\documentclass[10pt]{memoir}
\pdfoutput=1

\usepackage[english]{babel}

\usepackage[utf8]{inputenc}
\usepackage[T1]{fontenc}

\usepackage{enumitem}

\usepackage[margin = 0em, caption = off]{subfig}

\usepackage{longtable}

\usepackage{graphicx}

\usepackage[mode = image]{standalone}

\usepackage{todonotes}

\usepackage[reset = true, backref = true]{enotez}
\DeclareInstance{enotez-list}{plain}{paragraph}
{%
    heading = \subsection*{Notes},
    format = \normalsize,
    notes-sep = 0.4\baselineskip
}%
\usepackage{float}

\usepackage{siunitx}

\usepackage[nomessages]{fp}

\usepackage{mathtools}

\usepackage{dsfont}% double stroke
\usepackage{amsfonts}% blackboard bold
\usepackage{bm}% bold math

\usepackage{xparse}

\usepackage{amssymb}

\newcommand{\unit}{u}% complex unit root
\newcommand{\dequal}{\overset{\scriptscriptstyle\textup{d}}{=}}% equal in distribution
\newcommand{\euler}{\mathrm{e}}% euler constant
\newcommand{\BigO}{\mathrm{O}}% big O notation
\newcommand{\SmallO}{\mathrm{o}}% small O notation
\newcommand{\defi}{\coloneqq}% definition :=
\newcommand{\ifed}{\eqqcolon}% definition =:
\newcommand{\LSTold}[2]{\expandafter\tilde#1(#2)}% Laplace-Stieltjes transform
\newcommand{\PGFold}[2]{\expandafter\hat#1(#2)}% probability generating function
\newcommand{\PGFvar}{P}% letter used for a PGF
\newcommand{\LSTvar}{L}% letter used for a LST
\newcommand{\LSTarg}{\omega} % argument of a LST
\newcommand{\PGFarg}{z} % argument of a PGF
\NewDocumentCommand \PGF { m g }{%
    \IfNoValueTF{#2}% Check if second argument is given
        {\PGFvar(#1)}% If argument is not given, then use only the first argument
        {\PGFvar_{#1}(#2)} % If both arguments are given, use both
    }%
\NewDocumentCommand \PGFder { m g }{%
    \IfNoValueTF{#2}% Check if second argument is given
        {\PGFvar'(#1)}% If argument is not given, then use only the first argument
        {\PGFvar_{#1}'(#2)} % If both arguments are given, use both
    }%
\NewDocumentCommand \PGFvec { m g }{%
    \IfNoValueTF{#2}% Check if second argument is given
        {\mathbf{\PGFvar}(#1)}% If argument is not given, then use only the first argument
        {\mathbf{\PGFvar}_{#1}(#2)} % If both arguments are given, use both
    }%
\NewDocumentCommand \PGFvecder { m g }{%
    \IfNoValueTF{#2}% Check if second argument is given
        {\mathbf{\PGFvar}'(#1)}% If argument is not given, then use only the first argument
        {\mathbf{\PGFvar}_{#1}'(#2)} % If both arguments are given, use both
    }%
\NewDocumentCommand \LST {m g }{%
    \IfNoValueTF{#2}% Check if second argument is given
        {\LSTvar(#1)}% If argument is not given, then use only the first argument
        {\LSTvar_{#1}(#2)} % If both arguments are given, use both
    }%
\NewDocumentCommand \LSTder {m g }{%
    \IfNoValueTF{#2}% Check if second argument is given
        {\LSTvar'(#1)}% If argument is not given, then use only the first argument
        {\LSTvar_{#1}'(#2)} % If both arguments are given, use both
    }%
\NewDocumentCommand \LSTbig {m g }{%
    \IfNoValueTF{#2}% Check if second argument is given
        {\LSTvar\bigl(#1\bigr)}% If argument is not given, then use only the first argument
        {\LSTvar_{#1}\bigl(#2\bigr)} % If both arguments are given, use both
    }%

\newcommand{\approximate}[1]{\tilde{#1}} % approximate quantity

\newcommand{\RealPart}[1]{\mathrm{Re}(#1)}
\newcommand{\RealPartLargerBrackets}[1]{\mathrm{Re}\bigl( #1 \bigr)}
\newcommand{\ImagPart}[1]{\mathrm{Im}(#1)}
\newcommand{\complexunit}{\mathrm{i}}

\newcommand{\ind}[1]{\mathds{1}\{#1\}} % indicator function
\newcommand{\sprad}[1]{\operatorname{sp}(#1)} % spectral radius
\newcommand{\trace}[1]{\operatorname{trace}(#1)} % trace of a matrix

\newcommand{\dinf}{\textup{d}} % upright d

\newcommand{\Nat}{\mathbb{N}} % natural numbers starting at 1
\newcommand{\Int}{\mathbb{Z}} % integer numbers
\newcommand{\Real}{\mathbb{R}} % real numbers
\newcommand{\Complex}{\mathbb{C}} % complex numbers
\newcommand{\set}[1]{\mathcal{#1}} % set
\newcommand{\setUncountable}[1]{\mathcal{#1}} % set of uncountable elements

\newcommand{\closedunitdisc}{\setUncountable{U}} % closed unit disc
\newcommand{\unitcircle}{\partial \setUncountable{U}} % unit circle

\NewDocumentCommand \E { m g }{% expectation with regular braces
    \IfNoValueTF{#2}% Check if second argument is given
        {\mathbb{E}[#1]}% If argument is not given, then use only the first argument
        {\mathbb{E}_{#1}[#2]} % If both arguments are given, use both
    }%
\NewDocumentCommand \Efxd { m g }{% expectation with fixed, larger braces
    \IfNoValueTF{#2}% Check if second argument is given
        {\mathbb{E}\Bigl[#1\Bigr]}% If argument is not given, then use only the first argument
        {\mathbb{E}_{#1}\!\Bigl[#2\Bigr]} % If both arguments are given, use both
    }%
\NewDocumentCommand \Prob { m g }{% probability with regular parentheses
    \IfNoValueTF{#2}% Check if second argument is given
        {\mathbb{P}(#1)}% If argument is not given, then use only the first argument
        {\mathbb{P}_{#1}(#2)} % If both arguments are given, use both
    }%
\NewDocumentCommand \Probfxd { m g }{% probability with fixed, larger parentheses
    \IfNoValueTF{#2}% Check if second argument is given
        {\mathbb{P}\Bigl(#1\Bigr)}% If argument is not given, then use only the first argument
        {\mathbb{P}_{#1}\!\Bigl(#2\Bigr)} % If both arguments are given, use both
    }%
\newcommand{\Var}[1]{\operatorname{Var}(#1)} % variance
\newcommand{\Std}[1]{\sigma(#1)} % standard deviation

\newcommand{\Exp}[1]{\textup{Exp}(#1)}
\newcommand{\Erl}[2]{\textup{Erl}_{#1}(#2)}
\newcommand{\Poisson}[1]{\textup{Poi}(#1)}
\newcommand{\Geo}[1]{\textup{Geo}(#1)}

\newcommand{\Ber}[1]{\textup{Ber}(#1)}

\newcommand{\statespace}{\set{S}} % state space
\newcommand{\q}[1]{q_{#1}} % transition rates
\newcommand{\alt}[1]{\expandafter\bar#1} % alternative transition rates
\newcommand{\htt}[1]{\tau_{#1}} % hitting time

\newcommand{\bld}[1]{\bm{#1}} % bold notation
\newcommand{\vc}[1]{\bld{#1}} % vectors
\newcommand{\vca}[1]{\mathbf{#1}} % alternative notation for vectors
\newcommand{\oneb}{\vca{1}} % column vector of ones
\NewDocumentCommand \eb { m g }{%
    \IfNoValueTF{#2}% Check if second argument is given
        {\vca{e}_{#1}}% If argument is not given, then use only the first argument
        {\vca{e}^{(#1)}_{#2}} % If both arguments are given, use both
    }%
\NewDocumentCommand \zerob { g }{%
    \IfNoValueTF{#1}% Check if optional argument is given
        {\vca{0}}% If optional argument is not given, then use no arguments
        {\vca{0}^{(#1)}} % If optional arguments are given, use it
    }%
\newcommand{\transpose}{\intercal} % transpose of a vector or matrix
\newcommand{\diag}[1]{\operatorname{diag}(#1)} % main diagonal matrix
\newcommand{\I}{I} % identity matrix

\newcommand{\pb}{\vca{p}}% equilibrium probability vectors
\newcommand{\lvl}[1]{\set{L}_{#1}}% level i
\newcommand{\lvls}[1]{\set{L}_{\le #1}}% levels 0 until and including i

\newcommand{\la}{\lambda}
\newcommand{\La}{\Lambda}
\newcommand{\al}{\alpha}
\newcommand{\be}{\beta}
\newcommand{\ga}{\gamma}
\newcommand{\Ga}{\Gamma}

\newcommand{\rootx}{\xi} % x = root(y) in the kernel equation
\newcommand{\rooty}{\upsilon} % y = root(x) in the kernel equation
\newcommand{\Rooty}{\Upsilon} % function of y = root(x)
\newcommand{\radiusNI}{r} % radius used in numerical inversion
\newcommand{\discretizationError}{\gamma} % discretization error

\newcommand{\smallminus}{{\scriptscriptstyle -}}
\newcommand{\smallplus}{{\scriptscriptstyle +}}
\newcommand{\smallplusminus}{{\scriptscriptstyle \pm}}

\newcommand{\makeExercise}{}%{\sidepar{\boxed{\textup{exer.}}}} 

\setstocksize{235mm}{155mm}
\settrimmedsize{235mm}{155mm}{*}

\settrims{0mm}{0mm}

\settypeblocksize{185mm}{118mm}{*}

\setlrmargins{18.5mm}{*}{*}

\setulmargins{24mm}{*}{*}

\setheaderspaces{*}{7mm}{*}

\checkandfixthelayout

\setsecnumdepth{subsection}

\numberwithin{equation}{chapter}

\maxtocdepth{section}

\setpnumwidth{3em}
\setrmarg{4em}

\nouppercaseheads

\aliaspagestyle{title}{empty}

\aliaspagestyle{part}{empty}

\newif\ifusehyperref
\usehyperreftrue

\ifusehyperref
    \usepackage[colorlinks]{hyperref}

    \usepackage{refcount}
    \usepackage{bookmark}

    \usepackage{xcolor}

    \colorlet{coloreqref}{cyan}
    \colorlet{colorcref}{magenta}
    \colorlet{colorcite}{orange}

    \AtBeginDocument{\hypersetup{%
    linkcolor  = colorcref,
    citecolor  = colorcite
    }}%
\fi
\usepackage[capitalise,noabbrev]{cleveref}

\ifusehyperref
    \makeatletter
    \renewcommand{\toclevel@part}{10}
    \makeatother

    \newcommand*{\SavedEqref}{}
    \let\SavedEqref\eqref
    \renewcommand*{\eqref}[1]{%
    \begingroup
    \hypersetup{
      linkcolor=coloreqref
    }%
    \SavedEqref{#1}%
    \endgroup
    }%
\fi

\usepackage{amsthm}
\usepackage{thmtools}

\declaretheoremstyle[%
    spaceabove = 5pt,
    spacebelow = 5pt,
    postheadspace = 0.5em,
    headfont = \bfseries,
    bodyfont = \itshape,
    notefont = \normalfont\small\scshape,%\normalfont,
    notebraces = {\!\textup{(}}{\textup{)}},%{(}{)},
    headpunct = {},
    headformat = {\NAME~\NUMBER.\@\NOTE}
    ]{italic}

\declaretheoremstyle[%
    spaceabove = 5pt,
    spacebelow = 5pt,
    postheadspace = 0.5em,
    headfont = \bfseries,
    bodyfont = \upshape,
    notefont = \normalfont\small\scshape,%\normalfont\itshape,
    notebraces = {\textup{(}}{\textup{)}},%{\textup{(}}{\textup{)}},
    headpunct = {},
    headformat = {\NAME~\NUMBER.\@\NOTE}
    ]{upright}

\declaretheoremstyle[%
    spaceabove = 5pt,
    spacebelow = 5pt,
    postheadspace = 0.5em,
    headfont = \bfseries,
    bodyfont = \upshape,
    notefont = \normalfont\small\scshape,%\normalfont\itshape,
    notebraces = {\textup{(}}{\textup{)}},%{\textup{(}}{\textup{)}},
    headpunct = {},
    headformat = {\NAME.\@\NOTE}
    ]{upright_no_number}

\declaretheoremstyle[%
    spaceabove = 5pt,
    spacebelow = 5pt,
    postheadspace = 0.5em,
    headfont = \bfseries,
    bodyfont = \upshape,
    notefont = \normalfont\small\scshape,%\normalfont\itshape,
    notebraces = {\textup{(}}{\textup{)}},%{\textup{(}}{\textup{)}},
    headpunct = {},
    headformat = {\NAME~\NUMBER.\@\NOTE},
    qed = $\triangle$
    ]{upright_qed}

\declaretheorem[style = italic, parent = chapter]{theorem}
\declaretheorem[style = italic, sibling = theorem]{lemma}
\declaretheorem[style = italic, sibling = theorem]{proposition}

\declaretheorem[style = upright, sibling = theorem]{definition}
\declaretheorem[style = upright_qed, sibling = theorem]{remark}

\declaretheorem[style = upright_qed, sibling = theorem]{example}

\usepackage[chapter]{algorithm}
\usepackage{algpseudocode}

\begin{document}%

\frontmatter%
\thispagestyle{empty}%
%\begin{adjustwidth*}{7mm}{-7mm}%
\begin{vplace}[0.6]%
\begin{center}%
{\huge Analysis of structured Markov processes} \\
\vspace{1em}
{\Large Ivo Adan, Johan van Leeuwaarden, Jori Selen} \\
\vspace{2em}
{\normalsize version September 26, 2017}
\end{center}%
\end{vplace}%
%\end{adjustwidth*}%%
\chapter*{Preface}%
\label{ch:preface}%

Markov processes are popular mathematical models, studied by theoreticians for their intriguing properties, and applied by practitioners for their flexible structure. With this book we teach how to model and analyze Markov processes. We classify Markov processes based on their structural properties, which in turn determine which analytic methods are required for solving them. In doing so, we start in each chapter with specific examples that naturally lead up to general theory and general methods. In this way the reader learns about Markov processes on the job.

By studying this book, the reader becomes acquainted with the basic analytic methods that come into play when systems are modeled as structured Markov processes. These basic methods will likely prove useful, in real-time when studying the examples at hand, but more importantly for future encounters with Markov processes not covered in this book. Methods are more important than examples. The methods have a large scope of application, even outside the scope of Markov processes, in areas like probability theory, industrial engineering, mechanical engineering, physics and financial mathematics.

This book arose from various courses taught in the last decade at master level and postgraduate level. We thank the students and colleagues that participated in these courses for their valuable feedback.
\cleardoublepage%
{%
\ifusehyperref
    \hypersetup{hidelinks}
    \pdfbookmark[chapter]{\contentsname}{toc}
\fi
\tableofcontents*
}%

\mainmatter%

% Checked points 1-6 and a-i
\chapter{Introduction}%
\label{ch:introduction}%

Markov processes provide essential instruments for modeling and analyzing a large variety of systems and networks, including manufacturing systems, communication networks, traffic networks and service systems such as clinics or hospitals. This book provides the basic tools you need to build models that are detailed enough to capture the essential system dynamics, but are simple enough in terms of mathematical structure to be amenable for theoretical analysis and efficient numerical evaluation. The first two parts of this book assume only prior exposure to stochastic processes, linear algebra and basic analysis at the undergraduate level. The third part is meant for graduate students, researchers and practitioners, and requires more background in probability theory and complex analysis.

Markov processes fall under the umbrella of \textit{Stochastics}, the branch of mathematics that aims to establish rigorous statements about systems that are inherently uncertain, and therefore subject to some degree of randomness. A classical example is a queue, in which jobs need to wait for service. The queue grows when new jobs arrive and shrinks when jobs complete service. Queues occur virtually everywhere and can be seen as stochastic systems that are subject to variability in arrivals and services. Under certain assumptions, a queueing system can be modeled as a Markov process and analyzed using the techniques described in this book. This analytic treatment of a queue then leads to explicit formulas or algorithms for performance measures such as the mean queue length or the probability that the queue grows beyond a certain level. Such performance measures often reveal critical dependencies between the system performance and the system utilization. In fact, many real-life systems operate in regimes that dwarf the trade-off between high system utilization and short queues, two confliction goals. The analysis of Markov processes therefore also serves the purpose of dimensioning, with the objective to balance the system capacity and demand so as to achieve a certain target performance standard or optimize a certain cost criterion.

%We use as a compass for this book the closed loop in \cref{figINT:compass}. We often take real-life systems and show how the dynamics of these systems can be captured in terms of mathematical models. We then analyze the models to characterize the system performance. To close the loop, we then ask how to redesign or adjust the system to let it operate close to some desired performance target. The three technical steps that each require different skills are: \mytodo{Is dit nog wel zo relevant?}
%%
%\begin{enumerate}%
%\item Translating a system into a mathematical model;
%\item Analyzing a mathematical model to obtain performance measures;
%\item Using the performance measures to dimension a system.
%\end{enumerate}%
%%
%
%\begin{figure}%
%\centering%
%\includestandalone{Chapters/INT/TikZFiles/closed_loop}%
%\caption{Compass for this book.}%
%\label{figINT:compass}%
%\end{figure}%

%%%%%%%%%%%%%%%%%%%%%%%%%%%%%%%%%%%%%%%%%%%%%%%%%%%%%%%
%%%%%%%%%%%%%%%%%%%%%%%%%%%%%%%%%%%%%%%%%%%%%%%%%%%%%%%
%%%%%%%%%%%%%%%%%%%%% NEW SECTION %%%%%%%%%%%%%%%%%%%%%
%%%%%%%%%%%%%%%%%%%%%%%%%%%%%%%%%%%%%%%%%%%%%%%%%%%%%%%
%%%%%%%%%%%%%%%%%%%%%%%%%%%%%%%%%%%%%%%%%%%%%%%%%%%%%%%

\section{A balance act}%
\label{secINT:balance_act}%

This book deals with obtaining the equilibrium distribution that characterizes the long-term fractions of time that the Markov process spends in each of the possible states. Think of a queue that evolves in time. What is the long-term probability that the queue is empty? If we denote this probability by $p(0)$, we could estimate it by simply observing the queue for a very long time and divide the total time that the queue is empty by the total time we have observed the queue. We could similarly estimate the probability $p(i)$ of seeing a queue of size $i$.

\begin{figure}
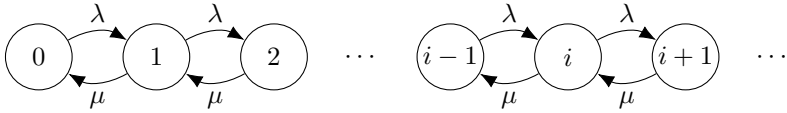
%
\centering%
\includestandalone{Chapters/INT/TikZFiles/transition_rate_diagram_MM1}%
\caption{Transition rate diagram of the Markov process of the simple queue.}%
\label{figINT:transition_rate_diagram_MM1}%
\end{figure}%

Instead of this brute-force approach to estimate $p(i)$ for all possible $i$, we will use the structure that is governed by the interaction between states. For a simple queue in which one job could leave or join, it is clear that $p(i)$ should be related to $p(i - 1)$ and $p(i + 1)$. And indeed, under some further assumptions, we could argue that the probabilities $p(i)$ should satisfy certain \textit{balance equations}. A famous example is the simplest possible queue that serves jobs at an exponential rate $\mu$ and to which new jobs arrive at exponential rate $\la$. Because of the exponential rates, at any moment in time only one event can happen: a new arrival or a service completion. The Markov process that describes the queue size evolves on the state space $\{ 0, 1, 2, \ldots \}$ according to the rates displayed in \cref{figINT:transition_rate_diagram_MM1}. \cref{figINT:transition_rate_diagram_MM1} is called a transition rate diagram and displays the states of the Markov process with the arrows depicting the rates at which the process transitions from one state to the other. Rate $\la$ should be smaller than $\mu$, otherwise the queue will grow to infinity, and under this assumption, the balance equations are given by
\begin{subequations}%
\begin{align}%
\la p(0) &= \mu p(1), \label{eqnINT:balance_equations_i=0} \\
(\la + \mu) p(i) &= \la p(i - 1) + \mu p(i + 1), \quad i \ge 1. \label{eqnINT:balance_equations_i>0}
\end{align}%
\label{eqnINT:balance_equations}%
\end{subequations}%
You can interpret these equations as what goes out should equal what comes in (either from the left or the right). These balance equations together can be written as the system of linear equations
\begin{equation}%
\pb Q = \zerob \label{eqnINT:balance_equations_vector-matrix_form}
\end{equation}%
with $\pb = \begin{bmatrix} p(0) & p(1) & p(2) & \cdots \end{bmatrix}$ and $Q$ the transition rate matrix given by
\begin{equation}%
Q = \begin{bmatrix}%
-\la & \la \\
\mu & -(\la + \mu) & \la \\
    & \mu & -(\la + \mu) & \la \\
    &     & \ddots & \ddots & \ddots
\end{bmatrix}.%
\end{equation}%
All Markov processes considered in this book can be brought into the matrix form \eqref{eqnINT:balance_equations_vector-matrix_form}. For the readers familiar with linear algebra, this makes available a powerful toolbox for numerically solving for $\pb$ as the null space of the kernel $Q$. But this is not the road we will pursue in this book. Instead, we will try to exploit additional structures that are hidden in the general matrix equation \eqref{eqnINT:balance_equations_vector-matrix_form}. For the simple queue we know for instance that $Q$ is extremely sparse and contains only elements on three diagonals. Moreover, the state space in \cref{figINT:transition_rate_diagram_MM1} and the form of \eqref{eqnINT:balance_equations} may allow an iterative solution. Indeed, using $\la p(0) = \mu p(1)$ one gets for $i = 1$ from \eqref{eqnINT:balance_equations_i>0} that $\la p(1) = \mu p(2)$, and more generally,
\begin{equation}%
\la p(i) = \mu p(i + 1), \quad i \ge 0.
\end{equation}%
Iteration then gives $p(i) = p(0) \rho^i$ with $\rho \defi \la/\mu$. Since $\sum_{i \ge 0} p(i) = 1$ we conclude that $p(0) = 1 - \rho$ to arrive at the elegant solution
\begin{equation}%
p(i) = (1 - \rho) \rho^i, \quad i \ge 0. \label{eqnINT:product-form_solution}
\end{equation}%
This is our very first product-form solution! And if you like it, many more will follow for more advanced, yet structured, Markov processes.

We call \eqref{eqnINT:product-form_solution} a product-form solution, because of the term $\rho^i$, the product of $i$ times $\rho$. Most of the Markov processes in this book are multi-dimensional, in which case we encounter multi-dimensional product forms, for instance of the types $\al^i \be^j$ or $R^i$ with $R$ some matrix (instead of a scalar $\rho$). While in most cases, finding these product forms will be less straightforward than in the case of \eqref{eqnINT:product-form_solution}, we will often use ways to exploit recurring structures. A second technique we will often use is that of making an educated guess. Through many examples we learn the reader when to expect a product form (and when not). If we return to \eqref{eqnINT:balance_equations} and we would guess that $p(i)$ is of the form $c \al^i$ with $c$ and $\al$ some unknown constants, we could simply substitute $p(i) = c \al^i$ in \eqref{eqnINT:balance_equations_i>0} to obtain
\begin{equation}%
(\la + \mu) c \al^i = \la c \al^{i - 1} + \mu c \al^{i + 1},
\end{equation}%
or equivalently,
\begin{equation}%
(\la + \mu) \al = \la + \mu \al^2, \label{eqnINT:balance_equations_educated_guess}
\end{equation}%
from which we conclude that $\al = \rho$ and $c = 1 - \rho$. Although this guessing technique appears naive at first sight, it is a mathematical rigorous way of proving that \eqref{eqnINT:product-form_solution} uniquely characterizes the equilibrium distribution. The substitution of product forms in difference equations like \eqref{eqnINT:balance_equations} is a well-known analytic technique, but when the difference equation is in fact a balance equation there are some specific features that can be exploited. For instance, we know from the start that all $p(i)$ are nonnegative and that $\sum_{i \ge 0} p(i) = 1$. The latter condition we have used in \eqref{eqnINT:balance_equations_educated_guess} to conclude that $\al = \rho$ is the unique solution, and not the other candidate solution of \eqref{eqnINT:balance_equations_educated_guess}. Indeed, only when $\al < 1$ the infinite series $\sum_{i \ge 0} p(i)$ converges to a finite constant. The final step concludes that $(1 - \al)^{-1} = c^{-1} = p(0)^{-1}$, which can be interpreted as solving the boundary condition $\sum_{i \ge 0} p(i) = 1$. For this simple Markov process this gives $p(0) \sum_{i \ge 0} \al^i = 1$. In this case this boundary condition gives one additional equation for solving $p(0)$. For the more advanced Markov processes in this book the boundary conditions give rise to an additional system of equations from which equally many remaining equilibrium probabilities need to be determined.

%%%%%%%%%%%%%%%%%%%%%%%%%%%%%%%%%%%%%%%%%%%%%%%%%%%%%%%
%%%%%%%%%%%%%%%%%%%%%%%%%%%%%%%%%%%%%%%%%%%%%%%%%%%%%%%
%%%%%%%%%%%%%%%%%%%%% NEW SECTION %%%%%%%%%%%%%%%%%%%%%
%%%%%%%%%%%%%%%%%%%%%%%%%%%%%%%%%%%%%%%%%%%%%%%%%%%%%%%
%%%%%%%%%%%%%%%%%%%%%%%%%%%%%%%%%%%%%%%%%%%%%%%%%%%%%%%

\section{Why this book?}%
\label{secINT:why_this_book}%

Many books have been written about general stochastic processes and Markov processes in particular. This book views Markov processes as continuous-time processes, and studies their equilibrium or long-term behavior. Finding the equilibrium distributions requires solving a system of difference or difference-differential equations. Each Markov process in this book comes with its own system of equations, and its own specific challenges. We classify the Markov processes by the analytic methods required to solve the system of equations, which in turn depends strongly on the underlying structure of Markov processes. The reader will learn to recognize these structures, and hence choosing the adequate methods for analysis. While this book puts much emphasis on basic real and complex analysis, less attention goes to the more formal or probabilistic aspects of Markov processes, for instance related to operators, function spaces, martingale characterizations, stability, and weak convergence to limiting processes. Excellent books exist that cover these topics in much depth, for example Ethier and Kurtz \cite{Ethier1986_Markov_processes}, Feller \cite{Feller1968_Probability_theory,Feller1971_Probability_theory_volume_II} and Whitt \cite{Whitt2002_Stochastic-process_limits}.

This book is not just about queueing theory. While queueing theory generates intriguing questions that can be answered using the theory of Markov processes, this book only introduces queueing models that ask for a different analytic method. More advanced queueing models---that arise for instance when relaxing Markovian assumptions---are not treated just for the sake of generalization or enhancing the scope of applicability. Books with more theory and examples of queues are for instance Cohen \cite{Cohen1969_Single_server_queue}, Gross and Harris \cite{Gross1974_Fundamentals_of_queueing_theory}, Kleinrock \cite{Kleinrock1975_Queueing_systems_volume_I}, Prabhu \cite{Prabhu1997_Foundations_of_QT}, Robert \cite{Robert2013_Stochastic_networks_and_queues} and Tak{\'a}cs \cite{Takacs1962_Theory_Queues}. Parts of the material covered in this book can also be found in text books on applied probability or Markov chains, such as Asmussen \cite{Asmussen2008_Applied_probability_and_queues}, Chung \cite{Chung2001_Course_in_probability_theory}, Grimmett and Stirzaker \cite{Grimmett2001_Probability_and_random_processes}, Karlin and Taylor \cite{Karlin1975_First_course_stochastic_processes}, Liggett \cite{Liggett2010_CTMC}, Norris \cite{Norris1997_Markov}, Resnick \cite{Resnick1992_Adventures} and Ross \cite{Ross1996_Stochastic_processes}, although the same topics are often presented in a different manner. This book is complementary, again because of the dominant role of exact analysis, product-form solutions, and structure of Markov processes.

%Learning the methods is more important than learning the models, also because the methods have a large scope of application in, e.g., manufacturing systems, combinatorics, probability theory, physics and financial mathematics. By going through this book, the reader becomes acquainted with the basic analytic tools that come into play when systems are modeled as Markov processes. These basic tools will likely prove useful, in real-time when studying the problems at hand, and for future encounters with Markov processes not covered in this book.

%Part I covers the essential theory of continuous-time Markov processes and some basic methods. Part II focusses on several classes of widely studied structured Markov processes, including birth--and--death processes, single-server queues, queueing networks and quasi-birth--and--death processes. Part III deals with more advanced processes that require more ingenuous mathematical treatment. In Part III the reader will learn advanced tools from analysis through the detailed study of a series of multi-dimensional, but highly structured Markov processes.

%%%%%%%%%%%%%%%%%%%%%%%%%%%%%%%%%%%%%%%%%%%%%%%%%%%%%%%
%%%%%%%%%%%%%%%%%%%%%%%%%%%%%%%%%%%%%%%%%%%%%%%%%%%%%%%
%%%%%%%%%%%%%%%%%%%%% NEW SECTION %%%%%%%%%%%%%%%%%%%%%
%%%%%%%%%%%%%%%%%%%%%%%%%%%%%%%%%%%%%%%%%%%%%%%%%%%%%%%
%%%%%%%%%%%%%%%%%%%%%%%%%%%%%%%%%%%%%%%%%%%%%%%%%%%%%%%

\section{Overview}%
\label{secINT:overview}%

In \textbf{Part I} we cover the essential theory of continuous-time Markov processes and some basic methods. We furthermore introduce some common queues with their associated Markov processes and present how transforms are useful in the analysis of such Markov processes.

\cref{ch:Markov_processes} covers the foundations needed to build a Markov process. Essential ingredients are the exponential distribution and its memoryless property, which makes that after each event that takes place in the Markov process, we can forget about the past and only use the current information. We introduce notions like irreducibility, positive recurrence and regularity. Brief consideration is given to the evolution of Markov processes as a function of time, but we will focus mostly on the long-term of equilibrium behavior.

\cref{ch:queues_and_transforms} introduces the Laplace-Stieltjes transform and the probability-generating function. Both transforms play a crucial role in the analysis of the equilibrium distribution and other related quantities. We demonstrate the use of the transforms together with other important results by analyzing single-server queues that are at the heart of queueing theory. Numerical inversion algorithms are provided to retrieve the underlying probability distributions from their transforms.

\vspace*{0.5em}

In \textbf{Part II} we focus on several classes of widely studied structured Markov processes, including birth--and--death processes, queueing networks, quasi-birth--and--death processes and quasi-skip-free processes. Each chapter is dedicated to one class of processes and introduces the techniques required to obtain their equilibrium distributions.

\cref{ch:birth--and--death_processes} is devoted entirely to birth--and--death (BD) processes, a highly structured class of Markov processes. The distinguishing feature of BD processes is that the state space can be ordered on a line and that transitions occur only between neighboring states. The queue in \cref{figINT:transition_rate_diagram_MM1} is an example of a BD process. Like that queue, all BD processes have product-form equilibrium distributions that can be solved iteratively. The class of BD processes contains many classical Markov processes that occur in queueing theory or in epidemics.

\cref{ch:reversible_networks} extends the one-dimensional BD processes to multi-dimensional network models, and hence multi-dimensional Markov processes. Although these Markov processes have multiple dimensions, the equilibrium distribution can often be derived by making an educated guess.

\cref{ch:quasi-birth--and--death_processes} again extends the BD processes of \cref{ch:birth--and--death_processes}, but now by including a finite second dimension. Here we encounter product-form solutions that involve matrices instead of the scalars that we have seen in this introduction. We discuss the matrix-geometric, matrix-analytic and spectral expansion method. Deriving an explicit expression for the matrices in the product-form solution proves to be difficult in many cases, so also numerical algorithms are provided to obtain these matrices.

\cref{ch:skip-free_one_direction} considers Markov processes on the same state space as the QBD processes of \cref{ch:quasi-birth--and--death_processes}. The difference is that in \cref{ch:skip-free_one_direction} we allow the process to have larger jumps in one direction. The structure of the solution for the equilibrium distribution is similar to the one for the QBD process, but calculating the matrices of interest is more involved.

\vspace*{0.5em}

In \textbf{Part III} we tackle specific models that require advanced techniques to obtain the equilibrium distribution. Each chapter in this part is devoted to a specific model and for each model we develop multiple techniques to calculate the equilibrium distribution. The models serve as a vehicle through which we can demonstrate various techniques and allow the reader to compare methods. While applying the methods, we regularly exploit structural properties of the Markov process to obtain explicit expression for the equilibrium probabilities.

\cref{ch:single-server_priority} considers a queueing system consisting of a single server and two priority classes, where low-priority jobs are only served when there are no high-priority jobs in the system. We model this system as a Markov process with two dimensions, where the dimensions keep track of the number of jobs of each class in the system. We demonstrate a difference equations approach, the generating function approach and two approaches related to QBD processes to obtain the equilibrium distribution.

\cref{ch:gated_single-server} describes a single-server queue where waiting jobs are only allowed into the system when the system empties. We present three methods to obtain the equilibrium distribution of the associated two-dimensional Markov process: the generating function approach, the matrix-geometric approach and the compensation approach.

\cref{ch:cyclic_production_systems} covers three different production systems that give rise to two-dimensional Markov processes. The first two models are QBD processes and the third model has two countably infinite dimensions. For each system we present a tailor-made solution method to obtain the equilibrium distribution.

\cref{ch:join_the_shortest_queue} analyzes a system consisting of two single-server queues where an arriving job joins the shortest of the two queues. The dynamics of this model are described by a Markov process that takes values in the positive half-plane. We use the compensation approach to determine the equilibrium probabilities.

%%%%%%%%%%%%%%%%%%%%%%%%%%%%%%%%%%%%%%%%%%%%%%%%%%%%%%%
%%%%%%%%%%%%%%%%%%%%%%%%%%%%%%%%%%%%%%%%%%%%%%%%%%%%%%%
%%%%%%%%%%%%%%%%%%%%% NEW SECTION %%%%%%%%%%%%%%%%%%%%%
%%%%%%%%%%%%%%%%%%%%%%%%%%%%%%%%%%%%%%%%%%%%%%%%%%%%%%%
%%%%%%%%%%%%%%%%%%%%%%%%%%%%%%%%%%%%%%%%%%%%%%%%%%%%%%%

%\section{For the reader}%
%\label{secINT:reader}%
%

%%%%%%%%%%%%%%%%%%%%%%%%%%%%%%%%%%%%%%%%%%%%%%%%%%%%%%%
%%%%%%%%%%%%%%%%%%%%%%%%%%%%%%%%%%%%%%%%%%%%%%%%%%%%%%%
%%%%%%%%%%%%%%%%%%%%% NEW SECTION %%%%%%%%%%%%%%%%%%%%%
%%%%%%%%%%%%%%%%%%%%%%%%%%%%%%%%%%%%%%%%%%%%%%%%%%%%%%%
%%%%%%%%%%%%%%%%%%%%%%%%%%%%%%%%%%%%%%%%%%%%%%%%%%%%%%%

%\section{For the instructor}%
%\label{secINT:instructor}%
%

%%%%%%%%%%%%%%%%%%%%%%%%%%%%%%%%%%%%%%%%%%%%%%%%%%%%%%%
%%%%%%%%%%%%%%%%%%%%%%%%%%%%%%%%%%%%%%%%%%%%%%%%%%%%%%%
%%%%%%%%%%%%%%%%%%%%%%%% NOTES %%%%%%%%%%%%%%%%%%%%%%%%
%%%%%%%%%%%%%%%%%%%%%%%%%%%%%%%%%%%%%%%%%%%%%%%%%%%%%%%
%%%%%%%%%%%%%%%%%%%%%%%%%%%%%%%%%%%%%%%%%%%%%%%%%%%%%%%

%\theendnotes%
%\setcounter{endnote}{0}
%\printendnotes%%

%%%%%%%%%%%%%%%%%%%%%%%%%%%%%%%%%%%%%%%%%%%%%%%%%%%%%%%
%%%%%%%%%%%%%%%%%%%%%%%%%%%%%%%%%%%%%%%%%%%%%%%%%%%%%%%
%%%%%%%%%%%%%%%%%%%%%%% NEW PART %%%%%%%%%%%%%%%%%%%%%%
%%%%%%%%%%%%%%%%%%%%%%%%%%%%%%%%%%%%%%%%%%%%%%%%%%%%%%%
%%%%%%%%%%%%%%%%%%%%%%%%%%%%%%%%%%%%%%%%%%%%%%%%%%%%%%%

\part{Basic methods}%

% Did not check due to rewrite upcoming
%\input{Chapters/BMP/basic_markov_processes}%

% Did not check due to rewrite upcoming
%\input{Chapters/AMP/advanced_markov_processes}%

\chapter{Markov processes}%
\label{ch:Markov_processes}%

Markov processes are stochastic processes whose future behavior only depends on the present and not on the past. This special property makes Markov processes mathematically tractable. Markov processes therefore serve as widely applied models in areas as diverse as biology, physics, chemistry, logistics, economics and social sciences.

In this book we analyze a host of Markov processes. In this chapter we present the mathematical notions that are required to define Markov processes. To that end we start with a discussion of the exponential distribution, which is essential in the construction of Markov processes. We then show how to build Markov processes and discuss some of the basic properties. We will study Markov processes as functions of time, but our main focus in this chapter and throughout the remainder of this book will be on the analysis of the long-term or equilibrium behavior.

%%%%%%%%%%%%%%%%%%%%%%%%%%%%%%%%%%%%%%%%%%%%%%%%%%%%%%%
%%%%%%%%%%%%%%%%%%%%%%%%%%%%%%%%%%%%%%%%%%%%%%%%%%%%%%%
%%%%%%%%%%%%%%%%%%%%% NEW SECTION %%%%%%%%%%%%%%%%%%%%%
%%%%%%%%%%%%%%%%%%%%%%%%%%%%%%%%%%%%%%%%%%%%%%%%%%%%%%%
%%%%%%%%%%%%%%%%%%%%%%%%%%%%%%%%%%%%%%%%%%%%%%%%%%%%%%%

\section{Exponential distribution}%
\label{secMP:exponential_distribution}%

The continuous random variable $X$ follows an \textit{exponential distribution} with parameter $\la > 0$, denoted by $X \sim \Exp{\la}$, if its probability density function is given by
\begin{equation}%
f_X(t) = \la \euler^{-\la t}, \quad t \ge 0,
\end{equation}%
and the associated cumulative distribution function is
\begin{equation}%
F_X(t) = \int_0^t f_X(u) \, \dinf u = 1 - \euler^{-\la t}, \quad t \ge 0.
\end{equation}%
It readily follows that the expectation and variance of $X$ are
\begin{equation}%
\E{X} = 1/\la \quad \textup{and} \quad \Var{X} = 1/\la^2.
\end{equation}%

The exponential distribution enjoys the so-called \textit{memoryless property} or \textit{Markov property}, which is arguably the most important property for analytic tractability of stochastic processes in this book. The property reads
\begin{align}%
\Prob{X > s + t \mid X > s} &= \frac{\Prob{X > s + t, X > s}}{\Prob{ X > s}} = \frac{\Prob{X > s + t}}{\Prob{ X > s}} \notag \\
&= \frac{ \euler^{-\la(s + t)}}{\euler^{-\la s}} = \euler^{-\la t} = \Prob{X > t}.
\end{align}%
Think of $X$ as the lifetime of some component. Then, the memoryless property states that the remaining lifetime of $X$, given that $X$ is still alive at time $s$, is again exponentially distributed with the same mean $1/\la$. In other words, the probability that $X$ dies in the next $t$ time units is independent of the current age $s$ of $X$. The exponential distribution is the only continuous distribution that satisfies this memoryless property. \makeExercise

Denote by $X_1,X_2,\ldots,X_n$ independent exponentially distributed random variables with parameters $\la_1,\la_2,\ldots,\la_n$. Define now the minimum over these random variables as $Y_n \defi \min(X_1,X_2,\ldots,X_n)$. We have
\begin{align}%
\Prob{Y_n \ge t} &= \Prob{ \min(X_1,X_2,\ldots,X_n) \ge t } \notag \\
&= \Prob{ X_1 \ge t, X_2 \ge t, \ldots, X_n \ge t } \notag \\
&= \Prob{ X_1 \ge t } \Prob{X_2 \ge t } \cdots \Prob{ X_n \ge t } \notag \\
&= \euler^{-\la_1 t} \euler^{-\la_2 t} \cdots \euler^{-\la_n t} = \euler^{-(\la_1 + \la_2 + \cdots + \la_n)t}, \label{eqnMP:minimum_over_exponential_is_again_exponential}
\end{align}%
where the third equality follows from the independence of the random variables. We have just proved the second important property of the exponential distribution: the minimum of $n$ exponential random variables is again an exponential random variable with parameter the sum of the $n$ parameters.

\begin{example}\label{exMP:minimum_exponential_random_variables}%
A printer can fail due to power outages, paper jams or ink shortages. Let these events be independent and occur after exponential times with rates $\la_{\textup{power}}$, $\la_{\textup{jam}}$ and $\la_{\textup{ink}}$. The up-time of the printer is the minimum time until any failure occurs and hence the up-time is exponentially distributed with parameter $\la_{\textup{power}} + \la_{\textup{jam}} + \la_{\textup{ink}}$ with mean up-time $1/(\la_{\textup{power}} + \la_{\textup{jam}} + \la_{\textup{ink}})$.
\end{example}%

Next consider the probability that an exponential random variable turns out to be the minimum among $n$ exponential random variables:
\begin{equation}%
\Prob{X_n = \min(X_1,X_2,\ldots,X_n)} \ifed \Prob{X_n = Y_n}. \label{eqnMP:exponential_rv_probability_to_be_minimum}
\end{equation}%
The event $\{ X_n = Y_n \}$ is the same as the event $\{ X_n \le Y_{n - 1} \}$ since it implies that the $n$-th exponential random variable is the minimum. Using $Y_{n - 1} \sim \Exp{\mu}$, where we abbreviated $\mu \defi \sum_{m = 1}^{n - 1} \la_m$, and conditioning on the length of $X_n$,
\begin{align}%
\Prob{X_n = Y_n} &= \Prob{X_n \le Y_{n - 1}} \notag \\
&= \int_0^\infty \Prob{Y_{n - 1} \ge t} f_{X_n}(t) \, \dinf t = \int_0^\infty \euler^{-\mu t} \la_n \euler^{-\la_n t} \, \dinf t \notag \\
&= \frac{\la_n}{\mu + \la_n} = \frac{\la_n}{\la_1 + \la_2 + \cdots + \la_n}.
\end{align}%
A similar reasoning shows that for any $k = 1,2,\ldots,n$,
\begin{equation}%
\Prob{X_k = \min(X_1,X_2,\ldots,X_n)} = \frac{\la_k}{\la_1 + \la_2 + \cdots + \la_n}.
\end{equation}%

Combining the two previous properties, one can even show that one particular $X_k$ being equal to $\min(X_1,X_2,\ldots,X_n)$ is independent of the value of $\min(X_1,X_2,\ldots,X_n)$. \makeExercise This will prove to be a very useful property when constructing Markov processes. Returning to \cref{exMP:minimum_exponential_random_variables}, this means that the printer fails due to an ink shortage with probability $\la_{\textup{ink}}/(\la_{\textup{power}} + \la_{\textup{jam}} + \la_{\textup{ink}})$.

%%%%%%%%%%%%%%%%%%%%%%%%%%%%%%%%%%%%%%%%%%%%%%%%%%%%%%%
%%%%%%%%%%%%%%%%%%%%%%%%%%%%%%%%%%%%%%%%%%%%%%%%%%%%%%%
%%%%%%%%%%%%%%%%%%%%% NEW SECTION %%%%%%%%%%%%%%%%%%%%%
%%%%%%%%%%%%%%%%%%%%%%%%%%%%%%%%%%%%%%%%%%%%%%%%%%%%%%%
%%%%%%%%%%%%%%%%%%%%%%%%%%%%%%%%%%%%%%%%%%%%%%%%%%%%%%%

\section{Poisson processes}%
\label{secMP:Poisson_processes}%

Before we introduce Markov processes in greater detail, we describe a specific type of Markov process called a Poisson process, a counting process that counts how many events have occurred in a time interval. For a Poisson process these events occur randomly in time and the time between two events is exponentially distributed with parameter $\la$. Denote by $\{N(t) \}_{t \ge 0}$ the Poisson process where $N(t)$ is the number of events that have occurred in the interval $[0,t]$ and set $N(0) = 0$.

We model the Poisson process as a collection of states representing the cumulative number of events that have occurred, and transitions between states that model the time needed to go to the next state, see \cref{figMP:Poisson_process}. A transition is marked with the rate at which it occurs. To be more precise, given that the process is in state $i$, a transition from state $i$ to state $i + 1$ occurs after an exponential amount of time with parameter $\la$, see \cref{figMP:Poisson_process} again. When modeling the Poisson process in this way, we have constructed a Markov process description of the Poisson process!

\begin{figure}
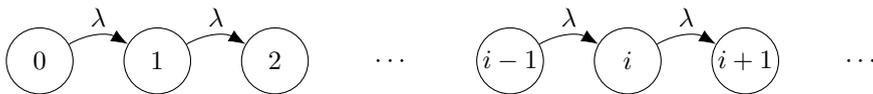
%
\centering%
\includestandalone{Chapters/MP/TikZFiles/Poisson_process}%
\caption{Modeling a Poisson process with rate $\la$.}%
\label{figMP:Poisson_process}%
\end{figure}%

Suppose that $X_1,X_2,\ldots$ are independent and identically exponentially distributed random variables with parameter $\la$. From \cref{figMP:Poisson_process} we find that the time to reach state $n$ is then $X_1 + X_2 + \cdots + X_n$. The probability that there are at most $n$ events in $[0,t]$ can be expressed as
\begin{equation}%
\Prob{N(t) \le n} = \Prob{X_1 + X_2 + \cdots + X_{n + 1} > t}, \quad n \ge 1, ~ t \ge 0. \label{eqnMP:Poisson_process_cdf}
\end{equation}%
In other words, the event to be in any of the states in $\{ 0,1,\ldots,n \}$ at time $t$ is equivalent to the event that the time it takes to reach state $n + 1$ is larger than $t$. The probability on the right-hand side of \eqref{eqnMP:Poisson_process_cdf} can be characterized further in terms of the \textit{Erlang distribution}. If $S_n \defi X_1 + X_2 + \cdots + X_n$, then $S_n$ follows an Erlang-$n$ distribution with parameter $\la$ denoted as $\Erl{n}{\la}$. The density of the $\Erl{n}{\la}$ distribution is given by
\begin{equation}%
f_{S_n}(t) = \la \frac{(\la t)^{n - 1}}{(n - 1)!} \euler^{-\la t}, \label{eqnMP:Erlang-n_distribution_pdf}
\end{equation}%
which we can prove using induction. For $n = 1$ we have $f_{S_1}(t) = \la \euler^{-\la t}$. Assume that $f_{S_n}(\cdot)$ satisfies \eqref{eqnMP:Erlang-n_distribution_pdf}. Then,
\begin{align}%
f_{S_{n + 1}}(t) &= f_{S_n + X_{n + 1}}(t) = \int_0^t f_{S_n}(u) f_{X_{n + 1}}(t - u) \, \dinf u \notag \\
&= \int_0^t \la \euler^{-\la u} \frac{(\la u)^{n - 1}}{(n - 1)!} \la \euler^{-\la(t - u)} \, \dinf u \notag \\
&= \la \euler^{-\la t} \frac{\la^n}{(n - 1)!} \int_0^t u^{n - 1} \, \dinf u \notag \\
&= \la \frac{(\la t)^n}{n!} \euler^{-\la t} ,
\end{align}%
showing that \eqref{eqnMP:Erlang-n_distribution_pdf} is correct. The cumulative distribution function we give without proof:
\begin{equation}%
F_{S_n}(t) = 1 - \sum_{m = 0}^{n - 1} \frac{(\la t)^m}{m!} \euler^{-\la t}, \quad n \ge 1, ~ t \ge 0. \label{eqnMP:Erlang-n_distribution_cdf}
\end{equation}%
Returning to \eqref{eqnMP:Poisson_process_cdf} armed with \eqref{eqnMP:Erlang-n_distribution_cdf}, we find
\begin{equation}%
\Prob{N(t) \le n} = \sum_{m = 0}^n \frac{(\la t)^m}{m!} \euler^{-\la t},
\end{equation}%
and therefore
\begin{equation}%
\Prob{N(t) = n} = \frac{(\la t)^n}{n!} \euler^{-\la t}, \quad n \ge 0, ~ t \ge 0.
\end{equation}%
The distribution of $N(t)$ is called a \textit{Poisson distribution} with parameter $\la t$. Some quick calculations show that
\begin{equation}%
\E{N(t)} = \la t \quad \textup{and} \quad \Var{N(t)} = \la t.
\end{equation}%
The expected number of events in $[0,t]$ is thus the rate $\la$ at which events occur multiplied by the length of the interval $t$.

The Poisson process is vital for modeling practical applications: to model the occurrence of software errors or machine breakdowns, the arrival of jobs at a processor, or the arrival of orders at a production system. It is empirically found that in many conditions the real-world processes can be well approximated by a Poisson process. We next establish a theoretical result that supports the assumption of Poisson processes in practical settings.

\begin{proposition}\label{propMP:binomial_distribution_to_Poisson}%
Let $X$ follow a \textup{binomial distribution} with parameters $n$ and $p$, that is
\begin{equation}%
\Prob{X = k} = \binom{n}{k} p^k (1 - p)^{n - k}, \quad k = 0,1,\ldots,n.
\end{equation}%
Let $p \to 0$ as $n \to \infty$ such that $n p = \la t$, then
\begin{equation}%
\Prob{X = k} \to \frac{(\la t)^k}{k!} \euler^{-\la t}, \quad k \ge 0, ~ t \ge 0.
\end{equation}%
\end{proposition}%

\begin{proof}%
Let $k$ be a fixed integer. Then we have
\begin{align}%
&\lim_{n \to \infty} \Prob{X = k} = \lim_{n \to \infty} \binom{n}{k} p^k (1 - p)^{n - k} \notag \\
&= \lim_{n \to \infty} \frac{n!}{(n - k)!k!} \bigl( \frac{\la t}{n} \bigr)^k \bigl( 1 -  \frac{\la t}{n} \bigr)^{n - k} \notag \\
%&= \lim_{n \to \infty} \frac{n (n - 1) (n - 2) \cdots (n - k + 1)}{k!} \bigl( \frac{\la t}{n} \bigr)^k \bigl( 1 -  \frac{\la t}{n} \bigr)^{n - k} \notag \\
&= \frac{(\la t)^k}{k!} \lim_{n \to \infty} \frac{n (n - 1) (n - 2) \cdots (n - k + 1)}{n^k} \bigl( 1 -  \frac{\la t}{n} \bigr)^{n - k} \notag \\
&= \frac{(\la t)^k}{k!} \lim_{n \to \infty} \bigl( 1 - \frac{1}{n} \bigr) \bigl( 1 - \frac{2}{n} \bigr) \cdots \bigl( 1 - \frac{k - 1}{n} \bigr) \bigl( 1 -  \frac{\la t}{n} \bigr)^n \bigl( 1 -  \frac{\la t}{n} \bigr)^{-k} \notag \\
&= \frac{(\la t)^k}{k!} 1 \cdot 1 \cdots 1 \cdot \euler^{-\la t} \cdot 1 = \frac{(\la t)^k}{k!} \euler^{-\la t},
\end{align}%
proving the statement.
\end{proof}%

Many real-world arrival processes fit into the framework of \cref{propMP:binomial_distribution_to_Poisson}. To see this, consider $n$ potential voters each having a small probability $p$ of arriving at a particular polling station in a small interval $[0,t]$. The probability that $k$ out of the $n$ voters show up in $[0,t]$ is binomially distributed: there are $\binom{n}{k}$ groups of size $k$ in a population of size $n$ and exactly $k$ voters arrive with probability $p^k (1 - p)^{n - k}$. If $n$ is large and $p$ is small, then the expression in terms of the Poisson distribution closely approximates the actual probability and is moreover easy to work with. In other words, if a large number $n$ of arrivals can occur in a time interval $[0,t]$ with a small probability $p$ and we can construct $\la$ such that $\la \approx np/t$, then the Poisson process with rate $\la$ closely approximates the arrival process at the polling station.

We next mention two important properties of a Poisson process. Suppose that $N_1(\cdot),N_2(\cdot),\ldots,N_n(\cdot)$ are independent Poisson processes with rates $\la_1,\la_2,\ldots,\la_n$. Define $N(t) \defi N_1(t) + N_2(t) + \cdots + N_n(t)$ for all $t \ge 0$. The time until a next event occurs for the counting process $N(t)$ is, by the memoryless property of the exponential distribution, the minimum over $n$ independent exponential random variables with parameters $\la_1,\la_2,\ldots,\la_n$. So, by \eqref{eqnMP:minimum_over_exponential_is_again_exponential} we have that the time until a next event is exponentially distributed with parameter $\la_1 + \la_2 + \cdots + \la_n$ and $\{ N(t) \}_{t \ge 0}$ is a Poisson process with rate $\la_1 + \la_2 + \cdots + \la_n$. This property is called the \textit{merging} property of independent Poisson processes.

For the second property, consider a Poisson process $\{ N(t) \}_{t \ge 0}$ with rate $\la$ where $X_1,X_2,\ldots$ are the times between events and each event is given a label out of $n$ possible labels. For each arrival, label $k$ is given with fixed probability $p_k > 0$ and $p_1 + p_2 + \cdots + p_n = 1$. Denote the number of events with label $k$ in the interval $[0,t]$ as $N_k(t)$. We determine the time $T$ until the next event occurs for the counting process $\{ N_k(t) \}_{t \ge 0}$. To that end, we require the total number of events that occur until the first time an event is given the label $k$ (this counts the last event with label $k$ as well). This is exactly a random variable $K$ with a \textit{geometric distribution}, that is,
\begin{equation}%
\Prob{K = i} = (1 - p_k)^{i - 1} p_k, \quad i \ge 1.
\end{equation}%
We can now express the time $T$ in terms of the time between events of the original Poisson process and $K$ as
\begin{equation}%
T \dequal \sum_{i = 1}^K X_i \sim \Exp{\la p_k}. \label{eqnMP:Poisson_process_splitting_still_need_to_prove}
\end{equation}%
%
%Now we will use a concept from \cref{ch:queues_and_transforms} called a Laplace-Stieltjes transform (LST), see \cref{subsecQTF:LST} in particular. For example, the LST of an exponential random variable $X$ with rate $\la$ is denoted by $\E{\euler^{- \LSTarg X}}$ and is equal to $\la/(\la + \LSTarg)$. If the LSTs of two random variables are equal, then the random variables have the same distribution. Taking the LST on the right-hand side and conditioning on $K$,
%%
%\begin{align}%
%\E{ \euler^{- \LSTarg \sum_{i = 1}^K X_i} } &= \sum_{j \ge 1} \E{ \euler^{- \LSTarg \sum_{i = 1}^K X_i} \mid K = j} \, \Prob{K = j} \notag \\
%&= \sum_{j \ge 1} \E{ \euler^{- \LSTarg \sum_{i = 1}^j X_i}} \, \Prob{K = j} = \sum_{j \ge 1} \E{ \euler^{- \LSTarg X_1}}^j \, \Prob{K = j} \notag \\
%&= \sum_{j \ge 1} \bigl( \frac{\la}{\la + \LSTarg} \bigr)^j (1 - p_k)^{j - 1} p_k = \frac{\la p_k }{\la + \LSTarg} \sum_{j \ge 0} \bigl( \frac{\la(1 - p_k)}{\la + \LSTarg} \bigr)^j \notag \\
%&= \frac{\la p_k}{\la + \LSTarg} \frac{\la + \LSTarg}{\la + \LSTarg - \la(1 - p_k)} = \frac{\la p_k}{\la p_k + \LSTarg},
%\end{align}%
%%
%which is exactly the LST of an exponential random variable with parameter $\la p_k$.
We will show \eqref{eqnMP:Poisson_process_splitting_still_need_to_prove} in \cref{remQTF:proof_Poisson_process_probabilistic_splitting} of \cref{ch:queues_and_transforms}, since the proof requires Laplace-Stieltjes transforms. For now, we can conclude that $\{ N_k(t) \}_{t \ge 0}$ is an independent Poisson process with parameter $\la p_k$. The second property thus says that under \textit{probabilistic splitting}, a Poisson process remains a Poisson process.

%%%%%%%%%%%%%%%%%%%%%%%%%%%%%%%%%%%%%%%%%%%%%%%%%%%%%%%
%%%%%%%%%%%%%%%%%%%%%%%%%%%%%%%%%%%%%%%%%%%%%%%%%%%%%%%
%%%%%%%%%%%%%%%%%%%%% NEW SECTION %%%%%%%%%%%%%%%%%%%%%
%%%%%%%%%%%%%%%%%%%%%%%%%%%%%%%%%%%%%%%%%%%%%%%%%%%%%%%
%%%%%%%%%%%%%%%%%%%%%%%%%%%%%%%%%%%%%%%%%%%%%%%%%%%%%%%

\section{General Markov processes}%
\label{secMP:Markov_processes}%

A continuous-time stochastic process $\{ X(t) \}_{t \ge 0}$ is called a \textit{Markov process} if it takes values in a countably infinite or finite state space $\statespace$ and satisfies the \textit{Markov property}. Let $\mathfrak{F}(s)$ be the history of the process until and including time $s$ at which $X(s) = x$. A process satisfies the Markov property if for all $x,y \in \statespace$ and $t,s \ge 0$,
\begin{equation}%
\Prob{X(t + s) = y \mid \mathfrak{F}(s)} = \Prob{X(t + s) = y \mid X(s) = x}. \label{eqnMP:Markov_property}
\end{equation}%
The Markov property states that the future state at time $t + s$ does not depend on the past states, but only on the current state at time $s$. The right-hand side of \eqref{eqnMP:Markov_property} is called a \textit{transition function}. A Markov process for which $\Prob{X(t + s) = y \mid X(s) = x}$ does not depend on $s$ is said to have \textit{stationary} transition functions, an assumption we shall make throughout this book.

A Markov process is a jump process. This means that the Markov process  stays in a state $x \in \statespace$ a certain amount of time and after that time, makes a transition to a different state $y \in \statespace, ~ y \neq x$. A transition alters the state of the process in a sudden and radical way, hence the name jump process. Both the time spent in a state and the possible transitions (and the probabilities with which they occur) are allowed to depend on the state. Because of the jumps a sample path of a Markov process is assumed continuous from the right and having a limit from the left.\endnote{For an elaborate and technical description of Markov jump processes we refer the interested reader to Asmussen \cite[Chapter~II]{Asmussen2008_Applied_probability_and_queues}.}

Assume that the Markov process is currently in state $x \in \statespace$. The event that causes a transition from state $x$ to $y$, where $x \neq y$, takes places after an exponential amount of time with parameter $\q{x,y} \ge 0$ (where 0 indicates a transition is not possible). Let us call this the \textit{transition time} from $x$ to $y$ and refer to $\q{x,y}$ as the \textit{transition rate} from $x$ to $y$. Clearly, the time spent in state $x$ until a transition occurs (the \textit{sojourn time} $H_x$) is the minimum over all end states $y$ of the transition times from $x$ to $y$. According to the properties of exponential random variables, we have that a Markov process obeys two basic rules (see also \cref{figMP:sample_path}):
\begin{enumerate}[label = (\roman*)]%
\item The sojourn time $H_x$ in state $x$ is exponentially distributed with parameter $\q{x} \defi \sum_{y \neq x} \q{x,y}$;
\item After the sojourn time the Markov process jumps from state $x$ to $y \neq x$ with probability $\q{x,y}/\q{x}$.
\end{enumerate}%
We require the sojourn time in each state $x$ to be positive. This means that we restrict our analysis to Markov processes that satisfy $0 \le \q{x} < \infty$ for all $x \in \statespace$. A state $x$ is called \textit{absorbing} if $\q{x} = 0$. An absorbing state is a state from which the Markov process cannot leave: once it reaches this state, it will stay there indefinitely.

\begin{figure}
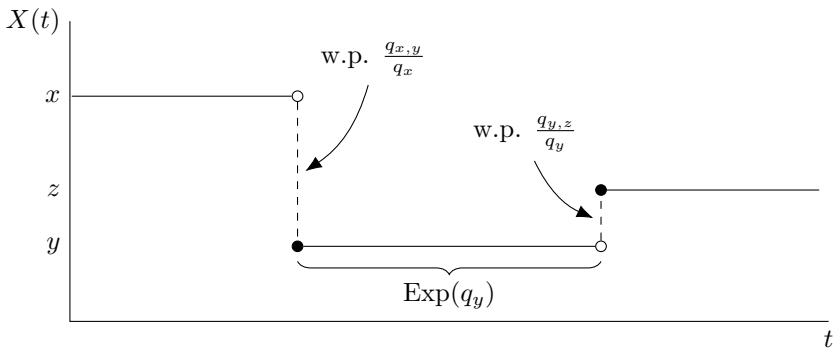
%
\centering%
\includestandalone{Chapters/MP/TikZFiles/sample_path_explanation}%
\caption{Sample path of a Markov process. With probability is abbreviated to w.p.}%
\label{figMP:sample_path}%
\end{figure}%

\begin{example}[Browsing the internet]\label{exMP:internet_browsing}%
The internet browsing behavior of a user is tracked for the purpose of ranking websites. From numerous previous observations, the behavior of this particular user has become apparent. The user starts his session at some website. He stays at each website $x$ an exponential amount of time with mean $1/\q{x}$. After that time, the user proceeds to a different website that he picks from a set of websites $n_x$, which is allowed to depend on the current website since the user might want to visit a website on a related topic. The set $n_x$ can also contain an element representing the end of the browsing session.

The browsing behavior is a Markov process $\{ X(t) \}_{t \ge 0}$. The states of the Markov process are the websites and state 0 is the end of the browsing session (with $\q{0} = 0$). Then, $X(t)$ is the website the user is on at time $t$. The sojourn time $H_x$ in state $x$ is exponentially distributed with rate $\q{x}$ and after the sojourn time, the Markov process transitions to a different website $y \in n_x$ with some probability $b_{x,y}$ that can be determined from previous browsing behavior. Notice that we require $\sum_{y \in n_x} b_{x,y} = 1$.
\end{example}%

We next discuss \textit{regularity},\endnote{A rigorous treatment of regularity and explosion can be found in Resnick \cite[Section~5.2]{Resnick1992_Adventures} and Norris \cite[Section~2.7]{Norris1997_Markov}.} the property that states that the Markov process makes a finite number of transitions in a finite length of time with probability 1. If a Markov process is not regular, we call it an explosive process. Explosive processes have the property that within a finite amount of time, an infinite number of transitions can occur. We assume throughout the book that all Markov processes are regular. This will always hold for Markov processes with a finite state space, or when $\sup_{x \in \statespace} \q{x} < \infty$. If $\sup_{x \in \statespace} \q{x} = \infty$, the Markov process might still be regular, however. Unless mentioned otherwise, we will henceforth assume that $\sup_{x \in \statespace} \q{x} < \infty$, since handling the other case introduces technical hurdles that detract from the book's main storyline.

\begin{example}[An explosive process]\label{exMP:explosive_process}%
Consider a Markov process labeled $\{ X(t) \}_{t \ge 0}$ with initial state $X(0) = 1$, transition rates $\q{x,x + 1} = x^2, ~ x \ge 1$ and all other transition rates are zero. Clearly, the Markov process proceeds through the numbered states $1,2,\ldots$ and resides in each state $x$ an exponential amount of time with mean $1/x^2$. Let $T_\infty$ be the time until the process reaches state $\infty$. Then $\E{T_\infty} = \sum_{x \ge 1} 1/x^2 = \pi^2/6$ and $\Prob{T_\infty < \infty} = 1$, showing that with probability 1 infinitely many transitions occur in a finite interval.
\end{example}%

The transition rates are the basic ingredients of the Markov process. We therefore introduce the transition rate matrix $Q$ of dimension $|\statespace| \times |\statespace|$, with as the elements the transition rates. A row of $Q$ indicates the state the process is currently in and the column is the target state. The diagonal elements are different in the sense that in row $x$, the element on the diagonal is $-\q{x}$. This makes the row sums equal to zero. For $\statespace = \Nat_0$, the transition rate matrix is then
\begin{equation}%
Q = \begin{bmatrix}%
- \q{0} & \q{0,1} & \q{0,2} & \q{0,3} & \cdots \\
\q{1,0} & - \q{1} & \q{1,2} & \q{1,3} & \cdots \\
\q{2,0} & \q{2,1} & - \q{2} & \q{2,3} & \cdots \\
\q{3,0} & \q{3,1} & \q{3,2} & - \q{3} & \cdots \\
\vdots  & \vdots  & \vdots  & \vdots  & \ddots
\end{bmatrix}.%
\end{equation}%
In general one needs to order the state space to be able to characterize the transition rate matrix $Q$.

The transition rate matrix $Q$ can be visualized in a \textit{transition rate diagram}. This diagram depicts the states of the Markov process, the possible transitions between the states and the rates at which they occur. The transition rate diagram can be incredibly helpful in recognizing the underlying structure of the transition rates of the Markov process. See \cref{figMP:transition_rate_diagram} for an example. Both the description of a Markov process in terms of the transition rate matrix $Q$ and the transition rate diagram are sufficient to fully characterize the Markov process.

\begin{figure}
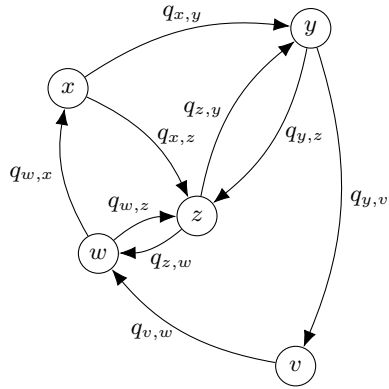
%
\centering%
\includestandalone{Chapters/MP/TikZFiles/transition_rate_diagram}%
\caption{An example of a transition rate diagram.}%
\label{figMP:transition_rate_diagram}%
\end{figure}%

A useful concept for Markov processes are \textit{stopping times}. Namely, a Markov process before a stopping time is independent of the Markov process after the stopping time. This property is called the \textit{strong Markov property}.\endnote{Establishing the strong Markov property in continuous time is actually more technical than we make it seem in \cref{secMP:Markov_processes}. We have opted for the current description to bring across the main idea of the strong Markov property without going into too much technical detail. For a technical and precise treatment of the strong Markov property in continuous time we refer the reader to Norris \cite[Section~6.5]{Norris1997_Markov}.} It essentially applies the Markov property at a `random time' with a clear definition of when this time stops. We briefly describe these two concepts.

A random variable $T$ is called a stopping time if its realization depends only on the history of the Markov process $\mathfrak{F}(T)$ until and including time $T$ and whose value is the time at which the process meets a `stopping rule'. A good example of a stopping time is the time $T$ it takes for the Markov process to go from state $x$ to state $y$. If asked to stop at time $T$, you only need to observe when the Markov process enters state $y$ for the first time. An example that is not a stopping time is the time $T$ at which the Markov process exits the set of states $\set{A}$ for the last time. Clearly, the future states of the Markov process are needed to determine if it actually was the last time the process exits the set of states $\set{A}$. So, in general, a last exit time is not a stopping time. The Markov process evaluated at a stopping time $T$, conditional on $\{ T < \infty \}$, starts anew from the state $X(T)$. More precisely, a Markov process $\{ X(t) \}_{t \ge 0}$ satisfies the strong Markov property, which says that for each stopping time $T$, conditioned on the event $\{ T < \infty \}$, we have that for each $t \ge 0$, $X(T + t)$ only depends on $X(T)$. As an example, say we have the time $T$ it takes to go from state $x$ to state $y$. Conditioning on the event that $T$ is finite,
\begin{align}%
\Prob{X(T + t) = z \mid X(0) = x} &= \Prob{X(T + t) = z \mid X(T) = y} \notag \\
&= \Prob{X(t) = z \mid X(0) = y},
\end{align}%
since $T$ is a stopping time and the Markov process has stationary transition functions.

%%%%%%%%%%%%%%%%%%%%%%%%%%%%%%%%%%%%%%%%%%%%%%%%%%%%%%%
%%%%%%%%%%%%%%%%%%%%%%%%%%%%%%%%%%%%%%%%%%%%%%%%%%%%%%%
%%%%%%%%%%%%%%%%%%%%% NEW SECTION %%%%%%%%%%%%%%%%%%%%%
%%%%%%%%%%%%%%%%%%%%%%%%%%%%%%%%%%%%%%%%%%%%%%%%%%%%%%%
%%%%%%%%%%%%%%%%%%%%%%%%%%%%%%%%%%%%%%%%%%%%%%%%%%%%%%%

\section{Classification of states}%
\label{secMP:classification_of_states}%

We now discuss the notions of irreducibility, recurrence and transience.
A state $y$ is said to be \textit{accessible} from state $x$ if there is a positive probability of ever reaching state $y$ given that the process starts in state $x$. If $x$ is also accessible from $y$, the states $x$ and $y$ are said to \textit{communicate} and is denoted by $x \leftrightarrow y$. Furthermore, if $x \leftrightarrow y$ and $y \leftrightarrow z$, then also $x \leftrightarrow z$.\endnote{The accessibility and communication properties are treated in many classical books, usually for Markov chains in discrete-time, see Feller \cite[Section~XV.6]{Feller1968_Probability_theory}, Karlin and Taylor \cite[Section~2.4]{Karlin1975_First_course_stochastic_processes} or Ross \cite[Section~4.2]{Ross1996_Stochastic_processes}. The concept is identical for Markov processes, however.}

States that communicate are said to be in the same \textit{equivalence class}, or class for short. This indicates that the state space of a Markov process can be partitioned into separate classes. If all states communicate with each other, then there is only one class and the Markov process is called \textit{irreducible}. Alternatively, a Markov process is irreducible if
\begin{equation}%
\Prob{X(t) = y \mid X(0) = x} > 0,
\end{equation}%
for all states $x,y \in \statespace$ and $t > 0$, indicating that there is a positive probability that the process is in state $y$ at time $t$ given it started in $x$. So, state $y$ is accessible from state $x$. Irreducibility is a direct property of the transition rate matrix $Q$, but a transition rate diagram such as the one in \cref{figMP:transition_rate_diagram}, can also be helpful in assessing if a Markov process is irreducible.

A state is said to be \textit{recurrent}\endnote{Liggett \cite[Section~2.6.2]{Liggett2010_CTMC} and Norris \cite[Section~3.4]{Norris1997_Markov} both have an excellent treatment of recurrence and transience for Markov processes.} if the Markov process returns to that state infinitely many times with probability 1. Otherwise the state is called \textit{transient}. So, a recurrent state is always visited a next time, but there exists a time at which a transient state is visited for the last time.

\begin{definition}[Recurrence and transience]\label{defMP:recurrence}%
State $x$ is recurrent if
\begin{equation}%
\Prob{x}{X(t) = x \text{ for arbitrary large $t$}} = 1
\end{equation}%
and transient otherwise, where the notation $\E{x}{f(X)}$ and $\Prob{x}{f(X)}$ are the expectation and probability of a functional of a process $\{ X(t) \}_{t \ge 0}$ given $X(0) = x$.
\end{definition}%

\begin{example}\label{exMP:four_states}%
\begin{figure}
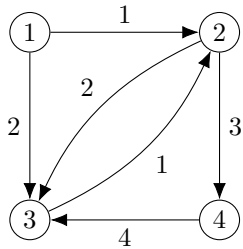
%
\centering%
\includestandalone{Chapters/MP/TikZFiles/example_Markov_process}%
\caption{A Markov process with four states.}%
\label{figMP:four_states_transition_rate_diagram}%
\end{figure}%

Consider a Markov process with state space $\statespace = \{ 1,2,3,4 \}$. Transitions can occur between these states. If the Markov process is in state 1, it transitions to state 2 after an exponentially distributed time with rate 1 and to state 3 with rate 2. From 2 the process transitions to state 3 with rate 2 and to state 4 with rate 3. With rate 1 the process transitions from state 3 to state 2 and with rate 4 from state 4 to state 3. This explanation is rather verbose and can be condensed by simply giving the transition rate matrix
\begin{equation}%
Q = \begin{bmatrix}%
-3 & 1  & 2  & 0 \\
 0 & -5 & 2  & 3 \\
 0 & 1  & -1 & 0 \\
 0 & 0  & 4  & -4
\end{bmatrix}.%
\end{equation}%
Another concise description of the behavior of the Markov process is the transition rate diagram shown in \cref{figMP:four_states_transition_rate_diagram}. The states are represented by the labeled circles and the transitions with their rates are described using the arrows. If we inspect the transition rate diagram in \cref{figMP:four_states_transition_rate_diagram}, we see that the process cannot return to state 1 since there are no transitions leading to this state and therefore state 1 is transient. The communicating class $\{ 2,3,4 \}$ is recurrent.
\end{example}%

Recurrence has a number of equivalent definitions. To that end we need the total time spent by the Markov process in a state and hitting-time random variables. Define
\begin{equation}%
T_y \defi \int_0^\infty \ind{X(t) = y} \, \dinf t
\end{equation}%
to be the total time spent in state $y$. Taking the expectation with respect to the initial state $x$ yields
\begin{equation}%
\E{x}{T_y} = \Efxd{x}{\int_0^\infty \ind{X(t) = y} \, \dinf t} = \int_0^\infty \Prob{x}{X(t) = y} \, \dinf t.
\end{equation}%
Introduce the hitting time random variables
\begin{equation}%
\htt{x,y} \defi \inf \{ t > 0 : \lim_{s \uparrow t} X(s) \neq X(t) = y \mid X(0) = x\}, \label{eqnMP:hitting_time_random_variable}
\end{equation}%
with the convention $\inf \emptyset = \infty$. Note that $\htt{x,x}$ is the time it takes the process to return to state $x$. The hitting time $\htt{x,y}$ is a stopping time.

Now, the first equivalent condition of recurrence is then as follows. A state $x$ is recurrent if
\begin{equation}%
\E{x}{T_x} = \infty
\end{equation}%
and transient otherwise. A second equivalent condition of recurrence is
\begin{equation}%
\Prob{ \htt{x,x} < \infty } = 1,
\end{equation}%
which indicates by the strong Markov property that the process returns to state $x$ unboundedly many times with probability 1. State $x$ is transient if $\Prob{ \htt{x,x} < \infty } < 1$. Both conditions can be understood from the viewpoint of the number of visits to a state. If state $x$ is transient and $X(0) = x$, then the number of visits to state $x$ follows a geometric distribution with failure probability $\Prob{x}{\text{the process returns to state $x$}} = \Prob{ \htt{x,x} < \infty }$. Given that the process starts in state $x$, the expected number of visits to state $x$ is
\begin{equation}%
\frac{1}{1 - \Prob{ \htt{x,x} < \infty }}.
\end{equation}%
Each time the process visits state $x$ it stays there, in expectation, $1/\q{x}$ time. Naturally, the expected total time spent in a transient state $x$ is finite, since $\Prob{ \htt{x,x} < \infty } < 1$. To be more precise, the total time $T_x$ spent in a transient state $x$ conditional on $X(0) = x$ is a sum of i.i.d.\@ exponential random variables with distribution $H_x$ where the number of terms in the summation is an independent geometric random variable with failure probability $\Prob{ \htt{x,x} < \infty }$. We conclude that for a transient state $x$, $T_x$ conditional on $X(0) = x$ is an exponential random variable. \makeExercise A recurrent state $x$ is visited infinitely often and thus the expected total time spent in state $x$ is infinite.

Recurrent states can be classified even further. A state is said to be \textit{positive recurrent} if the expected return time is finite and \textit{null recurrent} if the expected return time is infinite. Recurrent states in a Markov process with a finite number of states are always positive recurrent.

\begin{definition}\label{defMP:positive_null_recurrent}%
A recurrent state $x$ is positive recurrent iff $\E{\htt{x,x}} < \infty$ and null recurrent otherwise.
\end{definition}%

In \cref{exMP:four_states} it is easy to see that the expected returns times for the recurrent states 2, 3 and 4 are finite, which makes them \textit{positive} recurrent.

Recurrence and transience are class properties. If any one state in an equivalence class is (positive or null) recurrent, than all states in that class are (positive or null) recurrent. Equivalently, a transient state implies that all states in that class are transient. There are many more ways to characterize recurrence and transience, but the current level of discussion is sufficient for this book.

Most Markov processes are one of three types: (i) all states communicate and are recurrent; (ii) some transient classes and some recurrent classes and the Markov process eventually enters one of the recurrent classes; or (iii) all states in the countably infinite state space of the Markov process are transient. In this book we focus mostly on type-(i) Markov processes.

%%%%%%%%%%%%%%%%%%%%%%%%%%%%%%%%%%%%%%%%%%%%%%%%%%%%%%%
%%%%%%%%%%%%%%%%%%%%%%%%%%%%%%%%%%%%%%%%%%%%%%%%%%%%%%%
%%%%%%%%%%%%%%%%%%%%% NEW SECTION %%%%%%%%%%%%%%%%%%%%%
%%%%%%%%%%%%%%%%%%%%%%%%%%%%%%%%%%%%%%%%%%%%%%%%%%%%%%%
%%%%%%%%%%%%%%%%%%%%%%%%%%%%%%%%%%%%%%%%%%%%%%%%%%%%%%%

\section{Time-dependent behavior}%
\label{secMP:time-dependent_behavior}%

By the law of total probability the probability mass function of $X(t)$ satisfies
\begin{equation}%
\Prob{X(t) = y} = \sum_{x \in \statespace} \Prob{X(t) = y \mid X(0) = x} \Prob{X(0) = x} \label{eqnMP:pmf_X(t)}
\end{equation}%
and is thus uniquely characterized by the transition functions
\begin{equation}%
p_{x,y}(t) \defi \Prob{X(t) = y \mid X(0) = x}
\end{equation}%
and the matrix of transition functions
\begin{equation}%
P(t) \defi [p_{x,y}(t)]_{x,y \in \statespace}.
\end{equation}%
The transition functions satisfy the Chapman-Kolmogorov equations, which state that each transition can be split at any intermediate time. The proof of this theorem can be found in many textbooks, e.g., \cite[Sections~4.2 and 5.4]{Ross1996_Stochastic_processes}.

\begin{theorem}[Chapman-Kolmogorov equations]\label{thmMP:Chapman-Kolmogorov}%
For all $t,s \ge 0$,
\begin{equation}%
P(t + s) = P(t) P(s),
\end{equation}%
or, in scalar form with $x,y \in \statespace$,
\begin{equation}%
p_{x,y}(t + s) = \sum_{z \in \statespace} p_{x,z}(t) p_{z,y}(s).
\end{equation}%
\end{theorem}%

The transition functions satisfy two sets of differential equations called the Kolmogorov backward and forward equations. The Kolmogorov backward equations are derived from the Chapman-Kolmogorov equations by conditioning on the state at time $h$. We have
\begin{equation}%
p_{x,y}(t + h) = \sum_{z \in \statespace} p_{x,z}(h) p_{z,y}(t)
\end{equation}%
and subtracting $p_{x,y}(t)$ from both sides, dividing by $h$ and taking $h \downarrow 0$ yields
\begin{align}%
&\lim_{h \downarrow 0} \frac{p_{x,y}(t + h) - p_{x,y}(t)}{h} \notag \\
&= \lim_{h \downarrow 0} \sum_{z \neq x} \frac{p_{x,z}(h)}{h} p_{z,y}(t) - \lim_{h \downarrow 0} \frac{1 - p_{x,x}(h)}{h} p_{x,y}(t). \label{eqnMP:Kolmogorov_backward_still_need_to_interchange_limit_and_summation}
\end{align}%
By definition, the left-hand side of \eqref{eqnMP:Kolmogorov_backward_still_need_to_interchange_limit_and_summation} equals $\frac{\dinf}{\dinf t} p_{x,y}(t)$. On the right-hand side we have two limits. Since the transition functions satisfy
\begin{equation}%
p_{x,x}(t) = 1 - \q{x} t + \SmallO(t), \qquad p_{x,y}(t) = \q{x,y} t + \SmallO(t), \quad y \neq x, \label{eqnMP:behavior_transition_functions_small_time_interval}
\end{equation}%
which is proved in, e.g., \cite[Lemma~5.4.1]{Ross1996_Stochastic_processes}, these limits can be simplified. In case the state space $\statespace$ is finite, the interchange of the limit and the finite summation is clearly allowed. If the state space is countably infinite, the interchange is also allowed (see, e.g., \cite[Theorem~5.4.3]{Ross1996_Stochastic_processes}) and we obtain the Kolmogorov backward equations.

\begin{theorem}[Kolmogorov backward equations]\label{thmMP:Kolmogorov_backward}%
For all $t \ge 0$,
\begin{equation}%
\frac{\dinf}{\dinf t} P(t) = Q P(t),
\end{equation}%
or, in scalar form with $x,y \in \statespace$,
\begin{equation}%
\frac{\dinf}{\dinf t} p_{x,y}(t) = \sum_{z \neq x} \q{x,z} p_{z,y}(t) - \q{x} p_{x,y}(t)
\end{equation}%
and initial conditions $p_{x,x}(0) = 1$ and $p_{x,y}(0) = 0, ~ y \neq x$.
\end{theorem}%

The Kolmogorov forward equations are obtained by conditioning on the state at time $t$. We have
\begin{equation}%
p_{x,y}(t + h) = \sum_{z \in \statespace} p_{x,z}(t) p_{z,y}(h)
\end{equation}%
and subtracting $p_{x,y}(t)$ from both sides, dividing by $h$ and letting $h \downarrow 0$ gives
\begin{align}%
&\lim_{h \downarrow 0} \frac{p_{x,y}(t + h) - p_{x,y}(t)}{h} \notag \\
&= \lim_{h \downarrow 0} \sum_{z \in \statespace} p_{x,z}(t) \frac{p_{z,y}(h)}{h} - p_{x,y}(t) \lim_{h \downarrow 0} \frac{1 - p_{y,y}(h)}{h} .
\end{align}%
In this case, the interchange of limit and summation is not always allowed. For example, an explosive process does not satisfy the Kolmogorov forward equations as they are formulated in the following theorem, where we did interchange the limit and the summation. However, these equations do hold for all birth--and--death processes (see \cref{ch:birth--and--death_processes}) and for all Markov processes with a finite state space $\statespace$. We state the following theorem without a proof, since all results follow from the definition of the derivative and \eqref{eqnMP:behavior_transition_functions_small_time_interval}, assuming that the limit and summation can be interchanged.

\begin{theorem}[Kolmogorov forward equations]\label{thmMP:Kolmogorov_forward}%
For all $t \ge 0$ and under suitable regularity conditions\endnote{The conditions for which the Kolmogorov forward equations hold are formulated either fairly restrictive, as in the case in Gikhman and Skorohod \cite[Lemma~3, p.~204]{Gikhman1975_Theory_stochastic_processes}, or difficult to check, as in Liggett \cite[Theorem~2.39]{Liggett2010_CTMC}. The former conditions are $\sup_{x \in \statespace} \q{x} < \infty$ and the latter conditions are $\sum_{z \in \statespace} p_{x,z}(t) \q{z} < \infty$. If the former conditions hold, than they imply the latter conditions:
\begin{equation}%
\sum_{z \in \statespace} p_{x,z}(t) \q{z} \le \sup_{x \in \statespace} \q{x} \sum_{z \in \statespace} p_{x,z}(t) = \sup_{x \in \statespace} \q{x} \cdot 1 < \infty.
\end{equation}%
}
\begin{equation}%
\frac{\dinf}{\dinf t} P(t) = P(t) Q,
\end{equation}%
or, in scalar form with $x,y \in \statespace$,
\begin{equation}%
\frac{\dinf}{\dinf t} p_{x,y}(t) = \sum_{z \neq x} p_{x,z}(t) \q{z,y} - p_{x,y}(t) \q{y}
\end{equation}%
with initial conditions $p_{x,x}(0) = 1$ and $p_{x,y} = 0, y \neq x$.
\end{theorem}%

The Kolmogorov forward equations are often easier to solve, since these equations express the transition functions in terms of a common initial state $X(0) = x$. Still, obtaining explicit expressions for the transition functions is notoriously difficult, and can generally be done only for toy models or Markov processes with a pronounced structure in the transition rate matrix. Let us consider such an example.

\begin{example}[Star gazing]\label{exMP:light_source_time-dependent}%
We study the visibility of a star. Statistical analysis shows that the light source is visible for an exponential amount of time with parameter $\mu$ and remains invisible for an exponential amount of time with parameter $\la$. We denote by $X(t)$ if the star is visible or not at time $t$. Under this description, $X(t)$ has a finite state space $\statespace = \{ 0,1 \}$ and transition rates $\q{0,1} = \la$ and $\q{1,0} = \mu$ (all other rates are 0). The transition functions with $X(0) = 0$ satisfy the Kolmogorov forward equations, so
\begin{align}%
\frac{\dinf}{\dinf t} p_{0,0}(t) &= \mu p_{0,1}(t) - \la p_{0,0}(t), \\
\frac{\dinf}{\dinf t} p_{0,1}(t) &= \la p_{0,0}(t) - \mu p_{0,1}(t).
\end{align}%
We can solve this system of equations by noting that at time $t$ the star has to be either visible or invisible, or, symbolically, $p_{0,0}(t) + p_{0,1}(t) = 1$. From the first equation we derive
\begin{equation}%
\frac{\dinf}{\dinf t} p_{0,0}(t) = \mu - (\la + \mu) p_{0,0}(t),
\end{equation}%
which can be turned into a separable equation by
\begin{equation}%
\frac{\dinf}{\dinf t} \bigl( \euler^{(\la + \mu)t} p_{0,0}(t) \bigr) = p_{0,0}(t) \frac{\dinf}{\dinf t} \euler^{(\la + \mu) t}  + \euler^{(\la + \mu) t} \frac{\dinf}{\dinf t} p_{0,0}(t) = \mu \euler^{(\la + \mu)t}.
\end{equation}%
Integrating the above equation and using the initial condition $p_{0,0}(0) = 1$ finally gives
\begin{equation}%
p_{0,0}(t) = \frac{\mu}{\la + \mu} + \frac{\la}{\la + \mu} \euler^{-(\la + \mu)t}.
\end{equation}%
The transition functions for all initial states are derived in an identical way. The result is
\begin{equation}%
P(t) = \begin{bmatrix}%
p_{0,0}(t) & p_{0,1}(t) \\
p_{1,0}(t) & p_{1,1}(t)
\end{bmatrix}%
= \frac{1}{\la + \mu} \begin{bmatrix}%
\mu + \la \euler^{-(\la + \mu)t} & \la - \la \euler^{-(\la + \mu)t} \\
\mu - \mu \euler^{-(\la + \mu)t} & \la + \mu \euler^{-(\la + \mu)t}
\end{bmatrix},\label{eqnMP:star_gazing_example_transition_functions}%
\end{equation}%
which has a nice symmetrical form.
\end{example}%

\begin{remark}[Numerical analysis for finite state spaces]%
For finite state spaces, a solution to the Kolmogorov backward and forward equations always exists and it is given by\endnote{There are some technical hurdles that one needs to overcome to be able to formulate $P(t) = \euler^{Qt}$ as the solution to both the Kolmogorov backward and forward equations. In particular, the derivative needs to be interchanged with an infinite series. Norris \cite[Section~2.1]{Norris1997_Markov} shows that this indeed can be done by showing that $\euler^{Qt}$ has an infinite radius of convergence. For an infinite state space, the same author derives in \cite[Section~2.8]{Norris1997_Markov} that the minimal non-negative solution of the backward equation is also the minimal non-negative solution of the forward equation. However, the solution $P(t)$ is not characterized.}
\begin{equation}%
P(t) = \euler^{Qt}, \label{eqnMP:solution_P(t)}
\end{equation}%
where the matrix exponential is defined as
\begin{equation}%
\euler^{Qt} \defi \sum_{n \ge 0} \frac{(Qt)^n}{n!}
\end{equation}%
with $(Qt)^0 = \I$ and therefore $P(0) = \I$. Indeed, \eqref{eqnMP:solution_P(t)} is a solution to both the Kolmogorov backward equation
\begin{align}%
\frac{\dinf}{\dinf t} P(t) &= Q + t Q^2 + \frac{t^2}{2!} Q^3 + \frac{t^3}{3!} Q^4 + \cdots \notag \\
&= Q \Bigl( \I + t Q + \frac{t^2}{2!} Q^2 + \frac{t^3}{3!} Q^3 + \cdots \Bigr) = Q P(t)
\end{align}%
and the forward equation
\begin{align}%
\frac{\dinf}{\dinf t} P(t) &= Q + t Q^2 + \frac{t^2}{2!} Q^3 + \frac{t^3}{3!} Q^4 + \cdots \notag \\
&= \Bigl( \I + t Q + \frac{t^2}{2!} Q^2 + \frac{t^3}{3!} Q^3 + \cdots \Bigr) Q = P(t) Q.
\end{align}%
Computing the matrix exponential is difficult, especially since the matrix $Q$ has both negative and positive elements and subtractions can cause loss of significant digits. Since the state space $\statespace$ is finite, one can truncate the series to a finite sum to obtain a numerical approximation of $P(t)$.
\end{remark}%

In \cref{secMP:Markov_processes} we have seen that the transition rate matrix $Q$ is the primary ingredient for constructing a Markov process. Coming to the end of this section, we have shown that the matrix $Q$ governs the time-dependent behavior of the Markov process as well. In the next section we show that $Q$ again plays an important role in determining probabilities of interest when $t \to \infty$ and the Markov process reaches an equilibrium.

%%%%%%%%%%%%%%%%%%%%%%%%%%%%%%%%%%%%%%%%%%%%%%%%%%%%%%%
%%%%%%%%%%%%%%%%%%%%%%%%%%%%%%%%%%%%%%%%%%%%%%%%%%%%%%%
%%%%%%%%%%%%%%%%%%%%% NEW SECTION %%%%%%%%%%%%%%%%%%%%%
%%%%%%%%%%%%%%%%%%%%%%%%%%%%%%%%%%%%%%%%%%%%%%%%%%%%%%%
%%%%%%%%%%%%%%%%%%%%%%%%%%%%%%%%%%%%%%%%%%%%%%%%%%%%%%%

\section{Equilibrium behavior}%
\label{secMP:equilibrium_behavior}%

With time, Markov processes that are irreducible and positive recurrent converge to an equilibrium. This means that the probability distribution of $X(t)$ (which depends on $X(0)$) tends to some other probability distribution that does not depend on $X(0)$ as $t$ tends to infinity.

An interesting object to study is the long-term fraction of time that the Markov process occupies a state $y \in \statespace$ given some initial state $x \in \statespace$, which is given by
\begin{equation}%
\lim_{t \to \infty} \frac{1}{t} \int_0^t \ind{X(s) = y \mid X(0) = x} \, \dinf s, \quad x,y \in \statespace.
\end{equation}%
It seems likely, and is indeed true, that if the Markov process is irreducible and positive recurrent, then this long-term fraction of time does not depend on the initial state $x$. If we label
\begin{equation}%
p(x) = \lim_{t \to \infty} \frac{1}{t} \int_0^t \ind{X(s) = x \mid X(0) = y} \, \dinf s, \quad x,y \in \statespace, \label{eqnMP:occupation_distribution}
\end{equation}%
then it is easy to see that $p(x) > 0$ for each $x \in \statespace$ by positive recurrence and $\sum_{x \in \statespace} p(x) = 1$ since we are talking about fractions of time. The distribution $p(x), ~ x \in \statespace$ in \eqref{eqnMP:occupation_distribution} is called the \textit{occupancy distribution}. Now, one can prove that the occupancy distribution is uniquely given by
\begin{equation}%
p(x) = \frac{1}{\q{x} \E{\htt{x,x}}} > 0, \quad x \in \statespace. \label{eqnMP:occupation_distribution_explicit_expression}
\end{equation}%
The proof of this statement uses a renewal-reward process, but we will only give an intuitive explanation. Due to the strong Markov property, we can just look at paths (or cycles) of the Markov process that start and end at state $x$. These cycles occur infinitely often because the Markov process is positive recurrent. The expected time of such a cycle is $\E{\htt{x,x}} < \infty$. Within each cycle, the expected time spent in state $x$ is $1/\q{x}$. Dividing these two quantities as in \eqref{eqnMP:occupation_distribution_explicit_expression} exactly gives the fraction of time spent in state $x$ in the long run.

\begin{example}[Occupancy distribution in a complete digraph]%
Consider a Markov process with $N + 1$ states, where from each state every other state is reachable in one transition. The transition rate diagram of this Markov process constitutes a complete digraph; every state is connected to every state. We furthermore make the simplifying assumptions that the sojourn time in each state is exponentially distribution with mean 1 and the probability of making a transition to a particular state is $1/N$.

The described Markov process is irreducible and positive recurrent since the number of states is finite. It is moreover symmetric and so we already know that the occupancy distribution $p(x) = 1/(N + 1)$. We verify this by deriving the expected return times and using \eqref{eqnMP:occupation_distribution_explicit_expression}. Fix the initial state as 1 and abbreviate $R_x = \E{\htt{x,1}}$. By a one-step analysis we derive
\begin{align}%
R_1 &= 1 + \frac{1}{N} \sum_{y \neq 1} R_y, \label{eqnMP:digraph_htt(1,1)}\\
R_x &= 1 + \frac{1}{N} \sum_{y \neq 1,x} R_y, \quad x \neq 1. \label{eqnMP:digraph_htt(x,1)}
\end{align}%
Add $R_x/N$ to both sides of \eqref{eqnMP:digraph_htt(x,1)} to get
\begin{equation}%
R_x \bigl( 1 + \frac{1}{N} \bigr) = 1 + \frac{1}{N} \sum_{y \neq 1,x} R_y + \frac{1}{N} R_x = 1 + \frac{1}{N} \sum_{y \neq 1} R_y = R_1.
\end{equation}%
Now, sum over all $x \neq 1$ to obtain
\begin{equation}%
\bigl( 1 + \frac{1}{N} \bigr) \sum_{x \neq 1} R_x = N R_1 \quad \Rightarrow \quad \sum_{x \neq 1} R_x = \frac{N^2}{N + 1} R_1. \label{eqnMP:digraph_sum_in_terms_of_htt(1,1)}
\end{equation}%
Substituting \eqref{eqnMP:digraph_sum_in_terms_of_htt(1,1)} into \eqref{eqnMP:digraph_htt(1,1)} gives
\begin{equation}%
R_1 = 1 + \frac{1}{N} \frac{N^2}{N + 1} R_1 \quad \Rightarrow \quad R_1 = N + 1
\end{equation}%
and so $p(1) = 1/(\q{1} R_1) = 1/(N + 1)$. Since we fixed an arbitrary state and the Markov process is symmetric, all expected return times are $N + 1$ and the occupancy distribution follows.
\end{example}%

So far, we derived that an irreducible and positive recurrent Markov process has a unique occupancy distribution expressed in terms of the expected sojourn times and expected return times. The expected return times are usually difficult to determine. We wish to have an easier way of computing the occupancy distribution. To that end, we introduce two concepts and relate these to the occupancy distribution.

\begin{definition}\label{defMP:stationary_distribution}%
A probability distribution $p(x), ~ x \in \statespace$ with $\sum_{x \in \statespace} p(x) = 1$ is said to be a \textit{stationary distribution} for the Markov process if it satisfies
\begin{equation}%
p(y) = \sum_{x \in \statespace} p(x) p_{x,y}(t), \quad y \in \statespace, ~ t \ge 0. \label{eqnMP:stationary_distribution}
\end{equation}%
\end{definition}%

In light of \eqref{eqnMP:pmf_X(t)}, the above definition should be interpreted as follows: if the initial state is distributed according to a stationary distribution $\pb$, then the distribution of $X(t)$ is independent of $t$ and equal to the stationary distribution $\pb$. Moreover, in that case, $\{ X(t) \}_{t \ge 0}$ is called a \textit{stationary process}.

A more natural and intuitive distribution is the limiting distribution.

\begin{definition}\label{defMP:limiting_distribution}%
A probability distribution $p(x), ~ x \in \statespace$ with $\sum_{x \in \statespace} p(x) = 1$ is said to be a \textit{limiting distribution} for the Markov process if it satisfies
\begin{equation}%
\lim_{t \to \infty} p_{x,y}(t) = p(y), \quad x,y \in \statespace,
\end{equation}%
when the limits exist.
\end{definition}%

Taking expectations on both sides of \eqref{eqnMP:occupation_distribution} shows that the occupancy distribution can be expressed in terms of transition functions:\endnote{In this case, taking expectations of \eqref{eqnMP:occupation_distribution} requires some work. It requires the use of the dominated convergence theorem and Tonelli's theorem, but we do not show it here.}
\begin{equation}%
\lim_{t \to \infty} \frac{1}{t} \int_0^t p_{x,y}(s) \, \dinf s = p(y), \quad x,y \in \statespace.
\end{equation}%
So, the existence of a limiting distribution implies the existence of an occupancy distribution. More importantly, the three distributions mentioned in this section are equivalent. We present this fact here without proof, see \cite[Sections~3.5 and 3.6]{Norris1997_Markov} for an elaborate discussion and the proof.

\begin{theorem}%
An irreducible and positive recurrent Markov process has a unique occupancy distribution, a unique stationary distribution and a unique limiting distribution and all three distributions are identical.
\end{theorem}%

To calculate the occupancy distribution, we require the expected return times and to calculate the stationary and limiting distributions we require the transition functions. In most cases, this is prohibitively difficult. Thankfully, we can work with another distribution that is the unique solution to a system of linear equations called the balance equations.

\begin{theorem}\label{thmMP:equilibrium_equations_and_distribution}%
An irreducible and positive recurrent Markov process has a probability distribution $\pb = [ p(x) ]_{x \in \statespace}$ with $\pb \oneb = 1$ which is the unique solution of the \textup{balance equations}
\begin{equation}%
\pb Q = \zerob,
\end{equation}%
or, in scalar form,
\begin{equation}%
p(y) \q{y} = \sum_{x \neq y} p(x) \q{x,y}, \quad y \in \statespace.
\end{equation}%
This distribution is called the \textup{equilibrium distribution} and is equal to the occupancy, stationary and limiting distribution.
\end{theorem}%

Solving the balance equations proves to be very useful since it also ensures positive recurrence of the Markov process. The following theorem is a continuous-time version of Foster's theorem \cite[Theorem~1]{Foster1953_Ergodicity_condition}.

\begin{theorem}\label{thmMP:Foster}%
If there exists a non-zero solution of the balance equations and this solution is absolutely convergent, then the Markov process is positive recurrent and the solution can be normalized to obtain the equilibrium distribution.
\end{theorem}%

\begin{example}[Star gazing]\label{exMP:light_source_equilibrium}%
We consider again the star of \cref{exMP:light_source_time-dependent}. Recall the transition functions in \eqref{eqnMP:star_gazing_example_transition_functions}. The Markov process is irreducible and positive recurrent. We derive that the occupancy, stationary, limiting and equilibrium distribution are identical. The occupancy distribution is given by
\begin{align}%
p(0) &= \lim_{t \to \infty} \frac{1}{t} \int_0^t p_{0,0}(s) \, \dinf s \notag \\
&= \frac{1}{\la + \mu} \lim_{t \to \infty} \frac{1}{t} \Bigl( \mu t + \la \frac{1 - \euler^{-(\la + \mu)t}}{\la + \mu} \Bigr) = \frac{\mu}{\la + \mu}
\end{align}%
and similarly for $p(1)$ to obtain $p(1) = \la/(\la + \mu)$. Let us verify that the occupancy distribution is a stationary distribution. We have
\begin{align}%
\pb P(t) &= \begin{bmatrix} p(0) & p(1) \end{bmatrix} \begin{bmatrix}%
p_{0,0}(t) & p_{0,1}(t) \\
p_{1,0}(t) & p_{1,1}(t) \end{bmatrix} \notag \\
&= \begin{bmatrix} p_{1,0}(t) + p(0)(p_{0,0}(t) - p_{1,0}(t)) & p_{0,1}(t) + p(1)(p_{1,1}(t) - p_{0,1}(t)) \end{bmatrix} \notag \\
&= \begin{bmatrix} p_{1,0}(t) + p(0) \euler^{-(\la + \mu)t} & p_{0,1}(t) + p(1) \euler^{-(\la + \mu)t} \end{bmatrix} \notag \\
&= \begin{bmatrix} p(0) & p(1) \end{bmatrix} = \pb,
\end{align}%
where we used $p(0) + p(1) = 1$. The limiting distribution is found by taking the limit $t \to \infty$ for the transition functions:
\begin{equation}%
\lim_{t \to \infty} P(t) = \lim_{t \to \infty} \begin{bmatrix}%
p_{0,0}(t) & p_{0,1}(t) \\
p_{1,0}(t) & p_{1,1}(t) \end{bmatrix}%
= \begin{bmatrix}%
p(0) & p(1) \\
p(0) & p(1) \end{bmatrix}%
= \frac{1}{\la + \mu} \begin{bmatrix}%
\mu & \la \\
\mu & \la \end{bmatrix}.%
\end{equation}%
Finally, the balance equations read
\begin{align}%
p(0) \la = p(1) \mu, \\
p(1) \mu = p(0) \la,
\end{align}%
which is a dependent system of linear equations, as is required. Using $p(0) + p(1) = 1$ we also obtain $p(0) = \mu/(\la + \mu)$ and $p(1) = \la/(\la + \mu)$. So, for this simple two-state example the four probability distributions indeed agree, in line with \cref{thmMP:equilibrium_equations_and_distribution}.
\end{example}%

One can think of the balance equations as the result of taking $t \to \infty$ in the Kolmogorov forward equations of \cref{thmMP:Kolmogorov_forward}. Intuitively, an irreducible and positive recurrent Markov process reaches an equilibrium in which the transition functions do not change anymore and we heuristically argue that $\frac{\dinf}{\dinf t} p_{x,y}(t) \to 0$ for $t \to \infty$. Since an irreducible and positive recurrent Markov process has a limiting distribution, we have $\lim_{t \to \infty} p_{x,z}(t) = p(z)$, and the interchange of the limit and infinite summation is allowed by the regularity conditions that were assumed in \cref{thmMP:Kolmogorov_forward}.

\begin{figure}
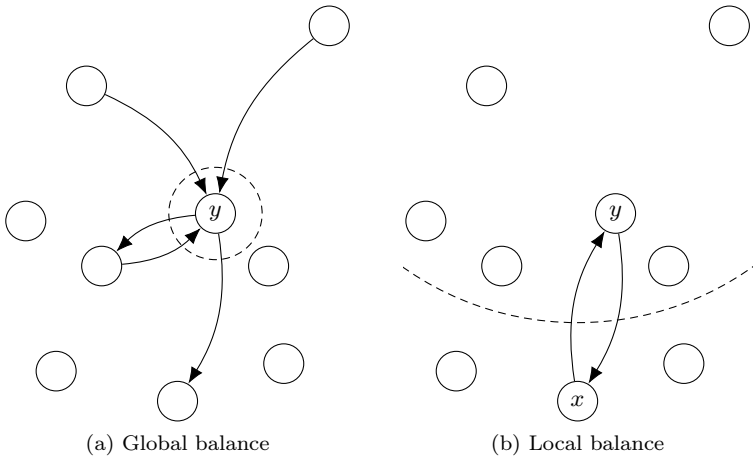
%
\centering%
\subfloat[Global balance]{%
\includestandalone{Chapters/MP/TikZFiles/global_balance}%
}%
\qquad
\subfloat[Local balance]{%
\includestandalone{Chapters/MP/TikZFiles/local_balance}%
}%
\caption{Two types of balance equations, where $\set{A}$ is the set of states inside the dashed circle.}%
\label{figMP:global_local_balance}%
\end{figure}%

A possibly more intuitive and natural interpretation of the balance equations is the following. If the Markov process is in an equilibrium, we require that the rate at which the process leaves a (set of) state(s) is equal to the rate at which the process enters the (set of) state(s). If this would not be the case, the Markov process is not in an equilibrium. Let us consider a countable set $\set{A} \subset \statespace$. Now, given that the Markov process is in state $y \in \set{A}$, the Markov process transitions to states outside $\set{A}$ with rate $\sum_{x \in \set{A}^c} \q{y,x}$. The probability that in equilibrium the Markov process is in state $y$ is given by the equilibrium distribution and is therefore equal to $p(y)$. Similarly, one derives the rate at which the Markov process transitions to states inside $\set{A}$ from a state $x \in \set{A}^c$. Balancing the two produces
\begin{equation}%
\sum_{y \in \set{A}} \sum_{x \in \set{A}^c} p(y) \q{y,x} = \sum_{x \in \set{A}^c} \sum_{y \in \set{A}} p(x) \q{x,y}.
\end{equation}%
The balance equations $\pb Q = \zerob$ follow from the above formula by taking $\set{A} = \{ y \}$. The set of equations $\pb Q = \zerob$ is also called the \textit{global} balance equations, see also \cref{figMP:global_local_balance}(a). Sometimes the set $\set{A}$ can be chosen in a way such that
\begin{equation}%
p(y) \q{y,x} = p(x) \q{x,y},
\end{equation}%
for all $x,y \in \statespace$. These equations are called the \textit{local} balance equations, see \cref{figMP:global_local_balance}(b). Local balance equations are ideal to work with. These equations make it far easier to determine the equilibrium probabilities since it allows one to express each equilibrium probability $p(y)$ in a specific other equilibrium probability, say $(0)$, and $p(0)$ follows from the normalization condition. Local balance equations do not exist in general, but they do exist for Markov processes with a specific type of structure in the transition rate matrix $Q$, such as the birth--and--death processes that we study in \cref{ch:birth--and--death_processes}, and for processes that are time-reversible. The topic of time-reversibility and its implications is studied in \cref{ch:reversible_networks}.

Choosing the set $\set{A}$ in a smart way and then invoking the balance principle is something that requires intuition, which can be trained through seeing and analyzing a variety of different Markov processes. This will be one of the goals of this book.

\begin{example}[Star topology]\label{exMP:star_topology}%
Consider a Markov process on the state space $\statespace = \Nat_0$. State 0 is central: from state 0 the process transitions to state $x$ with rate $\la^x$, but from state $x$ the process can only transition to state 0 with rate $\mu^x$, see \cref{figMP:star_topology}. Since we want all $\q{x}$ to be finite, we require $\la < 1$, otherwise the process leaves state 0 instantaneously. This gives $\q{0} = \sum_{x \ge 1} \la^x = \la/(1 - \la)$.

\begin{figure}%
\centering%
\includestandalone{Chapters/MP/TikZFiles/star_topology}%
\caption{The star topology of \protect\cref{exMP:star_topology}.}
\label{figMP:star_topology}%
\end{figure}%

The Markov process is irreducible and recurrent. It remains to see if the states are null recurrent or positive recurrent. The process transitions from state 0 to state $x$ with probability $(1 - \la)\la^{x - 1}$. If $\mu > 1$ ($\mu < 1$) the process resides in expectation a longer time at the states with a low (high) index.

We know that if an equilibrium distribution exists, the Markov process is positive recurrent, see \cref{thmMP:Foster}. We therefore investigate if a solution exists to the balance equations. This system of linear equations is given by
\begin{align}%
p(0) \frac{\la}{1 - \la} &= \sum_{x \ge 1} p(x) \mu^x, \label{eqnMP:star_topology_balance_x=0} \\
p(x) \mu^x &= p(0) \la^x, \quad x \ge 1. \label{eqnMP:star_topology_balance_x>0}
\end{align}%
Summing over all $x \ge 1$ on both sides of \eqref{eqnMP:star_topology_balance_x>0} produces \eqref{eqnMP:star_topology_balance_x=0} and the system of equations is dependent. Armed with the relation $p(x) = p(0) ( \la/\mu )^x$ and the normalization condition the equilibrium distribution can be obtained, if it exists. The normalization condition reads
\begin{equation}%
1 = \sum_{x \in \statespace} p(x) = p(0) \sum_{x \ge 0} \bigl( \frac{\la}{\mu} \bigr)^x.
\end{equation}%
We immediately see from the above equation that $\la < \mu$ is necessary for an equilibrium distribution to exist. Under this condition, the Markov process is indeed positive recurrent. Assuming $\la < \mu$, we find $p(0) = 1 - \la/\mu$ and all $p(x) = (1 - \la/\mu) ( \la/\mu )^x$. Equation~\eqref{eqnMP:occupation_distribution_explicit_expression} allows us to determine the expected return times from the occupancy distribution:
\begin{equation}%
\E{\htt{0,0}} = \frac{\mu(1 - \la)}{\la(\mu - \la)}, \quad \E{\htt{x,x}} = \frac{\mu}{\la^x (\mu - \la)}, \quad x \ge 1.
\end{equation}%
Since $\la < 1$, the expected return times grow unboundedly with increasing $x$, but for each state $x$ the expected return time is indeed finite.
\end{example}%

\begin{figure}
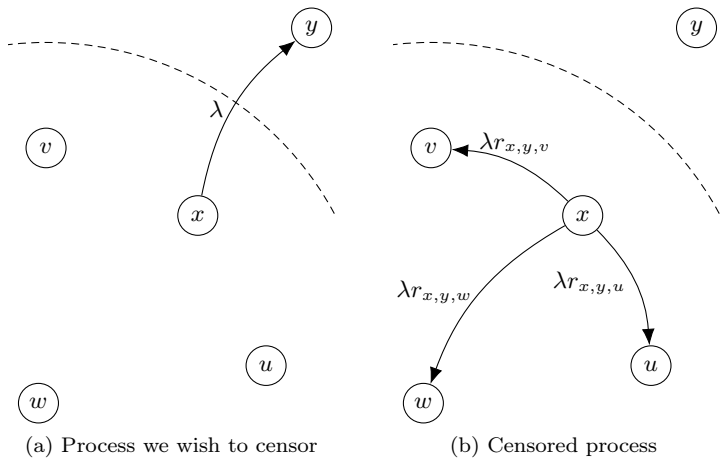
%
\centering%
\subfloat[Process we wish to censor]{%
\includestandalone{Chapters/MP/TikZFiles/want_to_censor_a_process}%
}%
\qquad
\subfloat[Censored process]{%
\includestandalone{Chapters/MP/TikZFiles/censored_process}%
}%
\caption{Example transition rate diagrams of the process that we want to censor to the set of states within the dashed circle and the censored process.}%
\label{figMP:censoring_a_process}%
\end{figure}%

A technique called \textit{censoring} can also be instrumental in calculating the equilibrium probabilities by allowing for the derivation of a different set of balance equations. We will use \cref{figMP:censoring_a_process} as a visual guide. Censoring a process to a set $\set{A}$ means that we only observe the process while it resides in this set. Practically it means that we can draw a new transition rate diagram: all transitions from states inside $\set{A}$ to states in $\set{A}^c$ are redirected to states within $\set{A}$. This redirection is done in a natural way, which we describe with an example. Say that state $x \in \set{A}$ has a single transition with rate $\la$ to a state $y$ outside $\set{A}$. With probability $r_{x,y,z}$ the process returns to $z \in \set{A}$ for the first time after leaving $\set{A}$ with a transition from state $x$ to state $y$. The transition with rate $\la$ is then split in many transitions according to these return probabilities: each new transition occurs with rate $\la r_{x,y,z}$ for all states $z \in \set{A}$. Notice that the potential transition from $x$ to $x$ does not need to be drawn, since it does not have any effect. Once the new transition rate diagram has been drawn, we can write down a different set of balance equations in the same way that we have described earlier.

\begin{example}[Censoring]\label{exMP:censoring}%
\begin{figure}
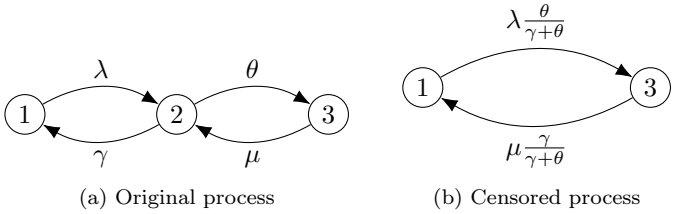
%
\centering%
\subfloat[Original process]{%
\includestandalone{Chapters/MP/TikZFiles/three-state_example_non-censored}%
}%
\qquad
\subfloat[Censored process]{%
\includestandalone{Chapters/MP/TikZFiles/three-state_example_censored}%
}%
\caption{Transition rate diagrams of the Markov process of \protect\cref{exMP:censoring}.}%
\label{figMP:example_censoring}%
\end{figure}%

Consider the Markov process with three states as shown in \cref{figMP:example_censoring}(a). The balance equations that we can derive from \cref{figMP:example_censoring}(a) are
\begin{align}%
p(1) \la &= p(2) \ga, \label{eqnMP:censoring_example_p(1)}\\
p(2) (\ga + \theta) &= p(1) \la + p(3) \mu, \\
p(3) \mu &= p(2) \theta.
\end{align}%
Let us censor the process to the set $\set{A} = \{ 1,3 \}$. So, we need to redirect all transitions that lead to state 2 to a state in $\set{A}$, since state 2 is outside this set. From state 1 the process can transition to state 2 with rate $\la$ and it returns to $\set{A}$ in state 1 with probability $\ga / (\ga + \theta)$ (but we do not need to draw that transition since it returns to the same state) and it returns to $\set{A}$ in state 3 with probability $\theta / (\ga + \theta)$. So from state 1 we need to draw a transition to state 3 with rate $\la \theta / (\ga + \theta)$. The same reasoning for state 3 leads to the transition rate diagram in \cref{figMP:example_censoring}(b). From \cref{figMP:example_censoring}(b) we derive another balance equation:
\begin{equation}%
p(1) \la \frac{\theta}{\ga + \theta} = p(3) \mu \frac{\ga}{\ga + \theta},
\end{equation}%
which gives us $p(1)= p(3) \mu \ga / (\theta \la)$ and therefore by \eqref{eqnMP:censoring_example_p(1)} shows that $p(2) = p(3) \mu / \theta$. The normalization condition $p(1) + p(2) + p(3) = 1$ then gives us that
\begin{equation}%
p(1) = \frac{\mu \ga}{\theta \la + \mu \ga + \mu \la}, \quad p(2) = \frac{\mu \la}{\theta \la + \mu \ga + \mu \la}, \quad p(3) = \frac{\theta \la}{\theta \la + \mu \ga + \mu \la}.
\end{equation}%
This simple example demonstrates how you can use censoring to derive new balance equations. This technique will prove useful when tackling more advanced processes.
\end{example}%

%The balance equations and the associated equilibrium distribution are the essential and main topic of study in this book. We study Markov processes motivated by diverse applications and seek to determine the equilibrium distribution as the solution of the balance equations. Instead of using the linear algebra approach to determine the solution $\pb Q = \zerob$ as a left eigenvector of $Q$ corresponding to the eigenvalue 0, we will exploit \textit{structural properties} of the $Q$ matrix associated to the Markov processes. These structural properties usually lead to explicit or exact expressions for the equilibrium probabilities.

%%%%%%%%%%%%%%%%%%%%%%%%%%%%%%%%%%%%%%%%%%%%%%%%%%%%%%%
%%%%%%%%%%%%%%%%%%%%%%%%%%%%%%%%%%%%%%%%%%%%%%%%%%%%%%%
%%%%%%%%%%%%%%%%%%%%% NEW SECTION %%%%%%%%%%%%%%%%%%%%%
%%%%%%%%%%%%%%%%%%%%%%%%%%%%%%%%%%%%%%%%%%%%%%%%%%%%%%%
%%%%%%%%%%%%%%%%%%%%%%%%%%%%%%%%%%%%%%%%%%%%%%%%%%%%%%%

\section{Manufacturing examples}%
\label{secMP:manufacturing_example}%
%

%A manufacturing network is a network of machines and processes from which a variety of products are produced. Different types of products are routed through this network in different ways, requiring dedicated and often different service at each machine. A wealth of modeling questions arise from these manufacturing networks ranging from deciding on a smart routing for the different types of products to optimizing the number of machines at a specific workstation to analyzing the production capacity of a single server. The purpose of this section is to examine a specific situation that appears in manufacturing networks using the knowledge on Markov processes that we obtained in this chapter. Specifically, we analyze the failure behavior of components in a single machine and predict the influence of the failing component on the production capacity of the machine.

We now apply our knowledge of Markov processes to some realistic manufacturing examples.

\begin{example}[A failing component]\label{exMP:single_failing_component}%
We assume that the quality of a component deteriorates through a total of $N$ phases where in each phase the component resides for an exponential amount of time with parameter $\theta$. After $N$ phases the component fails completely. So, the lifetime of a component has an $\Erl{N}{\theta}$ distribution. A lower quality component has a negative influence on the production capacity of the machine it resides in and therefore an operator visits the machine to check the quality of the component and replaces or repairs it whenever it is below perfect condition. The time between two visits of the operator is approximated by an exponential distribution with parameter $\ga$. Both replacing and repairing a component is assumed to take no time as it is short compared to the time between two successive visits of the operator.

Denote the quality of the component at time $t$ as $X(t)$. The process $\{ X(t) \}_{t \ge 0}$ is a Markov process with state space $\statespace \defi \{ 0,1,\ldots,N \}$ and transition rate matrix
\begin{equation}%
Q = \begin{bmatrix}%
-\theta & \theta \\
\ga     & -(\theta + \ga) & \theta \\
\ga     &                 & -(\theta + \ga) & \theta \\
\vdots  &                 &                 & \ddots & \ddots \\
\ga     &                 &                 &        & -(\theta + \ga) & \theta \\
\ga     &                 &                 &        &                 & -\ga \\
\end{bmatrix},%
\end{equation}%
where unspecified elements are zero. This Markov process is irreducible and positive recurrent because its state space is finite. So, the Markov process has an equilibrium distribution that we denote by $\pb \defi [ p(i) ]_{0 \le i \le N}$.

The global balance equations $\pb Q = \zerob$ read
\begin{align}%
p(0) \theta &= \ga \sum_{n = 1}^N p(n), \label{eqnMP:manufacturing_example_single_component_eq_eqs_p0}\\
p(i) (\theta + \ga) &= \theta p(i - 1), \quad i = 1,2,\ldots,N - 1, \\
p(N) \ga &= \theta p(N - 1).
\end{align}%
With the help of the normalization condition $\sum_{n = 0}^N p(n) = 1$ we are able to derive $p(0)$ from \eqref{eqnMP:manufacturing_example_single_component_eq_eqs_p0} as
\begin{equation}%
p(0) \theta = \ga (1 - p(0)) \quad \Rightarrow \quad p(0) = \frac{\ga}{\theta + \ga}.
\end{equation}%
The remaining balance equations are iterated to obtain
\begin{align}%
p(i) &= \bigl( \frac{\theta}{\theta + \ga} \bigr)^i p(0) = \bigl( \frac{\theta}{\theta + \ga} \bigr)^i \frac{\ga}{\theta + \ga}, \quad i = 0,1,\ldots,N - 1, \\
p(N) &= \frac{\theta}{\ga} \bigl( \frac{\theta}{\theta + \ga} \bigr)^{N - 1} \frac{\ga}{\theta + \ga}.
\end{align}%
From these equilibrium probabilities we see that if $\theta$ is large in comparison to $\gamma$, then $p(N)$ is large, which means that the component has deteriorated through all of its phases and has now completely failed. From these equilibrium expressions, an operator can, e.g., determine how often on average he needs to inspect the component so that with 99\% certainty it does not reach deterioration phases 5 and higher.
\end{example}%

\begin{example}[Multiple failing components]%
A machine naturally consists of multiple components that can be replaced or repaired if they are not in perfect condition. Let us consider a situation in which there are two components with each their own failure process. The behavior of the operator is the same as before, but now he replaces or repairs all components that are not in mint condition. Replacing or repairing both components at the same time makes the two failure processes dependent: if we know that one of the two components is in phase 0, then it is probable that both components were replaced or repaired recently, which shows that we also have information on the failure process of the other component. The time until failure for component 1 is $\Erl{N_1}{\theta_1}$ and $\Erl{N_2}{\theta_2}$ for component 2. Let $X_1(t)$ and $X_2(t)$ denote the quality level of component 1 and component 2 at time $t$ and let $X(t) \defi (X_1(t),X_2(t))$ describe the configuration of quality levels at time $t$. $\{ X(t) \}_{t \ge 0}$ describes an irreducible and positive recurrent Markov process with finite state space
\begin{equation}%
\statespace \defi \{ (i,j) \in \Nat_0^2 : 0 \le j \le N_1, ~ 0 \le j \le N_2 \}.
\end{equation}%
A transition rate diagram for a specific instance of $N_1$ and $N_2$ is shown in \cref{figMP:manufacturing_example_two_components_transition_rate_diagram}.

\begin{figure}
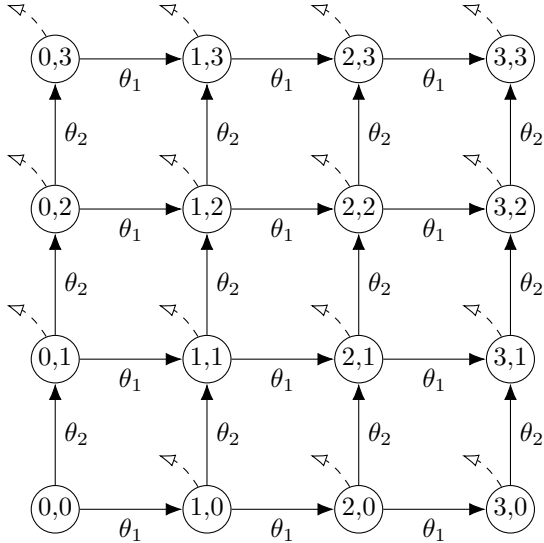
%
\centering%
\includestandalone{Chapters/MP/TikZFiles/manufacturing_example_two_components_transition_rate_diagram}%
\caption{Transition rate diagram of the Markov process associated with the failure processes of two components with $N_1 = N_2 = 3$. The dashed transitions are towards $(0,0)$ and occur with rate $\gamma$.}%
\label{figMP:manufacturing_example_two_components_transition_rate_diagram}%
\end{figure}%

Equilibrium probabilities of a two-dimensional Markov process are denoted as $p(i,j)$. For ease of exposition, we assume that $N_1$ and $N_2$ are large (so as to not worry about boundary behavior), but this approach works for any $N_1$ and $N_2$. The equilibrium probabilities $p(i,j)$ with $(i,j) \in \statespace$ can be solved in a recursive fashion. To start,
\begin{equation}%
p(0,0) (\theta_1 + \theta_2) = \ga \sum_{(i,j) \in \statespace \setminus \{ (0,0) \}} p(i,j),
\end{equation}%
which implies by the normalization condition $\sum_{(i,j) \in \statespace} p(i,j) = 1$ that
\begin{equation}%
p(0,0) = \frac{\ga}{\theta_1 + \theta_2 + \ga}.
\end{equation}%
Now that we have the equilibrium probability of state $(0,0)$ we can exploit the structure of the transition rate diagram in \cref{figMP:manufacturing_example_two_components_transition_rate_diagram}. In particular, we proceed along diagonals: the equilibrium probabilities of states $(1,0)$ and $(0,1)$ are expressed in terms of $(0,0)$ as
\begin{align}%
p(1,0)(\theta_1 + \theta_2 + \ga) &= p(0,0) \theta_1, \\
p(0,1)(\theta_1 + \theta_2 + \ga) &= p(0,0) \theta_2.
\end{align}%
Along the next diagonal, the equilibrium probabilities of states $(2,0)$, $(1,1)$ and $(0,2)$ are expressed in terms of the states on the previous diagonal:
\begin{align}%
p(2,0)(\theta_1 + \theta_2 + \ga) &= p(1,0) \theta_1, \\
p(1,1)(\theta_1 + \theta_2 + \ga) &= p(1,0) \theta_2 + p(0,1) \theta_1, \\
p(0,2)(\theta_1 + \theta_2 + \ga) &= p(0,1) \theta_2.
\end{align}%
Clearly, the equilibrium probabilities of the states on one diagonal can be expressed in terms of the equilibrium probabilities of the states on the preceding diagonal. When proceeding in this manner the complete equilibrium distribution can be obtained explicitly.

The recursive calculation of the equilibrium probabilities is not restricted to a system of two components, but can actually be applied to a system with an arbitrary number of components. For example, for a system with three components we can first determine $p(0,0,0)$ and from that find $p(1,0,0)$, $p(0,1,0)$ and $p(0,0,1)$ which leads to $p(1,1,0)$, $p(1,0,1)$ and $p(0,1,1)$ and ultimately gives $p(1,1,1)$. For the three-component example the sets of states are not diagonals but rather triangles.
\end{example}%

\begin{example}[Production capacity]%
We now study the impact of a single deteriorating component on the production capacity of a machine. Products arrive at the machine according to a Poisson process with rate $\la$ and are served in order of arrival. If the machine is already occupied, the products wait in a queue. The rate at which the machine serves a product depends on the quality level of the deteriorating product: if the component is in phase $n$ then the service rate is $\mu_n$ for $n = 0,1,\ldots,N$ with $\mu_1 \ge \mu_2 \ge \cdots \ge \mu_N = 0$. The operator behaves the same as before and replaces or repairs the component after an $\Exp{\gamma}$ amount of time and the lifetime of the component has an $\Erl{N}{\theta}$ distribution.

\begin{figure}
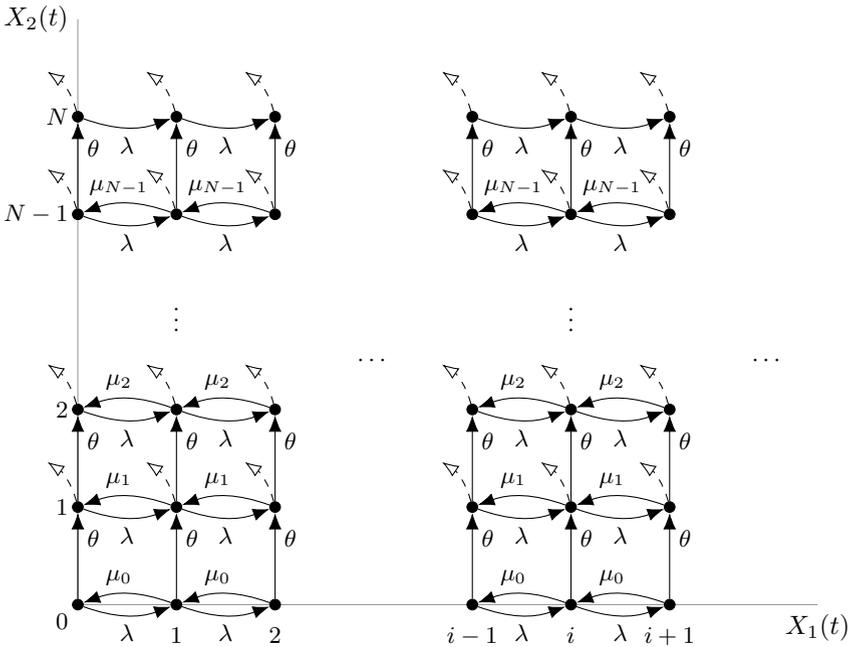
%
\centering%
\includestandalone{Chapters/MP/TikZFiles/manufacturing_example_machine_single_component_transition_rate_diagram}%
\caption{Transition rate diagram of the Markov process associated with the machine and the single component. The dashed transitions within $\lvl{i}$ are towards $(i,0)$ and occur with rate $\gamma$.}%
\label{figMP:manufacturing_example_machine_single_component_transition_rate_diagram}%
\end{figure}%

The Markov process associated with this system is two-dimensional: $X_1(t)$ denotes the number of products in the system at time $t$ and $X_2(t)$ is the quality level of the component at time $t$ and $X(t) \defi (X_1(t),X_2(t))$ is the state of the system at time $t$. The state space of this irreducible Markov process is
\begin{equation}%
\statespace \defi \{ (i,j) \in \Nat_0^2 : 0 \le j \le N \}
\end{equation}%
and the transition rate diagram is given in \cref{figMP:manufacturing_example_machine_single_component_transition_rate_diagram}. We note that the state space of this Markov process is countably infinite. A convenient way to partition the state space is by introducing \textit{levels}. A level is a vertically aligned set of states. Specifically, level $i$ is
\begin{equation}%
\lvl{i} \defi \{ (i,0),(i,1),\ldots,(i,N) \}, \quad i \ge 0,
\end{equation}%
so that
\begin{equation}%
\statespace = \lvl{0} \cup \lvl{1} \cup \lvl{2} \cup \cdots
\end{equation}%

At this point we will not determine the equilibrium distribution since it requires the theory of \cref{ch:quasi-birth--and--death_processes}. Rather, we derive the condition for which the process is positive recurrent. Intuitively, the states are positive recurrent if the Markov does not diverge `towards infinity', by which we mean that $X_1(t)$ does not grow without bound. For $X_1(t)$ to not grow without bound, we require that the average transition rate from level $i$ to level $i - 1$ (to the left) is greater than the average transition rate from level $i$ to level $i + 1$ (to the right). We can make this statement without specifying the exact level $i$ since the transition rate behavior is the same for any level greater than level 0. The average transition rate to the left is sum over $i$ of the the proportion of time spent in phase $i$ multiplied by $\mu_i$. We can similarly calculate the average transition rate to the right. Clearly the proportions sum to 1, so that the average transition rate to the right is exactly $\la$. So, to determine the average transition rate to the left we require to determine the fractions of time spent in each of the phases.

If we only observe transitions in the vertical direction, then we end up with exactly the Markov process of the failure process of a single component. Let us denote the equilibrium distribution of the phase process by $\pi(i), ~ i = 0,1,\ldots,N$ (we reserve $\pb$ for the equilibrium distribution of the Markov process). From our earlier analysis of the single component we know that
\begin{align}%
\pi(i) &= \bigl( \frac{\theta}{\theta + \ga} \bigr)^i \frac{\ga}{\theta + \ga}, \quad i = 0,1,\ldots,N - 1, \\
\pi(N) &= \frac{\theta}{\ga} \bigl( \frac{\theta}{\theta + \ga} \bigr)^{N - 1} \frac{\ga}{\theta + \ga}.
\end{align}%
The average transition rate to the right is therefore
\begin{equation}%
\sum_{i = 0}^N \pi(i) \mu_i = \frac{\ga}{\theta + \ga} \sum_{i = 0}^{N - 1} \bigl( \frac{\theta}{\theta + \ga} \bigr)^i \mu_i.
\end{equation}%
Under the \textit{stability condition}
\begin{equation}%
\la < \frac{\ga}{\theta + \ga} \sum_{i = 0}^{N - 1} \bigl( \frac{\theta}{\theta + \ga} \bigr)^i \mu_i,
\end{equation}%
$X_1(t)$ does not grow without bound and therefore the Markov process is positive recurrent. Compare this with the single-server system of \cref{secINT:balance_act}, where the stability condition is $\la < \mu$. This inequality also says that the average transition rate to the left is greater than the average transition rate to the right.
\end{example}%

%%%%%%%%%%%%%%%%%%%%%%%%%%%%%%%%%%%%%%%%%%%%%%%%%%%%%%%
%%%%%%%%%%%%%%%%%%%%%%%%%%%%%%%%%%%%%%%%%%%%%%%%%%%%%%%
%%%%%%%%%%%%%%%%%%%%% NEW SECTION %%%%%%%%%%%%%%%%%%%%%
%%%%%%%%%%%%%%%%%%%%%%%%%%%%%%%%%%%%%%%%%%%%%%%%%%%%%%%
%%%%%%%%%%%%%%%%%%%%%%%%%%%%%%%%%%%%%%%%%%%%%%%%%%%%%%%

\section{Takeaways}%
\label{secMP:takeaways}%

Markov processes can describe the evolution in time of many systems. This chapter discussed some of the prerequisites needed to define Markov processes in a mathematical way. For analyzing Markov processes, in order to quantify their behavior, we discussed three basic systems of equations: the Kolmogorov backward and forward equations, and the balance equations. The Kolmogorov equations capture the dynamics of the Markov process, over all time, while the balance equations describe long-term behavior. The focus of this book lies primarily with balance equations, although for all Markov processes discussed in the subsequent chapters one could state the Kolmogorov equations and study these as well. We do this in \cref{ch:birth--and--death_processes}, where we treat birth--and--death processes that have an exceptionally nice structure, leading to analytic solutions for both the balance and the Kolmogorov equations. In general, however, solving the Kolmogorov equations is more challenging than solving the balance equations. Solving the balance equations alone is challenging enough to write an entire book about.

From the theory side, much more can be said about the mathematics of Markov processes. While this chapter is restricted to the bare minimum needed to work with the mathematics in this book, there is a wealth of mathematical theory for Markov processes to be discovered. We encourage the interested reader to study for instance the books of Br{\'e}maud \cite{Bremaud1999_Markov_chains}, Chung \cite{Chung1967_Markov}, Ethier and Kurtz \cite{Ethier1986_Markov_processes}, Feller \cite{Feller1968_Probability_theory,Feller1971_Probability_theory_volume_II}, Jacod and Shiryaev \cite{Jacod2003_Limit_theorems_stochastic_processes}, Karlin and Taylor \cite{Karlin1975_First_course_stochastic_processes}, Liggett \cite{Liggett2010_CTMC}, Norris \cite{Norris1997_Markov}, Resnick \cite{Resnick1992_Adventures}, and Rogers and Williams \cite{Rogers1987_Diffusions_Markov_processes}.

From the practical side, much more can be said about the applications of Markov processes. Throughout the book we give examples of practical flavor, but these examples only serve the purpose of illustrating and practicing the mathematical methods. Those who want to learn more about modeling real-world applications as Markov processes can find many inspirational examples in books like Asmussen \cite{Asmussen2008_Applied_probability_and_queues}, Bruneel and Kim \cite{Bruneel1993_Discrete-time_communication}, Buzacott and Shantikumar \cite{Buzacott1993_Stochastic_models_of_manufacturing_systems},
Harchol-Balter \cite{Harchol2013_Computer_systems}, Kelly and Yudovina \cite{Kelly2014_Stochastic_Networks}, Kiss, Miller and Simon \cite{Kiss2017_Epidemics_on_networks} and Van Mieghem \cite{Mieghem2009_Performance_analysis_communications}.

If there is one thing we have learned from this chapter is that defining the Markov process in terms of its transition rate matrix or diagram is only the beginning. In order to study the Markov process, we are confronted with solving systems of equations. This challenge does not only require basic analysis or linear algebra, but should be combined with recognizing the structure hidden in the transition matrix. It is only then that the Markov process will reveal its beautiful properties, most notably the product-form solutions for the balance equations we encountered in \cref{exMP:star_topology,exMP:single_failing_component}. Many chapters now will follow, about classes of Markov processes, each with their specific structures and specific mathematical challenges. In all but a few cases we will be able to construct product-form solutions.

%%%%%%%%%%%%%%%%%%%%%%%%%%%%%%%%%%%%%%%%%%%%%%%%%%%%%%%
%%%%%%%%%%%%%%%%%%%%%%%%%%%%%%%%%%%%%%%%%%%%%%%%%%%%%%%
%%%%%%%%%%%%%%%%%%%%%%%% NOTES %%%%%%%%%%%%%%%%%%%%%%%%
%%%%%%%%%%%%%%%%%%%%%%%%%%%%%%%%%%%%%%%%%%%%%%%%%%%%%%%
%%%%%%%%%%%%%%%%%%%%%%%%%%%%%%%%%%%%%%%%%%%%%%%%%%%%%%%

%\theendnotes%
%\setcounter{endnote}{0}
\printendnotes%  

%\input{Chapters/Q/basic_queues}

% Checked points 1-7 and a-i
\chapter{Queues and transforms}%
\label{ch:queues_and_transforms}%
%

%A transform aggregates information about the probability mass or density function of a random variable. These transforms can be used to transform an infinite system of linear equations---such as the balance equations---to a single equation for the transform. So, transforms can reduce the complexity of the problem at hand, while also allowing for the calculation of the moments of the random variable and the underlying probability mass or density function. In this chapter we describe transforms for discrete and continuous random variables. For discrete random variables we introduce the \textit{probability generating function} and for continuous random variables the \textit{Laplace-Stieltjes transform}. We learn about both these transforms by studying canonical queueing systems.
This book is centered around analytic methods for finding the equilibrium distribution of a Markov process. So far, we have discussed methods targeted at directly solving the balancing equations, for instance by exploiting recursive structures or by substituting product forms. Transforms arise as an alternative method when an infinite system of linear equations---such as the balance equations---is converted into a single functional equation for the transform. The mathematical challenge then becomes to find the transform as the solution of the functional equation, which in some cases might prove the easiest or only method to tackle the problem. Once a transform is obtained, all information about the underlying distribution can be extracted from it. Taking derivatives of the transforms readily gives all moments. The underlying distribution can be retrieved by more advanced algorithms that invert the transform. This chapter covers the basics of transforms. For discrete random variables we introduce the \textit{probability generating function} and for continuous random variables the \textit{Laplace-Stieltjes transform}. We then learn how to work with these transforms by applying transform techniques to several classical queueing systems. We also introduce several numerical algorithms for transform inversion, which are largely based on Cauchy's formula. The transform technique and associated algorithms introduced in this chapter have a large scope of application, not only in later chapters in this book on more advanced Markov processes, but also in probability theory \cite{Feller1968_Probability_theory}, combinatorics \cite{Flajolet2009_Analytic_combinatorics} and digital signal processing \cite{Lyons2010_Understanding_digital_signal_processing}.

%%%%%%%%%%%%%%%%%%%%%%%%%%%%%%%%%%%%%%%%%%%%%%%%%%%%%%%
%%%%%%%%%%%%%%%%%%%%%%%%%%%%%%%%%%%%%%%%%%%%%%%%%%%%%%%
%%%%%%%%%%%%%%%%%%%%% NEW SECTION %%%%%%%%%%%%%%%%%%%%%
%%%%%%%%%%%%%%%%%%%%%%%%%%%%%%%%%%%%%%%%%%%%%%%%%%%%%%%
%%%%%%%%%%%%%%%%%%%%%%%%%%%%%%%%%%%%%%%%%%%%%%%%%%%%%%%

\section{Basic transforms}%
\label{secQTF:basic_transforms}%

We introduce basic properties of the probability generating function (PGF) for discrete random variables and the Laplace-Stieltjes transform (LST) for continuous random variables. We also give a first demonstration of how to use these transforms in the context of the basic single-server queue covered in \cref{ch:introduction}.

%%%%%%%%%%%%%%%%%%%%%%%%%%%%%%%%%%%%%%%%%%%%%%%%%%%%%%%
%%%%%%%%%%%%%%%%%%%%%%%%%%%%%%%%%%%%%%%%%%%%%%%%%%%%%%%
%%%%%%%%%%%%%%%%%%% NEW SUBSECTION %%%%%%%%%%%%%%%%%%%%
%%%%%%%%%%%%%%%%%%%%%%%%%%%%%%%%%%%%%%%%%%%%%%%%%%%%%%%
%%%%%%%%%%%%%%%%%%%%%%%%%%%%%%%%%%%%%%%%%%%%%%%%%%%%%%%

\subsection{Probability generating functions}%
\label{subsecQTF:PGF}%

The PGF of a non-negative random variable $X$ that takes values in the set $\{ 0,1,2,\ldots \}$ is defined as
\begin{equation}%
\PGF{X}{\PGFarg} \defi \E{\PGFarg^X} = \sum_{i \ge 0} p(i) \PGFarg^i,
\end{equation}%
where $p(i) = \Prob{X = i}$ is the probability mass function of $X$. A PGF of a random variable with a countably infinite support gives rise to an infinite series. Since we know that $p(\cdot)$ is a probability distribution and therefore $\sum_{i \ge 0} p(i) = 1$, we can conclude for $|\PGFarg| \le 1$ that
\begin{equation}%
|\PGF{X}{\PGFarg}| = \bigl| \sum_{i \ge 0} p(i) \PGFarg^i \bigr| \le \sum_{i \ge 0} p(i) |\PGFarg|^i \le \sum_{i \ge 0} p(i) = 1
\end{equation}%
and therefore the PGF converges for any $\PGFarg$ that is inside the closed unit disk. Depending on the form of $p(i)$ the PGF might converge for other values of $\PGFarg$ as well. Specifically, there exists an $r \ge 1$ such that the PGF converges absolutely for all $|\PGFarg| < r$ and diverges for all $|\PGFarg| > r$. This $r$ is called the \textit{radius of convergence} of the PGF.

\begin{example}[Geometric distribution]\label{exQTF:PGF_geometric_distribution}%
The probability mass function of the geometric distribution with failure probability $\rho$ is given by
\begin{equation}%
p(i) = (1 - \rho)\rho^i, \quad i \ge 0
\end{equation}%
and therefore its PGF is
\begin{equation}%
\PGF{\PGFarg} = \sum_{i \ge 0} p(i) \PGFarg^i = (1 - \rho) \sum_{i \ge 0} (\rho \PGFarg)^i = \frac{1 - \rho}{1 - \rho \PGFarg}. \label{eqnQTF:PGF_geometric_distribution}
\end{equation}%
The last equality only holds if $|\PGFarg| < 1/\rho$, which ensures that the series converges. Note that the radius of convergence $r$ is $1/\rho$.
\end{example}%

\begin{example}[Poisson distribution]\label{exQTF:PGF_Poisson_distribution}%
The probability mass function of the Poisson distribution with parameter $\la$ is given by
\begin{equation}%
p(i) = \frac{\la^i}{i!} \euler^{-\la}, \quad i \ge 0
\end{equation}%
and therefore its PGF is
\begin{equation}%
\PGF{\PGFarg} = \sum_{i \ge 0} p(i) \PGFarg^i = \euler^{-\la} \sum_{i \ge 0} \frac{(\la \PGFarg)^i}{i!} = \euler^{-\la(1 - \PGFarg)}. \label{eqnQTF:PGF_Poisson_distribution}
\end{equation}%
The last equality holds for all $\PGFarg \in \Complex$. So, the radius of convergence of the PGF of a random variable with a Poisson distribution with parameter $\la$ is infinite.
\end{example}%

\begin{remark}\label{remQTF:PGF_analytic_function}%
A PGF $\PGF{\cdot}$ is said to have radius of convergence $r$ when $\PGF{\PGFarg}$ is an \textit{analytic function} for all $\PGFarg \in \Complex$ satisfying $|\PGFarg| < r$ and has at least one singularity on the circle $|\PGFarg| = r$. A function that is analytic in a region $\setUncountable{A} \subset \Complex$ is a function that is complex differentiable at every $\PGFarg \in \setUncountable{A}$, or equivalently, if it has a convergent series expansion in an open disk around every $\PGFarg \in \setUncountable{A}$. For the mathematical definition of these terms we refer the reader to \cite{Marsden1998_Basic_complex_analysis}; we will only use the property that $\PGF{\PGFarg}$ is analytic for $|\PGFarg| < r$. Returning to \cref{exQTF:PGF_geometric_distribution}, we see that $\PGF{\PGFarg}$ in \eqref{eqnQTF:PGF_geometric_distribution} is an analytic function for all $\PGFarg \in \Complex$ satisfying $|\PGFarg| < 1/\rho$. This function has a pole (a simple singularity) at $\PGFarg = 1/\rho$. The PGF in \eqref{eqnQTF:PGF_Poisson_distribution} is called an \textit{entire function} because it is analytic for all $\PGFarg \in \Complex$.
\end{remark}%

The probability mass function can be retrieved from the PGF $\PGF{X}{\cdot}$ through $p(0) = \PGF{X}{0}$ and
\begin{equation}%
p(i) = \frac{1}{i!} \frac{\dinf^i}{\dinf \PGFarg^i} \PGF{X}{\PGFarg} \Big\vert_{\PGFarg = 0}, \quad i \ge 1.
\end{equation}%
All probabilities $\{ p(i) \}_{i \ge 0}$ thus follow by taking derivatives of the PGF at $\PGFarg = 0$. This observation leads to one of the most important properties of a PGF, which is that if $\PGF{X}{\PGFarg} = \PGF{Y}{\PGFarg}$, then $X \dequal Y$, and vice versa, if $X \dequal Y$, then $\PGF{X}{\PGFarg} = \PGF{Y}{\PGFarg}$. Moreover, since the derivatives are evaluated at $\PGFarg = 0$, we conclude that if two PGFs are equal on any real interval containing the value 0, then the underlying probability mass functions are equal. We should mention that taking derivatives can become computationally cumbersome, either because of the complexity of the symbolic expressions of the derivatives, or because of numerical inaccuracies, particularly for $p(i)$ with $i$ large. We therefore also present an alternative method for PGF inversion in \cref{secQTF:numerical_inversion} based on contour integrals.

One of the great advantages of using PGFs is that the moments of the random variables are easy to determine. For example,
\begin{equation}%
\frac{\dinf}{\dinf \PGFarg} \PGF{X}{\PGFarg} \Big\vert_{\PGFarg = 1} = \frac{\dinf}{\dinf \PGFarg} \sum_{i \ge 0} p(i) \PGFarg^i \Big\vert_{\PGFarg = 1} = \sum_{i \ge 0} p(i) \frac{\dinf}{\dinf \PGFarg} \PGFarg^i \Big\vert_{\PGFarg = 1} = \sum_{i \ge 0} i p(i) = \E{X},
\end{equation}%
where the interchange of derivative and summation is allowed because the series converges uniformly. More generally, the factorial moments are given by
\begin{equation}%
\E{X(X - 1) \cdots (X - k + 1)} = \frac{\dinf^k}{\dinf \PGFarg^k} \PGF{X}{\PGFarg} \Big\vert_{\PGFarg = 1}, \quad k \ge 1.
\end{equation}%
A PGF is also useful when considering sums of random variables. For example, if we set $Z \defi X + Y$ and $X$ and $Y$ are independent, then
\begin{equation}%
\PGF{Z}{\PGFarg} = \E{\PGFarg^Z} = \E{\PGFarg^{X + Y}} = \E{\PGFarg^X} \E{\PGFarg^Y} = \PGF{X}{\PGFarg} \PGF{Y}{\PGFarg}.
\end{equation}%
%
%and in case $Z$ is with probability $q$ equal to $X$ and with probability $1 - q$ equal to $Y$, then the PGF can be calculated by conditioning:
%%
%\begin{align}%
%\PGF{Z}{\PGFarg} &= \E{\PGFarg^Z} = \E{\PGFarg^Z \mid Z = X} \Prob{Z = X} + \E{\PGFarg^Z \mid Z = Y} \Prob{Z = Y} \notag \\
%&= \E{\PGFarg^X} q + \E{\PGFarg^Y} (1 - q) = q \PGF{X}{\PGFarg} + (1 - q) \PGF{Y}{\PGFarg}.
%\end{align}%
%%

%\myinline{Possibly write something about analytic functions, complex-differentiability, residues, Cauchy contour integral and geometric tail bounds?}

%%%%%%%%%%%%%%%%%%%%%%%%%%%%%%%%%%%%%%%%%%%%%%%%%%%%%%%
%%%%%%%%%%%%%%%%%%%%%%%%%%%%%%%%%%%%%%%%%%%%%%%%%%%%%%%
%%%%%%%%%%%%%%%%%%% NEW SUBSECTION %%%%%%%%%%%%%%%%%%%%
%%%%%%%%%%%%%%%%%%%%%%%%%%%%%%%%%%%%%%%%%%%%%%%%%%%%%%%
%%%%%%%%%%%%%%%%%%%%%%%%%%%%%%%%%%%%%%%%%%%%%%%%%%%%%%%

\subsection{Laplace-Stieltjes transforms}%
\label{subsecQTF:LST}%

The LST of a non-negative random variable $X$ is defined as
\begin{equation}%
\LST{X}{\LSTarg} \defi \E{\euler^{-\LSTarg X}} = \int_0^\infty \euler^{-\LSTarg t} \, \dinf F(t).
\end{equation}%
When the random variable $X$ has a density $f(\cdot)$, then the transform simplifies to
\begin{equation}%
\LST{X}{\LSTarg} = \int_0^\infty \euler^{-\LSTarg t} f(t) \, \dinf t.
\end{equation}%
The region of convergence of an LST is at least the complex numbers $\LSTarg$ that satisfy $\RealPart{\LSTarg} > 0$, but in most cases this region is larger. Notice that $|\LST{X}{\LSTarg}| \le 1$ for $\RealPart{\LSTarg} > 0$.

\begin{example}[Exponential distribution]\label{exQTF:LST_exponential_distribution}%
The exponential distribution with rate $\la$ has the probability density function $f(t) = \la \euler^{-\la t}$ for all $t \ge 0$. The LST of this distribution is therefore
\begin{equation}%
\LST{\LSTarg} = \int_0^\infty \euler^{-\LSTarg t} f(t) \, \dinf t = \la \int_0^\infty \euler^{-(\la + \LSTarg) t} \, \dinf t = \frac{\la}{\la + \LSTarg}.
\end{equation}%
The last integral is finite if $\RealPart{\LSTarg} > -\la$. So, the region of convergence of the LST associated with the exponential distribution with rate $\la$ is described by all $\LSTarg \in \Complex$ satisfying $\RealPart{\LSTarg} > -\la$.
\end{example}%

An LST uniquely determines the underlying distribution just as a PGF does: if $\LST{X}{\LSTarg} = \LST{Y}{\LSTarg}$, then $X \dequal Y$ and vice versa if $X \dequal Y$, then $\LST{X}{\LSTarg} = \LST{Y}{\LSTarg}$. In \cref{subsecQTF:numerical_inversion_univariate_LSTs} we show how to retrieve the probability distribution function $f(\cdot)$ using the Bromwich line integral or using an algorithm.

An LST satisfies many useful properties; some of the most important ones include
\begin{equation}%
\LST{X}{0} = 1, \quad \frac{\dinf}{\dinf \LSTarg} \LST{X}{\LSTarg} \big\vert_{\LSTarg = 0} = -\E{X}, \quad \frac{\dinf^k}{\dinf \LSTarg^k} \LST{X}{\LSTarg} \big\vert_{\LSTarg = 0} = (-1)^k \E{X^k}.
\end{equation}%
Furthermore, if $Z \defi X + Y$ and $X$ and $Y$ are independent, then
\begin{equation}%
\LST{Z}{\LSTarg} = \LST{X}{\LSTarg} \LST{Y}{\LSTarg}.
\end{equation}%
%
%and in case $Z$ is with probability $q$ equal to $X$ and with probability $1 - q$ equal to $Y$, then the LST is given by
%%
%\begin{equation}%
%\LST{Z}{\LSTarg} = q \LST{X}{\LSTarg} + (1 - q) \LST{Y}{\LSTarg}.
%\end{equation}%
%%

\begin{remark}[Probabilistic splitting of Poisson processes]\label{remQTF:proof_Poisson_process_probabilistic_splitting}%
Now that we have introduced LSTs, we are able to prove the second property of Poisson processes: under probabilistic splitting, a Poisson process remains a Poisson process. We are required to prove \eqref{eqnMP:Poisson_process_splitting_still_need_to_prove}.
Taking the LST on the right-hand side of \eqref{eqnMP:Poisson_process_splitting_still_need_to_prove} and conditioning on $K$,
\begin{align}%
\E{ \euler^{- \LSTarg \sum_{i = 1}^K X_i} } &= \sum_{j \ge 1} \E{ \euler^{- \LSTarg \sum_{i = 1}^K X_i} \mid K = j} \, \Prob{K = j} \notag \\
&= \sum_{j \ge 1} \E{ \euler^{- \LSTarg \sum_{i = 1}^j X_i}} \, \Prob{K = j} = \sum_{j \ge 1} \E{ \euler^{- \LSTarg X_1}}^j \, \Prob{K = j} \notag \\
&= \sum_{j \ge 1} \bigl( \frac{\la}{\la + \LSTarg} \bigr)^j (1 - p_k)^{j - 1} p_k = \frac{\la p_k }{\la + \LSTarg} \sum_{j \ge 0} \bigl( \frac{\la(1 - p_k)}{\la + \LSTarg} \bigr)^j \notag \\
&= \frac{\la p_k}{\la + \LSTarg} \frac{\la + \LSTarg}{\la + \LSTarg - \la(1 - p_k)} = \frac{\la p_k}{\la p_k + \LSTarg},
\end{align}%
which is exactly the LST of an exponential random variable with parameter $\la p_k$.
\end{remark}%

%%%%%%%%%%%%%%%%%%%%%%%%%%%%%%%%%%%%%%%%%%%%%%%%%%%%%%%
%%%%%%%%%%%%%%%%%%%%%%%%%%%%%%%%%%%%%%%%%%%%%%%%%%%%%%%
%%%%%%%%%%%%%%%%%%% NEW SUBSECTION %%%%%%%%%%%%%%%%%%%%
%%%%%%%%%%%%%%%%%%%%%%%%%%%%%%%%%%%%%%%%%%%%%%%%%%%%%%%
%%%%%%%%%%%%%%%%%%%%%%%%%%%%%%%%%%%%%%%%%%%%%%%%%%%%%%%

\subsection{Applying the transforms to a simple queue}%
\label{subsecQTF:simple_queue}%

In \cref{ch:introduction} we have introduced the simple queue where jobs arrive according to a Poisson process with rate $\la$ and are served by a single server with exponential rate $\mu$. This queueing system is denoted in Kendall's notation as the $M/M/1$ system. Here, $M$ stands for Markovian or memoryless (so exponentially distributed). In later sections we will also encounter the letter $G$, which stands for general. The order in which the letters appear in Kendall's notation matters: the first entry describes the distribution of the inter-arrival times, the second entry the distribution of the service times and the third entry the number of servers in the system.

In \cref{ch:introduction} we have demonstrated how to obtain the equilibrium distribution of the $M/M/1$ system in two ways. We now demonstrate a third way using transforms. Recall that the balance equations are given by
\begin{subequations}%
\label{eqnQTF:MM1_balance_equations}%
\begin{align}%
\la p(0) &= \mu p(1), \label{eqnQTF:MM1_balance_equations_i=0} \\
(\la + \mu) p(i) &= \la p(i - 1) + \mu p(i + 1), \quad i \ge 1. \label{eqnQTF:MM1_balance_equations_i>0}
\end{align}%
\end{subequations}%

We aim to find an expression for $\PGF{\PGFarg} \defi \sum_{i \ge 0} p(i) \PGFarg^i$ by manipulating the balance equations \eqref{eqnQTF:MM1_balance_equations}.

Multiply both sides of \eqref{eqnQTF:MM1_balance_equations_i>0} by $\PGFarg^i$ and sum on both sides over all $i \ge 1$ to obtain
\begin{equation}%
(\la + \mu) \sum_{i \ge 1} p(i) \PGFarg^i = \la \sum_{i \ge 1} p(i - 1) \PGFarg^i + \mu \sum_{i \ge 1} p(i + 1) \PGFarg^i. \label{eqnQTF:MM1_balance_equations_i>0_PGF}
\end{equation}%
By appropriately adding and subtracting terms on both sides of \eqref{eqnQTF:MM1_balance_equations_i>0_PGF} and multiplying by $\PGFarg$, we obtain
\begin{equation}%
(\la + \mu) \PGFarg ( \PGF{\PGFarg} - p(0) ) = \la \PGFarg^2 \PGF{\PGFarg} + \mu ( \PGF{\PGFarg} - p(1) \PGFarg - p(0) ).
\end{equation}%
Use \eqref{eqnQTF:MM1_balance_equations_i=0} to express $p(1)$ in terms of $p(0)$ and obtain the relation
\begin{equation}%
\bigl(1 - (1 + \rho) \PGFarg + \rho \PGFarg^2 \bigr) \PGF{\PGFarg} = (1 - \PGFarg) p(0).
\end{equation}%
Noticing that $1 - (1 + \rho) \PGFarg + \rho \PGFarg^2 = (1 - \PGFarg)(1 - \rho \PGFarg)$ gives
\begin{equation}%
\PGF{\PGFarg} = \frac{p(0)}{1 - \rho \PGFarg}.
\end{equation}%
Since $\PGF{1} = \sum_{i \ge 0} p(i) = 1$, we find that $p(0) = 1 - \rho$ and
\begin{equation}%
\PGF{\PGFarg} = \frac{1 - \rho}{1 - \rho \PGFarg}. \label{eqnQTF:MM1_PGF_explicit_expression}
\end{equation}%
From the geometric series $\sum_{i \ge 0} x^i = 1/(1 - x)$ if $|x| < 1$, we deduce that
\begin{equation}%
\PGF{\PGFarg} = (1 - \rho) \sum_{i \ge 0} (\rho \PGFarg)^i = \sum_{i \ge 0} (1 - \rho) \rho^i \PGFarg^i,
\end{equation}%
and hence $p(i) = (1 - \rho) \rho^i$. The generating function approach is a powerful approach that works well even if an explicit expression for $p(i)$ is difficult to obtain. In fact, if the expression of the PGF was not as nice as in \eqref{eqnQTF:MM1_PGF_explicit_expression}, then we could have stopped at that point and used algorithms that can numerically invert the PGF to calculate values for any $p(i)$, see \cref{secQTF:numerical_inversion}.

For queueing systems with Poisson arrivals, so for $M/\cdot/\cdot$ systems, the unusual property holds that arriving jobs find on average the same situation as an outside observer looking at the system at an arbitrary point in time. More precisely, the fraction of jobs finding on arrival the system in some state $i$ is exactly the same as the fraction of time the system is in state $i$. This is called the \textit{Poisson arrivals see time averages} (PASTA) property \cite{Wolff1982_PASTA}. This property is only true for Poisson arrivals, and can be explained intuitively by the fact that Poisson arrivals occur completely random in time. If we label the probability that an arriving job sees $i$ jobs in the system (excluding itself) as $a(i)$, then we conclude that $a(i) = p(i)$.

The PASTA property can be used to determine the distribution of how much time a job spends in the system, which is also called the \textit{sojourn time} $S$. With probability $a(i)$ an arriving job finds $i$ jobs in the system. Since the service times are exponentially distributed, we know that the sojourn time of the arriving jobs is the sum of $i + 1$ exponential phases, each with rate $\mu$. By conditioning on the number of jobs seen on arrival, we therefore find that
\begin{align}%
\LST{S}{\LSTarg} &= \sum_{i \ge 0} a(i) \bigl( \frac{\mu}{\mu + \LSTarg} \bigr)^{i + 1} = \frac{\mu(1 - \rho)}{\mu + \LSTarg} \sum_{i \ge 0} \bigl( \frac{\mu \rho}{\mu + \LSTarg} \bigr)^i \notag \\
&= \frac{\mu(1 - \rho)}{\mu + \LSTarg} \frac{1}{1 - \frac{\mu \rho}{\mu + \LSTarg}} = \frac{\mu(1 - \rho)}{\mu(1 - \rho) + \LSTarg}.
\end{align}%
Since we have learned earlier that an LST uniquely determines the distribution of a random variable, we conclude that the sojourn time is exponentially distributed with rate $\mu(1 - \rho)$.

%%%%%%%%%%%%%%%%%%%%%%%%%%%%%%%%%%%%%%%%%%%%%%%%%%%%%%%
%%%%%%%%%%%%%%%%%%%%%%%%%%%%%%%%%%%%%%%%%%%%%%%%%%%%%%%
%%%%%%%%%%%%%%%%%%%%% NEW SECTION %%%%%%%%%%%%%%%%%%%%%
%%%%%%%%%%%%%%%%%%%%%%%%%%%%%%%%%%%%%%%%%%%%%%%%%%%%%%%
%%%%%%%%%%%%%%%%%%%%%%%%%%%%%%%%%%%%%%%%%%%%%%%%%%%%%%%

\section{Single-server queue with general service times}%
\label{secQTF:MG1_queue}%

Consider a single-server queueing system where jobs arrive according to a Poisson process with rate $\la$---so with exponentially distributed inter-arrival times---and service times that are i.i.d.~copies of some random variable $B$. Assume that $B$ has a cumulative distribution function $F_B(\cdot)$ and a probability density function $f_B(\cdot)$. We require for stability that $\rho \defi \la \E{B} < 1$. This queueing system is denoted in Kendall's notation as the $M/G/1$ system.

%%%%%%%%%%%%%%%%%%%%%%%%%%%%%%%%%%%%%%%%%%%%%%%%%%%%%%%
%%%%%%%%%%%%%%%%%%%%%%%%%%%%%%%%%%%%%%%%%%%%%%%%%%%%%%%
%%%%%%%%%%%%%%%%%%% NEW SUBSECTION %%%%%%%%%%%%%%%%%%%%
%%%%%%%%%%%%%%%%%%%%%%%%%%%%%%%%%%%%%%%%%%%%%%%%%%%%%%%
%%%%%%%%%%%%%%%%%%%%%%%%%%%%%%%%%%%%%%%%%%%%%%%%%%%%%%%

\subsection{Departure distribution}%
\label{subsecQTF:MG1_departure_distribution}%

The state of the queueing system can be described by $(i,t)$ with $i$ the number of jobs in the system and $t$ the service time already received by the job in service. This state description is then two-dimensional with one discrete dimension and one continuous dimension. The continuous dimension makes the analysis prohibitively difficult, so we will look for another state description. If we observe the number of jobs in the system at the instant just after a job departs, then we know that $t = 0$, which essentially removes the second continuous dimension in the state description. In equilibrium, we denote by $d(i)$ the probability that a departing job leaves behind $i$ jobs. In other words, $d(i)$ is the fraction of departing jobs that leaves behind $i$ jobs.

From one departure instant to the next the number of jobs in the system reduces by one, but increases by the number of jobs that have arrived during its service time. We specify the probability $r_i$ that a change of size $i$ occurs in the number of jobs from one departure instant to the next. By conditioning on the length of the service time and using that the number of arrivals within the interval $[0,t]$ is Poisson distributed with parameter $\la t$, we establish that
\begin{equation}%
r_i = \int_0^\infty \frac{(\la t)^{i + 1}}{(i + 1)!} \, \euler^{-\la t} f_B(t) \, \dinf t, \quad i \ge - 1.
\end{equation}%
A departing job can leave behind zero jobs. In that state, we first wait for a job to arrive and depart before observing the number of jobs in the system. This means that from state 0, we return to state 0 with probability $r_{-1}$ and move to state $i \ge 1$ with probability $r_{i - 1}$.

By specifying the states and the transition probabilities $r_i$, we have in fact constructed an embedded Markov chain. It is called embedded because we only observe the process at embedded points in time (at departure instants) and the term `chain' indicates that it has transition probabilities instead of transition rates and that the time spent in each state is equal. The transition probability diagram of this Markov chain is presented in \cref{figQTF:transition_probability_diagram_MG1}.

\begin{figure}
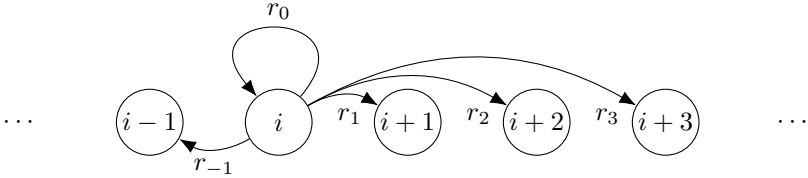
%
\centering%
\includestandalone{Chapters/QTF/TikZFiles/transition_probability_diagram_MG1}%
\caption{Transition probability diagram of the embedded Markov chain associated with the $M/G/1$ queue. Only transitions from state $i$ to other states are shown.}%
\label{figQTF:transition_probability_diagram_MG1}%
\end{figure}%

Each state has incoming transitions from states below itself, from itself, and one incoming transition from one state higher. This gives the following balance equations:
\begin{align}%
d(i) &=  r_i d(0) + r_{i - 1} d(1) + \cdots + r_0 d(i) + r_{-1} d(i + 1)  \notag \\
&= \sum_{j = 0}^{i + 1} r_{i - j} d(j). \label{eqnQTF:balance_equations_MG1}
\end{align}%
We manipulate the balance equations \eqref{eqnQTF:balance_equations_MG1} by making use of PGFs. Define
\begin{equation}%
\PGF{d}{\PGFarg} \defi \sum_{i \ge 0} d(i) \PGFarg^i, \quad \PGF{r}{\PGFarg} \defi \sum_{i \ge 0} r_{i - 1} \PGFarg^i, \quad |\PGFarg| \le 1. \label{eqnQTF:MG1_PGF_definition}
\end{equation}%
Multiply both sides of \eqref{eqnQTF:balance_equations_MG1} by $\PGFarg^i$ and sum over all $i$ to obtain
\begin{align}%
\PGF{d}{\PGFarg} &= \sum_{i \ge 0} \sum_{j = 0}^{i + 1} r_{i - j} d(j) \PGFarg^i \notag = \sum_{i \ge 0} r_i d(0) \PGFarg^i + \sum_{i \ge 0} \sum_{j = 1}^{i + 1} r_{i - j} d(j) \PGFarg^i \notag \\
&= d(0) \PGF{r}{\PGFarg} + \sum_{i \ge 0} \sum_{j = 1}^{i + 1} r_{i - j} d(j) \PGFarg^i.
\end{align}%
By changing the order of the double summation and writing $\PGFarg^i = \PGFarg^j \PGFarg^{i - j}$, we get
\begin{equation}%
\PGF{d}{\PGFarg} = d(0) \PGF{r}{\PGFarg} + \sum_{j \ge 1} d(j) \PGFarg^j \sum_{i \ge j - 1} r_{i - j}  \PGFarg^{i - j}.
\end{equation}%
Changing the summation index of the inner summation to $k = i - j + 1$ yields
\begin{align}%
\PGF{d}{\PGFarg} &= d(0) \PGF{r}{\PGFarg} + \sum_{j \ge 1} d(j) \PGFarg^j \sum_{k \ge 0} r_{k - 1}  \PGFarg^{k - 1} \notag \\
&= d(0) \PGF{r}{\PGFarg} + \sum_{j \ge 1} d(j) \PGFarg^j \frac{\PGF{r}{\PGFarg}}{\PGFarg} \notag \\
&= d(0) \PGF{r}{\PGFarg} + \bigl( \PGF{d}{\PGFarg} - d(0) \bigr) \frac{\PGF{r}{\PGFarg}}{\PGFarg},
\end{align}%
so that
\begin{equation}%
\PGF{d}{\PGFarg} = d(0) \PGF{r}{\PGFarg} \frac{1 - \frac{1}{\PGFarg}}{1 - \frac{\PGF{r}{\PGFarg}}{\PGFarg}} = d(0) \PGF{r}{\PGFarg} \frac{1 - \PGFarg}{\PGF{r}{\PGFarg} - \PGFarg}. \label{eqnQTF:MG1_PGF_almost_explicit_1}
\end{equation}%
It remains to determine $d(0)$ and $\PGF{r}{\PGFarg}$. We first find an expression for $\PGF{r}{\PGFarg}$:
\begin{align}%
\PGF{r}{\PGFarg} &= \sum_{i \ge 0} r_{i - 1} \PGFarg^i = \sum_{i \ge 0} \int_0^\infty \frac{(\la t)^{i}}{i!} \, \euler^{-\la t} f_B(t) \, \dinf t \, \PGFarg^i \notag \\
&= \int_0^\infty \sum_{i \ge 0} \frac{(\la \PGFarg t)^{i}}{i!} \, \euler^{-\la t} f_B(t) \, \dinf t \notag \\
&= \int_0^\infty \euler^{-\la(1 - \PGFarg)t} f_B(t) \, \dinf t = \LST{B}{\la(1 - \PGFarg)}, \label{eqnQTF:MG1_PGF_transition_probabilities_in_terms_of_service_time}
\end{align}%
where $\LST{B}{\cdot}$ is the LST of the service time $B$. Substituting \eqref{eqnQTF:MG1_PGF_transition_probabilities_in_terms_of_service_time} into \eqref{eqnQTF:MG1_PGF_almost_explicit_1} yields
\begin{equation}%
\PGF{d}{\PGFarg} = d(0) \LST{B}{\la(1 - \PGFarg)} \frac{1 - \PGFarg}{\LST{B}{\la(1 - \PGFarg)} - \PGFarg}, \label{eqnQTF:MG1_PGF_almost_explicit_2}
\end{equation}%
where $d(0)$ follows from $\lim_{\PGFarg \to 1} \PGF{d}{\PGFarg} = 1$. If we apply this limit to the right-hand side of \eqref{eqnQTF:MG1_PGF_almost_explicit_2}, then we get an indeterminate form. By taking $\PGFarg \to 1$ in \eqref{eqnQTF:MG1_PGF_almost_explicit_2} and applying l'H\^opital's rule to the fraction on the right-hand side, we obtain
\begin{equation}%
1 = d(0) \lim_{\PGFarg \to 1} \frac{1 - \PGFarg}{\LST{B}{\la(1 - \PGFarg)} - \PGFarg} = d(0) \lim_{\PGFarg \to 1} \frac{-1}{-\la \LSTder{B}{\la(1 - \PGFarg)} - 1} = \frac{d(0)}{1 - \rho}, \label{eqnQTF:MG1_PGF_almost_explicit_3}
\end{equation}%
so that $d(0) = 1 - \rho$. We finally obtain
\begin{equation}%
\PGF{d}{\PGFarg} = (1 - \rho) \frac{(1 - \PGFarg)\LST{B}{\la(1 - \PGFarg)}}{\LST{B}{\la(1 - \PGFarg)} - \PGFarg}, \label{eqnQTF:MG1_PGF_explicit}
\end{equation}%
which connects the PGF of the departure distribution to the LST of the service time distribution. This formula is referred to as the \textit{Pollaczek-Khinchin formula}\endnote{The Pollaczek-Khinchin formula comes in several variants. Accounts of this formula are published by Pollaczek \cite{Pollaczek1930_P-K_formula} in German and two years later by Khinchin in Russian (see \cite{Khinchin1967_P-K_formula} for a translation). Equation~\eqref{eqnQTF:MG1_PGF_explicit} can be found in \cite[Section~5.1.2]{Gross1974_Fundamentals_of_queueing_theory} or \cite[Section~5.6]{Kleinrock1975_Queueing_systems_volume_I}, whereas the variant for the LST of the sojourn time distribution presented in \eqref{eqnQTF:MG1_LST_sojourn_time_explicit} can be found in the classical textbooks \cite[Section~II.4.5]{Cohen1969_Single_server_queue}, \cite[Section~5.7]{Kleinrock1975_Queueing_systems_volume_I} or \cite[Section~VIII.5b]{Asmussen2008_Applied_probability_and_queues}.}. By differentiating \eqref{eqnQTF:MG1_PGF_explicit} we can determine the moments of the number of jobs in the system at a departure instant. To find its distribution, however, we have to invert \eqref{eqnQTF:MG1_PGF_explicit}, which under general conditions is not straightforward. If the LST $\LST{B}{\LSTarg}$ is a rational function---which means that it is a quotient of polynomials in $\LSTarg$---then the right-hand side of \eqref{eqnQTF:MG1_PGF_explicit} can be decomposed into partial fractions and the inverse transform can be easily determined. We now show this by example.

\begin{example}[Erlang services]\label{exQTF:MG1_Erlang_services_departure_distribution}%
Assume that the service times $B$ follow an Erlang distribution consisting of two exponential phases with rate $\mu$ in each phase. Label the two exponential phases as $B_1$ and $B_2$. The LST of $B$ is given by
\begin{equation}%
\LST{B}{\LSTarg} = \E{\euler^{-\LSTarg B}} = \E{\euler^{-\LSTarg (B_1 + B_2)}} = \E{\euler^{-\LSTarg B_1}} \E{\euler^{-\LSTarg B_2}} = \bigl( \frac{\mu}{\mu + \LSTarg} \bigr)^2,
\end{equation}%
and when evaluated in $\la(1 - \PGFarg)$ we can write
\begin{equation}%
\LST{B}{\la(1 - \PGFarg)} = \bigl( \frac{1}{1 + \frac{\rho}{2}(1 - \PGFarg)} \bigr)^2,
\end{equation}%
where in this case $\rho = 2\la/\mu$. Substituting this expression into \eqref{eqnQTF:MG1_PGF_explicit} gives
\begin{equation}%
\PGF{d}{\PGFarg} = (1 - \rho) \bigl( \frac{1}{1 + \frac{\rho}{2}(1 - \PGFarg)} \bigr)^2 \frac{1 - \PGFarg}{(\frac{1}{1 + \frac{\rho}{2}(1 - \PGFarg)})^2 - \PGFarg}.
\end{equation}%
Multiplying the numerator and denominator of the second fraction by the term $(1 + (\rho/2)(1 - \PGFarg))^2$ and simplifying gives
\begin{equation}%
\PGF{d}{\PGFarg} = \frac{4(1 - \rho)}{4 - (4\rho + \rho^2) \PGFarg + \rho^2 \PGFarg^2}.
\end{equation}%
If we now pick a value for $\rho$, then we can easily decompose $\PGF{d}{\PGFarg}$ into partial fractions. For example, let us choose $\rho = 1/3$ to obtain
\begin{align}%
\PGF{d}{\PGFarg} &= \frac{\frac{8}{3}}{4 - \frac{13}{9} \PGFarg + \frac{1}{9} \PGFarg^2} = \frac{24}{36 - 13 \PGFarg + \PGFarg^2} = \frac{24}{(4 - \PGFarg)(9 - \PGFarg)} \notag \\
&= \frac{6}{5} \frac{4}{4 - \PGFarg} - \frac{8}{15} \frac{9}{9 - \PGFarg} = \frac{6}{5} \frac{1}{1 - \frac{\PGFarg}{4}} - \frac{8}{15} \frac{1}{1 - \frac{\PGFarg}{9}} \notag \\
&= \frac{6}{5} \sum_{i \ge 0} \frac{1}{4^i} \PGFarg^i - \frac{8}{15} \sum_{i \ge 0} \frac{1}{9^i} \PGFarg^i = \sum_{i \ge 0} \Bigl( \frac{6}{5} \frac{1}{4^i} - \frac{8}{15} \frac{1}{9^i} \Bigr) \PGFarg^i.
\end{align}%
From this expression for $\PGF{d}{\PGFarg}$ we conclude that
\begin{equation}%
d(i) = \frac{6}{5} \frac{1}{4^i} - \frac{8}{15} \frac{1}{9^i}, \quad i \ge 0. \label{eqnQTF:MG1_example_d(i)}
\end{equation}%
Notice that \eqref{eqnQTF:MG1_example_d(i)} agrees with $d(0) = 1 - \rho = 2/3$, since $6/5 - 8/15 = 2/3$.
\end{example}%

We have determined the PGF of the departure distribution. However, as usual we are interested in the equilibrium probability $p(i)$ of having $i$ jobs in the system. We know from the PASTA property for $M/\cdot/\cdot$ systems that $a(i) = p(i)$ for all $i$. We now argue that $d(i)$ is also equal to $a(i)$. Taking the number of jobs in the system as the state of the queueing system, the changes in state are of a nearest-neighbor type: if the system is in state $i$, then an arrival of a job leads to a transition to state $i + 1$ and a departure of a job leads to a transition to state $i - 1$. Now, if the system is in equilibrium, then the number of transitions per unit time from state $i$ to $i + 1$ is equal to the number of transitions per unit time from state $i + 1$ to $i$. The former transitions correspond to jobs finding upon arrival $i$ jobs already in the system, which occurs at rate $\la a(i)$. The latter transitions correspond to departing jobs leaving behind $i$ jobs in the system, which occurs at rate $\la d(i)$ (under the stability condition $\rho < 1$, jobs depart at rate $\la$). Since these two transition rates are equal, we establish that $a(i) = d(i)$ and therefore $d(i) = p(i)$. Notice that in the argument establishing $a(i) = d(i)$ we did not use the distribution of the inter-arrival times or service times nor the number of servers; we only used that jobs depart one by one. So, the equality $a(i) = d(i)$ even holds for systems such as $G/G/c$ queues.

\begin{remark}[Partial fraction decomposition]\label{exQTF:MG1_PK_series_expression}%
In some case, we are able to write the Pollaczek-Khinchin formula \eqref{eqnQTF:MG1_PGF_explicit} as the ratio
\begin{equation}%
\PGF{d}{\PGFarg} = \frac{N(\PGFarg)}{D(\PGFarg)}
\end{equation}%
with both $N(\PGFarg)$ and $D(\PGFarg)$ polynomials without any common roots. Let $\PGFarg_1,\PGFarg_2,\ldots,\PGFarg_k$ be the roots of $D(\PGFarg) = 0$. Since the radius of convergence of $\PGF{d}{\PGFarg}$ is at least 1, we know that all $|\PGFarg_j| > 1$. We can write $D(\PGFarg)$ as
\begin{equation}%
D(\PGFarg) = (\PGFarg - \PGFarg_1)(\PGFarg - \PGFarg_2) \cdots (\PGFarg - \PGFarg_k),
\end{equation}%
which means that we can use a partial fraction decomposition to write
\begin{equation}%
\PGF{d}{\PGFarg} = \sum_{j = 1}^k \frac{n_j}{\PGFarg_j - \PGFarg},
\end{equation}%
where we still need to determine the coefficients $n_j$. If we restrict $\PGFarg$ to $|\PGFarg| < |\PGFarg_j|$ for all $\PGFarg_j$, then we can write
\begin{equation}%
\PGF{d}{\PGFarg} = \sum_{j = 1}^k \frac{n_j}{\PGFarg_j} \frac{1}{1 - \frac{\PGFarg}{\PGFarg_j}} = \sum_{j = 1}^k \frac{n_j}{\PGFarg_j} \sum_{i \ge 0} \bigl( \frac{\PGFarg}{\PGFarg_j} \bigr)^i = \sum_{i \ge 0} \Bigl( \sum_{j = 1}^k \frac{n_j}{\PGFarg_j^{i + 1}} \Bigr) \PGFarg^i.
\end{equation}%
Comparing this expression to \eqref{eqnQTF:MG1_PGF_definition} we conclude that
\begin{equation}%
d(i) = \sum_{j = 1}^k \frac{n_j}{\PGFarg_j^{i + 1}}. \label{eqnQTF:MG1_partial_fraction_decomposition_P-F_solution}
\end{equation}%
The coefficients $n_j$ follow from
\begin{align}%
n_j &= \lim_{\PGFarg \to \PGFarg_j} (\PGFarg_j - \PGFarg) \PGF{d}{\PGFarg} = \lim_{\PGFarg \to \PGFarg_j} (\PGFarg_j - \PGFarg) \frac{N(\PGFarg)}{D(\PGFarg)} \notag \\
&= \lim_{\PGFarg \to \PGFarg_j} - \frac{N(\PGFarg)}{(\PGFarg - \PGFarg_1) \cdots (\PGFarg - \PGFarg_{j - 1})(\PGFarg - \PGFarg_{j + 1}) \cdots (\PGFarg - \PGFarg_k)} = - \frac{N(\PGFarg_j)}{D'(\PGFarg_j)},
\end{align}%
where $D'(\cdot)$ is the derivative with respect to $\PGFarg$ of $D(\cdot)$.
\end{remark}%

%%%%%%%%%%%%%%%%%%%%%%%%%%%%%%%%%%%%%%%%%%%%%%%%%%%%%%%
%%%%%%%%%%%%%%%%%%%%%%%%%%%%%%%%%%%%%%%%%%%%%%%%%%%%%%%
%%%%%%%%%%%%%%%%%%% NEW SUBSECTION %%%%%%%%%%%%%%%%%%%%
%%%%%%%%%%%%%%%%%%%%%%%%%%%%%%%%%%%%%%%%%%%%%%%%%%%%%%%
%%%%%%%%%%%%%%%%%%%%%%%%%%%%%%%%%%%%%%%%%%%%%%%%%%%%%%%

\subsection{Sojourn time distribution}%
\label{subsecQTF:MG1_sojourn_time_distribution}%

We now ask how much time a job spends in the system and we show that there is a nice relationship between the transforms of the time spent in the system and the departure distribution.

Consider a job arriving to the system in equilibrium. Denote the sojourn time of this job by the random variable $S$ with cumulative distribution function $F_S(\cdot)$ and probability density function $f_S(\cdot)$. If we assume that jobs are served in first-come first-served order, then we know that a departing job leaves behind exactly those jobs that arrived during its sojourn time. By conditioning on the length of the sojourn time, we can construct the departure distribution:
\begin{equation}%
d(i) = \int_0^\infty \frac{(\la t)^i}{i!} \, \euler^{-\la t} f_S(t) \, \dinf t. \label{eqnQTF:MG1_d(i)_in_terms_of_sojourn_time}
\end{equation}%
Multiply both sides of \eqref{eqnQTF:MG1_d(i)_in_terms_of_sojourn_time} by $\PGFarg^i$ and sum over all $i$ to retrieve the PGF of the departure distribution on the left-hand side and the LST $\LST{S}{\cdot}$ of the sojourn time on the right-hand side (similar to the derivation in \eqref{eqnQTF:MG1_PGF_transition_probabilities_in_terms_of_service_time}):
\begin{equation}%
\PGF{d}{\PGFarg} = \LST{S}{\la(1 - \PGFarg)}. \label{eqnQTF:MG1_distributional_Little's_law_for_Poisson_arrivals}
\end{equation}%
Substituting this relation into \eqref{eqnQTF:MG1_PGF_explicit} and introducing $\LSTarg = \la(1 - \PGFarg)$, we finally arrive at
\begin{equation}%
\LST{S}{\LSTarg} = (1 - \rho) \LST{B}{\LSTarg} \frac{\LSTarg}{\la \LST{B}{\LSTarg} + \LSTarg - \la}, \label{eqnQTF:MG1_LST_sojourn_time_explicit}
\end{equation}%
which is, like \eqref{eqnQTF:MG1_PGF_explicit}, a form of the Pollaczek-Khinchin formula.

\begin{example}[Erlang services]\label{exQTF:MG1_Erlang_services_sojourn_time}%
Consider again the model described in \cref{exQTF:MG1_Erlang_services_departure_distribution}, where the service times $B$ follow an Erlang distribution consisting of two exponential phases with rate $\mu$ in each phase. We determine $\LST{S}{\LSTarg}$ and invert it to obtain $F_S(\cdot)$. From \eqref{eqnQTF:MG1_LST_sojourn_time_explicit} we find that
\begin{equation}%
\LST{S}{\LSTarg} = (1 - \rho) \bigl( \frac{\mu}{\mu + \LSTarg} \bigr)^2 \frac{\LSTarg}{\la (\frac{\mu}{\mu + \LSTarg})^2 + \LSTarg - \la}
\end{equation}%
Multiplying the numerator and denominator of the second fraction by the term $(\mu + \LSTarg)^2$ and simplifying gives
\begin{equation}%
\LST{S}{\LSTarg} = \frac{(1 - \rho)}{1 - \rho + \frac{1}{\mu}(2 - \frac{\rho}{2}) \LSTarg + \frac{1}{\mu^2} \LSTarg^2}.
\end{equation}%
Choose $\mu = 6$ and $\rho = 1/3$ so that
\begin{align}%
\LST{S}{\LSTarg} &= \frac{\frac{2}{3}}{\frac{2}{3} + \frac{1}{6}(2 - \frac{1}{6}) \LSTarg + \frac{1}{36} \LSTarg^2} = \frac{\frac{2}{3}}{\frac{2}{3} + \frac{11}{36} \LSTarg + \frac{1}{36} \LSTarg^2} \notag \\
&= \frac{24}{24 + 11 \LSTarg + \LSTarg^2} = \frac{24}{(8 + \LSTarg)(3 + \LSTarg)} = \frac{8}{8 + \LSTarg} \frac{3}{3 + \LSTarg}.
\end{align}%
From the LST of $S$ we can deduce that $S$ is the sum of two exponential random variables with rates 3 and 8, denoted by $X_1$ and $X_2$, respectively. Obtaining the cumulative distribution function $F_S(\cdot)$ requires some more work:
\begin{equation}%
F_S(t) = \Prob{S \le t} = \Prob{X_1 + X_2 \le t} = \int_0^t \Prob{X_1 \le t - x} f_{X_2}(x) \, \dinf x.
\end{equation}%
Solving the integral finally gives
\begin{equation}%
F_S(t) = \frac{8}{5}(1 - \euler^{-3t}) - \frac{3}{5}(1 - \euler^{-8t}), \quad t \ge 0.
\end{equation}%
where we recognize the cumulative distribution functions of $X_1$ and $X_2$ multiplied by some weights.
\end{example}%

%%%%%%%%%%%%%%%%%%%%%%%%%%%%%%%%%%%%%%%%%%%%%%%%%%%%%%%
%%%%%%%%%%%%%%%%%%%%%%%%%%%%%%%%%%%%%%%%%%%%%%%%%%%%%%%
%%%%%%%%%%%%%%%%%%% NEW SUBSECTION %%%%%%%%%%%%%%%%%%%%
%%%%%%%%%%%%%%%%%%%%%%%%%%%%%%%%%%%%%%%%%%%%%%%%%%%%%%%
%%%%%%%%%%%%%%%%%%%%%%%%%%%%%%%%%%%%%%%%%%%%%%%%%%%%%%%

\subsection{Distributional Little's law}%
\label{subsecQTF:Little's_law}%

The relation \eqref{eqnQTF:MG1_distributional_Little's_law_for_Poisson_arrivals} between the PGF of the number of jobs left behind upon departure and the LST of the sojourn time is a special case of \textit{distributional Little's law}\endnote{The distributional variant of Little's law is due to \cite{Haji1971_Distributional_Little's_law}. For the distributional law in the context of Poisson arrivals, we refer the reader to \cite{Keilson1988_Distributional_Little's_law}. A discussion of this law under milder conditions can be found in \cite{Bertsimas1995_Distributional_Little's_law} and the references therein. Many of the results presented in \cref{subsecQTF:Little's_law} on distributional Little's law are adapted from \cite{Bertsimas1995_Distributional_Little's_law}.}. This fundamental law holds under a number of conditions, namely
\begin{enumerate}[label = (\roman*)]%
\item All arriving jobs enter the system one at a time, remain in the system until served and leave one at a time;
\item Jobs leave the system in the order of arrival;
\item Jobs that arrive later in time do not affect the time spent in the system of jobs that arrived earlier in time.
\end{enumerate}%
Here, a system can be used to mean only the queue, only the server, or the complete queueing system. To formulate distributional Little's law, define $N(t)$ as the number of arrivals up to time $t$, where the first inter-arrival time is distributed as a residual inter-arrival time and all other inter-arrival times are distributed according to the stationary inter-arrival time. The residual inter-arrival time is the time between any given time $t$ and the next arrival epoch of the arrival process. We describe distributional Little's law in terms of the equilibrium number of jobs in the system $L$ and the equilibrium sojourn time $S$. Let $F_S(\cdot)$ denote the cumulative distribution function of $S$.

\begin{theorem}[Distributional Little's law \cite{Bertsimas1995_Distributional_Little's_law,Haji1971_Distributional_Little's_law}]\label{thmQTF:distributional_Little's_law}%
Under the conditions mentioned above and under the further assumption that $L$ and $S$ exist,
\begin{equation}%
L \dequal N(S),
\end{equation}%
or, in terms of the \textup{PGF}s of $L$ and $N$,,
\begin{equation}%
\PGF{L}{\PGFarg} = \int_0^\infty \PGF{N}{\PGFarg,t} \, \dinf F_S(t), \label{eqnQTF:distributional_Little's_law_transforms}
\end{equation}%
where
\begin{equation}%
\PGF{L}{\PGFarg} \defi \sum_{n \ge 0} \Prob{L = n} \PGFarg^n, \quad \PGF{N}{\PGFarg,t} \defi \sum_{n \ge 0} \Prob{N(t) = n} \PGFarg^n.
\end{equation}%
\end{theorem}%

\begin{proof}
This proof can be found in \cite{Bertsimas1995_Distributional_Little's_law,Haji1971_Distributional_Little's_law}. Define $t$ to be a random observation epoch and let $t_n$ be the arrival time of the $n$-th job still in the system at time $t$ and $S_n$ its sojourn time in the system. The order in which we number the jobs is important. Job 1 is the job that arrived most recently in time with respect to the random observation time $t$ and is therefore at the end of the queue. The job with the highest index is the one currently in service. So, the job with index $n$ departs the system at time $t_n + S_n$. The $t_n$ and $S_n$ are ordered in reverse time direction. Next, define the inter-arrival times as $A_1^* \defi t - t_1$ and $A_n \defi t_{n - 1} - t_n, ~ n \ge 1$. We note that $A_1^*$ is a residual inter-arrival time. \cref{figQTF:proof_distributional_Little's_law} displays the notation and indexing used.

\begin{figure}
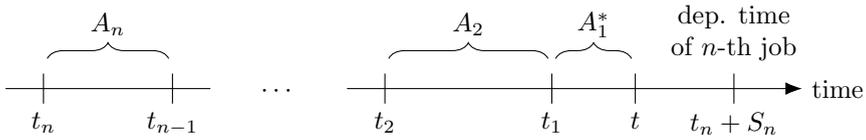
%
\centering%
\includestandalone{Chapters/QTF/TikZFiles/proof_distributional_Littles_law}%
\caption{Notation and indexing used for the proof of \protect\cref{thmQTF:distributional_Little's_law}.}%
\label{figQTF:proof_distributional_Little's_law}%
\end{figure}%

If, at the random observation time $t$, the observer sees at least $n$ jobs in the system, then the $n$-th most recently arrived job is still in the system at the observation time $t$. In particular, this means that the departure time $t_n + S_n$ of the $n$-th job is larger than $t$. So, $L \ge n$ if and only if $S_n > t - t_n$. This indicates that
\begin{equation}%
\Prob{L \ge n} = \Prob{S_n > t - t_n} = \Prob{S_n > A_1^* + \sum_{m = 2}^n A_m},
\end{equation}%
where we used a telescoping sum to derive the last equality. In equilibrium $S_n \dequal S$, so that conditioning on the length of the sojourn time leads to
\begin{equation}%
\Prob{L \ge n} = \int_0^\infty \Prob{t > A_1^* + \sum_{m = 2}^n A_m} \, \dinf F_S(t).
\end{equation}%
Finally, the probability inside the integral is exactly the probability that at least $n$ arrivals occur in $[0,t]$ where the first inter-arrival time is distributed according to the residual inter-arrival time and the other inter-arrival times are distributed according to the stationary inter-arrival time. Therefore,
\begin{equation}%
\Prob{L \ge n} = \int_0^\infty \Prob{N(t) \ge n} \, \dinf F_S(t) = \Prob{N(S) \ge n}, \label{eqnQTF:distributional_Little's_law_proof_1}
\end{equation}%
proving the first statement of the theorem. The probability that there are exactly $n$ jobs in equilibrium easily follows from \eqref{eqnQTF:distributional_Little's_law_proof_1} as
\begin{align}%
\Prob{L = n} &= \Prob{L \ge n} - \Prob{L \ge n + 1} \notag \\
&= \int_0^\infty \Prob{N(t) \ge n} \, \dinf F_S(t) - \int_0^\infty \Prob{N(t) \ge n + 1} \, \dinf F_S(t) \notag \\
&= \int_0^\infty \Prob{N(t) = n} \, \dinf F_S(t). \label{eqnQTF:distributional_Little's_law_proof_2}
\end{align}%
Multiplying both sides of \eqref{eqnQTF:distributional_Little's_law_proof_2} by $\PGFarg^n$, summing over all $n \ge 0$ and applying Tonelli's theorem to interchange the summation and integral on the right-hand side produces the second statement of the theorem.
\end{proof}%

For a Poisson arrival process, both the residual and stationary inter-arrival times are exponentially distributed with parameter $\la$. The PGF $\PGF{N}{\PGFarg,t}$ then reads
\begin{equation}%
\PGF{N}{\PGFarg,t} = \sum_{n \ge 0} \PGFarg^n \frac{(\la t)^n}{n!} \euler^{-\la t} = \euler^{-\la (1 - \PGFarg) t}.
\end{equation}%
Substituting this simplification into \eqref{eqnQTF:distributional_Little's_law_transforms} yields
\begin{equation}%
\PGF{L}{\PGFarg} = \int_0^\infty \euler^{-\la (1 - \PGFarg) t} \, \dinf F_S(t) = \LST{S}{\la(1 - \PGFarg)},
\end{equation}%
which we have seen before in \eqref{eqnQTF:MG1_distributional_Little's_law_for_Poisson_arrivals}.

Note that \cref{thmQTF:distributional_Little's_law} does not hold in general for the number of jobs in an $M/G/c$ system with $c > 1$ servers and a FCFS service discipline, since jobs may overtake other jobs and therefore violate the second condition. On the other hand, it does hold for the number of jobs in an $M/D/c$ system with a FCFS service discipline, since being taking into service guarantees a certain departure time.

%%%%%%%%%%%%%%%%%%%%%%%%%%%%%%%%%%%%%%%%%%%%%%%%%%%%%%%
%%%%%%%%%%%%%%%%%%%%%%%%%%%%%%%%%%%%%%%%%%%%%%%%%%%%%%%
%%%%%%%%%%%%%%%%%%%%% NEW SECTION %%%%%%%%%%%%%%%%%%%%%
%%%%%%%%%%%%%%%%%%%%%%%%%%%%%%%%%%%%%%%%%%%%%%%%%%%%%%%
%%%%%%%%%%%%%%%%%%%%%%%%%%%%%%%%%%%%%%%%%%%%%%%%%%%%%%%

\section{Single-server queue with general inter-arrival times}%
\label{secQTF:GM1_queue}%

The dual of the $M/G/1$ system discussed in \cref{secQTF:MG1_queue} is the $G/M/1$ system, which is a single-server queueing system with generally distributed inter-arrival times and exponential service times with rate $\mu$. We assume that the inter-arrival times have a cumulative distribution function $F_A(\cdot)$, a probability density function $f_A(\cdot)$ and have mean $1/\la$. For stability we require that $\rho \defi \la/\mu < 1$.

The state of the $G/M/1$ system can be described by a pair $(i,t)$ with $i$ the number of jobs in the system and $t$ the elapsed time since the last arrival. As we have argued for the $M/G/1$ systen, this state description leads to complications and the analysis simplifies considerably if we focus on special points in time. In this case, we look at the system at arrival instants so that in the state description $t$ is always 0 and we only keep track of the number of jobs in the system at an arrival instant. We denote by $a(i)$ the equilibrium probability that an arriving job encounters $i$ jobs in the system (excluding itself).

Unfortunately, since the arrivals do not follow a Poisson process, we cannot use PASTA to relate $a(i)$ to the equilibrium distribution $p(i)$ of the number of jobs in the system at arbitrary times. Nonetheless, we are still able to derive the distribution of the sojourn time using $a(i)$.

%%%%%%%%%%%%%%%%%%%%%%%%%%%%%%%%%%%%%%%%%%%%%%%%%%%%%%%
%%%%%%%%%%%%%%%%%%%%%%%%%%%%%%%%%%%%%%%%%%%%%%%%%%%%%%%
%%%%%%%%%%%%%%%%%%% NEW SUBSECTION %%%%%%%%%%%%%%%%%%%%
%%%%%%%%%%%%%%%%%%%%%%%%%%%%%%%%%%%%%%%%%%%%%%%%%%%%%%%
%%%%%%%%%%%%%%%%%%%%%%%%%%%%%%%%%%%%%%%%%%%%%%%%%%%%%%%

\subsection{Arrival distribution}%
\label{subsecQTF:GM1_arrival_distribution}%

We now derive the equilibrium probability $a(i)$ of encountering $i$ jobs in the system just before the arrival of a job. From one arrival instant to the next the number of jobs in the system increases by one, but decreases by the number of jobs that have arrived during its inter-arrival time. The number of jobs cannot decrease by more than the one plus the number of jobs present at the previous arrival instant. So, from state $i$ we can transition to any of the states $0,1,\ldots,i + 1$. Denote by $r_i$ the probability that a change of size $i$ occurs, under the assumption that this change does not bring us to state 0 (state 0 requires special treatment). We reuse the notation $r_i$ from the $M/G/1$ system because this probability has the same interpretation. By conditioning on the length of the inter-arrival time, we find that
\begin{equation}%
r_i = \int_0^\infty \frac{(\mu t)^{1 - i}}{(1 - i)!} \euler^{-\mu t} f_A(t) \, \dinf t, \quad i \le 1. \label{eqnQTF:GM1_transition_probabilities}
\end{equation}%
The transition probability from state $i$ to 0 is denoted by $q_{-i}$. Since the transition probabilities for each state sum to 1, it is easy to see that we must have
\begin{equation}%
q_{-i} = 1 - \sum_{j = -1}^{i - 1} r_{-i}.
\end{equation}%

By specifying the states and the transition probabilities, we have constructed the Markov chain associated with the $G/M/1$ system embedded at arrival instants. The transition probability diagram of this Markov chain is presented in \cref{figQTF:transition_probability_diagram_GM1}

\begin{figure}
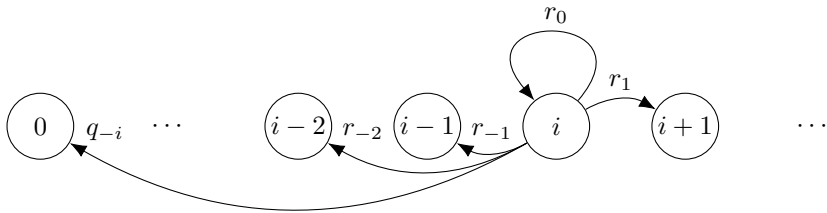
%
\centering%
\includestandalone{Chapters/QTF/TikZFiles/transition_probability_diagram_GM1}%
\caption{Transition probability diagram of the embedded Markov chain associated with the $G/M/1$ system. Only transitions from state $i$ to other states are shown.}%
\label{figQTF:transition_probability_diagram_GM1}%
\end{figure}%

The balance equations of this Markov chain are
\begin{equation}%
a(0) = a(0) q_0 + a(1) q_{-1} + a(2) q_{-2} + \cdots = \sum_{j \ge 0} a(j) q_{-j}, \label{eqnQTF:MG1_balance_equations_i=0}
\end{equation}%
and for $i \ge 1$,
\begin{equation}%
a(i) = a(i - 1) r_1 + a(i) r_0 + a(i + 1) r_{-1} + \cdots = \sum_{j \ge 0} a(i - 1 + j) r_{1 - j}. \label{eqnQTF:MG1_balance_equations_i>0}
\end{equation}%
It appears that the generating function approach does not work here. Instead, we guess that the solution to these balance equations is of the form
\begin{equation}%
a(i) = \al^i, \quad i \ge 0. \label{eqnQTF:GM1_educated_guess}
\end{equation}%
Substitution of \eqref{eqnQTF:GM1_educated_guess} into \eqref{eqnQTF:MG1_balance_equations_i>0} and dividing by $\al^{i - 1}$ yields
\begin{equation}%
\al = \sum_{j \ge 0} \al^j r_{1 - j}.
\end{equation}%
By also substituting \eqref{eqnQTF:GM1_transition_probabilities} for $r_{1 - j}$ we obtain
\begin{align}%
\al &= \sum_{j \ge 0} \al^j \int_0^\infty \frac{(\mu t)^j}{j!} \euler^{-\mu t} f_A(t) \, \dinf t \notag \\
&= \int_0^\infty \sum_{j \ge 0} \frac{(\al \mu t)^j}{j!} \euler^{-\mu t} f_A(t) \, \dinf t = \int_0^\infty \euler^{-\mu(1 - \al) t} f_A(t) \, \dinf t.
\end{align}%
The last integral can be recognized as the LST of the inter-arrival time and we obtain the equation
\begin{equation}%
\al = \LST{A}{\mu(1 - \al)}. \label{eqnQTF:GM1_alpha_satisfies}
\end{equation}%
Since $\LST{A}{0} = 1$, it is easy to see that $\al = 1$ is a root of \eqref{eqnQTF:GM1_alpha_satisfies}. However, this root is of no interest, since it does not produce a solution that can be normalized to obtain the equilibrium distribution. We show that you can obtain another root $\al \in (0,1)$, that does lead to a solution that can be normalized. Define
\begin{equation}%
f(\al) \defi \LST{A}{\mu(1 - \al)}.
\end{equation}%
We derive some properties of $f(\al)$ to show that it must intersect with the function $g(\al) = \al$ for $\al \in (0,1)$. First, it is easy to see that
\begin{equation}%
f(0) = \LST{A}{\mu} = \int_0^\infty \euler^{-\mu t} f_A(t) \, \dinf t > 0
\end{equation}%
and $f(1) = \LST{A}{0} = 1$, as we have already established. The derivative $f'(\al)$ of $f(\al)$ is given by
\begin{align}%
f'(\al) &= \frac{\dinf}{\dinf \al} \int_0^\infty \euler^{-\mu(1 - \al) t} f_A(t) \, \dinf t = \int_0^\infty \Bigl( \frac{\partial}{\partial \al} \euler^{-\mu(1 - \al) t} \Bigr) f_A(t) \, \dinf t \notag \\
&= \mu \int_0^\infty \euler^{-\mu(1 - \al) t} t f_A(t) \, \dinf t, \label{eqnQTF:GM1_derivative_root_equation}
\end{align}%
where the interchange of the derivative and the integral is allowed in this case by Leibniz's integral rule (see \cite{Flanders1973_Differentiation_under_the_integral_sign}) if we assume that $\rho < 1$. We will not discuss this interchange here. Substituting $\al = 1$ in \eqref{eqnQTF:GM1_derivative_root_equation} gives $f'(1) = 1/\rho > 1$ if $\rho < 1$. Pick $\al_1$ and $\al_2$ such that $0 \le \al_1 < \al_2 \le 1$, so that
\begin{equation}%
\euler^{-\mu(1 - \al_1) t} < \euler^{-\mu(1 - \al_2) t}, \quad t > 0.
\end{equation}%
By using this inequality, we see that $f'(\al)$ is increasing in $\al$ for $\al \in [0,1]$, we then say that $f(\al)$ is strictly convex for $\al \in [0,1]$. The properties of $f(\al)$ and $g(\al)$ are shown in \cref{figQTF:GM1_properties_of_equation_for_root}. Combining these properties we conclude that \eqref{eqnQTF:GM1_alpha_satisfies} has a single root $\al \in (0,1)$, which satisfies \eqref{eqnQTF:MG1_balance_equations_i>0} for $a(i) = \al^i$. Notice that the remaining balance equation \eqref{eqnQTF:MG1_balance_equations_i=0} is also satisfied, since the balance equations are dependent and one equation can therefore be omitted. We finally normalize the proposed solution to arrive at
\begin{equation}%
a(i) = (1 - \al) \al^i, \quad i \ge 0. \label{eqnQTF:GM1_P-F_solution}
\end{equation}%
Hence, the equilibrium number of jobs in the system just before arrival instants follows a geometric distribution with parameter $\al$, where $\al$ is the unique root of \eqref{eqnQTF:GM1_alpha_satisfies} in the interval $(0,1)$.

\begin{figure}
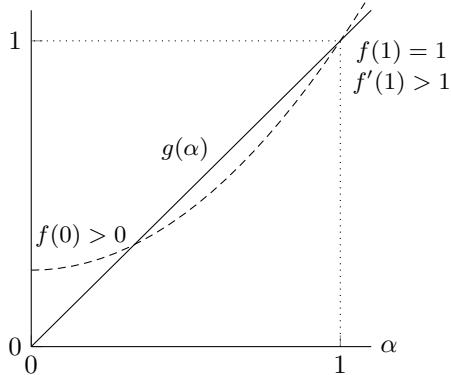
%
\centering%
\includestandalone{Chapters/QTF/TikZFiles/GM1_properties_of_equation_for_root}%
\caption{Properties of the functions $f(\al)$ (dashed) and $g(\al)$ (solid).}%
\label{figQTF:GM1_properties_of_equation_for_root}%
\end{figure}%

\begin{example}[Erlang arrivals]\label{exQTF:GM1_Erlang_arrivals_arrival_distribution}%
Suppose that the inter-arrival times $A$ follow an Erlang distribution consisting of two phases, where each exponential phase has rate $\la$. So, $\E{A} = 2/\la$ and $\rho = \la/(2\mu)$, where we assume that $\rho < 1$. The LST of $A$ is given by
\begin{equation}%
\LST{A}{\LSTarg} = \bigl( \frac{\la}{\la + \LSTarg} \bigr)^2,
\end{equation}%
and \eqref{eqnQTF:GM1_alpha_satisfies} becomes
\begin{equation}%
\al = \bigl( \frac{\la}{\la + \mu(1 - \al)} \bigr)^2,
\end{equation}%
which can be rewritten as
\begin{equation}%
\al (\la + \mu(1 - \al))^2 - \la^2 = 0.
\end{equation}%
Since we know that $\al = 1$ is a solution of this equation, we can write
\begin{equation}%
(\al - 1)\bigl( \al^2 \mu^2 - \al (\mu^2 + 2\la\mu) + \la^2 \bigr) = 0.
\end{equation}%
If we choose $\la = 3$ and $\mu = 4$, then we arrive at the solutions $\al = 1/4$, $\al = 1$ and $\al = 9/4$, so that
\begin{equation}%
a(i) = \frac{3}{4} \bigl( \frac{1}{4} \bigr)^i, \quad i \ge 0,
\end{equation}%
for this specific $G/M/1$ system.
\end{example}%

%%%%%%%%%%%%%%%%%%%%%%%%%%%%%%%%%%%%%%%%%%%%%%%%%%%%%%%
%%%%%%%%%%%%%%%%%%%%%%%%%%%%%%%%%%%%%%%%%%%%%%%%%%%%%%%
%%%%%%%%%%%%%%%%%%% NEW SUBSECTION %%%%%%%%%%%%%%%%%%%%
%%%%%%%%%%%%%%%%%%%%%%%%%%%%%%%%%%%%%%%%%%%%%%%%%%%%%%%
%%%%%%%%%%%%%%%%%%%%%%%%%%%%%%%%%%%%%%%%%%%%%%%%%%%%%%%

\subsection{Sojourn time distribution}%
\label{subsecQTF:GM1_sojourn_time_distribution}%

Since the arrival distribution is geometric, it is easy to determine the distribution of the sojourn time of a job. With probability $a(i)$ an arriving job finds $i$ jobs in the system. Because the service times are exponentially distributed, we know that the sojourn time of the arriving jobs is the sum of $i + 1$ exponential phases, each with rate $\mu$. By conditioning on the number of jobs seen on arrival, we therefore find that
\begin{align}%
\LST{S}{\LSTarg} &= \sum_{i \ge 0} a(i) \bigl( \frac{\mu}{\mu + \LSTarg} \bigr)^{i + 1} = \frac{\mu(1 - \al)}{\mu + \LSTarg} \sum_{i \ge 0} \bigl( \frac{\mu \al}{\mu + \LSTarg} \bigr)^i \notag \\
&= \frac{\mu(1 - \al)}{\mu + \LSTarg} \frac{1}{1 - \frac{\mu \al}{\mu + \LSTarg}} = \frac{\mu(1 - \al)}{\mu(1 - \al) + \LSTarg}.
\end{align}%
So the sojourn time is exponentially distributed with rate $\mu(1 - \al)$:
\begin{equation}%
F_S(t) = \Prob{S \le t} = 1 - \euler^{-\mu(1 - \al)t}, \quad t \ge 0.
\end{equation}%
%

%%%%%%%%%%%%%%%%%%%%%%%%%%%%%%%%%%%%%%%%%%%%%%%%%%%%%%%
%%%%%%%%%%%%%%%%%%%%%%%%%%%%%%%%%%%%%%%%%%%%%%%%%%%%%%%
%%%%%%%%%%%%%%%%%%%%% NEW SECTION %%%%%%%%%%%%%%%%%%%%%
%%%%%%%%%%%%%%%%%%%%%%%%%%%%%%%%%%%%%%%%%%%%%%%%%%%%%%%
%%%%%%%%%%%%%%%%%%%%%%%%%%%%%%%%%%%%%%%%%%%%%%%%%%%%%%%

\section{A reflected random walk}%
\label{secQTF:reflected_random_walk}%

In this section we introduce the reflected random walk, which can be seen as an extension of the embedded Markov chains associated with the $M/G/1$ and $G/M/1$ system. This reflected random walk can be modeled by a Markov chain with state space the non-negative integers $\Nat_0$. The term reflected refers to the fact that the Markov chain is reflected in state 0 back to the positive values. Let $X_n$ be the position of this random walk after $n$ steps with $X_0 \defi 0$ and satisfying the recursion
\begin{equation}%
X_{n + 1} = \max(0,X_n + A_n), \quad n \ge 0, \label{eqnQTF:random_walk_MC_recursion}
\end{equation}%
with $\{ A_n \}_{n \ge 0}$ a sequence of i.i.d.~discrete random variables that share the same distribution as some common random variable $A$. If we allow $A$ to take values in $\{ -1,0,1,2,\ldots \}$ then the Markov chain described by the recursion \eqref{eqnQTF:random_walk_MC_recursion} has the same transition structure as the embedded Markov chain associated with the $M/G/1$ system (see \cref{figQTF:transition_probability_diagram_MG1}) If we allow $A$ to take values in $\{ \ldots,-2,-1,0,1 \}$ then it has the same transition structure as the embedded Markov chain associated with the $G/M/1$ system (see \cref{figQTF:transition_probability_diagram_GM1}). To demonstrate some important techniques, we instead focus on the case where
\begin{equation}%
A \in \{ -s,-s + 1,\ldots,-1,0,1,2\ldots \}
\end{equation}%
with $s$ a positive integer and $\Prob{A = -s} > 0$. Notice that we can also write $A = B - s$ so that $B$ has as support the non-negative integers $\Nat_0$. The PGF of $A$ is therefore given by
\begin{equation}%
\PGF{A}{\PGFarg} = \E{\PGFarg^A} = \E{\PGFarg^{B - s}} = \frac{\PGF{B}{\PGFarg}}{\PGFarg^s}.
\end{equation}%

Assuming $\E{A} < 0$, which is equivalent to $\rho \defi \E{B}/s < 1$, the Markov chain is positive recurrent and we can study the equilibrium distribution. Denote by $X$ the equilibrium version of $X_n$. In equilibrium, the recursion \eqref{eqnQTF:random_walk_MC_recursion} becomes
\begin{equation}%
X \dequal \max(0,X + A) = \max(0,X + B - s).
\end{equation}%
From this relation we deduce that
\begin{equation}%
\Prob{X = 0} = \Prob{\max(0,X + B - s) = 0} = \sum_{i = 0}^s \Prob{X + B = i}
\end{equation}%
and for $k \ge 1$,
\begin{equation}%
\Prob{X = k} = \Prob{\max(0,X + B - s) = k} = \Prob{X + B = k + s}. \label{eqnQTF:random_walk_balance_equations_k>0}
\end{equation}%
Multiplying \eqref{eqnQTF:random_walk_balance_equations_k>0} by $\PGFarg^k$ and summing over all $k$ produces an expression for the PGF of $X$:
\begin{align}%
\PGF{X}{\PGFarg} &= \sum_{i = 0}^s \Prob{X + B = i} + \frac{1}{\PGFarg^s} \sum_{k \ge 1} \Prob{X + B = k + s} \PGFarg^{k + s} \notag \\
&= \sum_{i = 0}^s \Prob{X + B = i} + \frac{1}{\PGFarg^s} \Bigl[ \sum_{i \ge 0} \Prob{X + B = i} \PGFarg^i - \sum_{i = 0}^s \Prob{X + B = i} \PGFarg^i \Bigr]. \label{eqnQTF:random_walk_PGF_in_terms_of_summations}
\end{align}%
Recognizing the PGF of $X + B$, we can rewrite \eqref{eqnQTF:random_walk_PGF_in_terms_of_summations} as
\begin{equation}%
\PGF{X}{\PGFarg} = \frac{\sum_{i = 0}^{s - 1} \Prob{X + B = i}(\PGFarg^s - \PGFarg^i)}{\PGFarg^s - \PGF{B}{\PGFarg}}. \label{eqnQTF:random_walk_PGF_expression_1}
\end{equation}%
This expression still involves the $s$ unknowns $\Prob{X + B = i}, ~ 0 \le i \le s - 1$. Factorize the polynomial in $\PGFarg$ of degree $s$ in the numerator of \eqref{eqnQTF:random_walk_PGF_expression_1} as
\begin{equation}%
\sum_{i = 0}^{s - 1} \Prob{X + B = i}(\PGFarg^s - \PGFarg^i) = \gamma \prod_{k = 1}^s (\PGFarg - \PGFarg_k), \label{eqnQTF:random_walk_numerator_simplification}
\end{equation}%
where $\PGFarg_k$ are the $s$ roots of the polynomial and $\gamma$ is a constant. The values of the roots are still unknown, but we return to this issue later in \cref{subsecQTF:random_walk_finding_roots_z_k}. However, it is immediate that one of the roots, say $\PGFarg_s$, takes the value 1, so that we obtain
\begin{equation}%
\sum_{i = 0}^{s - 1} \Prob{X + B = i}(\PGFarg^s - \PGFarg^i) = \gamma (\PGFarg - 1) \prod_{k = 1}^{s - 1} (\PGFarg - \PGFarg_k). \label{eqnQTF:random_walk_PGF_expression_1_numerator}
\end{equation}%
What remains is to determine the constant $\gamma$. Taking derivatives with respect to $\PGFarg$ and substituting $\PGFarg = 1$ on both sides of \eqref{eqnQTF:random_walk_PGF_expression_1_numerator} yields
\begin{equation}%
\sum_{i = 0}^{s - 1} \Prob{X + B = i}(s - i) = \gamma \prod_{k = 1}^{s - 1} (1 - \PGFarg_k).
\end{equation}%
Now the function $\PGFarg^s - \PGF{B}{\PGFarg}$ (sometimes called the \textit{kernel}) comes into play. Since we know that $\PGF{X}{1} = 1$, we can apply l'H\^opital's rule to \eqref{eqnQTF:random_walk_PGF_expression_1} to find
\begin{equation}%
\sum_{i = 0}^{s - 1} \Prob{X + B = i}(s - i) = s - \PGFder{B}{1}, \label{eqnQTF:random_walk_numerator_normalization_condition}
\end{equation}%
which shows that
\begin{equation}%
\gamma = \frac{s - \PGFder{B}{1}}{\prod_{k = 1}^{s - 1} (1 - \PGFarg_k)}.
\end{equation}%
Returning to \eqref{eqnQTF:random_walk_PGF_expression_1} we finally obtain
\begin{equation}%
\PGF{X}{\PGFarg} = \frac{(s - \PGFder{B}{1})(\PGFarg - 1)}{\PGFarg^s - \PGF{B}{\PGFarg}}  \prod_{k = 1}^{s - 1} \frac{\PGFarg - \PGFarg_k}{1 - \PGFarg_k}. \label{eqnQTF:random_walk_PGF_explicit_expression}
\end{equation}%
%

%%%%%%%%%%%%%%%%%%%%%%%%%%%%%%%%%%%%%%%%%%%%%%%%%%%%%%%
%%%%%%%%%%%%%%%%%%%%%%%%%%%%%%%%%%%%%%%%%%%%%%%%%%%%%%%
%%%%%%%%%%%%%%%%%%% NEW SUBSECTION %%%%%%%%%%%%%%%%%%%%
%%%%%%%%%%%%%%%%%%%%%%%%%%%%%%%%%%%%%%%%%%%%%%%%%%%%%%%
%%%%%%%%%%%%%%%%%%%%%%%%%%%%%%%%%%%%%%%%%%%%%%%%%%%%%%%

\subsection{Finding the roots $\PGFarg_k$}%
\label{subsecQTF:random_walk_finding_roots_z_k}%

The roots $\PGFarg_k$ in \eqref{eqnQTF:random_walk_PGF_explicit_expression} are still unknown. We do not directly study $\PGFarg_k$, but instead focus on the properties of the PGF $\PGF{X}{\PGFarg}$. In particular, we show that $\PGFarg^s - \PGF{B}{\PGFarg}$ has $s$ roots in the closed unit disk and invoke the general properties of the PGF to conclude that these roots must coincide with the $\PGFarg_k$ in the numerator: otherwise $\PGF{X}{\PGFarg}$ would tend to infinity at those points, invalidating the analyticity of the function.

Recall from \cref{remQTF:PGF_analytic_function} that $\PGF{B}{\PGFarg}$ is an analytic function for all $\PGFarg \in \Complex$ satisfying $|\PGFarg| < 1$ and is moreover continuous up to the unit circle. We introduce Rouch\'e's theorem to show that $\PGFarg^s - \PGF{B}{\PGFarg} = 0$ has $s$ roots in the closed unit disk.

\begin{theorem}[Rouch\'e]\label{thm:Rouche}%
Consider a bounded region $\setUncountable{L}$ with continuous boundary $\partial \setUncountable{L}$ and two complex-valued functions $f(\cdot)$ and $g(\cdot)$ that are analytic on $\setUncountable{L}$. If
\begin{equation}%
|f(\PGFarg)| > |g(\PGFarg)|, \quad \PGFarg \in \partial \setUncountable{L},
\end{equation}%
then $f(\cdot)$ and $f(\cdot) + g(\cdot)$ have the same number of zeros in the interior of $\setUncountable{L}$.
\end{theorem}%

When the radius of convergence of $\PGF{B}{\PGFarg}$ exceeds 1, we can prove the following result concerning the number of zeros on and within the unit circle of $\PGFarg^s - \PGF{B}{\PGFarg}$ by using Rouch\'e's theorem.

\begin{lemma}\label{lemQTF:random_walk_location_zeros_PGF}%
Let $\PGF{B}{\PGFarg}$ be a PGF that is analytic in $|\PGFarg| \le 1 + \nu, ~ \nu > 0$. Assume that the condition $\PGFder{B}{1}< s$ for positive recurrence is satisfied. Then the function $\PGFarg^s - \PGF{B}{\PGFarg}$ has exactly $s$ zeros in $|\PGFarg| \le 1$.
\end{lemma}%

\begin{proof}%
Define the functions $f(\PGFarg) \defi \PGFarg^s$ and $g(\PGFarg) \defi - \PGF{B}{\PGFarg}$. Notice that both functions are analytic for $|\PGFarg| \le 1 + \nu$. It is clear that $f(\PGFarg)$ has $s$ roots within the closed unit circle. We aim to show that $|f(\PGFarg)| > |g(\PGFarg)|$ along the circle $|\PGFarg| = 1 + \epsilon$ for $0 < \epsilon < \nu$ so that by Rouch\'e's theorem $f(\cdot) + g(\cdot)$ has $s$ zeros inside the circle $|\PGFarg| = 1 + \epsilon$. Then, finally letting $\epsilon \downarrow 0$ proves the statement.

Observe that $|f(\PGFarg)| = f(|\PGFarg|)$ and $|g(\PGFarg)| = |\PGF{B}{\PGFarg}| \le \PGF{B}{|\PGFarg|}$ by the triangle inequality. So, instead we prove $f(|\PGFarg|) > \PGF{B}{|\PGFarg|}$ for $|\PGFarg| = 1 + \epsilon$. The Taylor series of $f(\PGFarg)$ and $\PGF{B}{|\PGFarg|}$ at $\PGFarg = 1$ evaluated in the point $\PGFarg = 1 + \epsilon$ are
\begin{align}%
f(1 + \epsilon) &= 1 + \epsilon s + \SmallO(\epsilon), \\
\PGF{B}{1 + \epsilon} &= 1 + \epsilon \PGFder{B}{1} + \SmallO(\epsilon).
\end{align}%
From the assumption $\PGFder{B}{1}< s$ and these Taylor expansions we conclude for sufficiently small $\epsilon$ satisfying $0 < \epsilon < \nu$ that $f(1 + \epsilon) > \PGF{B}{1 + \epsilon}$. Letting $\epsilon$ tend to zero yields the proof.
\end{proof}%

Note that the application of \cref{lemQTF:random_walk_location_zeros_PGF} is limited to the class of functions $\PGF{B}{\PGFarg}$ with a radius of convergence larger than 1, so random variables $B$ of which all moments (derivatives of $\PGF{B}{\PGFarg}$ at $\PGFarg = 1$) exist.

$\PGF{X}{\PGFarg}$ is an analytic function for at least all $|\PGFarg| < 1$. However, from \cref{lemQTF:random_walk_location_zeros_PGF} we see that term $\PGFarg^s - \PGF{B}{\PGFarg}$ in the denominator of \eqref{eqnQTF:random_walk_PGF_explicit_expression} approaches zero for $s$ values inside the closed unit disk. An analytic function in the region $|\PGFarg| < 1$ does not have singularities in that region, so at the $s$ values at which $\PGFarg^s - \PGF{B}{\PGFarg} = 0$, the numerator of \eqref{eqnQTF:random_walk_PGF_explicit_expression} must also approach zero. It is clear that one of the $s$ roots is $\PGFarg = 1$ and the other $s - 1$ roots must equal the $\PGFarg_k$ present in the numerator of \eqref{eqnQTF:random_walk_PGF_explicit_expression}.

When $\PGF{B}{\PGFarg}$ is assumed to not equal zero for all $|\PGFarg| \le 1$, we know that the $s$ roots of $\PGFarg^s = \PGF{B}{\PGFarg}$ in $|\PGFarg| \le 1$ satisfy
\begin{equation}%
\PGFarg = \unit \PGF{B}{\PGFarg}^{\frac{1}{s}}, \label{eqnQTF:random_walk_root_k}
\end{equation}%
where $\unit^s = 1$ are the roots of unity. For each unit root $\unit$, \eqref{eqnQTF:random_walk_root_k} can be shown to have a single root inside the closed unit disk $|\PGFarg| \le 1$. One could try to solve \eqref{eqnQTF:random_walk_root_k} by successive substitutions as
\begin{equation}%
\PGFarg_k^{(n + 1)} = \unit_k \PGF{B}{\PGFarg_k^{(n)}}^{\frac{1}{s}}, \quad k = 1,2,\ldots,s, \label{eqnQTF:random_walk_root_k_successive_substitutions}
\end{equation}%
with starting values $\PGFarg_k^{(0)} = 0$ and $\unit_k = \euler^{2\pi \complexunit k/s}$. Under the additional condition that for $|\PGFarg| \le 1$, the derivative $| \frac{\dinf}{\dinf \PGFarg} \PGF{B}{\PGFarg}^{\frac{1}{s}} | < 1$, it can be shown that \eqref{eqnQTF:random_walk_root_k_successive_substitutions} indeed converges to the desired roots $\PGFarg_k$ for $n \to \infty$.

\begin{example}[Poisson distribution]\label{exQTF:random_walk_Poisson_distribution}%
\begin{figure}%
\centering%
\includestandalone{Chapters/QTF/TikZFiles/roots_Szego_curve_random_walk_Poisson_distribution}%
\caption{Using successive substitution to approximate the roots $\PGFarg_k$ in \protect\cref{exQTF:random_walk_Poisson_distribution}.}%
\label{figQTF:roots_random_walk_Poisson_distribution}%
\end{figure}%
Let us assume that $B$ follows a Poisson distribution with rate $\la < s$ so that $\PGF{B}{\PGFarg} = \euler^{-\la(1 - \PGFarg)}$. It is readily seen that $\PGF{B}{\PGFarg}$ does not equal zero anywhere and $| \frac{\dinf}{\dinf \PGFarg} \PGF{B}{\PGFarg}^{\frac{1}{s}} | < 1$ for $|\PGFarg| \le 1$, so that the successive substitutions \eqref{eqnQTF:random_walk_root_k_successive_substitutions} converges to the correct root $\PGFarg_k$. In \cref{figQTF:roots_random_walk_Poisson_distribution} we show for various $\la$ and $s = 10$ the iterates $\PGFarg_k^{(100)}, ~ k = 1,2,\ldots,s$ and the curve on which they lie.
\end{example}%

\begin{remark}[Bounded support]\label{remQTF:random_walk_bounded_support}%
When $B$ has a bounded support, i.e., $B \le s + m$ with $m \ge 1$, we know that $\PGF{B}{\PGFarg}$ is a polynomial of degree $s + m$. From \cref{lemQTF:random_walk_location_zeros_PGF} it immediately follows that $\PGFarg^s = \PGF{B}{\PGFarg}$ has $m$ roots outside the closed unit disk, to be denoted by $\PGFarg_{s + 1},\PGFarg_{s + 2},\ldots,\PGFarg_{s + m}$. Write
\begin{equation}%
\PGFarg^s - \PGF{B}{\PGFarg} = \xi \prod_{k = 1}^{s + m} (\PGFarg - \PGFarg_k)
\end{equation}%
with $\xi$ a constant. Substituting this expression in \eqref{eqnQTF:random_walk_PGF_explicit_expression} provides an alternative expression for $\PGF{X}{\PGFarg}$ in terms of the roots outside the closed unit disk:
\begin{equation}%
\PGF{X}{\PGFarg} = \frac{(s - \PGFder{B}{1})}{\xi \prod_{k = s + 1}^{s + m} (\PGFarg - \PGFarg_k)} \prod_{k = 1}^{s - 1} \frac{1}{1 - \PGFarg_k}.
\end{equation}%
The constant $\xi$ is determined by setting $\PGFarg = 1$ and using $\PGF{X}{1} = 1$, which finally yields
\begin{equation}%
\PGF{X}{\PGFarg} = \prod_{k = s + 1}^{s + m} \frac{1 - \PGFarg_k}{\PGFarg - \PGFarg_k}.
\end{equation}%
This expression is amenable for explicit inversion. In particular, using partial fraction expansion gives
\begin{equation}%
\PGF{X}{\PGFarg} = \prod_{k = s + 1}^{s + m} \frac{1 - \PGFarg_k}{\PGFarg - \PGFarg_k} = \sum_{l = s + 1}^{s + m} \frac{x_l}{\PGFarg - \PGFarg_l}, \label{eqnQTF:random_walk_bounded_support_PGF}
\end{equation}%
where
\begin{equation}%
x_l = \lim_{\PGFarg \to \PGFarg_l} (\PGFarg - \PGFarg_l) \prod_{k = s + 1}^{s + m} \frac{1 - \PGFarg_k}{\PGFarg - \PGFarg_k} = \frac{\prod_{k = s + 1}^{s + m} (1 - \PGFarg_k)}{\prod_{k = s + 1, \, k \neq l}^{s + m} (\PGFarg_l - \PGFarg_k)}.
\end{equation}%
Dividing the numerator and denominator in \eqref{eqnQTF:random_walk_bounded_support_PGF} by $-\PGFarg_l$, we get
\begin{equation}%
\PGF{X}{\PGFarg} = - \sum_{l = s + 1}^{s + m} \frac{x_l}{\PGFarg_l} \frac{1}{1 - \frac{\PGFarg}{\PGFarg_l}} = - \sum_{l = s + 1}^{s + m} \frac{x_l}{\PGFarg_l} \sum_{k \ge 0} \bigl( \frac{\PGFarg}{\PGFarg_l} \bigr)^k, \label{eqnQTF:random_walk_bounded_support_PGF_explicit}
\end{equation}%
which gives
\begin{equation}%
\Prob{X = k} = - \sum_{l = s + 1}^{s + m} \frac{x_l}{\PGFarg_l^{k + 1}}, \quad k \ge 0. \label{eqnQTF:random_walk_bounded_support_eq_dist}
\end{equation}%
For $k$ large enough, the sum on the right-hand side of \eqref{eqnQTF:random_walk_bounded_support_eq_dist} is dominated by the pole of $\PGF{X}{\PGFarg}$ with the smallest modulus, to be denoted without loss of generality by $\PGFarg_{s + 1}$. Omitting all fractions in \eqref{eqnQTF:random_walk_bounded_support_eq_dist} other than the one that corresponds to $\PGFarg_{s + 1}$ gives the following approximation for the tail probabilities:
\begin{equation}%
\Prob{X = k} \approx - x_{s + 1} \bigl( \frac{1}{\PGFarg_{s + 1}} \bigr)^{k + 1}, \quad \textup{as~} k \to \infty.
\end{equation}%

By expressing the PGF of $X$ in terms of the roots outside the closed unit disk, we are able to obtain an explicit product-form solution for the equilibrium distribution.
\end{remark}%

\begin{remark}[Boundary probabilities]%
Armed with the values of $\PGFarg_k$ inside the closed unit disk, we can return to \eqref{eqnQTF:random_walk_numerator_simplification} to construct a linear system of equations for the $s$ unknown boundary probabilities $\Prob{X + B = i}$. In particular, we can substitute $\PGFarg_k$ for $k = 1,2,\ldots,s - 1$ into \eqref{eqnQTF:random_walk_numerator_simplification} to obtain the $s - 1$ equations
\begin{equation}%
\sum_{i = 0}^{s - 1} \Prob{X + B = i}(\PGFarg_k^s - \PGFarg_k^i) = 0, \quad k = 1,2,\ldots,s - 1,
\end{equation}%
which, together with the normalization condition \eqref{eqnQTF:random_walk_numerator_normalization_condition}, constitutes a system of $s$ linear equations for the $s$ unknowns $\Prob{X + B = i}$.
\end{remark}%

\section{Numerical inversion of transforms}%
\label{secQTF:numerical_inversion}%

In some cases it is difficult or even impossible to explicitly retrieve the probability mass function from a PGF or the probability density function from an LST. In this section we describe numerical inversion algorithms that approximate these probability mass and density functions to an arbitrary precision.

%%%%%%%%%%%%%%%%%%%%%%%%%%%%%%%%%%%%%%%%%%%%%%%%%%%%%%%
%%%%%%%%%%%%%%%%%%%%%%%%%%%%%%%%%%%%%%%%%%%%%%%%%%%%%%%
%%%%%%%%%%%%%%%%%%% NEW SUBSECTION %%%%%%%%%%%%%%%%%%%%
%%%%%%%%%%%%%%%%%%%%%%%%%%%%%%%%%%%%%%%%%%%%%%%%%%%%%%%
%%%%%%%%%%%%%%%%%%%%%%%%%%%%%%%%%%%%%%%%%%%%%%%%%%%%%%%

\subsection{Inverting univariate generating functions}%
\label{subsecQTF:numerical_inversion_univariate_PGFs}%

Recall that we denote the PGF by $\PGF{\PGFarg} \defi \sum_{k \ge 0} p(k) \PGFarg^k$, where $\PGFarg$ can be complex-valued, $p(k) \ge 0$ and $\sum_{k \ge 0} p(k) = 1$. To retrieve the probabilities $p(k)$ from $\PGF{\PGFarg}$, we use the fact that $\PGF{\PGFarg}$ is an analytic function for at least all $\PGFarg \in \Complex$ satisfying $|\PGFarg| < 1$ (see \cref{remQTF:PGF_analytic_function}), which allows us to apply the Cauchy contour integral. The Cauchy contour integral reads
\begin{equation}%
p(k) = \frac{1}{2\pi \complexunit} \oint_{C_{\radiusNI(k)}} \frac{\PGF{\PGFarg}}{\PGFarg^{k + 1}} \, \dinf \PGFarg \label{eqnQTF:numerical_inversion_PGF_Cauchy_contour_integral}
\end{equation}%
with $\complexunit$ the complex unit and $C_{\radiusNI(k)}$ a circle of radius $\radiusNI(k) \in (0,1)$ that depends on $k$. We make the change of variables $\PGFarg = \radiusNI(k) \euler^{\pi \theta}$ so that the contour integral \eqref{eqnQTF:numerical_inversion_PGF_Cauchy_contour_integral} can be written as
\begin{equation}%
p(k) = \frac{1}{2\pi \radiusNI(k)^k} \int_0^{2\pi} \PGF{\radiusNI(k) \, \euler^{\complexunit \theta}} \, \euler^{-\complexunit k \theta} \, \dinf \theta.
\end{equation}%
Use $\euler^{-\complexunit \PGFarg} = \cos(\PGFarg) - \complexunit \sin(\PGFarg)$ and $\PGF{\PGFarg} = \RealPart{\PGF{\PGFarg}} + \complexunit \, \ImagPart{\PGF{\PGFarg}}$ to rewrite the integral as
\begin{align}%
p(k) &= \frac{1}{2\pi \radiusNI(k)^k} \int_0^{2\pi} \Bigl[ \bigl(\RealPart{\PGF{\radiusNI(k) \, \euler^{\complexunit \theta}}} + \complexunit \, \ImagPart{\PGF{\radiusNI(k) \, \euler^{\complexunit \theta}}} \bigr) \notag \\
&\hphantom{= \frac{1}{2\pi \radiusNI(k)^k} \int_0^{2\pi}} \cdot \bigl( \cos(k \theta) - \complexunit \sin(k \theta) \bigr) \Bigr] \, \dinf \theta. \notag \\
&= \frac{1}{2\pi \radiusNI(k)^k} \Bigl[ \int_0^{2\pi} \bigl( \cos(k \theta) \RealPart{\PGF{\radiusNI(k) \, \euler^{\complexunit \theta}}} + \sin(k \theta) \ImagPart{\PGF{\radiusNI(k) \, \euler^{\complexunit \theta}}} \bigr) \, \dinf \theta \notag \\
&\hspace*{-1ex} + \complexunit \! \int_0^{2\pi} \bigl( \cos(k \theta) \ImagPart{\PGF{\radiusNI(k) \, \euler^{\complexunit \theta}}} - \sin(k \theta) \RealPart{\PGF{\radiusNI(k) \, \euler^{\complexunit \theta}}} \bigr) \, \dinf \theta \Bigr]. \label{eqnQTF:numerical_inversion_p(k)_in_terms_of_two_integrals}
\end{align}%
The last integral in \eqref{eqnQTF:numerical_inversion_p(k)_in_terms_of_two_integrals} equals zero because $\cos(\cdot)$ is an even function, $\sin(\cdot)$ is an odd function, $\ImagPart{\PGF{\PGFarg}} = - \ImagPart{\PGF{\bar{\PGFarg}}}$ and $\RealPart{\PGF{\PGFarg}} = \RealPart{\PGF{\bar{\PGFarg}}}$, where $\bar{\PGFarg}$ is the complex conjugate of $\PGFarg$.

It remains to determine the other integral in \eqref{eqnQTF:numerical_inversion_p(k)_in_terms_of_two_integrals}. We follow the approach outlined in \cite{Abate1992_Inversion_pgf}, which ultimately leads to an approximation $\approximate{p}(k)$ and a bound on the error $\approximate{e}(k)$, see \cite[Theorem~1]{Abate1992_Inversion_pgf}. We can use the trapezoidal rule to approximate the integral. If we use a step size of $\pi/k$, then we can write
\begin{equation}%
p(k) \approx \approximate{p}(k) = \frac{1}{2 k \radiusNI(k)^k} \sum_{l = 1}^{2k} (-1)^l \RealPart{\PGF{\radiusNI(k) \, \euler^{\complexunit \pi \frac{l}{k} }}}.
\end{equation}%
By using the inherent symmetry, we finally arrive at the following expression for the approximation, for $k \ge 1$,
\begin{align}%
\approximate{p}(k) = \frac{1}{2 k \radiusNI(k)^k} \Bigl(& \PGF{\radiusNI(k)} + (-1)^k \PGF{-\radiusNI(k)} \notag \\
&+ 2 \sum_{l = 1}^{k - 1} (-1)^l \RealPart{\PGF{\radiusNI(k) \, \euler^{\complexunit \pi \frac{l}{k}}}} \Bigr), \label{eqnQTF:univariate_PGF_numerical_inversion}
\end{align}%
where $\radiusNI(k) \in (0,1)$ is actually a tunable parameter that controls the error term $\approximate{e}(k) = p(k) - \approximate{p}(k)$, since
\begin{equation}%
|\approximate{e}(k)| \le \frac{\radiusNI(k)^{2k}}{1 - \radiusNI(k)^{2k}} \approx r(k)^{2k}. \label{eqnQTF:univariate_PGF_numerical_inversion_error_bound}
\end{equation}%
The approximate equality is valid if $r(k)^{2k}$ is small. Observe that $p(0)$ does not need to be approximated, since it easily follows from $p(0) = \PGF{0}$. With $\radiusNI(k) = 10^{-d/(2k)}$ we find that $|\approximate{e}(k)| \le 10^{-d}$ and therefore the approximation $\approximate{p}(k)$ in \eqref{eqnQTF:univariate_PGF_numerical_inversion} is accurate until at least the $d$-th decimal.

For reference in later chapters, we present in full the algorithm to numerically invert PGFs.

\begin{algorithm}%
\caption{Numerical inversion univariate PGF}%
\label{algQTF:numerical_inversion_univariate_PGF}%
\begin{algorithmic}[1]%
\State Input $\PGF{\PGFarg}$
\State Decide for which $k \ge 1$ you wish to approximate $p(k)$
\State Decide on the minimum number of correct decimals $d$
\State Set $\radiusNI(k) = 10^{-d/(2k)}$
\State Initialize
    \begin{equation}%
    \approximate{p}(k) = \frac{1}{2 k \radiusNI(k)^k} \bigl( \PGF{\radiusNI(k)} + (-1)^k \PGF{-\radiusNI(k)} \bigr)
    \end{equation}%
\For{$l = 1,2,\ldots,k - 1$} \Comment{Skip the for loop if $k = 1$}
    \State Set $\PGFarg_l = \radiusNI(k) \, \euler^{\complexunit \pi l/k}$
    \State Update
        \begin{equation}%
        \approximate{p}(k) = \approximate{p}(k) + \frac{1}{k \radiusNI(k)^k} (-1)^l \RealPart{\PGF{\PGFarg_l}}
        \end{equation}%
\EndFor
\State Approximate $p(k)$ as $p(k) \approx \approximate{p}(k)$
\end{algorithmic}%
\end{algorithm}%

\begin{example}[Gamma distributed service times]\label{exQTF:MG1_numerical_inversion_PGF}%
Consider the $M/G/1$ system with arrival rate $\la$ and service times $B$ that are distributed according to a gamma distribution with shape parameter $\al > 0$ and rate parameter $\be > 0$. Specifically, the probability density function of $B$ is given by
\begin{equation}%
f_B(t) = \frac{\be^\al t^{\al - 1}}{\Gamma(\al)} \euler^{-\be t}, \quad t \ge 0,
\end{equation}%
where $\Gamma(\al)$ is the complete gamma function. The mean is given by
\begin{equation}%
\E{B} = \frac{\al}{\be}
\end{equation}%
and the LST is
\begin{equation}%
\LST{B}{\LSTarg} = \bigl( \frac{\be}{\be + \LSTarg} \bigr)^\al.
\end{equation}%
The Pollaczek-Khinchin formula \eqref{eqnQTF:MG1_PGF_explicit} says that the PGF $\PGF{\PGFarg}$ of the equilibrium number of jobs in the system can be calculated from
\begin{equation}%
\PGF{\PGFarg} = (1 - \rho) \frac{(1 - \PGFarg) \bigl( \frac{\be}{\be + \la(1 - \PGFarg)} \bigr)^\al}{\bigl( \frac{\be}{\be + \la(1 - \PGFarg)} \bigr)^\al - \PGFarg},
\end{equation}%
where $\rho = \la \E{B}$. It is not immediate how we can explicitly invert this expression to obtain the equilibrium probabilities $p(k)$, especially if $\al$ is not an integer. To demonstrate the numerical inversion algorithm, we take $\al = 2\sqrt{2}$ and $\be = \sqrt{2}$ and invert the PGF to derive the equilibrium distribution. Notice that the load is given by $\rho = 2\la$. We select $d = 8$ in \cref{algQTF:numerical_inversion_univariate_PGF} and obtain for various values of $\la$ the equilibrium distribution, see \cref{figQTF:numerical_inversion_PGF_MG1_gamma_distribution}.
\begin{figure}
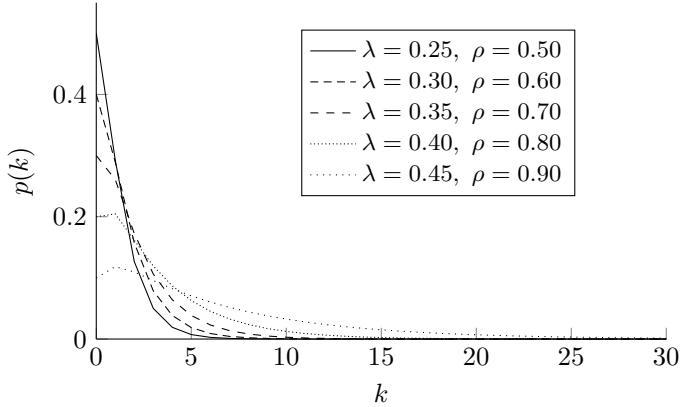
%
\centering%
\includestandalone{Chapters/QTF/TikZFiles/numerical_inversion_PGF_MG1_gamma_distribution}%
\caption{Equilibrium distribution of the $M/G/1$ system of \protect\cref{exQTF:MG1_numerical_inversion_PGF} with gamma distributed services times and varying arrival rate.}%
\label{figQTF:numerical_inversion_PGF_MG1_gamma_distribution}
\end{figure}%
%
%From \cref{figQTF:numerical_inversion_PGF_MG1_gamma_distribution} we see that if the arrival rate $\la$ is small, then the equilibrium distribution resembles a geometric distribution. However, when $\la$ increases, this resemblance becomes less and less. Nonetheless, the tail of the distribution does seem to remain geometrically distributed, where the first few equilibrium probabilities deviate from this geometric behavior. We also notice that if $\la$ increases, the geometric decay in the tail becomes smaller, it can be said that the tail becomes `fatter'.
\end{example}%

%%%%%%%%%%%%%%%%%%%%%%%%%%%%%%%%%%%%%%%%%%%%%%%%%%%%%%%
%%%%%%%%%%%%%%%%%%%%%%%%%%%%%%%%%%%%%%%%%%%%%%%%%%%%%%%
%%%%%%%%%%%%%%%%%%% NEW SUBSECTION %%%%%%%%%%%%%%%%%%%%
%%%%%%%%%%%%%%%%%%%%%%%%%%%%%%%%%%%%%%%%%%%%%%%%%%%%%%%
%%%%%%%%%%%%%%%%%%%%%%%%%%%%%%%%%%%%%%%%%%%%%%%%%%%%%%%

\subsection{Inverting bivariate generating functions}%
\label{subsecQTF:numerical_inversion_bivariate_PGFs}%

A \textit{bivariate} PGF is a PGF of the joint probability mass function of two random variables and therefore takes two arguments. We encounter bivariate PGFs in some of the more advanced chapters, where we would like to numerically invert them. So, we present a numerical inversion algorithm for PGFs of two variables. The bivariate PGF is defined as
\begin{equation}%
\PGF{x,y} \defi \sum_{k \ge 0} \sum_{l \ge 0} p(k,l) x^k y^l, \label{eqnQTF:bivariate_PGF}
\end{equation}%
where $x$ and $y$ can be complex-valued, $p(k,l) \ge 0$ and $\sum_{k,l \ge 0} p(k,l) = 1$. The bivariate PGF satisfies $\PGF{1,1} = 1$ and converges for at least all $|x|,|y| \le 1$ and is therefore analytic for at least all $x,y \in \Complex$ satisfying $|x|,|y| < 1$.

One of the standard numerical inversion algorithms is described in \cite[Section~3]{Choudhury1994_Multidimensional_transform_inversion}. Here we present a version of that algorithm with specific parameter choices so that it resembles the univariate case. The algorithm approximates $p(k,l)$ by
\begin{equation}%
p(k,l) = \approximate{p}(k,l) - \approximate{e}(k,l).
\end{equation}%
The approximation is given by
\begin{align}%
\approximate{p}(k,l) &= \frac{1}{4 j_1 j_2 \radiusNI_1(k)^k \radiusNI_2(l)^l} \notag \\
&\quad \cdot \sum_{m = - j_1}^{j_1 - 1} \sum_{n = - j_2}^{j_2 - 1} \euler^{- \complexunit \pi (k \frac{m}{j_1} + l \frac{n}{j_2})} \PGF{\radiusNI_1(k) \, \euler^{\complexunit \pi \frac{m}{j_1}}, \radiusNI_2(l) \, \euler^{\complexunit \pi \frac{n}{j_2}}}, \label{eqnQTF:bivariate_PGF_numerical_inversion}
\end{align}%
where $j_1,j_2 \in \Nat$, and $0 < \radiusNI_1(k),\radiusNI_2(l) < 1$ are tunable parameters that control the error:
\begin{align}%
|\approximate{e}(k,l)| &\le \frac{\radiusNI_1(k)^{2 j_1} + \radiusNI_2(l)^{2 j_2} - \radiusNI_1(k)^{2 j_1} \radiusNI_2(l)^{2 j_2}}{(1 - \radiusNI_1(k)^{2 j_1})(1 - \radiusNI_2(l)^{2 j_2})} \notag \\
&\approx \radiusNI_1(k)^{2 j_1} + \radiusNI_2(l)^{2 j_2}, \label{eqnQTF:bivariate_PGF_numerical_inversion_error_bound}
\end{align}%
where the approximate equality is a valid approximation if both $\radiusNI_1(k)^{2 j_1}$ and $\radiusNI_2(l)^{2 j_2}$ are small. When we are interested in $p(k,l)$ for $k,l \ge 1$, then we can set $j_1 = k$ and $j_2 = l$ to simplify the approximation and the bound on the error term. Moreover, if we then choose $\radiusNI_1(k) = 10^{-d/(2k)} / 2$ and $\radiusNI_2(l) = 10^{-d/(2l)} / 2$, then the resulting approximation is accurate until at least the $d$-th decimal.

\cref{algQTF:numerical_inversion_bivariate_PGF} summarizes the numerical scheme for inverting bivariate PGFs.

\begin{algorithm}%
\caption{Numerical inversion bivariate PGF}%
\label{algQTF:numerical_inversion_bivariate_PGF}%
\begin{algorithmic}[1]%
\State Input $\PGF{x,y}$
\State Decide for which $k,l \ge 0$ you wish to approximate $p(k,l)$
\State Pick $j_1,j_2 \in \Nat$ and $0 < \radiusNI_1(k),\radiusNI_2(l) < 1$
\State Initialize $\approximate{p}(k,l) = 0$
\For{$m = -j_1,-j_1 + 1,\ldots,j_1 - 1$}
    \State Set $x_m = \radiusNI_1(k) \, \euler^{\complexunit \pi m/j_1}$
    \For{$n = -j_2,-j_2 + 1,\ldots,j_2 - 1$}
        \State Set $y_n = \radiusNI_2(l) \, \euler^{\complexunit \pi n/j_2}$
        \State Update
            \begin{equation}%
            \approximate{p}(k,l) = \approximate{p}(k,l) + \euler^{- \complexunit \pi (k \frac{m}{j_1} + l \frac{n}{j_2})} \PGF{x_m,y_n}
            \end{equation}%
    \EndFor
\EndFor
\State Normalize
    \begin{equation}%
    \approximate{p}(k,l) = \frac{\approximate{p}(k,l)}{4 j_1 j_2 \radiusNI_1(k)^k \radiusNI_2(l)^l}
    \end{equation}%
\State Approximate $p(k,l)$ as $p(k,l) \approx \approximate{p}(k,l)$
\end{algorithmic}%
\end{algorithm}%

%%%%%%%%%%%%%%%%%%%%%%%%%%%%%%%%%%%%%%%%%%%%%%%%%%%%%%%
%%%%%%%%%%%%%%%%%%%%%%%%%%%%%%%%%%%%%%%%%%%%%%%%%%%%%%%
%%%%%%%%%%%%%%%%%%% NEW SUBSECTION %%%%%%%%%%%%%%%%%%%%
%%%%%%%%%%%%%%%%%%%%%%%%%%%%%%%%%%%%%%%%%%%%%%%%%%%%%%%
%%%%%%%%%%%%%%%%%%%%%%%%%%%%%%%%%%%%%%%%%%%%%%%%%%%%%%%

\subsection{Inverting univariate Laplace-Stieltjes transforms}%
\label{subsecQTF:numerical_inversion_univariate_LSTs}%

Most of the continuous random variables that we consider in this book are non-negative and have a continuous probability density function. With these characteristics the LST is given by
\begin{equation}%
\LST{\LSTarg} \defi \int_0^\infty \euler^{-\LSTarg t} f(t) \, \dinf t, \quad \RealPart{\LSTarg} > 0, \label{eqnQTF:univariate_LST}
\end{equation}%
where $f(\cdot)$ is a probability density function that we often wish to retrieve from $\LST{\cdot}$. An integral formula for the inverse Laplace transform called the Bromwich integral provides an expression for $f(t)$ in terms of a contour integral:
\begin{equation}%
f(t) = \frac{1}{2\pi \complexunit} \oint_{C_r} \euler^{\LSTarg t} \LST{\LSTarg} \, \dinf \LSTarg, \label{eqnQTF:numerical_inversion_LST_Bromwich_integral}
\end{equation}%
where $C_r$ is the vertical line in the complex plane with constant real part equal to $r$. The value of $r$ must be chosen such that all singularities of $\LST{\cdot}$ are to the left of the vertical line. Since we are dealing with LSTs, we can safely pick any positive value for $r$. Notice that \eqref{eqnQTF:numerical_inversion_LST_Bromwich_integral} establishes that an LST uniquely defines the underlying probability distribution function.

One of the standard inversion algorithm for LSTs is called the Euler method and is presented in \cite[Section~1]{Abate1995_LT_inversion_probability_distributions}. The derivation of the approximation resembles the derivation of the approximation for the univariate PGF presented in \cref{subsecQTF:numerical_inversion_univariate_PGFs}, so we omit it here. The algorithm approximates $f(t)$ by $\approximate{f}(t)$. To construct the approximation $\approximate{f}(t)$ we require the definition
\begin{equation}%
s_n(t) \defi \frac{\euler^{\discretizationError/2}}{2t} \RealPartLargerBrackets{L\bigl( \frac{\discretizationError}{2t} \bigr)} + \frac{\euler^{\discretizationError/2}}{t} \sum_{k = 1}^n (-1)^k \RealPartLargerBrackets{L\bigl( \frac{\discretizationError}{2t} + \frac{\complexunit \pi k}{t} \bigr)}, \label{eqnQTF:univariate_LST_definition_s_n(t)}
\end{equation}%
where we still need to choose $\discretizationError$. In \cite[Equation~(13)]{Abate1995_LT_inversion_probability_distributions} it is explained that $s_n(t)$ is an approximation of a more accurate infinite series expression for $f(t)$ by truncating the infinite series to $n$ terms. By increasing $n$ in \eqref{eqnQTF:univariate_LST_definition_s_n(t)}, the accuracy of the approximation increases. Euler summation can be used to accelerate convergence of the approximation (to get more accurate results with fewer computations):
\begin{equation}%
\approximate{f}(t) = \sum_{k = 0}^m \binom{m}{k} 2^{-m} s_{n + k}(t). \label{eqnQTF:bivariate_LST_numerical_inversion}
\end{equation}%
Since $\sum_{k = 0}^m \binom{m}{k} 2^{-m} = 1$ and the summands are positive, we see that $\approximate{f}(t)$ is the weighted average of the terms $s_n(t),s_{n + 1}(t),\ldots,s_{n + m}(t)$. More specifically, it is the \textit{binomial} average of those terms, since the weights are in terms of binomial coefficients.

It still remains to choose $\discretizationError$, $m$ and $n$. Typically, $m = 11$ and $n = 15$ produce accurate results. If more accurate results are required, the value of $n$ can be increased, but $m$ can usually remain fixed. There are various types of errors that decrease the quality of the approximation. One of those errors is the discretization error, which occurs when we replace an integral by a series, as was done here. The value of $\discretizationError$ directly influences the magnitude of this discretization error $\approximate{e}_{\textup{d}}(t)$, since
\begin{equation}%
|\approximate{e}_{\textup{d}}(t)| \le \frac{\euler^{-\discretizationError}}{1 - \euler^{-\discretizationError}} \approx \euler^{-\discretizationError},
\end{equation}%
where the approximate equality holds if $\euler^{-\discretizationError}$ is small. If we choose $\discretizationError$ too large, then we can run into computational difficulties, such as loss of significant digit, or roundoff errors. There is no exact error bound on the approximation \eqref{eqnQTF:bivariate_LST_numerical_inversion}, but in most cases, we can select $\discretizationError = d \log{10}$ to get $d - 1$ correct decimals. We often select $d = 8$ and use $\discretizationError = 8 \log{10} \approx 18.4$.

For reference in the following chapters, we present in full the algorithm to numerically invert univariate LSTs.

\begin{algorithm}%
\caption{Numerical inversion univariate LST}%
\label{algQTF:numerical_inversion_univariate_LST}%
\begin{algorithmic}[1]%
\State Input $\LST{\LSTarg}$
\State Decide for which $t > 0$ you wish to approximate $f(t)$
\State Pick $m,n \in \Nat$ and $\discretizationError \in \Real_+$ \Comment{$m = 11$, $n = 15$ and $\discretizationError = 18.4$ works well}
\For{$k = 0,1,\ldots,m + n$}
    \State Calculate $\LST{ \discretizationError/(2t) + \complexunit \pi k/t }$
\EndFor
\State Compute $s_n(t)$ from \eqref{eqnQTF:univariate_LST_definition_s_n(t)}
\For{$k = 1,2,\ldots,m$}
    \State Compute
        \begin{equation}%
        s_{n + k}(t) = s_{n + k - 1}(t) + \frac{\euler^{\discretizationError/2}}{t} (-1)^{n + k} \RealPartLargerBrackets{\LSTbig{ \frac{\discretizationError}{2t} + \frac{\complexunit \pi (n + k)}{t} }}
        \end{equation}%
\EndFor
\State Compute $\approximate{f}(t)$ using the binomial average \eqref{eqnQTF:bivariate_LST_numerical_inversion}
\State Approximate $f(t)$ as $f(t) \approx \approximate{f}(t)$
\end{algorithmic}%
\end{algorithm}%

The inversion algorithm also works for distributions that have discontinuities, but the results might be distorted due to some oscillations around the points of discontinuity. By increasing the accuracy of the method by, e.g., increasing $m$ and $n$ in \cref{algQTF:numerical_inversion_univariate_LST}, one can damp these oscillations. We treat an example to show how this works in practice.

\begin{example}[Uniform services]\label{exQTF:MG1_numerical_inversion_LST}%
Consider the $M/G/1$ system with arrival rate $\la = 0.35$ and service times $B$ that are distributed according to a uniform distribution on the interval $[1,3]$ and mean 2. Specifically, the probability density function of $B$ is given by
\begin{equation}%
f_B(t) = \frac{1}{2}, \quad t \in [1,3]
\end{equation}%
and the LST is
\begin{equation}%
\LST{B}{\LSTarg} = \frac{\euler^{-\LSTarg} - \euler^{-3\LSTarg}}{2\LSTarg}.
\end{equation}%

The LST of the sojourn time of an $M/G/1$ queue is given in \eqref{eqnQTF:MG1_LST_sojourn_time_explicit} and is in this case
\begin{equation}%
\LST{S}{\LSTarg} = (1 - \rho) \frac{\euler^{-\LSTarg} - \euler^{-3\LSTarg}}{2\LSTarg} \frac{\LSTarg}{\la \frac{\euler^{-\LSTarg} - \euler^{-3\LSTarg}}{2\LSTarg} + \LSTarg - \la},
\end{equation}%
where $\rho = \la \E{B} = 0.7$. Explicitly inverting this LST to obtain $f_S(\cdot)$ proves to be difficult due to the exponential functions. We therefore turn to the numerical inversion techniques presented in \cref{algQTF:numerical_inversion_univariate_LST}. We will see that the uniform service time distribution causes numerical inaccuracies due to the discontinuities of $f_B(t)$ at $t = 1$ and $t = 3$. For the algorithm settings, we will fix $\discretizationError = 18.4$ and show the influence of $m$ and $n$.

\begin{figure}
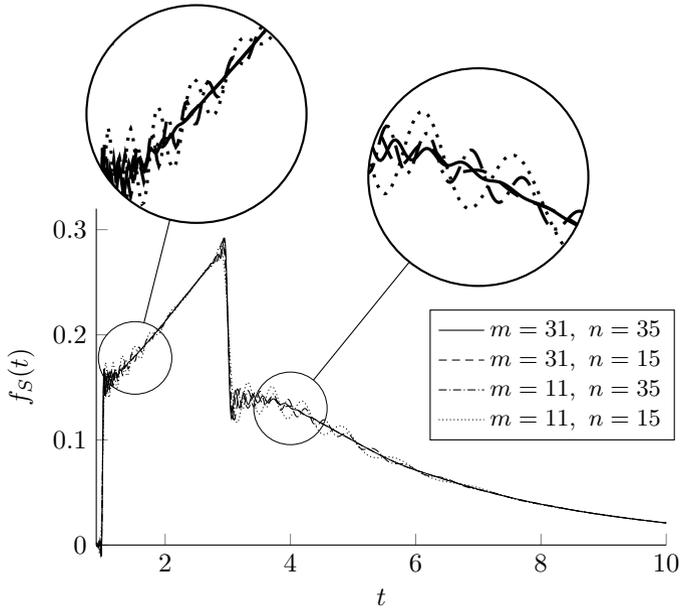
%
\centering%
\includestandalone{Chapters/QTF/TikZFiles/numerical_inversion_LST_MG1_uniform_distribution}%
\caption{Probability density function of the sojourn time of the $M/G/1$ queue of \protect\cref{exQTF:MG1_numerical_inversion_LST} with uniformly distributed services times for various inputs $m$ and $n$ of \protect\cref{algQTF:numerical_inversion_univariate_LST}.}%
\label{figQTF:numerical_inversion_LST_MG1_uniform_distribution}
\end{figure}%

\begin{table}%
\centering%
\begin{tabular}{cc|ccc}%
    &    & \multicolumn{3}{c}{$m$} \\
    &    & 11 & 21 & 31 \\
\hline
    & 15 & 6.001 & 8.803 & 11.61 \\
$n$ & 25 & 7.657 & 10.48 & 13.31 \\
    & 35 & 9.273 & 12.03 & 14.67
\end{tabular}%
\caption{Computation times (in seconds) required to numerically invert the LST of the sojourn time using \protect\cref{algQTF:numerical_inversion_univariate_LST} for the queueing system and algorithm settings described in \protect\cref{exQTF:MG1_numerical_inversion_LST} and varying $m$ and $n$.}%
\label{tblQTF:computation_times_numerical_inversion_LST_MG1_uniform_distribution}%
\end{table}%

\cref{figQTF:numerical_inversion_LST_MG1_uniform_distribution} shows that at the points of discontinuity, the approximation obtained from \cref{algQTF:numerical_inversion_univariate_LST} oscillates. This oscillations is damped when the values of $m$ and $n$ increase. It is important that the inverted function is checked for irregularities such as the one we encounter now. In \cref{tblQTF:computation_times_numerical_inversion_LST_MG1_uniform_distribution} we display the time required to compute $f_S(t)$ for each $t$ from 0.9 to 10 in steps of size 0.001 (so 9101 times) for each combination of algorithm settings.
\end{example}%

%%%%%%%%%%%%%%%%%%%%%%%%%%%%%%%%%%%%%%%%%%%%%%%%%%%%%%%
%%%%%%%%%%%%%%%%%%%%%%%%%%%%%%%%%%%%%%%%%%%%%%%%%%%%%%%
%%%%%%%%%%%%%%%%%%% NEW SUBSECTION %%%%%%%%%%%%%%%%%%%%
%%%%%%%%%%%%%%%%%%%%%%%%%%%%%%%%%%%%%%%%%%%%%%%%%%%%%%%
%%%%%%%%%%%%%%%%%%%%%%%%%%%%%%%%%%%%%%%%%%%%%%%%%%%%%%%

\section{Takeaways}%
\label{secQTF:takeaways}%

Transforms are powerful tools that can simplify and facilitate calculating with distributions. Transforms enjoy the property that they uniquely characterize probability distributions. Once the PGF or LST of a random variable is known, all moments and the probability distribution often readily follows. Another advantage of transforms, of particular use in this book, is that an infinite system of linear balance equations can be converted into a single functional equation for the PGF; see \cref{subsecQTF:simple_queue,subsecQTF:MG1_departure_distribution}.

Transforms need to be inverted. This can be done by differentiation or  integration. Both methods can be useful and will be applied in later chapters. Sometimes a PGF can be written in the form of an infinite sum involving powers of $\PGFarg$. In those cases, the coefficients of $\PGFarg^k$ together constitute the probability mass function. %A similar observation can be made for a characterization of an LST in terms of an integral involving $\euler^{-\LSTarg t} \, \dinf t$, but we do not encounter them throughout this book.

In this chapter we have embedded the $M/G/1$ queue at departure instants and the $G/M/1$ queue at arrival instants. Both approaches lead to a state space $\Nat_0$ with a particular transition structure for each queue. The linear systems of balance equations associated with these embedded Markov chains are amenable to transform analysis and lead to some canonical relations such as the Pollaczek-Khinchin formula and distributional Little's law. The embedding technique is not restricted to the $M/G/1$ or $G/M/1$ queue and can be used for many stochastic models.

The embedded Markov chains associated with the $M/G/1$ and $G/M/1$ system are skip-free to the left and right, respectively. In \cref{ch:skip-free_one_direction} we introduce processes that also possess the skip-free property, but each state is replaced by a finite set of states. For the skip-free to the right variant of these processes, the transform analysis that was used in this chapter can be extended to determine the equilibrium distribution. For the other variant we turn to matrix-analytic methods.

In this chapter we have encountered various product-form solutions. For the Erlang service time distribution, \eqref{eqnQTF:MG1_example_d(i)} shows that the departure distribution of the $M/G/1$ queue has a product-form solution. If we are able to write the PGF of the departure distribution in an $M/G/1$ as a ratio of polynomial without any common roots, then the departure distribution is given by a sum of product-form solutions, see \eqref{eqnQTF:MG1_partial_fraction_decomposition_P-F_solution}. For any inter-arrival time distribution, the arrival distribution is given by the product-form solution \eqref{eqnQTF:GM1_P-F_solution}. In case of bounded jumps in both directions in the random walk setting, we find the product-form solution \eqref{eqnQTF:random_walk_bounded_support_eq_dist} for the equilibrium distribution.

%%%%%%%%%%%%%%%%%%%%%%%%%%%%%%%%%%%%%%%%%%%%%%%%%%%%%%%
%%%%%%%%%%%%%%%%%%%%%%%%%%%%%%%%%%%%%%%%%%%%%%%%%%%%%%%
%%%%%%%%%%%%%%%%%%%%%%%% NOTES %%%%%%%%%%%%%%%%%%%%%%%%
%%%%%%%%%%%%%%%%%%%%%%%%%%%%%%%%%%%%%%%%%%%%%%%%%%%%%%%
%%%%%%%%%%%%%%%%%%%%%%%%%%%%%%%%%%%%%%%%%%%%%%%%%%%%%%%

%\theendnotes%
%\setcounter{endnote}{0}
\printendnotes%
%

%%%%%%%%%%%%%%%%%%%%%%%%%%%%%%%%%%%%%%%%%%%%%%%%%%%%%%%
%%%%%%%%%%%%%%%%%%%%%%%%%%%%%%%%%%%%%%%%%%%%%%%%%%%%%%%
%%%%%%%%%%%%%%%%%%%%%%% NEW PART %%%%%%%%%%%%%%%%%%%%%%
%%%%%%%%%%%%%%%%%%%%%%%%%%%%%%%%%%%%%%%%%%%%%%%%%%%%%%%
%%%%%%%%%%%%%%%%%%%%%%%%%%%%%%%%%%%%%%%%%%%%%%%%%%%%%%%

\part{Basic processes}%

% Checked points 1-7 and a-i
\chapter{Birth--and--death processes}%
\label{ch:birth--and--death_processes}%

In this chapter we introduce a structured class of Markov processes called the birth--and--death processes. This structure allows for local balance equations to be used in the derivation of the equilibrium distribution.

%%%%%%%%%%%%%%%%%%%%%%%%%%%%%%%%%%%%%%%%%%%%%%%%%%%%%%%
%%%%%%%%%%%%%%%%%%%%%%%%%%%%%%%%%%%%%%%%%%%%%%%%%%%%%%%
%%%%%%%%%%%%%%%%%%%%% NEW SECTION %%%%%%%%%%%%%%%%%%%%%
%%%%%%%%%%%%%%%%%%%%%%%%%%%%%%%%%%%%%%%%%%%%%%%%%%%%%%%
%%%%%%%%%%%%%%%%%%%%%%%%%%%%%%%%%%%%%%%%%%%%%%%%%%%%%%%

\section{General birth--and--death processes}%
\label{secBD:birth--and--death_processes}%

We start by defining the birth--and--death process.

\begin{definition}%
A {\itshape birth--and--death \textup{(BD)} process} is a Markov process on the state space $\statespace = \{0,1,\ldots,S\}$ with $S$ possibly infinite, where transitions are between adjacent states: from state $i$ to state $i + 1$ (a \textit{birth}) and to state $i - 1$ (a \textit{death}).
\end{definition}%
Unless stated otherwise, we focus on BD processes that have an infinite state space $\statespace = \Nat_0$ and all transition rates are strictly positive, leading to an irreducible Markov process. Birth rates are commonly denoted as $\la_i$ and death rates as $\mu_i$. This leads to the following transition rate matrix of the BD process:
\begin{equation}\label{eqnBD:Q_infinite_state_space}%
Q = \begin{bmatrix}%
-\la_0 & \la_0 \\
\mu_1  & -(\la_1 + \mu_1) & \la_1 \\
       & \mu_2            & -(\la_2 + \mu_2) & \la_2 \\
       &                  & \mu_3            & -(\la_3 + \mu_3) & \la_3 \\
       &                  &                  & \ddots           & \ddots & \ddots
\end{bmatrix},%
\end{equation}%
where unspecified elements are zero. A BD process with rates $\la_i = \la$ and $\mu_i = \mu$ is called \textit{homogeneous} and $\textit{inhomogeneous}$ otherwise. The transition rate diagram of the BD process is depicted in \cref{figBD:BD_infinite_state_space}.

\begin{figure}%
\centering%
\includestandalone{Chapters/BD/TikZFiles/transition_rate_diagram_infinite_state_space}%
\caption{A BD process on the state space $\statespace = \Nat_0$.}%
\label{figBD:BD_infinite_state_space}
\end{figure}%

The sojourn time in state $i$ is the minimum of the time to transit to state $i + 1$ and the time to transit to state $i - 1$. Since both of these times are exponentially distributed, the time spent in state $i$ until a transition occurs is exponentially distributed with parameter $\la_i + \mu_i$. Given that a transition occurs, we have a birth with probability $\la_i/(\la_i + \mu_i)$ or a death with probability $\mu_i/(\la_i + \mu_i)$.

The above reasoning indicates that for simulation purposes one needs to repeat these two steps: sample a sojourn time and flip a biased coin to determine to which state the process transitions. This simple procedure is summarized in \cref{algBD:simulation}.

\begin{algorithm}%
\caption{Simulation of a BD process (with $\mu_0 = 0$).}%
\label{algBD:simulation}%
\begin{algorithmic}[1]%
\State Input $t_{\textup{max}}$, $X(0)$, and birth and death rates
\State $t = 0$
\While{$t < t_{\textup{max}}$}
    \State Sample the sojourn time $t^*$ in state $X(t)$ from $\Exp{\la_{X(t)} + \mu_{X(t)}}$
    \State Sample $d$ from $\Ber{\mu_{X(t)}/(\la_{X(t)} + \mu_{X(t)})}$
    \State $X(t + t^*) = X(t) + (-1)^d$
    \State $t = t + t^*$
\EndWhile
\end{algorithmic}%
\end{algorithm}%

Using \cref{algBD:simulation}, we simulate one sample path each for three different homogeneous BD processes. These sample paths are depicted in \cref{figBD:simulation}. Notice that for $\la - \mu < 0$ the process seems to have a drift towards zero. On the other hand, if $\la - \mu > 0$, $X(t)$ seems to increase as time passes. For the case $\la - \mu = 0$ no clear conjectures can be made. Intuitively these three statements make sense, $\la$ is the rate at which the process transitions upwards and $\mu$ is the rate at which the process transitions downwards. So if $\la > \mu$ there is a net rate upwards and vice versa for $\la < \mu$. We formalize this intuition and extend it to inhomogeneous BD processes in \cref{secBD:equilibrium_distribution}. We will see that this net rate decides if the Markov process is transient or recurrent.

\begin{figure}
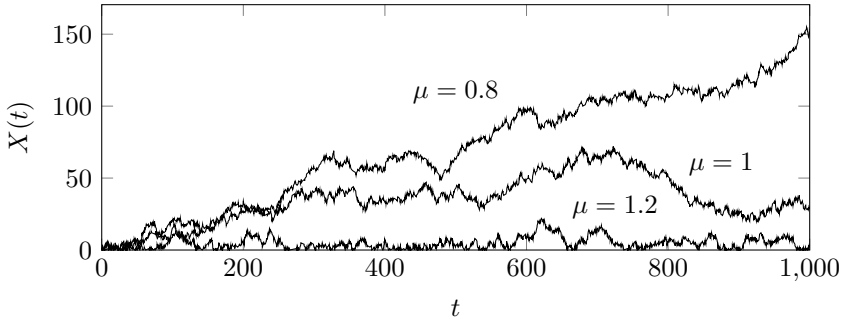
%
\centering%
\includestandalone{Chapters/BD/TikZFiles/simulation_BD_processes}%
\caption{Sample paths of a BD process with $\la_i = 1$, $\mu_i = \mu$ and $X(0) = 0$.}%
\label{figBD:simulation}%
\end{figure}%

%%%%%%%%%%%%%%%%%%%%%%%%%%%%%%%%%%%%%%%%%%%%%%%%%%%%%%%
%%%%%%%%%%%%%%%%%%%%%%%%%%%%%%%%%%%%%%%%%%%%%%%%%%%%%%%
%%%%%%%%%%%%%%%%%%%%% NEW SECTION %%%%%%%%%%%%%%%%%%%%%
%%%%%%%%%%%%%%%%%%%%%%%%%%%%%%%%%%%%%%%%%%%%%%%%%%%%%%%
%%%%%%%%%%%%%%%%%%%%%%%%%%%%%%%%%%%%%%%%%%%%%%%%%%%%%%%

\section{Time-dependent behavior}%
\label{secBD:time_dependent}%

Analyzing time-dependent behavior of BD processes is difficult. Explicit expressions for the transition functions
\begin{equation}%
p_{i,j}(t) \defi \Prob{ X(t) = j \mid X(0) = i}
\end{equation}%
exist, but only for special cases and often involve special functions related to orthogonal polynomials (see Karlin and McGregor \cite{Karlin1959_Random_walks} and Karlin and Taylor \cite{Karlin1975_First_course_stochastic_processes}). Nonetheless, we review some of the techniques used.

The transition functions satisfy both the Kolmogorov forward and backward equation, see \cref{thmMP:Kolmogorov_backward,thmMP:Kolmogorov_forward}. The Kolmogorov forward equation in case of a BD process reads in scalar form
\begin{align}%
\frac{\dinf}{\dinf t} p_{i,0}(t) &= -\underbrace{\la_0 p_{i,0}(t)}_{\textup{rate out}} + \underbrace{\mu_1 p_{i,1}(t)}_{\textup{rate in}}, \label{eqnBD:forward_Kolmogorov_p_i0} \\
\frac{\dinf}{\dinf t} p_{i,j}(t) &= \underbrace{\la_{j - 1} p_{i,j - 1}(t) + \mu_{j + 1} p_{i,j + 1}(t)}_{\textup{rate in}} - \underbrace{(\la_j + \mu_j) p_{i,j}(t)}_{\textup{rate out}}, \quad j \ge 1, \label{eqnBD:forward_Kolmogorov_p_ij}
\end{align}%
with the initial condition $p_{i,i}(0) = 1$.

\begin{example}[Poisson process]\label{exBD:Poisson_process_time-dependent}%
The homogeneous Poisson process can be seen as a BD process with $\la_i = \la$, $\mu_i = 0$ and $X(0) = 0$. This pure birth process will drift off towards infinity since all states are transient. The transition function $p_{0,j}(t)$ is the probability that $j$ births have occurred in the interval $[0,t]$. Obviously, the number of births in the interval $[0,t]$ is distributed according to a Poisson distribution with parameter $\la t$. We will verify this statement through \eqref{eqnBD:forward_Kolmogorov_p_i0}--\eqref{eqnBD:forward_Kolmogorov_p_ij}, which for $p_{0,j}(t)$ read
\begin{align}%
\frac{\dinf}{\dinf t} p_{0,0}(t) &= -\la p_{0,0}(t), \\
\frac{\dinf}{\dinf t} p_{0,j}(t) &= \la p_{0,j - 1}(t) - \la p_{0,j}(t), \quad j \ge 1. \label{eqnBD:Poisson_process_Kolmogorov_forward_equation_p_0,j}
\end{align}%
Together with $p_{0,0}(0) = 1$ this leads to $p_{0,0}(t) = \euler^{-\la t}$. Equation~\eqref{eqnBD:Poisson_process_Kolmogorov_forward_equation_p_0,j} is separable using
\begin{equation}%
\frac{\dinf}{\dinf t} \bigl( \euler^{\la t} p_{0,j}(t) \bigr) = p_{0,j}(t) \frac{\dinf}{\dinf t} \euler^{\la t}  + \euler^{\la t} \frac{\dinf}{\dinf t} p_{0,j}(t) = \la \euler^{\la t} p_{0,j - 1}(t).
\end{equation}%
So, by direct integration we obtain
\begin{equation}%
\euler^{\la t} p_{0,j}(t) = \la \int_0^t \euler^{\la u} p_{0,j - 1}(u) \, \dinf u.
\end{equation}%
The transition functions can be solved recursively starting from $p_{0,0}(t)$. Let us determine the first few terms. For $j = 1$, we derive
\begin{equation}%
p_{0,1}(t) = \la \euler^{-\la t} \int_0^t \euler^{\la u} p_{0,0}(u) \, \dinf u = \la \euler^{-\la t} \int_0^t \euler^{\la u} \euler^{-\la u} \, \dinf u = (\la t) \euler^{-\la t}.
\end{equation}%
The expression for $p_{0,1}(t)$ is used to determine the second term:
\begin{equation}%
p_{0,2}(t) = \la \euler^{-\la t} \int_0^t \euler^{\la u} (\la u)\euler^{-\la u} \, \dinf u = \frac{(\la t)^2}{2} \euler^{-\la t}.
\end{equation}%
The third term is
\begin{equation}%
p_{0,3}(t) = \la \euler^{-\la t} \int_0^t \euler^{\la u} \frac{(\la u)^2}{2} \euler^{-\la u} \, \dinf u = \frac{(\la t)^3}{3!} \euler^{-\la t}.
\end{equation}%
A pattern starts to show itself. Induction on $j$ is used to show that the explicit expression
\begin{equation}%
p_{0,j}(t) = \frac{(\la t)^j}{j!} \euler^{-\la t}, \quad j \in \statespace, ~ t \ge 0
\end{equation}%
is correct. This verifies that the number of births in the interval $[0,t]$ is indeed $\Poisson{\la t}$.
\end{example}%

\begin{example}[$M/M/\infty$ queue]\label{exBD:MMinf_time-dependent}%
We now set $\la_i = \la$ and $\mu_i = i \mu$. This BD process models for example a population that grows exclusively through immigration with rate $\la$ and all individuals die independently of each other with rate $\mu$ \cite[Section~4.6]{Karlin1975_First_course_stochastic_processes}; or packets arriving according to a Poisson process with rate $\la$ that are routed to their next destination after an exponential amount of time with rate $\mu$. In the queueing context we refer to a birth as an arrival of a job and a death as a departure of a job.

Suppose $X(0) = 0$ and we are interested in the transition functions $p_{0,j}(t)$. For the event $X(t) = j$ to occur, we require at least $j$ arrivals. If $k \ge j$ jobs arrive, we furthermore require $k - j$ departures. The probability that $k$ jobs arrive in the time interval $[0,t]$ follows from the Poisson distribution and is
\begin{equation}%
\euler^{-\la t} \frac{(\la t)^k}{k!}.
\end{equation}%
Conditioning on the fact that there are $k$ arrivals in the time interval $[0,t]$, we know that the arrival instant of each job is independent of the arrival instants of other jobs and is moreover uniformly distributed in the interval $[0,t]$. So, the probability $q(t)$ that a job is still in the system at time $t$ follows by conditioning on the arrival time:
\begin{equation}%
q(t) = \int_0^t \euler^{-\mu u} \frac{1}{t} \, \dinf u = \frac{1 - \euler^{-\mu t}}{\mu t}.
\end{equation}%
The probability that $j$ jobs remain at time $t$ conditioned on $k$ arriving in the interval $[0,t]$ follows a Bernoulli distribution and leads to an explicit expression for $p_{0,j}(t)$:
\begin{align}%
p_{0,j}(t) &= \Prob{0}{X(t) = j} \notag \\
&= \sum_{k \ge j} \Prob{0}{X(t) = j \mid \text{ $k$ arrivals in $[0,t]$}} \Prob{\text{$k$ arrivals in $[0,t]$}} \notag \\
&= \sum_{k \ge j} \binom{k}{j} (1 - q(t))^{k - j} q(t)^j \euler^{-\la t} \frac{(\la t)^k}{k!} \notag \\
&= \euler^{-\la t q(t)} \frac{(\la t q(t))^j}{j!} = \euler^{- \frac{\la}{\mu} (1 - \euler^{-\mu t})} \frac{\bigl( \frac{\la}{\mu} (1 - \euler^{-\mu t}) \bigr)^j}{j!}.
\end{align}%
The explicit expression for $p_{0,j}(t)$ allows for a simple determination of the transient mean as
\begin{equation}%
\E{0}{X(t)} = \sum_{j \ge 0} j p_{0,j}(t) = \frac{\la}{\mu} \bigl( 1 - \euler^{-\mu t} \bigr).
\end{equation}%
In conclusion, $X(t)$ conditional on $X(0) = 0$ is a Poisson distribution at each time $t$ with parameter $(\la/\mu) ( 1 - \euler^{-\mu t})$.\endnote{For a similar discussion of the $M/G/\infty$ queue, where the service time distribution is allowed to be any distribution (general, hence the G), see \cite[Section~3.2]{Takacs1962_Theory_Queues}.}
\end{example}%

We now consider the first time at which the BD process $\{ X(t) \}_{t \ge 0}$ enters a state $j$, starting from a state $i$. We recall the definition of a hitting time random variable in \eqref{eqnMP:hitting_time_random_variable} as
\begin{equation}%
\htt{i,j} \defi \inf\{ t > 0 : \lim_{s \uparrow t} X(s) \neq X(t) = j \mid X(0) = i \}, \label{eqnBD:hitting_time_random_variable}
\end{equation}%
We will make use of the LST
\begin{equation}%
\LST{i,j}{\LSTarg} \defi \E{\euler^{- \LSTarg \htt{i,j}}}, \quad \RealPart{\LSTarg} > 0.
\end{equation}%
Recall that a LST uniquely characterizes the distribution of a random variable.

\begin{example}[Regenerative structure]\label{exBD:MM1_busy_period}%
\begin{figure}
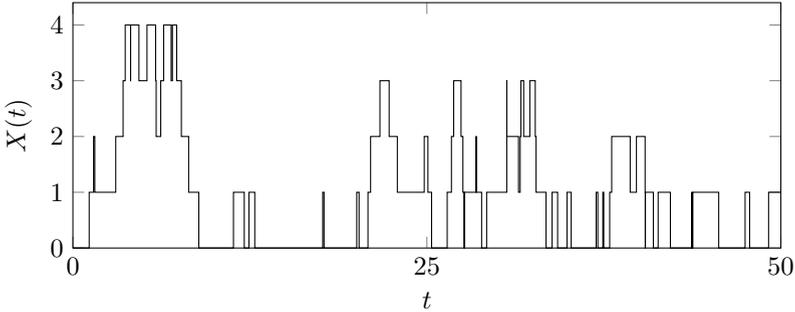
%
\centering%
\includestandalone{Chapters/BD/TikZFiles/simulation_homogeneous_BD_process}%
\caption{Sample path of a BD process with $\la_i = 1$, $\mu_i = 1.5$ and $X(0) = 0$.}%
\label{figBD:simulation_homogeneous}%
\end{figure}%
An irreducible BD process has a regenerative structure. Assume that at a particular time the BD process is in state 0. The process stays in state 0 for an exponential amount of time with parameter $\la_0$. After this time it transitions to state 1. Under the condition that the BD process is recurrent, it returns to state 0 after some time with probability 1. The time spent in state 0 is called an \textit{idle period} and the time it takes to go from state 1 to state 0 is called the \textit{busy period}. So, an irreducible BD process with recurrent states alternates between idle and busy periods, see \cref{figBD:simulation_homogeneous}. The terminology idle and busy period comes from the interpretation of a homogeneous BD process as the $M/M/1$ single server queue. In state 0 the server is idle and in all other states the server is busy serving jobs.

The length of a busy period is the hitting time random variable $\htt{1,0}$ with LST $\LST{1,0}{\LSTarg}$. Let us assume that the BD process is homogeneous with $\la_i = \la$ and $\mu_i = \mu$. Note that $\htt{1,0}$ is the sojourn time in state 1 plus the time it takes to reach state 0 from the state the process transitions to. We derive $\LST{1,0}{\LSTarg}$ using this observation, a one-step analysis and the strong Markov property:
\begin{equation}%
\LST{1,0}{\LSTarg} = \underbrace{\frac{\mu}{\la + \mu}}_{ \substack{\textup{prob. to jump}\\ \textup{to state 0}}} \underbrace{\frac{\la + \mu}{\la + \mu + \LSTarg}}_{\textup{LST of }H_1} \cdot \underbrace{1}_{\E{\euler^{-\LSTarg\cdot 0}}} + \underbrace{\frac{\la}{\la + \mu}}_{ \substack{\textup{prob. to jump}\\ \textup{to state 2}}} \frac{\la + \mu}{\la + \mu + \LSTarg} \LST{2,0}{\LSTarg}.
\end{equation}%
Due to the BD structure of the Markov process, we have $\htt{2,0} = \htt{2,1} + \htt{1,0}$, where $\htt{2,1}$ and $\htt{1,0}$ are independent random variables. More importantly, for homogeneous BD processes, the time it takes to go from state 2 to state 1 is exactly the same as the time it takes to go from state 1 to state 0 and in general the time it takes to go from state $n \ge 1$ to state $n - 1$. So, $\LST{2,0}{\LSTarg} = \LST{1,0}{\LSTarg}^2$ and we know that $\LST{1,0}{\LSTarg}$ is a solution to the polynomial
\begin{equation}%
\la x^2 - (\la + \mu + \LSTarg) x + \mu = 0.
\end{equation}%
This equation has the two roots
\begin{equation}%
x_{\pm}(\LSTarg) = \frac{\la + \mu + \LSTarg \pm \sqrt{(\la + \mu + \LSTarg)^2 - 4 \la \mu}}{2\la}.
\end{equation}%
A LST of a non-negative random variable has absolute value less than one for all $\LSTarg$ with $\RealPart{\LSTarg} > 0$. Since $0 < |x_-(\LSTarg)| < 1 < |x_+(\LSTarg)|$ for $\RealPart{\LSTarg} > 0$,
\begin{equation}%
\LST{1,0}{\LSTarg} = \frac{\la + \mu + \LSTarg - \sqrt{(\la + \mu + \LSTarg)^2 - 4 \la \mu}}{2\la}.
\end{equation}%
The expectation of the length of the busy period is determined from its LST
\begin{equation}%
\E{\htt{1,0}} = - \frac{\dinf}{\dinf \LSTarg} \LST{1,0}{\LSTarg} \Big\vert_{\LSTarg = 0} = \begin{cases}%
\frac{1}{\mu - \la}, & \la < \mu, \\
\infty, & \la = \mu,
\end{cases}%
\end{equation}%
and we agree to write $\E{\htt{1,0}} = \infty$ if $\int_0^\infty f_{\htt{1,0}}(t) \, \dinf t < 1$, which indicates that starting in state 1, there is a non-zero probability that state 0 will never be hit. This is the case if $\la > \mu$; we do not prove this statement. Here we already see the relation with positive recurrence ($\la < \mu$), null recurrence ($\la = \mu$) and transience ($\la > \mu$), that was observed in \cref{figBD:simulation}.
\end{example}%

If the system initially is in state 0 and the target state is $n$, we can write the hitting time $\htt{0,n}$ as a sum of \textit{independent} random variables:
\begin{equation}%
\htt{0,n} = \htt{0,1} + \htt{1,2} + \cdots + \htt{n - 1,n}.
\end{equation}%
The independence property is crucial in the analysis that will follow. Clearly, $\htt{0,1}$ is an exponential random variable with parameter $\la_0$. More importantly, $\htt{0,n}$ turns out to be a sum of $n$ exponential random variables. Albeit true, this result is rather counterintuitive. Consider for example $\htt{0,2} = \htt{0,1} + \htt{1,2}$. Here $\htt{0,1}$ is still an exponential random variable with parameter $\la_0$, while $\htt{1,2}$ is definitely not an exponential random variable, yet their sum is. The crux lies in the fact that $\htt{0,2}$ is the sum of two exponential random variables where both parameters are \textit{different} from $\la_0$.

\begin{theorem}\label{thmBD:first_passage_time}%
The hitting time $\htt{0,n}$ is distributed as the sum of $n$ exponential random variables:\endnote{Keilson provides two analytical proofs of \cref{thmBD:first_passage_time} in \cite[Section~1]{Keilson1971_Passage_time_densities} and \cite[Section~5.1]{Keilson1979_Markov_Chain_Models}, but does not characterize the parameters of the exponential random variables. Fill \cite[Theorem~1.1]{Fill2009_BD_passage_time_distribution} succeeds in proving the same result using probabilistic arguments and moreover characterizes the parameters of the exponential random variables.}
\begin{equation}%
\htt{0,n} = X_0^{(n)} + X_1^{(n)} + \cdots + X_{n - 1}^{(n)},
\end{equation}%
with $X_i^{(n)} \sim \Exp{\theta_i^{(n)}}$ and $\theta_i^{(n)}$ the $n$ positive eigenvalues of $-Q^{(n)}$, where $Q^{(n)}$ is the transition rate matrix of the BD process on the states $\{ 0,1,\ldots,n \}$ with $n$ an absorbing state.
\end{theorem}%

\begin{proof}%
In terms of the Laplace transform, we require to prove
\begin{equation}%
\LST{0,n}{\LSTarg} = \prod_{i = 0}^{n - 1} \frac{\theta_i^{(n)}}{\theta_i^{(n)} + \LSTarg} \ifed \frac{N^{(n)}}{D^{(n)}}. \label{eqnBD:hitting_time_LST}
\end{equation}%
A one-step analysis and the strong Markov property gives
\begin{equation}%
\LST{n,n + 1}{\LSTarg} = \frac{\la_n}{\la_n + \mu_n + \LSTarg} + \frac{\mu_n}{\la_n + \mu_n + \LSTarg} \LST{n - 1,n + 1}{\LSTarg}, \quad n \ge 1.
\end{equation}%
Using $\LST{n - 1,n + 1}{\LSTarg} = \LST{n - 1,n}{\LSTarg} \LST{n,n + 1}{\LSTarg}$ the above equation results in the recursion
\begin{equation}%
\LST{n,n + 1}{\LSTarg} = \frac{\la_n}{\la_n + \mu_n + \LSTarg - \mu_n \LST{n - 1,n}{\LSTarg}}, \quad n \ge 1. \label{eqnBD:hitting_time_recursion}
\end{equation}%
Next, multiply \eqref{eqnBD:hitting_time_recursion} by $\LST{0,n}{\LSTarg}$ and use $\LST{0,n}{\LSTarg} = \LST{0,n - 1}{\LSTarg} \LST{n - 1,n}{\LSTarg}$ to obtain
\begin{equation}%
\LST{0,n + 1}{\LSTarg} = \frac{\la_n \LST{0,n}{\LSTarg}}{\la_n + \mu_n + \LSTarg - \mu_n \frac{\LST{0,n}{\LSTarg}}{\LST{0,n - 1}{\LSTarg}}}, \quad n \ge 1. \label{eqnBD:hitting_time_recursion_used_for_induction}
\end{equation}%
We proceed by induction. The claim \eqref{eqnBD:hitting_time_LST} is true for $n = 1$, since $\LST{0,1}{\LSTarg} = \la_0/(\la_0 + \LSTarg)$. Assume the claim is true for $n$, then \eqref{eqnBD:hitting_time_recursion_used_for_induction} reads
\begin{equation}%
\LST{0,n + 1}{\LSTarg} = \frac{ \la_n N^{(n)} N^{(n - 1)} }{ (\la_n + \mu_n + \LSTarg) N^{(n - 1)} D^{(n)}  - \mu_n  N^{(n)} D^{(n - 1)} }.
\end{equation}%
The denominator of $\LST{0,n + 1}{\LSTarg}$ will be a polynomial of degree $n + 1$. Moreover, \cite[Sections~3, 4 and 5]{Keilson1964_Transient_behavior_BD} establishes that $\LST{0,n + 1}{\LSTarg}$ has $n + 1$ negative real simple poles. Thus, $\LST{0,n + 1}{\LSTarg}$ also has the form \eqref{eqnBD:hitting_time_LST}, proving the claim.

Proving that the $\theta_i^{(n)}$ are the $n$ positive eigenvalues of $-Q^{(n)}$ is outside the scope of this book, an interested reader is pointed to \cite{Fill2009_BD_passage_time_distribution}.
\end{proof}%

Hitting times and transition functions are inherently linked. Let $\{ Y(t) \}_{t \ge 0}$ be a modified process which is identical to the BD process $\{ X(t) \}_{t \ge 0}$, except that the target state $j$ is absorbing. Since state $j$ in the process $\{ Y(t) \}_{t \ge 0}$ is absorbing, we know that if $Y(t)$ reaches state $j$, it stays there forever. In other words, if $Y(t)$ reaches state $j$ at time $t^* < t$, it will still be in state $j$ at time $t$. This leads to a relation between the hitting time $\htt{i,j}$ defined in \eqref{eqnBD:hitting_time_random_variable} and the transition functions of the modified BD process $\{ Y(t) \}_{t \ge 0}$:\endnote{The intuition for \eqref{eqnBD:link_hitting_time_transition_functions} appears in \cite[Section~3.1]{Crawford2014_BD}, but also follows from the standard probabilistic reasoning outlined above the equation.}
\begin{equation}%
\Prob{\htt{i,j} < t} = \Prob{i}{Y(t) = j}. \label{eqnBD:link_hitting_time_transition_functions}
\end{equation}%
%

%%%%%%%%%%%%%%%%%%%%%%%%%%%%%%%%%%%%%%%%%%%%%%%%%%%%%%%
%%%%%%%%%%%%%%%%%%%%%%%%%%%%%%%%%%%%%%%%%%%%%%%%%%%%%%%
%%%%%%%%%%%%%%%%%%%%% NEW SECTION %%%%%%%%%%%%%%%%%%%%%
%%%%%%%%%%%%%%%%%%%%%%%%%%%%%%%%%%%%%%%%%%%%%%%%%%%%%%%
%%%%%%%%%%%%%%%%%%%%%%%%%%%%%%%%%%%%%%%%%%%%%%%%%%%%%%%

\section{Equilibrium distribution}%
\label{secBD:equilibrium_distribution}%

The BD processes that we study are irreducible. The irreducibility property implies that the BD process can go from any state $i$ to any state $j$. For an irreducible Markov process the unique equilibrium distribution exists if it is positive recurrent. For BD processes we derive a necessary and sufficient condition for positive recurrence and examine the equilibrium distribution.

In \cref{secMP:equilibrium_behavior} the concepts of global and local balance are introduced. In the context of a BD process, the global balance equations are constructed by equating the rate into and out of state $i$, yielding
\begin{align}%
\la_0 p(0) &=  \mu_1 p(1), \label{eqnBD:global_balance_equations_p_0}\\
(\la_i + \mu_i) p(i) &= \la_{i - 1} p(i - 1) + \mu_{i + 1} p(i + 1), \quad i \ge 1. \label{eqnBD:global_balance_equations_p_i}
\end{align}%
The latter equation is a second-order linear recurrence equation. Whenever possible, a relation between $p(i)$ and $p(i - 1)$ is far more convenient to work with and often leads to simple ways to determine explicit expressions for the equilibrium distribution $\pb = [ p(i) ]_{i \ge 0}$. Local balance equations give this relation between $p(i)$ and $p(i - 1)$. They are derived by equating the flow into and out of the set of states $\set{A}_{i - 1} = \{0,1,\ldots,i - 1\}$. Since this set of states has a single state through which the process can enter and leave, the local balance equations result in the simple expression
\begin{equation}%
\mu_i p(i) = \la_{i - 1} p(i - 1), \quad i \ge 1. \label{eqnBD:local_balance_equations}
\end{equation}%
Note that the local balance equations can also be obtained from the global balance equations by substitutions. The local balance equations may be solved in a recursive fashion, yielding
\begin{equation}%
p(i) = p(0) \prod_{j = 1}^i \frac{\la_{j - 1}}{\mu_j}, \quad i \ge 0, \label{eqnBD:local_balance_solution_p_i}
\end{equation}%
where the empty product $\prod_{j = 1}^0 = 1$. All equilibrium probabilities $p(i)$ are expressed in terms of $p(0)$. Finally, the normalization condition allows for the determination of $p(0)$ from
\begin{equation}%
1 = \sum_{i \ge 0} p(i)  = p(0) \sum_{i \ge 0} \prod_{j = 1}^i \frac{\la_{j - 1}}{\mu_j}. \label{eqnBD:normalization_condition_p_0}
\end{equation}%
The following theorem now summarizes when an irreducible Markov process is positive recurrent.

\begin{theorem}\label{thmBD:positive_recurrence}%
A necessary and sufficient condition for an irreducible \textup{(BD)} process to be positive recurrent is
\begin{equation}%
\sum_{i \ge 0} \prod_{j = 1}^i \frac{\la_{j - 1}}{\mu_j} < \infty, \label{eqnBD:stability_condition}
\end{equation}%
and ensures that all $p(i) > 0$.
\end{theorem}%

\begin{proof}%
By \cref{thmMP:Foster}, if an irreducible Markov process has a solution $\pb, ~ \pb \oneb = 1$ to the balance equations, then the Markov process is positive recurrent.
\end{proof}%

Condition \eqref{eqnBD:stability_condition} is sometimes referred to as the \textit{stability condition}. Note that this condition is trivially satisfied when the state space is finite, which is not surprising since we know that a finite irreducible Markov process possesses a unique equilibrium distribution.

Returning to the homogeneous BD processes, we see that the stability condition \eqref{eqnBD:stability_condition} reduces to
\begin{equation}%
\sum_{i \ge 0} \Bigl( \frac{\la}{\mu} \Bigr)^i < \infty. \label{eqnBD:stability_condition_homogeneous}
\end{equation}%
So, if $\la/\mu < 1$ the BD process is positive recurrent and an equilibrium distribution exists. The condition $\la/\mu < 1$ makes the intuition for the sample paths in \cref{figBD:simulation} rigorous.

The local balance equations \eqref{eqnBD:local_balance_equations} can be derived by censoring the parts of the sample path of the Markov process when it is not inside the set $\set{A}_i = \{0,1,\ldots,i\}$ with $i \ge 1$. The rate at which the process leaves state $i$ is $p(i) (\la_i + \mu_i)$. The rate at which we enter state $i$ is $p(i - 1) \la_{i - 1}$ plus the rate at which the process transitions to states outside $\set{A}_i$ that return to state $i$. There is only one transition from a state in $\set{A}_i$, state $i$, to a state outside $\set{A}_i$, state $i + 1$. Taking into account the above observations, the balance equations of the censored process are
\begin{equation}%
p(i) (\la_i + \mu_i) = p(i - 1) \la_{i - 1} + p(i) \la_i \Prob{\textup{return to $\set{A}_i$ in state $i$}}. \label{eqnBD:balance_equations_censored_process}
\end{equation}%
In this case, $\Prob{\textup{return to $\set{A}_i$ in state $i$}} = G_{i + 1,i}$, which is the probability that, starting from state $i + 1$, the Markov process reaches state $i$ in finite time. Assuming the Markov process is irreducible and positive recurrent, we know that the process always returns to $\set{A}_i$. More importantly, due to the BD structure the process always returns to $\set{A}_i$ via state $i$. Combining these two properties we derive $G_{i + 1,i} = 1$. The balance equations \eqref{eqnBD:balance_equations_censored_process} for the censored process reduce to
\begin{equation}%
p(i) \mu_i = p(i - 1) \la_{i - 1}, \label{eqnBD:balance_equations_censored_process_equal_to_local_balance_equations}
\end{equation}%
which is a local balance equation. The index $i$ in $\set{A}_i$ was arbitrary, so \eqref{eqnBD:balance_equations_censored_process_equal_to_local_balance_equations} holds for all $i \ge 1$.

\begin{example}[$M/M/\infty$ queue]\label{exBD:MMinf_equilibrium_distribution}%
We return to \cref{exBD:MMinf_time-dependent} concerning the $M/M/\infty$ queue. Regardless of the value of $\la$ and $\mu$, as long as they are finite, this BD process is positive recurrent:
\begin{equation}%
\sum_{i \ge 0} \prod_{j = 1}^i \frac{\la_{j - 1}}{\mu_j} = \sum_{i \ge 0} \prod_{j = 1}^i \frac{\rho}{j} = \sum_{i \ge 0} \frac{\rho^i}{i!} = \euler^{\rho} < \infty,
\end{equation}%
with $\rho \defi \la/\mu$. Since the BD process is positive recurrent, the limiting distribution exists and is found by taking $t \to \infty$ in the transition functions of \cref{exBD:MMinf_time-dependent}, resulting in
\begin{equation}%
p(j) = \lim_{t \to \infty} p_{0,j}(t) = \lim_{t \to \infty} \euler^{- \rho (1 - \euler^{-\mu t})} \frac{\bigl( \rho (1 - \euler^{-\mu t}) \bigr)^j}{j!} = \euler^{-\rho} \frac{\rho^j}{j!}. \label{eqnBD:MMinf_limiting_dist}
\end{equation}%
We showed in \cref{exBD:MMinf_time-dependent} that for each time $t$, $X(t)$ has a Poisson distribution and also in equilibrium it follows a Poisson distribution.

If the transition functions are not available, the equilibrium distribution can be derived using the balance equations. Let us do that now. Each BD process satisfies the local balance equations. In this case they read
\begin{equation}%
p(j) j \mu = \la p(j - 1) \quad \Rightarrow \quad p(j) = \frac{\rho}{j} p(j - 1) = \cdots = \frac{\rho^j}{j!} p(0). \label{eqnBD:MMinf_equilibrium_dist_p_j}
\end{equation}%
We have an expression for $p(j)$ in terms of $p(0)$. The probability of being in state 0 follows from the normalization condition as follows
\begin{equation}%
1 = \sum_{j \ge 0} p(j) = p(0) \sum_{j \ge 0} \frac{\rho^j}{j!} \quad \Rightarrow \quad p(0) = \euler^{-\rho}. \label{eqnBD:MMinf_equilibrium_dist_p_0}
\end{equation}%
Combining \eqref{eqnBD:MMinf_equilibrium_dist_p_j} and \eqref{eqnBD:MMinf_equilibrium_dist_p_0} shows that the equilibrium distribution is also given by \eqref{eqnBD:MMinf_limiting_dist}.
\end{example}%

\begin{example}[$M/M/s/s$ queue]\label{exBD:MMss_equilibrium_dist}%
We examine the $M/M/\infty$ queue but set $\la_i = 0$ for $i \ge s$. The interpretation of this queueing system is that of the $M/M/\infty$ queue, but if $s$ servers are occupied, no arriving jobs are allowed into the system. These jobs may be considered blocked, or lost, and correspondingly this system is referred to as the \textit{Erlang loss} or \textit{Erlang-B} system. An alternative interpretation is that of a system with $s$ servers that allows a maximum of $s$ jobs to be in the system simultaneously. The birth rates are $\la_i = \la, ~ 0 \le i < s$ and the death rates are $\mu_i = i \mu, ~ 1 \le i \le s$. Since it is an irreducible BD process with a finite state space, all states are positive recurrent. The local balance equations are, with $\rho \defi \la/\mu$,
\begin{equation}
p(j) j \mu = p(j - 1) \la, \quad 1 \le j \le s \quad \Rightarrow \quad p(j) = \frac{\rho^j}{j!} p(0), \quad 1 \le j \le s.
\end{equation}%
So, the equilibrium probabilities of the $M/M/s/s$ model have the same structure as the ones of the $M/M/\infty$ model seen in \cref{exBD:MMinf_equilibrium_distribution}, expect for the normalization. The normalization condition in this case is
\begin{equation}
1 = \sum_{j = 0}^s p(j) = p(0) \sum_{j = 0}^s \frac{\rho^j}{j!} \quad \Rightarrow \quad p(0) = \Bigl( \sum_{j = 0}^s \frac{\rho^j}{j!} \Bigr)^{-1}
\end{equation}%
and the equilibrium distribution has been determined. If $s \to \infty$ then $p(0)$ converges to $\euler^{-\rho}$ and the equilibrium distribution coincides with the one from the $M/M/\infty$ model.

A quantity of special significance is the probability that an arriving job is lost or blocked, which, by PASTA (see \cref{subsecQTF:simple_queue}), is given by
\begin{equation}%
B(s,\rho) = p(s) = \frac{\rho^s}{s!} \Bigl( \sum_{j = 0}^s \frac{\rho^j}{j!} \Bigr)^{-1}.
\end{equation}%
This is often termed the Erlang-B formula. It is easily verified that the probability of blocking satisfies the recurrence relation \makeExercise
\begin{equation}%
B(s + 1,\rho) = \frac{\rho B(s,\rho)}{\rho B(s,\rho) + s + 1},
\end{equation}%
which is useful for numerical computation.
\end{example}%

\begin{example}[$M/M/1$ queue]\label{exBD:MM1_equilibrium_dist}%
We consider a homogeneous BD process with $\la_i = \la$ and $\mu_i = \mu$. This is also called an $M/M/1$ queue in queueing terminology. From \eqref{eqnBD:stability_condition_homogeneous} we require $\la < \mu$ for the states to be positive recurrent. The equilibrium distribution is derived from \eqref{eqnBD:local_balance_solution_p_i}--\eqref{eqnBD:normalization_condition_p_0} and found to be
\begin{equation}%
p(i) = (1 - \rho) \rho^i, \quad i \ge 0,
\end{equation}%
with $\rho \defi \la/\mu$.
\end{example}%

\begin{example}[$M/M/s$ queue]\label{exBD:MMs_equilibrium_dist}%
Consider a queueing system consisting of $s$ servers and a common queue. Jobs arrive according to a Poisson process with rate $\la$ and enter service if a server is idle. Serving a job takes $\Exp{\mu}$ time. If all servers are occupied, the job joins the end of the waiting line in the common queue. When a server finishes serving a job, he takes the first job from the waiting line and starts serving that job. If there is no waiting job, the server becomes idle. This model is often referred to as the \textit{Erlang-C} model. Both \cref{exBD:MMinf_equilibrium_distribution} and \cref{exBD:MM1_equilibrium_dist} are special cases of this model.

The total number of jobs in the system at time $t$, labeled $X(t)$, evolves according to a BD process with $\la_i = \la$ and
\begin{equation}%
\mu_i = \begin{cases}%
i \mu, & 0 \le i \le s - 1, \\
s \mu, & i \ge s.
\end{cases}%
\end{equation}%
Applying \cref{thmBD:positive_recurrence}, the BD process is positive recurrent iff, with $\rho \defi \la/\mu$,
\begin{equation}%
\sum_{i \ge 0} \prod_{j = 1}^i \frac{\la_{j - 1}}{\mu_j} = \sum_{i = 0}^{s - 1} \prod_{j = 1}^i \frac{\la_{j - 1}}{\mu_j} + \sum_{i \ge s} \prod_{j = 1}^i \frac{\la_{j - 1}}{\mu_j} = \sum_{i = 0}^{s - 1} \frac{\rho^i}{i!} + \frac{\rho^s}{s!} \sum_{i \ge 0} \frac{\rho^i}{s^i} < \infty.
\end{equation}%
So, $\rho/s < 1$ ensures that an equilibrium distribution exists. From the local balance equations \eqref{eqnBD:local_balance_equations} the equilibrium distribution is
\begin{equation}%
p(i) = \begin{cases}%
p(0) \displaystyle \frac{\rho^i}{i!}, & 0 \le i \le s - 1, \\
p(0) \displaystyle \frac{\rho^i}{s!s^{i - s}}, & i \ge s,
\end{cases}%
\end{equation}%
with
\begin{equation}%
p(0) = \Bigl( \sum_{i = 0}^{s - 1} \frac{\rho^i}{i!} + \sum_{i \ge s}  \frac{\rho^i}{s!s^{i - s}} \Bigr)^{-1} = \Bigl( \sum_{i = 0}^{s - 1} \frac{\rho^i}{i!} + \frac{\rho^s}{s!} \frac{1}{1 - \frac{\rho}{s}} \Bigr)^{-1}
\end{equation}%
representing the probability that the system is empty.

A quantity of great importance is the probability that an arriving job must wait, which is, by the PASTA property,
\begin{equation}%
C(s,\rho) = \sum_{i \ge s} p(i) = \Bigl( 1 + \frac{s!}{\rho^s} (1 - \frac{\rho}{s}) \sum_{i = 0}^{s - 1} \frac{\rho^i}{i!} \Bigr)^{-1}.
\end{equation}%
This is often referred to as the Erlang-C formula. It is easily verified that the probability of waiting satisfies the recurrence relation \makeExercise
\begin{equation}%
C(s + 1,\rho) = \frac{(1 - \frac{\rho}{s}) C(s,\rho)}{s + 1 - \rho - \frac{\rho}{s} C(s,\rho)},
\end{equation}%
which is convenient for numerical calculations.

The waiting time $W$ of a job is the time between his arrival and the time he is taken into service, assuming an equilibrium state for the queueing system. We analyze the waiting time distribution of an arbitrary job. Note that, given that an arriving job must wait, the number of waiting jobs in front of him is geometrically distributed with parameter $\rho/s$. So, the number of service completions the arriving job must wait for is $G + 1$, where $G$ is a geometrically distributed random variable with parameter $\rho/s$. Also note that the times between successive service completions are independent and exponentially distributed random variables with
parameter $s \mu$. Now observe that the sum of $G + 1$ independent and exponentially distributed random variables with parameter $\nu$, where $G$
itself is a geometrically distributed random variable with parameter $p$, is again exponentially distributed with parameter $\nu (1 - p)$. \makeExercise So given that an arriving job must wait, its waiting time is
exponentially distributed with parameter $s \mu (1 - \rho / s) = s \mu - \la$. Therefore the unconditional waiting-time distribution is given by
\begin{equation}%
\Prob{W > t} = C(s,\rho) \euler^{- (s \mu - \la) t},
\end{equation}%
since the probability that an arbitrary job needs to wait is $C(s,\rho)$.
\end{example}%

Denote by $R_{i,j}, ~ j > i$ the expected time spent in state $j$ multiplied by $\la_i + \mu_i$ between two successive visits to state $i$. Conditioning on the state visited after the first jump of the Markov process gives, for $i \ge 1$,
\begin{align}%
R_{i,i + 1} &= (\la_i + \mu_i) \Efxd{ i }{\int_0^{\htt{i,i}} \ind{X(t) = i + 1} \, \dinf t} \notag \\
&= (\la_i + \mu_i) \Bigl( \frac{\la_i}{\la_i + \mu_i} \Efxd{ i + 1 }{\int_0^{\htt{i + 1,i}} \ind{X(t) = i + 1} \, \dinf t} \notag \\
&\hspace{0.17\linewidth} + \frac{\mu_i}{\la_i + \mu_i} \Efxd{ i - 1 }{\int_0^{\htt{i - 1,i}} \ind{X(t) = i + 1} \, \dinf t} \Bigr) \notag \\
&= \la_i \Efxd{ i + 1 }{\int_0^{\htt{i + 1,i}} \ind{X(t) = i + 1} \, \dinf t} \ifed \la_i M_{i + 1,i}.
\end{align}%
$M_{i + 1,i}$ is interpreted as the expected time spent in state $i + 1$ before the process reaches state $i$, given that the process starts in state $i + 1$. This quantity is determined from a one-step analysis,
\begin{equation}%
M_{i + 1,i} = \underbrace{\frac{1}{\la_{i + 1} + \mu_{i + 1}}}_{\E{H_{i + 1}}} + \frac{\la_{i + 1}}{\la_{i + 1} + \mu_{i + 1}} G_{i + 1,i} M_{i + 1,i}.
\end{equation}%
A positive recurrent Markov process has $G_{i + 1,i} = 1$ and therefore
\begin{equation}%
M_{i + 1,i} = \frac{1}{\mu_{i + 1}} \quad \Rightarrow \quad R_{i,i + 1} = \frac{\la_i}{\mu_{i + 1}},
\end{equation}%
which holds for all $i$. It seems that
\begin{equation}%
p(i) = p(i - 1) R_{i - 1,i} = \cdots = p(0) \prod_{j = 0}^{i - 1} R_{j,j + 1} = p(0) \prod_{j = 0}^{i - 1} \frac{\la_{j - 1}}{\mu_j}, \quad i \ge 1,
\end{equation}%
which can be proven to hold.\endnote{A formal proof of the relation $p(i) = p(0) \prod_{j = 0}^{i - 1} R_{j,j + 1}$ can be found in Latouche and Ramaswami \cite[Section~4.5]{Latouche1999_Matrix-analytic}. Section~4.6 in the same reference derives more properties of the $R_{j,j + 1}$. Neuts also discusses this quantity in \cite{Neuts1994_Matrix-geometric}.} Plugging $p(i) = p(0) \prod_{j = 0}^{i - 1} R_{j,j + 1}$ into the global balance equations \eqref{eqnBD:global_balance_equations_p_0}--\eqref{eqnBD:global_balance_equations_p_i} gives
\begin{align}%
\la_0 p(0) &= \mu_1 p(0) R_{0,1}, \label{eqnBD:R_0,1} \\
(\la_i + \mu_i) p(0) \prod_{j = 0}^{i - 1} R_{j,j + 1} &= \la_{i - 1} p(0) \prod_{j = 0}^{i - 2} R_{j,j + 1} \notag \\
&\quad + \mu_{i + 1} p(0) \prod_{j = 0}^{i} R_{j,j + 1}, \quad i \ge 1. \label{eqnBD:R_i,i+1}
\end{align}%
$R_{0,1}$ is determined from \eqref{eqnBD:R_0,1}. Dividing \eqref{eqnBD:R_i,i+1} by $p(0) \prod_{j = 0}^{i - 2} R_{j,j + 1}$ shows that $R_{i,i + 1}$ satisfies
\begin{equation}%
\mu_{i + 1} R_{i - 1,i} R_{i,i + 1} - (\la_i + \mu_i) R_{i - 1,i} + \la_{i - 1} = 0, \quad i \ge 1,
\end{equation}%
or
\begin{equation}%
R_{i,i + 1} = \frac{\la_i + \mu_i}{\mu_{i + 1}} - \frac{\la_{i - 1}}{\mu_{i + 1}} \frac{1}{R_{i - 1,i}}, \quad i \ge 1.
\end{equation}%
If the BD process is homogeneous with $\la_i = \la$ and $\mu_i = \mu$, then from the definition of $R_{i,j}$ we deduce that all $R_{i,i + 1}$ are identical and we denote it by $R$. This implies that $R$ is the solution to the quadratic equation
\begin{equation}%
\mu R^2 - (\la + \mu) R + \la = 0, \quad i \ge 1. \label{eqnBD:R}
\end{equation}%
If the BD process is positive recurrent, then $R$ is the minimal non-negative solution to \eqref{eqnBD:R}. We return to these equations for $R_{i,i + 1}$ and $R$ in \cref{ch:quasi-birth--and--death_processes}.

%%%%%%%%%%%%%%%%%%%%%%%%%%%%%%%%%%%%%%%%%%%%%%%%%%%%%%%
%%%%%%%%%%%%%%%%%%%%%%%%%%%%%%%%%%%%%%%%%%%%%%%%%%%%%%%
%%%%%%%%%%%%%%%%%%% NEW SUBSECTION %%%%%%%%%%%%%%%%%%%%
%%%%%%%%%%%%%%%%%%%%%%%%%%%%%%%%%%%%%%%%%%%%%%%%%%%%%%%
%%%%%%%%%%%%%%%%%%%%%%%%%%%%%%%%%%%%%%%%%%%%%%%%%%%%%%%

\section{Takeaways}%
\label{secBD:takeaways}%

Many probability text books cover birth--and--death (BD) processes, ranging from short descriptions of the balance equations and equilibrium distribution, to extensive chapters including many special cases and time-dependent analysis \cite{Asmussen2008_Applied_probability_and_queues,Karlin1975_First_course_stochastic_processes,Robert2013_Stochastic_networks_and_queues}. In fact, we also decided to include some time-dependent analysis starting from the Kolmogorov forward equations that describe the relations between transition functions. The time-dependent analysis of all Markov processes, also the ones treated in this book, can depart from Kolmogorov equations, but only exceptional cases like BD processes lead to equations that are amenable to analysis, let alone result in compact solutions like in some of the examples. For a more extensive treatment of the time-dependent analysis of BD process, including some deep connections with orthogonal polynomials, we refer to the classic work of Karlin and McGregor \cite{Karlin1962_Orthogonal_polynomials}.

BD processes give rise to Markov process with states that can be arranged on a half-line. This special structures makes that instead of global balance, it suffices to work with local balance, which considerably reduces the complexity of the system of equations. While we see more examples in this book where local balance suffices (\cref{ch:reversible_networks}), for BD processes the local balance equations are particularly neat, and solved by the product-form solution in \eqref{eqnBD:local_balance_solution_p_i}. This solution can be obtained by a recursive argument that starts in state zero and follows the half-line from one state to the other. We will exploit such recursive structures more often, for instance in developing the theory of matrix-geometric methods presented in \cref{ch:quasi-birth--and--death_processes,ch:skip-free_one_direction}.

We saw that the equilibrium distribution of a BD process can also be found using the global balance equations, for instance using generating functions. For BD processes this is a naive method that forgets to exploit the unique state space structure, but still gives the product-form solutions. In this book we see more examples that can be approached by either global or (more) local views. In these more involved examples of \cref{ch:reversible_networks,ch:join_the_shortest_queue}, the global view leads nowhere, while the local view (not necessarily local balance, but at least a flow argument between a reduced number of states) provides a handle for finding a product-form solution.

%%%%%%%%%%%%%%%%%%%%%%%%%%%%%%%%%%%%%%%%%%%%%%%%%%%%%%%
%%%%%%%%%%%%%%%%%%%%%%%%%%%%%%%%%%%%%%%%%%%%%%%%%%%%%%%
%%%%%%%%%%%%%%%%%%%%%%%% NOTES %%%%%%%%%%%%%%%%%%%%%%%%
%%%%%%%%%%%%%%%%%%%%%%%%%%%%%%%%%%%%%%%%%%%%%%%%%%%%%%%
%%%%%%%%%%%%%%%%%%%%%%%%%%%%%%%%%%%%%%%%%%%%%%%%%%%%%%%

%\theendnotes%
%\setcounter{endnote}{0}
\printendnotes% %

% Checked points 1-7 and a-i
\chapter{Queueing networks}%
\label{ch:reversible_networks}%

This chapter deals with structured classes of Markov processes that arise from considering queueing networks, so systems of queues in which jobs or customers following routes to traverse multiple stations. The structure of these Markov processes shows strong dependencies between customers and stations, but nevertheless product-form solutions arise for some classes of networks.

%%%%%%%%%%%%%%%%%%%%%%%%%%%%%%%%%%%%%%%%%%%%%%%%%%%%%%%
%%%%%%%%%%%%%%%%%%%%%%%%%%%%%%%%%%%%%%%%%%%%%%%%%%%%%%%
%%%%%%%%%%%%%%%%%%%%% NEW SECTION %%%%%%%%%%%%%%%%%%%%%
%%%%%%%%%%%%%%%%%%%%%%%%%%%%%%%%%%%%%%%%%%%%%%%%%%%%%%%
%%%%%%%%%%%%%%%%%%%%%%%%%%%%%%%%%%%%%%%%%%%%%%%%%%%%%%%

\section{Reversibility}%
\label{secRN:reversibility}%

For the purpose of introducing \textit{reversibility}, or \textit{time-reversibility}, we assume that the time index $t$ belongs to $\Real$, so that a Markov process is referred to as $\{ X(t) \}_{t \in \Real}$. In this context, a stationary process has $\Prob{X(0) = x} = p(x)$, where $\pb = [p(x)]_{x \in \statespace}$ is the stationary distribution.

\begin{theorem}\label{thmRN:stationary_Markov_processes_-t}%
Consider a stationary Markov process $\{ X(t) \}_{t \in \Real}$. Then the process $\{ X(-t) \}_{t \in \Real}$ is a stationary Markov process with the same equilibrium distribution $\pb = [ p(x) ]_{x \in \statespace}$ and transition rates, for $x \neq y$,
\begin{equation}%
\alt{\q{x,y}} \defi \frac{p(y)}{p(x)} \q{y,x}, \quad x,y \in \statespace,
\end{equation}%
and $\alt{\q{x}} \defi \sum_{y \neq x} \alt{\q{x,y}} = \sum_{y \neq x} \q{x,y} \ifed \q{x}$.
\end{theorem}%

\begin{proof}%
$\{ X(-t) \}_{t \in \Real}$ is a stationary process since $\Prob{X(-t) = x} = p(x)$. Define $Y(t) \defi X(-t)$. Now, for $x \neq y$,
\begin{align}%
\Prob{Y(t + h) = y \mid Y(t) = x} &= \frac{\Prob{Y(t + h) = y, ~ Y(t) = x}}{\Prob{Y(t) = x}} \notag \\
&= \frac{\Prob{X(-t - h) = y, ~ X(-t) = x}}{\Prob{X(-t) = x}} \notag \\
&= \frac{p(y) p_{y,x}(h)}{p(x)}.
\end{align}%
Dividing both sides by $h$, letting $h \downarrow 0$ and recalling \eqref{eqnMP:behavior_transition_functions_small_time_interval} gives the result.\endnote{A similar proof of \cref{thmRN:stationary_Markov_processes_-t} appears in Chen and Yao \cite[Lemma~1.3]{Chen2001_Queueing_Networks}.}
\end{proof}%

\begin{definition}[Reversibility]\label{defRN:reversibility}%
If a Markov process satisfies, for $x \neq y$,
\begin{equation}%
p(x) \q{x,y} = p(y) \q{y,x}, \quad x,y \in \statespace
\end{equation}%
then the process is reversible.
\end{definition}%

This definition implicates that all Markov processes that have a solution to the local balance equations are reversible Markov processes. In particular, all BD processes are reversible.

\begin{example}[$M/M/s/s$ queue]\label{exRN:MMss_reversibility}%
Recall the Erlang-B model, which is a BD process with $\q{i,i + 1} = \la$ and $\q{i + 1,i} = (i + 1) \mu$ for $0 \le i \le s - 1$. The equilibrium distribution was derived in \cref{exBD:MMss_equilibrium_dist} and is
\begin{equation}%
p(i) = p(0) \frac{(\la/\mu)^i}{i!}, \quad 0 \le i \le s
\end{equation}%
with $p(0)$ given in \cref{exBD:MMss_equilibrium_dist}. Using \cref{defRN:reversibility}, for $0 \le i \le s - 1$,
\begin{equation}%
p(i) \q{i,i + 1} = p(0) \frac{(\la/\mu)^i}{i!} \la = p(0) \frac{(\la/\mu)^{i + 1}}{(i + 1)!} (i + 1)\mu = p(i + 1) \q{i + 1,i},
\end{equation}%
verifying that the Markov process associated with the $M/M/s/s$ queue is reversible.
\end{example}%

The following theorem plays a key part in the analysis of stochastic networks that are reversible.

\begin{theorem}\label{thmRN:truncating_reversible_process}%
A reversible Markov process with state space $\statespace$ and equilibrium distribution $\pb = [p(x)]_{x \in \statespace}$ that is truncated to $\set{A} \subset \statespace$ is again a reversible Markov process with equilibrium distribution
\begin{equation}%
\alt{p(x)} = \frac{p(x)}{\sum_{y \in \set{A}} p(y)}, \quad x \in \set{A}. \label{eqnRN:equilibrium_distribution_truncated_reversible_process}
\end{equation}%
\end{theorem}%

\begin{proof}%
Note that $\alt{p(x)} \q{x,y} = \alt{p(y)} \q{y,x}$ by reversibility of the original process, so detailed balance is satisfied.\endnote{See also Kelly and Yudovina \cite[Section~3.3]{Kelly2014_Stochastic_Networks} for some examples of the uses of truncating a reversible Markov process.}
\end{proof}%

\begin{example}[$M/M/s/s$ queue]\label{exRN:MMss_truncation}%
Employing \cref{thmRN:truncating_reversible_process}, the equilibrium distribution of the Markov process associated with the $M/M/s/s$ queue is the same as both the equilibrium distribution of the reversible Markov processes of the $M/M/s$ queue and of the $M/M/\infty$ queue truncated to the set $\set{A} = \{ 0,1,\ldots,s \}$. From \cref{exBD:MMs_equilibrium_dist,exBD:MMinf_equilibrium_distribution}, we know that the equilibrium probabilities are
\begin{equation}%
p(i)^{(M/M/s)} = \begin{cases}%
p(0)^{(M/M/s)} \frac{\rho^i}{i!}, & 0 \le i \le s - 1, \\
p(0)^{(M/M/s)} \frac{\rho^i}{s!s^{i - s}}, & i \ge s,
\end{cases}
\end{equation}%
and
\begin{equation}%
p(i)^{(M/M/\infty)} = p(0)^{(M/M/\infty)} \frac{\rho^i}{i!},
\end{equation}%
with $\rho \defi \la/\mu$. Plugging both equilibrium probabilities into the right-hand side of \eqref{eqnRN:equilibrium_distribution_truncated_reversible_process} produces the equilibrium distribution of the Markov process associated with the $M/M/s/s$ queue.
\end{example}%

The queueing systems that we consider in this book have Poisson arrival processes. For many of these systems, the departure process is also a Poisson process where the departure rate is equal to the arrival rate, which we show in the following theorem. In queueing networks, the departure process of one queue can be the arrival process of another queue. Knowing that this arrival process is again a Poisson process makes the analysis of the network a lot easier.

\begin{theorem}\label{thmRN:Poisson_departure_process_state-dependent_service_times}%
Consider a queue where jobs arrive according to a Poisson process with rate $\la$ and leave at rate $\mu_i$ when $i$ jobs are in the system. In equilibrium, the inter-departure times of jobs are exponentially distributed with mean $1/\la$ and are independent of the number of jobs in the system.
\end{theorem}%

\begin{proof}%
Denote by $X(t)$ the number of jobs in the system at time $t$. The system is in equilibrium, which is equivalent to $X(0)$ being distributed according to the equilibrium distribution $\pb$. Let $T$ be the time at which the first departure occurs and recall that $X(T)$ is the number of jobs left behind by the first departure. Define the conditional joint transform
\begin{equation}%
R_i(\LSTarg,\PGFarg) \defi \E{\euler^{-\LSTarg T} \PGFarg^{X(T)} \mid X(0) = i}, \quad i \ge 0.
\end{equation}%
For $i \ge 1$, either the first event is an arrival with probability $\la/(\la + \mu_i)$ or a departure with probability $\mu/(\la + \mu_i)$. So, by the strong Markov property, for $i \ge 1$,
\begin{align}%
R_i(\LSTarg,\PGFarg) &= \frac{\la + \mu_i}{\la + \mu_i + \LSTarg} \Bigl( \frac{\la}{\la + \mu_i} R_{i + 1}(\LSTarg,\PGFarg) + \frac{\mu_i}{\la + \mu_i} \PGFarg^{i - 1} \Bigr) \notag \\
&=  \frac{\la}{\la + \mu_i + \LSTarg} R_{i + 1}(\LSTarg,\PGFarg) + \frac{\mu_i}{\la + \mu_i + \LSTarg} \PGFarg^{i - 1}
\end{align}%
and
\begin{equation}%
R_0(\LSTarg,\PGFarg) = \frac{\la}{\la + \LSTarg} R_1(\LSTarg,\PGFarg).
\end{equation}%
This gives the functional equations
\begin{align}%
(\la + \LSTarg) R_0(\LSTarg,\PGFarg) &= \la R_1(\LSTarg,\PGFarg), \label{eqnRN:Poisson_departure_process_state-dependent_service_times_functional_equation_i=0}\\
(\la + \mu_i + \LSTarg) R_i(\LSTarg,\PGFarg) &= \la R_{i + 1}(\LSTarg,\PGFarg) + \mu_i \PGFarg^{i - 1}, \quad i \ge 1. \label{eqnRN:Poisson_departure_process_state-dependent_service_times_functional_equation_i>0}
\end{align}%
Define the PGF
\begin{equation}%
\PGF{X}{\PGFarg} \defi \sum_{i \ge 0} p(i) \PGFarg^i
\end{equation}%
and consider
\begin{equation}%
\E{\euler^{-\LSTarg T} \PGFarg^{X(T)}} = \sum_{i \ge 0} p(i) R_i(\LSTarg,\PGFarg).
\end{equation}%
Multiply the $i$-th equation of \eqref{eqnRN:Poisson_departure_process_state-dependent_service_times_functional_equation_i>0} by $p(i)$ and sum over all $i \ge 1$ to obtain
\begin{align}%
&(\la + \LSTarg) \sum_{i \ge 1} p(i) R_i(\LSTarg,\PGFarg) + \sum_{i \ge 1} \mu_i p(i) R_i(\LSTarg,\PGFarg) \notag \\
&= \la \sum_{i \ge 1} p(i) R_{i + 1}(\LSTarg,\PGFarg) + \sum_{i \ge 1} \mu_i p(i) \PGFarg^{i - 1}.
\end{align}%
Adding and subtracting $p(0) R_0(\LSTarg,\PGFarg)$ on the left-hand side, using the local balance equations $\la p(i - 1) = \mu_i p(i), ~ i \ge 1$ and \eqref{eqnRN:Poisson_departure_process_state-dependent_service_times_functional_equation_i=0}, results in
\begin{equation}%
\E{\euler^{-\LSTarg T} \PGFarg^{X(T)}} = \frac{\la}{\la + \LSTarg} \PGF{X}{\PGFarg}.
\end{equation}%
So, the inter-departure time is exponentially distributed with parameter $\la$ and is moreover independent of the number of jobs that are left behind by the departing job.
\end{proof}%

%%%%%%%%%%%%%%%%%%%%%%%%%%%%%%%%%%%%%%%%%%%%%%%%%%%%%%%
%%%%%%%%%%%%%%%%%%%%%%%%%%%%%%%%%%%%%%%%%%%%%%%%%%%%%%%
%%%%%%%%%%%%%%%%%%%%% NEW SECTION %%%%%%%%%%%%%%%%%%%%%
%%%%%%%%%%%%%%%%%%%%%%%%%%%%%%%%%%%%%%%%%%%%%%%%%%%%%%%
%%%%%%%%%%%%%%%%%%%%%%%%%%%%%%%%%%%%%%%%%%%%%%%%%%%%%%%

\section{Loss networks}%
\label{secRN:loss_networks}%

A loss network is a stochastic network consisting of nodes with links between nodes and jobs travelling over routes in the network. Jobs for each route arrive according to a Poisson process. A route is described by a number of links and for each link the number of capacity unit that is required to serve the job. A job holds the capacities in each link of its route simultaneously for an exponential amount of time, leaving the system afterwards. The capacity on each link is finite, however. So, an arriving job does not enter the network if it finds that a link on its route does not have enough free capacity. Such a job is lost, and therefore the network is called a \textit{loss network}. Besides the equilibrium distribution, a key quantity in these networks is the probability that a job is lost.

\begin{example}[A loss network]\label{exRN:simple_loss_network}%
\begin{figure}
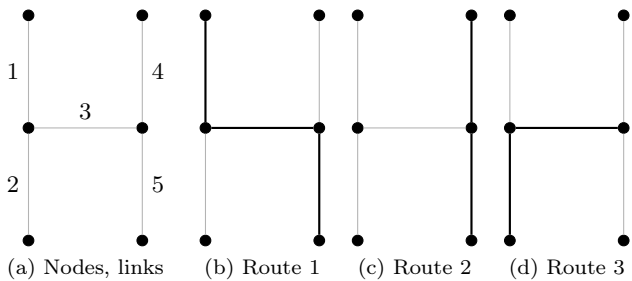
%
\centering%
\subfloat[Nodes, links]{%
\includestandalone{Chapters/RN/TikZFiles/loss_network_example}%
}%
\quad%
\subfloat[Route $1$]{%
\includestandalone{Chapters/RN/TikZFiles/loss_network_example_route_A}%
}%
\quad%
\subfloat[Route $2$]{%
\includestandalone{Chapters/RN/TikZFiles/loss_network_example_route_B}%
}%
\quad%
\subfloat[Route $3$]{%
\includestandalone{Chapters/RN/TikZFiles/loss_network_example_route_C}%
}%
\caption{A simple loss network.}%
\label{figRN:loss_network_example}%
\end{figure}%
Consider a network of six nodes and links with capacities as shown in \cref{figRN:loss_network_example}(a). There are three different routes in this network, see \cref{figRN:loss_network_example}(b)-(d). Jobs for route 1 use the links 1, 3 and 5, arrive according to a Poisson process with rate $\la_1$ and, if admitted, hold simultaneously one unit of capacity on all three links in its route for an exponential amount of time with parameter $\mu_1$. So, an arriving route-1 job is lost if there is no capacity available on links 1, 3 or 5. Route-2 jobs share link 5 with route-1 jobs and route-3 jobs share link 3 with route-1 jobs, but route-2 and route-3 jobs do not share a link. Nonetheless, there is still a large influence of route-2 jobs on the performance of route-3 jobs and vice versa. For example, if the arrival rate of route-3 jobs is large, all the capacity units of link 3 will be occupied. This means that almost all route-1 jobs will be lost and as a result, almost all route-2 jobs are admitted.
\end{example}%

In the following subsection we treat an example in greater detail.

%%%%%%%%%%%%%%%%%%%%%%%%%%%%%%%%%%%%%%%%%%%%%%%%%%%%%%%
%%%%%%%%%%%%%%%%%%%%%%%%%%%%%%%%%%%%%%%%%%%%%%%%%%%%%%%
%%%%%%%%%%%%%%%%%%% NEW SUBSECTION %%%%%%%%%%%%%%%%%%%%
%%%%%%%%%%%%%%%%%%%%%%%%%%%%%%%%%%%%%%%%%%%%%%%%%%%%%%%
%%%%%%%%%%%%%%%%%%%%%%%%%%%%%%%%%%%%%%%%%%%%%%%%%%%%%%%

\subsection{Multi-class Erlang-B model}%
\label{subsecRN:multi-class_Erlang-B_model}%

Consider a pool of $c$ identical servers offered traffic from $M$ job classes and denote the set of classes as $\set{M} \defi \{1,2,\ldots,M\}$. Class-$m$ jobs arrive according to a Poisson process with rate $\la_m$ and require an exponentially distributed service time with parameter $\mu_m$. Denote by $\rho_m \defi \la_m/\mu_m$ the offered traffic from class-$m$ jobs. A class-$m$ job requires the simultaneous use of $b_m$ servers for the duration of its service. Arriving jobs for which there are not sufficiently many servers available leave the system immediately.

%Such models are commonly encountered in communication networks, where transmissions, such as calls, occupy a link for a variable amount of time, depending on the type of transmission.

The state of the system at time $t$ may be described by a vector $X(t) \defi (X_1(t),X_2(t),\ldots,X_M(t))$ with $X_m(t)$ representing the number of class-$m$ jobs in the system at time $t$. Define
\begin{equation}%
\statespace \defi \{ \vca{x} \in \Nat_0^M : \sum_{m = 1}^M b_m x_m \le c \}
\end{equation}%
as the set of all feasible states. The process $\{ X(t) \}_{t \ge 0}$ is an irreducible Markov process with state space $\statespace$. Since its state space is finite, the equilibrium probabilities, now denoted as $p(\vca{x})$, exist.

Let $\eb{m}$ be a vector of dimension $M$ with a $1$ at position $m$, where indexing starts at 1. The equilibrium distribution satisfies the global balance equations
\begin{align}%
&\Bigl( \sum_{m = 1}^M \la_m \ind{\vca{x} + \eb{m} \in \statespace} + \sum_{m = 1}^M x_m \mu_m \Bigr) p(\vca{x}) \notag \\
&= \sum_{m = 1}^M \la_m \ind{x_m > 0} p(\vca{x} - \eb{m}) \notag \\
&\quad + \sum_{m = 1}^M (x_m + 1) \mu_m \ind{\vca{x} + \eb{m} \in \statespace} p(\vca{x} + \eb{m}), \label{eqnRN:multi-class_Erlang-B_global_balance}
\end{align}%
for all states $\vca{x} \in \statespace$, together with the normalization condition
\begin{equation}%
\sum_{\vca{x} \in \statespace} p(\vca{x}) = 1. \label{eqnRN:multi-class_Erlang-B_normalization_condition}
\end{equation}%
Let us try to solve for $p(\vca{x})$ using an educated guess. If there would be infinite number of servers, then jobs of all classes are allowed to enter the system, removing dependencies between classes and we would expect a product-form solution. So, let us see if a product-form solution works here as well. In particular, use the form
\begin{equation}%
p(\vca{x}) = \frac{1}{G(c,M)} \prod_{m = 1}^M \frac{y_m^{x_m}}{x_m!}, \label{eqnRN:multi-class_Erlang-B_educated_guess}
\end{equation}%
where $y_m$ still needs to be determined and $G(c,M)$ is a normalization constant. Assume all indicator functions in \eqref{eqnRN:multi-class_Erlang-B_global_balance} evaluate to 1. This indicates that we are in the interior of the state space $\statespace$. Plugging \eqref{eqnRN:multi-class_Erlang-B_educated_guess} into \eqref{eqnRN:multi-class_Erlang-B_global_balance} and multiplying both sides by $G(c,M)$ gives
\begin{align}%
&\sum_{m = 1}^M \la_m \prod_{n = 1}^M \frac{y_n^{x_n}}{x_n!} + \sum_{m = 1}^M x_m \mu_m \prod_{n = 1}^M \frac{y_n^{x_n}}{x_n!} \notag \\
&= \sum_{m = 1}^M \la_m \frac{y_m^{x_m - 1}}{(x_m - 1)!} \prod_{n \in \set{M} \setminus \{ m \}} \frac{y_n^{x_n}}{x_n!} \notag \\
&\quad + \sum_{m = 1}^M (x_m + 1) \mu_m \frac{y_m^{x_m + 1}}{(x_m + 1)!} \prod_{n \in \set{M} \setminus \{ m \}} \frac{y_n^{x_n}}{x_n!}.
\end{align}%
If we now choose $y_m = \rho_m$, then the first summation on the left-hand side is equal to the second summation on the right-hand side and the second summation on the left-hand side is equal to the first summation on the right-hand side. We conclude that
\begin{equation}%
p(\vca{x}) = \frac{1}{G(c,M)} \prod_{m = 1}^M \frac{\rho_m^{x_m}}{x_m!}, \quad \vca{x} \in \statespace \label{eqnRN:multi-class_Erlang-B_educated_guess_correct}
\end{equation}%
satisfies \eqref{eqnRN:multi-class_Erlang-B_global_balance} if all indicator functions evaluate to 1, but can also be shown to satisfy \eqref{eqnRN:multi-class_Erlang-B_global_balance} if this assumption is dropped. The normalization constant follows from the normalization condition \eqref{eqnRN:multi-class_Erlang-B_normalization_condition} and is
\begin{equation}%
G(c,M) = \sum_{\vca{x} \in \statespace} \prod_{m = 1}^M \frac{\rho_m^{x_m}}{x_m!}. \label{eqnRN:multi-class_Erlang-B_normalization_constant}
\end{equation}%

We now consider the system occupancy in terms of the number of busy servers. Denote by $p(i)$ the probability that $i$ servers are busy for $i = 0,1,\ldots,c$. Define
\begin{equation}%
\statespace_i \defi \{ \vca{x} \in \statespace : \sum_{m = 1}^M b_m x_m = i \}
\end{equation}%
as the set of all states with exactly $i$ servers busy. The probabilities $p(i)$ may then be formally expressed in terms of the probabilities $p(\vca{x})$ as
\begin{equation}%
p(i) = \sum_{\vca{x} \in \statespace_i} p(\vca{x}).
\end{equation}%
The probability that a class-$m$ job is blocked can directly be obtained from the probabilities $p(i)$ as
\begin{equation}%
B_m = \sum_{i = c - b_m + 1}^c p(i).
\end{equation}%
Of course, the blocking probability $B_m$ may also be directly expressed in terms of the probabilities $p(\vca{x})$ as
\begin{equation}%
B_m = \sum_{\vca{x} \in \statespace : \vca{x} + \eb{m} \notin \statespace} p(\vca{x}) = 1 - \sum_{\vca{x} + \eb{m} \in \statespace} p(\vca{x}).
\end{equation}%
This last summation can be rewritten as
\begin{equation}%
\sum_{\vca{x} + \eb{m} \in \statespace} p(\vca{x}) = \frac{1}{G(c,M)} \sum_{\vca{x} + \eb{m} \in \statespace} \prod_{n = 1}^M \frac{\rho_n^{x_n}}{x_n!} = \frac{G(c - b_m,M)}{G(c,M)}.
\end{equation}%
Summarizing, the blocking probabilities $B_m$ can be obtained from the ratio of the normalization constants for two systems with a different number of servers.

\begin{remark}[Insensitivity property]%
In fact, the equilibrium distribution given above holds for \textit{any} service time distribution with mean $1/\mu_m$ (without proof). This means that the stationary distribution only depends on the service time distribution through its mean, and not on any higher moments. This is called an \textit{insensitivity} property that is also encountered in the ordinary Erlang-B model, but also the Erlang-C model.
\end{remark}%

Despite the elegant form, the expression \eqref{eqnRN:multi-class_Erlang-B_normalization_constant} is typically impractical for computing the probabilities $p(i)$ and $B_m$. The number of feasible states in the above model and therefore also the number of terms in the normalization constant, grows rapidly with $c$ and $M$. This makes the numerical evaluation of the normalization constant directly through brute-force summation prohibitively demanding for even moderately large values of $c$ and $M$.

We now discuss an alternative procedure for calculating the probabilities $p(i)$ and the blocking probabilities first described in Kaufman \cite{Kaufman1981_Blocking_recursion} and Roberts \cite{Roberts1981_Blocking_Recursion}.

\begin{lemma}[Kaufman-Roberts recursion]\label{lemRN:multi-class_Erlang-B_Kaufman-Roberts_recursion}%
The probabilities $p(i)$ satisfy the recurrence relation
\begin{equation}%
i p(i) = \sum_{m = 1}^M \rho_m b_m \ind{i \ge b_m} p_{i - b_m}, \quad i = 0,1,\ldots,c.
\end{equation}%
\end{lemma}%

\begin{proof}%
Define
\begin{equation}%
R_m(i) \defi \sum_{\vca{x} \in \statespace_i} x_m p(\vca{x}). \label{eqnRN:multi-class_Erlang-B_Kaufman-Roberts_recursion_proof_1}
\end{equation}%
First observe that, from the definition of $p(i)$ and $\statespace_i$,
\begin{align}%
i p(i) &= \sum_{\vca{x} \in \statespace_i} i p(\vca{x}) = \sum_{\vca{x} \in \statespace_i} \sum_{m = 1}^M b_m x_m p(\vca{x}) \notag \\
&= \sum_{m = 1}^M b_m \sum_{\vca{x} \in \statespace_i} x_m p(\vca{x}) = \sum_{m = 1}^M b_m R_m(i). \label{eqnRN:multi-class_Erlang-B_Kaufman-Roberts_recursion_proof_2}
\end{align}%
From \eqref{eqnRN:multi-class_Erlang-B_educated_guess_correct}, we have
\begin{align}%
x_n p(\vca{x}) &= \frac{x_n}{G(c,M)} \prod_{m = 1}^M \frac{\rho_m^{x_m}}{x_m!} \notag \\
&= \frac{\rho_n}{G(c,M)} \frac{\rho_n^{x_n - 1}}{(x_n - 1)!} \prod_{m \in \set{M} \setminus \{ n \}} \frac{\rho_m^{x_m}}{x_m!} = \rho_n p(\vca{x} - \eb{n}). \label{eqnRN:multi-class_Erlang-B_Kaufman-Roberts_recursion_proof_3}
\end{align}%
Substituting \eqref{eqnRN:multi-class_Erlang-B_Kaufman-Roberts_recursion_proof_3} into \eqref{eqnRN:multi-class_Erlang-B_Kaufman-Roberts_recursion_proof_1} yields
\begin{equation}%
R_m(i) = \rho_m \sum_{\vca{x} \in \statespace_i} p(\vca{x} - \eb{m}) = \rho_m \sum_{\vca{x} \in \statespace_{i - b_m}} p(\vca{x}) = \rho_m \ind{i \ge b_m} p_{i - b_m}. \label{eqnRN:multi-class_Erlang-B_Kaufman-Roberts_recursion_proof_4}
\end{equation}%
Plugging \eqref{eqnRN:multi-class_Erlang-B_Kaufman-Roberts_recursion_proof_4} into \eqref{eqnRN:multi-class_Erlang-B_Kaufman-Roberts_recursion_proof_2} proves the claim.
\end{proof}%

%%%%%%%%%%%%%%%%%%%%%%%%%%%%%%%%%%%%%%%%%%%%%%%%%%%%%%%
%%%%%%%%%%%%%%%%%%%%%%%%%%%%%%%%%%%%%%%%%%%%%%%%%%%%%%%
%%%%%%%%%%%%%%%%%%% NEW SUBSECTION %%%%%%%%%%%%%%%%%%%%
%%%%%%%%%%%%%%%%%%%%%%%%%%%%%%%%%%%%%%%%%%%%%%%%%%%%%%%
%%%%%%%%%%%%%%%%%%%%%%%%%%%%%%%%%%%%%%%%%%%%%%%%%%%%%%%

\subsection{Equilibrium distributions for loss networks}%
\label{subsecRN:loss_networks_equilibrium_distributions}%

The multi-class Erlang-B model described in \cref{subsecRN:multi-class_Erlang-B_model} may be interpreted as a single `link' or transmission resource with $c$ `circuits' or `trunks' (represented by the servers) offered `calls' or `connections' (represented by the jobs) from $M$ classes. The single-link model may be generalized to networks of multiple links, where the various classes correspond to jobs that may traverse different routes (subsets of links), require different numbers of circuits, or a combination of these two features. Specifically, consider a network consisting of $L$ links indexed by the set $\set{L} \defi \{ 1,2,\ldots,L \}$, offered traffic from $M$ distinct job classes. Denote by $c_l$ the capacity of, or, number of circuits in, link $l$. Class-$m$ jobs arrive as a Poisson process with rate $\lambda_m$, and have exponentially distributed holding times with parameter $\mu_m$. Denote by $\rho_m \defi \la_m/\mu_m$ the offered traffic from class-$m$ jobs. Class-$m$ jobs require the simultaneous use of $b_{m,l}$ circuits on link $l$ for the duration of their holding time. Arriving jobs for which there are not sufficiently many circuits available leave the system immediately. The set of links $R_m \defi \{l \in \set{L} : b_{m,l} > 0 \}$ may be interpreted as the route of class-$k$ jobs. The route sets $R_m$ need to satisfy certain `logical' constraints in order for routes to be contiguous paths in some underlying physical network topology. However, the subsequent analysis applies for completely arbitrary values of $b_{m,l}$.

It is easily verified that the analysis in \cref{subsecRN:multi-class_Erlang-B_model} for the single-link model, in particular the equilibrium distribution given in \eqref{eqnRN:multi-class_Erlang-B_educated_guess_correct}, readily extends to the above network scenario, with state space now replaced by
\begin{equation}%
\statespace \defi \{ \vca{x} \in \Nat_0^M : \sum_{m = 1}^M b_{m,l} x_m \le c_l \textup{ for all } l \in \set{L} \}.
\end{equation}%

\begin{theorem}\label{thmRN:loss_networks_equilibrium_distribution}%
The equilibrium distribution of a loss network is given by
\begin{equation}%
p(\vca{x}) = \frac{1}{G} \prod_{m = 1}^M \frac{\rho_m^{x_m}}{x_m!},
\end{equation}%
with normalization constant
\begin{equation}%
G = \sum_{\vca{x} \in \statespace} \prod_{m = 1}^M \frac{\rho_m^{x_m}}{x_m!}.
\end{equation}%
\end{theorem}%

\begin{proof}%
We now give a proof using the concept of reversibility. Consider the case with infinite capacity $c_1 = c_2 = \cdots = c_L = \infty$ (abbreviated as ic). In this case all jobs are accepted to the system and jobs in different classes are independent of each other. By this independence we have a product-form solution originating from the $M/M/\infty$ queue:
\begin{equation}%
p_{\textup{ic}}(\vca{x}) = \prod_{m = 1}^M \euler^{-\rho_m} \frac{\rho_m^{x_m}}{x_m!}.
\end{equation}%
Truncating the state space from $\Nat_0^M$ to $\statespace$ and using \cref{thmRN:truncating_reversible_process} gives the result.
\end{proof}%

Clearly, the evaluation of the normalization constant $G$ will be even more computationally demanding than in the single-link model. In general, there is no efficient numerical equivalent of the Kaufman-Roberts recursion presented in \cref{lemRN:multi-class_Erlang-B_Kaufman-Roberts_recursion}. In the important special case where $b_{m,l} \in \{ 0, 1 \}$ for all $m \in \set{M}$ and $l \in \set{L}$, the blocking probabilities for the various classes may be approximated using the so-called Erlang fixed-point approximation.

Denote the blocking probability on link $l$ as $B_l'$. Then the probability of a class-$m$ job being blocked is expressed in terms of these link blocking probabilities as
\begin{equation}%
1 - B_m = \prod_{l \in \set{L} \,:\, b_{m,l} = 1} (1 - B_l'),
\end{equation}%
since each link in the class-$m$ route needs to have at least one unit of capacity available. Now, assume that the blocking probabilities $B_l'$ of blocking on link $l$ are independent from link to link (which in a real network they are not!). In that case, the traffic offered to link $l$ would be Poisson with rate
\begin{equation}%
\sigma_l = \sum_{m = 1}^M b_{m,l} \rho_m \prod_{k \in \set{L} \setminus \{ l \} \,:\, b_{m,k} = 1} (1 - B_k'). \label{eqnRN:loss_networks_relation_link_class_blocking_probabilities}
\end{equation}%
By the Erlang-B formula, see also \cref{exBD:MMss_equilibrium_dist}, the link blocking probabilities satisfy
\begin{equation}%
B_l' = B(c_l,\sigma_l) = \frac{\sigma_l^{c_l}}{c_l!} \Bigl( \sum_{j = 0}^{c_l} \frac{\sigma_l^j}{j!} \Bigr)^{-1}, \quad l \in \set{L}. \label{eqnRN:Erlang_fixed-point_approximation}
\end{equation}%
A unique solution to these equations exists and therefore we are able to obtain the blocking probabilities for each class of jobs.\endnote{The Brouwer fixed-point theorem states that a continuous map from a compact, convex set to itself has at least one fixed point. In our case \eqref{eqnRN:Erlang_fixed-point_approximation} defines a continuous map $F : [0,1]^L \to [0,1]^L$ and $[0,1]^L$ is compact and convex, so at least one solution to \eqref{eqnRN:Erlang_fixed-point_approximation} exists, see also \cite[Section~3.2]{Kelly2014_Stochastic_Networks}. The uniqueness of this solution is established in \cite[Theorem~3.20]{Kelly2014_Stochastic_Networks}.}

\begin{example}[Blocking probabilities in a simple loss network]\label{exRN:simple_loss_network_blocking_probabilities}%
\begin{table}%
\centering%
\begin{tabular}{*{6}{c}}%
Iteration & $B_1'$ & $B_2'$ & $B_3'$ & $B_4'$ & $B_5'$ \\
\hline
1 & 6.25\% & 6.25\% & 21.05\% & 6.25\% & 21.05\% \\
2 & 2.17\% & 3.76\% & 16.15\% & 3.76\% & 15.15\% \\
5 & 2.70\% & 4.16\% & 17.47\% & 4.16\% & 17.47\% \\
10 & 2.68\% & 4.15\% & 17.44\% & 4.15\% & 17.44\% \\
20 & 2.68\% & 4.15\% & 17.44\% & 4.15\% & 17.44\%
\end{tabular}%
\caption{Solving \protect\eqref{eqnRN:Erlang_fixed-point_approximation} for \protect\cref{exRN:simple_loss_network_blocking_probabilities} using successive substitutions.}%
\label{tblRN:iterations_Erlang_fixed-point_simple_loss_network}%
\end{table}%

Consider again the loss network of \cref{exRN:simple_loss_network} and \cref{figRN:loss_network_example}, where all three classes require one unit of capacity at each link in their route. Set $c_l = 3, ~ l \in \set{L}$ and $\rho_m = 1, ~ m \in \set{M}$. Under the assumption of independent blocking probabilities, the Poisson traffic offered to each link is
\begin{gather}%
\sigma_1 = (1 - B_3')(1 - B_5'), ~ \sigma_2 = (1 - B_3'), ~ \sigma_3 = (1 - B_1')(1 - B_5') + (1 - B_2'), \notag \\
\sigma_4 = (1 - B_5'), ~ \sigma_5 = (1 - B_1')(1 - B_3') + (1 - B_4').
\end{gather}%
We wish to determine these link blocking probabilities through the Erlang fixed-point equations \eqref{eqnRN:Erlang_fixed-point_approximation}. A possible method of obtaining the solution is through straightforward successive substitutions. This method, however, does not guarantee convergence to the solution, but usually works in practice. Let us take this approach and use as an initial guess $B_l' = 0, ~ l \in \set{L}$, see \cref{tblRN:iterations_Erlang_fixed-point_simple_loss_network}. From this approximation we find that link 3 and 5 are blocked most often and there are two pairs of links that have the same blocking probabilities. The last observation can be explained by the fact that both pairs of links are on a route consisting of three links and a route of two links, and furthermore, the load offered by each class is the same. The blocking probabilities for a class (or route) are calculated from \eqref{eqnRN:loss_networks_relation_link_class_blocking_probabilities}: $B_1 = 33.67\%$, $B_2 = B_3 = 20.87\%$.
\end{example}%

%%%%%%%%%%%%%%%%%%%%%%%%%%%%%%%%%%%%%%%%%%%%%%%%%%%%%%%
%%%%%%%%%%%%%%%%%%%%%%%%%%%%%%%%%%%%%%%%%%%%%%%%%%%%%%%
%%%%%%%%%%%%%%%%%%%%% NEW SECTION %%%%%%%%%%%%%%%%%%%%%
%%%%%%%%%%%%%%%%%%%%%%%%%%%%%%%%%%%%%%%%%%%%%%%%%%%%%%%
%%%%%%%%%%%%%%%%%%%%%%%%%%%%%%%%%%%%%%%%%%%%%%%%%%%%%%%

\section{Jackson networks}%
\label{secRN:jackson_networks}%

In this section we consider the class of so-called Jackson networks, named after the queueing theorist J.R.~Jackson. A Jackson network consists of $M$ queues (or stations) with possibly state-dependent service rates. Specifically, when there are a total of $x_m$ jobs at queue $m$, the service rate is $v_m(x_m)$, with $v_m(0) = 0, ~ m = 1,2,\ldots,M$. Note that for example $v_m(x_m) = \min(x_m,s_m)$ models a situation where queue $m$ has $s_m$ identical servers. The service times at queue $m$ are independent and exponentially distributed with parameter $\mu_m$. At each of the queues, the jobs are served in order of arrival. Upon service completion at queue $m$, jobs either proceed to queue $n$ with probability $r_{m,n}$ or leave the system with probability $r_{m,0} = 1 - \sum_{n = 1}^M p_{m,n}$, where `0' refers to outside the network. The probabilities $r_{m,n}$ are commonly called routing probabilities, and the $M \times M$ matrix $R = [r_{m,n}]_{m,n = 1,2,\ldots,M}$ the routing matrix. Jobs can arrive from outside the network to any of the queues in the network.

Let us first treat two examples of Jackson networks.

%%%%%%%%%%%%%%%%%%%%%%%%%%%%%%%%%%%%%%%%%%%%%%%%%%%%%%%
%%%%%%%%%%%%%%%%%%%%%%%%%%%%%%%%%%%%%%%%%%%%%%%%%%%%%%%
%%%%%%%%%%%%%%%%%%% NEW SUBSECTION %%%%%%%%%%%%%%%%%%%%
%%%%%%%%%%%%%%%%%%%%%%%%%%%%%%%%%%%%%%%%%%%%%%%%%%%%%%%
%%%%%%%%%%%%%%%%%%%%%%%%%%%%%%%%%%%%%%%%%%%%%%%%%%%%%%%

\subsection{Tandem queues}%
\label{subsecRN:tandem_queues}%

Consider a system of $M$ queues in series with $s_m$ servers at the $m$-th queue. Jobs arrive to the first queue according to a Poisson process with rate $\la$ and require independent and exponentially distributed service times with parameter $\mu_m$ at the $m$-th queue. Upon service completion at the $m$-th queue, jobs proceed to the $(m + 1)$-th queue, $m = 1,2,\ldots,M - 1$ and a service completion at the final queue leads to the job departing the system. Define $\rho_m \defi \la/\mu_m$ as the offered load at the $m$-th queue. For stability, assume $\rho_m < s_m$ for all $m = 1,2,\ldots,M$.here should a homogeneous structure in terms of the transitions. That is, the transition structure and the rate at which these transitions occur should be the same for all states in the interior; for all states on the vertical boundary; and for all states on the horizontal boundary.

The state of the system at time $t$ may be described by a vector $X(t) = (X_1(t),X_2(t),\ldots,X_M(t))$ with $X_m(t)$ representing the number of jobs at the $m$-th queue at time $t$. It is easily verified that the process $\{ X(t) \}_{t \ge 0}$ is a Markov process with state space $\statespace = \Nat_0^M$. Denote by $p(\vca{x})$ the equilibrium probability of being in state $\vca{x} \in \statespace$. The equilibrium probabilities satisfy the global balance equations
\begin{align}%
&\Bigl( \la + \sum_{m = 1}^M \min( x_m, s_m ) \mu_m \Bigr) p(\vca{x}) = \la \ind{x_1 > 0} p(\vca{x} - \eb{1}) \notag \\
&\quad + \sum_{m = 1}^{M - 1} \min( x_m + 1, s_m ) \mu_m \ind{x_{m + 1} > 0} p(\vca{x} + \eb{m} - \eb{m + 1}) \notag \\
&\quad + \min( x_M + 1, s_M ) \mu_M p(\vca{x} + \eb{M}), \label{eqnRN:tandem_queues_global_balance}
\end{align}%
for all states $\vca{x} \in \statespace$, and the normalization condition
\begin{equation}%
\sum_{\vca{x} \in \statespace} p(\vca{x}) = 1.
\end{equation}%
It is easily verified through substitution, as we did in \cref{subsecRN:multi-class_Erlang-B_model}, that the equilibrium distribution is a product-form solution
\begin{equation}%
p(\vca{x}) = \prod_{m = 1}^M p_m(x_m),
\end{equation}%
with
\begin{equation}%
p_m(i) = \begin{cases}%
p_m(0) \frac{\rho_m^i}{i!}, & 0 \le i \le s_m - 1, \\
p_m(0) \frac{\rho_m^i}{s_m! s_m^{i - s_m}}, & i \ge s_m,
\end{cases}%
\end{equation}%
and
\begin{equation}%
p_m(0) = \Bigl( \sum_{i = 0}^{s_m - 1} \frac{\rho_m^i}{i!} + \frac{\rho_m^{s_m}}{s_m!} \frac{1}{1 - \frac{\rho_m}{s_m}} \Bigr)^{-1}.
\end{equation}%
This indicates that the number of jobs at the various stations are independent, and the number of jobs at the $m$-th station is distributed as the number of jobs in an isolated queue with $s_m$ servers, Poisson arrival at rate $\la$ and exponentially distributed service times with parameter $\mu_m$. Looking back at \cref{thmRN:Poisson_departure_process_state-dependent_service_times}, the departure process of the first queue, which is the arrival process of the second queue, is Poisson with rate $\la$ and moreover independent of the number of jobs in the first queue. So, we could have expected the product-form equilibrium distribution.

%%%%%%%%%%%%%%%%%%%%%%%%%%%%%%%%%%%%%%%%%%%%%%%%%%%%%%%
%%%%%%%%%%%%%%%%%%%%%%%%%%%%%%%%%%%%%%%%%%%%%%%%%%%%%%%
%%%%%%%%%%%%%%%%%%% NEW SUBSECTION %%%%%%%%%%%%%%%%%%%%
%%%%%%%%%%%%%%%%%%%%%%%%%%%%%%%%%%%%%%%%%%%%%%%%%%%%%%%
%%%%%%%%%%%%%%%%%%%%%%%%%%%%%%%%%%%%%%%%%%%%%%%%%%%%%%%

\subsection{Closed tandem queues}%
\label{subsecRN:closed_tandem_queues}%

Suppose that the previous tandem queue is modified as follows. Instead of a Poisson arrival process, we assume that there is a finite population of $K$ jobs circulating through the tandem queues. Upon service completion at the $M$-th queue, jobs return to the first queue. To avoid trivialities, $K > \min(s_1,s_2,\ldots,s_M)$, because otherwise there is no interaction between jobs, and each of the simply cycle through the $M$ queues, independently of all others. We no longer need to assume that $\rho_m < s_m$ since this system is always stable. This system is called a \textit{closed} system because no outside arrivals are allowed into the system. The previous tandem queue is then aptly named \textit{open}.

As before, let $X(t)$ be the vector of the number of jobs at time $t$ at each queue. The process $\{ X(t) \}_{t \ge 0}$ is a Markov process on the state space
\begin{equation}%
\statespace = \{ \vca{x} \in \Nat_0^M : \sum_{m = 1}^M x_m = K \},
\end{equation}%
with equilibrium probabilities $p(\vca{x})$. These probabilities satisfy the global balance equations
\begin{align}%
&\sum_{m = 1}^M \min( x_m, s_m ) \mu_m p(\vca{x}) \notag \\
&= \sum_{m = 1}^{M - 1} \min( x_m + 1, s_m ) \mu_m \ind{x_{m + 1} > 0} p(\vca{x} + \eb{m} - \eb{m + 1}) \notag \\
&\quad + \min( x_M + 1, s_M ) \mu_M \ind{x_1 > 0} p(\vca{x} + \eb{M} - \eb{1}), \label{eqnRN:closed_tandem_queues_global_balance}
\end{align}%
for all states $\vca{x} \in \statespace$, and the normalization condition
\begin{equation}%
\sum_{\vca{x} \in \statespace} p(\vca{x}) = 1.
\end{equation}%
For the closed tandem queuing network, the equilibrium distribution is given by
\begin{equation}%
p(\vca{x}) = \frac{1}{G} \prod_{m = 1}^M p_m(x_m),
\end{equation}%
with
\begin{equation}%
p_m(i) = \begin{cases}%
\frac{\rho_m^i}{i!}, & 0 \le i \le s_m - 1, \\
\frac{\rho_m^i}{s_m! s_m^{i - s_m}}, & i \ge s_m,
\end{cases}%
\end{equation}%
with the normalization constant
\begin{equation}%
G = \sum_{\vca{x} \in \statespace} \prod_{m = 1}^M p_m(x_m),
\end{equation}%
and $\rho_m \defi \gamma/\mu_m$ for some arbitrary constant $\gamma > 0$. This constant is arbitrary since it appears only as $\gamma^K$ in both the numerator and denominator of $p(\vca{x})$.

Obviously, the number of jobs at the various stations are no longer independent, but the equilibrium distribution retains a remarkably simple structure. It looks as if we applied \cref{thmRN:truncating_reversible_process} to the equilibrium distribution of the open tandem queue to obtain the equilibrium distribution of the closed tandem queue. However, \cref{thmRN:truncating_reversible_process} requires the Markov process to be reversible, but that is not the case here. Consider $M = 3$ and examine the transition rate from state $\vca{y} = (1,0,0)$ to state $\vca{z} = (0,1,0)$, which is $\mu_1$. For the Markov process to be reversible, we require
\begin{equation}%
p(\vca{y}) q_{\vca{y},\vca{z}} = p(\vca{z}) q_{\vca{z},\vca{y}},
\end{equation}%
but $q_{\vca{z},\vca{y}} = 0$. So, the Markov process associated with the open tandem queue is in general not reversible and \cref{thmRN:truncating_reversible_process} cannot be applied.

%%%%%%%%%%%%%%%%%%%%%%%%%%%%%%%%%%%%%%%%%%%%%%%%%%%%%%%
%%%%%%%%%%%%%%%%%%%%%%%%%%%%%%%%%%%%%%%%%%%%%%%%%%%%%%%
%%%%%%%%%%%%%%%%%%% NEW SUBSECTION %%%%%%%%%%%%%%%%%%%%
%%%%%%%%%%%%%%%%%%%%%%%%%%%%%%%%%%%%%%%%%%%%%%%%%%%%%%%
%%%%%%%%%%%%%%%%%%%%%%%%%%%%%%%%%%%%%%%%%%%%%%%%%%%%%%%

\subsection{Open Jackson networks}%
\label{subsecRN:open_Jackson_networks}%

The tandem queue of \cref{subsecRN:tandem_queues} belongs to the class of open Jackson networks. As it turns out, open Jackson networks have a similar product-form solution for the equilibrium distribution. Here we treat open Jackson networks in full.

In open Jackson networks jobs arrive from the external environment, and eventually leave the system. Specifically, jobs are assumed to arrive at queue $m$ as a Poisson process with rate $\la_m$, $m = 1,2,\ldots,M$. We will assume that $r_{m,0} > 0$ for at least one value of $m$, because otherwise it would be impossible for jobs to leave, and the system would definitely be unstable.

Denote by $\La_m$ the total arrival rate at queue $m$, including both external arrivals and transitions from other queues or queue $m$ itself. In case the system is stable, $\La_m$ must equal the total departure rate at queue $m$, including both external departures and transitions to other queues or queue $m$ itself, and will also be called the \textit{throughput} of queue $m$. In case the system is stable, the throughputs satisfy the following set of linear equations describing the flow of jobs through the system, the so-called traffic equations,
\begin{equation}%
\La_m = \la_m + \sum_{n = 1}^M \La_n r_{n,m}, \quad m = 1,2,\ldots,M,
\end{equation}%
which may be written in vector-matrix notation as $\vc{\La} = \vc{\la} + \vc{\La} R$, or equivalently $\vc{\La} (\I - R) = \vc{\la}$, with $\I$ the identity matrix, $\vc{\La} = [\La_m]_{m = 1,2,\ldots,M}$ the throughput vector, and $\vc{\la} = [\la_m]_{m = 1,2,\ldots,M}$ the vector of exogenous arrival rates. The assumption that $r_{m,0} > 0$ for at least one value of $m$ implies that the matrix $R$ has spectral radius strictly less than unity, and ensures that the matrix $\I - R$ has a positive inverse, so that the throughput vector may be expressed as $\vc{\La} = \vc{\la} (\I - R)^{-1}$. Note that the service rates $v_m(\cdot)$ and parameters $\mu_m$ do not occur in the traffic equations, but of course they do determine whether or note the system is stable, and in turn determine when the traffic equations actually apply. Without proof, we state that the system is stable if $\La_m < \mu_m v_m^*$ for all $m = 1,2,\ldots,M$, with $v_m^* \defi \liminf_{x \to \infty} v_m(x)$. In particular, when $r_m(x) = \min(x,s_m)$, the system is stable when $\La_m < \mu_m s_m$.

The state of the system at time $t$ may be described by a vector $X(t) \defi (X_1(t),X_2(t),\ldots,X_M(t))$, with $X_m(t)$ the total number of jobs present at queue $m$ at time $t$. It is easily verified that the process $\{ X(t) \}_{t \ge 0}$ is a Markov process with state space $\statespace = \Nat_0^M$. Assuming the stability condition to be fulfilled, denote by $p(\vca{x})$ the equilibrium probability that the system is in state $\vca{x}$. These probabilities satisfy the global balance equations
\begin{align}%
&\sum_{m = 1}^M \bigl( \la_m + \mu_m r_m(x_m) \bigr) p(\vca{x}) = \sum_{m = 1}^M \la_m \ind{x_m > 0} p(\vca{x} - \eb{m}) \notag \\
&\quad + \sum_{m = 1}^M \mu_m r_m(x_m + 1) r_{m,0} p(\vca{x} + \eb{m}) \notag \\
&\quad + \sum_{m = 1}^M \sum_{n = 1}^M \mu_m r_m(x_m + 1) r_{m,n} \ind{x_n > 0} p(\vca{x} + \eb{m} - \eb{n})
\end{align}%
for all states $\vca{x} \in \statespace$, along with the normalization condition
\begin{equation}%
\sum_{\vca{x} \in \statespace} p(\vca{x}) = 1.
\end{equation}%
The equilibrium distribution has the product form
\begin{equation}%
p(\vca{x}) = \frac{1}{G} \prod_{m = 1}^M g_m(x_m),
\end{equation}%
with $\rho_m \defi \La_m/\mu_m$, $g_m(x) \defi \rho_m^x / (\prod_{n = 1}^x v_m(n))$, and
\begin{equation}%
G = \sum_{\vca{x} \in \statespace} \prod_{m = 1}^M g_m(x_m) = \prod_{m = 1}^M \sum_{x_m \ge 0} g_m(x_m) \ifed \prod_{m = 1}^M G_m.
\end{equation}%
Since $\liminf_{x_m \to \infty} v_m(x_m) = v_m^*$ and $\rho_m < v_m^*$ we know that there exists an $\epsilon > 0$ and $N < \infty$ such that $v_m(x_m) > v_m^* - \epsilon > \rho_m$ for $x_m > M$. Now,
\begin{align}%
G_m &= \sum_{k \ge 0} \frac{\rho_m^{k}}{\prod_{l = 1}^{k} v_m(l)} \notag \\
&= \sum_{k = 0}^N \frac{\rho_m^{k}}{\prod_{l = 1}^{k} v_m(l)} + \sum_{k \ge N + 1} \frac{\rho_m^{k}}{\bigl( \prod_{l = 1}^{N} v_m(l) \bigr) \bigl( \prod_{n = N + 1}^{k} v_m(n) \bigr)} \notag \\
&= \sum_{k = 0}^N \frac{\rho_m^{k}}{\prod_{l = 1}^{k} v_m(l)} + \frac{\rho_m^N}{\prod_{l = 1}^{N} v_m(l)} \sum_{k \ge 1} \frac{\rho_m^{k}}{\prod_{n = N + 1}^{N + k} v_m(n)} \notag \\
&< \sum_{k = 0}^N \frac{\rho_m^{k}}{\prod_{l = 1}^{k} v_m(l)} + \frac{\rho_m^N}{\prod_{l = 1}^{N} v_m(l)} \sum_{k \ge 1} \bigl( \frac{\rho_m}{v_m^* - \epsilon} \bigr)^k < \infty.
\end{align}%
So, the assumption $\La_m < \mu_m v_m^* \Leftrightarrow \rho_m < v_m^*$ ensures that $G_m < \infty$. Note that $p(\vca{x})$ may be written as $p(\vca{x}) = \prod_{m = 1}^M \psi_m(x_m)$, where $\psi_m(x) \defi g_m(x)/G_m$. Further observe that $\psi_m(\cdot)$ is the equilibrium distribution of the number of jobs at an isolated queue with a Poisson arrival process with rate $\La_m$, exponentially distributed service times with parameter $\mu_m$ and state-dependent service rate $v_m(\cdot)$.

In case queue $m$ has just a single unit-rate ($v_m(x) = 1$) server, the factor $\rho_m$ represent the utilization of the server at queue $m$, and $G_m = 1/(1 - \rho_m)$, so that its equilibrium distribution (indexed by a subscript $m$) is given by
\begin{equation}%
p_m(x_m) = (1 - \rho_m) \rho_m^{x_m}.
\end{equation}%
In case there are infinitely many servers at queue $m$ ($v_m(x) = x$), we obtain that $G_m = \euler^{\rho_m}$, which means that
\begin{equation}%
p_m(x_m) = \euler^{-\rho_m} \frac{\rho_m^{x_m}}{x_m!}.
\end{equation}%

Now let $p_m(x_m)$ be the equilibrium probability that there are a total of $x_m$ jobs present at queue $m$. Since $\psi_m(\cdot)$ is a probability distribution, it follows that
\begin{align}%
p_m(x_m) &= \sum_{\vca{y} \in \statespace \,:\, y_m = x_m} p(\vca{y}) = \sum_{\vca{y} \in \statespace \,:\, y_m = x_m} \prod_{n = 1}^M \psi_n(y_n) \notag \\
&= \psi_m(x_m) \prod_{n \neq m} \sum_{y_n \ge 0} \psi_n(y_n) = \psi_m(x_m),
\end{align}%
and therefore $p(\vca{x}) = \prod_{m = 1}^M p_m(x_m)$.

This implies two important properties of open Jackson networks. First of all, the total number of jobs present at queue $m$ has the same equilibrium distribution as that of an isolated queue with a Poisson arrival process of rate $\La_m$, exponentially distributed service times with parameter $\mu_m$, and state-dependent service rate $v_m(\cdot)$. Second, the numbers of jobs present at the various queues are independent in equilibrium. These are two crucial properties that however need to be applied and interpreted with caution. For example, the first property might suggest that the aggregate arrival process at queue $m$, including both external arrivals and transitions from other queues, is Poisson with rate $\La_m$. This is indeed the case in some particular Jackson networks such as the open tandem queues considered in this chapter. However, in general this is \textit{not} the case. Also, the second property is rather striking in view of the strong interaction due to the transitions among the various queues. The interaction in fact implies that the state of one queue can influence the state of other queues at future time instants, which might seem to contradict the stated independence. In order to resolve the paradoxical situation, it is critical to note that the independence only holds for the joint number of jobs at each queue at the same time epoch in equilibrium, and \textit{not} for the states of
different queues at different instants in time or in transient regimes.

%%%%%%%%%%%%%%%%%%%%%%%%%%%%%%%%%%%%%%%%%%%%%%%%%%%%%%%
%%%%%%%%%%%%%%%%%%%%%%%%%%%%%%%%%%%%%%%%%%%%%%%%%%%%%%%
%%%%%%%%%%%%%%%%%%% NEW SUBSECTION %%%%%%%%%%%%%%%%%%%%
%%%%%%%%%%%%%%%%%%%%%%%%%%%%%%%%%%%%%%%%%%%%%%%%%%%%%%%
%%%%%%%%%%%%%%%%%%%%%%%%%%%%%%%%%%%%%%%%%%%%%%%%%%%%%%%

\subsection{Closed Jackson networks}%
\label{subsecRN:closed_Jackson_networks}%

In closed Jackson networks there are no external arrivals or departures. Instead, there is a fixed population of $K$ jobs which circulate through the system. In contrast to the case of open networks, we now have
$\sum_{m = 1}^{M} r_{m,n} = 1$ for all $m = 1,2,\ldots,M$; the
routing matrix $R = [r_{m,n}]_{m,n = 1,2,\ldots,M}$ is stochastic. In order to ensure that the equilibrium distribution does not depend on the initial state, we assume that the matrix $R$ is irreducible, which means that the matrix $\I - R$ has rank $M - 1$.

Like in the case of open networks, denote by $\La_m$ the total
arrival rate at queue $m$, now however consisting exclusively of
transitions from other queues or queue $m$ itself. Without any further assumptions, $\La_m$ will be equal to the total departure rate or throughput at queue $m$, again now consisting however exclusively of transitions to other queues or queue $m$ itself. The throughputs satisfy the following set of linear equations, the so-called traffic equations,
\begin{equation}%
\La_m = \sum_{n = 1}^M \La_n r_{n,m}, \quad m = 1,2,\ldots,M, \label{eqnRN:closed_Jackson_networks_traffic_equations}
\end{equation}%
which may be written in vector-matrix notation as $\vc{\La} = \vc{\La} R$, or equivalently $\vc{\La} (\I - R) = \zerob$, with $\vc{\La} = [\La_m]_{m = 1,2,\ldots,M}$ the throughput vector. In contrast to the case of open networks, the traffic equations no longer have a unique solution. Note that scaling a solution, that is, multiplying all throughputs with a common scalar value, will again yield a solution since the traffic equations are first-order homogeneous. Because the matrix $\I - R$ has rank $M - 1$, the traffic equations do however uniquely determine the \textit{relative} values of the throughputs: they determine the throughputs up to a common scaling factor.

As in the case of open networks, the state of the system at time $t$ may be described by $X(t) \defi (X_1(t),X_2(t),\ldots,X_M(t))$, with $X_m(t)$ representing the total number of jobs present at queue $m$ at time $t$. It is easily verified that the process $\{ X(t) \}_{t \ge 0}$ is a Markov process with state space
\begin{equation}%
\statespace \defi \{ \vca{x} \in \Nat_0^M : \sum_{m = 1}^{M} x_m = K \}.
\end{equation}%
Denote by $p(\vca{x})$ the equilibrium probability that the system is in state $\vca{x} \in \statespace$. These probabilities satisfy the global balance equations
\begin{align}%
&\sum_{m = 1}^M \mu_m v_m(x_m) p(\vca{x}) \notag \\
&= \sum_{m = 1}^M \sum_{n = 1}^M \mu_m v_m(x_m + 1) r_{m,n} \ind{x_n > 0} p(\vca{x} + \eb{m} - \eb{n})
\end{align}%
for all states $\vca{x} \in \statespace$, along with the normalization condition
\begin{equation}%
\sum_{\vca{x} \in \statespace} p(\vca{x}) = 1.
\end{equation}%
Through substitution it can be verified that the equilibrium distribution is
\begin{equation}%
p(\vca{x}) = \frac{1}{G} \prod_{m = 1}^M g_m(x_m),
\end{equation}%
with $\rho_m \defi \kappa \La_m/\mu_m$, $g_m(x) \defi \rho_m^x / (\prod_{n = 1}^x v_m(n))$, and
\begin{equation}%
G = \sum_{\vca{x} \in \statespace} \prod_{m = 1}^M g_m(x_m). \label{eqnRN:closed_Jackson_networks_normalization_constant}
\end{equation}%
Here $\La_m$ is the relative throughput at queue $m$, so that $\vca{\La}$ is the solution to the traffic equations satisfying $\vca{\La} \oneb = 1$. The scaling factor $\kappa$ may be chosen arbitrarily, for example so as to obtain `convenient' $\rho_m$ values. In order to see that $\kappa$ may be chosen arbitrarily, observe that the numerator and denominator of $p(\vca{x})$ both have the term $\kappa^K$ and therefore cancels.

Note that the equilibrium distribution has a product form, just like in the case of open networks. While the various terms in the product look similar, they are no longer distributions, and hence the two important properties that we observed for open Jackson networks no longer hold. Some reflection indeed shows that it is not possible for the number of jobs present at queue $m$ to have the same equilibrium distribution as that in an isolated queue with a Poisson arrival process, for the simple reason that the number of jobs at queue $m$ is at most $K$, whereas the number of jobs in the latter situation is unbounded. Likewise, it follows that it is not possible for the various numbers of jobs at the queues to be independent, for the simple reason that if there are $K$ jobs at one of the queues for example, all the other queues are known to be empty.

%%%%%%%%%%%%%%%%%%%%%%%%%%%%%%%%%%%%%%%%%%%%%%%%%%%%%%%
%%%%%%%%%%%%%%%%%%%%%%%%%%%%%%%%%%%%%%%%%%%%%%%%%%%%%%%
%%%%%%%%%%%%%%%%%%% NEW SUBSECTION %%%%%%%%%%%%%%%%%%%%
%%%%%%%%%%%%%%%%%%%%%%%%%%%%%%%%%%%%%%%%%%%%%%%%%%%%%%%
%%%%%%%%%%%%%%%%%%%%%%%%%%%%%%%%%%%%%%%%%%%%%%%%%%%%%%%

\subsection{Normalization constant in closed Jackson networks}%
\label{subsecRN:normalization_constant_closed_Jackson_networks}%

Although \eqref{eqnRN:closed_Jackson_networks_normalization_constant} provides a simple expression for the normalization constant $G$ in closed networks, brute-force numerical evaluation is prohibitively demanding for all but the smallest networks. The number of terms in the summation is
\begin{equation}%
\binom{M + K - 1}{M - 1}
\end{equation}%
which rapidly grows with the values of $M$ and $K$.

We now describe a more efficient numerical procedure for calculating the normalization constant. For convenience, we assume that the various queues either have a single server ($v_m(x) = 1$) or infinitely many servers ($v_m(x) = x$), and are labeled such that queues $1,2,\ldots,J$ are infinite-server queues while queues $J + 1,J + 2,\ldots,M$ are single-server queues. Infinite-server queues are not really `queues', in the sense that jobs never need to wait but immediately enter service upon
arrival. However, they provide a useful paradigm for modeling various kinds of delays, such as think times of users, availability periods of
machines, or transit times among queues. The normalization constant may then be expressed as
\begin{equation}%
G = \sum_{\vca{x} \in \statespace} \Bigl( \prod_{m = 1}^J \frac{\rho_m^{x_m}}{x_m!} \Bigr) \prod_{m = J + 1}^M \rho_m^{x_m} = \sum_{\vca{x} \in \statespace} \Bigl( \prod_{m = 1}^J \frac{1}{x_m!} \Bigr) \prod_{m = 1}^M \rho_m^{x_m}.
\end{equation}%
Now define $\statespace_{j,k} \defi \{ \vca{x} \in \Nat_0^j : \sum_{m = 1}^j x_m = k \}$ and let
\begin{equation}%
G(j,k) \defi \sum_{\vca{x} \in \statespace_{j,k}} \prod_{m = 1}^j \frac{\rho_m^{x_m}}{x_m!}, \quad j = 0,1,\ldots,J,
\end{equation}%
and
\begin{align}%
G(j,k) &\defi \sum_{\vca{x} \in \statespace_{j,k}} \Bigl( \prod_{m = 1}^J \frac{\rho_m^{x_m}}{x_m!} \Bigr) \prod_{m = J + 1}^j \rho_m^{x_m} \notag \\
&= \sum_{\vca{x} \in \statespace_{j,k}} \Bigl( \prod_{m = 1}^J \frac{1}{x_m!} \Bigr) \prod_{m = 1}^j \rho_m^{x_m}, \quad j = J + 1,J + 2,\ldots,M.
\end{align}%
Note that $G = G(M,K)$. Observe that for $j = 1,2,\ldots,J$,
\begin{align}%
G(j,k) &= \sum_{\substack{\vca{x} \in \Nat_0^j :\\ \vca{x} \oneb = k}} \Bigl( \prod_{m = 1}^j \frac{1}{x_m!} \Bigr) \prod_{m = 1}^j \rho_m^{x_m} \notag \\
&= \frac{1}{k!} \sum_{\substack{\vca{x} \in \Nat_0^j :\\ \vca{x} \oneb = k}} \binom{k}{x_1,\ldots,x_j} \prod_{m = 1}^j \rho_m^{x_m} = \frac{1}{k!} \Bigl( \sum_{m = 1}^j \rho_m \Bigr)^k,
\end{align}%
where we used the multinomial theorem. For $j = J + 1,J + 2,\ldots,M$ we obtain a recursion instead of an explicit expression:
\begin{align}%
G(j,k) &= \sum_{\substack{\vca{x} \in \Nat_0^j :\\ \vca{x} \oneb = k}} \Bigl( \prod_{m = 1}^J \frac{1}{x_m!} \Bigr) \prod_{m = 1}^j \rho_m^{x_m} \notag \\
&= \sum_{\substack{\vca{x} \in \Nat_0^j :\\ \vca{x} \oneb = k, \, x_j = 0}} \Bigl( \prod_{m = 1}^J \frac{1}{x_m!} \Bigr) \prod_{m = 1}^j \rho_m^{x_m} + \sum_{\substack{\vca{x} \in \Nat_0^j :\\ \vca{x} \oneb = k, \, x_j > 0}} \Bigl( \prod_{m = 1}^J \frac{1}{x_m!} \Bigr) \prod_{m = 1}^j \rho_m^{x_m} \notag \\
&= \sum_{\substack{\vca{x} \in \Nat_0^{j - 1} :\\ \vca{x} \oneb = k}} \Bigl( \prod_{m = 1}^J \frac{1}{x_m!} \Bigr) \prod_{m = 1}^{j - 1} \rho_m^{x_m} + \rho_j \sum_{\substack{\vca{x} \in \Nat_0^j :\\ \vca{x} \oneb = k - 1}} \Bigl( \prod_{m = 1}^J \frac{1}{x_m!} \Bigr) \prod_{m = 1}^j \rho_m^{x_m} \notag \\
&= G(j - 1,k) + \rho_j G(j,k - 1).
\end{align}%
Using the above recursive relationship, $G(M,K)$ can be efficiently computed starting from $G(J,k) = 1/k! ( \sum_{m = 1}^j \rho_m )^k$ and $G(j,0) = 1, ~ j = J + 1,J + 2,\ldots,M$.

%%%%%%%%%%%%%%%%%%%%%%%%%%%%%%%%%%%%%%%%%%%%%%%%%%%%%%%
%%%%%%%%%%%%%%%%%%%%%%%%%%%%%%%%%%%%%%%%%%%%%%%%%%%%%%%
%%%%%%%%%%%%%%%%%%% NEW SUBSECTION %%%%%%%%%%%%%%%%%%%%
%%%%%%%%%%%%%%%%%%%%%%%%%%%%%%%%%%%%%%%%%%%%%%%%%%%%%%%
%%%%%%%%%%%%%%%%%%%%%%%%%%%%%%%%%%%%%%%%%%%%%%%%%%%%%%%

\subsection{Mean-value analysis in closed Jackson networks}%
\label{subsecRN:MVA_closed_Jackson_networks}%

Previously, we described an efficient numerical procedure for calculating the normalization constant associated with the equilibrium distribution in closed Jackson networks with only single-server and infinite-server queues. In case we are not interested in the entire equilibrium distribution, but only in mean number of jobs at each queue or the mean \textit{sojourn} times (time spent by a job in a station), there exists an even more efficient recursive procedure, called mean-value analysis (MVA).\endnote{The mean-value analysis algorithm was developed by Reiser and Lavenberg \cite{Reiser1980_MVA} and relies heavily on the arrival theorem. The MVA algorithm as presented in \cite{Reiser1980_MVA} is more general than what is presented here. For example, different service mechanism are allowed and the authors consider networks of multiple closed systems that share stations (queues), yet still retain the product-form equilibrium distribution.} Just like in the previous section, we assume that queues $1,2,\ldots,J$ are infinite-server queues, while queues $J + 1,J + 2,\ldots,M$ are single-server queues.

Mean-value analysis is based on the following property, often referred to as `arrival theorem', which we state without proof. Consider an arbitrary arrival instant at queue $m$, that is, a time epoch where a job makes a transition to queue $m$ (possibly coming from queue $m$ itself after a service completion). Then the joint equilibrium distribution at that time instant, not counting the arriving job, is the same as the joint equilibrium distribution of the same system, but with $K - 1$ rather than $K$ jobs. In other words, when the job arrives at queue $m$, it sees the
system \textit{as if} it had never been present.\endnote{The arrival theorem originated from Lavenberg and Reiser \cite{Lavenberg1980_Arrival_Theorem}, see also Sevcik and Mitrani \cite{Sevcik1981_Arrival_Theorem}. In open Jackson networks, the distribution of the number of jobs at each queue is identical at arrival instants, departure instants and random points in time.}

Since we are interested in results in stationarity, we abuse notation and remove the time index $t$ from the state variables. In order to formally state and use the above property, it is convenient to add a superscript $a$ to indicate state variables at arrival instants (excluding the arriving job itself), and further explicitly indicate the population size in brackets. Then the above property may be written as
\begin{equation}%
\Prob{X^{\textup{a}}(K)= \vca{x}} = \Prob{X(K - 1) = \vca{x}},
\end{equation}%
for all $\vca{x} \in \statespace$, or equivalently,
\begin{equation}%
\Prob{X^{\textup{a}}(K + 1) = \vca{x}} = \Prob{X(K) = \vca{x}},
\end{equation}%
which implies for example
\begin{equation}%
\E{X_m^{\textup{a}}(K + 1)} = \E{X_m(K)}.
\end{equation}%
In case of a infinite-server queue, the mean sojourn time $S_m$ of a job at queue $m$ is simply the mean service time
\begin{equation}%
\E{S_m(K)} = 1/\mu_m, \quad m = 1,2,\ldots,J. \label{eqnRN:MVA_sojourn_times_infinite-server}
\end{equation}%
In case of a single-server queue, the mean sojourn time of a job at queue $m$ can be easily related to the number of jobs found upon arrival:
\begin{align}%
\E{S_m(K)} &= (\E{X_m^{\textup{a}}(K)} + 1)/\mu_m \notag \\
&= (\E{X_m(K - 1)} + 1)/\mu_m, \quad m = J + 1,J + 2,\ldots,M. \label{eqnRN:MVA_sojourn_times_single-server}
\end{align}%
In turn, the mean sojourn time is related to the mean queue length via Little's law:
\begin{equation}%
\La_m(K) \E{S_m(K)} = \E{X_m(K)} \label{eqnRN:MVA_Little's_law}
\end{equation}%
with $\La_m(K)$ the throughput at queue $m$ given that there are $K$ jobs in total in the system. The throughputs may be determined from the traffic equations \eqref{eqnRN:closed_Jackson_networks_traffic_equations}, up to a common scaling factor $\kappa(K)$, namely
\begin{equation}%
\vca{\La}(K) = \kappa(K) \vca{\La}, \label{eqnRN:MVA_throughputs}
\end{equation}%
where $\vca{\La}$ represents the vector of relative throughput with $\vca{\La} \oneb = 1$, which can be uniquely determined from the traffic equations, and $\kappa(K)$ is a common scaling factor depending on the total number of jobs in the system. Now, the number of jobs in the system is constant, so by summing over all $m = 1,2,\ldots,M$ on both sides of \eqref{eqnRN:MVA_Little's_law} we obtain
\begin{equation}%
\sum_{m = 1}^M \La_m(K) \E{S_m(K)} = \sum_{m = 1}^M \E{X_m(K)} = K
\end{equation}%
which gives an expression for $\kappa(K)$:
\begin{equation}%
\kappa(K) = K \Bigl( \sum_{m = 1}^M \La_m \E{S_m(K)} \Bigr)^{-1}. \label{eqnRN:MVA_scaling_factor_throughputs}
\end{equation}%
Together, the above relationships provide a recursive procedure for calculating $\E{S_m(K)}$ and $\E{X_m(K)}$ for any desired value of $K$, starting from $\E{X_m(0)} = 0, ~ m = 1,2,\ldots,M$. We summarize the mean-value analysis in \cref{algRN:MVA}.

\begin{algorithm}%
\caption{Mean-value analysis for closed Jackson networks.}%
\label{algRN:MVA}%
\begin{algorithmic}[1]%
\State Calculate $\vca{\La}$ such that $\vca{\La} \oneb = 1$ from \eqref{eqnRN:closed_Jackson_networks_traffic_equations}
\State Choose a desired population $K_{\textup{max}}$ and set $K = 1$
\While{$K < K_{\textup{max}}$}
    \State Calculate $\E{S_m(K)}, ~ m = 1,2,\ldots,M$ from \eqref{eqnRN:MVA_sojourn_times_infinite-server} and \eqref{eqnRN:MVA_sojourn_times_single-server}
    \State Calculate $\kappa(K)$ from \eqref{eqnRN:MVA_scaling_factor_throughputs}
    \State Calculate $\E{X_m(K)}, ~ m = 1,2,\ldots,M$ from \eqref{eqnRN:MVA_Little's_law}
    \State $K = K + 1$
\EndWhile
\end{algorithmic}%
\end{algorithm}%

\begin{example}[A trucking company]\label{exRN:MVA_trucking_company}%
\begin{figure}%
\centering%
\includestandalone{Chapters/RN/TikZFiles/closed_Jackson_network_MVA_example}%
\caption{The closed Jackson network of \protect\cref{exRN:MVA_trucking_company}.}
\label{figRN:MVA_trucking_company}
\end{figure}%

A large international trucking company has to move its spare parts from warehouse $A$ to warehouse $B$. Upon arriving to warehouse $A$, the trucks wait to be served by a crew that loads the spare parts into the truck. The crew takes an exponential amount of time with mean 1 to load a single truck. The loaded truck then drives to warehouse $B$ where another crew unloads the truck, taking an exponential amount of time with mean $6/5$ per truck. The time it takes to drive from warehouse $A$ to warehouse $B$ (or back) takes is exponential distributed with mean 4. The trucking company does not comply with the regulations and laws and allows the truck drivers to make as many trips in a row as they want to earn some extra money. After unloading at warehouse $B$ one third of the truck drivers decides to drive back to warehouse $A$ and make another trip. The remaining fraction of the drivers goes to a motel nearby warehouse $B$ and starts the drive to warehouse $A$ after an exponential amount of time with mean 12. The trucking company is interested in the impact of the number of trucks (and drivers) on the number of trucks per time unit that unload at warehouse $B$.

The moving operation can be modeled as a closed Jackson network. We identify two single-server queues (warehouses $A$ and $B$) and three infinite-server queues (drive from $A$ to $B$, drive from $B$ to $A$ and stay at the motel). See \cref{figRN:MVA_trucking_company} for the numbering of the stations. The trucking company is interested in $\La_5(K)$ for various values of $K$.

The routing matrix is
\begin{equation}%
R = \begin{bmatrix}%
0 & 0   & 0   & 0 & 1 \\
0 & 0   & 0   & 1 & 0 \\
0 & 1   & 0   & 0 & 0 \\
1 & 0   & 0   & 0 & 0 \\
0 & 1/3 & 2/3 & 0 & 0
\end{bmatrix}.%
\end{equation}%
The relative throughputs $\vca{\La}$ are determined from the traffic equations \eqref{eqnRN:closed_Jackson_networks_traffic_equations} with $\vca{\La} \oneb = 1$ and we find $\vca{\La} = 1/14 \begin{bmatrix} 3 & 3 & 2 & 3 & 3 \end{bmatrix}$. So, $\La_1(K) = \La_2(K) = \La_4(K) = \La_5(K)$ and $\E{S_1(K)} = \E{S_2(K)} = 4$ and $\E{S_3(K)} = 12$ for all $K$. So, we only report $\E{S_4(K)}$, $\E{S_5(K)}$, $\La_3(K)$, $\La_5(K)$ and $\E{X_m(K)}, ~ m = 1,2,\ldots,5$. Applying \cref{algRN:MVA} produces the results in \cref{tblRN:MVA_trucking_company}.

\begin{table}%
\centering%
\begin{tabular}{c|*{5}{c}}%
             & \multicolumn{5}{c}{Number of trucks $K$} \\
             & 1 & 2 & 5 & 10 & 20 \\
\hline
$\E{S_4(K)}$ & 1.0000 & 1.0549 & 1.2533 & 1.7367 & 3.4488 \\
$\E{S_5(K)}$ & 1.2000 & 1.2791 & 1.5775 & 2.3917 & 6.5797 \\
$\La_3(K)$   & 0.0366 & 0.0727 & 0.1770 & 0.3312 & 0.5123 \\
$\La_5(K)$   & 0.0549 & 0.1091 & 0.2655 & 0.4968 & 0.7684 \\
$\E{X_1(K)}$ & 0.2198 & 0.4363 & 1.0621 & 1.9872 & 3.0735 \\
$\E{X_2(K)}$ & 0.2198 & 0.4363 & 1.0621 & 1.9872 & 3.0735 \\
$\E{X_3(K)}$ & 0.4396 & 0.8727 & 2.1242 & 3.9745 & 6.1471 \\
$\E{X_4(K)}$ & 0.0549 & 0.1151 & 0.3328 & 0.8628 & 2.6500 \\
$\E{X_5(K)}$ & 0.0659 & 0.1395 & 0.4188 & 1.1882 & 5.0558 \\
\end{tabular}%
\caption{Results from \protect\cref{algRN:MVA} for \protect\cref{exRN:MVA_trucking_company}.}%
\label{tblRN:MVA_trucking_company}%
\end{table}%

There are many scenarios to consider that could improve on this situation. The trucking company can train the unloading crew to become faster, if the mean unloading time reduces to $5/6$, the throughput for $K = 20$ trucks increases by 13.7\% to $\La_5(20) = 0.8738$. If the trucking company were to increase the money earned per trip, a fraction 1/2 does another trip. In that case, the throughput for $K = 20$ trucks only slightly increases by 2.5\% to $\La_5(20) = 0.7876$. The first option seems better, but it does increase the mean sojourn time for the loading station by approximately 1 time unit.
\end{example}%

%%%%%%%%%%%%%%%%%%%%%%%%%%%%%%%%%%%%%%%%%%%%%%%%%%%%%%%
%%%%%%%%%%%%%%%%%%%%%%%%%%%%%%%%%%%%%%%%%%%%%%%%%%%%%%%
%%%%%%%%%%%%%%%%%%% NEW SUBSECTION %%%%%%%%%%%%%%%%%%%%
%%%%%%%%%%%%%%%%%%%%%%%%%%%%%%%%%%%%%%%%%%%%%%%%%%%%%%%
%%%%%%%%%%%%%%%%%%%%%%%%%%%%%%%%%%%%%%%%%%%%%%%%%%%%%%%

\section{Takeaways}%
\label{secRN:takeaways}%

Starting from a basic birth--and--death process, and the concept of reversibility, we were able to find in an elegant manner the equilibrium distribution for the rich class of loss networks. While loss networks give rise to multi-dimensional Markov processes, their state space allows for local balance arguments with balance equations that are readily solved, leading to the beautiful product-form solution in \cref{thmRN:loss_networks_equilibrium_distribution}. As pointed out, the catch here is the normalization constant, whose computation requires the enumeration of all states in the state space and needs tailor-made algorithms.

Markov processes intimately related to loss networks are also studied in statistical mechanics, in the form of interacting particle systems. While the terminology is different (Markovian assumptions become Glauber dynamics, product-form solution is called Gibbs measure and the normalization constant is known as the partition function), the Markov process description and analytic methods are largely the same. For thorough treatments of such interacting particle systems we refer to Liggett \cite{Liggett1985_Interacting_particle_systems}.

We then proceeded to queueing networks, again giving rise to multi-dimensional Markov processes. But for these processes, local balance fails, and the global balance equations are then the unavoidable point of departure. Nevertheless, structure was there to be discovered, the first glimpse captured by Burke’s Theorem, telling us that the output process of one queue with Poisson arrivals is again Poisson. This property then naturally leads to the guess that networks of queues with external Poisson arrivals can be decomposed into isolated queues with arrival rates that match in expectation the arrival rates in the networks. Mathematically, such an educated guess translates into substituting a product of product forms into the global balance equations, and showing that indeed this is the unique solution and hence the unique equilibrium distribution. Although elegant and sound, this educated guess approach is somewhat unsatisfying, because it is non-constructive. In analysis, however, solving a difference or differential equations by clever substitutions is one of the key techniques. We shall continue to work with educated guesses for finding product-form solution whenever this is necessary, e.g., for more advanced Markov processes in \cref{ch:gated_single-server,ch:join_the_shortest_queue,ch:cyclic_production_systems}.

The network models in this chapter make it possible to consider real-world networked systems with a host of applications. Loss networks were used for instance to describe the topology and performance of the internet \cite{Kelly1979_Reversibility_and_stochastic_networks} and queueing models can describe complex manufacturing processes \cite{Buzacott1993_Stochastic_models_of_manufacturing_systems}.

Loss networks and queueing networks are examples of stochastic networks, one of the richest topics in the field of applied probability. Text books with prominent roles for such network are Buzacott and Shantikumar \cite{Buzacott1993_Stochastic_models_of_manufacturing_systems}, Chen and Yao \cite{Chen2001_Queueing_Networks}, Kelly and Yudovina \cite{Kelly2014_Stochastic_Networks} and Whittle \cite{Whittle2007_Networks}.

%%%%%%%%%%%%%%%%%%%%%%%%%%%%%%%%%%%%%%%%%%%%%%%%%%%%%%%
%%%%%%%%%%%%%%%%%%%%%%%%%%%%%%%%%%%%%%%%%%%%%%%%%%%%%%%
%%%%%%%%%%%%%%%%%%%%%%%% NOTES %%%%%%%%%%%%%%%%%%%%%%%%
%%%%%%%%%%%%%%%%%%%%%%%%%%%%%%%%%%%%%%%%%%%%%%%%%%%%%%%
%%%%%%%%%%%%%%%%%%%%%%%%%%%%%%%%%%%%%%%%%%%%%%%%%%%%%%%

%\theendnotes%
%\setcounter{endnote}{0}
\printendnotes%%

% Checked points 1-7 and a-i
\chapter{Quasi-birth--and--death processes}%
\label{ch:quasi-birth--and--death_processes}%

Quasi-birth--and--death (QBD) processes are the natural two-dimensional generalization of the birth--and--death process. QBDs live on a countable state space that consists of one infinite dimension and one finite dimension. The finite dimension is added on top of what would otherwise be a BD process. Before we develop the general theory of a QBD process, let us treat some examples that show the extension of a BD process to a QBD process.

%%%%%%%%%%%%%%%%%%%%%%%%%%%%%%%%%%%%%%%%%%%%%%%%%%%%%%%
%%%%%%%%%%%%%%%%%%%%%%%%%%%%%%%%%%%%%%%%%%%%%%%%%%%%%%%
%%%%%%%%%%%%%%%%%%%%% NEW SECTION %%%%%%%%%%%%%%%%%%%%%
%%%%%%%%%%%%%%%%%%%%%%%%%%%%%%%%%%%%%%%%%%%%%%%%%%%%%%%
%%%%%%%%%%%%%%%%%%%%%%%%%%%%%%%%%%%%%%%%%%%%%%%%%%%%%%%

\section{Variations of birth--and--death processes}%
\label{secQBD:variations}%
%

%%%%%%%%%%%%%%%%%%%%%%%%%%%%%%%%%%%%%%%%%%%%%%%%%%%%%%%
%%%%%%%%%%%%%%%%%%%%%%%%%%%%%%%%%%%%%%%%%%%%%%%%%%%%%%%
%%%%%%%%%%%%%%%%%%% NEW SUBSECTION %%%%%%%%%%%%%%%%%%%%
%%%%%%%%%%%%%%%%%%%%%%%%%%%%%%%%%%%%%%%%%%%%%%%%%%%%%%%
%%%%%%%%%%%%%%%%%%%%%%%%%%%%%%%%%%%%%%%%%%%%%%%%%%%%%%%

\subsection{Machine with setup times}%
\label{subsecQBD:setup_times}%

Let us consider a machine processing jobs in order of arrival. Jobs arrive according to a Poisson process with rate $\la$ and the processing times are exponential with mean $1/\mu$. For stability we assume that $\rho \defi \la/\mu < 1$. The machine is turned off when the system is empty and it is turned on again when a new job arrives. The setup time is exponentially distributed with mean $1/\theta$. Turning off the machine takes no time. We are interested in the effect of the setup time on the sojourn time of a job.

The state of the system may be described by $X(t) \defi (X_1(t),X_2(t))$ with $X_1(t)$ representing the number of jobs in the system at time $t$ and $X_2(t)$ describes if the machine is turned off (0) or on (1) at time $t$. The process $\{ X(t) \}_{t \ge 0}$ is a Markov process with state space $\statespace \defi \{ (i,j) \in \Nat_0 \times \{ 0,1 \} \}$. The transition rate diagram is displayed in \cref{figQBD:transition_rate_diagram_setup_times}. It looks similar to the one for the BD process, in for example \cref{figBD:BD_infinite_state_space}, except that each state $i$ has been replaced by a set of states $\{ (i,0),(i,1) \}$. This set of states is called \textit{level} $i$. In BD processes transitions are restricted to neighbouring states, while in QBD processes the transitions are restricted to neighbouring levels. The horizontally aligned set of states $\{ (0,j),(1,j),\ldots \}$ is often referred to as \textit{phase} $j$.

\begin{figure}
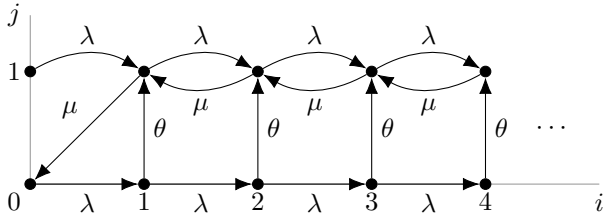
%
\centering%
\includestandalone{Chapters/QBD/TikZFiles/transition_rate_diagram_setup_times}%
\caption{Transition rate diagram of the machine with setup times.}
\label{figQBD:transition_rate_diagram_setup_times}%
\end{figure}%

For the current model, define $\lvl{i} = \{ (i,0),(i,1) \}, ~ i \ge 0$ as the set of states with $i$ jobs in the system, that is $\lvl{i}$ is level $i$. We can then write the state space as
\begin{equation}%
\statespace \defi \lvl{0} \cup \lvl{1} \cup \lvl{2} \cup \cdots
\end{equation}%

Let $p(i,j)$ denote the equilibrium probability of state $(i,j) \in \statespace$. Clearly, $p(0,1) = 0$ since state $(0,1)$ is transient. State $(0,1)$ is included in the state space for notational convenience: all levels consists of two states. From the transition rate diagram we obtain by equating the flow out of a state and the flow into that state the following set of global balance equations,
\begin{align}%
\la p(0,0)  &= \mu p(1,1), \label{eqnQBD:setup_times_global_balance_(0,0)} \\
(\la + \theta) p(i,0) &= \la p(i - 1,0), \quad i \ge 1, \label{eqnQBD:setup_times_global_balance_phase_0} \\
(\la + \mu) p(i,1) &= \la p(i - 1,0) + \theta p(i,0) + \mu p(i + 1,1), \quad i \ge 1. \label{eqnQBD:setup_times_global_balance_phase_1}
\end{align}%
The structure of the equations \eqref{eqnQBD:setup_times_global_balance_(0,0)}--\eqref{eqnQBD:setup_times_global_balance_phase_1} is closely related to balance equations of the $M/M/1$ queueing model, see for example \cref{exBD:MM1_equilibrium_dist}. This becomes more striking by introducing vectors of equilibrium probabilities $\pb_i = \begin{bmatrix} p(i,0) & p(i,1) \end{bmatrix}$ and writing \eqref{eqnQBD:setup_times_global_balance_(0,0)}--\eqref{eqnQBD:setup_times_global_balance_phase_1} in vector-matrix notation:
\begin{align}%
\pb_0 \La_{0}^{(0)} + \pb_1 \La_{-1}^{(1)} &= \zerob, \label{eqnQBD:setup_times_global_balance_vector-matrix_level_0} \\
\pb_{i - 1} \La_1 + \pb_i \La_0 + \pb_{i + 1} \La_{-1} &= \zerob, \quad i \ge 1, \label{eqnQBD:setup_times_global_balance_vector-matrix}
\end{align}%
where
\begin{gather}%
\La_{-1} = \begin{bmatrix}%
0 & 0 \\ 0 & \mu
\end{bmatrix},%
\quad
\La_0 = \begin{bmatrix}%
-(\la + \theta) & \theta \\ 0 & -(\la + \mu)
\end{bmatrix},%
\quad
\La_1 = \begin{bmatrix}%
\la & 0 \\ 0 & \la
\end{bmatrix},\\%
\La_{0}^{(0)} = -\La_1,
\quad
\La_{-1}^{(1)} = \begin{bmatrix}%
0 & 0 \\ \mu & 0
\end{bmatrix}.%
\end{gather}%
Obviously, if we can determine the equilibrium probabilities $p(i,j)$, then we also compute the mean number of jobs in the system, and by Little's law, the mean sojourn time. We now present three methods to determine the equilibrium probabilities. The first one is known as the \textit{matrix-geometric method}, the second is referred to as the \textit{spectral expansion method}, and the third one employs partial generating functions. Let us start with the matrix-geometric approach. We will introduce the first two methods in greater detail in the later sections of this chapter. The last method will appear in various places of this book, but is more well-known overall.

We first simplify the balance equations \eqref{eqnQBD:setup_times_global_balance_vector-matrix} by eliminating the vector $\pb_{i + 1}$. By equating the flow from level $i$ to level $i + 1$ to the flow from level $i + 1$ to level $i$ we obtain
\begin{equation}%
(p(i,0) + p(i,1))\la = p(i + 1,1)\mu,
\end{equation}%
or, in vector-matrix notation,
\begin{equation}%
\pb_i \La^* = \pb_{i + 1} \La_{-1}
\end{equation}%
where
\begin{equation}%
\La^* = \begin{bmatrix}%
0 & \la \\ 0 & \la
\end{bmatrix}.%
\end{equation}%
Substituting this relation into \eqref{eqnQBD:setup_times_global_balance_vector-matrix} produces
\begin{equation}%
\pb_{i - 1} \La_1 + \pb_i \bigl( \La_0 + \La^* \bigr) = \zerob, \quad i \ge 1,
\end{equation}%
which allows us to express $\pb_i$ in terms of $\pb_{i - 1}$:
\begin{equation}%
\pb_i = - \pb_{i - 1} \La_1 \bigl( \La_0 + \La^* \bigr)^{-1} = \pb_{i - 1} R, \label{eqnQBD:setup_times_relation_level_i_level_i-1}
\end{equation}%
where
\begin{equation}%
R \defi - \La_1 \bigl( \La_0 + \La^* \bigr)^{-1} = \begin{bmatrix}%
\frac{\la}{\la + \theta} & \frac{\la}{\mu} \\ 0 & \frac{\la}{\mu}
\end{bmatrix}.%
\end{equation}%
Iterating \eqref{eqnQBD:setup_times_relation_level_i_level_i-1} leads to the matrix-geometric solution
\begin{equation}%
\pb_i = \pb_0 R^i, \quad i \ge 0. \label{eqnQBD:setup_times_equilibrium_distribution_vector-matrix}
\end{equation}%
Notice that this is very similar to the solution for the $M/M/1$ model, which is $p(i) = p(0) \rho^i, ~ i \ge 0$. Finally, $\pb_0$ follows from the equations \eqref{eqnQBD:setup_times_global_balance_vector-matrix_level_0} and the normalization condition
\begin{equation}%
\sum_{(i,j) \in \statespace} p(i,j) = \pb_0( \I - R )^{-1} \oneb = 1.
\end{equation}%
From \eqref{eqnQBD:setup_times_equilibrium_distribution_vector-matrix} we obtain the mean number of jobs in the system as
\begin{equation}%
\E{X_1} = \sum_{i \ge 1} i \pb_i \oneb = \sum_{i \ge 1} i \pb_0 R^i \oneb = \pb_0 R (I - R)^{-2} \oneb,
\end{equation}%
and the mean sojourn time is $\E{S} = \E{X_1}/\la$.

The matrix $R$ is critical in the matrix-geometric approach. It is called the \textit{rate matrix} and has an interesting and useful probabilistic interpretation. Element $(j,k)$ of $R$ is the expected time spent in state $(i + 1,k)$ multiplied by element $(j,j)$ of $-\La_0$ before the first transition to a state in level $i$, given the initial state $(i,j)$. This immediately means that zero rows in $\La_1$ lead to zero rows in $R$. Recall the hitting-time random variables of \cref{ch:Markov_processes}, which we now use with a slight modification. For any set $\set{A} \subset \statespace$,
\begin{equation}%
\htt{\set{A}} \defi \inf \{ t > 0 : \lim_{s \uparrow t} X(s) \neq X(t) \in \set{A} \}.
\end{equation}%
Note that we suppress the dependence on the initial state, since that will be clear from the expectation that we are determining. Using the hitting-time random variable, we can write $(R)_{j,k}$ as
\begin{equation}%
(R)_{j,k} = (-\La_0)_{j,j} \Efxd{ (i,j) }{ \int_0^{\htt{\lvl{i}}} \ind{X(t) = (i + 1,k)} \, \dinf t }.
\end{equation}%
Let us derive element $(0,0)$ of $R$. Using a one-step analysis and the strong Markov property at the sojourn time in state $(i,0)$ we have
\begin{align}%
(R)_{0,0} &= (\la + \theta) \Efxd{ (i,0) }{ \int_0^{\htt{\lvl{i}}} \ind{X(t) = (i + 1,0)} \, \dinf t } \notag \\
&= (\la + \theta) \frac{\la}{\la + \theta} \Efxd{ (i + 1,0) }{ \int_0^{\htt{\lvl{i}}} \ind{X(t) = (i + 1,0)} \, \dinf t } \notag \\
&= \la \Efxd{ (i + 1,0) }{\int_0^{H_{(i + 1,0)}} 1 \, \dinf t} = \frac{\la}{\la + \theta},
\end{align}%
where $H_x$ was defined as the sojourn time in state $x$. Possibly more interesting is the derivation of element $(1,1)$ of $R$. Using a similar analysis as above we get
\begin{align}%
(R)_{1,1} &= (\la + \mu) \Efxd{ (i,1) }{ \int_0^{\htt{\lvl{i}}} \ind{X(t) = (i + 1,1)} \, \dinf t } \notag \\
&= \la \Efxd{ (i + 1,1) }{ \int_0^{\htt{\lvl{i}}} \ind{X(t) = (i + 1,1)} \, \dinf t }. \label{eqnQBD:setup_times_R_11}
\end{align}%
We continue by conditioning on the number of times the process visits state $(i + 1,1)$ before reaching level $i$. The probability $q(n)$ that state $(i + 1,1)$ is visited $n$ times (where the initial visit is counted) before reaching level $i$ is
\begin{equation}%
q(n) = \frac{\mu}{\la + \mu} \Bigl( \frac{\la}{\la + \mu} \Bigr)^{n - 1}, \quad n \ge 1,
\end{equation}%
since if the process transitions to state $(i + 2,1)$, it returns to state $(i + 1,1)$ with probability 1 by positive recurrence due to $\rho < 1$. If the process visits state $(i + 1,1)$ a total of $n$ times before reaching level $i$, then it spends in expectation $n/(\la + \mu)$ time in state $(i + 1,1)$. Combining these observations, we find
\begin{align}%
(R)_{1,1} &= \la \sum_{n \ge 1} \frac{n}{\la + \mu} \cdot \frac{\mu}{\la + \mu} \Bigl( \frac{\la}{\la + \mu} \Bigr)^{n - 1} \notag \\
&= \frac{\mu}{\la + \mu} \sum_{n \ge 1} n \Bigl( \frac{\la}{\la + \mu} \Bigr)^n = \frac{\la}{\mu}.
\end{align}%
Element $(R)_{0,1}$ is equal to $(R)_{1,1}$, because if the process transitions from state $(i,0)$ to $(i + 1,0)$, it reaches state $(i + 1,1)$ before level $i$ with probability 1, allowing for the exact same analysis and result.

We now demonstrate the spectral expansion method. This method first seeks solutions of the equations \eqref{eqnQBD:setup_times_global_balance_vector-matrix} of the simple form
\begin{equation}%
\pb_i = \vca{y} x^i, \quad i \ge 0,
\end{equation}%
where $\vca{y} = \begin{bmatrix} y(0) & y(1) \end{bmatrix}$ is a non-zero vector and $|x| < 1$. The latter is required, since we want to be able to normalize the solution afterwards. Substitution of this form into \eqref{eqnQBD:setup_times_global_balance_vector-matrix} and dividing by common powers of $x$ gives
\begin{equation}%
\vca{y} \bigl( \La_1 + x \La_0 + x^2 \La_{-1} \bigr) = \zerob.
\end{equation}%
So, the desires values of $x$ are the roots inside the unit circle of the determinant equation
\begin{equation}%
\det( \La_1 + x \La_0 + x^2 \La_{-1} ) = 0.
\end{equation}%
In this case we have
\begin{equation}%
\det( \La_1 + x \La_0 + x^2 \La_{-1} ) = (\la - (\la + \theta)x)(\mu x - \la)(x - 1) = 0.
\end{equation}%
We can read of the roots with $|x| < 1$, which are
\begin{equation}%
x_1 = \frac{\la}{\la + \theta}, \quad x_2 = \frac{\la}{\mu}.
\end{equation}%
For $i = 1,2$, let $\vca{y}_i$ be the non-zero solution of
\begin{equation}%
\vca{y}_i \bigl( \La_1 + x_i \La_0 + x_i^2 \La_{-1} \bigr) = \zerob.
\end{equation}%
Solving this linear system of equations gives the solutions
\begin{equation}%
\vca{y}_1 = \begin{bmatrix}%
1 & \displaystyle\frac{-\theta x_1}{\la - (\la + \mu)x_1 + \mu x_1^2}
\end{bmatrix},%
\quad
\vca{y}_2 = \begin{bmatrix}%
0 & 1
\end{bmatrix}.%
\end{equation}%
Note that, since the balance equations are linear, any linear combination of the two simple solutions satisfies \eqref{eqnQBD:setup_times_global_balance_vector-matrix}. Now the final step of the spectral expansion method is to determine a linear combination that also satisfies the boundary equations \eqref{eqnQBD:setup_times_global_balance_vector-matrix_level_0}. So we set
\begin{equation}%
\pb_i = \xi_1 \vca{y}_1 x_1^i + \xi_2 \vca{y}_2 x_2^i, \quad i \ge 0, \label{eqnQBD:setup_times_equilibrium_distribution_spectral_expansion}
\end{equation}%
where the coefficients $\xi_1$ and $\xi_2$ follow from the boundary equations \eqref{eqnQBD:setup_times_global_balance_vector-matrix_level_0} and the normalization condition
\begin{equation}%
\sum_{(i,j) \in \statespace} p(i,j) = \sum_{i \ge 0} \Bigl( \xi_1 \vca{y}_1 x_1^i + \xi_2 \vca{y}_2 x_2^i \Bigr) \oneb = \frac{\xi_1 \vca{y}_1 \oneb}{1 - x_1} + \frac{\xi_2 \vca{y}_2 \oneb}{1 - x_2} = 1.
\end{equation}%
Since the balance equations are dependent, we may omit one of the equations of \eqref{eqnQBD:setup_times_global_balance_vector-matrix_level_0}, and, for example, only use
\begin{equation}%
0 = p(0,1) = \xi_1 y_1(1) + \xi_2 y_2(1),
\end{equation}%
together with the normalization condition to determine the (unique) coefficients $\xi_1$ and $\xi_2$.

Using representation \eqref{eqnQBD:setup_times_equilibrium_distribution_spectral_expansion} we obtain
\begin{equation}%
\E{X_1} = \sum_{i \ge 1} i \pb_i \oneb = \frac{\xi_1 \vca{y}_1 x_1 \oneb}{(1 - x_1)^2} + \frac{\xi_2 \vca{y}_2 x_2 \oneb}{(1 - x_2)^2}.
\end{equation}%

The two methods presented above are closely related: $x_1$ and $x_2$ are the eigenvalues of the rate matrix $R$ and $\vca{y}_1$ and $\vca{y}_2$ are the corresponding left eigenvectors.

The third and final method uses generating functions. Introduce the partial generating functions
\begin{equation}%
\PGF{j}{\PGFarg} \defi \sum_{i \ge 0} p(i,j) \PGFarg^i, \quad j = 0,1,
\end{equation}%
defined for all $|\PGFarg| \le 1$. Multiplying \eqref{eqnQBD:setup_times_global_balance_phase_0} and \eqref{eqnQBD:setup_times_global_balance_phase_1} by $\PGFarg^i$ and summing over all $i \ge 1$ yields
\begin{align}%
(\la + \mu)(\PGF{0}{\PGFarg} - p(0,0)) &= \la \PGFarg \PGF{0}{\PGFarg} , \\
(\la + \mu)(\PGF{1}{\PGFarg} - p(0,1)) &= \la \PGFarg \PGF{1}{\PGFarg} + \theta (\PGF{0}{\PGFarg} - p(0,0)) \notag \\
&\quad + \frac{\mu}{\PGFarg} (\PGF{1}{\PGFarg} - p(0,1) - p(1,1) \PGFarg).
\end{align}%
Using $p(0,1) = 0$ and \eqref{eqnQBD:setup_times_global_balance_(0,0)}, we get
\begin{equation}%
\PGF{0}{\PGFarg} = \frac{p(0,0)}{1 - \frac{\la}{\la + \theta} \PGFarg}, \label{eqnQBD:setup_times_partial_GF_phase_0}
\end{equation}%
and
\begin{align}%
\PGF{1}{\PGFarg} &= \frac{\theta \PGF{0}{\PGFarg} - (\la + \theta) p(0,0)}{(\PGFarg - 1)(\frac{\mu}{\PGFarg} - \la)} = \frac{\rho \PGFarg p(0,0)}{(1 - \frac{\la}{\la + \theta} \PGFarg)(1 - \rho \PGFarg)} \notag \\
&= \frac{p(0,0)\rho}{\rho - \frac{\la}{\la + \theta}} \Bigl( \frac{1}{1 - \rho \PGFarg} - \frac{1}{1 - \frac{\la}{\la + \theta} \PGFarg} \Bigr). \label{eqnQBD:setup_times_partial_GF_phase_1}
\end{align}%
The probability $p(0,0)$ follows from the normalization condition
\begin{equation}%
\PGF{0}{1} + \PGF{1}{1} = 1,
\end{equation}%
which results in
\begin{equation}%
p(0,0) = (1 - \rho) \frac{\theta}{\la + \theta}.
\end{equation}%
From \eqref{eqnQBD:setup_times_partial_GF_phase_0} and \eqref{eqnQBD:setup_times_partial_GF_phase_1} and $1/(1 - x) = \sum_{i \ge 0} x^i, ~ |x| < 1$ we conclude that for $i \ge 0$,
\begin{align}%
p(i,0) &= p(0,0) \Bigl( \frac{\la}{\la + \theta} \Bigr)^i, \\
p(i,1) &= \frac{p(0,0)\rho}{\rho - \frac{\la}{\la + \theta}} \Bigl( \rho^i - \Bigl( \frac{\la}{\la + \theta} \Bigr)^i \Bigr),
\end{align}%
which agrees with the form \eqref{eqnQBD:setup_times_equilibrium_distribution_spectral_expansion}.

\begin{remark}[Mean value analysis]%
The mean number of jobs in the system $\E{X_1}$ and the mean sojourn time $\E{S}$ can also be determined by combining the PASTA property and Little's law. Based on PASTA we know that the average number of jobs in the system seen by an arriving job equals $\E{X_1}$, and each of them (also the one being processed) has a (residual) processing time with mean $1/\mu$. With probability $1 - \rho$ the machine is not in operation on arrival, so that the job also has to wait for the setup phase with mean $1/\theta$. Further, the job has to wait for its own processing time. Combining these observations, we have
\begin{equation}%
\E{S} = (1 - \rho) \frac{1}{\theta} + \E{X_1} \frac{1}{\mu} + \frac{1}{\mu},
\end{equation}%
and together with Little's law
\begin{equation}%
\E{X_1} = \la \E{S},
\end{equation}%
we find
\begin{equation}%
\E{S} = \frac{\frac{1}{\mu}}{1 - \rho} + \frac{1}{\theta}.
\end{equation}%
The first term at the right-hand side is the mean sojourn time in the system without setup times (the machine is always on). The second term is the mean setup time. Clearly, the mean setup time is exactly the extra mean delay caused by turning off the machine when there is no work. In fact, it can be shown (by using, for example, a sample path argument) that the extra delay is an exponential time with mean $1/\theta$.
\end{remark}%

%%%%%%%%%%%%%%%%%%%%%%%%%%%%%%%%%%%%%%%%%%%%%%%%%%%%%%%
%%%%%%%%%%%%%%%%%%%%%%%%%%%%%%%%%%%%%%%%%%%%%%%%%%%%%%%
%%%%%%%%%%%%%%%%%%% NEW SUBSECTION %%%%%%%%%%%%%%%%%%%%
%%%%%%%%%%%%%%%%%%%%%%%%%%%%%%%%%%%%%%%%%%%%%%%%%%%%%%%
%%%%%%%%%%%%%%%%%%%%%%%%%%%%%%%%%%%%%%%%%%%%%%%%%%%%%%%

\subsection{Erlang services}%
\label{subsecQBD:Erlang_services}%

We consider a single-server queue. Jobs arrive according to a Poisson process with rate $\la$ and they are served in order of arrival. The service times are Erlang-$r$ distributed with mean $r/\mu$. For stability we require that the occupation rate
\begin{equation}%
\rho \defi \la \frac{r}{\mu}
\end{equation}%
is less than one. This system can be described by a QBD process $\{ X(t) \}_{t \ge 0}$ with levels $\lvl{0} = \{(0,0),(0,1),\ldots,(0,r) \}$ and $\lvl{i} = \{ (i,1),(i,2),\ldots,(i,r) \}$, with $i \ge 1$, where level $i$ indicates the number of jobs waiting in the queue and phase $j$ is the remaining number of service phases of the job in service. The state space is denoted by $\statespace \defi \lvl{0} \cup \lvl{1} \cup \cdots$. The transition rate diagram is shown in \cref{figQBD:transition_rate_diagram_Erlang_services}. Note that by setting $r = 1$ we get a homogeneous BD process on the states $\Nat_0$, modeling the $M/M/1$ queue.

\begin{figure}
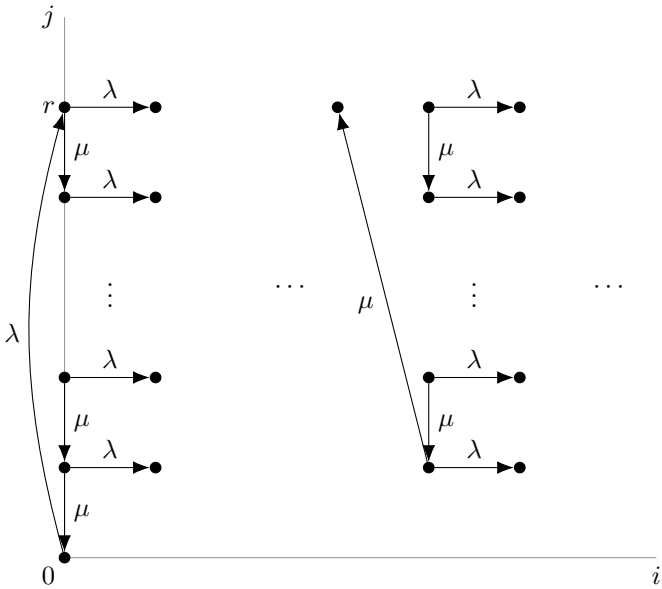
%
\centering%
\includestandalone{Chapters/QBD/TikZFiles/transition_rate_diagram_Erlang_services}%
\caption{Transition rate diagram of the single-server system with Erlang services.}
\label{figQBD:transition_rate_diagram_Erlang_services}%
\end{figure}%

Let $p(i,j)$ denote the equilibrium probability of state $(i,j) \in \statespace$. From the transition rate diagram we get the following balance equations for the states $(i,j)$ with $i \ge 1$,
\begin{align}%
(\la + \mu) p(i,j) &= \la p(i - 1,j) + \mu p(i,j + 1), \quad j = 1,2,\ldots,r - 1, \label{eqnQBD:Erlang_services_global_balance_middle_phases} \\
(\la + \mu) p(i,r) &= \la p(i - 1,r) + \mu p(i + 1,1), \label{eqnQBD:Erlang_services_global_balance_top_phase}
\end{align}%
or, in vector-matrix notation,
\begin{equation}%
\pb_{i - 1} \La_1 + \pb_i \La_0 + \pb_{i + 1}\La_{-1} = \zerob, \quad i \ge 1, \label{eqnQBD:Erlang_services_global_balance}
\end{equation}%
where $\pb_i = \begin{bmatrix} p(i,1) & p(i,2) & \cdots & p(i,r) \end{bmatrix}$,
\begin{equation}%
\La_{-1} = \begin{bmatrix}%
 & \cdots & 0 & \mu \\
 &        &   &  0  \\
 &        &   & \vdots \\
 &        &   &
\end{bmatrix},%
\quad
\La_0 = \begin{bmatrix}%
0   & \\
\mu & 0 \\
    & \ddots & \ddots \\
    &        & \mu    & 0
\end{bmatrix} - (\la + \mu)\I,
\end{equation}%
and $\La_1 = \la \I$, where all unlabeled entries are 0. We first determine the probabilities $p(i,j)$ using the \textit{matrix-analytic method}. Define an \textit{excursion} as a sample path of the process starting in level $i$, moving to levels higher than $i$ and ending at first return to level $i$. From the transition rate diagram we see that the number of excursions per time unit that end in state $(i,r)$ is $p(i + 1,1) \mu$. Note that this is the only state in which an excursion can end. On the other hand, the number of excursions per time unit that start in state $(i,k)$, immediately go to state $(i + 1,k)$ and ultimately end the excursion in state $(i,r)$ is $p(i,k) \la$. The number of excursions per time unit that end in state $(i,r)$ is found by summing over all possible starting states, so we get $\sum_{k = 1}^r p(i,k) \la$. In vector-matrix form this leads to
\begin{equation}%
\pb_{i + 1} \La_{-1} = \pb_i \La_1 G, \label{eqnQBD:Erlang_services_excursions}
\end{equation}%
where the matrix $G$ is called the \textit{auxiliary matrix} of the matrix-analytic method and element $(j,k)$ of $G$ is interpreted as the probability that, starting in state $(i,j), ~ i \ge 1$, the first passage to level $i - 1$ happens in state $(i - 1,k)$. This immediately means that zero columns in $\La_{-1}$ lead to zero columns in $G$. For the model at hand
\begin{equation}%
G = \begin{bmatrix}%
0       & \cdots & 0       & 1 \\
\vdots  &         & \vdots & \vdots \\
0       & \cdots & 0       & 1
\end{bmatrix}.%
\end{equation}%
We can substitute the relation \eqref{eqnQBD:Erlang_services_excursions} into \eqref{eqnQBD:Erlang_services_global_balance} to obtain
\begin{equation}%
\pb_i = - \pb_{i - 1} \La_1 ( \La_0 + \La_1 G)^{-1},
\end{equation}%
where we note that the inverse exists. By iterating this equation we get
\begin{equation}%
\pb_i = \pb_0 \bigl( - \La_1 ( \La_0 + \La_1 G)^{-1} \bigr)^i, \quad i \ge 0.
\end{equation}%
Finally the probabilities $p(0,0)$ and $\pb_0$ follow from the balance equations for the states in $\lvl{0}$ and the normalization condition. The above relation also shows that the matrix-geometric and matrix-analytic methods for QBD processes are closely related: the rate matrix $R = - \La_1 ( \La_0 + \La_1 G)^{-1}$, since a relation like \eqref{eqnQBD:setup_times_relation_level_i_level_i-1} of the previous example also holds for the current model.

We again apply the spectral expansion method to find the equilibrium distribution. We substitute the simple form
\begin{equation}%
p(i,j) = y(j) x^i, \quad i \ge 0, ~ j = 1,2,\ldots,r,
\end{equation}%
into the balance equations \eqref{eqnQBD:Erlang_services_global_balance_middle_phases}--\eqref{eqnQBD:Erlang_services_global_balance_top_phase} and divide by common powers of $x$ to find
\begin{align}%
(\la + \mu) y(j) x &= \la y(j) + \mu y(j + 1) x, \quad j = 1,2,\ldots,r - 1, \label{eqnQBD:Erlang_services_spectral_expansion_middle_phases} \\
(\la + \mu) y(r) x &= \la y(r) + \mu y(1) x^2, \label{eqnQBD:Erlang_services_spectral_expansion_top_phase}
\end{align}%
From \eqref{eqnQBD:Erlang_services_spectral_expansion_middle_phases} we deduce
\begin{equation}%
\frac{y(j + 1)}{y(j)} = \frac{(\la + \mu) x - \la}{\mu x} = \textup{constant} \ifed \be,
\end{equation}%
so we can set $y(j) = \be^j, ~ j = 1,2,\ldots,r$. Substituting this back into \eqref{eqnQBD:Erlang_services_spectral_expansion_middle_phases}--\eqref{eqnQBD:Erlang_services_spectral_expansion_top_phase} gives
\begin{align}%
(\la + \mu) x &= \la + \mu \be x, \\
(\la + \mu) x &= \la + \frac{\mu x^2}{\be^{r - 1}}.
\end{align}%
This set of equations is equivalent to
\begin{align}%
x &= \beta^r, \\
0 &= (\la + \mu) \beta^r - (\la + \mu \beta^{r + 1}). \label{eqnQBD:Erlang_services_apply_Rouche_on}
\end{align}%
We will apply Rouch\'e's theorem, see \cref{thm:Rouche}, to establish that \eqref{eqnQBD:Erlang_services_apply_Rouche_on} has $r$ roots inside the unit disk. Define $f(\beta) \defi (\la + \mu) \beta^r$ and $g(\beta) \defi - (\la + \mu \beta^{r + 1})$. Since both functions are polynomials, they are analytic functions for all $\beta \in \Complex$. Clearly, $f(\beta)$ has $r$ roots in the complex unit disk and we wish to establish $|f(\beta)| > |g(\beta)|$ for $|\beta| = 1$ so that by Rouch\'e's theorem, we have that \eqref{eqnQBD:Erlang_services_apply_Rouche_on} has $r$ roots in the complex unit disk. Observe that $|f(\beta)| = f(|\beta|)$ and $|g(\beta)| \le - g(|\beta|)$. Therefore we only require to that $f(|\beta|) > - g(|\beta|)$ for $|\beta| = 1$, but, for $|\beta| = 1$,
\begin{equation}%
f(|\beta|) = f(1) = \la + \mu = - g(1) = - g(|\beta|).
\end{equation}%
To resolve this issue, we essentially evaluate $f(|\beta|)$ and $g(|\beta|)$ along the circle $|\beta| = 1 - \epsilon$. We use the Taylor expansion at $|\beta| = 1$ to get $f(1 - \epsilon) = f(1) - \epsilon f'(1) + o(\epsilon)$ and similarly for $g(1 - \epsilon)$. So, we require to show that $f(1 - \epsilon) > g(1 - \epsilon)$ for $\epsilon$ small. However, since $f(1) = - g(1)$ the only thing we need is $f'(1) < - g'(1)$, with $f'(1) = r(\la + \mu)$ and $-g'(1) = (r + 1)\mu$, which is indeed the case by the stability condition $r \la < \mu$. Finally, we have established that \eqref{eqnQBD:Erlang_services_apply_Rouche_on} has $r$ roots in the complex unit disk. In \cite[Appendix~A]{Adan1996_Erlang-Erlang-c} the authors establish for a more general model that these roots are unique.

Label the $r$ roots inside the unit disk of \eqref{eqnQBD:Erlang_services_apply_Rouche_on} as $\beta_1,\beta_2,\ldots,\beta_r$ with corresponding $x_k = \beta_k^r, ~ k = 1,2,\ldots,r$. We have $r$ basis solutions of the form
\begin{equation}%
p(i,j) = \beta_k^j x_k^i, \quad k = 1,2,\ldots,r.
\end{equation}%
The next step is to take a linear combination of these basis solutions
\begin{equation}%
p(i,j) = \sum_{k = 1}^r \al_k \beta_k^j x_k^i, \quad i \ge 0, ~ j = 1,2,\ldots,r
\end{equation}%
and determine the coefficients $\al_1,\al_2,\ldots,\al_r$ and $p(0,0)$ from the balance equations of level 0 and the normalization condition.

%%%%%%%%%%%%%%%%%%%%%%%%%%%%%%%%%%%%%%%%%%%%%%%%%%%%%%%
%%%%%%%%%%%%%%%%%%%%%%%%%%%%%%%%%%%%%%%%%%%%%%%%%%%%%%%
%%%%%%%%%%%%%%%%%%% NEW SUBSECTION %%%%%%%%%%%%%%%%%%%%
%%%%%%%%%%%%%%%%%%%%%%%%%%%%%%%%%%%%%%%%%%%%%%%%%%%%%%%
%%%%%%%%%%%%%%%%%%%%%%%%%%%%%%%%%%%%%%%%%%%%%%%%%%%%%%%

\subsection{Tandem queue with blocking}%
\label{subsecQBD:tandem_queue}%

The final example is related to both loss networks and open Jackson networks, but is not an example of either of the two. The network consists of two stations. Jobs arrive to the first station according to a Poisson process with rate $\la$. The first station is an single-server queue where jobs are served in order of arrival and service takes an exponential amount of time with mean $1/\mu_1$. Jobs leaving the first station are routed to the second station. The second station is an Erlang-B model with $r$ servers. A service at station 2 is exponentially distributed with mean $1/\mu_2$. After receiving service at station 2, the job leaves the system. A departing job from station 1 that finds all servers occupied in station 2 is blocked and leaves the system as well. The first station is an $M/M/1$ queue and is therefore stable if $\la < \mu$; the second station is always stable.

The state of the system may be described by $X(t) \defi (X_1(t),X_2(t))$ with $X_i(t)$ the number of jobs at station $i$ at time $t$. The state space of this QBD process is $\statespace \defi \lvl{0} \cup \lvl{1} \cup \cdots$ with levels $\lvl{i} = \{ (i,0),(i,1),\ldots,(i,r) \}, ~ i \ge 0$. The transition rate diagram is shown in \cref{figQBD:transition_rate_diagram_tandem_queue}.

\begin{figure}
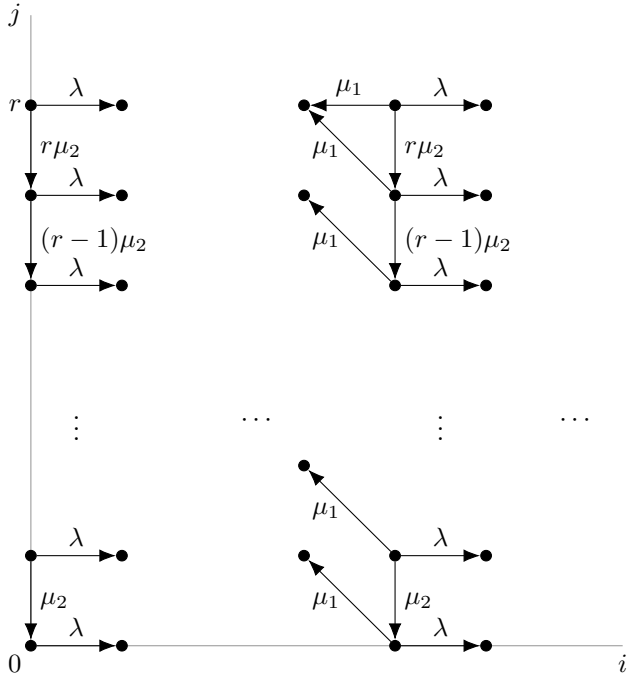
%
\centering%
\includestandalone{Chapters/QBD/TikZFiles/transition_rate_diagram_tandem_queues}%
\caption{Transition rate diagram of the tandem queue with blocking.}
\label{figQBD:transition_rate_diagram_tandem_queue}%
\end{figure}%

Let $p(i,j)$ denote the equilibrium probability of state $(i,j) \in \statespace$. From the transition rate diagram we obtain the global balance equations for level 0, with $0 < j < r$,
\begin{subequations}%
\label{eqnQBD:tandem_queue_global_balance_level_0}%
\begin{align}%
\la p(0,0) &= \mu_2 p(0,1), \\
(\la + j \mu_2) p(0,j) &= \mu_1 p(1,j - 1) + (j + 1) \mu_2 p(0,j + 1), \\
(\la + r \mu_2) p(0,r) &= \mu_1 \bigl( p(1,r - 1) + p(1,r) \bigr),
\end{align}%
\end{subequations}%
and for level $i \ge 1$, with $0 < j < r$,
\begin{subequations}%
\label{eqnQBD:tandem_queue_global_balance_level_i}%
\begin{align}%
(\la + \mu_1) p(i,0) &= \la p(i - 1,0) + \mu_2 p(i,1), \\
(\la + \mu_1 + j \mu_2) p(i,j) &= \la p(i - 1,j) + \mu_1 p(i + 1,j - 1) \notag \\
&\quad + (j + 1) \mu_2 p(i,j + 1), \\
(\la + \mu_1 +  r \mu_2) p(i,r) &= \la p(i - 1,r) \notag \\
&\quad + \mu_1 \bigl( p(i + 1,r - 1) + p(i + 1,r) \bigr),
\end{align}%
\end{subequations}%
We have learned from \cref{thmRN:Poisson_departure_process_state-dependent_service_times} that the output process of the first station is a Poisson process with rate $\la$. So it is not at all unreasonable to think that both stations operate independently and therefore the equilibrium distribution is a product of the equilibrium distributions of the first and second station. Both equilibrium distributions were already derived in \cref{ch:birth--and--death_processes}, see \cref{exBD:MMss_equilibrium_dist,exBD:MM1_equilibrium_dist}. Define $\rho_1 \defi \la/\mu_1$ and $\rho_2 \defi \la/\mu_2$, and let us validate if
\begin{equation}%
p(i,j) = \frac{1 - \rho_1}{\sum_{k = 0}^{r} \frac{\rho_2^k}{k!}} \rho_1^i \frac{\rho_2^j}{j!} \label{eqnQBD:tandem_queue_equilibrium_distribution}
\end{equation}%
is the equilibrium distribution of the tandem queue. Note that the normalization condition $\sum_{i \ge 0} \sum_{j = 0}^r p(i,j) = 1$ is satisfied. It can be easily verified that \eqref{eqnQBD:tandem_queue_equilibrium_distribution} is a solution to the global balance equations by substituting \eqref{eqnQBD:tandem_queue_equilibrium_distribution} into \eqref{eqnQBD:tandem_queue_global_balance_level_0} and \eqref{eqnQBD:tandem_queue_global_balance_level_i}. In conclusion, even though this tandem queue network has state-dependent routing, it still retains the explicit product-form equilibrium distribution that was encountered in the open Jackson networks of \cref{secRN:jackson_networks}.

The approach of making an educated guess for the equilibrium distribution and verifying its correctness through the global balance equations and the normalization condition is a powerful approach that can quickly lead to the solution. However, it is crucial that the problem is well understood so that intuition can lead to a correct guess for the form of the equilibrium distribution. Alternatively, the spectral expansion method leads to the same result, but more computations should be done to get there. For this method, we substitute the simple form $p(i,j) = y(j) x^i$ into the global balance equations and try to determine both parameters. This takes considerably more work than immediately guessing the correct expression for $p(i,j)$ as we have done for the current model.

%%%%%%%%%%%%%%%%%%%%%%%%%%%%%%%%%%%%%%%%%%%%%%%%%%%%%%%
%%%%%%%%%%%%%%%%%%%%%%%%%%%%%%%%%%%%%%%%%%%%%%%%%%%%%%%
%%%%%%%%%%%%%%%%%%%%% NEW SECTION %%%%%%%%%%%%%%%%%%%%%
%%%%%%%%%%%%%%%%%%%%%%%%%%%%%%%%%%%%%%%%%%%%%%%%%%%%%%%
%%%%%%%%%%%%%%%%%%%%%%%%%%%%%%%%%%%%%%%%%%%%%%%%%%%%%%%

\section{General quasi-birth--and--death processes}%
\label{secQBD:QBD}%

From the previous three examples we have seen that a QBD process consists of one infinite dimension and one finite dimension. The state space of a QBD processes can be partitioned in levels, where level 0 sometimes has a different number of states. This structure holds for the QBD processes that we are interested in. In particular,
\begin{equation}%
\lvl{0} \defi \{ (0,0),(0,1),\ldots,(0,b) \}, \quad \lvl{i} \defi \{ (i,0),(i,1),\ldots,(i,r) \}, ~ i \ge 1,
\end{equation}%
with $b$ and $r$ non-negative finite integers, so that the state space is given by
\begin{equation}%
\statespace \defi \lvl{0} \cup \lvl{1} \cup \lvl{2} \cup \cdots
\end{equation}%
We denote the state of the QBD process at time $t$ as $X(t) \defi (X_1(t),X_2(t))$ where $X_1(t)$ describes the level and $X_2(t)$ describes the phase at time $t$.

Throughout this chapter we focus on \textit{homogeneous} QBD processes, which means that transition rates are level-independent, possibly except for the transition rates from and to level 0.\endnote{QBD processes that are not homogeneous are called level-dependent QBD processes. These processes admit similar solutions to those presented for the three examples in \cref{secQBD:variations}, but now require level-dependent $R$ or $G$ matrices. For a detailed description of the methods involved see Bright and Taylor \cite{Bright1995_Level-DependentQBD} or Kharoufeh \cite{Kharoufeh2011_Level-DependentQBD}.} As stated in \cref{subsecQBD:setup_times}, the analogy with a BD process follows from the fact that transitions from a state within level $i \ge 1$ can only go to a state within level $i - 1$, $i$, or $i + 1$. The transition rate matrix of a homogeneous QBD process has block-tridiagonal structure
\begin{equation}%
Q = \begin{bmatrix}%
\La_{0}^{(0)}  & \La_{1}^{(0)}  \\
\La_{-1}^{(1)} & \La_{0}^{(1)} & \La_1 \\
               & \La_{-1}      & \La_0    & \La_1  \\
               &               & \La_{-1} & \La_0  & \La_1  \\
               &               &          & \ddots & \ddots & \ddots \\
\end{bmatrix}, \label{eqnQBD:definition_transition_rate_matrix}%
\end{equation}%
when the states are ordered according to their level and in increasing order within a level. The subscript $n$ of $\La_n$ denotes the change in levels for a transition. Element $(j,k)$ of $\La_n$ is the transition rate from state $(i,j)$ to state $(i + n,k)$ with $i \ge 2$. Note that elements $(j,j), ~ 0 \le j \le r$ of $\La_0$ are the exception to this rule; element $(j,j)$ is negative, but it's absolute value is exactly the rate at which the process leaves $(i,j), ~ i \ge 2$. This makes the row sums of $Q$ zero. The additional superscript $l$ in $\La_{n}^{(l)}$ indicates the dependence on the level $l$.

The matrix $\La_{0}^{(0)}$ is of dimension $(b + 1) \times (b + 1)$; $\La_{1}^{(0)}$ is of dimension $(b + 1) \times (r + 1)$; $\La_{-1}^{(1)}$ is of dimension $(r + 1) \times (b + 1)$; and $\La_{0}^{(1)}$, $\La_{-1}$, $\La_0$ and $\La_1$ are square matrices of dimension $r + 1$. Note that $\La \defi \La_{-1} + \La_0 + \La_1$ is a transition rate matrix that describes the behavior of the QBD process in the vertical direction only. The matrix $\La$ has negative entries on the main diagonal and non-negative entries elsewhere with row sums equal to zero.

%%%%%%%%%%%%%%%%%%%%%%%%%%%%%%%%%%%%%%%%%%%%%%%%%%%%%%%
%%%%%%%%%%%%%%%%%%%%%%%%%%%%%%%%%%%%%%%%%%%%%%%%%%%%%%%
%%%%%%%%%%%%%%%%%%%%% NEW SECTION %%%%%%%%%%%%%%%%%%%%%
%%%%%%%%%%%%%%%%%%%%%%%%%%%%%%%%%%%%%%%%%%%%%%%%%%%%%%%
%%%%%%%%%%%%%%%%%%%%%%%%%%%%%%%%%%%%%%%%%%%%%%%%%%%%%%%

\section{Modeling QBD processes}%
\label{secQBD:modeling_QBD_processes}%

In this section we present some examples in various application fields of Markov processes that are QBD processes. We focus mainly on the modeling aspect: the translation of a problem description to a QBD process with a state space and transition matrices.

\begin{example}[An insurance company]%
Claims arrive to an insurance company according to a Poisson process with rate $\la$. A claim is important with probability $p$. To achieve low waiting times for important claims, the insurance company is allowed to hold at most $r$ of these important claims at the same time; if new important claims arrive, they are diverted to a different insurance company. Standard claims are resolved one-by-one independently of the important claims and take $\Exp{\mu_1}$ time each. Important claims are also resolved one-by-one and take $\Exp{\mu_2}$ time each.

Denote by $X_1(t)$ and $X_2(t)$ the number of standard and important claims at time $t$ and by $X(t) \defi (X_1(t),X_2(t))$ the state of the system. Then, $\{ X(t) \}_{t \ge 0}$ is a QBD process with $b = r$, $\La_{-1} = \mu_1 \I$, $\La_1 = \la(1 - p) \I$, and
\[%
\La_0 = \begin{bmatrix}%
-\mu_1 & \la p \\
\mu_2  & -\mu  & \la p \\
       & \mu_2 & -\mu   & \la p \\
       &       & \ddots & \ddots & \ddots \\
       &       &        & \mu_2  & -\mu   & \la p \\
       &       &        &        & \mu_2  & -\mu + \la p
\end{bmatrix} - \la \I,%
\]%
where $\mu \defi \mu_1 + \mu_2$, $\I$ is the identity matrix and unlabeled elements of $\La_0$ are zero.
\end{example}%

\begin{example}[Experiments that require setup]%
A scientist is performing experiments. Requests for an additional experiment arrive according to a Poisson process with rate $\la$. An experiment requires two phases of setup; both take $\Exp{\theta}$ time. Once setup is completed, experiments can be performed one after the other, where an experiment takes $\Exp{\mu}$ time. However, when a request for an additional experiment arrives, the scientist gets distracted and the current experiment and the setup process have to be redone. Knowing this, the scientist does no setup when there are no experiments to be done.

Denote by $X_1(t)$ the number of experiments that still need to be done at time $t$, let $X_2(t)$ be the number of setup phases completed and let the state of the system be described by $X(t) \defi (X_1(t),X_2(t))$. Then, $\{ X(t) \}_{t \ge 0}$ is a QBD process with $b = 0$, $r = 2$,
\[%
\La_{-1} = \begin{bmatrix}%
0 & 0 & 0 \\
0 & 0 & 0 \\
0 & 0 & \mu
\end{bmatrix},%
\quad
\La_{0} = \begin{bmatrix}%
-\theta & \theta  & 0 \\
0       & -\theta & \theta \\
0       & 0       & -\mu
\end{bmatrix} - \la \I,%
\quad
\La_1 = \begin{bmatrix}%
\la & 0 & 0 \\
\la & 0 & 0 \\
\la & 0 & 0
\end{bmatrix}. \qedhere
\]%
\end{example}%

\begin{example}[A single-server queue in a random environment]%
A single server with an infinite capacity queue is serving jobs one at a time. An exogenous process (the random environment) changes the parameters of the system, where this process can be in any of $r + 1$ phases. If the random environment is in phase $n$, then jobs arrive according to a Poisson process with rate $\la_n$ and are served with exponential rate $\mu_n$. The only restriction required on the exogenous process is that transitions occur after some exponential time and that all states within a level can be reached.

Let the state of the system be denoted by $X(t) \defi (X_1(t),X_2(t)$, where $X_1(t)$ is the total number of jobs in the system at time $t$ and $X_2(t)$ is the phase of the random environment at time $t$. Then, $\{ X(t) \}_{t \ge 0}$ is a QBD process with $b = r$, $\La_{-1} = \diag{\mu_0,\mu_1,\ldots,\mu_r}$, $\La_{1} = \diag{\la_0,\la_1,\ldots,\la_r}$, and $\La_0 = E - \La_{-1} - \La_1$, where $E$ is the generator of the exogenous process and $\diag{\vca{x}}$ is a square matrix with the vector $\vca{x}$ on the main diagonal.
\end{example}%

\begin{example}[Make to order and make to stock \textup{\cite{Adan1998_MTO_MTS}}]%
Standard products and customer-specific prototypes are produced by the same high-tech company. Demand for standard products arrives according to a Poisson process with rate $\la_1$ and demand for prototypes according to a Poisson process with rate $\la_2$. If the company has no outstanding orders, it makes standard products to stock. The company is willing to have at most $r$ standard products on stock to avoid high holding costs. A demand for a standard product is immediately satisfied whenever stock is available, otherwise the standard product is produced to order. Prototypes are customer specific and are therefore made to order. Since prototypes yield higher monetary returns, producing these products has preemptive priority over producing standard products. Producing either product takes $\Exp{\mu}$ time.

If we denote by $X_1(t)$ the total number of outstanding orders (standard plus prototypes) at time $t$, by $X_2(t)$ the number of standard products on stock at time $t$ and by $X(t) \defi (X_1(t),X_2(t))$ the state of the system, then $\{ X(t) \}_{t \ge 0}$ is a QBD process with $b = r$, $\La_{-1} = \mu \I$, $\La_1 = \diag{\la_1 + \la_2,\la_2,\ldots,\la_2}$, and
\[%
\La_0 = \begin{bmatrix}%
0 \\
\la_1 & 0 \\
      & \la_1 & 0 \\
      &       & \ddots & \ddots \\
      &       &        & \la_1 & 0
\end{bmatrix} -(\la + \mu)\I,%
\quad
\La_{0}^{(0)} = \La_0 + \begin{bmatrix}%
0 & \mu \\
  & 0 & \mu \\
  &   & \ddots & \ddots \\
  &   &        & 0      & \mu \\
  &   &        &        & 0
\end{bmatrix},%
\]%
where $\la \defi \la_1 + \la_2$.
\end{example}%

\begin{example}[An encryption server with inspection]\label{exQBD:encryption_server}%
A computing facility has a single server that encrypts files. Tasks arrive according to a Poisson process with rate $\la$ and wait in an infinite queue if the server is busy. Before encrypting a file, the server inspects the contents of the file and decides on a certain technology to use. Inspecting a file takes $\Exp{\theta}$ time. The server uses three different encryption types: Data Encryption Standard (type 1), Advanced Encryption Standard (type 2), and RC4 (type 3). A file requires type $n$ encryption with probability $p_n$ and $p_1 + p_2 + p_3 = 1$. Type $n$ encryption takes $\Exp{\mu_n}$ time.

Let $X_1(t)$ be the number of files that still need to be encrypted at time $t$ and let $X_2(t)$ be the encryption type of the file being encrypted at time $t$, where $X_2(t) = 0$ indicates that the server is still in the process of inspecting the file (or idle, if $X_1(t) = 0$ as well). The state of the system is denoted by $X(t) \defi (X_1(t),X_2(t))$. The Markov process $\{ X(t) \}_{t \ge 0}$ is a QBD process with $b = 0$, $r = 3$, $\La_1 = \la \I$,
\[%
\La_{-1} = \begin{bmatrix}%
0 & 0 & 0 & 0 \\
\mu_1 & 0 & 0 & 0 \\
\mu_2 & 0 & 0 & 0 \\
\mu_3 & 0 & 0 & 0
\end{bmatrix},%
\quad
\La_0 = \begin{bmatrix}%
-\theta & \theta p_1 & \theta p_2 & \theta p_3 \\
0       & -\mu_1     & 0          & 0 \\
0       & 0          & -\mu_2     & 0 \\
0       & 0          & 0          & -\mu_3
\end{bmatrix} - \la \I. \qedhere%
\]%
\end{example}%

%%%%%%%%%%%%%%%%%%%%%%%%%%%%%%%%%%%%%%%%%%%%%%%%%%%%%%%
%%%%%%%%%%%%%%%%%%%%%%%%%%%%%%%%%%%%%%%%%%%%%%%%%%%%%%%
%%%%%%%%%%%%%%%%%%%%% NEW SECTION %%%%%%%%%%%%%%%%%%%%%
%%%%%%%%%%%%%%%%%%%%%%%%%%%%%%%%%%%%%%%%%%%%%%%%%%%%%%%
%%%%%%%%%%%%%%%%%%%%%%%%%%%%%%%%%%%%%%%%%%%%%%%%%%%%%%%

\section{Stability condition}%
\label{secQBD:stability_condition}%

From here on we will assume that the QBD process $\{ X(t) \}_{t \ge 0}$ is irreducible and that the transition rate matrix $\La$ has exactly one communicating class. The condition for this Markov process to be positive recurrent is an intuitive one and is easily extended from the homogeneous BD processes case. In the homogenous BD case the process is positive recurrent (or stable) if the birth rate is smaller than the death rate. This implies that the process does not drift off to infinity, because the net drift (birth rate minus death rate) is negative. For the QBD case we establish a similar condition for positive recurrence based on the net drift.

The QBD process adds a finite number of phases to the BD process and transition rates to the left and right can vary from phase to phase. Just like the BD case, the QBD process should be stable if the mean drift to the left is larger than the mean drift to the right. This way, the process does not drift off to higher and higher levels. Now, the mean drift to the left or right depends on the transition rates at each phase, and more importantly, depends on the fraction of time the process spends in each of its phases. The last quantity is determined from the transition rate matrix $\La$ of the vertical direction. Let $\vca{x}$ be the equilibrium distribution of the vertical direction:
\begin{equation}%
\vca{x} \La = \zerob, \quad \vca{x} \oneb = 1.
\end{equation}%
Element $j$ of $\vca{x}$ is interpreted as the fraction of time that the QBD process is in phase $j$ when it is far away from level 0 (so that boundary effects do not play a role). With this in mind, the mean drift from level $i$ to level $i - 1$ is $\vca{x} \La_{-1} \oneb$ and the mean drift from level $i$ to level $i + 1$ is $\vca{x} \La_1 \oneb$. The net mean drift is then $\vca{x} \La_1 \oneb - \vca{x} \La_{-1} \oneb$ and the process is positive recurrent---also called stable---if the net mean drift is negative. This condition is known as Neuts' \textit{mean drift condition} \cite[Theorem~3.1.1]{Neuts1994_Matrix-geometric} or the stability condition and we present it here as a theorem.

\begin{theorem}[Stability condition]\label{thmQBD:stability_condition}%
The \textup{QBD} process $\{ X(t) \}_{t \ge 0}$ is positive recurrent if and only if
\begin{equation}%
\vca{x} \La_1 \oneb < \vca{x} \La_{-1} \oneb \label{eqnQBD:stability_condition}
\end{equation}%
with $\vca{x} = \begin{bmatrix} x(0) & x(1) & \cdots & x(r) \end{bmatrix}$  the equilibrium distribution of the Markov process with transition rate matrix $\La \defi \La_{-1} + \La_0 + \La_1$\textup{:}
\begin{equation}%
\vca{x} \La = \zerob, \quad \vca{x} \oneb = 1. \label{eqnQBD:stability_condition_vertical_eq_dist}
\end{equation}%
\end{theorem}%

\begin{example}[An encryption server with inspection]%
We derive the stability condition of the QBD process of \cref{exQBD:encryption_server}. The transition rate matrix of the vertical direction is
\begin{equation}%
\La = \begin{bmatrix}%
-\theta & \theta p_1 & \theta p_2 & \theta p_3 \\
\mu_1   & -\mu_1     & 0          & 0 \\
\mu_2   & 0          & -\mu_2     & 0 \\
\mu_3   & 0          & 0          & -\mu_3
\end{bmatrix}.\label{eqnQBD:encryption_server_vertical_generator}%
\end{equation}%
The dependent system of equations \eqref{eqnQBD:stability_condition_vertical_eq_dist} can easily be solved by replacing the left-most column in the generator \eqref{eqnQBD:encryption_server_vertical_generator} by ones (we briefly refer to this modified generator as $\La^*$) and solve the system $\vca{x} \La^* = \begin{bmatrix} 1 & 0 & 0 & 0 \end{bmatrix}$ to obtain
\begin{equation}%
\vca{x} = \frac{1}{1 + \sum_{n = 1}^3 \frac{\theta p_n}{\mu_n}} \begin{bmatrix}%
1 & \frac{\theta p_1}{\mu_1} & \frac{\theta p_2}{\mu_2} & \frac{\theta p_3}{\mu_3}
\end{bmatrix}.%
\end{equation}%
So, for this process, the stability condition \eqref{eqnQBD:stability_condition} reads
\begin{equation}%
\la < \frac{\theta}{1 + \sum_{n = 1}^3 \frac{\theta p_n}{\mu_n}} = \frac{1}{\frac{1}{\theta} + \sum_{n = 1}^3 p_n \frac{1}{\mu_n}}.
\end{equation}%
The mean drift to the right is clear: from every phase an arrival can occur with rate $\la$. The mean drift to the left is the inverse of the mean service time. The service time consists of the setup phase (exponential with mean $1/\theta$) plus the encryption, where type-$n$ encryption occurs with probability $p_n$ and is exponential with mean $1/\mu_n$.
\end{example}%

%%%%%%%%%%%%%%%%%%%%%%%%%%%%%%%%%%%%%%%%%%%%%%%%%%%%%%%
%%%%%%%%%%%%%%%%%%%%%%%%%%%%%%%%%%%%%%%%%%%%%%%%%%%%%%%
%%%%%%%%%%%%%%%%%%%%% NEW SECTION %%%%%%%%%%%%%%%%%%%%%
%%%%%%%%%%%%%%%%%%%%%%%%%%%%%%%%%%%%%%%%%%%%%%%%%%%%%%%
%%%%%%%%%%%%%%%%%%%%%%%%%%%%%%%%%%%%%%%%%%%%%%%%%%%%%%%

\section{Matrix-geometric method}%
\label{secQBD:matrix-geometric_method}%

The aim of the \textit{matrix-geometric method}\endnote{The matrix-geometric method was pioneered by Evans \cite{Evans1967_Matrix-geometric} and Wallace \cite{Wallace1969_QBD}, fully developed by Neuts \cite{Neuts1994_Matrix-geometric}, and discussed at length in the classical work of Latouche and Ramaswami \cite{Latouche1999_Matrix-analytic}.} is to characterize the equilibrium probabilities
\begin{equation}%
p(i,j) \defi \lim_{t \to \infty} \Prob{X_1(t) = i, X_2(t) = j}, \quad (i,j) \in \statespace,
\end{equation}%
as a matrix-geometric distribution in terms of the levels. In the examples of \cref{secQBD:variations} we have seen that the rate matrix $R$ plays a key role. This is also true for the general class of QBD processes. Recall from \cref{subsecQBD:setup_times} that element $(R)_{j,k}$ is the expected time spent in state $(i + 1,k)$ multiplied by $-(\La_0)_{j,j}$ before the first return to level $i$, given the initial state $(i,j)$ with $i \ge 1$. Note that $-(\La_0)_{j,j}$ is the rate at which the process leaves state $(i,j), ~ i \ge 1$. From the interpretation of $R$ we directly conclude that zero rows of $\La_{1}$ correspond to zero rows in $R$.

We denote the equilibrium probability vectors as
\begin{align*}%
\pb_0 &\defi \begin{bmatrix} p(0,0) & p(0,1) & \cdots & p(0,b) \end{bmatrix}, \\
\pb_i &\defi \begin{bmatrix} p(i,0) & p(i,1) & \cdots & p(i,r) \end{bmatrix}, \quad i \ge 1,
\end{align*}%
and $\pb = \begin{bmatrix} \pb_0 & \pb_1 & \cdots \end{bmatrix}$. The balance equations for the QBD process with transition rate matrix $Q$ partitioned by levels are given by
\begin{align}%
\pb_0 \La_{0}^{(0)} + \pb_1 \La_{-1}^{(1)} &= \zerob, \label{eqnQBD:balance_equations_level_0} \\
\pb_0 \La_{1}^{(0)} + \pb_1 \La_{0}^{(1)} + \pb_2 \La_{-1} &= \zerob, \label{eqnQBD:balance_equations_level_1} \\
\pb_{i - 1} \La_1 + \pb_i \La_0 + \pb_{i + 1} \La_{-1} &= \zerob, \quad i \ge 2. \label{eqnQBD:balance_equations_level_i}
\end{align}%
The next result, appearing in \cite[Theorem~3.1.1]{Neuts1994_Matrix-geometric}, describes the matrix-geometric structure of the equilibrium probability vectors. Notice that \eqref{eqnQBD:R_matrix_equation} is obtained by substituting \eqref{eqnQBD:eq_dist_R} into \eqref{eqnQBD:balance_equations_level_i}.

\begin{theorem}\label{thmQBD:matrix-geometric_distribution}%
Provided the \textup{QBD} process $\{ X(t) \}_{t \ge 0}$ is irreducible and positive recurrent, the equilibrium probability vector $\pb$, satisfying $\pb Q = \zerob, ~ \pb \oneb = 1$, is given by
\begin{equation}%
\pb_{i + 1} = \pb_i R = \pb_1 R^i, \quad i \ge 1, \label{eqnQBD:eq_dist_R}
\end{equation}%
where $R$, called the \textup{rate matrix}, is the minimal non-negative solution of the matrix-quadratic equation
\begin{equation}%
R^2 \La_{-1} + R \La_0 + \La_1 = 0. \label{eqnQBD:R_matrix_equation}
\end{equation}%
The equilibrium probability vectors $\pb_0$ and $\pb_1$ follow from the system of equations
\begin{align*}%
\pb_0 \La_{0}^{(0)} + \pb_1 \La_{-1}^{(1)} &= \zerob, \\
\pb_0 \La_{1}^{(0)} + \pb_1 \bigl( \La_{0}^{(1)} + R \La_{-1} \bigr) &= \zerob,
\end{align*}%
and the normalization condition $\pb_0 \oneb + \pb_1 (\I - R)^{-1} \oneb = 1$.
\end{theorem}

For the computation of the rate matrix $R$ we may rewrite \eqref{eqnQBD:R_matrix_equation} in the form
\begin{equation}%
R = - \bigl( \La_1 + R^2 \La_{-1} \bigr) \La_0^{-1}. \label{eqnQBD:fixed_point_equation_R}
\end{equation}%
The matrix $\La_0$ is indeed invertible, since it is a transient generator, which means that it is the transition rate matrix of a transient Markov process. More precisely, for a transient (substochastic) transition probability matrix $P$ associated with a Markov chain, we know that $P^n \to 0$ as $n \to 0$. This convergence is geometric, the decay parameter of which is the largest eigenvalue of $P$, which is less than 1. So, the series $\sum_{n \ge 0} P^n = (\I - P)^{-1}$ converges and therefore the inverse of $I - P$ exists. Now, in the continuous time setting we can construct the transient transition probability matrix $P$ as $P = \I + \Delta Q$, where $0 < \Delta < \max_i - (Q)_{i,i}$. Since $P$ is a transient transition probability matrix, the series $\sum_{n \ge 0} P^n = (\I - P)^{-1} = (-\Delta Q)^{-1}$, and therefore $Q$ is invertible.

The above fixed point equation \eqref{eqnQBD:fixed_point_equation_R} may be solved by straightforward successive substitutions, so
\begin{equation}%
R_{n + 1} = - \bigl( \La_1 + R_n^2 \La_{-1} \bigr) \La_0^{-1}, \quad n \ge 0, \label{eqnQBD:R_successive_substitutions}
\end{equation}%
starting with $R_0 = 0$ and $R_n \uparrow R$ as $n \to \infty$.\endnote{Neuts shows that $R_n \uparrow R$ as $n \to \infty$, see \cite[Lemma~1.2.3]{Neuts1994_Matrix-geometric} for the proof in case of a discrete-time QBD chain. More sophisticated and efficient numerical schemes for determining $R$ have been developed, in particular, cyclic reduction \cite{Bini1996_Cyclic_Reduction} and logarithmic reduction \cite{Latouche1993_Logarithmic_Reduction}.} For computational purposes a stopping criterion is required: we show one choice for a stopping criterion in \cref{algQBD:matrix-geometric_successive_substitutions}. Note that for a matrix $A$ the max norm is $\| A \|_{\textup{max}} \defi \max_{i,j} |(A)_{i,j}|$.

\begin{algorithm}%
\caption{Calculating $R$ using successive substitutions}%
\label{algQBD:matrix-geometric_successive_substitutions}%
\begin{algorithmic}[1]%
\State Pick $\epsilon$ small and positive
\State Set $R_0 = 0$, $R_1 = - \La_1 \La_0^{-1}$ and $n = 1$
\While{$\| R_n - R_{n - 1} \|_{\textup{max}} > \epsilon$}
    \State Compute $R_{n + 1}$ according to \eqref{eqnQBD:R_successive_substitutions}
    \State Update $n = n + 1$
\EndWhile
\end{algorithmic}%
\end{algorithm}%

%%%%%%%%%%%%%%%%%%%%%%%%%%%%%%%%%%%%%%%%%%%%%%%%%%%%%%%
%%%%%%%%%%%%%%%%%%%%%%%%%%%%%%%%%%%%%%%%%%%%%%%%%%%%%%%
%%%%%%%%%%%%%%%%%%% NEW SUBSECTION %%%%%%%%%%%%%%%%%%%%
%%%%%%%%%%%%%%%%%%%%%%%%%%%%%%%%%%%%%%%%%%%%%%%%%%%%%%%
%%%%%%%%%%%%%%%%%%%%%%%%%%%%%%%%%%%%%%%%%%%%%%%%%%%%%%%

\subsection{Explicit solutions for the rate matrix}%
\label{subsecQBD:R_explicit_solutions}%

We have seen in \cref{subsecQBD:setup_times} that the rate matrix $R$ can be determined explicitly. Also the example of \cref{subsecQBD:Erlang_services} admits an explicit solution for $R$. This is not always the case, however. We now review two cases in which the rate matrix $R$ can be determined explicitly.\endnote{Ramaswami \cite{Ramaswami1986_Explicit_R} treats the two cases in which the $R$ matrix can be obtained explicitly.}

The first case assumes that the transition rate matrix $\La_{-1}$ with transitions to the left is of the form
\begin{equation}%
\La_{-1} = \vc{\al} \vc{\be}, \label{eqnQBD:R_explicit_La_-1}
\end{equation}%
where $\vc{\al}$ is a column vector and $\vc{\be}$ is a \textit{stochastic} row vector, both of dimension $r + 1$:
\begin{equation}%
\vc{\al} = \begin{bmatrix}%
\al_0 \\
\al_1 \\
\vdots \\
\al_r
\end{bmatrix},%
\quad
\vc{\be} = \begin{bmatrix}%
\be_0 & \be_1 & \cdots & \be_r
\end{bmatrix},%
\quad
\vc{\be} \oneb = 1, ~ \vc{\al},\vc{\be} > \zerob
\end{equation}%
where $\vca{x} > \zerob$ indicates that all elements of $\vca{x}$ are non-negative and at least one element is positive. This means that all rows of $\La_{-1}$ are the same up to some scaling: from all states $(i + 1,j), ~ j = 0,1,\ldots,r$ the probability of jumping to state $(i,k), ~ k = 0,1,\ldots,r$ is \textit{independent} of the starting state in level $i + 1$. We investigate the consequences for the rate matrix $R$.

Substitution of \eqref{eqnQBD:R_explicit_La_-1} into the balance equations for level $i > 1$ yields
\begin{equation}%
\pb_{i - 1} \La_1 + \pb_i \La_0 + \pb_{i + 1} \vc{\al} \vc{\be} = \zerob. \label{eqnQBD:R_explicit_La_-1_eq_eqs}
\end{equation}%
To eliminate $\pb_{i + 1}$ from this equation we derive a relation between $\pb_i$ and $\pb_{i + 1}$ by equating the flow between level $i$ and level $i + 1$:
\begin{equation}%
\pb_i \La_1 \oneb = \pb_{i + 1} \La_{-1} \oneb = \pb_{i + 1} \vc{\al} \vc{\be} \oneb = \pb_{i + 1} \vc{\al}. \label{eqnQBD:R_explicit_La_-1_relation_level_i_and_i+1}
\end{equation}%
Substituting \eqref{eqnQBD:R_explicit_La_-1_relation_level_i_and_i+1} into \eqref{eqnQBD:R_explicit_La_-1_eq_eqs}, we obtain
\begin{equation}%
\pb_{i - 1} \La_1 + \pb_i \La_0 + \pb_i \La_1 \oneb \vc{\be} = \zerob,
\end{equation}%
which can be rewritten as
\begin{equation}%
\pb_i = \pb_{i - 1} R, \quad i > 1,
\end{equation}%
with the explicit formulation
\begin{equation}%
R = - \La_1 \bigl( \La_0 + \La_1 \oneb \vc{\be} \bigr)^{-1},
\end{equation}%
where $\La_0 + \La_1 \oneb \vc{\be}$ is invertible, since it is a transient generator.

The second case for which $R$ can be solved explicitly is when $\La_1$ is of the form
\begin{equation}%
\La_1 = \vc{\al} \vc{\be}.
\end{equation}%
Similarly to the first case, this means that all rows of $\La_1$ are the same up to some scaling.

From the recursive scheme \eqref{eqnQBD:R_successive_substitutions} we obtain
\begin{equation}%
R_0 = 0, \quad R_1 = -\La_1 \La_0^{-1} = - \vc{\al} \vc{\be} \La_0^{-1} = \vc{\al} \vc{\ga}_1,
\end{equation}%
with row vector $\vc{\ga}_1 = - \vc{\be} \La_0^{-1}$. Repeating the iteration shows that all $R_n$'s are of the form $R_n = \vc{\al} \vc{\ga}_n$, where $\vc{\ga}_n > \zerob$ is a row vector of dimension $r + 1$. Since $R_n \uparrow R$ as $n \to \infty$, we conclude that $\vc{\ga}_n \uparrow \vc{\ga}$ and
\begin{equation}%
R = \vc{\al} \vc{\ga}, \label{eqnQBD:R_explicit_La_1_rank_1}
\end{equation}%
for some vector $\vc{\ga} > \zerob$. So, $R$ is a matrix of rank 1 and has a single non-zero eigenvalue which is equal to $\trace{R}$. This implies that
\begin{equation}%
R^i = (\vc{\ga} \vc{\al})^{i - 1} R = \eta^{i - 1} R, \quad i \ge 1 \label{eqnQBD:R_explicit_La_1_sprad}
\end{equation}%
where $\eta \defi \vc{\ga} \vc{\al} = \trace{R}$ and equal to the spectral radius $\sprad{R}$ defined as
\begin{equation}%
\sprad{A} \defi \max \{ |\eta_0|,|\eta_1|,\ldots,|\eta_r| \},
\end{equation}%
where $\eta_0,\eta_1,\ldots,\eta_r$ are the eigenvalues of a matrix $A$ of dimension $r + 1$.

Observing \eqref{eqnQBD:R_explicit_La_1_sprad}, the matrix-geometric form in \cref{thmQBD:matrix-geometric_distribution} reduces to
\begin{equation}%
\pb_{i + 1} = \pb_1 R^i = \eta^{i - 1} \pb_1 R = \eta^{i - 1} \pb_2, \quad i \ge 1. \label{eqnQBD:R_explicit_La_1_eq_dist}
\end{equation}%
What remains is to determine $\eta$.

\begin{proposition}%
The spectral radius $\eta$ of $R$ for the case $\La_1 = \vc{\al} \vc{\be}$ can be characterized as the unique root in $(0,1)$ of the determinant equation
\begin{equation}%
\det \bigl( \La_1 + \eta \La_0 + \eta^2 \La_{-1} \bigr) = 0.
\end{equation}%
\end{proposition}%

\begin{proof}%
Using \eqref{eqnQBD:R_explicit_La_1_sprad} in \eqref{eqnQBD:R_matrix_equation} results in
\begin{equation}%
R = - \La_1 \bigl( \La_0 + \eta \La_{-1} \bigr)^{-1}.
\end{equation}%
The eigenvalue $\eta$ then follows from
\begin{align}%
0 &= \det \bigl( R - \eta \I \bigr) \notag \\
  &= \det \bigl( - \La_1 \bigl( \La_0 + \eta \La_{-1} \bigr)^{-1}   - \eta \I \bigr) \notag \\
  &= \det \Bigl( \bigl( \La_1 + \eta ( \La_0 + \eta \La_{-1}) \bigr) \bigl( - \La_0 - \eta \La_{-1} \bigr)^{-1} \Bigr) \notag \\
  &= \det \bigl( \La_1 + \eta \La_0 + \eta^2 \La_{-1} \bigr) \det \bigl( \bigl( - \La_0 - \eta \La_{-1} \bigr)^{-1} \bigr).
\end{align}%
Since $\eta < 1$, $- \La_0 - \eta \La_{-1}$ is nonsingular. So, $\eta$ satisfies
\begin{equation}%
0 = \det \bigl( \La_1 + \eta \La_0 + \eta^2 \La_{-1} \bigr).
\end{equation}%
Establishing that there is a single $\eta \in (0,1)$ follows from \cite[Proof of Theorem~4]{Ramaswami1986_Explicit_R}.
\end{proof}%

%%%%%%%%%%%%%%%%%%%%%%%%%%%%%%%%%%%%%%%%%%%%%%%%%%%%%%%
%%%%%%%%%%%%%%%%%%%%%%%%%%%%%%%%%%%%%%%%%%%%%%%%%%%%%%%
%%%%%%%%%%%%%%%%%%% NEW SUBSECTION %%%%%%%%%%%%%%%%%%%%
%%%%%%%%%%%%%%%%%%%%%%%%%%%%%%%%%%%%%%%%%%%%%%%%%%%%%%%
%%%%%%%%%%%%%%%%%%%%%%%%%%%%%%%%%%%%%%%%%%%%%%%%%%%%%%%

\subsection{Exact solutions for the rate matrix}%
\label{subsecQBD:R_exact_solutions}%

If the transition matrices $\La_{-1}$, $\La_0$, and $\La_1$ are all upper or all lower triangular, then the interpretation of the elements of $R$ shows us that the rate matrix is also upper or lower triangular. Determining $R$ in this case is made easier by exploiting its structure.\endnote{This class of QBD processes (or slight variations of it) is briefly mentioned in \cite[Section~6.5]{Neuts1994_Matrix-geometric} and has hence been studied extensively in the literature and has lead to many variations of exact or explicit solutions for $R$. Some examples include exploiting structural properties of \eqref{eqnQBD:R_matrix_equation} in \cite{Houdt2011_Triangular_R}; lattice path counting solutions \cite{Leeuwaarden2009_Lattice_path_counting,Leeuwaarden2006_Explicit_R}; and renewal-reward based approaches in \cite{Gandhi2014_RRR}.}

We will highlight the application of the methodology in \cite{Houdt2011_Triangular_R} to QBD processes with upper triangular transition matrices. Since $R$ is upper triangular, $R^2$ has the same upper triangular structure and has elements $(R^2)_{i,j} = \sum_{k = i}^j (R)_{i,k} (R)_{k,j}, ~ 0 \le i \le j \le r$. From \eqref{eqnQBD:R_matrix_equation} we deduce that the diagonal elements are the minimal non-negative solution of
\begin{equation}%
(R)_{i,i}^2 (\La_{-1})_{i,i} + (R)_{i,i} (\La_0)_{i,i} + (\La_1)_{i,i} = 0, \quad 0 \le i \le r.
\end{equation}%
So,
\begin{equation}%
(R)_{i,i} = \frac{ -(\La_0)_{i,i} - \sqrt{ \bigl( (\La_0)_{i,i} \bigr)^2 - 4 (\La_{-1})_{i,i} (\La_1)_{i,i} } }{2(\La_{-1})_{i,i}}.
\end{equation}%
The elements on the superdiagonal of $R$ are determined recursively and also follow from \eqref{eqnQBD:R_matrix_equation}, for $0 \le i < j \le r$:
\begin{equation}%
\sum_{k = i}^j \sum_{l = i}^k (R)_{i,l} (R)_{l,k} (\La_{-1})_{k,j} + \sum_{k = i}^j (R)_{i,k} (\La_0)_{i,j} + (\La_1)_{i,j} = 0.
\end{equation}%
Solving the above equation for $(R)_{i,j}$ yields
\begin{equation}%
(R)_{i,j} = - \frac{R_{\textup{nmr}}(i,j) + (\La_1)_{i,j}}{\bigl( (R)_{i,i} + (R)_{j,j} \bigr) (\La_{-1})_{j,j} + (\La_0)_{j,j} }, \quad 0 \le i < j \le r,
\end{equation}%
where
\begin{align}%
R_{\textup{nmr}}(i,j) &= \sum_{k = i}^{j - 1} \Bigl( \sum_{l = i}^k (R)_{i,l} (R)_{l,k} (\La_{-1})_{k,j} + (R)_{i,k}(\La_0)_{k,j} \Bigr) \notag \\
&\quad + \sum_{k = i + 1}^{j - 1} (R)_{i,k} (R)_{k,j} (\La_{-1})_{j,j},
\end{align}%
with the convention $\sum_{n = n_0}^{n_1} f(n) = 0$ if $n_0 > n_1$. The above equation is a recursion along the superdiagonals of $R$, which should be solved starting at the superdiagonal closest to the main diagonal and moving to the top right corner of the matrix.

We finally note that the inverse of an upper triangular matrix is again upper triangular; the same applies for lower triangular matrices. For the determination of $\pb_0$ and $\pb_1$ the inverse $(I - R)^{-1}$ is required. Provided the diagonal elements of an upper triangular matrix $A$ of dimension $r + 1$ are non-zero, the inverse can be determined as
\begin{align}%
(A^{-1})_{i,i} &= \frac{1}{(A)_{i,i}}, \quad 0 \le i \le r, \\
(A^{-1})_{i,j} &= - (A^{-1})_{i,i} \sum_{k = i + 1}^j (A)_{i,k} (A^{-1})_{k,j}, \quad 0 \le i < j \le r,
\end{align}%
where the recursion should be solved along the superdiagonals, starting at the main diagonal, followed by the superdiagonal closest to the main diagonal, and so forth, exactly as for $R$.

%%%%%%%%%%%%%%%%%%%%%%%%%%%%%%%%%%%%%%%%%%%%%%%%%%%%%%%
%%%%%%%%%%%%%%%%%%%%%%%%%%%%%%%%%%%%%%%%%%%%%%%%%%%%%%%
%%%%%%%%%%%%%%%%%%%%% NEW SECTION %%%%%%%%%%%%%%%%%%%%%
%%%%%%%%%%%%%%%%%%%%%%%%%%%%%%%%%%%%%%%%%%%%%%%%%%%%%%%
%%%%%%%%%%%%%%%%%%%%%%%%%%%%%%%%%%%%%%%%%%%%%%%%%%%%%%%

\section{Matrix-analytic method}%
\label{secQBD:matrix-analytic_method}%

The main object of study in the \textit{matrix-analytic method}\endnote{The \textit{matrix-analytic} method was pioneered by Neuts \cite{Neuts1989_M-G-1_type,Neuts1994_Matrix-geometric} in the 1980's and developed further in \cite{Latouche1999_Matrix-analytic}. As we will see in later chapters, the matrix-geometric and matrix-analytic methods have their own area of application, but these areas overlap in case of QBD processes. A review of the matrix-analytic method and recent tutorial on the application of this method can be found in \cite{Neuts1984_Matrix-analytic_review} and \cite{Riska2002_M-G-1_tutorial}, respectively.} is the auxiliary matrix $G$. Element $(G)_{j,k}$ is a \textit{first passage probability}, defined as the probability that, starting at level $i \ge 2$ in state $(i,j)$, the first passage to levels $i - 1$ and below happens in state $(i - 1,k)$. Note that indeed, if $i \ge 2$, the first passage probabilities do not depend on $i$ due to the homogeneous transition structure. Moreover, if the QBD process is positive recurrent and irreducible, then $G$ is a right stochastic matrix.

Similar to the derivation of $G$ in \cref{subsecQBD:Erlang_services}, define an \textit{excursion} as a sample path of the process starting at level $i$, moving to levels higher than $i$ and ending at first return to level $i$. Clearly, the number of excursions per time unit that end in state $(i,j)$ is equal to $\sum_{k = 0}^r p(i + 1,k) (\La_{-1})_{k,j}$. Next, the number of excursions per time unit that start from state $(i,k)$, immediately go to state $(i + 1,l)$, and ultimately end the excursion in state $(i,j)$ is $p(i,k) (\La_1)_{k,l} (G)_{l,j}$, where we exploit the interpretation of the elements of $G$. Summing over all possible starting states in level $i$ and the state first visited in level $i + 1$ also gives us the number of excursions per time unit that end in state $(i,j)$, namely $\sum_{k = 0}^r p(i,k) \sum_{l = 0}^r (\La_1)_{k,l} (G)_{l,j}$. Equating both expressions for the number of excursions per time unit that end in state $(i,j), ~ 0 \le j \le r$ and writing it in vector-matrix form yields
\begin{equation}%
\pb_{i + 1} \La_{-1} = \pb_i \La_1 G.
\end{equation}%
Substituting this relation in \eqref{eqnQBD:balance_equations_level_i} produces
\begin{equation}%
\pb_{i - 1} \La_1 + \pb_i ( \La_0 + \La_1 G) = \zerob.
\end{equation}%
Based on this probabilistic derivation we have the following equivalent result to \cref{thmQBD:matrix-geometric_distribution} for the matrix-analytic method.\endnote{The formal proof of \cref{thmQBD:matrix-analytic_distribution} follows from \cref{thmQBD:matrix-geometric_distribution} in combination with \cite[Proposition~6.4.2]{Latouche1999_Matrix-analytic}.}

\begin{theorem}\label{thmQBD:matrix-analytic_distribution}%
Provided the \textup{QBD} process $\{ X(t) \}_{t \ge 0}$ is irreducible and positive recurrent, the stationary probability vector $\pb$, satisfying $\pb Q = \zerob, ~ \pb \oneb = 1$, is given by
\begin{equation}%
\pb_{i + 1} = \pb_i \La_1(- \La_0 - \La_1 G)^{-1} = \pb_1 \bigl( \La_1(- \La_0 - \La_1 G)^{-1} \bigr)^i, \quad i \ge 1, \label{eqnQBD:eq_dist_G}
\end{equation}%
where $G$ is the right stochastic solution of the matrix-quadratic equation
\begin{equation}%
\La_{-1} + \La_0 G + \La_1 G^2 = 0.
\end{equation}%
The equilibrium probability vectors $\pb_0$ and $\pb_1$ follow from the system of equations
\begin{align*}%
\pb_0 \La_{0,0} + \pb_1 \La_{1,-1} &= \zerob, \\
\pb_0 \La_{0,1} + \pb_1 \bigl( \La_{1,0} + \La_1 G \bigr) &= \zerob,
\end{align*}%
and the normalization condition $\pb_0 \oneb + \pb_1 (\I - R)^{-1} \oneb = 1$.
\end{theorem}%

We can immediately conclude from \cref{thmQBD:matrix-geometric_distribution,thmQBD:matrix-analytic_distribution} and (or \cite[Proposition~6.4.2]{Latouche1999_Matrix-analytic}) that there exist multiple relations between the rate matrix $R$ and the auxiliary matrix $G$. We have $\La_1 G = R \La_{-1}$, $R = \La_1(- \La_0 - \La_1 G)^{-1}$ and $G = (-\La_0 - R \La_{-1})^{-1} \La_{-1}$.

In a similar fashion as for the rate matrix $R$, the auxiliary matrix $G$ can be determined by successive substitution\endnote{Faster and more efficient algorithms than successive substitution exist to determine $G$; a good example is cyclic reduction \cite{Bini2000_Cyclic_reduction_M-G-1}.}, see \cref{algQBD:matrix-analytic_successive_substitutions}. Explicit and exact results exist also for the auxiliary matrix $G$, derived in an analogous manner to \cref{subsecQBD:R_explicit_solutions,subsecQBD:R_exact_solutions}.

\begin{algorithm}%
\caption{Calculating $G$ using successive substitutions}%
\label{algQBD:matrix-analytic_successive_substitutions}%
\begin{algorithmic}[1]%
\State Pick $\epsilon$ small and positive
\State Set $G_0 = 0$, $G_1 = - \La_0^{-1} \La_{-1}$ and $n = 1$
\While{$\| G_n - G_{n - 1} \|_{\textup{max}} > \epsilon$}
    \State Compute
            \begin{equation}%
            G_{n + 1} = - \La_0^{-1} \bigl( \La_{-1} + \La_1 G_n^2 \bigr)
            \end{equation}%
    \State Update $n = n + 1$
\EndWhile
\end{algorithmic}%
\end{algorithm}%

%%%%%%%%%%%%%%%%%%%%%%%%%%%%%%%%%%%%%%%%%%%%%%%%%%%%%%%
%%%%%%%%%%%%%%%%%%%%%%%%%%%%%%%%%%%%%%%%%%%%%%%%%%%%%%%
%%%%%%%%%%%%%%%%%%%%% NEW SECTION %%%%%%%%%%%%%%%%%%%%%
%%%%%%%%%%%%%%%%%%%%%%%%%%%%%%%%%%%%%%%%%%%%%%%%%%%%%%%
%%%%%%%%%%%%%%%%%%%%%%%%%%%%%%%%%%%%%%%%%%%%%%%%%%%%%%%

\section{Spectral expansion method}%
\label{secQBD:spectral_expansion_method}%

A third approach to determining the equilibrium probabilities does not make use of matrices but rather eigenvectors and eigenvalues. This approach is called the \textit{spectral expansion method}.\endnote{The spectral expansion method was developed by Mitrani \cite{Mitrani1995_Spectral_expansion,Mitrani1992_Spectral_expansion}.}
The basic idea of this method is to first try and find \textit{basis solutions} of the form
\begin{equation}%
\pb_i = \vca{y} x^{i - 1}, \quad i \ge 1, \label{eqnQBD:spectral_expansion_basis_solution}
\end{equation}%
where $\vca{y} = \begin{bmatrix} y(0) & y(1) & \cdots & y(r) \end{bmatrix} \neq \zerob$ and $|x| < 1$, satisfying the balance equations \eqref{eqnQBD:balance_equations_level_i} for $i \ge 2$. We require that $|x| < 1$, since we want to be able to normalize the equilibrium distribution. Substitution of \eqref{eqnQBD:spectral_expansion_basis_solution} in \eqref{eqnQBD:balance_equations_level_i} and dividing by common powers of $x$ yields
\begin{equation}%
\vca{y} \bigl( \La_{-1} + x \La_0 + x^2 \La_1 \bigr) = \zerob. \label{eqnQBD:spectral_expansion_y_and_x}
\end{equation}%
These equations have a non-zero solution for $\vca{y}$ if
\begin{equation}%
\det \bigl( \La_{-1} + x \La_0 + x^2 \La_1 \bigr) = 0. \label{eqnQBD:spectral_expansion_determinant_equation}
\end{equation}%
The desired values of $x$ are the roots $x$ with $|x| < 1$ of the determinant equation \eqref{eqnQBD:spectral_expansion_determinant_equation}. Equation \eqref{eqnQBD:spectral_expansion_determinant_equation} is a polynomial equation of degree $2(r + 1)$. Suppose that $\tilde{r} + 1$ roots $x$ satisfy $|x| < 1$ and for now let us assume that these roots are different. Let $\vca{y}_k, ~ k = 0,1,\ldots,\tilde{r}$ be the corresponding non-zero solutions of \eqref{eqnQBD:spectral_expansion_y_and_x} for $x = x_k$. Each solution $\pb_i = \vca{y}_k x_k^i, ~ k = 0,1,\ldots,\tilde{r}$ satisfies the global balance equations \eqref{eqnQBD:balance_equations_level_i}. These solutions are moreover linearly independent. We can linearly combine the $\tilde{r} + 1$ solutions to obtain a solution that satisfies the global balance equations \eqref{eqnQBD:balance_equations_level_i}:
\begin{equation}%
\pb_i = \sum_{k = 0}^{\tilde{r}} \xi_k \vca{y}_k x_k^i, \quad i \ge 1,
\end{equation}%
where $\xi_k, ~ k = 0,1,\ldots,\tilde{r}$ are arbitrary constants. So far, we have obtained expressions for $\pb_1,\pb_2,\ldots$, which still contains $\tilde{r} + 1$ unknowns $\xi_k$. Now, to determine these unknowns and $\pb_0$, we turn to the global balance equations for levels 0 and 1. Equations \eqref{eqnQBD:balance_equations_level_0} and \eqref{eqnQBD:balance_equations_level_1} are a set of $b + 1 + r + 1$ linear equations involving the $b + 1$ unknown probabilities of level 0 and the $\tilde{r} + 1$ unknowns constants $\xi_k$. The set of equations \eqref{eqnQBD:balance_equations_level_0} and \eqref{eqnQBD:balance_equations_level_1} only has $b + 1 + r$ linearly independent equations, but an additional independent equation is provided by the normalization condition. In conclusion, the set of equations \eqref{eqnQBD:balance_equations_level_0} and \eqref{eqnQBD:balance_equations_level_1} only has a unique solution if the number of equations and unknowns match, which means that $\tilde{r}$ is required to be equal to $r$. Since an irreducible and positive recurrent (by \cref{thmQBD:stability_condition}) QBD process has a unique solution to the global balance equations, we have the following theorem.

\begin{theorem}\label{thmQBD:spectral_expansion}%
An irreducible and positive recurrent \textup{QBD} process has $r + 1$ solutions $x$ with $|x| < 1$ of \eqref{eqnQBD:spectral_expansion_determinant_equation}. Assume these roots are different and label them $x_0,x_1,\ldots,x_r$. Let $\vca{y}_k, ~ k = 0,1,\ldots,r$ be the non-zero solution of \eqref{eqnQBD:spectral_expansion_y_and_x} for $x = x_k$. The linear combination of basis solutions
\begin{equation}%
\pb_i = \sum_{k = 0}^r \xi_k \vca{y}_k x_k^i, \quad i \ge 1, \label{eqnQBD:spectral_expansion_equilibrium_probability_vectors}
\end{equation}%
is the unique equilibrium distribution. The coefficients $\xi_0,\xi_1,\ldots,\xi_r$ and $\pb_0$ are the unique solution to the global balance equations of levels $0$ and $1$,
\begin{align}%
\pb_0 \La_{0,0} + \sum_{k = 0}^r \xi_k \vca{y}_k \La_{1,-1} &= \zerob, \\
\pb_0 \La_{0,1} + \sum_{k = 0}^r \xi_k \vca{y}_k \La_{1,0} + \sum_{k = 0}^r \xi_k \vca{y}_k x_k \La_{-1} &= \zerob,
\end{align}%
and the normalization condition
\begin{equation}%
\pb_0 \oneb + \sum_{k = 0}^r \xi_k \frac{\vca{y}_k \oneb}{1 - x_k} = 1.
\end{equation}%
\end{theorem}%

The roots $x_0,x_1,\ldots,x_r$ do not have to be different. If we assume that, when a root $x$ occurs $k$ times, it is possible to find $k$ linearly independent solutions of \eqref{eqnQBD:spectral_expansion_y_and_x}, then the analysis proceeds in exactly the same way. In case there are less than $k$ independent solutions, we would also have to consider more complicated basis solutions of the form $i \vca{y} x^{i - 1}$ (or even higher powers of $i$).

The relation between the matrix-geometric representation \eqref{eqnQBD:eq_dist_R} and the spectral expansion \eqref{eqnQBD:spectral_expansion_equilibrium_probability_vectors} is clear: the roots $x_0,x_1,\ldots,x_r$ are the eigenvalues of $R$ with corresponding left eigenvectors $\vca{y}_0,\vca{y}_1,\ldots,\vca{y}_r$. %We immediately see that in the special case of \eqref{eqnQBD:R_explicit_La_1_rank_1} in \eqref{eqnQBD:R_explicit_La_1_eq_dist} that $\omega$ and the single eigenvalue $x$ coincide and $\pb_2 = \xi \omega \vca{y}$ where $\xi$ is left to be determined from the equilibrium equations of levels 0 and 1 and $\vca{y}$ follows from \eqref{eqnQBD:spectral_expansion_y_and_x}.

\section{Takeaways}%
\label{secQBD:takeaways}%

Birth--and--death (BD) processes live on the positive half-line and move either to the left or right after exponential times. These basic features of BD processes, laid out in \cref{ch:birth--and--death_processes}, proved essential for the theory developed in this part of the book. In \cref{ch:reversible_networks} we exploited the BD structure to construct multi-dimensional versions, first for loss networks and then for queueing networks.

In this chapter we added a finite number of states to each state of that BD half-line, to construct quasi-birth--and--death (QBD) processes that live on a semi-infinite strip of states. The basic recursion method for BD processes was then lifted to the more general QBD setting to obtain the equilibrium distribution. Where BD processes result in geometric equilibrium distributions, QBD processes obey similar geometric forms, but with scalars replaced by matrices. This also explains why the main analytic technique introduced in this chapter is called the matrix-geometric method.

The matrix-geometric method exploits the fact that the QBD process has a highly structured state space, which allows for describing the balance equations in terms of transitions in the horizontal direction only. All transitions in the vertical directions are described in terms of finite matrices that appear in the balance equations. The matrix-geometric method is than the analytic methods for solving the system of matrix equations. The central step is to prove the existence of the unique rate matrix $R$ in \cref{thmQBD:matrix-geometric_distribution}. We have also presented efficient algorithms to determine $R$ numerically. Taken together, this provides a powerful computational framework for QBD processes.

Besides the matrix-geometric method, we have also demonstrated the matrix-analytic method with its auxiliary matrix $G$ and the spectral expansion method. The matrix-analytic method is similar in scope to the matrix-geometric method, but its focus is on the transitions to the left with as a result the first passage probabilities in the $G$ matrix. The spectral expansion method decomposes the $R$ matrix into its eigenvectors and eigenvalues and linearly combines them to construct the product-form solution. %The generating function approach gives a simple expression for the generating function of the equilibrium probability vectors. This function can easily be inverted to obtain numerical values for each equilibrium probability.

At first sight the extension from BD processes to QBD processes might seem less spectacular than the extension to the network models in \cref{ch:reversible_networks}. To fully appreciate the wide scope of QBD processes, the key insight is that the computational complexity of the matrix-geometric method is determined by the finite dimension, i.e., the dimension of the rate matrix $R$. This is remarkable, because without exploiting the special QBD structure, we would face a Markov process living on an infinite state space, and simply trying to solve the global balance equations would in many cases be prohibitively difficult. So if a Markov process can be brought into a QBD form, this brings enormous computational advantages.

Take as an example a single-server queue with generally distributed i.i.d.\@ inter-arrival times and generally distributed i.i.d.\@ service times. Approximate the inter-arrival and service times by phase-type distributions and use the finite dimension to keep track of these phases. Only arrivals and service completions then result in horizontal transitions, while all other events trigger transitions in the vertical direction. The fairly intractable general queueing systems is then converted into a QBD process, and performance analysis of the systems becomes straightforward.

This example shows that effort should be put in constructing the QBD process, and then one can reap the benefits of reduced complexity. More generally, the additional finite dimension in QBDs can keep track of enormous amounts of information and this partly explains why so many real-world systems can be modeled as QBD processes \cite{Latouche1999_Matrix-analytic,Neuts1994_Matrix-geometric}. In the next chapter we discuss the extension of QBD processes to Markov process of a similar structure, but with the possibility to take large steps in the horizontal direction.

%%%%%%%%%%%%%%%%%%%%%%%%%%%%%%%%%%%%%%%%%%%%%%%%%%%%%%%
%%%%%%%%%%%%%%%%%%%%%%%%%%%%%%%%%%%%%%%%%%%%%%%%%%%%%%%
%%%%%%%%%%%%%%%%%%%%%%%% NOTES %%%%%%%%%%%%%%%%%%%%%%%%
%%%%%%%%%%%%%%%%%%%%%%%%%%%%%%%%%%%%%%%%%%%%%%%%%%%%%%%
%%%%%%%%%%%%%%%%%%%%%%%%%%%%%%%%%%%%%%%%%%%%%%%%%%%%%%%

%\theendnotes%
%\setcounter{endnote}{0}
\printendnotes% %

% Checked points 1-7 and a-i
\chapter{Quasi-skip-free processes}%
\label{ch:skip-free_one_direction}%

Quasi-skip-free (QSF) processes are the generalization to two dimensions of the Markov processes associated with the $M/G/1$ system and the $G/M/1$ system. The QSF process has the same state space as the QBD process, but its transition structure is different. Whereas the QBD process is skip-free in both directions, the QSF process allows for transitions of larger size in one of the two directions. A distinction is made for processes that are QSF to the right and to the left, since each process requires a different solution approach for the equilibrium distribution.

%%%%%%%%%%%%%%%%%%%%%%%%%%%%%%%%%%%%%%%%%%%%%%%%%%%%%%%
%%%%%%%%%%%%%%%%%%%%%%%%%%%%%%%%%%%%%%%%%%%%%%%%%%%%%%%
%%%%%%%%%%%%%%%%%%%%% NEW SECTION %%%%%%%%%%%%%%%%%%%%%
%%%%%%%%%%%%%%%%%%%%%%%%%%%%%%%%%%%%%%%%%%%%%%%%%%%%%%%
%%%%%%%%%%%%%%%%%%%%%%%%%%%%%%%%%%%%%%%%%%%%%%%%%%%%%%%

\section{Variations of skip-free processes}%
\label{secQSF:variations}%

In this section we analyze two QSF processes that both are constructed from a QBD process. The resulting QSF processes are skip-free in different directions and therefore require different solution methods to obtain their equilibrium distributions.

%%%%%%%%%%%%%%%%%%%%%%%%%%%%%%%%%%%%%%%%%%%%%%%%%%%%%%%
%%%%%%%%%%%%%%%%%%%%%%%%%%%%%%%%%%%%%%%%%%%%%%%%%%%%%%%
%%%%%%%%%%%%%%%%%%% NEW SUBSECTION %%%%%%%%%%%%%%%%%%%%
%%%%%%%%%%%%%%%%%%%%%%%%%%%%%%%%%%%%%%%%%%%%%%%%%%%%%%%
%%%%%%%%%%%%%%%%%%%%%%%%%%%%%%%%%%%%%%%%%%%%%%%%%%%%%%%

\subsection{Machine with setup times and batch arrivals}%
\label{subsecQSF:setup_times}%

Let us consider an adaptation of the machine with setup times as mentioned before in \cref{subsecQBD:setup_times}. The machine processes jobs in order of arrival. Jobs arrive in batches to the system: batches of size 1 and 2 arrive according to Poisson processes with rates $\la_1$ and $\la_2$. The processing time of a job is exponentially distributed with mean $1/\mu$. For stability we assume that $\rho \defi (\la_1 + 2\la)/\mu < 1$. The machine is turned off when the system is empty and it is turned on again when a new batch of jobs arrives. The setup time is exponentially distributed with mean $1/\theta$. Turning off the machine takes an exponential amount of time with mean $1/\gamma$.

The state of the system may be described by $X(t) \defi (X_1(t),X_2(t))$ with $X_1(t)$ representing the number of jobs in the system at time $t$ and $X_2(t)$ indicates whether the machine is turned off (0) or on (1) at time $t$. The process $\{ X(t) \}_{t \ge 0}$ is a Markov process with state space $\statespace \defi \{ (i,j) \in \Nat_0 \times \{ 0,1 \} \}$. The transition rate diagram is displayed in \cref{figQSF:transition_rate_diagram_setup_times_batch_arrivals} and resembles the QBD variant in \cref{figQBD:transition_rate_diagram_setup_times}, depicting the transition rate diagram of the system where jobs arrive one by one and turning off the machine takes no time.

\begin{figure}
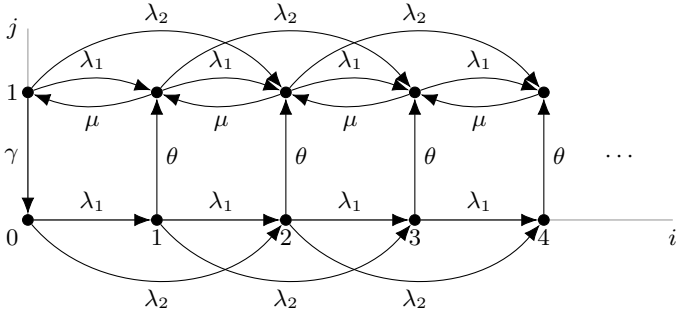
%
\centering%
\includestandalone{Chapters/QSF/TikZFiles/transition_rate_diagram_setup_times_batch_arrivals}%
\caption{Transition rate diagram of the machine with setup times and batch arrivals of size 2.}
\label{figQSF:transition_rate_diagram_setup_times_batch_arrivals}%
\end{figure}%

For the current model, define $\lvl{i} = \{ (i,0),(i,1) \}, ~ i \ge 0$ as the set of states with $i$ jobs in the system, that is, $\lvl{i}$ is level $i$. We can then partition the state space as
\begin{equation}%
\statespace \defi \lvl{0} \cup \lvl{1} \cup \lvl{2} \cup \cdots
\end{equation}%

The Markov process $\{ X(t) \}_{t \ge 0}$ is QSF to the left: the process cannot skip any levels when transitioning to the left, whereas it can skip a level when transitioning to the right due to a batch arrival of size 2. We demonstrate how the matrix-analytic method can be applied to determine the equilibrium distribution of this QSF process.

Let $p(i,j)$ denote the equilibrium probability of state $(i,j) \in \statespace$. From the transition rate diagram we can obtain the balance equations by equating the flow out of a state and the flow into that state. For the boundary states we have, with $\la \defi \la_1 + \la_2$,
\begin{align}%
\la p(0,0)  &= \gamma p(0,1), \label{eqnQSF:setup_times_batch_arrivals_global_balance_(0,0)} \\
(\la + \gamma) p(0,1)  &= \mu p(1,1), \label{eqnQSF:setup_times_batch_arrivals_global_balance_(0,1)} \\
(\la + \theta) p(1,0)  &= \la_1 p(0,0), \label{eqnQSF:setup_times_batch_arrivals_global_balance_(1,0)} \\
(\la + \mu) p(1,1)  &= \la_1 p(0,1) + \theta p(1,0) + \mu p(2,1), \label{eqnQSF:setup_times_batch_arrivals_global_balance_(1,1)}
\end{align}%
and for $i \ge 2$,
\begin{align}%
(\la + \theta) p(i,0) &= \la_1 p(i - 1,0) + \la_2 p(i - 2,0), \label{eqnQSF:setup_times_batch_arrivals_global_balance_phase_0} \\
(\la + \mu) p(i,1) &= \la_1 p(i - 1,0) + \la_2 p(i - 2,0) + \theta p(i,0) + \mu p(i + 1,1). \label{eqnQSF:setup_times_batch_arrivals_global_balance_phase_1}
\end{align}%
Let us introduce the vectors of equilibrium probabilities $\pb_i = \begin{bmatrix} p(i,0) & p(i,1) \end{bmatrix}$ and write \eqref{eqnQSF:setup_times_batch_arrivals_global_balance_(0,0)}--\eqref{eqnQSF:setup_times_batch_arrivals_global_balance_phase_1} in vector-matrix notation:
\begin{align}%
\pb_0 \La_{0}^{(0)} + \pb_1 \La_{-1} &= \zerob, \label{eqnQSF:setup_times_batch_arrivals_global_balance_vector-matrix_level_0} \\
\pb_0 \La_1 + \pb_1 \La_0 + \pb_2 \La_{-1} &= \zerob, \label{eqnQSF:setup_times_batch_arrivals_global_balance_vector-matrix_level_1} \\
\pb_{i - 2} \La_2 + \pb_{i - 1} \La_1 + \pb_i \La_0 + \pb_{i + 1} \La_{-1} &= \zerob, \quad i \ge 2, \label{eqnQSF:setup_times_batch_arrivals_global_balance_vector-matrix}
\end{align}%
where
\begin{gather}%
\La_{-1} = \begin{bmatrix}%
0 & 0 \\ 0 & \mu
\end{bmatrix},%
\quad
\La_0 = \begin{bmatrix}%
-(\la + \theta) & \theta \\ 0 & -(\la + \mu)
\end{bmatrix}, \notag \\
\La_k = \begin{bmatrix}%
\la_k & 0 \\ 0 & \la_k
\end{bmatrix}, ~ k = 1,2, \quad%
\La_{0}^{(0)} = \begin{bmatrix}%
-\la & 0 \\ \ga & -(\la + \ga)
\end{bmatrix}.%
\end{gather}%
The balance equations \eqref{eqnQSF:setup_times_batch_arrivals_global_balance_vector-matrix_level_0}--\eqref{eqnQSF:setup_times_batch_arrivals_global_balance_vector-matrix_level_1}  are referred to as the boundary equations. We show how the matrix-analytic method for QBD processes can be applied to processes that are QSF to the left.

The auxiliary matrix $G$ plays a key role in the matrix-analytic method. Element $(j,k)$ of $G$ is interpreted as the probability that, starting in state $(i,j), ~ i \ge 2$, the first passage to level $i - 1$ happens in state $(i - 1,k)$. More generally, element $(j,k)$ of $G^n$ is interpreted as the probability that, starting in state $(i,j), ~ i \ge n + 1$, the first passage to level $i - n$ happens in state $(i - n,k)$. This immediately implies that zero columns in $\La_{-1}$ lead to zero columns in $G$. For the model at hand
\begin{equation}%
G = \begin{bmatrix}%
0 & 1 \\
0 & 1 \\
\end{bmatrix}.
\end{equation}%
This matrix appears when we censor the Markov process to particular sets of states. Define the union of levels $0,1,\ldots,i$ as $\lvls{i} \defi \lvl{0} \cup \lvl{1} \cup \cdots \cup \lvl{i}$. Censoring the Markov process to $\lvls{i}$ means that we only observe the Markov process when it resides in a state in $\lvls{i}$ and transitions to states outside this set are redirected appropriately to states inside the set.

Suppose we censor the Markov process to $\lvls{i}, ~ i \ge 2$ and write down the balance equations for level $\lvl{i}$. From $\lvl{i - 2}$ the process transitions to $\lvl{i}$ with rate $\pb_{i - 2} \La_2$. From $\lvl{i - 1}$ the process transitions to $\lvl{i}$ with rate $\pb_{i - 1} \La_1$, but also with rate $\pb_{i - 1} \La_2 G$, since in that case the process transitions to $\lvl{i + 1}$ and returns to $\lvl{i}$ according to the probabilities described in $G$. Similarly, the contribution of $\lvl{i}$ to the balance equations is $\pb_i (\La_0 + \La_1 G + \La_2 G^2)$. So, for the censored process the balance equations for $\lvl{i}$ are
\begin{equation}%
\pb_{i - 2} \La_2 + \pb_{i - 1} \bigl( \La_1 + \La_2 G \bigr) + \pb_i \bigl( \La_0 + \La_1 G + \La_2 G^2 \bigr) = \zerob, \quad i \ge 2. \label{eqnQSF:setup_times_batch_arrivals_censored_level_i}
\end{equation}%
We can derive similar balance equations for $\lvl{1}$ when censoring the process to $\lvls{1}$ and for $\lvl{0}$ when censoring the process to $\lvls{0} = \lvl{0}$:
\begin{align}%
\pb_{0} \bigl( \La_1 + \La_2 G \bigr) + \pb_1 \bigl( \La_0 + \La_1 G + \La_2 G^2 \bigr) &= \zerob, \label{eqnQSF:setup_times_batch_arrivals_censored_level_1}\\
\pb_0 \bigl( \La_0^{(0)} + \La_1 G + \La_2 G^2 \bigr) &= \zerob. \label{eqnQSF:setup_times_batch_arrivals_censored_level_0}
\end{align}%
Equation~\eqref{eqnQSF:setup_times_batch_arrivals_censored_level_0} is a homogeneous system of equations which does not have a unique solution, but  we can at least conclude that $\pb_0$ is proportional to $\begin{bmatrix} \ga & \la \end{bmatrix}$. We supplement this system of equations with the normalization condition to be able to uniquely determine $\pb_0$.

We examine the series $\sum_{i \ge 0} \pb_i$ with the goal of finding another expression involving $\pb_0$. Abbreviate
\begin{equation}%
\Ga_0 = \La_0 + \La_1 G + \La_2 G^2, \quad \Ga_1 = \La_1 + \La_2 G, \quad \Ga_2 = \La_2.
\end{equation}%
Use \eqref{eqnQSF:setup_times_batch_arrivals_censored_level_i}--\eqref{eqnQSF:setup_times_batch_arrivals_censored_level_1} to rewrite the summation $\sum_{i \ge 0} \pb_i$:
\begin{align}%
\sum_{i \ge 0} \pb_i &= \pb_0 + \pb_1 + \sum_{i \ge 2} \pb_i \notag \\
&= \pb_0 - \pb_0 \Ga_1 \Ga_0^{-1} - \sum_{i \ge 2} \bigl( \pb_{i - 2} \Ga_2 + \pb_{i - 1} \Ga_1 \bigr) \Ga_0^{-1} \notag \\
&= \pb_0 - \pb_0 \Ga_1 \Ga_0^{-1} - \sum_{i \ge 0} \pb_{i - 2} \Ga_2 \Ga_0^{-1} - \sum_{i \ge 1} \pb_{i - 1} \Ga_1 \Ga_0^{-1} \notag \\
&= \pb_0 - \sum_{i \ge 0} \pb_i \bigl( \Ga_1 + \Ga_2 \bigr) \Ga_0^{-1}.
\end{align}%
Hence,
\begin{equation}%
\sum_{i \ge 0} \pb_i \Bigl( \I + \bigl( \Ga_1 + \Ga_2 \bigr) \Ga_0^{-1} \Bigr) = \pb_0,
\end{equation}%
which gives
\begin{equation}%
\sum_{i \ge 0} \pb_i = \pb_0 \Ga_0 \bigl( \Ga_0 + \Ga_1 + \Ga_2 \bigr)^{-1}, \label{eqnQSF:setup_times_batch_arrivals_expression_for_p_0}
\end{equation}%
where the inverse is given by
\begin{equation}%
\bigl( \Ga_0 + \Ga_1 + \Ga_2 \bigr)^{-1} = \begin{bmatrix}%
-\frac{1}{\theta} & \frac{1}{\theta}\bigl( 1 + \frac{\theta + \mu}{\la_1 + 2\la_2 - \mu} \bigr) \\
 0                & \frac{1}{\la_1 + 2\la_2 - \mu}
\end{bmatrix}.%
\end{equation}%
So, the normalization condition is
\begin{equation}%
1 = \sum_{i \ge 0} \pb_i \oneb = \pb_0 \Ga_0 \bigl( \Ga_0 + \Ga_1 + \Ga_2 \bigr)^{-1} \oneb. \label{eqnQSF:setup_times_batch_arrivals_normalization_condition}
\end{equation}%
Substituting the normalization condition \eqref{eqnQSF:setup_times_batch_arrivals_normalization_condition} for any of the equations in \eqref{eqnQSF:setup_times_batch_arrivals_censored_level_0} allows us to calculate $\pb_0$. Armed with $\pb_0$ we are able to solve \eqref{eqnQSF:setup_times_batch_arrivals_censored_level_1} for $\pb_1$, after which we can iteratively calculate $\pb_i, ~ i \ge 2$ from \eqref{eqnQSF:setup_times_batch_arrivals_censored_level_i}, stopping when the accumulated probability mass is close to 1.

\subsection{A batch machine subject to breakdowns}%
\label{subsecQSF:batch_machine_breakdowns}%

Consider a batch machine that is subject to breakdowns. Depending on the details of the jobs, the batch machine can sometimes serve two jobs at the same time, but sometimes only a single job. For simplicity we assume that with probability $1/2$ the batch machine serves a single job and with the same probability two jobs. The service time is independent of the number of jobs in service and is exponentially distributed with rate $2\mu$. Jobs arrive to the system according to a Poisson process with rate $\la$. The machine breaks down after an exponential amount of time with rate $\ga$ (irrespective of whether it is serving a job or not) and repair takes an exponential amount of time with rate $\theta$. Every time the machine breaks down, the repair is started immediately. If there is only a single job in the system, the machine serves this single job with probability 1.

Notice that the machine works a fraction $\theta/(\ga + \theta)$ of the time and is in repair a fraction $\ga/(\ga + \theta)$ of the time. So, the rate at which the server can serve jobs is $\theta/(\ga + \theta) \cdot 2\mu (1 \cdot 1/2 + 2 \cdot 1/2)$ and therefore the stability condition is
\begin{equation}%
\rho \defi \frac{\la}{\frac{\theta}{\ga + \theta} (\mu + 2\mu)} < 1.
\end{equation}%

The state of the system may be described by $X(t) \defi (X_1(t),X_2(t))$ with $X_1(t)$ representing the number of jobs in the system at time $t$ and $X_2(t)$ describes if the machine is working (1) or not (0) at time $t$. The process $\{ X(t) \}_{t \ge 0}$ is a Markov process with state space $\statespace \defi \{ (i,j) \in \Nat_0 \times \{ 0,1 \} \}$. The transition rate diagram is displayed in \cref{figQSF:transition_rate_diagram_batch_services_breakdowns}. We use the same levels as in the previous example, i.e., $\lvl{i} = \{ (i,0),(i,1) \}, ~ i \ge 0$.

\begin{figure}
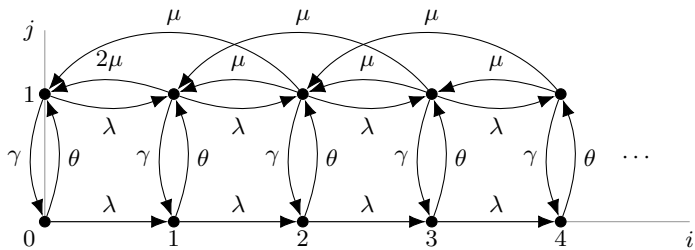
%
\centering%
\includestandalone{Chapters/QSF/TikZFiles/transition_rate_diagram_batch_services_breakdowns}%
\caption{Transition rate diagram of the batch machine subject to breakdowns.}
\label{figQSF:transition_rate_diagram_batch_services_breakdowns}%
\end{figure}%

The Markov process $\{ X(t) \}_{t \ge 0}$ is QSF to the right: the process cannot skip any levels when transitioning to the right, whereas it can skip a level when transitioning to the left with a batch service of size 2. We demonstrate how the matrix-geometric method and spectral expansion method can be adapted to determine the equilibrium distribution of this QSF process.

Let $p(i,j)$ denote the equilibrium probability of state $(i,j) \in \statespace$. From the transition rate diagram we can obtain the balance equations by equating the flow out of a state and the flow into that state. For the boundary states we have,
\begin{align}%
(\la + \theta) p(0,0)  &= \ga p(0,1), \label{eqnQSF:batch_services_breakdowns_global_balance_(0,0)} \\
(\la + \ga) p(0,1)  &= 2\mu p(1,1) + \mu p(2,1) + \theta p(0,0), \label{eqnQSF:batch_services_breakdowns_global_balance_(0,1)} \\
(\la + \theta) p(1,0)  &= \la p(0,0) + \ga p(1,1), \label{eqnQSF:batch_services_breakdowns_global_balance_(1,0)} \\
(\la + 2\mu + \ga) p(1,1)  &= \la p(0,1) + \theta p(1,0) + \mu (p(2,1) + p(3,1)), \label{eqnQSF:batch_services_breakdowns_global_balance_(1,1)}
\end{align}%
and for $i \ge 2$,
\begin{align}%
(\la + \theta) p(i,0) &= \la p(i - 1,0) + \ga p(i,1), \label{eqnQSF:batch_services_breakdowns_global_balance_phase_0} \\
(\la + 2\mu + \ga) p(i,1) &= \la p(i - 1,0) + \theta p(i,0) + \mu (p(i + 1,1) + p(i + 2,1)). \label{eqnQSF:batch_services_breakdowns_global_balance_phase_1}
\end{align}%
Let us introduce the vectors of equilibrium probabilities $\pb_i = \begin{bmatrix} p(i,0) & p(i,1) \end{bmatrix}$ and write \eqref{eqnQSF:batch_services_breakdowns_global_balance_(0,0)}--\eqref{eqnQSF:batch_services_breakdowns_global_balance_phase_1} in vector-matrix notation:
\begin{align}%
\pb_0 \La_{0}^{(0)} + \pb_1 \La_{-1}^{(1)} + \pb_2 \La_{-2} &= \zerob, \label{eqnQSF:batch_services_breakdowns_global_balance_vector-matrix_level_0} \\
%\pb_0 \La_1 + \pb_1 \La_0 + \pb_2 \La_{-1} + \pb_3 \La_{-2} &= \zerob, \label{eqnQSF:batch_services_breakdowns_global_balance_vector-matrix_level_1} \\
\pb_{i - 1} \La_1 + \pb_i \La_0 + \pb_{i + 1} \La_{-1} + \pb_{i + 2} \La_{-2} &= \zerob, \quad i \ge 1, \label{eqnQSF:batch_services_breakdowns_global_balance_vector-matrix}
\end{align}%
where
\begin{gather}%
\La_{-2} = \La_{-1} = \begin{bmatrix}%
0 & 0 \\ 0 & \mu
\end{bmatrix},%
\quad
\La_0 = \begin{bmatrix}%
-(\la + \theta) & \theta \\ \ga & -(\la + 2\mu + \ga)
\end{bmatrix}, \notag \\
\La_1 = \begin{bmatrix}%
\la & 0 \\ 0 & \la
\end{bmatrix},
\quad
\La_{0}^{(0)} = \begin{bmatrix}%
-(\la + \theta) & \theta \\ \ga & -(\la + \ga)
\end{bmatrix},
\quad
\La_{-1}^{(1)} = \begin{bmatrix}%
0 & 0 \\ 0 & 2\mu
\end{bmatrix}.%
\end{gather}%
The balance equations \eqref{eqnQSF:batch_services_breakdowns_global_balance_vector-matrix_level_0} are referred to as the boundary equations. We will first show how the matrix-geometric method for QBD processes can be applied to processes that are QSF to the right.

Recall from the QBD processes that the matrix-geometric method expresses the equilibrium probability vectors as
\begin{equation}%
\pb_i = \pb_0 R^i, \quad i \ge 0. \label{eqnQSF:batch_services_breakdowns_matrix-geometric_relation}
\end{equation}%
Substituting \eqref{eqnQSF:batch_services_breakdowns_matrix-geometric_relation} into \eqref{eqnQSF:batch_services_breakdowns_global_balance_vector-matrix} gives
\begin{equation}%
\pb_{i - 1} \Bigl( \La_1 + R \La_0 + R^2 \La_{-1} + R^3 \La_{-2} \Bigr) = \zerob, \quad i \ge 1,
\end{equation}%
which holds if the rate matrix $R$ is the solution to the matrix equation
\begin{equation}%
\La_1 + R \La_0 + R^2 \La_{-1} + R^3 \La_{-2} = 0. \label{eqnQSF:batch_services_breakdowns_matrix_equation_R}
\end{equation}%
It can be shown that $R$ is the unique minimal non-negative solution of \eqref{eqnQSF:batch_services_breakdowns_matrix_equation_R} and has spectral radius less than 1, which shows that $(\I - R)^{-1}$ exists. Equation~\eqref{eqnQSF:batch_services_breakdowns_matrix_equation_R} can be solved via successive substitutions in a similar fashion as \cref{algQBD:matrix-geometric_successive_substitutions}.

The boundary equilibrium probability vector $\pb_0$ can be determined by substituting \eqref{eqnQSF:batch_services_breakdowns_matrix-geometric_relation} in the balance equations  \eqref{eqnQSF:batch_services_breakdowns_global_balance_vector-matrix_level_0}:
\begin{equation}%
\pb_0 \bigl( \La_{0}^{(0)} + R \La_{-1}^{(1)} + R^2 \La_{-2} \bigr) = \zerob.
\end{equation}%
This homogeneous system of equations does not have a unique solution. However, if we substitute any of its equations by the normalization condition
\begin{equation}%
1 = \sum_{i \ge 0} \pb_i \oneb = \pb_0 \oneb + \sum_{i \ge 1} \pb_1 R^{i - 1} \oneb = \pb_0 \oneb + \pb_1 ( \I - R )^{-1} \oneb,
\end{equation}%
we get a non-homogeneous system of equations with unique solution $\pb_0$, and through \eqref{eqnQSF:batch_services_breakdowns_matrix-geometric_relation} we find all $\pb_i$.

We now demonstrate the spectral expansion method. Recall that this method tries to find basis solutions of the form
\begin{equation}%
\pb_i = \vca{y} x^i, \quad i \ge 0, \label{eqnQSF:batch_services_breakdowns_spectral_expansion_basis_solutions}
\end{equation}%
where $\vca{y} = \begin{bmatrix} y(0) & y(1) \end{bmatrix} \neq \zerob$ and $|x| < 1$, satisfying the balance equations \eqref{eqnQSF:batch_services_breakdowns_global_balance_vector-matrix}. We require that $|x| < 1$, since we want to be able to normalize the $\pb_i$. Substitution of \eqref{eqnQSF:batch_services_breakdowns_spectral_expansion_basis_solutions} in \eqref{eqnQSF:batch_services_breakdowns_global_balance_vector-matrix} and dividing by common powers of $x$ yields
\begin{equation}%
\vca{y} \bigl( \La_1 + x \La_0 + x^2 \La_{-1} + x^3 \La_{-2} \bigr) = \zerob. \label{eqnQSF:batch_services_breakdowns_spectral_expansion_y_and_x}
\end{equation}%
These equations have a non-zero solution for $\vca{y}$ if
\begin{equation}%
\det \bigl( \La_1 + x \La_0 + x^2 \La_{-1} + x^3 \La_{-2} \bigr) = 0. \label{eqnQSF:batch_services_breakdowns_spectral_expansion_determinant_equation}
\end{equation}%
The desired values of $x$ are the roots $x$ with $|x| < 1$ of the determinant equation \eqref{eqnQSF:batch_services_breakdowns_spectral_expansion_determinant_equation}. In this case \eqref{eqnQSF:batch_services_breakdowns_spectral_expansion_determinant_equation} is a polynomial of degree four in $x$. One of the solutions of \eqref{eqnQSF:batch_services_breakdowns_spectral_expansion_determinant_equation}  is $x = 1$, since $\La_1 + \La_0 + \La_{-1} + \La_{-2}$ is the transition rate matrix of the Markov process that describes the phase transitions. Now, \eqref{eqnQSF:batch_services_breakdowns_spectral_expansion_determinant_equation} reads
\begin{equation}%
(x - 1)\bigl( x^3(\la + \theta)\mu + x^2 (\la + 2\theta)\mu - x\la (\la + 2\mu + \theta + \ga ) + \la^2 \bigr) = 0,
\end{equation}%
which has two roots $x_1$ and $x_2$ inside the open unit disk. %\mytodo{The proof of this statement will follow.} 
These roots  have an explicit expression, but are more difficult to write down because they originate from a cubic equation. For $k = 1,2$, let $\vca{y}_k$ be the non-zero solution of
\begin{equation}%
\vca{y}_k \bigl( \La_1 + x_k \La_0 + x_k^2 \La_{-1} + x_k^3 \La_{-2} \bigr) = \zerob. \label{eqnQSF:batch_services_breakdowns_spectral_expansion_linear_combination}
\end{equation}%
Note that, since the balance equations are linear, any linear combination of the two solutions satisfies \eqref{eqnQSF:batch_services_breakdowns_global_balance_vector-matrix}. Now the final step of the spectral expansion method is to determine a linear combination that also satisfies the boundary equations \eqref{eqnQSF:batch_services_breakdowns_global_balance_vector-matrix_level_0}. So we set
\begin{equation}%
\pb_i = \xi_1 \vca{y}_1 x_1^i + \xi_2 \vca{y}_2 x_2^i, \quad i \ge 0. \label{eqnQSF:batch_services_breakdowns_equilibrium_distribution_spectral_expansion}
\end{equation}%
We can determine the coefficients $\xi_1$ and $\xi_2$ by substituting \eqref{eqnQSF:batch_services_breakdowns_equilibrium_distribution_spectral_expansion} into \eqref{eqnQSF:batch_services_breakdowns_global_balance_vector-matrix_level_0}, which gives
\begin{equation}%
\sum_{k = 1}^2 \xi_k \vca{y}_k \bigl( \La_0^{(0)} + x_k \La_{-1}^{(1)} + x_k^2 \La_{-2} \bigr) = \zerob. \label{eqnQSF:batch_services_breakdowns_equilibrium_distribution_spectral_expansion_boundary_level_0} \\
\end{equation}%
%
%If we use \eqref{eqnQSF:batch_services_breakdowns_spectral_expansion_linear_combination}, we can simplify
%%
%\begin{equation}%
%\vca{y}_k \bigl( \La_0 + x_k \La_{-1} + x_k^2 \La_{-2} \bigr) = - \vca{y}_k \frac{1}{x_k} \La_1,
%\end{equation}%
%%
%so that \eqref{eqnQSF:batch_services_breakdowns_equilibrium_distribution_spectral_expansion_boundary_level_1} becomes, in combination with $\La_1 = \la \I$,
%%
%\begin{equation}%
%\pb_0 - \xi_1 \vca{y}_1 \frac{1}{x_1} - \xi_2 \vca{y}_2 \frac{1}{x_2} = \zerob,
%\end{equation}%
%%
%which shows that $\pb_0 = \xi_1 \vca{y}_1 \frac{1}{x_1} + \xi_2 \vca{y}_2 \frac{1}{x_2}$ and therefore we could just as well set
%%
%\begin{equation}%
%\pb_i = \xi_1 \vca{y}_1 x_1^i + \xi_2 \vca{y}_2 x_2^i, \quad i \ge 0. \label{eqnQSF:batch_services_breakdowns_equilibrium_distribution_spectral_expansion_updated_to_level_0}
%\end{equation}%
%%
%So, the only equation that remains is \eqref{eqnQSF:batch_services_breakdowns_equilibrium_distribution_spectral_expansion_boundary_level_0}, which, with our new definition \eqref{eqnQSF:batch_services_breakdowns_equilibrium_distribution_spectral_expansion_updated_to_level_0}, reads
%%
%\begin{equation}%
%\xi_1 \vca{y}_1 \bigl( \La_{0}^{(0)} + x_1 \La_{-1}^{(1)} + x_1^2 \La_{-2} \bigr) + \xi_2 \vca{y}_2 \bigl( \La_{0}^{(0)} + x_2 \La_{-1}^{(1)} + x_2^2 \La_{-2} \bigr) = \zerob.
%\end{equation}%
%%
We can substitute the normalization condition for one of the above homogeneous equations to uniquely determine the coefficients $\xi_1$ and $\xi_2$. The normalization condition states that
\begin{equation}%
1 = \sum_{i \ge 0} \pb_i = \frac{\xi_1 \vca{y}_1}{1 - x_1} + \frac{\xi_2 \vca{y}_2}{1 - x_2}.
\end{equation}%
Determination of the coefficients is now a straightforward task.

%%%%%%%%%%%%%%%%%%%%%%%%%%%%%%%%%%%%%%%%%%%%%%%%%%%%%%%
%%%%%%%%%%%%%%%%%%%%%%%%%%%%%%%%%%%%%%%%%%%%%%%%%%%%%%%
%%%%%%%%%%%%%%%%%%%%% NEW SECTION %%%%%%%%%%%%%%%%%%%%%
%%%%%%%%%%%%%%%%%%%%%%%%%%%%%%%%%%%%%%%%%%%%%%%%%%%%%%%
%%%%%%%%%%%%%%%%%%%%%%%%%%%%%%%%%%%%%%%%%%%%%%%%%%%%%%%

\section{General quasi-skip-free processes}%
\label{secQSF:QSF_processes}%

From the previous examples we have seen that processes that are QSF to the left or right share the same state space as the QBD process that we have encountered in \cref{ch:quasi-birth--and--death_processes}. We can therefore use the similar level definitions as before:
\begin{equation}%
\lvl{0} \defi \{ (0,0),(0,1),\ldots,(0,b) \}, \quad \lvl{i} \defi \{ (i,0),(i,1),\ldots,(i,r) \}, ~ i \ge 1,
\end{equation}%
with $b$ and $r$ non-negative finite integers and partition the state space as
\begin{equation}%
\statespace \defi \lvl{0} \cup \lvl{1} \cup \lvl{2} \cup \cdots
\end{equation}%
We denote the state of the QSF process at time $t$ as $X(t) \defi (X_1(t),X_2(t))$ where $X_1(t)$ describes the level and $X_2(t)$ describes the phase at time $t$.

Throughout this chapter we focus on \textit{homogeneous} QSF processes, which means that transition rates are level-independent, possibly except for the transition rates from and to level 0. We can now identify the two types of QSF processes. Using the level-independent $\La_n$ and level-dependent $\La_n^{(m)}$ transition sub-matrices, we have that the transition rate matrix $Q$ of a process that is QSF to the left is of the form
\begin{equation}%
Q = \begin{bmatrix}%
\La_0^{(0)}    & \La_1^{(0)} & \La_2^{(0)} & \La_3^{(0)} & \La_4^{(0)} & \cdots \\
\La_{-1}^{(1)} & \La_{0}     & \La_1       & \La_2       & \La_3       & \cdots \\
               & \La_{-1}    & \La_0       & \La_1       & \La_2       & \cdots \\
               &             & \La_{-1}    & \La_0       & \La_1       & \cdots \\
               &             &             & \ddots      & \ddots      & \ddots
\end{bmatrix}. \label{eqnQSF:transition_rate_matrix_QSF_to_the_left}%
\end{equation}%
The transition rate matrix $Q$ of a process that is QSF to the right is given by
\begin{equation}%
Q = \begin{bmatrix}%
\La_0^{(0)}    & \La_1^{(0)} &          &        &        &  \\
\La_{-1}^{(1)} & \La_{0}     & \La_1    &        &        &  \\
\La_{-2}^{(2)} & \La_{-1}    & \La_0    & \La_1  &        &  \\
\La_{-3}^{(3)} & \La_{-2}    & \La_{-1} & \La_0  & \La_1  &  \\
\vdots         & \vdots      & \vdots   & \vdots & \vdots & \ddots
\end{bmatrix}. \label{eqnQSF:transition_rate_matrix_QSF_to_the_right}%
\end{equation}%

By comparing these transition rate matrices \eqref{eqnQSF:transition_rate_matrix_QSF_to_the_left}--\eqref{eqnQSF:transition_rate_matrix_QSF_to_the_right} with the transition rate matrix of the QBD process in \eqref{eqnQBD:definition_transition_rate_matrix}, we see that the QBD process is a process that is QSF to both the left and the right.

The balance equations $\pb Q = \zerob$ in case of a process that is QSF to the left are given by
\begin{align}%
\pb_0 \La_0^{(0)} + \pb_1 \La_{-1}^{(1)} &= \zerob, \label{eqnQSF:QSF_left_balance_equations_level_0} \\
\pb_0 \La_i^{(0)} + \sum_{k = 1}^{i + 1} \pb_k \La_{i - k} &= \zerob, \quad i \ge 1, \label{eqnQSF:QSF_left_balance_equations_level_i}
\end{align}%
and for a process that is QSF to the right we have
\begin{align}%
\sum_{k \ge 0} \pb_k \La_{-k}^{(k)} &= \zerob, \label{eqnQSF:QSF_right_balance_equations_level_0} \\
\pb_0 \La_1^{(0)} + \sum_{k \ge 1} \pb_k \La_{1 - k} &= \zerob, \label{eqnQSF:QSF_right_balance_equations_level_1} \\
\sum_{k \ge i - 1} \pb_k \La_{i - k} &= \zerob, \quad i \ge 2. \label{eqnQSF:QSF_right_balance_equations_level_i}
\end{align}%

Each of the two examples in \cref{secQSF:variations} are analyzed by using different methods. The success of a method depends on the skip-free direction of the QSF process. That is, the matrix-geometric method or the spectral expansion method do not work for processes that are QSF to the left and one cannot use the matrix-analytic method to analyze processes that are QSF to the right. The application of these methods to QSF processes is similar to their application to QBD processes, so the treatment of these methods in the next sections will be brief.

%%%%%%%%%%%%%%%%%%%%%%%%%%%%%%%%%%%%%%%%%%%%%%%%%%%%%%%
%%%%%%%%%%%%%%%%%%%%%%%%%%%%%%%%%%%%%%%%%%%%%%%%%%%%%%%
%%%%%%%%%%%%%%%%%%%%% NEW SECTION %%%%%%%%%%%%%%%%%%%%%
%%%%%%%%%%%%%%%%%%%%%%%%%%%%%%%%%%%%%%%%%%%%%%%%%%%%%%%
%%%%%%%%%%%%%%%%%%%%%%%%%%%%%%%%%%%%%%%%%%%%%%%%%%%%%%%

\section{Stability condition}%
\label{secQSF:stability_condition}%

There is a natural extension of the stability condition for QBD processes seen in \cref{thmQBD:stability_condition} to the stability condition for processes that are QSF to the left or right. We summarize the results for both types in a single theorem.

Recall that the transition rate matrix $\La$ describes the transition behavior of the phases. We unify both types of QSF processes by setting
\begin{equation}%
\La \defi \sum_{i = -\infty}^\infty \La_i.
\end{equation}%
From here on we will assume that the QSF process $\{ X(t) \}_{t \ge 0}$ is irreducible and that the transition rate matrix $\La$ has exactly one communicating class. Let $\vca{x}$ be the equilibrium distribution of the Markov process with transition rate matrix $\La$:
\begin{equation}%
\vca{x} \La = \zerob, \quad \vca{x} \oneb = 1.
\end{equation}%
Using distribution $\vca{x}$ we can formulate a mean drift condition that generalizes the one for QBD processes. It asserts that the mean drift to the right is smaller than the mean drift to the left.

\begin{theorem}[Stability condition]\label{thmQSG:stability_condition}%
The \textup{QSF} process $\{ X(t) \}_{t \ge 0}$ is positive recurrent if and only if
\begin{equation}%
\vca{x} \sum_{i \ge 1} \La_i \oneb < \vca{x} \sum_{i \le -1} \La_i \oneb \label{eqnQSF:stability_condition}
\end{equation}%
with $\vca{x} = \begin{bmatrix} x(0) & x(1) & \cdots & x(r) \end{bmatrix}$ the equilibrium distribution of the Markov process with transition rate matrix $\La$\textup{:}
\begin{equation}%
\vca{x} \La = \zerob, \quad \vca{x} \oneb = 1. \label{eqnQSF:stability_condition_vertical_eq_dist}
\end{equation}%
\end{theorem}%

%%%%%%%%%%%%%%%%%%%%%%%%%%%%%%%%%%%%%%%%%%%%%%%%%%%%%%%
%%%%%%%%%%%%%%%%%%%%%%%%%%%%%%%%%%%%%%%%%%%%%%%%%%%%%%%
%%%%%%%%%%%%%%%%%%%%% NEW SECTION %%%%%%%%%%%%%%%%%%%%%
%%%%%%%%%%%%%%%%%%%%%%%%%%%%%%%%%%%%%%%%%%%%%%%%%%%%%%%
%%%%%%%%%%%%%%%%%%%%%%%%%%%%%%%%%%%%%%%%%%%%%%%%%%%%%%%

\section{Matrix-geometric method}%
\label{secQSF:matrix-geometric_method}%

The matrix-geometric method is applicable to processes that are QSF to the right and have as their transition rate matrix the one shown in \eqref{eqnQSF:transition_rate_matrix_QSF_to_the_right}. Instrumental for the approach is the rate matrix $R$ that gives rise to the matrix-geometric relation
\begin{equation}%
\pb_i = \pb_1 R^{i - 1}, \quad i \ge 1. \label{eqnQSF:matrix-geometric_relation}
\end{equation}%
In the QBD case, the matrix $R$ is the minimal non-negative solution of the matrix-quadratic equation \eqref{eqnQBD:R_matrix_equation}. In the QSF case this equation is no longer quadratic. In fact, $R$ is the minimal non-negative solution of
\begin{equation}%
\sum_{i \ge 0} R^i \La_{1 - i} = 0. \label{eqnQSF:rate_matrix_R_satisfies}
\end{equation}%
The largest eigenvalue (in terms of absolute value) of the matrix $R$ is less than one, which ensures that $\I - R$ is invertible. Of course, given that $R$ satisfies \eqref{eqnQSF:rate_matrix_R_satisfies}, it readily follows that the matrix-geometric representation \eqref{eqnQSF:matrix-geometric_relation} satisfies the balance equations \eqref{eqnQSF:QSF_right_balance_equations_level_i}: substitution of \eqref{eqnQSF:matrix-geometric_relation} into \eqref{eqnQSF:QSF_right_balance_equations_level_i} gives
\begin{equation}%
\pb_{i - 1} \sum_{k \ge 0} R^k \La_{1 - k} = 0, \quad i \ge 2,
\end{equation}%
which is valid because of \eqref{eqnQSF:rate_matrix_R_satisfies}. Equation~\eqref{eqnQSF:rate_matrix_R_satisfies} can be rewritten as
\begin{equation}%
R = - \bigl( \La_1 + \sum_{i \ge 2} R^i \La_{1 - i} \bigr) \La_0^{-1}.
\end{equation}%
To numerically solve this equation we first have to truncate the infinite sum at $K$ say, and then compute an approximation for $R$ by successive substitutions as done in \cref{algQBD:matrix-geometric_successive_substitutions}. The larger $K$, the better the resulting approximation for $R$, but also the higher the computational effort to compute this approximation. We finally mention that the rate matrix $R$ has the same probabilistic interpretation as in a QBD process.

\begin{algorithm}%
\caption{Calculating $R$ using successive substitutions}%
\label{algQSF:matrix-geometric_successive_substitutions}%
\begin{algorithmic}[1]%
\State Pick $\epsilon$ small and positive and $K$ a large integer
\State Set $R_0 = 0$, $R_1 = - \La_1 \La_0^{-1}$ and $n = 1$
\While{$\| R_n - R_{n - 1} \|_{\textup{max}} > \epsilon$}
    \State Compute $R_{n + 1}$ from
    \begin{equation}%
    R_{n + 1} = - \bigl( \La_1 + \sum_{i = 2}^K R_n^i \La_{1 - i} \bigr) \La_0^{-1}
    \end{equation}%
    \State Update $n = n + 1$
\EndWhile
\end{algorithmic}%
\end{algorithm}%

Once we have determined $R$, we can solve for the remaining equilibrium probability vectors $\pb_0$ and $\pb_1$. Substituting \eqref{eqnQSF:matrix-geometric_relation} into the balance equations \eqref{eqnQSF:QSF_right_balance_equations_level_0}--\eqref{eqnQSF:QSF_right_balance_equations_level_1} for levels 0 and 1 gives
\begin{align}%
\pb_0 \La_0^{(0)} + \pb_1 \sum_{k \ge 1} R^{k - 1} \La_{-k}^{(k)} &= \zerob, \label{eqnQSF:matrix-geometric_method_boundary_equations_level_0} \\
\pb_0 \La_1^{(0)} + \pb_1 \sum_{k \ge 1} R^{k - 1} \La_{1 - k} &= \zerob. \label{eqnQSF:matrix-geometric_method_boundary_equations_level_1}
\end{align}%
%
%If the inverse of $R$ exists, we can pre-multiply the infinite sum in \eqref{eqnQSF:matrix-geometric_method_boundary_equations_level_1} by $R^{-1} R$ and use \eqref{eqnQSF:rate_matrix_R_satisfies} to simplify the infinite sum:
%%
%\begin{equation}%
%\pb_0 \La_1^{(0)} - \pb_1 R^{-1} \La_1 = \zerob.
%\end{equation}%
%%
Replacing one of the boundary equations with the normalization condition
\begin{equation}%
1 = \sum_{i \ge 0} \pb_i \oneb = \pb_0 \oneb + \pb_1 \sum_{i \ge 1} R^{i - 1} \oneb = \pb_0 \oneb + \pb_1 (\I - R)^{-1} \oneb
\end{equation}%
allow us to uniquely determine $\pb_0$ and $\pb_1$. We do, however, need to truncate the infinite series in \eqref{eqnQSF:matrix-geometric_method_boundary_equations_level_0}--\eqref{eqnQSF:matrix-geometric_method_boundary_equations_level_1} to be able to numerically determine $\pb_0$ and $\pb_1$.

%%%%%%%%%%%%%%%%%%%%%%%%%%%%%%%%%%%%%%%%%%%%%%%%%%%%%%%
%%%%%%%%%%%%%%%%%%%%%%%%%%%%%%%%%%%%%%%%%%%%%%%%%%%%%%%
%%%%%%%%%%%%%%%%%%%%% NEW SECTION %%%%%%%%%%%%%%%%%%%%%
%%%%%%%%%%%%%%%%%%%%%%%%%%%%%%%%%%%%%%%%%%%%%%%%%%%%%%%
%%%%%%%%%%%%%%%%%%%%%%%%%%%%%%%%%%%%%%%%%%%%%%%%%%%%%%%

\section{Matrix-analytic method}%
\label{secQSF:matrix-analytic_method}%

Processes that are QSF to the left do not have a matrix-geometric representation of the equilibrium probability vectors. In \cref{subsecQSF:setup_times} we have developed a recursive scheme involving the auxiliary matrix $G$ to determine the equilibrium probability vectors. In general, for processes that are QSF to the left, $G$ is the minimal non-negative solution of
\begin{equation}%
\sum_{i \ge 0} \La_{-1 + i} G^i = 0.
\end{equation}%
Similarly as for the calculation of $R$, we are required to truncate the infinite sum at $K$ say, and then approximate $G$ using successive substitutions, which leads to \cref{algQSF:matrix-analytic_successive_substitutions}. The matrix $G$ has the same probabilistic interpretation as in a QBD process.

\begin{algorithm}%
\caption{Calculating $G$ using successive substitutions}%
\label{algQSF:matrix-analytic_successive_substitutions}%
\begin{algorithmic}[1]%
\State Pick $\epsilon$ small and positive and $K$ a large integer
\State Set $G_0 = 0$, $G_1 = -\La_0^{-1} \La_{-1}$ and $n = 1$
\While{$\| G_n - G_{n - 1} \|_{\textup{max}} > \epsilon$}
    \State Compute $G_{n + 1}$ from
    \begin{equation}%
    G_{n + 1} = -\La_0^{-1} \bigl( \La_{-1} + \sum_{i = 2}^K \La_{-1 + i} G_n^i \bigr)
    \end{equation}%
    \State Update $n = n + 1$
\EndWhile
\end{algorithmic}%
\end{algorithm}%

In \cref{subsecQSF:setup_times} we censored the QSF process to $\lvls{i}$ for all $i$ and wrote down the balance equations for $\lvl{i}$. We can use the same principal in the present case to develop a recursive scheme for $\pb_i$. We will use the following notation to develop that scheme:
\begin{equation}%
\Gamma_i^{(0)} \defi \sum_{k \ge 0} \La_{i + k}^{(0)} G^k, ~ i \ge 0, \quad \Gamma_i \defi \sum_{k \ge 0} \La_{i + k} G^k, ~ i \ge 0.
\end{equation}%
The matrix $\Gamma_i^{(0)}$ describes the rates at which the QSF process enters $\lvl{i}$ from $\lvl{0}$ in a single transition when the process is censored to $\lvls{i}$. The matrix $\Gamma_i$ describes the rates at which the QSF process enters $\lvl{i + j}$ from $\lvl{j}, ~ j \ge 1$ in a single transition when the process is censored to $\lvls{i + j}$.

Using these definitions, we can censor the QSF process to $\lvl{0}$ to develop the relation
\begin{equation}%
\pb_0 \Gamma_0^{(0)} =  \zerob. \label{eqnQSF:matrix-analytic_method_censored_balance_level_0}
\end{equation}%
Censoring the QSF process to $\lvls{i}, ~ i \ge 1$, provides the recursive relation\endnote{The recursive scheme in the matrix-analytic method for levels $i \ge 1$ is called Ramaswami's formula and was developed by Ramaswami in \cite{Ramaswami1988_Matrix-analytic_stable_recursion}. This formula is numerically stable because it involves only additions and multiplications of non-negative matrices and vectors; subtractions could lead to numerical instability due to the loss of significant figures.}
\begin{equation}%
\pb_0 \Gamma_i^{(0)} + \sum_{j = 1}^i \pb_j \Gamma_{i - j} = \zerob, \quad i \ge 1. \label{eqnQSF:matrix-analytic_method_Ramaswami_formula}
\end{equation}%
Now, if we are able to determine $\pb_0$, then we can use \eqref{eqnQSF:matrix-analytic_method_Ramaswami_formula} to determine any $\pb_i$. The homogeneous system of equations \eqref{eqnQSF:matrix-analytic_method_censored_balance_level_0} does not have a unique solution, so we aim to supplement this system with the normalization condition. Let us add \eqref{eqnQSF:matrix-analytic_method_Ramaswami_formula} over all $i \ge 1$, which gives
\begin{equation}%
\pb_0 \sum_{i \ge 1} \Gamma_i^{(0)} + \sum_{i \ge 1} \sum_{j = 1}^i \pb_j \Gamma_{i - j} = \zerob.
\end{equation}%
Interchanging the double summation yields
\begin{equation}%
\pb_0 \sum_{i \ge 1} \Gamma_i^{(0)} + \sum_{j \ge 1} \pb_j \sum_{i \ge j} \Gamma_{i - j} = \zerob,
\end{equation}%
or
\begin{equation}%
\pb_0 \sum_{i \ge 1} \Gamma_i^{(0)} + \sum_{j \ge 1} \pb_j \sum_{i \ge 0} \Gamma_i = \zerob.
\end{equation}%
Post-multiplying by the inverse of $\sum_{i \ge 0} \Gamma_i$ (assuming it exists) and then post-multiplying by $\oneb$ gives
\begin{equation}%
\pb_0 \sum_{i \ge 1} \Gamma_i^{(0)} \Bigl( \sum_{i \ge 0} \Gamma_i \Bigr)^{-1} \oneb + \sum_{j \ge 1} \pb_j \oneb = 0,
\end{equation}%
and by the normalization condition $\sum_{i \ge 0} \pb_i \oneb = 1$ this leads to
\begin{equation}%
\pb_0 \sum_{i \ge 1} \Gamma_i^{(0)} \Bigl( \sum_{i \ge 0} \Gamma_i \Bigr)^{-1} \oneb + 1 - \pb_0 \oneb  = 0,
\end{equation}%
and therefore
\begin{equation}%
\pb_0 \Bigl( \oneb - \sum_{i \ge 1} \Gamma_i^{(0)} \Bigl( \sum_{i \ge 0} \Gamma_i \Bigr)^{-1} \oneb \Bigr) = 1. \label{eqnQSF:matrix-analytic_method_normalization_condition_p_0}
\end{equation}%
By substituting \eqref{eqnQSF:matrix-analytic_method_normalization_condition_p_0} for any of the equations in \eqref{eqnQSF:matrix-analytic_method_censored_balance_level_0} the value of $\pb_0$ can be uniquely determined.

Many of the equations required to determine the equilibrium probability vectors involve infinite sums that need to be truncated for actual computations. Furthermore, there needs to be a criterion for when the computations can be stopped. A natural stopping criterion is by examining the accumulated probability mass. In \cref{algQSF:matrix-analytic_method} we demonstrate the implementation of the matrix-analytic approach.

\begin{algorithm}%
\caption{Matrix-analytic method}%
\label{algQSF:matrix-analytic_method}%
\begin{algorithmic}[1]%
\State Pick $\epsilon$ small and positive and $K$ a large integer
\State Approximate $G$ using \cref{algQSF:matrix-analytic_successive_substitutions}
\State Approximate and store $\{ \Gamma_i^{(0)} \}_{0 \le i \le K}$ and $\{ \Gamma_i \}_{0 \le i \le K}$. Use only the first $K$ terms in each infinite sum
\State Calculate $\pb_0$ from \eqref{eqnQSF:matrix-analytic_method_censored_balance_level_0} and \eqref{eqnQSF:matrix-analytic_method_normalization_condition_p_0} by truncating each infinite sum to $K$ terms
\State Set $n = 0$
\While{$\sum_{m = 0}^n \pb_m \oneb < 1 - \epsilon$}
    \State Compute $\pb_{n + 1}$ from the truncated version of \eqref{eqnQSF:matrix-analytic_method_Ramaswami_formula}, which reads
    \begin{equation}%
    \pb_0 \Gamma_i^{(0)} + \sum_{j = i - K}^i \pb_j \Gamma_{i - j} = \zerob
    \end{equation}%
    \State Update $n = n + 1$
\EndWhile
\end{algorithmic}%
\end{algorithm}

%%%%%%%%%%%%%%%%%%%%%%%%%%%%%%%%%%%%%%%%%%%%%%%%%%%%%%%
%%%%%%%%%%%%%%%%%%%%%%%%%%%%%%%%%%%%%%%%%%%%%%%%%%%%%%%
%%%%%%%%%%%%%%%%%%%%% NEW SECTION %%%%%%%%%%%%%%%%%%%%%
%%%%%%%%%%%%%%%%%%%%%%%%%%%%%%%%%%%%%%%%%%%%%%%%%%%%%%%
%%%%%%%%%%%%%%%%%%%%%%%%%%%%%%%%%%%%%%%%%%%%%%%%%%%%%%%

\section{Spectral expansion method}%
\label{secQSF:spectral_expansion_method}%

The spectral expansion method only works for processes that are QSF to the right. This follows naturally from the fact that the spectral expansion method uses the eigenvalues and left eigenvectors of the rate matrix $R$ of the matrix-geometric method, which is also only applicable to processes that are QSF to the right. Recall that the aim of the spectral expansion method is to linearly combine basis solutions of the form
\begin{equation}%
\pb_i = \vca{y} x^{i - 1}, \quad i \ge 1. \label{eqnQSF:spectral_expansion_basis_solution}
\end{equation}%
By substituting \eqref{eqnQSF:spectral_expansion_basis_solution} into \eqref{eqnQSF:QSF_right_balance_equations_level_i} and dividing by common powers of $x$ we obtain
\begin{equation}%
\vca{y} \bigl( \La_1 + x \La_0 + x^2 \La_{-1} + x^3 \La_{-2} + \cdots \bigr) = \zerob. \label{eqnQSF:spectral_expansion_y_and_x}
\end{equation}%
As we have argued before in the QBD case, these equations have a non-zero solution for $\vca{y}$ if
\begin{equation}%
\det \bigl( \La_1 + x \La_0 + x^2 \La_{-1} + x^3 \La_{-2} + \cdots \bigr) = 0. \label{eqnQSF:spectral_expansion_determinant_equation}
\end{equation}%
Even though this determinant equation involves unbounded powers of $x$, it still provides us with exactly $r + 1$ solutions for $x$ that lie inside the open unit disk. For a detailed discussion as to why this is the case, see \cref{secQBD:spectral_expansion_method}. For numerical calculation purposes the determinant equation needs to be truncated. A rule of thumb could be to discard any terms with powers higher $3r$, for example.

Label the roots $x$ of \eqref{eqnQSF:spectral_expansion_determinant_equation} inside the closed unit disk as $x_0,x_1,\ldots,x_r$ and associate with the roots the corresponding non-zero eigenvectors $\vca{y}_0,\vca{y}_1,\ldots,\vca{y}_r$ found from \eqref{eqnQSF:spectral_expansion_y_and_x}. We assume that the eigenvectors are linearly independent, which is the case if roots $x_k$ are different, but independence can also be the case even if some of the roots are identical. Now, each solution $\pb_i = \vca{y}_k x_k^i, ~ k = 0,1,\ldots,r$ satisfies the global balance equations \eqref{eqnQSF:QSF_right_balance_equations_level_i}. We can linearly combine these solutions as
\begin{equation}%
\pb_i = \sum_{k = 0}^r \xi_k \vca{y}_k x_k^{i - 1}, \quad i \ge 1, \label{eqnQSF:spectral_expansion_linear_combination_of_basis_solutions}
\end{equation}%
where $\xi_k, ~ k = 0,1,\ldots,r$ are constants that we still need to determine. Substituting \eqref{eqnQSF:spectral_expansion_linear_combination_of_basis_solutions} into the balance equations for levels 0 and 1 gives
\begin{align}%
\pb_0 \La_0^{(0)} + \sum_{l = 0}^r \xi_l \vca{y}_l \sum_{k \ge 1} x_l^{k - 1} \La_{-k}^{(k)} &= \zerob, \label{eqnQSF:spectral_expansion_boundary_equation_level_0} \\
\pb_0 \La_1^{(0)} + \sum_{l = 0}^r \xi_l \vca{y}_l \sum_{k \ge 1}  x_l^{k - 1} \La_{1 - k} &= \zerob. \label{eqnQSF:spectral_expansion_boundary_equation_level_1}
\end{align}%
We can simplify the infinite sum in \eqref{eqnQSF:spectral_expansion_boundary_equation_level_1} by using \eqref{eqnQSF:spectral_expansion_y_and_x}:
\begin{equation}%
\pb_0 \La_1^{(0)} - \sum_{l = 0}^r \xi_l \vca{y}_l \frac{1}{x_l} \La_1 = \zerob.
\end{equation}%
Replacing one of the boundary equations with the normalization condition
\begin{equation}%
1 = \sum_{i \ge 0} \pb_i \oneb = \pb_0 \oneb + \sum_{i \ge 1} \sum_{k = 0}^r \xi_k \vca{y}_k x_k^{i - 1} \oneb = \pb_0 \oneb + \pb_1 \sum_{k = 0}^r \frac{\xi_k \vca{y}_k}{1 - x_k} \oneb
\end{equation}%
allow us to uniquely determine $\pb_0$ and $\pb_1$. We do, however, need to truncate the infinite series in \eqref{eqnQSF:spectral_expansion_boundary_equation_level_0} to be able to numerically determine $\pb_0$ and $\pb_1$.

\section{Takeaways}%
\label{secQSF:takeaways}%

This chapter exploited two structural properties we have encountered in earlier chapters. In \cref{ch:queues_and_transforms} we saw how constructing embedded Markov processes could help analyze processes with larger jumps to the left or to the right. In \cref{ch:quasi-birth--and--death_processes} we saw how the skip-free structure of birth--and--death processes could be extended to two-dimensional quasi-birth--and--death processes that remained skip-free in one dimension, but could skip in the other. This chapter combined both features in quasi-skip-free (QSF) processes that can skip states in both dimensions.

This additional flexibility comes with mathematical challenges, but the main methods used in earlier chapters again work, albeit in a more advanced form. The matrix-geometric and matrix-analytic method involved solutions of matrix equations with unbounded powers of $R$ and $G$. The spectral expansion method worked after considering the determinant equation over infinitely many powers. For pratical purposes, these techniques require truncation of the infinite sums to be amenable for numerical calculations. The three techniques together provide a good handle, both analytically and algorithmically, on analyzing the rich class of QSF processes.

%%%%%%%%%%%%%%%%%%%%%%%%%%%%%%%%%%%%%%%%%%%%%%%%%%%%%%%
%%%%%%%%%%%%%%%%%%%%%%%%%%%%%%%%%%%%%%%%%%%%%%%%%%%%%%%
%%%%%%%%%%%%%%%%%%%%%%%% NOTES %%%%%%%%%%%%%%%%%%%%%%%%
%%%%%%%%%%%%%%%%%%%%%%%%%%%%%%%%%%%%%%%%%%%%%%%%%%%%%%%
%%%%%%%%%%%%%%%%%%%%%%%%%%%%%%%%%%%%%%%%%%%%%%%%%%%%%%%

%\theendnotes%
%\setcounter{endnote}{0}
\printendnotes% %

%%%%%%%%%%%%%%%%%%%%%%%%%%%%%%%%%%%%%%%%%%%%%%%%%%%%%%%
%%%%%%%%%%%%%%%%%%%%%%%%%%%%%%%%%%%%%%%%%%%%%%%%%%%%%%%
%%%%%%%%%%%%%%%%%%%%%%% NEW PART %%%%%%%%%%%%%%%%%%%%%%
%%%%%%%%%%%%%%%%%%%%%%%%%%%%%%%%%%%%%%%%%%%%%%%%%%%%%%%
%%%%%%%%%%%%%%%%%%%%%%%%%%%%%%%%%%%%%%%%%%%%%%%%%%%%%%%

\part{Advanced processes}%

% Checked points 1-7 and a-i
\chapter{Priority systems}%
\label{ch:single-server_priority}%

In this chapter we consider a priority system\endnote{Priority queueing systems have been studied for a long time. Cobham \cite{Cobham1954_Priority_assignment}, Davis \cite{Davis1966_M_n-M-c_N-class_prio_waiting_time_distributions} and Jaiswal \cite{Jaiswal1961_Preemptive_resume_priority,Jaiswal1968_Priority_queues} are among the first to formulate and analyze priority models. The earliest studies derive expressions for the marginal distribution of each class, whereas we are interested in the joint distribution of the underlying Markov. More recently, attention has been given to systems with an arbitrary number of priority classes in Sleptchenko et al  \cite{Sleptchenko2015_M-M-1_priority} or to multi-server priority systems in Wang, Baron, Scheller-Wolf \cite{Wang2015_M-M-c_2-class_prio} and Selen and Fralix \cite{Selen2017_MMc_prio}.}  with a single exponential server that serves jobs of high and low priority arriving according to Poisson streams. High-priority jobs are served before low-priority jobs, low-priority jobs are only served when there are no high-priority jobs in the system. Whenever a high-priority job enters the system and a low-priority job is in service, the low-priority job is removed from service and placed at the head of the line while the high-priority job is immediately taken into service. This priority rule is referred to as \textit{preemptive priority}: the high-priority job preempts the service of the low-priority job.

To determine the equilibrium distribution of the two-dimensional Markov process of the number of high- and low-priority jobs associated with this single-server priority system we take three approaches. The first approaches recursively solves the balance equations by using second-order difference equations. The second approach translates the balance equations into a quadratic (functional) equation to find the bivariate PGF of the joint equilibrium distribution of high- and low-priority jobs. The third approach casts the balance equations into a QBD matrix structure, and uses the matrix-geometric and matrix-analytic methods to find a product-form solution for the equilibrium distribution. Due to the structure of the transition rate diagram the elements of the infinite-dimensional rate matrix $R$ and auxiliary matrix $G$ are easily determined.

%%%%%%%%%%%%%%%%%%%%%%%%%%%%%%%%%%%%%%%%%%%%%%%%%%%%%%%
%%%%%%%%%%%%%%%%%%%%%%%%%%%%%%%%%%%%%%%%%%%%%%%%%%%%%%%
%%%%%%%%%%%%%%%%%%%%% NEW SECTION %%%%%%%%%%%%%%%%%%%%%
%%%%%%%%%%%%%%%%%%%%%%%%%%%%%%%%%%%%%%%%%%%%%%%%%%%%%%%
%%%%%%%%%%%%%%%%%%%%%%%%%%%%%%%%%%%%%%%%%%%%%%%%%%%%%%%

\section{Model description and balance equations}%
\label{secPRIO:model_description}%

We distinguish the two job classes by numbering them: class-1 jobs have preemptive priority over class-2 jobs. The arrival process of class-$n$ jobs is a Poisson process with rate $\la_n$. Each class-$n$ job requires an exponentially distributed service time with rate $\mu_n$. Since the service requirements are exponentially distributed and thus memoryless, the residual service time of a class-2 job that was removed from service again has an exponential distribution with the same rate $\mu_2$. Denote by $\rho_n \defi \la_n/\mu_n$ the amount of work brought into the system per time unit by class-$n$ jobs.

Let $X_n(t)$ be the number of class-$n$ jobs in the system at time $t$ and denote the state of the system by $X(t) \defi (X_1(t),X_2(t))$. Then $\{ X(t) \}_{t \ge 0}$ is a Markov process on the state space $\statespace \defi \Nat_0^2$. It is apparent from the transition rate diagram in \cref{figPRIO:transition_rate_diagram_single-server_prio} that the state space is irreducible. The system is stable if the total amount of work brought into the system per time unit is strictly less than one. We therefore assume
\begin{equation}%
\rho \defi \rho_1 + \rho_2 < 1,
\end{equation}%
to guarantee positive recurrence and the existence of the equilibrium distribution. Denote the equilibrium probability of being in state $(i,j)$ as $p(i,j)$.

\begin{figure}
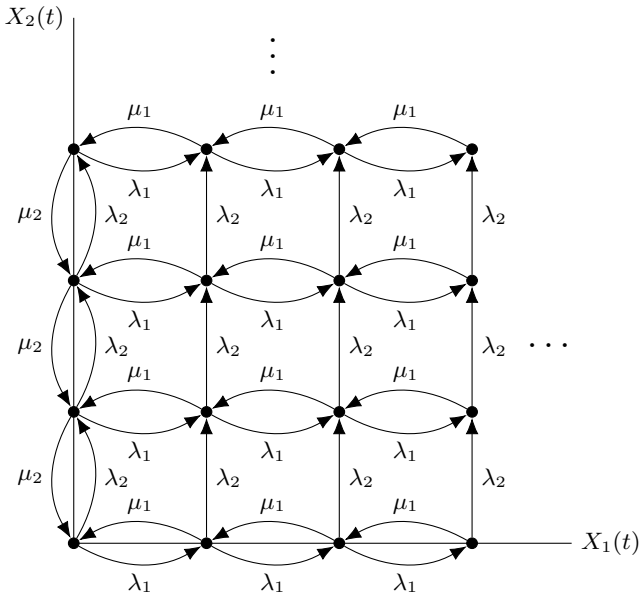
%
\centering%
\includestandalone{Chapters/PRIO/TikZFiles/transition_rate_diagram_single-server_prio}%
\caption{Transition rate diagram of the Markov process associated with the single-server priority system.}%
\label{figPRIO:transition_rate_diagram_single-server_prio}
\end{figure}%

The balance equations for the interior of the state space are given for $i,j \ge 1$ by
\begin{align}%
(\la + \mu_1) p(i,j) &= \la_1 p(i - 1,j) + \mu_1 p(i + 1,j) + \la_2 p(i,j - 1) \label{eqnPRIO:balance_equations_interior}
\end{align}%
with $\la \defi \la_1 + \la_2$. For the horizontal and vertical axis we have
\begin{align}%
(\la + \mu_1) p(i,0) &= \la_1 p(i - 1,0) + \mu_1 p(i + 1,0), \quad i \ge 1, \label{eqnPRIO:balance_equations_horizontal} \\
(\la + \mu_2) p(0,j) &= \la_2 p(0,j - 1) + \mu_2 p(0,j + 1) + \mu_1 p(1,j), \quad j \ge 1. \label{eqnPRIO:balance_equations_vertical}
\end{align}%
Finally, for the origin the balance equation is
\begin{equation}%
\la p(0,0) = \mu_1 p(1,0)+ \mu_2 p(0,1). \label{eqnPRIO:balance_equations_origin}
\end{equation}%
Since the server always works at unit rate whenever there is work to do, $p(0,0) = 1 - \rho$.

%%%%%%%%%%%%%%%%%%%%%%%%%%%%%%%%%%%%%%%%%%%%%%%%%%%%%%%
%%%%%%%%%%%%%%%%%%%%%%%%%%%%%%%%%%%%%%%%%%%%%%%%%%%%%%%
%%%%%%%%%%%%%%%%%%%%% NEW SECTION %%%%%%%%%%%%%%%%%%%%%
%%%%%%%%%%%%%%%%%%%%%%%%%%%%%%%%%%%%%%%%%%%%%%%%%%%%%%%
%%%%%%%%%%%%%%%%%%%%%%%%%%%%%%%%%%%%%%%%%%%%%%%%%%%%%%%

\section{Difference equations approach}%
\label{secPRIO:difference_equations}%

We exploit the upward structure of the transition rate diagram by first solving the balance equations for $j = 0$ and working our way up by increasing $j$ one step at a time. For $j = 0$, \eqref{eqnPRIO:balance_equations_horizontal} is a homogeneous difference equation of order two:
\begin{equation}%
\mu_1 p(i + 1,0) - (\la + \mu_1) p(i,0) + \la_1 p(i - 1,0) = 0, \quad i \ge 1. \label{eqnPRIO:diff_j=0}
\end{equation}%
We have the general solution
\begin{equation}%
p(i,0) = c_{0,0} x_1^i + c_{0,1} x_2^i, \quad i \ge 0 \label{eqnPRIO:diff_p(i,0)_guess}
\end{equation}%
where $x_1$ and $x_2$ are the roots of the quadratic equation
\begin{equation}%
\mu_1 x^2 - (\la + \mu_1) x + \la_1 = 0 \label{eqnPRIO:diff_roots_x}
\end{equation}%
satisfying $0 < x_1 < 1 < x_2$. Since the solution of \eqref{eqnPRIO:diff_p(i,0)_guess} needs to be normalized, and $\sum_{i \ge 1} x_2^i = \infty$, we set $c_{0,1} = 0$. This gives with $x \defi x_1$,
\begin{equation}%
p(i,0) = c_{0,0} x^i, \quad i \ge 0,
\end{equation}%
where we leave $c_{0,0}$ undetermined for now.

For $j = 1$, \eqref{eqnPRIO:balance_equations_interior} is a nonhomogeneous difference equation of order two:
\begin{equation}%
\mu_1 p(i + 1,1) - (\la + \mu_1) p(i,1) + \la_1 p(i - 1,1) = - \la_2 p(i,0), \quad i \ge 1 \label{eqnPRIO:diff_j=1}
\end{equation}%
and its solution will be a combination of the general solution of the homogeneous equation and a particular solution of the nonhomogeneous equation. Clearly, the solution of the homogeneous equation is $c_{1,0} x^i$, where $c_{1,0}$ is a constant that we determine later. For the solution to the nonhomogeneous equation we guess that it is of the form $p(i,1) = c_{1,1} \binom{i + 1}{1} x^i$. Substituting this guess into \eqref{eqnPRIO:diff_j=1} and dividing by $x^{i - 1}$ gives
\begin{equation}%
\mu_1 c_{1,1} \binom{i + 2}{1} x^2 - (\la + \mu_1) c_{1,1} \binom{i + 1}{1} x + \la_1 \binom{i}{1} = - \la_2 c_{0,0} x.
\end{equation}%
Since $x$ satisfies \eqref{eqnPRIO:diff_roots_x}, we obtain
\begin{equation}%
c_{1,1} = \frac{\la_2 c_{0,0}}{\la + \mu_1 - 2 \mu_1 x}
\end{equation}%
and
\begin{equation}%
p(i,1) = c_{1,0} \binom{i + 1}{0} x^i + c_{1,1} \binom{i + 1}{1} x^i.
\end{equation}%

For $j = 2$, \eqref{eqnPRIO:balance_equations_interior} gives
\begin{equation}%
\mu_1 p(i + 1,2) - (\la + \mu_1) p(i,2) + \la_1 p(i - 1,2) = - \la_2 p(i,1), \quad i \ge 1. \label{eqnPRIO:diff_j=2}
\end{equation}%
The solution of the homogeneous version of \eqref{eqnPRIO:diff_j=2} is $c_{2,0} x^i$, where $c_{2,0}$ is a constant that we determine later. Substituting the guess $p(i,2) = c_{2,1} \binom{i + 2}{1} x^i + c_{2,2} \binom{i + 2}{2} x^i$ into \eqref{eqnPRIO:diff_j=2} and dividing by $x^{i - 1}$ gives
\begin{align}%
&\mu_1 x^2 \Bigl( c_{2,1} \binom{i + 3}{1} + c_{2,2} \binom{i + 3}{2} \Bigr) - (\la + \mu_1) x \Bigl( c_{2,1} \binom{i + 2}{1} + c_{2,2} \binom{i + 2}{2} \Bigr) \notag \\
&\quad + \la_1 \Bigl( c_{2,1} \binom{i + 1}{1} + c_{2,2} \binom{i + 1}{2} \Bigr) = -\la_2 x \Bigl( c_{1,0} \binom{i + 1}{0} + c_{1,1} \binom{i + 1}{1} \Bigr). \label{eqnPRIO:diff_j=2_subs_1}
\end{align}%
We use
\begin{align*}%
\binom{i + 3}{1} &= \bigl( 1 + \frac{2}{i + 1} \bigr) \binom{i + 1}{1}, \quad \binom{i + 3}{2} = \bigl( 1 + \frac{4}{i} + \frac{2}{i(i + 1)} \bigr) \binom{i + 1}{2}, \\
\binom{i + 2}{1} &= \bigl( 1 + \frac{1}{i + 1} \bigr) \binom{i + 1}{1}, \quad \binom{i + 2}{2} = \bigl(1 + \frac{2}{i} \bigr) \binom{i + 1}{2},
\end{align*}%
and the fact that $x$ satisfies \eqref{eqnPRIO:diff_roots_x} to simplify \eqref{eqnPRIO:diff_j=2_subs_1} to
\begin{align}%
&c_{2,1} \frac{1}{i + 1} \binom{i + 1}{1} (2 \mu_1 x - (\la + \mu_1)) \notag \\
&+ c_{2,2} \binom{i + 1}{2} \Bigl( \bigl( \frac{4}{i} + \frac{2}{i(i + 1)} \bigr) \mu_1 x - \frac{2}{i} (\la + \mu_1) \Bigr) \notag \\
&= -\la_2 \Bigl( c_{1,0} \binom{i + 1}{0} + c_{1,1} \binom{i + 1}{1} \Bigr).
\end{align}%
Grouping terms in front of the binomial coefficients gives
\begin{align}%
&\binom{i + 1}{0} \bigl( c_{2,1} ( 2 \mu_1 x - (\la + \mu_1) ) + c_{2,2} \mu_1 x \bigr) + \binom{i + 1}{1} c_{2,2} \bigl( 2 \mu_1 x - (\la + \mu_1) \bigr) \notag \\
&= - \binom{i + 1}{0} c_{1,0} \la_1 - \binom{i + 1}{1} c_{1,1} \la_1.
\end{align}%
Matching the coefficients of the binomial coefficients finally shows that
\begin{equation}%
c_{2,2} = \frac{\la_2 c_{1,1}}{\la + \mu_1 - 2 \mu_1 x}, \quad c_{2,1} = \frac{\la_2 c_{1,0} + \mu_1 x c_{2,2}}{\la + \mu_1 - 2 \mu_1 x}.
\end{equation}%

Repeating this procedure leads to the general expression
\begin{equation}%
p(i,j) = \sum_{k = 0}^j c_{j,k} \binom{i + j}{k} x^i, \quad i,j \ge 0, \label{eqnPRIO:diff_p(i,j)}
\end{equation}%
where the coefficients satisfy the recursion, for $j \ge 1$,
\begin{align}%
c_{j,j} &= \frac{\la_2 c_{j - 1,j - 1}}{\la + \mu_1 - 2 \mu_1 x}, \label{eqnPRIO:diff_c_j,j} \\
c_{j,k} &= \frac{\la_2 c_{j - 1,k - 1} + \mu_1 x c_{j,k + 1}}{\la + \mu_1 - 2 \mu_1 x}, \quad 1 \le k \le j - 1. \label{eqnPRIO:diff_c_j,k}
\end{align}%
The coefficients $c_{j,0}, ~ j \ge 0$ still need to be determined. Since $p(0,0) = 1 - \rho$ we have $c_{0,0} = 1 - \rho$. Substituting \eqref{eqnPRIO:diff_p(i,j)} into the balance equation \eqref{eqnPRIO:balance_equations_origin} gives
\begin{equation}%
c_{1,0} = \frac{(\la - \mu_1 x)}{\mu_2} c_{0,0} - c_{1,1}. \label{eqnPRIO:diff_c_1,0}
\end{equation}%
From \eqref{eqnPRIO:balance_equations_vertical} we obtain, for $j \ge 1$,
\begin{align}%
c_{j + 1,0} &= \frac{1}{\mu_2} \sum_{k = 0}^j c_{j,k} \Bigl( \mu_1 x \binom{j + 1}{k} - (\la + \mu_2) \binom{j}{k} \Bigr) \notag \\
&\quad + \rho_2 \sum_{k = 0}^{j - 1} c_{j - 1,k} \binom{j - 1}{k} - \sum_{k = 1}^{j + 1} c_{j + 1,k} \binom{j + 1}{k}. \label{eqnPRIO:diff_c_j,0}
\end{align}%
We outline the computation of the coefficients in \cref{algPRIO:diff_computation_coefficients}.

\begin{algorithm}%
\caption{Calculating the coefficients}%
\label{algPRIO:diff_computation_coefficients}%
\begin{algorithmic}[1]%
\State Pick a positive integer $j_{\textup{max}}$
\State Set $c_{0,0} = 1 - \rho$
\State Calculate $c_{1,1}$ from \eqref{eqnPRIO:diff_c_j,j}
\State Calculate $c_{1,0}$ from \eqref{eqnPRIO:diff_c_1,0}
\For{$j = 2,3,\ldots,j_{\textup{max}}$}
    \State Calculate $c_{j,j}$ from \eqref{eqnPRIO:diff_c_j,j}
    \For{$k = j - 1,j - 2,\ldots,1$}
        \State Calculate $c_{j,k}$ from \eqref{eqnPRIO:diff_c_j,k}
    \EndFor
    \State Calculate $c_{j,0}$ from \eqref{eqnPRIO:diff_c_j,0}
\EndFor
\end{algorithmic}%
\end{algorithm}%

%%%%%%%%%%%%%%%%%%%%%%%%%%%%%%%%%%%%%%%%%%%%%%%%%%%%%%%
%%%%%%%%%%%%%%%%%%%%%%%%%%%%%%%%%%%%%%%%%%%%%%%%%%%%%%%
%%%%%%%%%%%%%%%%%%%%% NEW SECTION %%%%%%%%%%%%%%%%%%%%%
%%%%%%%%%%%%%%%%%%%%%%%%%%%%%%%%%%%%%%%%%%%%%%%%%%%%%%%
%%%%%%%%%%%%%%%%%%%%%%%%%%%%%%%%%%%%%%%%%%%%%%%%%%%%%%%

\section{Generating function approach}%
\label{secPRIO:pgf_approach}%

Define the bivariate PGF
\begin{equation}%
\PGF{x,y} \defi \sum_{i \ge 0} \sum_{j \ge 0} p(i,j) x^i y^j, \quad |x| \le 1, ~ |y| \le 1.
\end{equation}%
Note that $\PGF{x,0} = \sum_{i \ge 0} p(i,0) x^i$ and $\PGF{0,y} = \sum_{j \ge 0} p(0,j) y^j$ are the probability generating functions of the equilibrium probabilities of the states on the horizontal and vertical axis. Furthermore, $\PGF{x,1}$ and $\PGF{1,y}$ are the probability generating functions of the number of class-1 and class-2 jobs in the system, respectively.

We shall now perform a series of operations on the balance equations to obtain an expression for $\PGF{x,y}$. Multiply both sides of \eqref{eqnPRIO:balance_equations_interior} by $x^i y^j$ and sum over all $i,j \ge 1$ to obtain \eqref{eqnPRIO:pgf_balance_equations_summed_interior}. Multiply both sides of \eqref{eqnPRIO:balance_equations_horizontal} by $x^i$ and sum over all $i \ge 1$ to obtain \eqref{eqnPRIO:pgf_balance_equations_summed_horizontal}. Finally, multiply both sides of \eqref{eqnPRIO:balance_equations_vertical} by $y^j$ and sum over all $j \ge 1$ to obtain \eqref{eqnPRIO:pgf_balance_equations_summed_vertical}.
\begin{align}%
(\la + \mu_1) \sum_{i \ge 1} \sum_{j \ge 1} p(i,j) x^i y^j &= \la_1 \sum_{i \ge 1} \sum_{j \ge 1} p(i - 1,j) x^i y^j \notag \\
&\quad + \mu_1 \sum_{i \ge 1} \sum_{j \ge 1} p(i + 1,j) x^i y^j \notag \\
&\quad + \la_2 \sum_{i \ge 1} \sum_{j \ge 1} p(i,j - 1) x^i y^j, \label{eqnPRIO:pgf_balance_equations_summed_interior} \\
(\la + \mu_1) \sum_{i \ge 1} p(i,0) x^i &= \la_1 \sum_{i \ge 1} p(i - 1,0) x^i \notag \\
&\quad + \mu_1 \sum_{i \ge 1} p(i + 1,0) x^i, \label{eqnPRIO:pgf_balance_equations_summed_horizontal} \\
(\la + \mu_2) \sum_{j \ge 1} p(0,j) y^j &= \la_2 \sum_{j \ge 1} p(0,j - 1) y^j + \mu_2 \sum_{j \ge 1} p(0,j + 1) y^j \notag \\
&\quad + \mu_1 \sum_{j \ge 1} p(1,j) y^j.
\label{eqnPRIO:pgf_balance_equations_summed_vertical}
\end{align}%
Summing \eqref{eqnPRIO:pgf_balance_equations_summed_interior}--\eqref{eqnPRIO:pgf_balance_equations_summed_vertical} and \eqref{eqnPRIO:balance_equations_origin} and using simplifications based on the definition of $\PGF{x,y}$ such as
\begin{align}%
\sum_{i \ge 1} \sum_{j \ge 1} p(i,j) x^i y^j &= \sum_{i \ge 0} \sum_{j \ge 0} p(i,j) x^i y^j - \sum_{i \ge 0} p(i,0) x^i - \sum_{j \ge 0} p(0,j) y^j + p(0,0) \notag \\
&= \PGF{x,y} - \PGF{x,0} - \PGF{y,0} + \PGF{0,0},
\end{align}%
shows that $\PGF{x,y}$ satisfies the functional equation
\begin{equation}%
h_1(x,y) \PGF{x,y} = h_2(x,y) \PGF{0,y} + h_3(x,y) \PGF{0,0} \label{eqnPRIO:pgf_functional}
\end{equation}%
with
\begin{align}%
h_1(x,y) &\defi \la_1 x y (1 - x) + \la_2 x y (1 - y) - \mu_1 y (1 - x), \\
h_2(x,y) &\defi - \mu_1 y (1 - x) + \mu_2 x (1 - y), \\
h_3(x,y) &\defi - \mu_2 x (1 - y).
\end{align}%
The question is whether we can solve functional equation \eqref{eqnPRIO:pgf_functional}. Recall that $\PGF{0,0} = p(0,0) = 1 - \rho$.

\begin{remark}[Distribution of the number of class-1 jobs]%
Setting $y = 1$ in \eqref{eqnPRIO:pgf_functional} gives
\begin{equation}%
\PGF{x,1} = \frac{\PGF{0,1}}{1 - \rho_1 x},
\end{equation}%
where $\PGF{0,1}$ is the probability that there are no class-1 jobs in the system. Clearly, $\PGF{0,1} = 1 - \rho_1$ and therefore
\begin{equation}%
\PGF{x,1} = \frac{1 - \rho_1}{1 - \rho_1 x} = \sum_{i \ge 0} (1 - \rho_1) \rho_1^i x^i,
\end{equation}%
which shows that the number of class-1 jobs follows a geometric distribution with parameter $\rho_1$. Due to the preemptive priority, class-1 jobs do not have to wait for class-2 jobs and therefore experience the system as if it were a standard $M/M/1$ queue.
\end{remark}%

Now choose $x$ so that the left-hand side of \eqref{eqnPRIO:pgf_functional} vanishes to obtain an expression for $\PGF{0,y}$. For a fixed $y$ with $0 < |y| \le 1$, $h_1(x,y)$ is a second degree polynomial in $x$, and hence
\begin{equation}%
0 = h_1(x,y) \quad \Leftrightarrow \quad 0 = \rho_1 x^2 - \bigl(1 + \rho_1 + \frac{\la_2}{\mu_1}(1 - y) \bigr) x + 1. \label{eqnPRIO:pgf_kernel}
\end{equation}%

\begin{lemma}%
For a fixed $y$ with $0 < |y| \le 1$, \eqref{eqnPRIO:pgf_kernel} has a unique solution $x = \rootx(y)$ with $|x| \le 1$.
\end{lemma}%

\begin{proof}%
For now, fix a $y$ with $0 < |y| < 1$. We use Rouch\'e's theorem, see \cref{thm:Rouche}, to prove that \eqref{eqnPRIO:pgf_kernel} has a unique solution within the closed unit disk. Denote the closed unit disk by $\closedunitdisc$ and the unit circle by $\unitcircle$. Define the functions
\begin{equation}%
f_1(x,y) \defi - \bigl(1 + \rho_1 + \frac{\la_2}{\mu_1}(1 - y) \bigr) x, \quad g(x) \defi \rho_1 x^2 + 1.
\end{equation}%
Clearly, $f_1(x,y)$ has only one root $x = 0$ in $\closedunitdisc$. We aim to show that
\begin{equation}%
|f_1(x,y)| > |g(x)|, \quad x \in \unitcircle,
\end{equation}%
so that it follows from Rouch\'e's theorem that $f_1(x,y) + g(x)$ also has one root in $\closedunitdisc$.

Then,
\begin{align}%
|f_1(x,y)| &= \bigl| 1 + \rho_1 + \frac{\la_2}{\mu_1} - \frac{\la_2}{\mu_1} y \bigr| |x| \notag \\
&\ge \bigl( 1 + \rho_1 + \frac{\la_2}{\mu_1} - \frac{\la_2}{\mu_1} |y| \bigr) |x| \ifed f_2(|x|,|y|)
\end{align}%
and
\begin{equation}%
|g(x)| = | \rho_1 x^2 + 1 | \le \rho_1 |x|^2 + 1 = g(|x|).
\end{equation}%
It suffices to show that $f_2(|x|,|y|) > g(|x|)$ for $|x| = 1$ and $0 < |y| < 1$, which is clearly the case.

However, when $|x| = |y| = 1$ we have $f_2(1,1) = g(1)$. In order to use Rouch\'e's theorem for that particular case, we essentially evaluate $f_2(|x|,1)$ and $g(|x|)$ on the circle $|x| = 1 + \epsilon$ with $\epsilon$ small and positive. To accomplish this, we use the Taylor expansion $f_2(1,1 + \epsilon) = f_2(1,1) + \epsilon f_2'(1,1) + o(\epsilon)$ and verify that $f_2(1,1 + \epsilon) > g(1 + \epsilon)$. Since $f_2(1,1) = g(1)$ we are left to show that $f_2'(1,1) > g'(1)$. Now,
\begin{equation}%
f_2'(1,1) = \frac{\dinf}{\dinf |x|} f_2(|x|,1) \Big\vert_{|x| = 1} = 1 + \rho_1
\end{equation}%
and
\begin{equation}%
g'(1) = \frac{\dinf}{\dinf |x|} g(|x|) \Big\vert_{|x| = 1} = 2\rho_1,
\end{equation}%
which proves $f_2'(1,1) > g'(1)$ since $\rho_1 < 1$. So, for sufficiently small $\epsilon > 0$ we have that $f_2(|x|,1) > g(|x|)$ for $|x| \in (1,1 + \epsilon]$, which proves the claim.
\end{proof}%

The unique solution $x = \rootx(y)$ within the closed unit disk can easily be computed from the second-degree polynomial \eqref{eqnPRIO:pgf_kernel}:
\begin{equation}%
\rootx(y) = \frac{\la_1 + \mu_1 + \la_2(1 - y) - \sqrt{(\la_1 + \mu_1 + \la_2(1 - y))^2 - 4 \la_1 \mu_1}}{2\la_1}. \label{eqnPRIO:pgf_root_xi(y)}
\end{equation}%
We proceed by plugging $x = \rootx(y)$ and $\PGF{0,0} = 1 - \rho$ into \eqref{eqnPRIO:pgf_functional} to obtain
\begin{equation}%
\PGF{0,y} = \frac{-h_3(\rootx(y),y) \PGF{0,0}}{h_2(\rootx(y),y)} = \frac{\mu_2 \rootx(y)(1 - y)(1 - \rho)}{-\mu_1 y (1 - \rootx(y)) + \mu_2 \rootx(y)(1 - y)}
\end{equation}%
so that \eqref{eqnPRIO:pgf_kernel} gives the expression
\begin{equation}%
\PGF{x,y} = \frac{1 - \rho}{h_1(x,y)} \Bigl( \frac{h_2(x,y) \mu_2 \rootx(y)(1 - y)}{-\mu_1 y (1 - \rootx(y)) + \mu_2 \rootx(y)(1 - y)} + h_3(x,y) \Bigr). \label{eqnPRIO:pgf_P(x,y)}
\end{equation}%

We have converted the balance equations \eqref{eqnPRIO:balance_equations_interior}--\eqref{eqnPRIO:balance_equations_origin} into the functional equation \eqref{eqnPRIO:pgf_functional} and found a solution for $\PGF{x,y}$ in \eqref{eqnPRIO:pgf_P(x,y)}. So we took the direct, explicit relations between the equilibrium probabilities to the transform domain to find an indirect description of the equilibrium probabilities in terms of a complex-valued function $\PGF{x,y}$. Given this bivariate PGF of the joint equilibrium distribution, we can use \cref{algQTF:numerical_inversion_bivariate_PGF} to numerically invert $\PGF{x,y}$ to obtain any $p(i,j)$. In \cref{tblPRIO:numerical_inversion_pgf} we demonstrate how the algorithm parameters $j_1$ and $j_2$ influence the accuracy of the solution and the computation time required to obtain this solution.

\begin{table}%
\centering%
\begin{tabular}{*{4}{c}}
$j$ & $\bar{p}(0,0)$ & error bound \eqref{eqnQTF:bivariate_PGF_numerical_inversion_error_bound} & time (\si{\micro\second}) \\
\hline
1  & 0.1899979672       & 0.7777777777 & 9.36 \\
2  & 0.1694992925       & 0.1377777777 & 29.6 \\
3  & 0.1671319561       & 0.0319979843 & 56.2 \\
5  & 0.1666811419       & 0.0019559897 & 151 \\
10 & 0.1666666694       & 0.0000019073 & 563 \\
$\infty$ & 0.1666666666 &              &
\end{tabular}%
\caption{Comparing $\approximate{p}(0,0)$ obtained from \protect\cref{algQTF:numerical_inversion_bivariate_PGF} with the exact expression. Parameter values are $\la_1 = 1$, $\la_2 = 2$, $\mu_1 = 3$ and $\mu_2 = 4$. Algorithm settings are $\radiusNI_1(0) = \radiusNI_2(0) = 0.5$ and $j_1 = j_2 = j$.}
\label{tblPRIO:numerical_inversion_pgf}
\end{table}%

\begin{remark}[Distribution of the number of class-2 jobs]\label{remPRIO:distribution_class-2_jobs}%
Setting $x = 1$ in \eqref{eqnPRIO:pgf_P(x,y)} gives
\begin{equation}%
\PGF{1,y} = \frac{1 - \rho}{\rho_2} \frac{\mu_1 (1 - \rootx(y))}{-\mu_1 y (1 - \rootx(y)) + \mu_2 \rootx(y) (1 - y)}. \label{eqnPRIO:pgf_P(1,y)}
\end{equation}%
Denote the equilibrium distribution of the number of class-2 jobs by $p_2(\cdot)$ so that $\PGF{1,y} = \sum_{k \ge 0} p_2(k) y^k$. In \cref{secPRIO:relations_busy_period_transforms} we will see that the root $\rootx(y)$ is a PGF: $\rootx(y) = \sum_{k \ge 0} g_k y^k$ with $\{ g_k \}_{k \ge 0}$ the elements of the auxiliary matrix $G$ for which we have an exact expression, see \cref{secPRIO:matrix-geometric_and_matrix-analytic_methods} and \cref{propPRIO:first_passage_probabilities}. We derive a recursion for the probabilities $p_2(\cdot)$ by matching coefficients of the generating functions on both sides of \eqref{eqnPRIO:pgf_P(1,y)}. Substituting in \eqref{eqnPRIO:pgf_P(1,y)} the series expression $\rootx(y) = \sum_{k \ge 0} g_k y^k$ and $\PGF{1,y} = \sum_{k \ge 0} p_2(k) y^k$ gives
\begin{align}%
&\Bigl( \mu_2 g_0 - \mu_1 y + \sum_{l \ge 1} \bigl( \mu_2 g_l + (\mu_1 - \mu_2) g_{l - 1} \bigr) y^l \Bigr) \sum_{k \ge 0} p_2(k) y^k \notag \\
&= \frac{1 - \rho}{\rho_2} \mu_1 \Bigl( 1 - \sum_{l \ge 0} g_l y^l \Bigr). \label{eqnPRIO:pgf_P(1,y)_recursion_1}
\end{align}%
For convenience, define
\begin{align}%
a_l &\defi \begin{cases}%
\frac{1 - \rho}{\rho_2} \mu_1 (1 - g_0), & l = 0, \\
- \frac{1 - \rho}{\rho_2} \mu_1 g_l, & l \ge 1,
\end{cases}%
\\
b_l &\defi \begin{cases}%
\mu_2 g_0, & l = 0, \\
\mu_2 g_1 + (\mu_1 - \mu_2) g_0 - \mu_1, & l = 1, \\
\mu_2 g_l + (\mu_1 - \mu_2) g_{l - 1}, & l \ge 2,
\end{cases}%
\end{align}%
so that \eqref{eqnPRIO:pgf_P(1,y)_recursion_1} becomes
\begin{equation}%
\sum_{k \ge 0} \sum_{l \ge 0} b_l p_2(k) y^{k + l} = \sum_{m \ge 0} \sum_{n = 0}^m b_{m - n} p_2(n) y^m = \sum_{m \ge 0} a_m y^m
\end{equation}%
which, by coefficient matching, leads to the recursion
\begin{align}%
b_0 p_2(0) &= a_0, \label{eqnPRIO:pgf_p_2(0)}\\
b_0 p_2(m) &= a_m - \sum_{n = 0}^{m - 1} b_{m - n} p_2(n), \quad m \ge 1. \label{eqnPRIO:pgf_p_2(m)}
\end{align}%

The numerical inversion algorithm shown in \cref{algQTF:numerical_inversion_univariate_PGF} can also be used to determine the equilibrium probabilities. We show the equilibrium distribution for an example in \cref{figPRIO:equilibrium_distribution_class-2_jobs}. From the ratio $p_2(j + 1)/p_2(j)$ in \cref{figPRIO:equilibrium_distribution_class-2_jobs} it is clear that the distribution of the number of class-2 jobs is not geometric. \qedhere
%A standard inversion algorithm for univariate probability generating functions is \cite[Theorem~1]{Abate1992_Inversion_pgf}. The equilibrium probabilities follow from $p_2(j) = \bar{p}_2(j) - \bar{e}(j)$ where $\bar{p}_2(j)$ is the approximation and $\bar{e}(j)$ the error. The approximation is given by, for $j \ge 1$,
%%
%\begin{equation}%
%\bar{p}_2(j) = \frac{1}{2 j \radiusNI^j} \Bigl( \PGF{1,\radiusNI} + (-1)^j \PGF{1,-\radiusNI} + 2 \sum_{k = 1}^{j - 1} (-1)^k \RealPart{\PGF{1,\radiusNI \euler^{\complexunit \pi \frac{k}{j}}}} \Bigr) \label{eqnPRIO:pgf_numerical_inversion_univariate}
%\end{equation}%
%%
%with $0 < \radiusNI < 1$ a tunable parameter that controls the error as
%%
%\begin{equation}%
%|\bar{e}(j)| \le \frac{\radiusNI^{2j}}{1 - \radiusNI^{2j}}.
%\end{equation}%
%%
%With $\radiusNI = 10^{-d/(2j)}$ the approximation $\bar{p}_2(j)$ in \eqref{eqnPRIO:pgf_numerical_inversion_univariate} is accurate upto the $d$-th decimal. We show the equilibrium distribution for an example in \cref{figPRIO:equilibrium_distribution_class-2_jobs}. From the ratio $p_2(j + 1)/p_2(j)$ in \cref{figPRIO:equilibrium_distribution_class-2_jobs} it is clear that the distribution of the number of class-2 jobs is not geometric.

\begin{figure}
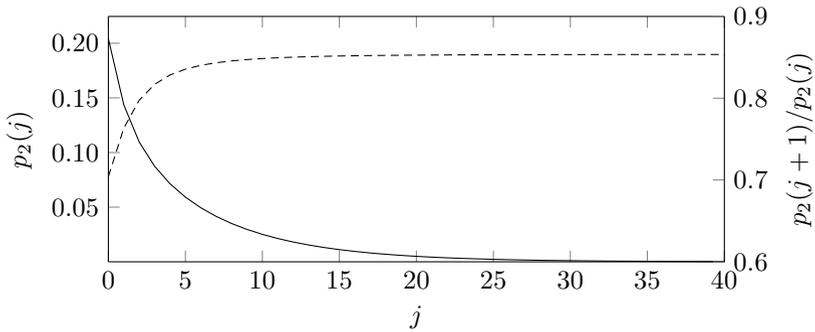
%
\centering%
\includestandalone{Chapters/PRIO/TikZFiles/equilibrium_distribution_class-2_jobs}%
\caption{Equilibrium distribution of the number of class-2 jobs determined using \protect\cref{algQTF:numerical_inversion_univariate_PGF}. Parameter values are $\la_1 = 1$, $\la_2 = 2$, $\mu_1 = 3$ and $\mu_2 = 4$. Algorithm setting is $d = 10$. Solid line belongs to the left $y$-axis.}%
\label{figPRIO:equilibrium_distribution_class-2_jobs}
\end{figure}%
\end{remark}%

%%%%%%%%%%%%%%%%%%%%%%%%%%%%%%%%%%%%%%%%%%%%%%%%%%%%%%%
%%%%%%%%%%%%%%%%%%%%%%%%%%%%%%%%%%%%%%%%%%%%%%%%%%%%%%%
%%%%%%%%%%%%%%%%%%%%% NEW SECTION %%%%%%%%%%%%%%%%%%%%%
%%%%%%%%%%%%%%%%%%%%%%%%%%%%%%%%%%%%%%%%%%%%%%%%%%%%%%%
%%%%%%%%%%%%%%%%%%%%%%%%%%%%%%%%%%%%%%%%%%%%%%%%%%%%%%%

\section{QBD approaches}%
\label{secPRIO:matrix-geometric_and_matrix-analytic_methods}%

The two-dimensional Markov process $\{ X(t) \}_{t \ge 0}$ is a QBD process with levels $\lvl{i} = \{ (i,0),(i,1),\ldots \}, ~ i \ge 0$ and an infinite number of phases per level. The infinite-dimensional transition matrix $Q$ can be partitioned into levels as
\begin{equation}%
Q = \begin{bmatrix}%
\La_{0}^{(0)} & \La_{1}  \\
\La_{-1}      & \La_{0}  & \La_1 \\
              & \La_{-1} & \La_0    & \La_1  \\
              &          & \La_{-1} & \La_0  & \La_1  \\
              &          &          & \ddots & \ddots & \ddots \\
\end{bmatrix},%
\end{equation}%
where $\La_{-1} = \mu_1 \I$, $\La_1 = \la_1 \I$, with $\I$ the infinite-dimensional identity matrix, and
\begin{equation}%
\La_{0} = - (\la + \mu_1) \I + \begin{bmatrix}%
0 & \la_2 \\
  & 0     & \la_2 \\
  &       & \ddots & \ddots
\end{bmatrix}
\end{equation}%
and
\begin{equation}%
\La_{0}^{(0)} = -\la \I + \begin{bmatrix}%
0     & \la_2 \\
\mu_2 & -\mu_2 & \la_2   \\
      & \mu_2  & -\mu_2 & \la_2 \\
      &        & \ddots & \ddots & \ddots
\end{bmatrix}.%
\end{equation}%

We first use the matrix-geometric method\endnote{Miller \cite{Miller1981_M-M-1_2-class_prio_matrix-geometric} was the first to model the single-server priority system with two classes as a QBD process. He exploits the structure of the transition rate diagram to determine the joint equilibrium distribution.} to determine the equilibrium distribution. Define the vectors $\pb_i \defi \begin{bmatrix} p(i,0) & p(i,1) & \cdots \end{bmatrix}$. The rate matrix $R$ satisfies the matrix-quadratic equation
\begin{equation}%
R^2 \La_{-1} + R\La_0 + \La_1 = 0 \label{eqnPRIO:mgm_R}
\end{equation}%
and the equilibrium probabilities follow from
\begin{equation}%
\pb_{i + 1} = \pb_i R, \quad i \ge 0, \label{eqnPRIO:mgm_p_i}
\end{equation}%
where the boundary probabilities are computed as
\begin{equation}%
\pb_0 \La_0^{(0)} + \pb_1 \La_{-1} = \zerob \quad \Leftrightarrow \quad \pb_0 \bigr( \La_0^{(0)} + R \La_{-1} \bigl) = \zerob \label{eqnPRIO:mgm_p_0}
\end{equation}%
and the normalization condition is $\pb_0 (\I - R)^{-1} \oneb = 1$, where $\oneb$ is a vector of ones.

At this point we can already obtain highly accurate approximations of the equilibrium distribution by truncating all matrices in \eqref{eqnPRIO:mgm_R} to size $n \times n$ with $n$ large and use successive substitutions to determine $R$ (see \cref{algQBD:matrix-geometric_successive_substitutions}). However, we can do better than that by exploiting the specific structure of the transition rate diagram.

Since the transitions within levels $i \ge 1$ are strictly upward in the vertical direction and the transition rate diagram is homogeneous, we already know from the probabilistic interpretation of the elements of the rate matrix (see \cref{secQBD:matrix-geometric_method}) that
\begin{equation}%
R = \begin{bmatrix}%
r_0 & r_1 & r_2 & \cdots \\
    & r_0 & r_1 & \cdots \\
    &     & r_0 & \cdots \\
    &     &     & \ddots
\end{bmatrix}.%
\end{equation}%
It is precisely this structure that makes it possible to solve for the elements of $R$ using a recursive procedure. That is, component-wise the equations \eqref{eqnPRIO:mgm_R} read
\begin{align}%
\mu_1 r_0^2 - (\la + \mu_1)r_0 + \la_1 &= 0, \label{eqnPRIO:mgm_r_0} \\
\mu_1 \sum_{l = 0}^{k} r_{k - l} r_l - (\la + \mu_1) r_k + \la_2 r_{k - 1}  &= 0, \quad k \ge 1. \label{eqnPRIO:mgm_r_k}
\end{align}%
Since $R$ is the minimal non-negative solution of \eqref{eqnPRIO:mgm_R} we know that the solution of \eqref{eqnPRIO:mgm_r_0} is given by
\begin{equation}%
r_0 = \frac{\la_1 + \mu_1 + \la_2 - \sqrt{(\la_1 + \mu_1 + \la_2)^2 - 4 \la_1 \mu_1}}{2 \mu_1}. \label{eqnPRIO:mgm_r_0_explicit}
\end{equation}%
Substituting the solution for $r_0$ into \eqref{eqnPRIO:mgm_r_k} gives the recursion
\begin{equation}%
r_k = \frac{\la_2 r_{k - 1} + \mu_1 \sum_{l = 1}^{k - 1} r_{k - l} r_l}{\sqrt{(\la + \mu_1)^2 - 4 \la_1 \mu_1}}, \quad k \ge 0, \label{eqnPRIO:mgm_r_k_explicit}
\end{equation}%
where the empty sum $\sum_{j = 1}^0$ is zero.

The boundary probabilities can also be determined recursively due to the specific form of $R$ and $p(0,0) = 1 - \rho$. In particular,
\begin{align}%
\mu_2 p(0,1) &= (\la - \mu_1 r_0) p(0,0) = (\la - \mu_1 r_0) (1 - \rho), \label{eqnPRIO:mgm_p(0,1)} \\
\mu_2 p(0,j + 1) &= \Bigl( (\la + \mu_2)p(0,j) - \la_2 p(0,j - 1) \notag \\
&\hspace{0.10\linewidth} - \mu_1 \sum_{k = 0}^j p(0,j - k) r_k \Bigr), \quad j \ge 1. \label{eqnPRIO:mgm_p(0,j)}
\end{align}%
A similar simplification as for $p(0,j)$ is obtained for the equilibrium probabilities $p(i,j)$:
\begin{equation}%
p(i + 1,j) = \sum_{k = 0}^j p(i,j - k) r_k. \label{eqnPRIO:mgm_p(i,j)}
\end{equation}%

The numerical implementation of the matrix-geometric method is explained in \cref{algPRIO:matrix-geometric_method}. The positive integers $i_{\textup{max}}$ and $j_{\textup{max}}$ determine the subset of the state space $\statespace' \defi \{ (i,j) \in \statespace : 0 \le i \le i_{\textup{max}}, ~ 0 \le j \le j_{\textup{max}} \}$ for which the equilibrium probabilities are determined exactly.

\begin{algorithm}%
\caption{Numerical implementation matrix-geometric method}%
\label{algPRIO:matrix-geometric_method}%
\begin{algorithmic}[1]%
\State Pick positive integers $i_{\textup{max}}$ and $j_{\textup{max}}$
\State Set $r_0$ according to \eqref{eqnPRIO:mgm_r_0_explicit}
\For{$k = 1,2,\ldots,j_{\textup{max}}$}
    \State Determine $r_k$ from \eqref{eqnPRIO:mgm_r_k_explicit}
\EndFor
\State Set $p(0,0) = 1 - \rho$ and determine $p(0,1)$ from \eqref{eqnPRIO:mgm_p(0,1)}
\For{$j = 1,2,\ldots,j_{\textup{max}} - 1$}
    \State Determine $p(0,j + 1)$ from \eqref{eqnPRIO:mgm_p(0,j)}
\EndFor
\For{$i = 0,1,\ldots,i_{\textup{max}} - 1$}
    \For{$j = 0,1,\ldots,j_{\textup{max}}$}
        \State Determine $p(i + 1,j)$ from \eqref{eqnPRIO:mgm_p(i,j)}
    \EndFor
\EndFor
\end{algorithmic}%
\end{algorithm}%

For a QBD process, the matrix-geometric and matrix-analytic method are nearly identical. For that reason we do not describe how to determine the auxiliary matrix $G$ of the matrix-analytic method. Instead, we focus on the probabilistic interpretation of the elements of the $G$ (see \cref{secQBD:matrix-analytic_method}) to immediately derive a recursion for the equilibrium probabilities.\endnote{Sleptchenko et al \cite{Sleptchenko2015_M-M-1_priority} developed the method that exploits the transition structure to determine the elements $\{ g_k \}_{k \ge 0}$ and use these elements to write the excursions in two ways as we have done in \cref{secPRIO:matrix-geometric_and_matrix-analytic_methods} for the matrix-analytic method.}

The auxiliary matrix $G$ is given by
\begin{equation}%
G = \begin{bmatrix}%
g_0 & g_1 & g_2 & \cdots \\
    & g_0 & g_1 & \cdots \\
    &     & g_0 & \cdots \\
    &     &     & \ddots
\end{bmatrix}.%
\end{equation}%
The element $g_k$ can be interpreted as a \textit{first passage probability}: it is the probability that, starting at level $i \ge 1$ in state $(i,j)$, the first passage to level $i - 1$ occurs in state $(i - 1,j + k)$. The first passage probabilities do not depend on the starting state due to the homogeneous transition behavior in the interior of the state space. The $\{ g_k \}_{k \ge 0}$ are obtained from a recursion relation similar to the one for $\{ r_k \}_{k \ge 0}$ and given by
\begin{align}%
g_0 &= \frac{\mu_1}{\la + \mu_1} + \frac{\la_1}{\la + \mu_1} g_0^2, \label{eqnPRIO:fpp_g_0}\\
g_k &= \frac{\la_2}{\la + \mu_1} g_{k - 1} + \frac{\la_1}{\la + \mu_1} \sum_{l = 0}^k g_{k - l} g_l, \quad k \ge 1. \label{eqnPRIO:fpp_g_k}
\end{align}%

We now use the first passage probabilities to derive an expression for the equilibrium probabilities in the interior of the state space. Let an \textit{excursion} refer to a sample path of the Markov process that starts in level $\lvl{i}$, reaches levels higher than $\lvl{i}$ and ends on first passage to level $\lvl{i}$. The number of excursions per time unit that ends in state $(i,j)$ is $p(i + 1,j) \mu_1$. Alternatively, this rate is also given by the number of excursions starting from level $\lvl{i}$ per time unit that end in state $(i,j)$. The number of excursions per time unit that starts in state $(i,j - k)$ is $p(i,j - k) \la_1$, a fraction $g_k$ of which ends its excursion in $(i,j)$. Since these two rates are equal, we obtain the recursion
\begin{equation}%
\mu_1 p(i + 1,j) = \la_1 \sum_{k = 0}^j p(i,j - k) g_k, \quad i,j \ge 0. \label{eqnPRIO:fpp_p(i,j)}
\end{equation}%
It remains to determine the equilibrium probabilities on the vertical boundary. To that end, we censor the Markov process to $\lvl{0}$. This leads to the transition rate diagram in \cref{figPRIO:transition_rate_diagram_single-server_prio_censored}. For the censored Markov process we can simply equate the number of transitions per time unit that enter and leave the set $\{ (0,0),(0,1),\ldots,(0,j) \}$, which yields
\begin{equation}%
\mu_2 p(0,j + 1) = \la_2 p(0,j) + \la_1 \sum_{k = 0}^j p(0,j - k) \Bigl( 1 - \sum_{l = 0}^{k} g_l \Bigr). \label{eqnPRIO:fpp_p(0,j)}
\end{equation}%
Starting from $p(0,0) = 1 - \rho$, all equilibrium probabilities can be obtained through \eqref{eqnPRIO:fpp_p(i,j)}--\eqref{eqnPRIO:fpp_p(0,j)}.

\begin{figure}
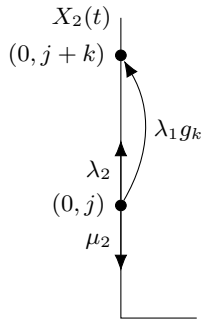
%
\centering%
\includestandalone{Chapters/PRIO/TikZFiles/transition_rate_diagram_single-server_prio_censored}%
\caption{Transition rate diagram of the Markov process censored on $\lvl{0}$.}%
\label{figPRIO:transition_rate_diagram_single-server_prio_censored}
\end{figure}%

\begin{remark}[Alternative levels]%
An alternative choice for a level is the vertically aligned set of states $\lvl{j} = \{ (0,j),(1,j),\ldots \}, ~ j \ge 0$. We write variables with a bar to reflect that they belong to the alternative choice for the levels.

The transition matrix $\bar{Q}$ can be partitioned into these levels as
\begin{equation}%
\bar{Q} = \begin{bmatrix}%
\bar{\La}_{0}^{(0)} & \bar{\La}_{1}  \\
\bar{\La}_{-1}      & \bar{\La}_{0}  & \bar{\La}_1    \\
                    & \bar{\La}_{-1} & \bar{\La}_0    & \bar{\La}_1 \\
                    &                & \bar{\La}_{-1} & \bar{\La}_0 & \bar{\La}_1  \\
                    &                &                & \ddots      & \ddots & \ddots \\
\end{bmatrix}.%
\end{equation}%
The matrix $\bar{\La}_{-1}$ has zeroes everywhere except $(\bar{\La}_{-1})_{0,0} = \mu_2$. The Markov process can only go from level $\lvl{j + 1}$ to level $\lvl{j}$ by using the transition from $(0,j + 1)$ to $(0,j)$. Due to the probabilistic interpretation of the elements of the auxiliary matrix $\bar{G}$ we can immediately write that
\begin{equation}%
\bar{G} = \begin{bmatrix}%
1 & 0 & \cdots \\
1 & 0 & \cdots \\
\vdots & \vdots & \ddots
\end{bmatrix}.%
\end{equation}%

Denote the vectors $\bar{\pb}_j \defi \begin{bmatrix} p(0,j) & p(1,j) & \cdots \end{bmatrix}$. By censoring the Markov process to the set of states $\lvl{0} \cup \lvl{1} \cup \cdots \cup \lvl{j}$ we can write the balance equations for level $\lvl{j}$ as
\begin{equation}%
\bar{\pb}_{j - 1} \bar{\La}_1 + \bar{\pb}_j \bigl( \bar{\La}_0 + \bar{\La}_1 \bar{G} \bigr) = \zerob, \quad j \ge 1,
\end{equation}%
Finally, censoring the Markov process to level $\lvl{0}$ shows that
\begin{equation}%
\bar{\pb}_{0} \bigl( \bar{\La}_0^{(0)} + \bar{\La}_1 \bar{G} \bigr) = \zerob.
\end{equation}%
These balance equations can be solved recursively in essentially the same way as the difference equations approach of \cref{secPRIO:difference_equations} solves the balance equations.
\end{remark}%

%%%%%%%%%%%%%%%%%%%%%%%%%%%%%%%%%%%%%%%%%%%%%%%%%%%%%%%
%%%%%%%%%%%%%%%%%%%%%%%%%%%%%%%%%%%%%%%%%%%%%%%%%%%%%%%
%%%%%%%%%%%%%%%%%%%%% NEW SECTION %%%%%%%%%%%%%%%%%%%%%
%%%%%%%%%%%%%%%%%%%%%%%%%%%%%%%%%%%%%%%%%%%%%%%%%%%%%%%
%%%%%%%%%%%%%%%%%%%%%%%%%%%%%%%%%%%%%%%%%%%%%%%%%%%%%%%

\section{Busy period transforms}%
\label{secPRIO:relations_busy_period_transforms}%

Key elements of both the generating function approach and the QBD approach are related to busy periods in a single-server system. We explain what a busy period is, derive its LST and mention where busy periods play a role in both approaches.

A busy period in a single-server system is a length of time that starts when a first job arrives to an empty system and ends when a departing job leaves the system empty again. Let $B_{\la,\mu}$ denote the length of a busy period in an $M/M/1$ queue with arrival rate $\la$ and service rate $\mu$. In \cref{exBD:MM1_busy_period} we have proven that the LST of $B_{\la,\mu}$---now denoted by $\LST{\la,\mu}{\LSTarg}$---is given by
\begin{equation}%
\LST{\la,\mu}{\LSTarg} = \frac{\la + \mu + \LSTarg - \sqrt{(\la + \mu + \LSTarg)^2 - 4\la\mu}}{2\la}, \quad \RealPart{\LSTarg} > 0.
\end{equation}%

We now mention a few relations between the busy period $B_{\la,\mu}$ and its transform $\LST{\la,\mu}{\LSTarg}$ and key elements of the approaches used in this chapter to determine the equilibrium distribution.

The unique root $\rootx(y)$ of the generating function approach can be expressed in terms of the busy period transform as
\begin{equation}%
\rootx(y) = \LST{\la_1,\mu_1}{\la_2(1 - y)}.
\end{equation}%
The root has an interpretation in terms of the PGF of the number of class-2 jobs $A$ that arrives during a busy period of class-1 jobs. Condition on the length $S_1$ of the service of the first class-1 job to obtain
\begin{equation}%
\E{y^A} = \int_0^\infty \E{y^A \mid S_1 = t} f_{S_1}(t) \, \dinf t. \label{eqnPRIO:pgf_number_class-2_arrivals_BP_class-1_1}
\end{equation}%
A number $K_2$ of class-2 jobs joins the queue during this first service, but the busy period of class-1 jobs might not have ended yet. In particular, during this first service a number $K_1$ of class-1 jobs has joined the queue and each class-1 job induces a busy period of class-1 jobs that generates a number of class-2 arrivals, which is statistically identical to $A$. Note that $K_1 \sim \Poisson{\la_1 t}$ and $K_2 \sim \Poisson{\la_2 t}$. By conditioning on the number of class-1 and class-2 arrivals in the interval $[0,t]$, we see that
\begin{align}%
\E{y^A} &= \int_0^\infty \sum_{i \ge 0} \sum_{j \ge 0} \E{y^A \mid S_1 = t, K_1 = i, K_2 = j} \notag \\
& \hspace{0.2\linewidth} \cdot f_{S_1}(t) \Prob{K_1 = i} \Prob{K_2 = j} \, \dinf t \notag \\
&= \int_0^\infty \sum_{i \ge 0} \sum_{j \ge 0} \E{y^{j + A^{(1)} + \cdots + A^{(i)}}} \notag \\
& \hspace{0.2\linewidth} \cdot f_{S_1}(t) \Prob{K_1 = i} \Prob{K_2 = j} \, \dinf t, \label{eqnPRIO:pgf_number_class-2_arrivals_BP_class-1_2}
\end{align}%
where $A^{(k)}$ denotes the number of class-2 arrivals during the class-1 busy period started by the $k$-th class-1 job during the service of the first class-1 job. The random variables $A^{(k)}$ are i.i.d. We can now substitute the probability density and mass functions of $S_1$, $K_1$ and $K_2$ to obtain
\begin{align}%
\E{y^A} &= \int_0^\infty \sum_{i \ge 0} \sum_{j \ge 0} \E{y^A}^i y^j \mu_1 \euler^{-\mu_1 t} \frac{(\la_1 t)^i}{i!} \euler^{-\la_1 t} \frac{(\la_2 t)^j}{j!} \euler^{-\la_2 t} \, \dinf t \notag \\
&= \mu_1 \int_0^\infty \euler^{-t(\la_1(1 - \E{y^A}) + \la_2(1 - y) + \mu_1)} \, \dinf t \notag \\
&= \frac{\mu_1}{\la_1(1 - \E{y^A}) + \la_2(1 - y) + \mu_1}. \label{eqnPRIO:BP_E[y^A]}
\end{align}%
From \eqref{eqnPRIO:pgf_kernel}, we see that $\rootx(y)$ satisfies the exact same equation as \eqref{eqnPRIO:BP_E[y^A]} and since $|\E{y^A}| < 1$ we conclude that $\rootx(y) = \E{y^A}$.

The first passage probabilities $g_k$ can be given in terms of a busy period $B_{\la_1,\mu_1}$ and a Poisson process $\{ N_{\la_2}(t) \}_{t \ge 0}$ with rate $\la_2$. By examining the transition rate diagram in \cref{figPRIO:transition_rate_diagram_single-server_prio}, we see that
\begin{equation}%
g_k = \Prob{ N_{\la_2}(B_{\la_1,\mu_1}) = k},
\end{equation}%
which is the probability that during a class-1 busy period, exactly $k$ class-2 jobs arrive. The PGF of $\{ g_k \}_{k \ge 0}$ is given by
\begin{align}%
\sum_{k \ge 0} g_k y^k &= \sum_{k \ge 0} \Prob{ N_{\la_2}(B_{\la_1,\mu_1}) = k} y^k \notag \\
&= \sum_{k \ge 0} \Efxd{ \ind{N_{\la_2}(B_{\la_1,\mu_1}) = k} y^k } \notag \\
&= \Efxd{ y^{N_{\la_2}(B_{\la_1,\mu_1})}}.
\end{align}%
Conditioning on the length of the busy period and using the PGF of a Poisson distribution with parameter $\la_2 t$ yields
\begin{align}%
\sum_{k \ge 0} g_k y^k &= \int_0^\infty \Efxd{ y^{N_{\la_2}(B_{\la_1,\mu_1})} \mid B_{\la_1,\mu_1} = t} f_{B_{\la_1,\mu_1}}(t) \, \dinf t \notag \\
&= \int_0^\infty \euler^{-\la_2(1 - y)t} f_{B_{\la_1,\mu_1}}(t) \, \dinf t \notag \\
&= \LST{\la_1,\mu_1}{\la_2(1 - y)},
\end{align}%
which shows that $\sum_{k \ge 0} g_k y^k = \E{y^A} = \rootx(y)$.

\begin{proposition}\label{propPRIO:first_passage_probabilities}%
The first passage probabilities are explicitly given by $g_0 = \LST{\la_1,\mu_1}{\la_2}$ and for $k \ge 1$,
\begin{equation}%
g_k = \gamma_1^k \LST{\la_1,\mu_1}{\la_2} \sum_{l = 0}^{k - 1} C_l \binom{k - 1 + l}{k - 1 - l} \gamma_2^l,
\end{equation}%
where $C_l \defi 1/(l + 1) \binom{2l}{l}$ are the Catalan numbers and
\begin{align}%
\gamma_1 &= \frac{\la_2}{\la_1(1 - 2 \LST{\la_1,\mu_1}{\la_2}) + \mu_1 + \la_2}, \\
\gamma_2 &= \frac{\la_1 \LST{\la_1,\mu_1}{\la_2}}{\la_1(1 - 2 \LST{\la_1,\mu_1}{\la_2}) + \mu_1 + \la_2}.
\end{align}%
\end{proposition}%

\begin{proof}%
We prove the claim by verifying that $\sum_{k \ge 0} g_k y^k = \LST{\la_1,\mu_1}{\la_2(1 - y)}$. For now, abbreviate $\LST{\la_1,\mu_1}{\la_2}$ as $L$.

We have
\begin{align}%
\sum_{k \ge 0} g_k y^k &= \LSTvar \Bigl( 1 + \sum_{k \ge 1} (\gamma_1 y)^k \sum_{l = 0}^{k - 1} C_l \binom{k - 1 + l}{k - 1 - l} \gamma_2^l \Bigr) \notag \\
&= \LSTvar \Bigl( 1 + \gamma_1 y \sum_{k \ge 0} (\gamma_1 y)^k \sum_{l = 0}^{k} C_l \binom{k + l}{k - l} \gamma_2^l \Bigr). \label{eqnPRIO:first_passage_probabilities_proof_1}
\end{align}%
Interchanging the two summations gives
\begin{align}%
\sum_{k \ge 0} (\gamma_1 y)^k \sum_{l = 0}^{k} C_l \binom{k + l}{k - l} \gamma_2^l &= \sum_{l \ge 0} C_l \gamma_2^l \sum_{k \ge l} \binom{k + l}{k - l} (\gamma_1 y)^k \notag \\
&= \sum_{l \ge 0} C_l (\gamma_1 \gamma_2 y)^l \sum_{k \ge 0} \binom{k + 2l}{k} (\gamma_1 y)^k.
\end{align}%
From the negative binomial distribution we know that the generating function of the binomial coefficient is
\begin{equation}%
\sum_{k \ge 0} \binom{k + K}{k} \PGFarg^k = \frac{1}{(1 - \PGFarg)^{K + 1}},
\end{equation}%
so that
\begin{equation}%
\sum_{l \ge 0} C_l (\gamma_1 \gamma_2 y)^l \sum_{k \ge 0} \binom{k + 2l}{k} (\gamma_1 y)^l = \frac{1}{1 - \gamma_1 y} \sum_{l \ge 0} C_l \Bigl( \frac{\gamma_1 \gamma_2 y}{(1 - \gamma_1 y)^2} \Bigr)^l.
\end{equation}%
Now use the generating function of the Catalan numbers
\begin{equation}%
\sum_{k \ge 0} C_k \PGFarg^k = \frac{1 - \sqrt{1 - 4\PGFarg}}{2\PGFarg}.
\end{equation}%
to get
\begin{equation}%
\frac{1}{1 - \gamma_1 y} \sum_{l \ge 0} C_l \Bigl( \frac{\gamma_1 \gamma_2 y}{(1 - \gamma_1 y)^2} \Bigr)^l = \frac{1 - \sqrt{1 - 4\frac{\gamma_1 \gamma_2 y}{(1 - \gamma_1 y)^2}}}{2\frac{\gamma_1 \gamma_2 y}{1 - \gamma_1 y}}
\end{equation}%
Substituting this back into \eqref{eqnPRIO:first_passage_probabilities_proof_1} yields
\begin{align}%
\sum_{k \ge 0} g_k y^k &= \LSTvar \Bigl( 1 + \frac{1 - \sqrt{1 - 4\frac{\gamma_1 \gamma_2 y}{(1 - \gamma_1 y)^2}}}{2\frac{\gamma_2}{1 - \gamma_1 y}} \Bigr) \notag \\
&= \LSTvar \frac{2 \frac{\gamma_2}{1 - \gamma_1 y} + 1 - \sqrt{1 - 4\frac{\gamma_1 \gamma_2 y}{(1 - \gamma_1 y)^2}}}{2\frac{\gamma_2}{1 - \gamma_1 y}}. \label{eqnPRIO:first_passage_probabilities_proof_2}
\end{align}%
Substituting
\begin{align}%
\frac{\gamma_1 \gamma_2 y}{(1 - \gamma_1 y)^2} &= \frac{\la_1 \la_2 \LSTvar y}{(\la_1(1 - 2 \LSTvar) + \mu_1 + \la_2(1 - y))^2}, \\
\frac{\gamma_2}{1 - \gamma_1 y} &= \frac{\la_1 \LSTvar}{\la_1(1 - 2 \LSTvar) + \mu_1 + \la_2(1 - y)},
\end{align}%
into \eqref{eqnPRIO:first_passage_probabilities_proof_2}, multiplying the numerator and denominator of \eqref{eqnPRIO:first_passage_probabilities_proof_2} by $\la_1(1 - 2 \LSTvar) + \mu_1 + \la_2(1 - y)$ and recognizing that $\LSTvar$ satisfies $\la_1 \LSTvar^2 - (\la + \mu_1) \LSTvar + \mu_1 = 0$ finally proves the claim.
\end{proof}%

%%%%%%%%%%%%%%%%%%%%%%%%%%%%%%%%%%%%%%%%%%%%%%%%%%%%%%%
%%%%%%%%%%%%%%%%%%%%%%%%%%%%%%%%%%%%%%%%%%%%%%%%%%%%%%%
%%%%%%%%%%%%%%%%%%%%% NEW SECTION %%%%%%%%%%%%%%%%%%%%%
%%%%%%%%%%%%%%%%%%%%%%%%%%%%%%%%%%%%%%%%%%%%%%%%%%%%%%%
%%%%%%%%%%%%%%%%%%%%%%%%%%%%%%%%%%%%%%%%%%%%%%%%%%%%%%%

\section{Takeaways}%
\label{secPRIO:what_have_we_learned}%

The Markov process associated with the single-server priority system has no downward transitions in the interior of the state space. This structure allowed for a simple solution using the generating function approach and while modeling the Markov process as a QBD process. A Markov process with a structure in which there are no upward transitions is amenable to the same solution approaches. The single-server priority system is one of many models that possesses this structure; a few others can be found in \cite{Gandhi2014_RRR,Gandhi2010_Server_farms_setup_costs,Selen2016_Snowball,Houdt2011_Triangular_R}.

Due to the upward structure the balance equations could be solved recursively by treating them as second-order difference equations. The balance equations for $i > 0$ and $j = 0$ are homogeneous difference equations and were easily solved by substituting a product-form solution. The balance equations for $i,j > 0$ are nonhomogeneous difference equations, where the constant term is the rate at which the process enters the state from the state directly below. The nonhomogeneous difference equations could also be solved and the final expressions for the equilibrium probabilities involve coefficients that needed to be calculated recursively.

For the generating function approach the upward structure meant that the functional equation for $\PGF{x,y}$ did not involve $\PGF{x,0}$ associated with the equilibrium probabilities of the states on the horizontal axis. This allowed for a direct determination of $\PGF{0,y}$ as a function of the root $\rootx(y)$ and ultimately led to an explicit expression for $\PGF{x,y}$.

In terms of the QBD approach the upward structure ensured that the infinite-dimensional rate matrix $R$ and auxiliary matrix $G$ were upper triangular. Both $R$ and $G$ satisfied a matrix-quadratic equation, that, due to the upper triangular structure, could be solved recursively. The upward structure was used once more to derive recursions for the equilibrium probabilities on the vertical boundary and in the interior of the state space.

The transition behavior in the interior of the state space has some additional structure: on top of being strictly upward in the vertical direction, the transition behavior in the horizontal direction mimics the transition behavior of an $M/M/1$ queue with arrival rate $\la_1$ and service rate $\mu_1$. The relations between busy periods of an $M/M/1$ queue and key elements of the generating function approach and the QBD approach came as no surprise.

The approaches we saw in this chapter are not restricted to Markov process with no downward transitions in the interior of the state space. Specifically, the approaches work whenever there are no downward, upward, leftward or rightward transitions, see \cref{figPRIO:possible_transitions_interior_state_space} and \cite{Leeuwaarden2009_Lattice_path_counting,Leeuwaarden2006_Explicit_R}. For example, when there are no upward transitions---the second case in \cref{figPRIO:possible_transitions_interior_state_space}---$\PGF{x,0}$ appears in the functional equation \eqref{eqnPRIO:pgf_functional} instead of $\PGF{0,y}$ and the matrices $R$ and $G$ are lower triangular instead of upper triangular. By swapping the two coordinates the third case with no rightward transitions in \cref{figPRIO:possible_transitions_interior_state_space} reduces to the first case with no upward transitions and the fourth case with the leftward transitions reduces to the second case with no downward transitions.

\begin{figure}
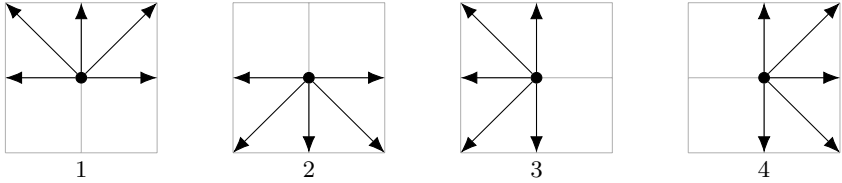
%
\centering%
\includestandalone{Chapters/PRIO/TikZFiles/possible_transitions_interior_state_space}%
\caption{Types of transitions in the interior of the state space for which the approaches of this chapter are suited to solve for the equilibrium distribution.}%
\label{figPRIO:possible_transitions_interior_state_space}
\end{figure}%

%%%%%%%%%%%%%%%%%%%%%%%%%%%%%%%%%%%%%%%%%%%%%%%%%%%%%%%
%%%%%%%%%%%%%%%%%%%%%%%%%%%%%%%%%%%%%%%%%%%%%%%%%%%%%%%
%%%%%%%%%%%%%%%%%%%%%%%% NOTES %%%%%%%%%%%%%%%%%%%%%%%%
%%%%%%%%%%%%%%%%%%%%%%%%%%%%%%%%%%%%%%%%%%%%%%%%%%%%%%%
%%%%%%%%%%%%%%%%%%%%%%%%%%%%%%%%%%%%%%%%%%%%%%%%%%%%%%%

%\theendnotes%
%\setcounter{endnote}{0}
\printendnotes% %

% Checked points 1-7 and a-i
\chapter{Gated systems}%
\label{ch:gated_single-server}%

In this chapter we consider an exponential single-server queueing system where access to the system is regulated by a gate. Jobs arrive according to a Poisson process and first have to wait behind this gate. Whenever there are no jobs left in the system, the gate opens and all waiting jobs are transferred to the system without further delay. The gate closes immediately after the transfer and the server starts service. If there are no jobs in the system nor behind the gate, then the gate remains open until a job arrives. An arriving job is then immediately transferred to the system, the gate closes and service starts. Notice that the system cannot be empty unless there are no jobs behind the gate. We are interested in the joint distribution of the number of jobs behind the gate and in the system.

The system can be described as a two-dimensional Markov process with as dimensions the number of jobs behind the gate and in the system. We shall determine the equilibrium distribution of this Markov process using three different approaches. The first approach casts the balance equations into the generating function domain and determines the generating function of the joint equilibrium distribution using an iterative approach.\endnote{Rietman and Resing \cite{Rietman2004_Gated_MG1} study the gated system with general service times using the generating function approach. In the exponential case that we consider in this chapter, we are able to develop an explicit expression for the equilibrium distribution in terms of an infinite sum of product forms.} The second approach uses the matrix-geometric method and exploits the downward structure in the interior of the state space to explicitly determine the elements of the rate matrix $R$. However, the second approach strands here and numerical approximations are required to determine the equilibrium probabilities. The third approach is called the compensation approach\endnote{Resing and Rietman \cite{Resing2004_Gated} introduced the gated single-server system and determined the equilibrium distribution using the compensation approach of Adan, Wessels and Zijm \cite{Adan1990_JSQ_symmetric}.} and exploits the fact that a product-form solution satisfies the balance equations for the states in the interior of the state space. These product-form solutions are linearly combined to also satisfy the remaining balance equations. The first and third approach both lead to infinite sum expressions for the equilibrium probabilities.

%%%%%%%%%%%%%%%%%%%%%%%%%%%%%%%%%%%%%%%%%%%%%%%%%%%%%%%
%%%%%%%%%%%%%%%%%%%%%%%%%%%%%%%%%%%%%%%%%%%%%%%%%%%%%%%
%%%%%%%%%%%%%%%%%%%%% NEW SECTION %%%%%%%%%%%%%%%%%%%%%
%%%%%%%%%%%%%%%%%%%%%%%%%%%%%%%%%%%%%%%%%%%%%%%%%%%%%%%
%%%%%%%%%%%%%%%%%%%%%%%%%%%%%%%%%%%%%%%%%%%%%%%%%%%%%%%

\section{Model description and balance equations}%
\label{secGTD:model_description}%

Jobs arrive according to a Poisson process with rate $\la$. Each job requires an exponentially distributed service time with rate $\mu$. Denote by $\rho \defi \la/\mu$ the amount of work brought into the system per time unit.

Let $X_1(t)$ be the number of jobs behind the gate at time $t$ and let $X_2(t)$ be the number of jobs in the system at time $t$. Further, denote the state of the system by $X(t) \defi (X_1(t),X_2(t))$. Then $\{ X(t) \}_{t \ge 0}$ is a Markov process with state space
\begin{equation}%
\statespace \defi \{ (i,j) : i \in \Nat_0, ~ j \in \Nat \} \cup \{ (0,0) \}.
\end{equation}%
It is apparent from the transition rate diagram in \cref{figGTD:transition_rate_diagram_gated_single-server} that the state space is irreducible. To guarantee positive recurrence and the existence of the equilibrium distribution we assume that $\rho < 1$. Let $p(i,j)$ denote the equilibrium probability of being in state $(i,j)$.

\begin{figure}
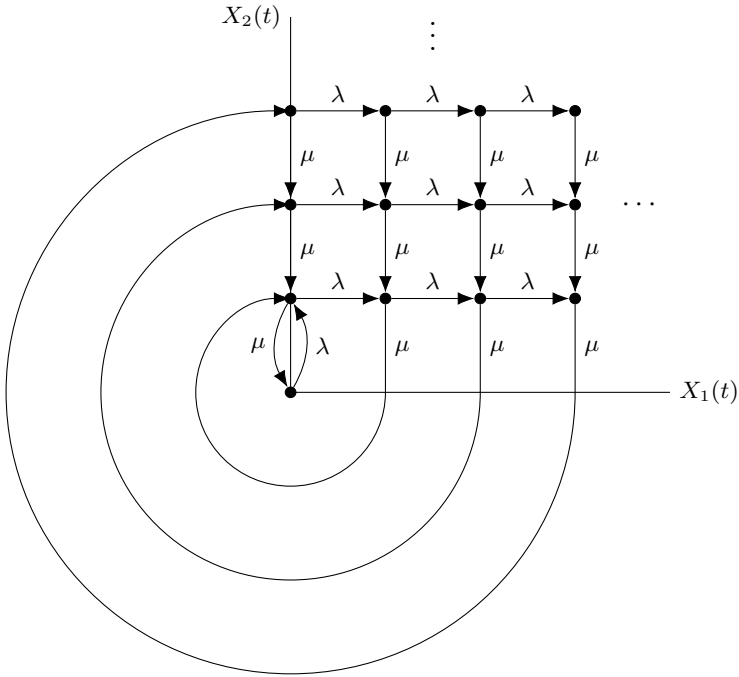
%
\centering%
\includestandalone{Chapters/GTD/TikZFiles/transition_rate_diagram_gated_single-server}%
\caption{Transition rate diagram of the Markov process associated with the gated single-server system.}%
\label{figGTD:transition_rate_diagram_gated_single-server}%
\end{figure}%

The balance equations for the interior of the state space are given by
\begin{equation}%
(\la + \mu) p(i,j) = \la p(i - 1,j) + \mu p(i,j + 1), \quad i,j \ge 1. \label{eqnGTD:balance_equations_interior}
\end{equation}%
For the vertical axis we have
\begin{align}%
(\la + \mu) p(0,1) &= \la p(0,0) + \mu p(0,2) + \mu p(1,1), \label{eqnGTD:balance_equations_vertical_p(0,1)} \\
(\la + \mu) p(0,j) &= \mu p(0,j + 1) + \mu p(j,1), \quad j \ge 2, \label{eqnGTD:balance_equations_vertical_p(0,j)}
\end{align}%
and the balance equation at the origin is
\begin{equation}%
\la p(0,0) = \mu p(0,1). \label{eqnGTD:balance_equations_origin}
\end{equation}%
Combining \eqref{eqnGTD:balance_equations_vertical_p(0,1)} with \eqref{eqnGTD:balance_equations_origin} gives
\begin{equation}%
\la p(0,1) = \mu p(0,2) + \mu p(1,1). \label{eqnGTD:balance_equations_vertical_p(0,1)_changed}
\end{equation}%

Observe that $\{ X_1(t) + X_2(t) \}_{t \ge 0}$ is also a Markov process. More specifically, it is the Markov process associated with an $M/M/1$ queue with arrival rate $\la$ and service rate $\mu$ and equilibrium probabilities $p(k)$, where
\begin{equation}%
p(k) = \sum_{i + j = k} p(i,j) = (1 - \rho) \rho^k. \label{eqnGTD:MM1_equilibrium_probabilities}
\end{equation}%
Since an empty system can only occur when there are no jobs behind the gate we clearly have $p(0,0) = 1 - \rho$ and therefore
\begin{equation}%
\sum_{i \ge 0} \sum_{j \ge 1} p(i,j) = \rho. \label{eqnGTD:normalization_condition}
\end{equation}%
%

%%%%%%%%%%%%%%%%%%%%%%%%%%%%%%%%%%%%%%%%%%%%%%%%%%%%%%%
%%%%%%%%%%%%%%%%%%%%%%%%%%%%%%%%%%%%%%%%%%%%%%%%%%%%%%%
%%%%%%%%%%%%%%%%%%%%% NEW SECTION %%%%%%%%%%%%%%%%%%%%%
%%%%%%%%%%%%%%%%%%%%%%%%%%%%%%%%%%%%%%%%%%%%%%%%%%%%%%%
%%%%%%%%%%%%%%%%%%%%%%%%%%%%%%%%%%%%%%%%%%%%%%%%%%%%%%%

\section{Generating function approach}%
\label{secGTD:pgf_approach}%

Define the bivariate generating function
\begin{equation}%
\PGF{x,y} \defi \sum_{i \ge 0} \sum_{j \ge 1} p(i,j) x^i y^{j - 1}, \quad |x| \le 1, ~ |y| \le 1. \label{eqnGTD:pgf_P(x,y)_definition}
\end{equation}%
Since $\PGF{1,1} = \rho < 1$, $\PGF{x,y}$ is not a \textit{probability} generating function. From \eqref{eqnGTD:balance_equations_origin} or \eqref{eqnGTD:MM1_equilibrium_probabilities} we derive that $\PGF{0,0} = p(0,1) = (1 - \rho) \rho$.

We shall now perform a series of operations on the balance equations to obtain an expression for $\PGF{x,y}$. Multiply both sides of \eqref{eqnGTD:balance_equations_interior} by $x^i y^{i - 1}$ and sum over all $i,j \ge 1$ to obtain
\begin{align}%
&(\la + \mu) \sum_{i \ge 1} \sum_{j \ge 1} p(i,j) x^i y^{j - 1} \notag \\
&= \la \sum_{i \ge 1} \sum_{j \ge 1} p(i - 1,j) x^i y^{j - 1} + \mu \sum_{i \ge 1} \sum_{j \ge 1} p(i,j + 1) x^i y^{j - 1}. \label{eqnGTD:pgf_balance_equations_summed_interior} \\
\intertext{Multiply both sides of \eqref{eqnGTD:balance_equations_vertical_p(0,j)} by $y^{j - 1}$, sum over all $j \ge 2$, and add \eqref{eqnGTD:balance_equations_vertical_p(0,1)_changed} to obtain}
&\la \sum_{j \ge 1} p(0,j) y^{j - 1} + \mu \sum_{j \ge 2} p(0,j) y^{j - 1} \notag \\
&= \mu \sum_{j \ge 1} p(0,j + 1) y^{j - 1} + \mu \sum_{j \ge 1} p(j,1) y^{j - 1}. \label{eqnGTD:pgf_balance_equations_summed_horizontal}
\end{align}%
Summing \eqref{eqnGTD:pgf_balance_equations_summed_interior}--\eqref{eqnGTD:pgf_balance_equations_summed_horizontal}, multiplying both sides by $y/\mu$ and using simplifications based on the definition of $\PGF{x,y}$ such as
\begin{align}%
\sum_{i \ge 1} \sum_{j \ge 1} p(i,j + 1) x^i y^{j - 1} &= \sum_{i \ge 1} \sum_{j \ge 0} p(i,j + 1) x^i y^{j - 1} - \frac{1}{y} \sum_{i \ge 1} p(i,1) x^i \notag \\
&= \frac{1}{y} \bigl( \PGF{x,y} - \PGF{0,y} - \PGF{x,0} + \PGF{0,0} \bigr),
\end{align}%
shows that $\PGF{x,y}$ satisfies the functional equation
\begin{equation}%
h(x,y) \PGF{x,y} = \PGF{y,0} - \PGF{x,0} + (y - 1) \PGF{0,0} \label{eqnGTD:pgf_kernel}
\end{equation}%
with
\begin{equation}%
h(x,y) \defi y \bigl( 1 + \rho(1 - x) \bigr) - 1.
\end{equation}%

It is easily seen that for a fixed $x$ we have that $h(x,y) = 0$ for $y = \rooty(x)$ with
\begin{equation}%
\rooty(x) = \frac{1}{1 + \rho(1 - x)}.
\end{equation}%
Notice that $v(x)$ has a simple pole at $x = 1 + 1/\rho > 1$. Since $|\rooty(x)| \le 1$ if $|x| \le 1$ we have by substituting $\rooty(x)$ into \eqref{eqnGTD:pgf_kernel} that
\begin{equation}%
\PGF{x,0} = \PGF{\rooty(x),0} + (\rooty(x) - 1) \PGF{0,0}. \label{eqnGTD:pgf_P(x,0)}
\end{equation}%
We will iterate \eqref{eqnGTD:pgf_P(x,0)} to obtain an expression for $\PGF{x,0}$. First, define
\begin{equation}%
\rooty^{\circ k}(x) \defi ( \underbrace{\rooty \circ \rooty \circ \cdots \circ \rooty}_{\textup{$k$ times}})(x), \quad n \ge 1,
\end{equation}%
where the operator $\circ$ denotes a composition: $(f \circ g)(x) = f(g(x))$.

\begin{lemma}\label{lemGTD:pgf_composite_root}%
The composition $\rooty^{\circ k}(x)$ is explicitly given by
\begin{equation}%
\rooty^{\circ k}(x) = \frac{1 - \rho^k - x \rho (1 - \rho^{k - 1})}{1 - \rho^{k + 1} - x \rho (1 - \rho^k)}, \quad k \ge 1. \label{eqnGTD:pgf_composite_root}
\end{equation}%
\end{lemma}%

\begin{proof}%
For $k = 1$ the claim is true. Assume that \eqref{eqnGTD:pgf_composite_root} holds for $k$. We show that it also holds for $k + 1$. We have
\begin{equation}%
\rooty^{\circ (k + 1)}(x) = \frac{1}{1 + \rho(1 - \rooty^{\circ k}(x))}
\end{equation}%
and by substituting in the expression for $\rooty^{\circ k}(x)$ and multiplying the denominator and numerator by $1 - \rho^{k + 1} - x \rho (1 - \rho^k)$ we get that $\rooty^{\circ (k + 1)}(x)$ is given by \eqref{eqnGTD:pgf_composite_root}. Hence, by induction, we conclude that the claim holds for all $k \ge 1$.
\end{proof}%

Iterating \eqref{eqnGTD:pgf_P(x,0)} gives
\begin{equation}%
\PGF{x,0} = \lim_{k \to \infty} \PGF{\rooty^{\circ k}(x),0} + \PGF{0,0} \sum_{k \ge 1} (\rooty^{\circ k}(x) - 1). \label{eqnGTD:pgf_P(x,0)_iterated}
\end{equation}%
Since $\rooty^{\circ k}(x) - 1 = \BigO(\rho^k)$ as $k$ tends to $\infty$, the infinite series in \eqref{eqnGTD:pgf_P(x,0)_iterated} is convergent. It is easily seen from \cref{lemGTD:pgf_composite_root} that $\lim_{k \to \infty} \rooty^{\circ k}(x) = 1$ independent of $x$ and therefore
\begin{equation}%
\PGF{x,0} = \PGF{1,0} + \PGF{0,0} \sum_{k \ge 1} (\rooty^{\circ k}(x) - 1). \label{eqnGTD:pgf_P(x,0)_iterated_2}
\end{equation}%
We determine $\PGF{1,0}$ by setting $x = 0$ in \eqref{eqnGTD:pgf_P(x,0)_iterated_2}. Since
\begin{equation}%
\rooty^{\circ k}(0) - 1 = - \frac{1 - \rho}{\rho} \frac{\rho^{k + 1}}{1 - \rho^{k + 1}},
\end{equation}%
we get
\begin{equation}%
\PGF{0,0} = \PGF{1,0} - \PGF{0,0} \frac{1 - \rho}{\rho} \sum_{k \ge 1} \frac{\rho^{k + 1}}{1 - \rho^{k + 1}},
\end{equation}%
which, by $\PGF{0,0} = (1 - \rho)\rho$, indicates that
\begin{equation}%
\PGF{1,0} = (1 - \rho)^2 \sum_{k \ge 1} \frac{\rho^k}{1 - \rho^k}.
\end{equation}%
Substituting the expressions for $\PGF{1,0}$ and $\PGF{0,0}$ into \eqref{eqnGTD:pgf_P(x,0)_iterated_2} yields
\begin{equation}%
\PGF{x,0} = (1 - \rho) \Bigl( (1 - \rho) \sum_{k \ge 1} \frac{\rho^k}{1 - \rho^k} + \rho \sum_{k \ge 1} (\rooty^{\circ k}(x) - 1) \Bigr). \label{eqnGTD:pgf_P(x,0)_iterated_3}
\end{equation}%
Define
\begin{equation}%
\Rooty(x) \defi \sum_{k \ge 1} ( \rooty^{\circ k}(x) - 1 )
\end{equation}%
and substitute $\PGF{x,0}$, $\PGF{y,0}$ and $\PGF{0,0}$ into \eqref{eqnGTD:pgf_kernel} to find that $\PGF{x,y}$ satisfies
\begin{align}%
\PGF{x,y} &= \frac{(1 - \rho) \rho}{h(x,y)} \bigl( \Rooty(y) - \Rooty(x) + y - 1 \bigr) \notag \\
&= \frac{(1 - \rho)\rho \rooty(x)}{\rooty(x) - y} \bigl( 1 - y + \Rooty(x) - \Rooty(y) \bigr).
\end{align}%
Since $\Rooty(x) = \Rooty(\rooty(x)) - 1 + \rooty(x)$ we can write the expression for $\PGF{x,y}$ as
\begin{equation}%
\PGF{x,y} = (1 - \rho)\rho \rooty(x) \Bigl( 1 + \frac{\Rooty(\rooty(x)) - \Rooty(y)}{\rooty(x) - y} \Bigr). \label{eqnGTD:pgf_P(x,y)_almost_explicit}
\end{equation}%
The fraction is
\begin{equation}%
\frac{\Rooty(\rooty(x)) - \Rooty(y)}{\rooty(x) - y} = \sum_{k \ge 1} \frac{\rooty^{\circ (k + 1)}(x) - \rooty^{\circ k}(y)}{\rooty(x) - y},
\end{equation}%
where the $k$-th summand is equal to
\begin{equation}%
\frac{(1 - \rho)^2 \rho^k \frac{1}{\rooty(x)}}{(1 - \rho^{k + 2})(1 - \rho^{k + 1})} \frac{1}{1 - x \rho \frac{1 - \rho^{k + 1}}{1 - \rho^{k + 2}}} \frac{1}{1 - y \rho \frac{1 - \rho^{k}}{1 - \rho^{k + 1}}}. \label{eqnGTD:pgf_P(x,y)_summands}
\end{equation}%
If we substitute $k = 0$ into \eqref{eqnGTD:pgf_P(x,y)_summands} we get 1. So, we find from \eqref{eqnGTD:pgf_P(x,y)_almost_explicit} that
\begin{equation}%
\PGF{x,y} = \sum_{k \ge 0} \frac{(1 - \rho)^3 \rho^{k + 1}}{(1 - \rho^{k + 2})(1 - \rho^{k + 1})} \frac{1}{1 - x \rho \frac{1 - \rho^{k + 1}}{1 - \rho^{k + 2}}} \frac{1}{1 - y \rho \frac{1 - \rho^{k}}{1 - \rho^{k + 1}}}.
\end{equation}%
Expanding the terms $(1 - x \rho (1 - \rho^{k + 1})/(1 - \rho^{k + 2}))^{-1}$ and $(1 - y \rho (1 - \rho^{k})/(1 - \rho^{k + 1}))^{-1}$ as geometric series shows that $\PGF{x,y}$ is given by, for $|x|,|y| \le 1$,
\begin{equation}%
\sum_{i \ge 0} \sum_{j \ge 1} \sum_{k \ge 0} \frac{(1 - \rho)^3 \rho^{k + 1}}{(1 - \rho^{k + 2})(1 - \rho^{k + 1})} \bigl( \rho \frac{1 - \rho^{k + 1}}{1 - \rho^{k + 2}} \bigr)^i \bigl( \rho \frac{1 - \rho^{k}}{1 - \rho^{k + 1}} \bigr)^{j - 1} x^i y^{j - 1}. \label{eqnGTD:pgf_P(x,y)_explicit}
\end{equation}%
Comparing \eqref{eqnGTD:pgf_P(x,y)_explicit} with the definition of $\PGF{x,y}$ in \eqref{eqnGTD:pgf_P(x,y)_definition} shows that the equilibrium probabilities are explicitly given by
\begin{equation}%
p(i,j) = \sum_{k \ge 0} c_k \al_k^i \be_k^{j - 1}, \quad i \ge 0, ~ j \ge 1
\end{equation}%
with
\begin{align}%
\al_k = \rho \frac{1 - \rho^{k + 1}}{1 - \rho^{k + 2}}, ~ \be_k = \rho \frac{1 - \rho^{k}}{1 - \rho^{k + 1}}, ~ c_k = \frac{(1 - \rho)^3 \rho^{k + 1}}{(1 - \rho^{k + 2})(1 - \rho^{k + 1})}.
\end{align}%
%

%%%%%%%%%%%%%%%%%%%%%%%%%%%%%%%%%%%%%%%%%%%%%%%%%%%%%%%
%%%%%%%%%%%%%%%%%%%%%%%%%%%%%%%%%%%%%%%%%%%%%%%%%%%%%%%
%%%%%%%%%%%%%%%%%%%%% NEW SECTION %%%%%%%%%%%%%%%%%%%%%
%%%%%%%%%%%%%%%%%%%%%%%%%%%%%%%%%%%%%%%%%%%%%%%%%%%%%%%
%%%%%%%%%%%%%%%%%%%%%%%%%%%%%%%%%%%%%%%%%%%%%%%%%%%%%%%

\section{Matrix-geometric method}%
\label{secGTD:matrix-geometric_approach}%

The two-dimensional Markov process $\{ X(t) \}_{t \ge 0}$ is a Markov process that is QSF to the right (also called a $G/M/1$-type Markov process) with levels $\lvl{i} = \{ (i,1),(i,2),\ldots \}, ~ i \ge 0$. We ignore state $(0,0)$ since it does not appear in the balance equations \eqref{eqnGTD:balance_equations_interior}, \eqref{eqnGTD:balance_equations_vertical_p(0,j)} and \eqref{eqnGTD:balance_equations_vertical_p(0,1)_changed}. Consistent with the indexing of levels, in this section the indexing of vectors and matrices starts at 1.

The infinite-dimensional transition matrix $Q$ can be partitioned into levels as
\begin{equation}%
Q = \begin{bmatrix}%
\La_{0}^{(0)} & \La_{1}  \\
\La_{-1}      & \La_{0} & \La_1 \\
\La_{-2}      &         & \La_0  & \La_1 \\
\La_{-3}      &         &        & \La_0 & \La_1  \\
%\La_{-4}      &         &        &       & \La_0  & \La_1 \\
\vdots        &         &        &       & \ddots & \ddots \\
\end{bmatrix},%
\end{equation}%
where $\La_{-k}$ has zeroes everywhere, except $(\La_{-k})_{1,k} = \mu$, $\La_1 = \la \I$, with $\I$ the infinite-dimensional identity matrix,
\begin{equation}%
\La_0 = -(\la + \mu)\I + \begin{bmatrix}%
0 \\
\mu & 0   \\
    & \mu & 0      \\
    &     & \ddots & \ddots
\end{bmatrix}%
\end{equation}%
and
\begin{equation}%
\La_0^{(0)} = \La_0 + \begin{bmatrix}%
\mu    & 0 & \cdots \\
0      & 0 &        \\
\vdots &   & \ddots
\end{bmatrix}.%
\end{equation}%

Define the vectors $\pb_i \defi \begin{bmatrix} p(i,1) & p(i,2) & \cdots \end{bmatrix}$. The rate matrix $R$ satisfies the matrix equation
\begin{equation}%
R \La_0 + \La_1 = 0 \label{eqnGTD:mgm_R}
\end{equation}%
and the equilibrium probabilities follow from
\begin{equation}%
\pb_{i + 1} = \pb_i R, \quad i \ge 0, \label{eqnGTD:mgm_p_i}
\end{equation}%
where the boundary probabilities are computed as
\begin{equation}%
\pb_0 \La_0^{(0)} + \sum_{k \ge 1} \pb_k \La_{-k} = \zerob \quad \Leftrightarrow \quad \pb_0 \bigl( \La_0^{(0)} + \sum_{k \ge 1} R^k \La_{-k} \bigr) = \zerob \label{eqnGTD:mgm_p_0}
\end{equation}%
and the normalization condition is $\pb_0 (\I - R)^{-1} \oneb = \rho$ since $p(0,0) = 1 - \rho$, where $\oneb$ is a vector of ones.

A highly accurate approximation of the equilibrium distribution can be obtained by truncating all matrices in \eqref{eqnGTD:mgm_R} to size $K \times K$ with $K$ large and using successive substitutions to determine $R$ (see \cref{algQBD:matrix-geometric_successive_substitutions}). To determine $\pb_0$ from \eqref{eqnGTD:mgm_p_0} both the matrices and the infinite sum must be truncated. However, we can do better than that by exploiting the specific structure of the transition rate diagram.

Since the transitions between the levels are not upward, we know from the probabilistic interpretation of the elements of the rate matrix (see \cref{secQBD:matrix-geometric_method}) that
\begin{equation}%
R = \begin{bmatrix}%
r_0    \\
r_1    & r_0    \\
r_2    & r_1    & r_0    \\
\vdots & \vdots & \vdots & \ddots
\end{bmatrix}.\label{eqnGTD:mgm_R_lower-triangular}%
\end{equation}%
The elements of $R$ can be determined explicitly. Component-wise the equations \eqref{eqnGTD:mgm_R} read
\begin{align}%
-(\la + \mu) r_0 + \la &= 0, \\
-(\la + \mu) r_k + \mu r_{k - 1} &= 0, \quad k \ge 1,
\end{align}%
so we obtain
\begin{equation}%
r_k = \bigl( \frac{\mu}{\la + \mu} \bigr)^k \frac{\la}{\la + \mu}, \quad k \ge 0. \label{eqnGTD:mgm_r_k_explicit}
\end{equation}%

Determining $\pb_0$ exactly is difficult. The balance equations \eqref{eqnGTD:mgm_p_0} for level $\lvl{0}$ involve infinite-dimensional matrices and an infinite sum. We propose the following approximation scheme for $\pb_0$: truncate the vector and all matrices in \eqref{eqnGTD:mgm_p_0} to have dimension $K$; truncate the infinite sum to $K$; replace one equation with the normalization condition $\pb_0 (\I - R)^{-1} \oneb = \rho$; and numerically solve for $\pb_0$. The inverse of $\I - R$ can be calculated exactly, see \cref{subsecQBD:R_exact_solutions}.

The vectors $\pb_i, ~ i \ge 1$ follow from \eqref{eqnGTD:mgm_p_i}, which reads as
\begin{equation}%
p(i + 1,j) = \sum_{k \ge 0} p(i,j + k) r_k.
\end{equation}%
Unfortunately, also this expression involves an infinite sum. Truncating the sum once more to $K$ finally gives an approximation for the equilibrium probabilities.

\cref{algGTD:mgm_approximation} shows how to derive the approximate equilibrium distribution using the matrix-geometric method. The parameter $K$ determines the accuracy of the obtained approximate equilibrium probabilities: the dimension of all matrices and infinite sums are truncated to $K$. So, increasing $K$ increases the accuracy of the results, but also requires more computation time.

\begin{algorithm}%
\caption{Matrix-geometric method}%
\label{algGTD:mgm_approximation}%
\begin{algorithmic}[1]%
\State Pick a large positive integer $K$
\State Calculate $\{ r_k \}_{0 \le k \le K}$ from \eqref{eqnGTD:mgm_r_k_explicit}
\State Construct the $K \times K$ matrix $R$ using \eqref{eqnGTD:mgm_R_lower-triangular}
\State Determine $(\I - R)^{-1}$ using the theory of \cref{subsecQBD:R_exact_solutions}
\State Construct the matrix
    \begin{equation}%
    A = \La_0^{(0)} + \sum_{k = 1}^K R^k \La_{-k},
    \end{equation}%
    replace any column (say column $i$) of $A$ by $(\I - R)^{-1} \oneb$ and construct the vector $b$ with zeroes everywhere, except $(b)_i = \rho$
\State Solve $\pb_0$ from $\pb_0 A = b$ using a standard numerical solver
\State Compute $\pb_i, ~ 1 \le i \le K$ from \eqref{eqnGTD:mgm_p_i}
\end{algorithmic}%
\end{algorithm}%

Recall that the total number of jobs behaves like an $M/M/1$ queue and therefore we have the exact equilibrium probabilities in \eqref{eqnGTD:MM1_equilibrium_probabilities}. Clearly, for $k \ge 1$,
\begin{equation}%
p(k) = \sum_{i = 0}^{k - 1} p(i,k - i).
\end{equation}%
In \cref{tblGTD:mgm_comparison_algorithm_exact} we compare $p(k)$ obtained using \cref{algGTD:mgm_approximation} with the exact values of \eqref{eqnGTD:MM1_equilibrium_probabilities}.

\begin{table}%
\centering%
\begin{tabular}{*{6}{c}}
         & \multicolumn{4}{c}{$k$} & \\
$K$      & 1 & 3 & 5 & 10 & time (\si{\milli\second}) \\
\hline
10       & 0.17049075 & 0.10911408 & 0.06983301 & 0.02288288 & 1.09 \\
20       & 0.16084709 & 0.10294214 & 0.06588296 & 0.02158853 & 3.43 \\
30       & 0.16007624 & 0.10244879 & 0.06556722 & 0.02148506 & 6.71 \\
40       & 0.16000709 & 0.10240454 & 0.06553890 & 0.02147578 & 32.8 \\
50       & 0.16000067 & 0.10240043 & 0.06553627 & 0.02147492 & 48.4 \\
$\infty$ & 0.16000000 & 0.10240000 & 0.06553600 & 0.02147483 &
\end{tabular}%
\caption{Comparing $p(k)$ obtained using \protect\cref{algGTD:mgm_approximation} for various values of $K$ and $k$ with the exact values \protect\eqref{eqnGTD:MM1_equilibrium_probabilities}. Parameter values are $\la = 0.8$ and $\mu = 1$.}
\label{tblGTD:mgm_comparison_algorithm_exact}
\end{table}%

%%%%%%%%%%%%%%%%%%%%%%%%%%%%%%%%%%%%%%%%%%%%%%%%%%%%%%%
%%%%%%%%%%%%%%%%%%%%%%%%%%%%%%%%%%%%%%%%%%%%%%%%%%%%%%%
%%%%%%%%%%%%%%%%%%%%% NEW SECTION %%%%%%%%%%%%%%%%%%%%%
%%%%%%%%%%%%%%%%%%%%%%%%%%%%%%%%%%%%%%%%%%%%%%%%%%%%%%%
%%%%%%%%%%%%%%%%%%%%%%%%%%%%%%%%%%%%%%%%%%%%%%%%%%%%%%%

\section{Compensation approach}%
\label{secGTD:compensation_approach}%

We make the educated guess that $p(i,j)$ is of the form $\al^i \be^{j - 1}$. Substitute this guess into the balance equations \eqref{eqnGTD:balance_equations_interior} and divide by common powers to obtain
\begin{equation}%
(\rho + 1) \al = \rho + \al \be \quad \Rightarrow \quad \al = \frac{\rho}{\rho + 1 - \be} \ifed f(\be). \label{eqnGTD:comp_interior}
\end{equation}%
Any pair $(\al,\be)$ that satisfies \eqref{eqnGTD:comp_interior}, satisfies the balance equations \eqref{eqnGTD:balance_equations_interior}. Moreover, any linear combination of product-form solutions that each, by itself, satisfies \eqref{eqnGTD:comp_interior} also satisfies \eqref{eqnGTD:balance_equations_interior}, which is a crucial property that we shall exploit. Since the equilibrium distribution must be normalized, only solution pairs $(\al,\be)$ with $|\al|,|\be| < 1$ are of interest.

We will construct a linear combination of solutions that satisfy the balance equations for the states in the interior to also satisfy the balance equations \eqref{eqnGTD:balance_equations_vertical_p(0,j)} on the vertical axis. If both \eqref{eqnGTD:balance_equations_interior} and \eqref{eqnGTD:balance_equations_vertical_p(0,j)} are satisfied, then the remaining balance equation \eqref{eqnGTD:balance_equations_vertical_p(0,1)_changed} is automatically satisfied, since the balance equations are dependent.

Rearrange \eqref{eqnGTD:balance_equations_vertical_p(0,j)} to
\begin{equation}%
(\rho + 1)p(0,j) - p(0,j + 1) = p(j,1), \quad j \ge 2. \label{eqnGTD:balance_equations_vertical_p(0,j)_rearranged}
\end{equation}%
Let us take as initial term $p(i,j) = c_0 \al_0^i \be_0^{j - 1}$ with $\be_0 = 0$, $\al_0 = f(\be_0) = \rho / (\rho + 1)$ and $c_0 > 0$ some constant. The choice $\be_0 = 0$ is essential; we argue why in \cref{remGTD:comp_alternative_beta_0}. Since the pair $(\al_0,\be_0)$ satisfies \eqref{eqnGTD:comp_interior}, the initial term $p(i,j)$ satisfies the balance equations of the interior, but does it also satisfy \eqref{eqnGTD:balance_equations_vertical_p(0,j)_rearranged}? Substitute $p(i,j) = c_0 \al_0^i \be_0^{j - 1}$ into \eqref{eqnGTD:balance_equations_vertical_p(0,j)_rearranged} to get
\begin{equation}%
0 = c_0 \al_0^j.
\end{equation}%
It is clear that the above equality does not hold, however, in this section we will abuse notation and write `$=$' anyway. Clearly, $p(i,j)$ does not satisfy \eqref{eqnGTD:balance_equations_vertical_p(0,j)_rearranged}. Let us therefore add another product-form term to compensate for the error $c_0 \al_0^j$. Set $p(i,j) = c_0 \al_0^i \be_0^{j - 1} + c_1 \al_1^i \be_1^{j - 1}$ and substitute this into \eqref{eqnGTD:balance_equations_vertical_p(0,j)_rearranged} to get
\begin{equation}%
c_1(\rho + 1) \be_1^{j - 1} - c_1 \be_1^j = c_0 \al_0^j + c_1 \al_1^j.
\end{equation}%
Since we want to compensate for the error introduced by the initial term, we chose $c_1$ and $\be_1$ such that
\begin{equation}%
c_1(\rho + 1) \be_1^{j - 1} - c_1 \be_1^j = c_0 \al_0^j. \label{eqnGTD:comp_second_term_eliminating_error_term}
\end{equation}%
Equation~\eqref{eqnGTD:comp_second_term_eliminating_error_term} must hold for all $j \ge 1$ and it is therefore immediate that we must choose $\be_1 = \al_0$. We want the pair $(\al_1,\be_1)$ to satisfy \eqref{eqnGTD:comp_interior} and therefore conclude that
\begin{equation}%
\be_1 = \al_0, \quad \al_1 = f(\be_1), \quad c_1 = c_0 \frac{\al_0}{\rho + 1 - \al_0}. \label{eqnGTD:comp_second_term}
\end{equation}%
By compensating once and choosing $c_1$, $\al_1$ and $\be_1$ as in \eqref{eqnGTD:comp_second_term} we have introduced a new error on the right-hand side of \eqref{eqnGTD:balance_equations_vertical_p(0,j)_rearranged}, namely $c_1 \al_1^j$. We compensate a second time: add a product-form term to the solution to get $p(i,j) = c_0 \al_0^i \be_0^{j - 1} + c_1 \al_1^i \be_1^{j - 1} + c_2 \al_2^i \be_2^{j - 1}$ and compensate for the error term $c_1 \al_1^j$ introduced by the previous compensation step. Similarly as for the previous compensation step, we set
\begin{equation}%
\be_2 = \al_1, \quad \al_2 = f(\be_2), \quad c_2 = c_1 \frac{\al_1}{\rho + 1 - \al_1}. \label{eqnGTD:comp_third_term}
\end{equation}%
Substituting this three-term solution $p(i,j)$ into \eqref{eqnGTD:balance_equations_vertical_p(0,j)_rearranged} gives zero on the left-hand side, but an error term $c_2 \al_2^j$ on the right-hand side.

The procedure is clear: compensation step $k$ adds a term $c_k \al_k^i \be_k^{j - 1}$ to the current solution to compensate for the error term $c_{k - 1} \al_{k - 1}^j$ introduced during compensation step $k - 1$. The terms are chosen according to
\begin{equation}%
\be_k = \al_{k - 1}, \quad \al_k = f(\be_k), \quad c_k = c_{k - 1} \frac{\al_{k - 1}}{\rho + 1 - \al_{k - 1}}. \label{eqnGTD:comp_general_term}
\end{equation}%

Now, if the error terms $c_k \al_k^j$ tend to zero sufficiently fast as $k \to \infty$, then the linear combination of product-form solutions
\begin{equation}%
p(i,j) = \sum_{k \ge 0} c_k \al_k^i \be_k^{j - 1}, \quad i \ge 0, ~ j \ge 1, \label{eqnGTD:comp_p(i,j)}
\end{equation}%
is finite and satisfies \eqref{eqnGTD:balance_equations_interior} and \eqref{eqnGTD:balance_equations_vertical_p(0,j)}. From $\al_k = f(\al_{k - 1})$ and $\be_k = \al_{k - 1}$ it can be verified that
\begin{align}%
\al_k &= \rho \frac{1 - \rho^{k + 1}}{1 - \rho^{k + 2}}, \quad \be_k = \rho \frac{1 - \rho^k}{1 - \rho^{k + 1}}, \label{eqnGTD:comp_alpha_k_beta_k} \\
c_k &= c_0 \prod_{l = 0}^{k - 1} \frac{\al_l}{\rho + 1 - \al_l} = c_0 \rho^k \frac{1 - \rho^2}{1 - \rho^{k + 2}} \frac{1 - \rho}{1 - \rho^{k + 1}}. \label{eqnGTD:comp_c_k_still_function_of_c_0}
\end{align}%
From these explicit expressions it is clear that $c_k \al_k^j \to 0$ for $k \to \infty$. In fact, the error terms $c_k \al_k^j$ tend to zero geometrically fast (with rate $\rho$). Since $0 < \al_k,\be_k < 1$ and $c_k > 0$, we know that \eqref{eqnGTD:comp_p(i,j)} is maximal if $i = 0$ and $j = 1$. Therefore, if \eqref{eqnGTD:comp_p(i,j)} is finite for $i = 0$ and $j = 1$, then it is finite for all $i \ge 0, ~ j \ge 1$. We have
\begin{equation}%
p(0,1) = \sum_{k \ge 0} c_k = c_0 \sum_{k \ge 0} \rho^k \frac{1 - \rho^2}{1 - \rho^{k + 2}} \frac{1 - \rho}{1 - \rho^{k + 1}} < c_0 \sum_{k \ge 0} \rho^k < \infty.
\end{equation}%
The constant $c_0$ follows from the normalization condition \eqref{eqnGTD:normalization_condition}:
\begin{equation}%
\rho = \sum_{i \ge 0} \sum_{j \ge 1} \sum_{k \ge 0} c_k \al_k^i \be_k^{j - 1} = \sum_{k \ge 0} c_k \frac{1}{1 - \al_k} \frac{1}{1 - \be_k} = c_0 \frac{1 - \rho^2}{(1 - \rho)^2}.
\end{equation}%
So,
\begin{equation}%
c_0 = \rho \frac{(1 - \rho)^2}{1 - \rho^2} \label{eqnGTD:comp_c_0}
\end{equation}%
and therefore
\begin{equation}%
c_k = \frac{(1 - \rho)^3 \rho^{k + 1}}{(1 - \rho^{k + 2})(1 - \rho^{k + 1})}. \label{eqnGTD:comp_c_k}
\end{equation}%
\cref{figGTD:generating_compensation_parameters} shows how the compensation parameters $\al_k$ and $\be_k$ are generated.

\begin{figure}%
\centering%
\includestandalone{Chapters/GTD/TikZFiles/generating_compensation_parameters}%
\caption{Generating the compensation parameters $\al_k$ and $\be_k$.}%
\label{figGTD:generating_compensation_parameters}%
\end{figure}%

\begin{remark}[Alternative $\be_0$]\label{remGTD:comp_alternative_beta_0}%
From \cref{figGTD:generating_compensation_parameters} it is clear that if $\be_0 > \rho$ then $\al_k \to 1$ and $\be_k \to 1$ for $k \to \infty$. This in turn means that by \eqref{eqnGTD:comp_c_k_still_function_of_c_0} $c_k \to \infty$ as $k \to \infty$ and the error terms $c_k \al_k^j \to \infty$. Hence, it is clear that $\be_0$ must satisfy $0 \le \be_0 < \rho$. However, we have made the specific choice $\be_0 = 0$. We demonstrate why that choice is essential.

\begin{figure}
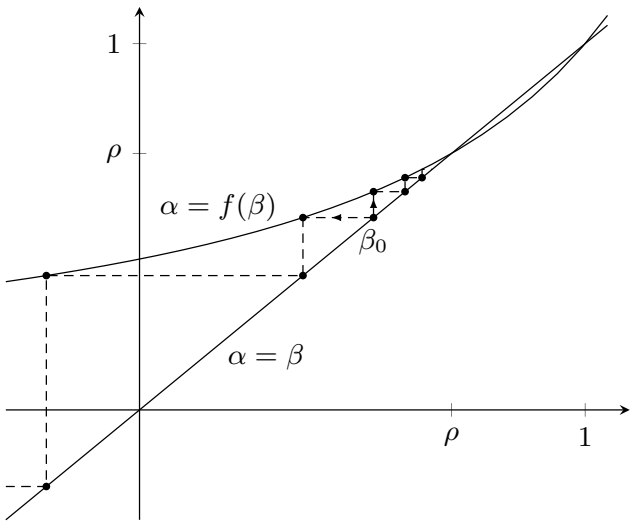
%
\centering%
\includestandalone{Chapters/GTD/TikZFiles/generating_compensation_parameters_alternative_beta_0}%
\caption{Generating the compensation parameters with initial value $0 < \be_0 < \rho$.}%
\label{figGTD:generating_compensation_parameters_alternative_beta_0}%
\end{figure}%

Let us fix an alternative $\be_0$ with $0 < \be_0 < \rho$. In that case, substituting the initial term $p(i,j) = c_0 \al_0^i \be_0^{j - 1}$ in \eqref{eqnGTD:balance_equations_vertical_p(0,j)_rearranged} results in two error terms: the term $c_0 \al_0^j$ on the right-hand side and the term $(\rho + 1) \be_0^{j - 1} - \be_0^j$ on the left-hand side. So, we would need to add two terms to compensate for the two errors. \cref{figGTD:generating_compensation_parameters_alternative_beta_0} shows that an infinite sequence of $\al_k$ and $\be_k$ is generated in two directions, where in one direction $\al_k$ and $\be_k$ tend to $-\infty$, thus leading to a divergent infinite series expression. Continuing in this way leads to a divergent infinite series expression for the equilibrium probabilities.

On the contrary, the choice $\be_0 = 0$ results in only one error term, which generates a convergent infinite series. As an edge case, choosing $\be_0 = f^{\circ k}(0)$ for some $k$, the sequence that is generated in the left- and downward direction terminates when the coordinate $(0,0)$ is hit (the correct initial value!).
\end{remark}%

We have seen that the compensation approach solves the balance equations by inserting a linear combination of product-form solutions. The linear combination contains a countably infinite number of product-form solutions and therefore a procedure is required to select the right product-form terms. These product-form solutions all have one thing in common: they satisfy the balance equations \eqref{eqnGTD:balance_equations_interior} of the states in the interior of the state space.

For numerical purposes the infinite sum expression must be truncated. We outline a simple procedure to determine an approximation of any equilibrium probability in \cref{algGTD:compensation_approach}. Just as for the matrix-geometric approach, we compare the values obtained for $p(k)$ from \cref{algGTD:compensation_approach} and the exact values in \eqref{eqnGTD:MM1_equilibrium_probabilities}. Comparing \cref{tblGTD:comp_comparison_algorithm_exact} with \cref{tblGTD:mgm_comparison_algorithm_exact}, it seems that the compensation approach produces better approximations of the equilibrium probabilities while requiring less computation time.

\begin{algorithm}%
\caption{Compensation approach}%
\label{algGTD:compensation_approach}%
\begin{algorithmic}[1]%
\State Pick a large positive integer $K$
\State Calculate $\{ \al_k \}_{0 \le k \le K}$ from \eqref{eqnGTD:comp_alpha_k_beta_k}.
\State Set $\be_0 = 0$ and $\be_k = \al_{k - 1}, ~ 1 \le k \le K$
\State Calculate $\{ c_k \}_{0 \le k \le n}$ from \eqref{eqnGTD:comp_c_k}
\State Compute any $p(i,j)$ from \eqref{eqnGTD:comp_p(i,j)} with the infinite sum truncated to $K$
\end{algorithmic}%
\end{algorithm}%

\begin{table}%
\centering%
\begin{tabular}{*{6}{c}}
         & \multicolumn{4}{c}{$k$} & \\
$K$      & 1 & 3 & 5 & 10 & time (\si{\milli\second}) \\
\hline
10       & 0.15704838 & 0.09683680 & 0.05971004 & 0.01782644 & 0.14 \\
30       & 0.15996828 & 0.10233911 & 0.06547107 & 0.02143230 & 0.21 \\
50       & 0.15999963 & 0.10239929 & 0.06553525 & 0.02147434 & 0.31 \\
100      & 0.15999999 & 0.10239999 & 0.06553599 & 0.02147483 & 0.78 \\
200      & 0.16000000 & 0.10240000 & 0.06553600 & 0.02147483 & 0.94 \\
$\infty$ & 0.16000000 & 0.10240000 & 0.06553600 & 0.02147483 &
\end{tabular}%
\caption{Comparing $p(k)$ obtained using \protect\cref{algGTD:compensation_approach} for various values of $K$ and $k$ with the exact values \protect\eqref{eqnGTD:MM1_equilibrium_probabilities}. Parameter values are $\la = 0.8$ and $\mu = 1$.}
\label{tblGTD:comp_comparison_algorithm_exact}
\end{table}%

%%%%%%%%%%%%%%%%%%%%%%%%%%%%%%%%%%%%%%%%%%%%%%%%%%%%%%%
%%%%%%%%%%%%%%%%%%%%%%%%%%%%%%%%%%%%%%%%%%%%%%%%%%%%%%%
%%%%%%%%%%%%%%%%%%%%% NEW SECTION %%%%%%%%%%%%%%%%%%%%%
%%%%%%%%%%%%%%%%%%%%%%%%%%%%%%%%%%%%%%%%%%%%%%%%%%%%%%%
%%%%%%%%%%%%%%%%%%%%%%%%%%%%%%%%%%%%%%%%%%%%%%%%%%%%%%%

\section{Takeaways}%
\label{secGTD:what_have_we_learned}%

The Markov process associated with the gated single-server system has no upward transitions in the interior of the state space. However, it does have transition from the states on the horizontal axis to states on the vertical axis, a property that makes the analysis of the gated single-server system challenging.

For the generating function approach, the transitions from the horizontal axis to the vertical axis ensured that both $\PGF{x,0}$ and $\PGF{y,0}$ appear, while $\PGF{0,y}$ did not appear in the functional equation for $\PGF{x,y}$. Substituting the root $\rooty(x)$ into the functional equation led to an expression of $\PGF{x,0}$ in terms of the same generating function evaluated in a different point, namely $\PGF{\rooty(x),0}$. By iteratively substituting $\PGF{\rooty(x),0},\PGF{\rooty(\rooty(x)),0},\ldots$ an infinite sum expression was obtained for $\PGF{x,0}$. In our case the function $\rooty(x)$ was easy to work with and allowed for an explicit determination of $\PGF{x,0}$. Finally, an explicit expression involving three infinite summations was obtained for $\PGF{x,y}$. The expression for $\PGF{x,y}$ revealed that each equilibrium probability has an explicit expression in terms of an infinite sum of product-form terms.

Even though the elements of the rate matrix $R$ of the matrix-geometric method were determined explicitly, it seems that this method suffered the most from the complex balance equations of the states on the vertical axis. We had to resort to numerical approximations of the equilibrium probabilities by truncating relevant matrices, vectors and infinite summations.

The compensation approach was well-suited for Markov processes with this structure in the transition rate diagram. The approach identified that a product-form solution satisfies the balance equations of the states in the interior of the state space. These product-form solutions were then linearly combined to also satisfy the balance equations of the states on the vertical axis. Finally, we showed that the error terms tend to zero and that the infinite sum expression is convergent so that the infinite sum expression indeed described the equilibrium probabilities.

The generating function approach and the compensation approach both led to the same product-form solution. Whereas the generating function approach can be used to obtain the generating function of the equilibrium probabilities for a broad class of Markov processes, the compensation approach is more limited in scope. However, if the compensation approach can be applied, then it leads to an explicit expression for the equilibrium probabilities. We consider another model where the compensation approach can be applied in \cref{ch:join_the_shortest_queue}.

%%%%%%%%%%%%%%%%%%%%%%%%%%%%%%%%%%%%%%%%%%%%%%%%%%%%%%%
%%%%%%%%%%%%%%%%%%%%%%%%%%%%%%%%%%%%%%%%%%%%%%%%%%%%%%%
%%%%%%%%%%%%%%%%%%%%%%%% NOTES %%%%%%%%%%%%%%%%%%%%%%%%
%%%%%%%%%%%%%%%%%%%%%%%%%%%%%%%%%%%%%%%%%%%%%%%%%%%%%%%
%%%%%%%%%%%%%%%%%%%%%%%%%%%%%%%%%%%%%%%%%%%%%%%%%%%%%%%

%\theendnotes%
%\setcounter{endnote}{0}
\printendnotes% %

% Checked points 1-7 and a-i
\chapter{Production systems}%
\label{ch:cyclic_production_systems}%

In this chapter we consider three production systems that give rise to two-dimensional Markov processes. The first system produces standard items to stock and non-standard items to demand. The second system produces items in two phases. When all demand for items is fulfilled, the system is allowed to complete the first phase of the production and place these half-finished items on stock.\endnote{The first and second system are hybrid systems that combine two production disciplines: make-to-order and make-to-stock. The analysis of these two systems is presented in Adan and van der Wal \cite{Adan1998_MTO_MTS}. Williams \cite{Williams1984_MTO_MTS}, Van Donk \cite{Donk2001_MTO_MTS} and Carr et al \cite{Carr1993_MTO_MTS} treat hybrid systems and answer questions such as which item to stock and which item to produce to order and what capacity is required.} The third system is a production line with two machines and three processing steps. The first and last step are both executed by machine one. Machine one works on step one items and immediately switches to items that require processing in the last step whenever they become available.\endnote{The third model is called a re-entrant line model and is analyzed in Adan and Weiss \cite{Adan2006_Push-pull} using the same techniques as in this book. The stability condition is derived in Weiss \cite{Weiss2004_Stability_push-pull}. Similar re-entrant line models (without the infinite supply of work) can be found in Chen and Meyn \cite{Chen1999_Multiclass_queueing_networks,Chen2003_Sensitivity_network_optimization} and Dai and Weiss \cite{Dai1996_Stability_reentrant_lines}.} For each system we present a tailor-made solution method to obtain the equilibrium distribution.

\section{Stocking standard items}%
\label{secCPM:model1}%

Consider a single-server system that produces both standard items to stock and non-standard items to demand. When there is no unfulfilled demand for either product, the server (machine or worker) produces standard items to stock in anticipation of future demand. We assume that at most $J$ units of standard items can be placed on stock. Demand for standard items are delivered from stock. However, if there is no stock, then the server produces standard items to satisfy the demand. Non-standard items are never delivered from stock, but are produced to order. Demand for standard and non-standard items arrives according to Poisson processes with rates $\la_1$ and $\la_2$. We denote $\la \defi \la_1 + \la_2$. The production times for both items are exponentially distributed with rate $\mu$. Producing items to satisfy demand preempts the production of standard items to stock. The sample paths of this system alternate between the server producing as many standard items to stock as possible in its otherwise idle time and the server satisfying demand for both standard and non-standard items.

Let $X_1(t)$ be the total number of unfulfilled demand (both standard and non-standard items) at time $t$ and let $X_2(t)$ be the number of standard items on stock. Denote the state of the system by $X(t) \defi (X_1(t),X_2(t))$. Then $\{ X(t) \}_{t \ge 0}$ is a Markov process with state space
\begin{equation}%
\statespace \defi \{ (i,j) \in \Nat_0^2 : 0 \le j \le J \}.
\end{equation}%
It is apparent from the transition rate diagram in \cref{figCPM:model1_transition_rate_diagram} that the state space is irreducible. To guarantee positive recurrence and the existence of the equilibrium distribution we require that
\begin{equation}%
\rho \defi \la / \mu < 1.
\end{equation}%
Let $p(i,j)$ denote the equilibrium probability of being in state $(i,j)$.

\begin{figure}
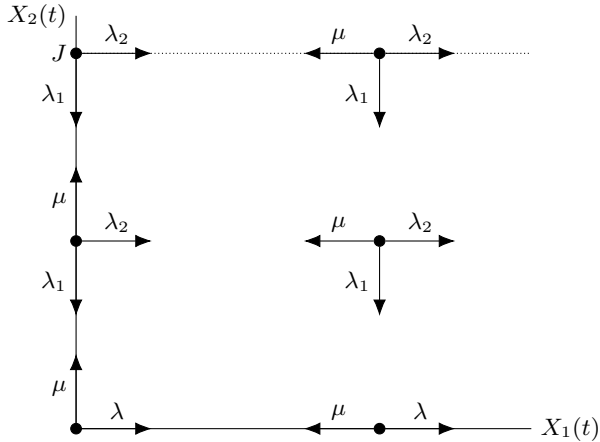
%
\centering%
\includestandalone{Chapters/CPM/TikZFiles/model1_transition_rate_diagram}%
\caption{Structure of the transition rate diagram of the Markov process associated with the first cyclic production system.}%
\label{figCPM:model1_transition_rate_diagram}%
\end{figure}%

The Markov process $\{ X(t) \}_{t \ge 0}$ is a QBD process with levels
\begin{equation}%
\lvl{i} \defi \{ (i,0),(i,1),\ldots,(i,J) \}, \quad i \ge 0. \label{eqnCPM:model1_definition_level}
\end{equation}%
We use the matrix-geometric method to determine the equilibrium distribution. To that end, define the vectors
\begin{equation}%
\pb_i \defi \begin{bmatrix} p(i,0) & p(i,1) & \cdots & p(i,J) \end{bmatrix}.
\end{equation}%

We display the balance equations in vector-matrix notation. The balance equations for the interior levels $\lvl{i}, ~ i \ge 1$ are
\begin{equation}%
\pb_{i - 1} \La_1 + \pb_i \La_0 + \pb_{i + 1} \La_{-1} = \zerob, \label{eqnCPM:model1_balance_equations_interior}
\end{equation}%
where $\La_{-1} = \mu \I$,
\begin{equation}%
\La_1 = \la_2 \I + \begin{bmatrix}%
\la_1  & 0 & \cdots & 0 \\
0      & 0 &        &   \\
\vdots &   & \ddots &   \\
0      &   &        & 0
\end{bmatrix}%
\end{equation}%
and
\begin{equation}%
\La_0 = -(\la + \mu) \I + \begin{bmatrix}%
0     &       \\
\la_1 & 0     \\
      & \la_1 & 0 \\
      &       & \ddots & \ddots \\
      &       &        & \la_1   & 0
\end{bmatrix}.%
\end{equation}%
For the boundary level $\lvl{0}$ we have the balance equation
\begin{equation}%
\pb_0 \La_0^{(0)} + \pb_1 \La_{-1} = \zerob, \label{eqnCPM:model1_balance_equations_boundary}
\end{equation}%
where
\begin{equation}%
\La_0^{(0)} = \La_0 + \begin{bmatrix}%
0 & \mu \\
  & 0   & \mu    \\
  &     & \ddots & \ddots \\
  &     &        & 0      & \mu \\
  &     &        &        & 0
\end{bmatrix} + \begin{bmatrix}%
0 &        &   & 0  \\
  & \ddots &   & \vdots \\
  &        & 0 & 0 \\
0 & \cdots & 0 & \mu
\end{bmatrix}.%
\end{equation}%

The rate matrix $R$ satisfies the matrix-quadratic equation
\begin{equation}%
R^2 \La_{-1} + R \La_0 + \La_1 = 0, \label{eqnCPM:model1_R_satisfies}
\end{equation}%
and the equilibrium probabilities follow from
\begin{equation}%
\pb_{i + 1} = \pb_i R, \quad i \ge 0.
\end{equation}%
The boundary probabilities $\pb_0$ are computed from \eqref{eqnCPM:model1_balance_equations_boundary} by inserting $\pb_1 = \pb_0 R$:
\begin{equation}%
\pb_0 \bigl( \La_0^{(0)} + R \La_{-1} \bigr) = \zerob \label{eqnCPM:model1_balance_equations_boundary_using_R}
\end{equation}%
and from either the normalization condition $\pb_0 (\I - R)^{-1} \oneb = 1$ or using $p(0,J) = 1 - \rho$. Clearly, we can use successive substitutions, see \cref{algQBD:matrix-geometric_successive_substitutions}, to determine the rate matrix $R$ from \eqref{eqnCPM:model1_R_satisfies}. However, we can do better and obtain exact expressions by exploiting the structure of the transition rate diagram.

In all levels except $\lvl{0}$ the process cannot move upwards. So, from the probabilistic interpretation (see \cref{secQBD:matrix-geometric_method}) we know that the rate matrix $R$ is a lower triangular matrix. We have seen this before in \cref{secPRIO:matrix-geometric_and_matrix-analytic_methods} and \cref{secGTD:matrix-geometric_approach}. Moreover, many of its elements are identical due to the homogeneous transition structure for phases $1$ until $J$. In particular, we can write
\begin{equation}%
R = \begin{bmatrix}%
b_0 \\
b_1 & r_0 \\
b_2 & r_1 & r_0 \\
b_3 & r_2 & r_1 & r_0 \\
 & \vdots & & & \ddots\\
b_J & r_{J - 1} & & \cdots & & r_0
\end{bmatrix}.%
\label{eqnCPM:model1_form_R}
\end{equation}%
With this representation in mind, the system of equations \eqref{eqnCPM:model1_R_satisfies} can be written component-wise as
\begin{align}%
\mu b_0^2 - (\la + \mu) b_0 + \la &= 0, \\
\mu r_0^2 - (\la + \mu) r_0 + \la_2 &= 0,
\end{align}%
and
\begin{align}%
\mu \bigl( b_i b_0 + \sum_{k = 1}^i r_{i - k} b_k \bigr) + \la_1 r_{i - 1} - (\la + \mu) b_i &= 0, \quad 1 \le i \le J, \label{eqnCPM:model1_elements_d_i} \\
\mu \sum_{k = j}^i r_{i - k} r_{k - j} + \la_1 r_{i - (j + 1)} - (\la + \mu) r_{i - j} &= 0, \quad 1 \le j < i \le J. \label{eqnCPM:model1_elements_r_i}
\end{align}%
Since $R$ is the minimal non-negative solution, we get that $b_0 = \rho$ and
\begin{equation}%
r_0 = \frac{\la + \mu - \sqrt{(\la + \mu)^2 - 4\la_2 \mu}}{2\mu}. \label{eqnCPM:model1_r_0}
\end{equation}%
Many of the equations \eqref{eqnCPM:model1_elements_r_i} are identical. We introduce $d = i - j$ in \eqref{eqnCPM:model1_elements_r_i} and find
\begin{equation}%
\mu \sum_{k = 0}^d r_{d - k} r_k + \la_1 r_{d - 1} - (\la + \mu) r_d = 0, \quad 1 \le d \le J - 1. \label{eqnCPM:model1_elements_r_i_subdiagonals}
\end{equation}%
Starting from the initial values $b_0$ and $r_0$ we can solve for $r_d$ for $1 \le d \le J - 1$ using \eqref{eqnCPM:model1_elements_r_i_subdiagonals} and then solve for $b_i$ for $1 \le i \le J$ using \eqref{eqnCPM:model1_elements_d_i}. Finally, we construct the matrix $R$ according to \eqref{eqnCPM:model1_form_R}.

We still need to solve for the boundary probabilities $\pb_0$. The balance equations \eqref{eqnCPM:model1_balance_equations_boundary_using_R} can be solved iteratively. Component-wise these equations read, for $1 \le j \le J - 1$,
\begin{align}%
0 &= \mu p(0,J - 1) - \la p(0,J) + \mu p(0,J) r_0, \label{eqnCPM:model1_balance_equations_boundary_phase_J} \\
0 &= \mu p(0,j - 1) - (\la + \mu) p(0,j) + \la_1 p(0,j + 1) \notag \\
&\quad + \mu \sum_{k = 0}^{J - j} p(0,j + k) r_k, \label{eqnCPM:model1_balance_equations_boundary_phase_j} \\
0 &= -(\la + \mu) p(0,0) + \la_1 p(0,1) + \mu \sum_{k = 0}^{J} p(0,k) b_k. \label{eqnCPM:model1_balance_equations_boundary_phase_0}
\end{align}%
Using $p(0,J) = 1 - \rho$ we can solve for all boundary probabilities by starting with the equation of phase $J$ and working our way down. Since we have the additional equation $p(0,J) = 1 - \rho$, equation \eqref{eqnCPM:model1_balance_equations_boundary_phase_0} is redundant, since we can determine $p(0,0)$ from \eqref{eqnCPM:model1_balance_equations_boundary_phase_j} for $j = 1$.

\cref{algCPM:model1_matrix_geometric_method} summarizes the matrix-geometric method for the model that combines production of standard items to stock and non-standard items to demand.

\begin{algorithm}%
\caption{Matrix-geometric method}%
\label{algCPM:model1_matrix_geometric_method}%
\begin{algorithmic}[1]%
\State Set $b_0 = \rho$ and $r_0$ according to \eqref{eqnCPM:model1_r_0}
\For{$d = 1,2,\ldots,J - 1$}
    \State Calculate $r_d$ from \eqref{eqnCPM:model1_elements_r_i_subdiagonals}
\EndFor
\For{$i = 1,2,\ldots,J$}
    \State Calculate $b_i$ from \eqref{eqnCPM:model1_elements_d_i}
\EndFor
\State Set $p(0,J) = 1 - \rho$ and calculate $p(0,J - 1)$ from \eqref{eqnCPM:model1_balance_equations_boundary_phase_J}
\For{$j = J - 1,J - 2,\ldots,0$}
    \State Calculate $p(0,j)$ from \eqref{eqnCPM:model1_balance_equations_boundary_phase_j}
\EndFor
\State Construct the matrix $R$ according to \eqref{eqnCPM:model1_form_R}
\State All equilibrium probabilities now follow from $\pb_i = \pb_0 R^i$
\end{algorithmic}%
\end{algorithm}%

The number of unfulfilled demand in equilibrium is denoted by $X_1$. Using \cref{algCPM:model1_matrix_geometric_method} we can determine key performance indicators such as the expected number off unfulfilled demand $\E{X_1}$ and the probability $\Prob{X_1 \ge 2}$ that two or more unfulfilled orders are in the system. We show both performance indicators in \cref{figCPM:model1_performance_indicators_unfulfilled_demand} as a function of the maximum stock level $J$. Clearly, increasing $J$ when $J$ is relatively small has a larger positive impact on these indicators than when $J$ is already relatively large.

\begin{figure}
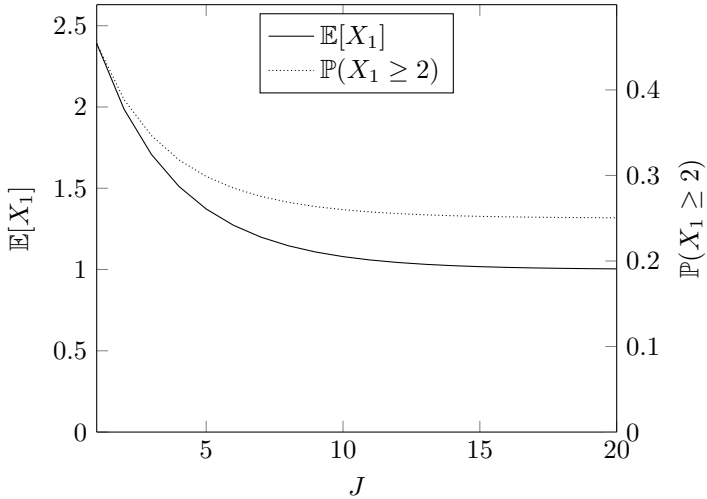
%
\centering%
\includestandalone{Chapters/CPM/TikZFiles/model1_EX1_and_ProbX1geq2}%
\caption{Performance indicators based on the number of unfulfilled demand $X_1$. Computed using \protect\cref{algCPM:model1_matrix_geometric_method}.}%
\label{figCPM:model1_performance_indicators_unfulfilled_demand}%
\end{figure}%

%%%%%%%%%%%%%%%%%%%%%%%%%%%%%%%%%%%%%%%%%%%%%%%%%%%%%%%
%%%%%%%%%%%%%%%%%%%%%%%%%%%%%%%%%%%%%%%%%%%%%%%%%%%%%%%
%%%%%%%%%%%%%%%%%%%%% NEW SECTION %%%%%%%%%%%%%%%%%%%%%
%%%%%%%%%%%%%%%%%%%%%%%%%%%%%%%%%%%%%%%%%%%%%%%%%%%%%%%
%%%%%%%%%%%%%%%%%%%%%%%%%%%%%%%%%%%%%%%%%%%%%%%%%%%%%%%

\section{Stocking half-finished items}%
\label{secCPM:model2}%

The next production system that we consider is one that produces items in two phases. The first and second phases take an exponential amount of time with parameters $\mu_1$ and $\mu_2$. There is a single server (machine or worker) that produces the items. The first phase is identical for all items. Therefore, some half-finished items (items for which only the first phase is completed) can be placed on stock in anticipation of future demand. We assume that at most $J$ units of half-finished items can be placed on stock. Demand for a single item arrives according to a Poisson process with rate $\la$. When demand arrives, the server immediately takes a half-finished item from stock and finishes its second phase, or, if there is no stock, starts immediately with the first phase. The behavior of the production system is cyclical: the server produces as much stock as possible in its otherwise idle time and then satisfies demand as it comes in until all demand is satisfied and the server returns to producing stock.

Let $X_1(t)$ be the number of unfulfilled demand at time $t$ and let $X_2(t)$ be the number of half-finished items in the system at time $t$. Denote the state of the system by $X(t) \defi (X_1(t),X_2(t))$. Then $\{ X(t) \}_{t \ge 0}$ is a Markov process with state space
\begin{equation}%
\statespace \defi \{ (i,j) \in \Nat_0^2 : 0 \le j \le J \}.
\end{equation}%
\cref{figCPM:model2_transition_rate_diagram} shows the transition rate diagram. The state space is irreducible because from each state all other states can be reached. The Markov process is positive recurrent if
\begin{equation}%
\rho \defi \la (\frac{1}{\mu_1} + \frac{1}{\mu_2}) < 1
\end{equation}%
and then the equilibrium distribution exists. Let $p(i,j)$ denote the equilibrium probability of being in state $(i,j)$.

\begin{figure}
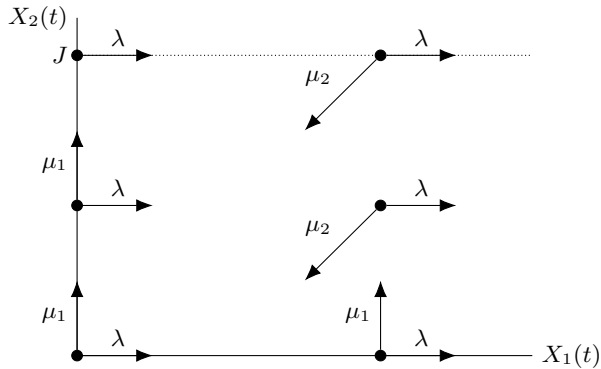
%
\centering%
\includestandalone{Chapters/CPM/TikZFiles/model2_transition_rate_diagram}%
\caption{Structure of the transition rate diagram of the Markov process associated with the second cyclic production system.}%
\label{figCPM:model2_transition_rate_diagram}%
\end{figure}%

The Markov process $\{ X(t) \}_{t \ge 0}$ is a QBD process with levels as in \eqref{eqnCPM:model1_definition_level}. We use generating functions to determine the equilibrium distribution. Since it is a QBD process, also other approaches such as the matrix-geometric or matrix-analytical methods are applicable, but we do not demonstrate them.

The balance equations for the interior levels $\lvl{i}, ~ i \ge 1$ are
\begin{align}%
(\la + \mu_2) p(i,J) &= \la p(i - 1,J), \label{eqnCPM:model2_balance_equations_interior_phase_J}\\
(\la + \mu_2) p(i,j) &= \la p(i - 1,j) + \mu_2 p(i + 1,j + 1), \quad 2 \le j \le J - 1, \label{eqnCPM:model2_balance_equations_interior_phase_j} \\
(\la + \mu_2) p(i,1) &= \la p(i - 1,1) + \mu_2 p(i + 1,2) + \mu_1 p(i,0), \label{eqnCPM:model2_balance_equations_interior_phase_1} \\
(\la + \mu_1) p(i,0) &= \la p(i - 1,0) + \mu_2 p(i + 1,1). \label{eqnCPM:model2_balance_equations_interior_phase_0}
\end{align}%
For $\lvl{0}$ we have the balance equations
\begin{align}%
\la p(0,J) &= \mu_1 p(0,J - 1), \label{eqnCPM:model2_balance_equations_boundary_phase_J}\\
(\la + \mu_1) p(0,j) &= \mu_1 p(0,j - 1) + \mu_2 p(1,j + 1), \quad 1 \le j \le J - 1, \label{eqnCPM:model2_balance_equations_boundary_phase_j} \\
(\la + \mu_1) p(0,0) &= \mu_2 p(1,1). \label{eqnCPM:model2_balance_equations_boundary_phase_0}
\end{align}%

Define the generating functions
\begin{equation}%
\PGF{j}{\PGFarg} \defi \sum_{i \ge 0} p(i,j) \PGFarg^i, \quad |\PGFarg| < 1, ~ j = 0,1,\ldots,J \label{eqnCPM:model2_definition_generating_function}
\end{equation}%
associated with the equilibrium probabilities of phase $j$. We derive expressions for these generating functions, starting with phase $J$ and working our way down.

Multiplying \eqref{eqnCPM:model2_balance_equations_interior_phase_J} by $\PGFarg^i$ and summing over all $i \ge 1$ yields
\begin{equation}%
(\la + \mu_2) \sum_{i \ge 1} p(i,J) \PGFarg^i = \la \sum_{i \ge 1} p(i - 1,J) \PGFarg^i.
\end{equation}%
Adding and subtracting $(\la + \mu_2)\PGF{J}{0}$ on the left-hand side and bringing all $\PGF{J}{\PGFarg}$ terms to one side allows us to write
\begin{equation}%
\PGF{J}{\PGFarg} = \PGF{J}{0} \frac{\la + \mu_2}{\la + \mu_2 - \la \PGFarg} = \PGF{J}{0} \frac{1}{1 - \frac{\la}{\la + \mu_2} \PGFarg}. \label{eqnCPM:model2_generating_function_phase_J}
\end{equation}%
Since $\PGF{J}{0} = p(0,J) = 1 - \rho$, we obtain an explicit expression for $\PGF{J}{\PGFarg}$. In \eqref{eqnCPM:model2_generating_function_phase_J} we recognize the geometric series
\begin{equation}%
p(i,J) = (1 - \rho) \bigl( \frac{\la}{\la + \mu_2} \bigr)^i.
\end{equation}%

Next, we multiply \eqref{eqnCPM:model2_balance_equations_interior_phase_j} by $\PGFarg^{i + 1}$ and sum over all $i \ge 1$ to obtain, for $2 \le j \le J - 1$,
\begin{align}%
&(\la + \mu_2) \PGFarg \sum_{i \ge 1} p(i,j) \PGFarg^i \notag \\
&= \la \PGFarg^2 \sum_{i \ge 1} p(i - 1,j) \PGFarg^{i - 1} + \mu_2 \sum_{i \ge 1} p(i + 1,j + 1) \PGFarg^{i + 1}.
\end{align}%
Using definition \eqref{eqnCPM:model2_definition_generating_function} we can write this as
\begin{align}%
& (\la + \mu_2) \PGFarg ( \PGF{j}{\PGFarg} - \PGF{j}{0} ) \notag \\
&= \la \PGFarg^2 \PGF{j}{\PGFarg}  + \mu_2 \Bigr( \PGF{j + 1}{\PGFarg} - \PGF{j + 1}{0} - \PGFarg \frac{\dinf}{\dinf y} \PGF{j + 1}{y} \Big\vert_{y = 0} \Bigr). \label{eqnCPM:model2_generating_function_phase_j_eq1}
\end{align}%
Equation \eqref{eqnCPM:model2_generating_function_phase_j_eq1} for $\PGF{j}{\PGFarg}$ involves $\PGF{j}{0}$, which is unknown at this point. We derive an additional equation to eliminate $\PGF{j}{0}$ from \eqref{eqnCPM:model2_generating_function_phase_j_eq1}. Define the set of states in phase $j$ as $\set{A}_j = \{ (0,j),(1,j),\ldots \}$ and the union of the sets in the first $j$ phases as $\set{A}_{\le j} = \sum_{k = 0}^j \set{A}_k$. Since the Markov process is in equilibrium, the rate at which the process enters and leaves the set of states $\set{A}_{\le j}$ is equal. For $1 \le j \le J - 1$, this balance equation reads
\begin{equation}%
\mu_1 p(0,j) = \mu_2 \sum_{k \ge 1} p(k,j + 1),
\end{equation}%
or, in terms of the generating functions,
\begin{equation}%
\mu_1 \PGF{j}{0} = \mu_2 (\PGF{j + 1}{1} - \PGF{j + 1}{0}). \label{eqnCPM:model2_balance_equations_between_phases}
\end{equation}%
Using \eqref{eqnCPM:model2_balance_equations_between_phases} to eliminate $\PGF{j}{0}$ from \eqref{eqnCPM:model2_generating_function_phase_j_eq1} yields
\begin{align}%
( \la + \mu_2 - \la \PGFarg) \PGFarg \PGF{j}{\PGFarg} %\notag \\
&= \mu_2 \Bigr( \PGF{j + 1}{\PGFarg} - \PGF{j + 1}{0} - \PGFarg \frac{\dinf}{\dinf y} \PGF{j + 1}{y} \Big\vert_{y = 0} \Bigr) \notag \\
&\quad + ( \la + \mu_2 ) \frac{\mu_2}{\mu_1} \PGFarg (\PGF{j + 1}{1} - \PGF{j + 1}{0}). \label{eqnCPM:model2_generating_function_phase_j_eq2}
\end{align}%
As a result, the generating function $\PGF{j}{\PGFarg}$ is expressed in terms of $\PGF{j + 1}{\PGFarg}$ evaluated at some points.

We continue by examining \eqref{eqnCPM:model2_balance_equations_interior_phase_1}. Multiply both sides by $\PGFarg^{i + 1}$ and sum over all $i \ge 1$ to obtain
\begin{align}%
(\la + \mu_2) \PGFarg \sum_{i \ge 1} p(i,1) \PGFarg^i &= \la \PGFarg^2 \sum_{i \ge 1} p(i - 1,1) \PGFarg^{i - 1} + \mu_2 \sum_{i \ge 1} p(i + 1,2) \PGFarg^{i + 1} \notag \\
&\quad + \mu_1 \PGFarg \sum_{i \ge 1} p(i,0) \PGFarg^i.
\end{align}%
Adding \eqref{eqnCPM:model2_balance_equations_boundary_phase_j} for $j = 1$ and simplifying using the definition \eqref{eqnCPM:model2_definition_generating_function} yields
\begin{align}%
(\la + \mu_2 - \la \PGFarg) \PGFarg \PGF{1}{\PGFarg} &= (\mu_2 - \mu_1) \PGFarg \PGF{1}{0} + \mu_2 ( \PGF{2}{\PGFarg} - \PGF{2}{0} ) \notag \\
&\quad + \mu_1 \PGFarg \PGF{0}{\PGFarg}.
\end{align}%
Eliminate $\PGF{1}{0}$ using \eqref{eqnCPM:model2_balance_equations_between_phases} to derive
\begin{align}%
(\la + \mu_2 - \la \PGFarg) \PGFarg \PGF{1}{\PGFarg} &= (\mu_2 - \mu_1) \PGFarg \frac{\mu_2}{\mu_1} (\PGF{2}{1} - \PGF{2}{0}) \notag \\
&\quad + \mu_2 ( \PGF{2}{\PGFarg} - \PGF{2}{0} ) + \mu_1 \PGFarg \PGF{0}{\PGFarg}. \label{eqnCPM:model2_generating_function_phase_1}
\end{align}%

We derive a second expression for $\PGF{1}{\PGFarg}$ and $\PGF{0}{\PGFarg}$. Multiply both sides of \eqref{eqnCPM:model2_balance_equations_interior_phase_0} by $\PGFarg^{i + 1}$, sum over all $i \ge 1$ and add \eqref{eqnCPM:model2_balance_equations_boundary_phase_0} to obtain
\begin{align}%
(\la + \mu_1) \PGFarg \sum_{i \ge 0} p(i,0) \PGFarg^i &= \la \PGFarg^2 \sum_{i \ge 1} p(i - 1,0) \PGFarg^{i - 1} \notag \\
&\quad + \mu_2 \sum_{i \ge 0} p(i + 1,1) \PGFarg^{i + 1}.
\end{align}%
Simplify this expression by using definition \eqref{eqnCPM:model2_definition_generating_function} and balance equation \eqref{eqnCPM:model2_balance_equations_between_phases}:
\begin{equation}%
(\la + \mu_1 - \la \PGFarg) \PGFarg \PGF{0}{\PGFarg} = \mu_2 \PGF{1}{\PGFarg} - \frac{\mu_2^2}{\mu_1} (\PGF{2}{1} - \PGF{2}{0}). \label{eqnCPM:model2_generating_function_phase_0}
\end{equation}%

Substituting \eqref{eqnCPM:model2_generating_function_phase_0} into \eqref{eqnCPM:model2_generating_function_phase_1} then completes the system of equations for $\PGF{j}{\PGFarg}, ~ 0 \le j \le J$. Starting from the explicit expression of $\PGF{J}{\PGFarg}$ in \eqref{eqnCPM:model2_generating_function_phase_J} we iteratively solve $\PGF{j}{\PGFarg}$ for $j = J - 1, J - 2, \ldots, 0$ from \eqref{eqnCPM:model2_generating_function_phase_j_eq2}, \eqref{eqnCPM:model2_generating_function_phase_1} and \eqref{eqnCPM:model2_generating_function_phase_0}.

It must be noted here that it seems that the terms $\PGF{J - j}{\PGFarg}, ~ 0 \le j \le J - 2$ can be written as a polynomial of degree $j + 1$ in $1/(1 - \la/(\la + \mu_2) \PGFarg)$. However, an explicit expression of the coefficients in each polynomial is difficult to obtain, so we do not that discuss this here.

The equilibrium probabilities can be determined by taking derivatives, or by using a standard inversion algorithm for univariate generating function such as the one we presented in \cref{algQTF:numerical_inversion_univariate_PGF}.

Let $X_2$ be the number of half-finished items on stock in equilibrium. As a performance indicator of the system, we can compute, for $J$ even,
\begin{equation}%
\Prob{X_2 \le J/2} = \sum_{j = 0}^{J/2} \PGF{j}{1},
\end{equation}%
which is the probability to be low on stock. No inversion algorithm is required to determine $\Prob{X_2 \le J/2}$, since $\PGF{j}{1}$ can easily be computed from the solution $\PGF{j}{\PGFarg}$. In \cref{tblCPM:model2_probability_low_on_stock} we show this probability for various values of $J$. By increasing $J$, $\Prob{X_2 \le J/2}$ decreases, which indicates that for a larger fraction of orders, only the second processing phase remains at the arrival instant.

\begin{table}%
\centering%
\begin{tabular}{c|*{5}{c}}
$J$                & 2        & 4        & 6        & 8        & 10 \\
\hline
$\Prob{X_2 \le J/2}$ & 0.777778 & 0.652778 & 0.556424 & 0.477322 & 0.410824
\end{tabular}%
\caption{Probability to be low on stock for $\la = 1$, $\mu_1 = 2$ and $\mu_2 = 3$. Computed from the generating functions $\PGF{j}{\PGFarg}, ~ 0 \le j \le J/2$.}
\label{tblCPM:model2_probability_low_on_stock}
\end{table}%

%%%%%%%%%%%%%%%%%%%%%%%%%%%%%%%%%%%%%%%%%%%%%%%%%%%%%%%
%%%%%%%%%%%%%%%%%%%%%%%%%%%%%%%%%%%%%%%%%%%%%%%%%%%%%%%
%%%%%%%%%%%%%%%%%%%%% NEW SECTION %%%%%%%%%%%%%%%%%%%%%
%%%%%%%%%%%%%%%%%%%%%%%%%%%%%%%%%%%%%%%%%%%%%%%%%%%%%%%
%%%%%%%%%%%%%%%%%%%%%%%%%%%%%%%%%%%%%%%%%%%%%%%%%%%%%%%

\section{Re-entrant line}%
\label{secCPM:model3}%

The third system is a re-entrant line consisting of two machines that produce items. Each item undergoes three processing steps. In the first step it is processed by the first machine, in the second step by the second machine and it finally returns to the first machine for its third processing step. As is typical in a manufacturing environment, there are always items that can be processed in the first step. We therefore assume that there is an infinite number of items awaiting the first processing step. Service times in each step are exponentially distributed with rates $\mu_1$, $\mu_2$ and $\mu_3$. \cref{figCPM:model3_schematic_of_the_system} shows the re-entrant line.

\begin{figure}
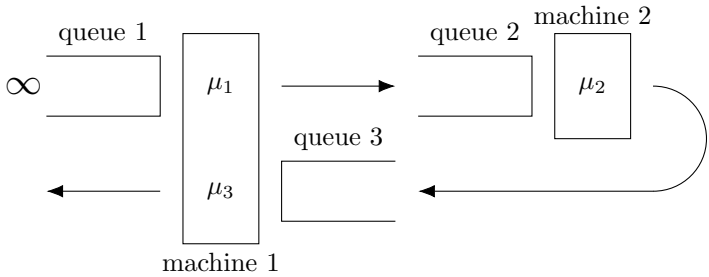
%
\centering%
\includestandalone{Chapters/CPM/TikZFiles/model3_schematic_of_the_system}%
\caption{A re-entrant line with an infinite number of items in the first queue.}%
\label{figCPM:model3_schematic_of_the_system}%
\end{figure}%

Since the first machine processes items for the first step and for the third step, we need a policy that dictates which item the first machine should serve whenever there are items in both queues. The policy we study here prioritizes processing the items in the third queue. More precisely, we assume that this priority is preemptive: whenever an item arrives in the third queue, machine 1 will stop processing an item from queue 1 and start processing the item from queue 3, only to resume processing items in queue 1 when queue 3 is empty. So, machine 1 will undergo cycles of work on items in queue 1, which are called \textit{push} periods (pushing items into the system), and work on items in queue 3, which are called \textit{pull} periods (pulling items from the system). The re-entrant line is therefore also sometimes called a push-pull system \cite{Adan2006_Push-pull}.

Let $X_2(t)$ and $X_3(t)$ be the number of items at the second and third queue at time $t$. Denote the state of the system by $X(t) \defi (X_2(t),X_3(t))$. Then $\{ X(t) \}_{t \ge 0}$ is a Markov process with state space $\statespace \defi \Nat_0^2$. The transition rate diagram in \cref{figCPM:model3_transition_rate_diagram} shows that the state space is irreducible. The states are positive recurrent if
\begin{equation}%
\frac{1}{\mu_1} + \frac{1}{\mu_3} > \frac{1}{\mu_2}. \label{eqnCPM:model3_stability_condition}
\end{equation}%
The intuition behind this condition is that if it does not hold, then the arrival rate to machine 2 will be $1/(1/\mu_1 + 1/\mu_3)$, which exceeds its service rate $\mu_2$. The proof of \eqref{eqnCPM:model3_stability_condition} is shown in \cite{Weiss2004_Stability_push-pull}. For now we assume that the condition holds and we prove that it is a sufficient condition later. Let $p(i,j)$ denote the equilibrium probability of being in state $(i,j)$.

\begin{figure}
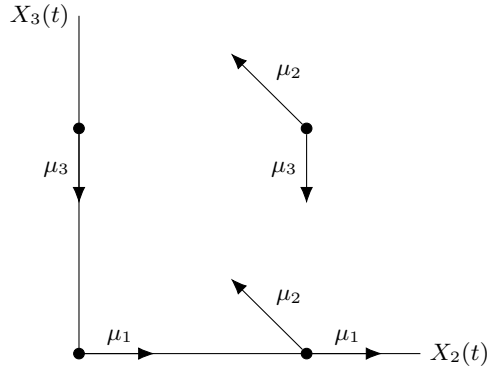
%
\centering%
\includestandalone{Chapters/CPM/TikZFiles/model3_transition_rate_diagram}%
\caption{Structure of the transition rate diagram of the Markov process associated with the re-entrant line model.}%
\label{figCPM:model3_transition_rate_diagram}%
\end{figure}%

The balance equations for the interior states are
\begin{equation}%
(\mu_2 + \mu_3) p(i,j) = \mu_3 p(i,j + 1) + \mu_2 p(i + 1,j - 1), \quad i,j \ge 1. \label{eqnCPM:model3_balance_equations_interior}
\end{equation}%
On the vertical axis we have the balance equations
\begin{equation}%
\mu_3 p(0,j) = \mu_2 p(1,j - 1) + \mu_3 p(0,j + 1), \quad j \ge 1. \label{eqnCPM:model3_balance_equations_vertical}
\end{equation}%
The balance equations on the horizontal axis are
\begin{equation}%
(\mu_1 + \mu_2) p(i,0) = \mu_1 p(i - 1,0) + \mu_3 p(i,1), \quad i \ge 1, \label{eqnCPM:model3_balance_equations_horizontal}
\end{equation}%
and at the origin we have
\begin{equation}%
\mu_1 p(0,0) = \mu_3 p(0,1). \label{eqnCPM:model3_balance_equations_origin}
\end{equation}%

We solve for the equilibrium probabilities by directly working with the balance equations. We attempt to solve the balance equations by inserting a product-form solution $\al^i \be^j$ with $0 < |\al|,|\be| < 1$ to ensure that we can normalize the solution. Substituting this product form in both \eqref{eqnCPM:model3_balance_equations_interior} and \eqref{eqnCPM:model3_balance_equations_horizontal} and dividing by common powers results in the following system of equations:
\begin{align}%
(\mu_2 + \mu_3) \be &= \mu_3 \be^2 + \mu_2 \al, \label{eqnCPM:model3_alpha_beta_satisfy_1} \\
(\mu_1 + \mu_2) \al &= \mu_1 + \mu_3 \al \be. \label{eqnCPM:model3_alpha_beta_satisfy_2}
\end{align}%
We determine $\al$ and $\be$ by solving this system of equations. From \eqref{eqnCPM:model3_alpha_beta_satisfy_2} we have that
\begin{equation}%
\al = \frac{\mu_1}{\mu_1 + \mu_2 - \mu_3 \be}. \label{eqnCPM:model3_alpha_in_terms_of_beta}
\end{equation}%
Substituting \eqref{eqnCPM:model3_alpha_in_terms_of_beta} in \eqref{eqnCPM:model3_alpha_beta_satisfy_1} and multiplying both sides by $\mu_1 + \mu_2 - \mu_3 \be$ yields a cubic equation in $\be$:
\begin{equation}%
f(\be) \defi \mu_3^2 \be^3 - \mu_3 (\mu_1 + 2\mu_2 + \mu_3) \be^2 + (\mu_1 + \mu_2) (\mu_2 + \mu_3) \be - \mu_1 \mu_2 = 0. \label{eqnCPM:model3_beta_cubic_equation}
\end{equation}%
One of the roots of this equation is $\be = \mu_2/\mu_3$. We can therefore factorize \eqref{eqnCPM:model3_beta_cubic_equation} as
\begin{equation}%
f(\be) = (\mu_3 \be - \mu_2) \bigl( \mu_3 \be^2 - (\mu_1 + \mu_2 + \mu_3) \be + \mu_1 \bigr) = 0. \label{eqnCPM:model3_beta_cubic_equation_2}
\end{equation}%
Let us study the function $f(\be)$ in more detail, see also \cref{figCPM:model3_cubic_function}. We know that $f(0) = - \mu_1 \mu_2 < 0$ and $\lim_{\be \to \infty} f(\be) = \infty$. At $\be = \mu_2/\mu_3$ we compute the derivative $f'(\be)$:
\begin{equation}%
f'\bigl( \frac{\mu_2}{\mu_3} \bigr) = \mu_1 \mu_3 - \mu_1 \mu_2 - \mu_2 \mu_3 = \mu_1 \mu_2 \mu_3 \bigl( \frac{1}{\mu_2} - \frac{1}{\mu_3} - \frac{1}{\mu_1} \bigr) < 0,
\end{equation}%
due to the stability condition \eqref{eqnCPM:model3_stability_condition}. Because of these properties and the fact that $f(\be) = 0$ is a cubic equation, we know that $f(\be) = 0$ has three positive roots: one at $\mu_2/\mu_3$, and one smaller and one larger than $\mu_2/\mu_3$. We now show that the smallest root is in $(0,1)$.

\begin{figure}
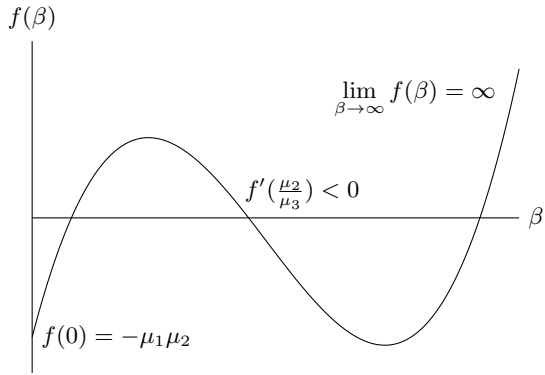
%
\centering%
\includestandalone{Chapters/CPM/TikZFiles/model3_cubic_function}%
\caption{Cubic function $f(\cdot)$ for $\mu_1 = 1$, $\mu_2 = 2$ and $\mu_3 = 3/2$.}%
\label{figCPM:model3_cubic_function}%
\end{figure}%

Define
\begin{equation}%
g(\be) \defi - (\mu_1 + \mu_2 + \mu_3) \be, \quad h(\be) \defi \mu_3 \be^2 + \mu_1.
\end{equation}%
Clearly, $g(\be)$ has a single root in $|\be| < 1$. Now, for $|\be| = 1$, we have $|g(\be)| = (\mu_1 + \mu_2 + \mu_3) |\be| = \mu_1 + \mu_2 + \mu_3$ and $|h(\be)| = |\mu_3 \be^2 + \mu_1| \le \mu_3 |\be^2| + \mu_1 = \mu_1 + \mu_3$. So, for $|\be| = 1$, we know that $|g(\be)| > |h(\be)|$ and according to Rouch\'e's theorem, see \cref{thm:Rouche}, $g(\be) + h(\be)$ has a single root in $|\be| < 1$. This root is given by
\begin{equation}%
\be = \frac{\mu_1 + \mu_2 + \mu_3 - \sqrt{(\mu_1 + \mu_2 + \mu_3)^2 - 4\mu_1 \mu_3}}{2\mu_3}. \label{eqnCPM:model3_root_beta}
\end{equation}%
Substituting the root $\be$ presented in \eqref{eqnCPM:model3_root_beta} into \eqref{eqnCPM:model3_alpha_in_terms_of_beta} yields after some manipulations
\begin{equation}%
\al = \frac{\mu_1}{\mu_2} (1 - \be). \label{eqnCPM:model3_alpha_final_expression}
\end{equation}%
Since the root $\be$ satisfies both $0 < \be < 1$ and $\be < \mu_2/\mu_3$, we know from the latter condition and \eqref{eqnCPM:model3_alpha_in_terms_of_beta} that also $0 < \al < 1$.

At this point we have a solution for the balance equations of the states $(i,j)$ with $i \ge 1$ and $j \ge 0$. We substitute this solution in \eqref{eqnCPM:model3_balance_equations_horizontal} to get
\begin{equation}%
\mu_3 p(0,j) = \mu_2 \al \be^{j - 1} + \mu_3 p(0,j + 1), \quad j \ge 1. \label{eqnCPM:model3_horizontal_boundary_what_type_of_solution}
\end{equation}%
Since this equation holds for all $j \ge 1$, we must have $p(0,j) = c \be^j, ~ j \ge 1$, where $c$ follows from substituting this solution in \eqref{eqnCPM:model3_horizontal_boundary_what_type_of_solution}:
\begin{equation}%
c = \frac{\mu_2 \al}{\mu_3 \be (1 - \be)} = \frac{\mu_1}{\mu_3 \be}.
\end{equation}%

The remaining (not yet normalized) equilibrium probability $p(0,0)$ is determined from \eqref{eqnCPM:model3_balance_equations_origin} and can be seen to equal 1.

At this point we return to the stability condition \eqref{eqnCPM:model3_stability_condition}. The solution of the balance equations is non-zero and since it is geometric, it is immediately seen to be absolutely convergent. So, as a result of \cref{thmMP:Foster}, the Markov process is positive recurrent. Since we assumed \eqref{eqnCPM:model3_stability_condition} to hold, we know that it is a sufficient condition for positive recurrence.

The solutions that we have obtained are not yet normalized. By multiplying them by the normalization constant $C$, we obtain
\begin{equation}%
p(i,j) = \begin{cases}%
C, & i = j = 0, \\
C \frac{\mu_1}{\mu_3} \be^{j - 1}, & i = 0, ~ j \ge 1, \\
C \al^i \be^j, & i \ge 1, ~ j \ge 0.
\end{cases}%
\end{equation}%
The normalization condition reads
\begin{equation}%
1 = \sum_{i \ge 0} \sum_{j \ge 0} p(i,j) = C \Bigl( 1 + \frac{\mu_1}{\mu_3} \frac{1}{1 - \be} + \frac{\al}{1 - \al} \frac{1}{1 - \be} \Bigr).
\end{equation}%
Using \eqref{eqnCPM:model3_alpha_final_expression} we write
\begin{align}%
1 &= C \Bigl( 1 + \frac{\mu_1}{\mu_2} \frac{1}{1 - \al} + \frac{\mu_1}{\mu_3} \frac{1}{1 - \be} \Bigr) \notag \\
&= C \frac{1}{\mu_2} \frac{1}{\mu_3} \frac{1}{1 - \al} \frac{1}{1 - \be} \notag \\
&\qquad\qquad \cdot \Bigl( \mu_2 \mu_3 (1 - \al)(1 - \be) + \mu_1 \mu_3 (1 - \be) + \mu_1 \mu_2 (1 - \al) \Bigr).
\end{align}%
We focus on the term in parentheses. Eliminate $\al$ using \eqref{eqnCPM:model3_alpha_final_expression} to obtain
\begin{align}%
&\mu_2 \mu_3 (1 - \frac{\mu_1}{\mu_2}(1 - \be))(1 - \be) + \mu_1 \mu_3(1 - \be) + \mu_1 \mu_2 (1 - \frac{\mu_1}{\mu_2}(1 - \be)) \notag \\
&= \mu_2 \mu_3 (1 - \be) + \mu_1 \mu_3 (\be - \be^2) + \mu_1 \mu_2 - \mu_1^2 (1 - \be) \notag \\
&= \mu_2 \mu_3 (1 - \be) + \mu_1 ( - \mu_3 \be^2 + (\mu_1 + \mu_3 ) \be - \mu_1 + \mu_2). \label{eqnCPM:model3_simplify_C_parentheses}
\end{align}%
Since $\be$ satisfies \eqref{eqnCPM:model3_beta_cubic_equation_2}, we can simplify \eqref{eqnCPM:model3_simplify_C_parentheses} and finally obtain
\begin{equation}%
C = \frac{\mu_3}{\mu_1 + \mu_3} (1 - \al).
\end{equation}%

In conclusion, provided \eqref{eqnCPM:model3_stability_condition} holds, the Markov process associated with the re-entrant line has the equilibrium probabilities
\begin{equation}%
p(i,j) = \begin{cases}%
\frac{\mu_3}{\mu_1 + \mu_3} (1 - \al), & i = j = 0, \\
\frac{\mu_1}{\mu_1 + \mu_3} (1 - \al) \be^{j - 1}, & i = 0, ~ j \ge 1, \\
\frac{\mu_3}{\mu_1 + \mu_3} (1 - \al) \al^i \be^j, & i \ge 1, ~ j \ge 0,
\end{cases}%
\label{eqnCPM:model3_normalized_equilibrium_distribution}
\end{equation}%
where $\al$ and $\be$ are given in \eqref{eqnCPM:model3_alpha_final_expression} and \eqref{eqnCPM:model3_root_beta}.

From the explicit expression \eqref{eqnCPM:model3_normalized_equilibrium_distribution} we can easily determine other key performance indicators. For example, the marginal distribution $p_2(\cdot)$ of the number of items at machine 2 is given by
\begin{equation}%
p_2(i) = \sum_{j \ge 0} p(i,j) = \begin{cases}%
\frac{\mu_3}{\mu_1 + \mu_3} (1 - \al) \bigl(1  + \frac{1}{1 - \be} \bigr), & i = 0, \\
\frac{\mu_3}{\mu_1 + \mu_3} (1 - \al) \frac{\al^i}{1 - \be}, & i \ge 1.
\end{cases}%
\end{equation}%
%

%%%%%%%%%%%%%%%%%%%%%%%%%%%%%%%%%%%%%%%%%%%%%%%%%%%%%%%
%%%%%%%%%%%%%%%%%%%%%%%%%%%%%%%%%%%%%%%%%%%%%%%%%%%%%%%
%%%%%%%%%%%%%%%%%%%%% NEW SECTION %%%%%%%%%%%%%%%%%%%%%
%%%%%%%%%%%%%%%%%%%%%%%%%%%%%%%%%%%%%%%%%%%%%%%%%%%%%%%
%%%%%%%%%%%%%%%%%%%%%%%%%%%%%%%%%%%%%%%%%%%%%%%%%%%%%%%

\section{Takeaways}%
\label{secCPM:what_have_we_learned}%

The three production systems in this chapter shared the common property that they produce items whenever they would otherwise be idle. For the first and second model this was clear: if the server would otherwise be idle, then in the first case standard items are produced to stock and in the second case the first phase of the production process is completed. In the third production system, machine one produces items from queue one whenever there are no items awaiting their third processing step.

The Markov processes associated with the first and second production system both led to two-dimensional Markov processes with one finite dimension and only nearest-neighbor transitions, which made them QBD processes. The transition rate diagrams of both systems had no upward transitions in all levels except for level 0. This structure was exploited to obtain exact expressions for the equilibrium distribution. In the first model we used the matrix-geometric method to determine the equilibrium distribution. But, instead of using the successive substitutions algorithm to determine $R$, we noticed that $R$ must be lower triangular and that many of its elements must be identical. We have seen this before in \cref{secPRIO:matrix-geometric_and_matrix-analytic_methods,secGTD:matrix-geometric_approach}. However, in this case the transition structure in phase $0$ was different from the transition structure in all other phases, which made that the boundary elements of $R$ are different from the other elements.

For the second model we again exploited the downward transition structure. By introducing a generating function for each phase, we could recursively determine all generating functions, starting from the known generating function for phase $J$. This approach allowed for an easy determination of the probability that there are $j$ half-finished items on stock, since this is equal to $\PGF{j}{1}$.

Both the first and the second model could be analyzed in multiple ways. The matrix-geometric (and matrix-analytic) method of the first model could be used to determine the equilibrium distribution of the second model. In \cite{Adan1998_MTO_MTS}, the two production systems are analyzed using the difference equations approach outlined in \cref{secPRIO:difference_equations}. Instead of starting in phase $0$ as in the priority systems of \cref{ch:single-server_priority}, we started with phase $J$ and worked our way down to phase 0. This is in line with the two approaches that we have seen in this chapter.

The third system has two infinite dimensions. To determine the equilibrium distribution of this model, we showed that the balance equations in the interior and on the horizontal axis were satisfied by a product-form solution. This product-form solution could be extended to also hold on the vertical axis and the origin by suitably multiplying it by a constant, see also \cite{Adan2006_Push-pull}.

%%%%%%%%%%%%%%%%%%%%%%%%%%%%%%%%%%%%%%%%%%%%%%%%%%%%%%%
%%%%%%%%%%%%%%%%%%%%%%%%%%%%%%%%%%%%%%%%%%%%%%%%%%%%%%%
%%%%%%%%%%%%%%%%%%%%%%%% NOTES %%%%%%%%%%%%%%%%%%%%%%%%
%%%%%%%%%%%%%%%%%%%%%%%%%%%%%%%%%%%%%%%%%%%%%%%%%%%%%%%
%%%%%%%%%%%%%%%%%%%%%%%%%%%%%%%%%%%%%%%%%%%%%%%%%%%%%%%

%\theendnotes%
%\setcounter{endnote}{0}
\printendnotes% %

\chapter{Join the shortest queue}%
\label{ch:join_the_shortest_queue}%

In this chapter we consider a system consisting of two exponential single-server queues in parallel. Jobs arrive to the system according to a Poisson process and join the shortest of the two queues. If the queue lengths are equal, then the job joins either queue with equal probability. Once a job has joined one of the two queues, it stays there until it has completed service. We are interested in the joint distribution of the number of jobs in both queues.

This join the shortest queue system\endnote{The join the shortest queue system is a classical topic in queueing theory and can be analyzed in various ways. Haight \cite{Haight1958_Two_parallel_queues} originally introduced the problem. Kingman \cite{Kingman1961_JSQ} uses generating functions and complex analysis to determine the equilibrium distribution. Cohen \cite{Cohen1998_JSQ_asymmetric} and Cohen and Boxma \cite{Cohen2000_Boundary_value_problems} reduce the analysis of the equilibrium distribution to finding a solution of a boundary value problem. Hooghiemstra, Keane and Van De Ree \cite{Hooghiemstra1988_Power_series} introduce a power-series method to calculate the equilibrium distribution for a more general class of queueing systems. Blanc \cite{Blanc1992_Power-series_JSQ} applies this numerical algorithm to obtain results for the join the shortest queue model. Halfin \cite{Halfin1985_JSQ} obtains bounds for the equilibrium distribution by using linear programming techniques.} gives rise to a Markov process in two dimensions describing the joint queue-length distribution. However, this state description leads to a transition rate diagram that is inhomogeneous in the interior of the state space and thus, is difficult to analyze. We therefore transform the state space and create a Markov process on the positive half-plane, where we can eliminate one of the two quadrants due to symmetry. This leaves us to analyze a Markov process on the positive quadrant with a homogeneous transition rate diagram in the interior of this quadrant. We then use the compensation approach to determine the equilibrium distribution of this process in the form of an infinite series of product forms.

%%%%%%%%%%%%%%%%%%%%%%%%%%%%%%%%%%%%%%%%%%%%%%%%%%%%%%%
%%%%%%%%%%%%%%%%%%%%%%%%%%%%%%%%%%%%%%%%%%%%%%%%%%%%%%%
%%%%%%%%%%%%%%%%%%%%% NEW SECTION %%%%%%%%%%%%%%%%%%%%%
%%%%%%%%%%%%%%%%%%%%%%%%%%%%%%%%%%%%%%%%%%%%%%%%%%%%%%%
%%%%%%%%%%%%%%%%%%%%%%%%%%%%%%%%%%%%%%%%%%%%%%%%%%%%%%%

\section{Model description and balance equations}%
\label{secJSQ:model_description}%

Jobs arrive according to a Poisson process with rate $2\rho$ to two parallel queues. Each job requires an exponentially distributed service time with rate 1. Due to symmetry, $\rho$ is the average amount of work brought into each queue per time unit.

Let $X_1(t)$ and $X_2(t)$ be the number of jobs at the first and second queue at time $t$. Denote the state of the system of the system by $X(t) \defi (X_1(t),X_2(t))$. Then $\{ X(t) \}_{t \ge 0}$ is a Markov process with state space $\statespace \defi \Nat_0^2$. It is apparent from the transition rate diagram in \cref{figJSQ:transition_rate_diagram_join_the_shortest_queue} that the state space is irreducible. To guarantee positive recurrence and the existence of the equilibrium distribution we assume that $\rho < 1$.

\begin{figure}
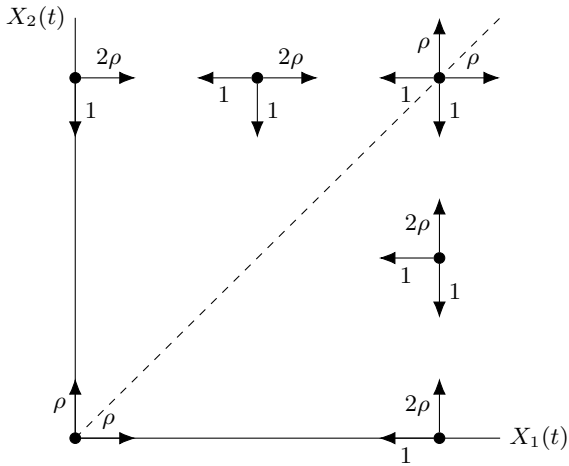
%
\centering%
\includestandalone{Chapters/JSQ/TikZFiles/transition_rate_diagram_join_the_shortest_queue_simplified}%
\caption{Structure of the transition rate diagram of the Markov process $\{ X(t) \}_{t \ge 0}$ associated with the join the shortest queue system.}%
\label{figJSQ:transition_rate_diagram_join_the_shortest_queue}%
\end{figure}%

\cref{figJSQ:transition_rate_diagram_join_the_shortest_queue} shows that the states on the diagonal $(i,i), ~ i \ge 0$ divide the state space into regions with different transition structures. In each state below the diagonal the first queue has more jobs and an arriving job joins the second queue. For the states above the diagonal the situation is reversed. This is why the Markov process has an inhomogeneous transition structure, which complicates the analysis of the equilibrium distribution. We make the analysis easier by moving to a different state description and using a symmetry argument. Define $Y_1(t) \defi \min(X_1(t),X_2(t))$ as the minimum queue length at time $t$ and $Y_2(t) \defi X_2(t) - X_1(t)$ as the difference between the two queue lengths at time $t$. The state of the system is now $Y(t) \defi (Y_1(t),Y_2(t))$ and the process $\{ Y(t) \}_{t \ge 0}$ is a Markov process on the state space $\statespace' \defi \Nat_0 \times \Int.$

Each element of $\statespace$ corresponds to exactly one element of $\statespace'$ and vice versa. For example, the state $(m,n) \in \statespace'$ corresponds to the state $(m,m + n) \in \statespace$ if $n > 0$ and to $(m - n,m)$ if $n \le 0$. The state space $\statespace'$ is irreducible because $\statespace$ is irreducible and $\{ Y(t) \}_{t \ge 0}$ is positive recurrent if $\rho < 1$. So, determining the equilibrium distribution of the Markov process $\{ Y(t) \}_{t \ge 0}$ gives us the equilibrium distribution of the Markov process $\{ X(t) \}_{t \ge 0}$. Let $p(m,n)$ denote the equilibrium probability of $\{ Y(t) \}_{t \ge 0}$ being in state $(m,n) \in \statespace'$.

The join the shortest queue policy does not favor any of the two servers in particular and the servers are identical, which makes the queue index interchangeable. As a result, the equilibrium probability that there are $i$ jobs in the first queue and $j$ jobs in the second queue is equal to the equilibrium probability that there are $j$ jobs in the first queue and $i$ jobs in the second queue. Hence, $p(m,n) = p(m,-n), ~ n > 0$ by symmetry. If we can calculate $p(m,n)$ for $m,n \ge 0$, then we know the complete equilibrium distribution.

\begin{figure}
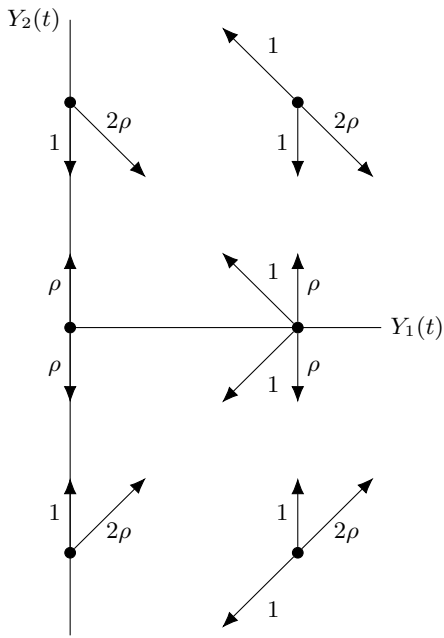
%
\centering%
\includestandalone{Chapters/JSQ/TikZFiles/transition_rate_diagram_join_the_shortest_queue_different_states_simplified}%
\caption{Structure of the transition rate diagram of the Markov process $\{ Y(t) \}_{t \ge 0}$ associated with the join the shortest queue system.}%
\label{figJSQ:transition_rate_diagram_join_the_shortest_queue_different_states}%
\end{figure}%

The transition rate diagram of $\{ Y(t) \}_{t \ge 0}$ is shown in \cref{figJSQ:transition_rate_diagram_join_the_shortest_queue_different_states}. Notice that the transition structure in each quadrant is homogeneous. Since determining $p(m,n), ~ m,n \ge 0$ is enough to obtain the complete equilibrium distribution, we only present the balance equations for the states in the positive quadrant ($(m,n)$ with $m,n \ge 0$). To that end, we exploit the symmetry property $p(m,n) = p(m,-n), ~ n > 0$ and consider balance equations that only involve the equilibrium probabilities $p(m,n), ~ m,n \ge 0$. For $m \ge 1, ~ n \ge 2$,
\begin{equation}%
2(1 + \rho) p(m,n) = 2\rho p(m - 1,n + 1) + p(m,n + 1) + p(m + 1,n - 1), \label{eqnJSQ:balance_equations_interior}
\end{equation}%
and for $m \ge 1$,
\begin{equation}%
2(1 + \rho) p(m,1) = 2\rho p(m - 1,2) + p(m,2) + \rho p(m,0) + p(m + 1,0). \label{eqnJSQ:balance_equations_interior_boundary}
\end{equation}%
For the vertical axis we have, for $n \ge 2$,
\begin{equation}%
(1 + 2\rho) p(0,n) = p(0,n + 1) + p(1,n - 1), \label{eqnJSQ:balance_equations_vertical}
\end{equation}%
and
\begin{equation}%
(1 + 2\rho) p(0,1) = p(0,2) + p(1,0) + \rho p(0,0). \label{eqnJSQ:balance_equations_vertical_boundary}
\end{equation}%
The balance equations for the horizontal axis are, for $m \ge 1$,
\begin{equation}%
(1 + \rho) p(m,0) = 2\rho p(m - 1,1) + p(m,1), \label{eqnJSQ:balance_equations_horizontal}
\end{equation}%
and at the origin
\begin{equation}%
\rho p(0,0) = p(0,1). \label{eqnJSQ:balance_equations_origin}
\end{equation}%
Substituting \eqref{eqnJSQ:balance_equations_horizontal} and \eqref{eqnJSQ:balance_equations_origin} into \eqref{eqnJSQ:balance_equations_interior_boundary} and \eqref{eqnJSQ:balance_equations_vertical_boundary} gives, for $m \ge 1$,
\begin{align}%
2(1 + \rho) p(m,1) &= 2\rho p(m - 1,2) + p(m,2) \notag \\
&\quad + \frac{\rho}{1 + \rho} \bigl( 2 \rho p(m - 1,1) + p(m,1) \bigr) \notag \\
&\quad + \frac{1}{1 + \rho} \bigl( 2 \rho p(m,1) + p(m + 1,1) \bigr), \label{eqnJSQ:balance_equations_interior_boundary_simplified} \\
(1 + 2\rho) p(0,1) &= p(0,2) + \frac{1}{1 + \rho} \bigl( 2 \rho p(0,1) + p(1,1) \bigr) + p(0,1). \label{eqnJSQ:balance_equations_vertical_boundary_simplified}
\end{align}%

The equations \eqref{eqnJSQ:balance_equations_interior}, \eqref{eqnJSQ:balance_equations_vertical}, \eqref{eqnJSQ:balance_equations_interior_boundary_simplified} and \eqref{eqnJSQ:balance_equations_vertical_boundary_simplified} together form the balance equations of the states $(m,n)$ with $m \ge 0, ~ n \ge 1$. These equations only involve the equilibrium probabilities $p(m,n), ~ m \ge 0, ~ n \ge 1$. If we can determine these equilibrium probabilities as a solution to \eqref{eqnJSQ:balance_equations_interior}, \eqref{eqnJSQ:balance_equations_vertical}, \eqref{eqnJSQ:balance_equations_interior_boundary_simplified} and \eqref{eqnJSQ:balance_equations_vertical_boundary_simplified}, then, through \eqref{eqnJSQ:balance_equations_horizontal} and \eqref{eqnJSQ:balance_equations_origin} and the symmetry $p(m,n) = p(m,-n), ~ n \ge 0$, we obtain $p(m,n)$ for all $(m,n)$.

\begin{remark}[Terminology]%
We adopt the following terminology for balance equations in three subsets of the state space $\{ (m,n) : m \ge 0, ~ n \ge 1 \}$. We refer to the balance equations \eqref{eqnJSQ:balance_equations_interior} as the balance equations of the \textit{interior}; to \eqref{eqnJSQ:balance_equations_vertical} as the balance equations of the \textit{vertical boundary}; and to
\eqref{eqnJSQ:balance_equations_interior_boundary_simplified} as the balance equations of the \textit{horizontal boundary}.
\end{remark}%

Define the bivariate PGF
\begin{equation}%
\PGF{x,y} \defi \sum_{m \ge 0} \sum_{n \ge 0} p(m,n) x^m y^n, \quad |x| \le 1, ~ |y| \le 1. \label{eqnJSQ:definition_probability_generating_function}
\end{equation}%
We can obtain an expression for $\PGF{x,y}$ by manipulating the balance equations \eqref{eqnJSQ:balance_equations_interior}--\eqref{eqnJSQ:balance_equations_origin}. Multiplying the balance equation of state $(m,n), ~ m,n \ge 0$ by $x^m y^n$ and summing over all $m,n \ge 0$ produces the functional equation
\begin{equation}%
h_1(x,y) \PGF{x,y} = h_2(x,y) \PGF{x,0} + h_3(x,y) \PGF{0,y} \label{eqnJSQ:functional_equation}
\end{equation}%
with
\begin{align}%
h_1(x,y) &\defi (1 + 2\rho x)x - 2(1 + \rho)xy + y^2, \\
h_2(x,y) &\defi (1 + 2\rho x)x - (1 + \rho)xy - \rho x y^2, \\
h_3(x,y) &\defi (y - x)y.
\end{align}%

We will not use the functional equation to determine $\PGF{x,y}$, but instead work directly with the balance equations to determine $p(m,n), ~ m,n \ge 0$ using the compensation approach. The functional equation will appear to be useful later on to determine the normalization constant, see also Kingman \cite{Kingman1961_JSQ}.

%%%%%%%%%%%%%%%%%%%%%%%%%%%%%%%%%%%%%%%%%%%%%%%%%%%%%%%
%%%%%%%%%%%%%%%%%%%%%%%%%%%%%%%%%%%%%%%%%%%%%%%%%%%%%%%
%%%%%%%%%%%%%%%%%%%%% NEW SECTION %%%%%%%%%%%%%%%%%%%%%
%%%%%%%%%%%%%%%%%%%%%%%%%%%%%%%%%%%%%%%%%%%%%%%%%%%%%%%
%%%%%%%%%%%%%%%%%%%%%%%%%%%%%%%%%%%%%%%%%%%%%%%%%%%%%%%

\section{Compensation approach}%
\label{secJSQ:compensation_approach}%

We have already seen the compensation approach\endnote{Adan, Wessels and Zijm \cite{Adan1990_JSQ_symmetric,Adan1991_JSQ_asymmetric_compensation_approach} develop the compensation approach which can be used to analyze the join the shortest queue model, but also many related models with state-dependent routing \cite{Adan2016_Polling_JSQ,Adan1996_SED_Erlang_servers,Selen2016_SED}.} in \cref{ch:gated_single-server}. Recall that the compensation approach linearly combines product-form solutions $\al^m \be^n$. Each product-form solution is chosen such that it satisfies the balance equations \eqref{eqnJSQ:balance_equations_interior} of the interior. In each compensation step a single product-form solution is added.

In a \textit{vertical} compensation step a product-form solution is added such that the resulting linear combination of product-form solutions satisfies the balance equations of both the states in the interior \eqref{eqnJSQ:balance_equations_interior} and on the vertical boundary \eqref{eqnJSQ:balance_equations_vertical}. However, in doing so, the resulting linear combination does not satisfy the balance equations \eqref{eqnJSQ:balance_equations_interior_boundary_simplified} on the horizontal boundary.

Each vertical compensation step is followed by a \textit{horizontal} compensation step. In this step, a product-form solution is added such that the resulting linear combination of product-form solutions satisfies the balance equations of both the states in the interior \eqref{eqnJSQ:balance_equations_interior} and on the horizontal boundary \eqref{eqnJSQ:balance_equations_interior_boundary_simplified}. Similarly to the vertical compensation step, the horizontal compensation step results in a linear combination that does not satisfy the balance equations \eqref{eqnJSQ:balance_equations_vertical} on the vertical boundary.

The procedure is repeated and each horizontal compensation step is followed by a vertical compensation step. We ultimately obtain two countably infinite linear combinations of product-form solution (one series each for the horizontal and vertical compensation steps). If the two series converge absolutely, then the error terms on each boundary converge sufficiently fast to zero. Finally, if the sum of the equilibrium probabilities is absolutely convergent, then by \cref{thmMP:Foster}, the solution can be normalized to obtain the equilibrium distribution.

%%%%%%%%%%%%%%%%%%%%%%%%%%%%%%%%%%%%%%%%%%%%%%%%%%%%%%%
%%%%%%%%%%%%%%%%%%%%%%%%%%%%%%%%%%%%%%%%%%%%%%%%%%%%%%%
%%%%%%%%%%%%%%%%%%%%% NEW SECTION %%%%%%%%%%%%%%%%%%%%%
%%%%%%%%%%%%%%%%%%%%%%%%%%%%%%%%%%%%%%%%%%%%%%%%%%%%%%%
%%%%%%%%%%%%%%%%%%%%%%%%%%%%%%%%%%%%%%%%%%%%%%%%%%%%%%%

\subsection{Constructing the equilibrium distribution}%
\label{subsecJSQ:constructing_equilibrium_probabilities}%

We make the educated guess that $p(m,n)$ in the interior is of the form $\al^m \be^n$. Substitute this guess into the balance equations \eqref{eqnJSQ:balance_equations_interior} and divide by common powers to obtain
\begin{equation}%
0 = \al^2 + 2 \rho \be^2 + \al \be^2 - 2(1 + \rho) \al \be. \label{eqnJSQ:interior_educated_guess}
\end{equation}%
We have the following result regarding roots of \eqref{eqnJSQ:interior_educated_guess}.

\begin{lemma}\label{lemJSQ:interior_product-form_solution}\hspace*{1em}%
\begin{enumerate}[label = \textup{(\roman*)}]%
\item For every fixed $\al$ with $|\al| \in (0,1)$, equation \eqref{eqnJSQ:interior_educated_guess} has exactly one root $\be$ inside the open circle of radius $|\al|$.
\item For every fixed $\be$ with $|\be| \in (0,1)$, equation \eqref{eqnJSQ:interior_educated_guess} has exactly one root $\al$ inside the open circle of radius $|\be|$.
\end{enumerate}%
\end{lemma}%

\begin{proof}%
(i) Divide \eqref{eqnJSQ:interior_educated_guess} by $\al^2$ and set $z = \be/\al$ to obtain the second-degree polynomial
\begin{equation}%
0 = (2\rho + \al) z^2 - 2(1 + \rho) z + 1.
\end{equation}%
Define $f(z) \defi -2(1 + \rho) z$, $g(z) \defi (2\rho + \al) z^2 + 1$ and the region $\closedunitdisc$ as the unit disk with the unit circle as the boundary $\unitcircle$. Clearly, $f(z)$ has a single root in $\closedunitdisc$. Now, for $z \in \unitcircle$, or equivalently $|z| = 1$,
\begin{align}%
|f(z)| &= 2(1 + \rho)|z| = 2 + 2\rho, \\
|g(z)| &= | (2\rho + \al) z^2 + 1| \le (2\rho + |\al|) |z|^2 + 1 = 2\rho + |\al| + 1.
\end{align}%
Since $|\al| < 1$ we conclude that $|f(z)| > |g(z)|$ for $z \in \unitcircle$. Then, by Rouch\'e's theorem, see \cref{thm:Rouche}, $f(z) + g(z)$ has a single root inside the unit circle. This proves that \eqref{eqnJSQ:interior_educated_guess} has a single root $\be$ inside the circle with radius $|\al|$.

\noindent (ii) Divide \eqref{eqnJSQ:interior_educated_guess} by $\be^2$ and set $z = \al/\be$ to obtain the second-degree polynomial
\begin{equation}%
0 = z^2 + (\be - 2(1 + \rho))z + 2\rho.
\end{equation}%
Define $f(z) \defi (\be - 2(1 + \rho))z$, $g(z) \defi z^2 + 2\rho$ and the same $\closedunitdisc$ and $\unitcircle$ as in (i). Clearly, $f(z)$ has a single root in $\closedunitdisc$. Now, for $z \in \unitcircle$, or equivalently $|z| = 1$,
\begin{align}%
|f(z)| &= | \be - 2(1 + \rho) | |z|^2 \ge | |\be| - 2(1 + \rho) | > 1 + 2\rho, \\
|g(z)| &= |z^2 + 2\rho| \le |z|^2 + 2\rho = 1 + 2\rho,
\end{align}%
where the last inequality for $f(z)$ follows from $|\be| \in (0,1)$. So, $|f(z)| > |g(z)|$ for $z \in \closedunitdisc$ and Rouch\'e's theorem proves the claim.
\end{proof}%

Let us, for now, further assume that the equilibrium probabilities along the horizontal and vertical boundary are also satisfied by a product-form solution $p(m,n) = \al^m \be^n$. We can substitute this solution in the balance equations for the horizontal boundary \eqref{eqnJSQ:balance_equations_interior_boundary_simplified}:
\begin{align}%
0 = \al^2 + \al \bigl( \be(1 + \rho) + 3\rho - 2(1 + \rho)^2 \bigr) + 2\rho (\be(1 + \rho) + \rho) \label{eqnJSQ:horizontal_educated_guess}
\end{align}%
and for the vertical boundary \eqref{eqnJSQ:balance_equations_vertical}:
\begin{equation}%
0 = \be^2 - \be(1 + 2\rho) + \al. \label{eqnJSQ:vertical_educated_guess}
\end{equation}%
In \cref{figJSQ:pairs_alpha_beta} we show the curves $(\al,\be)$ satisfying \eqref{eqnJSQ:interior_educated_guess}, \eqref{eqnJSQ:horizontal_educated_guess} and \eqref{eqnJSQ:vertical_educated_guess}, respectively. Wherever two curves intersect, we know that that pair $(\al,\be)$ satisfies those balance equations simultaneously. We find four of such pairs. Three of them are not useful since they do not satisfy $0 < |\al|,|\be| < 1$. The remaining fourth pair satisfies simultaneously the balance equations of the interior \eqref{eqnJSQ:interior_educated_guess} and the horizontal boundary \eqref{eqnJSQ:horizontal_educated_guess}. In general we can state that there is no pair $(\al,\be)$ with $0 < |\al|, |\be| < 1$ that satisfies simultaneously the balance equations of the interior and the vertical boundary, but there is a single pair $(\al,\be)$ that satisfies simultaneously the balance equations of the interior and the horizontal boundary. It is easy to derive this pair from the system of equations \eqref{eqnJSQ:interior_educated_guess} and \eqref{eqnJSQ:horizontal_educated_guess}: $(\al,\be) = (\rho^2,\rho^2/(2 + \rho))$. In \cref{tblJSQ:simulated_equilibrium_probabilities} we numerically verify that this pair dictates the tail behavior of the equilibrium probabilities for $m$ and $n$ large.

\begin{figure}
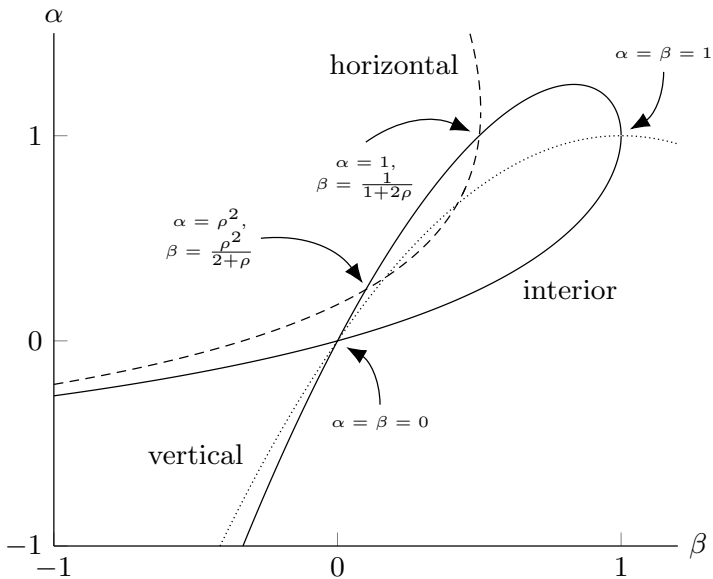
%
\centering%
\includestandalone{Chapters/JSQ/TikZFiles/product_form_solutions_interior_horizontal_vertical_balance_equations}%
\caption{Under the assumption $p(m,n) = \al^m \be^n$ for all $m \ge 0, ~ n \ge 1$ and $\rho = 1/2$, we plot the (solid) curve $(\al,\be)$ \protect\eqref{eqnJSQ:interior_educated_guess}, the (dashed) curve \protect\eqref{eqnJSQ:horizontal_educated_guess}, and the (dotted) curve \protect\eqref{eqnJSQ:vertical_educated_guess}.}%
\label{figJSQ:pairs_alpha_beta}%
\end{figure}%

We see that if $m$ and $n$ are large, then $p(m,n) \approx h_0 \al_0^m \be_0^n$ for some coefficient $h_0$, and parameters $\al_0$ and $\be_0$. We have simulated a join the shortest queue model with $\rho = 0.8$ to determine $\al_0$ and $\be_0$ from the ratios $p(m + 1,n)/p(m,n)$ and $p(m,n + 1)/p(m,n)$, see \cref{tblJSQ:simulated_equilibrium_probabilities}. The simulation confirms that $\al_0 = \rho^2$ and $\be_0 = \rho^2/(2 + \rho)$ describe the tail behavior for large $m$ and $n$. A rigorous derivation of $\al_0$ and $\be_0$ is given in, e.g., \cite[Theorem~5]{Kingman1961_JSQ}, but we do not show it here. Since normalization follows at the end of the compensation procedure, we can now set $h_0 = 1$.

\begin{table}%
\captionsetup[subfloat]{textfont = normalsize}
\centering%
\subfloat[$p(m + 1,n)/p(m,n)$]{%
\begin{tabular}{*{8}{c}}%
 & \multicolumn{7}{c}{$m$} \\
$n$ & 0 & 1 & 2 & 3 & 4 & 5 & 6 \\
\hline
0 &   &   &   &   &   &   \\
1 & 0.73 & 0.66 & 0.64 & 0.64 & 0.64 & 0.64 & 0.64 \\
2 & 0.57 & 0.62 & 0.64 & 0.64 & 0.64 & 0.64 & 0.64 \\
3 & 0.55 & 0.62 & 0.64 & 0.64 & 0.64 & 0.64 & 0.64 \\
4 & 0.54 & 0.62 & 0.63 & 0.64 & 0.64 & 0.64 & 0.64 \\
5 & 0.54 & 0.62 & 0.63 & 0.64 & 0.64 & 0.64 & 0.64 \\
6 & 0.54 & 0.62 & 0.63 & 0.64 & 0.64 & 0.63 & 0.65
\end{tabular}%
}%
\\%
\subfloat[$p(m,n + 1)/p(m,n)$]{%
\begin{tabular}{*{8}{c}}%
 & \multicolumn{7}{c}{$m$} \\
$n$ & 0 & 1 & 2 & 3 & 4 & 5 & 6 \\
\hline
0 &   &   &   &   &   &   & \\
1 & 0.31 & 0.24 & 0.23 & 0.23 & 0.23 & 0.23 & 0.23 \\
2 & 0.24 & 0.23 & 0.23 & 0.23 & 0.23 & 0.23 & 0.23 \\
3 & 0.23 & 0.23 & 0.23 & 0.23 & 0.23 & 0.23 & 0.23 \\
4 & 0.23 & 0.23 & 0.23 & 0.23 & 0.23 & 0.23 & 0.23 \\
5 & 0.23 & 0.23 & 0.23 & 0.23 & 0.23 & 0.23 & 0.23 \\
6 & 0.23 & 0.23 & 0.23 & 0.23 & 0.23 & 0.23 & 0.23
\end{tabular}%
}%
\caption{Simulated equilibrium probabilities for $\rho = 0.8$.}%
\label{tblJSQ:simulated_equilibrium_probabilities}%
\end{table}%

The pair $(\al_0,\be_0) = (\rho^2,\rho^2/(2 + \rho))$ is the only pair that satisfies simultaneously the balance equations of the interior and the horizontal boundary. In fact, this property characterizes the initial product-form solution $h_0 \al_0^m \be_0^n$. Since the initial solution does not satisfy the balance equations \eqref{eqnJSQ:balance_equations_vertical} on the vertical boundary---as we have already concluded from \cref{figJSQ:pairs_alpha_beta}---we need to compensate for the error introduced on the vertical boundary. It is important that in each compensation step---vertical or horizontal---the correction term that is added should be small compared to $h_0 \al_0^m \be_0^n$ in order to not disturb the asymptotic behavior for large $m$ or $n$.

In the vertical compensation step we add a single product-form term to the initial solution and construct $h_0 \al_0^m \be_0^n + v \al^m \be^n$. We refer to $v \al^m \be^n$ as the compensation term. We will choose $v$, $\al$ and $\be$ such that this linear combination satisfies both the balance equations of the interior \eqref{eqnJSQ:balance_equations_interior} and the vertical boundary \eqref{eqnJSQ:balance_equations_vertical}. Inserting it into \eqref{eqnJSQ:balance_equations_vertical} gives for all $n \ge 2$,
\begin{equation}%
(1 + 2\rho) \bigl( h_0 \be_0^n + v \be^n \bigr) = h_0 \be_0^{n + 1} + v \be^{n + 1} + h_0 \al_0 \be_0^{n - 1} + v \al \be^{n - 1}.
\end{equation}%
Since this equation holds for all $n \ge 2$, we must have that $\be = \be_0$. We further want the pair $(\al,\be_0)$ to satisfy the balance equations of the interior, so we pick $\al = \al_1$ as the root of \eqref{eqnJSQ:interior_educated_guess} for fixed $\be = \be_0$ satisfying $|\al_1| < |\be_0|$. There also exists the root $\al_0$ of \eqref{eqnJSQ:interior_educated_guess} for fixed $\be = \be_0$ satisfying $|\be_0| < |\al_0|$, but that would turn the compensation term into the initial term, which makes that root not useful. By choosing $\al = \al_1$ and $\be = \be_0$, we know that the linear combination $h_0 \al_0^m \be_0^n + v \al_1^m \be_0^n$ satisfies the balance equations of the interior. What remains is to choose $v = v_0$ in such a way that the linear combination $h_0 \al_0^m \be_0^n + v_0 \al_1^m \be_0^n$ satisfies \eqref{eqnJSQ:balance_equations_vertical}. We now describe the method of choosing this coefficient in a general setting.

\begin{lemma}[Vertical compensation step]\label{lemJSQ:vertical_compensation}%
Consider the product form $h \al^m \be^n$ with $0 < |\be| < |\al| < 1$ and some coefficient $h$, that satisfies the balance equations \eqref{eqnJSQ:balance_equations_interior} of the interior and stems from a solution that satisfies the balance equations of the interior and the horizontal boundary. For this fixed $\be$, let $\hat{\al}$ be the root that satisfies \eqref{eqnJSQ:interior_educated_guess} with $|\hat{\al}| < |\be|$. Then there exists a coefficient $v$ such that
\begin{equation}%
p(m,n) = h \al^m \be^n + v \hat{\al}^m \be^n
\end{equation}%
satisfies \eqref{eqnJSQ:balance_equations_interior} and \eqref{eqnJSQ:balance_equations_vertical}. The coefficient $v$ is given by
\begin{equation}%
v = - \frac{\hat{\al} - \be}{\al - \be} h. \label{eqnJSQ:vertical_compensation_coefficient}
\end{equation}%
\end{lemma}%

\begin{proof}%
Notice that both $(\al,\be)$ and $(\hat{\al},\be)$ satisfy \eqref{eqnJSQ:interior_educated_guess}. So, the linear combination $h \al^m \be^n + v \hat{\al}^m \be^n$ satisfies \eqref{eqnJSQ:balance_equations_interior} for any $h$ and $v$.

Inserting the linear combination into \eqref{eqnJSQ:balance_equations_vertical} and dividing by common powers yields
\begin{equation}%
(h + v) \bigl( 2(1 + \rho)\be - \be^2 - \be \bigr) = h \al + v \hat{\al}. \label{eqnJSQ:vertical_compensation_proof_1}
\end{equation}%
Since $(\al,\be)$ and $(\hat{\al},\be)$ both satisfy \eqref{eqnJSQ:interior_educated_guess} we know that $\al + \hat{\al} = 2(1 + \rho)\be - \be^2$. Substituting this relation into \eqref{eqnJSQ:vertical_compensation_proof_1} proves the claim.
\end{proof}%

We apply \cref{lemJSQ:vertical_compensation} to find that we must choose
\begin{equation}%
v_0 = - \frac{\al_1 - \be_0}{\al_0 - \be_0} h_0.
\end{equation}%
With these choices for the coefficient and the parameters of the compensation term, the linear combination $h_0 \al_0^m \be_0^n + v_0 \al_1^m \be_0^n$ satisfies \eqref{eqnJSQ:balance_equations_interior} and \eqref{eqnJSQ:balance_equations_vertical}. However, adding the term $v_0 \al_1^m \be_0^n$ introduces an error on the horizontal boundary for which we need to compensate.

In a horizontal compensation step we add a compensation term to compensate for the error introduced during the vertical compensation step. So, we form the linear combination $h_0 \al_0^m \be_0^n + v_0 \al_1^m \be_0^n + h \al^m \be^n$. We will choose $h$, $\al$ and $\be$ such that this linear combination satisfies both the balance equations of the interior \eqref{eqnJSQ:balance_equations_interior} and on the horizontal boundary \eqref{eqnJSQ:balance_equations_interior_boundary_simplified}. We know that $h_0 \al_0^m \be_0^n$ already satisfies \eqref{eqnJSQ:balance_equations_interior} and \eqref{eqnJSQ:balance_equations_interior_boundary_simplified}, so we do not need to take this term into account. Substituting the sum of the remaining two terms into \eqref{eqnJSQ:balance_equations_interior_boundary_simplified} gives for $m \ge 1$,
\begin{align}%
&2(1 + \rho) \bigl( v_0 \al_1^m \be_0 + h \al^m \be \bigr) \notag \\
&= 2\rho \bigl( v_0 \al_1^{m - 1} \be_0^2 + h \al^{m - 1} \be^2 \bigr) + v_0 \al_1^m \be_0^2 + h \al^m \be^2 \notag \\
&\quad + \frac{\rho}{1 + \rho} \bigl( 2 \rho \bigl( v_0 \al_1^{m - 1} \be_0 + h \al^{m - 1} \be \bigr) + v_0 \al_1^m \be_0 + h \al^m \be \bigr) \notag \\
&\quad + \frac{1}{1 + \rho} \bigl( 2 \rho \bigl( v_0 \al_1^m \be_0 + h \al^m \be \bigr) + v_0 \al_1^{m + 1} \be_0 + h \al^{m + 1} \be \bigr).
\end{align}%
Since this equation holds for all $m \ge 1$, we must have that $\al = \al_1$. We want the pair $(\al_1,\be)$ to satisfy the balance equations of the interior, so we pick $\be = \be_1$ as the root of \eqref{eqnJSQ:interior_educated_guess} for fixed $\al = \al_1$ satisfying $|\be_1| < |\al_1|$. Just as in the vertical compensation step, we can discard the other root of \eqref{eqnJSQ:interior_educated_guess}. So, by choosing $\al = \al_1$ and $\be = \be_1$, we know that the linear combination $h_0 \al_0^m \be_0^n + v_0 \al_1^m \be_0^n + h \al_1^m \be_1^n$ satisfies the balance equations of the interior. What remains is to choose $h = h_1$ in such a way that the linear combination $h_0 \al_0^m \be_0^n + v_0 \al_1^m \be_0^n + h_1 \al_1^m \be_1^n$ satisfies \eqref{eqnJSQ:balance_equations_interior_boundary_simplified}. We now describe the method of choosing this coefficient in a general setting.

\begin{lemma}[Horizontal compensation step]\label{lemJSQ:horizontal_compensation}%
Consider the product form $v \al^m \be^n$ with $0 < |\al| < |\be| < 1$ and some coefficient $v$, that satisfies the balance equations \eqref{eqnJSQ:balance_equations_interior} of the interior and stems from a solution that satisfies the balance equations of the interior and the vertical boundary. For this fixed $\al$, let $\hat{\be}$ be the root that satisfies \eqref{eqnJSQ:interior_educated_guess} with $|\hat{\be}| < |\al|$. Then there exists a coefficient $h$ such that
\begin{equation}%
p(m,n) = v \al^m \be^n + h \al^m \hat{\be}^n
\end{equation}%
satisfies \eqref{eqnJSQ:balance_equations_interior} and \eqref{eqnJSQ:balance_equations_interior_boundary_simplified}. The coefficient $h$ is given by
\begin{equation}%
h = - \frac{(\rho + \al)/\hat{\be} - (1 + \rho)}{(\rho + \al)/\be - (1 + \rho)} v. \label{eqnJSQ:horizontal_compensation_coefficient}
\end{equation}%
\end{lemma}%

\begin{proof}%
Notice that both $(\al,\be)$ and $(\al,\hat{\be})$ satisfy \eqref{eqnJSQ:interior_educated_guess}. So, the linear combination $v \al^m \be^n + h \al^m \hat{\be}^n$ satisfies \eqref{eqnJSQ:balance_equations_interior} for any $v$ and $h$.

Inserting the linear combination into \eqref{eqnJSQ:balance_equations_interior_boundary_simplified} and dividing by common powers yields
\begin{align}%
&v \bigl( 2(1 + \rho) \al \be - 2\rho \be^2 - \al \be^2 \bigr) + h \bigl( 2(1 + \rho) \al \hat{\be} - 2\rho \hat{\be}^2 - \al \hat{\be}^2 \bigr) \notag \\
&= \frac{\rho}{1 + \rho} \bigl( 2\rho (v \be + h \hat{\be}) + v \al \be + h \al \hat{\be} \bigr) \notag \\
&\quad + \frac{1}{1 + \rho} \bigl( 2\rho (v \al \be + h \al \hat{\be}) + v \al^2 \be + h \al^2 \hat{\be} \bigr).
\end{align}%
Since $(\al,\be)$ and $(\al,\hat{\be})$ satisfy \eqref{eqnJSQ:interior_educated_guess} we can simplify the coefficients of $v$ and $h$ on the left-hand side to obtain
\begin{align}%
v \al^2 (1 + \rho) + h \al^2 (1 + \rho) &= \rho \bigl( 2\rho (v \be + h \hat{\be}) + v \al \be + h \al \hat{\be} \bigr) \notag \\
&\quad + \bigl( 2\rho (v \al \be + h \al \hat{\be}) + v \al^2 \be + h \al^2 \hat{\be} \bigr).
\end{align}%
So,
\begin{equation}%
h = - \frac{\be (2\rho + \al)(\rho + \al) - \al^2(1 + \rho)}{\hat{\be} (2\rho + \al)(\rho + \al) - \al^2(1 + \rho)} v.
\end{equation}%
Since $\be$ and $\hat{\be}$ are roots of \eqref{eqnJSQ:interior_educated_guess} we have the relation $\be \hat{\be} (2 \rho + \al) = \al^2$. Using this relation proves the claim.
\end{proof}%

Applying \cref{lemJSQ:horizontal_compensation} shows that we must choose
\begin{equation}%
h_1 = - \frac{(\rho + \al_1)/\be_1 - (1 + \rho)}{(\rho + \al_1)/\be_0 - (1 + \rho)} v_0
\end{equation}%
to ensure that the linear combination $h_0 \al_0^m \be_0^n + v_0 \al_1^m \be_0^n + h_1 \al_1^m \be_1^n$ satisfies \eqref{eqnJSQ:balance_equations_interior} and \eqref{eqnJSQ:balance_equations_interior_boundary_simplified}. Adding the compensation term $h_1 \al_1^m \be_1^n$, however, introduces an error on the vertical boundary for which another vertical compensation step needs to be performed.

It is clear how the compensation procedure works: after an initial product-form solution is constructed, it alternates between horizontal and vertical compensation steps to compensate for the error introduced on the vertical or horizontal boundary in the previous compensation step. In every vertical compensation step we just need to compensate for the error introduced by the compensation term of the previous horizontal compensation step; the linear combination of product-form solutions at the time of the previous vertical compensation step namely already satisfies the balance equations of the interior and on the vertical boundary! Obviously, the same statement can be made for the horizontal compensation step.

\cref{figJSQ:indexing_compensation_parameters} shows the indexing of the terms of the compensation procedure. \cref{algJSQ:compensation_approach} can be used to generate a finite number of compensation parameters and \cref{figJSQ:generating_compensation_parameters} shows how the compensation parameters $\al_i$ and $\be_i$ are generated.

\begin{figure}
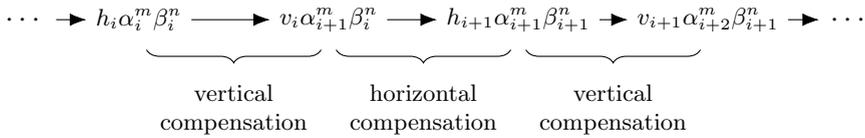
%
\centering%
\includestandalone{Chapters/JSQ/TikZFiles/indexing_compensation_parameters}%
\caption{Indexing of the terms of the compensation procedure.}%
\label{figJSQ:indexing_compensation_parameters}%
\end{figure}%

\begin{algorithm}%
\caption{Generating the compensation parameters}%
\label{algJSQ:compensation_approach}%
\begin{algorithmic}[1]%
\State Pick a large positive integer $K$
\State Set $h_0 = 1$, $\al_0 = \rho^2$ and $\be_0 = \rho^2/(2 + \rho)$
\State Calculate $\al_1$ from \eqref{eqnJSQ:interior_educated_guess} with fixed $\be = \be_0$ and $|\al_1| < |\be_0|$
\State Calculate $v_0$ using \cref{lemJSQ:vertical_compensation} with $h_0 \al_0^m \be_0^n$ as the original product form and $\hat{\al} = \al_1$
\For{$i = 1,2,\ldots,K$}
    \State Calculate $\be_i$ from \eqref{eqnJSQ:interior_educated_guess} with fixed $\al = \al_i$ and $|\be_i| < |\al_i|$
    \State Calculate $h_i$ using \cref{lemJSQ:horizontal_compensation} with $v_{i - 1} \al_i^m \be_{i - 1}^n$ as the original
    \Statex \hspace\algorithmicindent product form and $\hat{\be} = \be_i$
    \State Calculate $\al_{i + 1}$ from \eqref{eqnJSQ:interior_educated_guess} with fixed $\be = \be_i$ and $|\al_{i + 1}| < |\be_i|$
    \State Calculate $v_i$ using \cref{lemJSQ:vertical_compensation} with $h_i \al_i^m \be_i^n$ as the original product
    \Statex \hspace\algorithmicindent form and $\hat{\al} = \al_{i + 1}$
\EndFor
\end{algorithmic}%
\end{algorithm}%

\begin{figure}
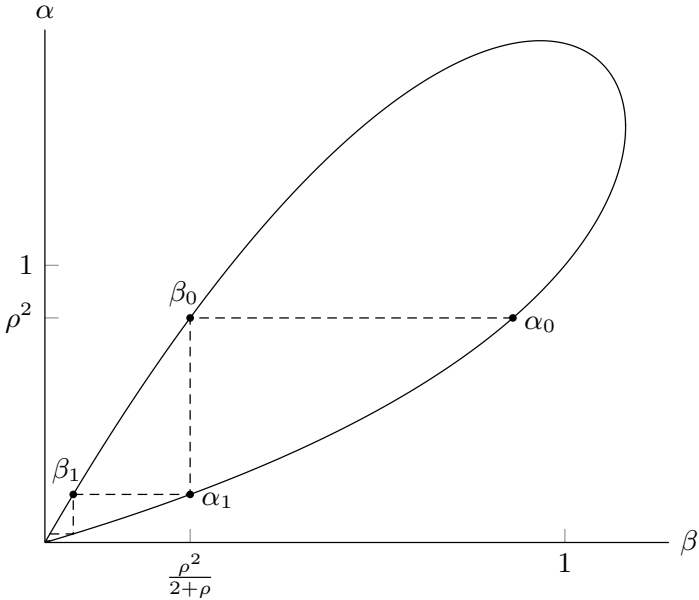
%
\centering%
\includestandalone{Chapters/JSQ/TikZFiles/generating_compensation_parameters}%
\caption{Generating the compensation parameters $\al_i$ and $\be_i$.}%
\label{figJSQ:generating_compensation_parameters}%
\end{figure}%

The compensation procedure ultimately leads to a series expression for the equilibrium probabilities:
\begin{equation}%
p(m,n) = \sum_{i \ge 0} h_i \al_i^m \be_i^n + \sum_{i \ge 0} v_i \al_{i + 1}^m \be_i^n, \quad m \ge 0, ~ n \ge 1. \label{eqnJSQ:series_expression_p(m,n)}
\end{equation}%
If the errors terms converge sufficiently fast to zero, then the series converges. Moreover, if the sum of $p(m,n)$ over all states is absolutely convergent, then it can be normalized to produce the equilibrium distribution and the balance equation \eqref{eqnJSQ:balance_equations_vertical_boundary_simplified} in state $(0,1)$ is also satisfied by this series expression, because we can sum over all other balance equations---which are already satisfied---to produce the balance equation in state $(0,1)$. Hence, what remains to be done is (i) to show that the two series in \eqref{eqnJSQ:series_expression_p(m,n)} converge absolutely and that
\begin{equation}%
\sum_{m \ge 0} \sum_{n \ge 1} |p(m,n)| < \infty;
\end{equation}%
and (ii) to determine the normalization constant.

%%%%%%%%%%%%%%%%%%%%%%%%%%%%%%%%%%%%%%%%%%%%%%%%%%%%%%%
%%%%%%%%%%%%%%%%%%%%%%%%%%%%%%%%%%%%%%%%%%%%%%%%%%%%%%%
%%%%%%%%%%%%%%%%%%%%% NEW SECTION %%%%%%%%%%%%%%%%%%%%%
%%%%%%%%%%%%%%%%%%%%%%%%%%%%%%%%%%%%%%%%%%%%%%%%%%%%%%%
%%%%%%%%%%%%%%%%%%%%%%%%%%%%%%%%%%%%%%%%%%%%%%%%%%%%%%%

\subsection{Proving convergence of the series}%
\label{subsecJSQ:convergence}%

We will study the absolute convergence of the two series in \eqref{eqnJSQ:series_expression_p(m,n)} by determining, for $m \ge 0, ~ n \ge 1$,
\begin{equation}%
R_1(m,n) \defi \lim_{i \to \infty} \biggl| \frac{h_{i + 1} \al_{i + 1}^m \be_{i + 1}^n}{h_i \al_i^m \be_i^n} \biggr|, ~ R_2(m,n) \defi \lim_{i \to \infty} \biggl| \frac{v_{i + 1} \al_{i + 2}^m \be_{i + 1}^n}{v_i \al_{i + 1}^m \be_i^n} \biggr|. \label{eqnJSQ:absolute_convergence_ratios_we_need_to_determine}
\end{equation}%
The coefficients $h_i$ and $v_i$ and the roots $\al_i$ and $\be_i$ are non-zero for all $i$, which allows us to divide by these quantities in \eqref{eqnJSQ:absolute_convergence_ratios_we_need_to_determine}. The coefficients cannot be zero, since this would indicate that there exists a product-form solution that satisfies the balance equations of the interior, horizontal boundary and the vertical boundary. From \cref{figJSQ:pairs_alpha_beta} we know that such solution does not exist. By inspecting \eqref{eqnJSQ:interior_educated_guess} we know that all roots $\al_i$ and $\be_i$ are non-zero.

If the limits \eqref{eqnJSQ:absolute_convergence_ratios_we_need_to_determine} exist and are strictly less than one, then we have proven that the two series in \eqref{eqnJSQ:series_expression_p(m,n)} converge absolutely. We can rewrite \eqref{eqnJSQ:absolute_convergence_ratios_we_need_to_determine} as
\begin{equation}%
R_1(m,n) = \lim_{i \to \infty} \biggl| \frac{\frac{h_{i + 1}}{v_i} \frac{\al_{i + 1}^m}{\be_{i + 1}^m} \frac{\be_{i + 1}^{m + n}}{\al_{i + 1}^{m + n}}}{\frac{h_i}{v_i} \frac{\al_i^m}{\be_i^m} \frac{\be_i^{m + n}}{\al_{i + 1}^{m + n}}} \biggr|, ~ R_2(m,n) = \lim_{i \to \infty} \biggl| \frac{\frac{v_{i + 1}}{h_{i + 1}} \frac{\al_{i + 2}^m}{\be_{i + 1}^m} \frac{\be_{i + 1}^{m + n}}{\al_{i + 1}^{m + n}}}{\frac{v_i}{h_{i + 1}} \frac{\al_{i + 1}^m}{\be_i^m} \frac{\be_i^{m + n}}{\al_{i + 1}^{m + n}}} \biggr|. \label{eqnJSQ:absolute_convergence_ratios_we_need_to_determine_in_ratios_of_compensation_parameters}
\end{equation}%
If we can determine the limits of the fractions present in \eqref{eqnJSQ:absolute_convergence_ratios_we_need_to_determine_in_ratios_of_compensation_parameters} as $i \to \infty$, then we can also determine $R_1(m,n)$ and $R_2(m,n)$.

First, let us study the sequence of $\al$'s and $\be$'s in greater detail. Each $\al_i$ generates a $\be_i$ through \eqref{eqnJSQ:interior_educated_guess} that satisfies $|\be_i| < |\al_i|$ and each $\be_i$ generates an $\al_{i + 1}$ through \eqref{eqnJSQ:interior_educated_guess} that satisfies $|\al_{i + 1}| < |\be_i|$. So, we have the ordering
\begin{equation}%
|\al_0| > |\be_0| > |\al_1| > |\be_1| > \cdots \label{eqnJSQ:ordering_al_be}
\end{equation}%
This indicates that $\al_i$ and $\al_{i + 1}$ are the two roots of \eqref{eqnJSQ:interior_educated_guess} for a fixed $\be = \be_i$ with $|\al_{i + 1}| < |\be_i| < |\al_i|$ and $\be_i$ and $\be_{i + 1}$ are the two roots of \eqref{eqnJSQ:interior_educated_guess} for a fixed $\al = \al_{i + 1}$ with $|\be_{i + 1}| < |\al_{i + 1}| < |\be_i|$. We therefore have that $\al_i$ and $\al_{i + 1}$ satisfy
\begin{equation}%
\al_i \al_{i + 1} = 2 \rho \be_i^2, \quad \al_i + \al_{i + 1} = 2(1 + \rho) \be_i - \be_i^2 \label{eqnJSQ:relation_al}
\end{equation}%
and $\be_i$ and $\be_{i + 1}$ satisfy
\begin{equation}%
\be_i \be_{i + 1} = \frac{\al_{i + 1}^2}{2\rho + \al_{i + 1}}, \quad \be_i + \be_{i + 1} = \frac{2(1 + \rho)}{2\rho + \al_{i + 1}}\al_{i + 1}. \label{eqnJSQ:relation_be}
\end{equation}%
Since $\al_0,\be_0 > 0$ it follows from \eqref{eqnJSQ:relation_al} and \eqref{eqnJSQ:relation_be} by induction that all $\al_i$ and $\be_i$ are positive. More importantly, the parameters $\al_i$ and $\be_i$ decrease geometrically fast, which we establish now.

\begin{lemma}\label{lemJSQ:geometric_decrease_al_and_be}%
There exists $c \in (0,1)$ such that $0 < \al_i, \be_i < c^i, ~ i \ge 0$.
\end{lemma}%

\begin{proof}%
For a fixed $\al$, let $\be$ be the root of \eqref{eqnJSQ:interior_educated_guess} satisfying $\be < \al$. Define $t(\al) \defi \be/\al$. In \cref{lemJSQ:asymptotic_ratios_al_be} we show that $\lim_{\al \downarrow 0} t(\al)$ exists and is less than 1, so that $t(\al) < 1$ for $\al \in [0,\rho^2]$ by \cref{lemJSQ:interior_product-form_solution}. Since the interval $[0,\rho^2]$ is closed and bounded, we have that $c_1 \defi \max_{\al \in [0,\rho^2]} t(\al) < 1$. Perform the same procedure for a fixed $\be$ to obtain a second bound $c_2$. So, $\be_i < \al_i c_1$ and $\al_{i + 1} < \be_i c_2$. Set $c \defi c_1 c_2$ to prove the claim.
\end{proof}%

A consequence of \cref{lemJSQ:geometric_decrease_al_and_be} is that $\al_i \to 0$ and $\be_i \to 0$ as $i \to \infty$.

The following results on the asymptotic behavior of $\be_i/\al_i$ and $\al_{i + 1}/\be_i$ will be used to evaluate \eqref{eqnJSQ:absolute_convergence_ratios_we_need_to_determine_in_ratios_of_compensation_parameters}.

\begin{lemma}[Asymptotic ratios $\al$ and $\be$]\label{lemJSQ:asymptotic_ratios_al_be}\hspace{1em}%
\begin{enumerate}[label = \textup{(\roman*)}]%
\item For a fixed $\al_i$, let $\be_i$ be the root of \eqref{eqnJSQ:interior_educated_guess} with $\be_i < \al_i$. Then, as $i \to \infty$ the ratio $\be_i/\al_i \to \ga_\smallminus$ with $\ga_\smallminus < 1$ the smaller root of
    \begin{equation}%
    0 = 2 \rho \ga^2 - 2(1 + \rho) \ga + 1, \label{eqnJSQ:limiting_ratio_be/al}
    \end{equation}%
    where the roots are
    \begin{equation}%
    \ga_\smallplusminus = \frac{1 + \rho \pm \sqrt{1 + \rho^2}}{2 \rho}.
    \end{equation}%
\item For a fixed $\be_i$, let $\al_{i + 1}$ be the root of \eqref{eqnJSQ:interior_educated_guess} with $\al_{i + 1} < \be_i$. Then, as $i \to \infty$ the ratio $\al_{i + 1}/\be_i \to 1/\ga_\smallplus$ with $\ga_\smallplus > 1$ the larger root of \eqref{eqnJSQ:limiting_ratio_be/al}.
\end{enumerate}%
\end{lemma}%

\begin{proof}%
(i) In \eqref{eqnJSQ:interior_educated_guess}, set $\al = \al_i$ and $\be = \be_i$, divide by $\al_i^2$, set $\ga = \be_i/\al_i$ and let $i \to \infty$ to obtain \eqref{eqnJSQ:limiting_ratio_be/al}. It is easy to see that $\ga_\smallplus > 1/(2\rho)$ for $0 < \rho < 1$ and since $\ga_\smallminus \ga_\smallplus = 1/(2\rho)$ we conclude that $\ga_\smallminus < 1$.

\noindent (ii) In \eqref{eqnJSQ:interior_educated_guess}, set $\al = \al_{i + 1}$ and $\be = \be_i$, divide by $\be_i^2$, set $\zeta = \al_{i + 1}/\be_i$ and let $i \to \infty$ to obtain
\begin{equation}%
0 = \zeta^2 - 2(1 + \rho) \zeta + 2 \rho. \label{eqnJSQ:limiting_ratio_al/be}
\end{equation}%
We are interested in the root of \eqref{eqnJSQ:limiting_ratio_al/be} smaller than one, which is $1/\ga_\smallplus$, since $\zeta$ satisfies the same equation as $1/\ga$.
\end{proof}%

We can also determine $v_i/h_i$ and $h_{i + 1}/v_i$ as $i \to \infty$. This is the final ingredient in the evaluation of \eqref{eqnJSQ:absolute_convergence_ratios_we_need_to_determine_in_ratios_of_compensation_parameters}.

\begin{lemma}[Asymptotic ratios coefficients $h$ and $v$]\label{lemJSQ:asymptotic_ratios_h_v}\hspace{1em}%
\begin{enumerate}[label = \textup{(\roman*)}]%
\item Consider the setting of \textup{\cref{lemJSQ:vertical_compensation}}. Then, as $i \to \infty$,
    \begin{equation}%
    \frac{v_i}{h_i} \to \frac{1/(2\rho) - \ga_\smallminus}{\ga_\smallplus - 1/(2\rho)}.
    \end{equation}%
\item Consider the setting of \textup{\cref{lemJSQ:horizontal_compensation}}. Then, as $i \to \infty$,
    \begin{equation}%
    \frac{h_{i + 1}}{v_i} \to - \frac{\ga_\smallplus}{\ga_\smallminus}.
    \end{equation}%
\end{enumerate}%
\end{lemma}%

\begin{proof}%
(i) Using the indexing of the compensation parameters, \eqref{eqnJSQ:vertical_compensation_coefficient} becomes
\begin{equation}%
v_i = - \frac{\al_{i + 1} - \be_i}{\al_i - \be_i} h_i. \label{eqnJSQ:vertical_compensation_coefficient_with_indices}
\end{equation}%
Divide both sides of \eqref{eqnJSQ:vertical_compensation_coefficient_with_indices} by $h_i$ and multiply by $\be_i/\be_i$ to obtain
\begin{equation}%
\frac{v_i}{h_i} = \frac{1 - \al_{i + 1}/\be_i}{\al_i/\be_i - 1}.
\end{equation}%
For $i \to \infty$, we have by \cref{lemJSQ:asymptotic_ratios_al_be} that $\al_{i + 1}/\be_i \to 1/\ga_\smallplus$ and $\al_i/\be_i \to 1/\ga_\smallminus$. So, for $i \to \infty$,
\begin{equation}%
\frac{v_i}{h_i} \to \frac{1 - 1/\ga_\smallplus}{1/\ga_\smallminus - 1} = \frac{\ga_\smallminus \ga_\smallplus - \ga_\smallminus}{\ga_\smallplus - \ga_\smallminus \ga_\smallplus},
\end{equation}%
and then $\ga_\smallminus \ga_\smallplus = 1/(2\rho)$ proves the claim.

\noindent (ii)
Using the indexing of the compensation parameters, \eqref{eqnJSQ:horizontal_compensation_coefficient} becomes
\begin{equation}%
h_{i + 1} = - \frac{(\rho + \al_{i + 1})/\be_{i + 1} - (1 + \rho)}{(\rho + \al_{i + 1})/\be_i - (1 + \rho) } v_i. \label{eqnJSQ:horizontal_compensation_coefficient_with_indices}
\end{equation}%
Divide both sides of \eqref{eqnJSQ:horizontal_compensation_coefficient_with_indices} by $v_i$ and multiply by $\be_i/\be_i$ to obtain
\begin{equation}%
\frac{h_{i + 1}}{v_i} = - \frac{(\rho + \al_{i + 1})\be_i/\be_{i + 1} - (1 + \rho)\be_i}{(\rho + \al_{i + 1}) - (1 + \rho)\be_i}.
\end{equation}%
For $i \to \infty$ we have that $\al_{i + 1} \to 0$, $\be_i \to 0$ and $\be_i/\be_{i + 1} = \be_i/\al_{i + 1} \cdot \al_{i + 1}/\be_{i + 1} \to \ga_\smallplus / \ga_\smallminus$, which proves the claim.
\end{proof}%

We can now determine the limits \eqref{eqnJSQ:absolute_convergence_ratios_we_need_to_determine_in_ratios_of_compensation_parameters}.
Applying \cref{lemJSQ:asymptotic_ratios_al_be,lemJSQ:asymptotic_ratios_h_v} produces
\begin{align}%
R_1(m,n) = R_2(m,n) = \frac{1/(2\rho) - \ga_\smallminus}{\ga_\smallplus - 1/(2\rho)} \Bigl( \frac{\ga_\smallminus}{\ga_\smallplus} \Bigr)^{m + n - 1}.
\end{align}%
If we define $\theta_\smallplusminus \defi 2\rho \ga_\smallplusminus = 1 + \rho \pm \sqrt{1 + \rho^2}$, then it is easy to see that $\theta_\smallminus < 1$ and $\theta_\smallplus > 1$ for $0 < \rho < 1$. More importantly, for $m \ge 0, ~ n \ge 1$,
\begin{equation}%
R_1(m,n) = R_2(m,n) = \frac{1 - \theta_\smallminus}{\theta_\smallplus - 1} \Bigl( \frac{\theta_\smallminus}{\theta_\smallplus} \Bigr)^{m + n - 1} < 1,
\end{equation}%
because, for $0 < \rho < 1$,
\begin{equation}%
\frac{1 - \theta_\smallminus}{\theta_\smallplus - 1} = 1 + 2\rho \bigl( \rho - \sqrt{1 + \rho^2} \bigr) < 1.
\end{equation}%
Since $R_1(m,n)$ and $R_2(m,n)$ are both less than one, we know that the two series in \eqref{eqnJSQ:series_expression_p(m,n)} converge absolutely. For a series to converge, its summands must tend to zero. So, for $m \ge 0, ~ n \ge 1$,
\begin{equation}%
\lim_{i \to \infty} h_i \al_i^m \be_i^n = 0, \quad \lim_{i \to \infty} v_i \al_{i + 1}^m \be_i^n = 0.
\end{equation}%
This shows that the error terms introduced in each vertical and horizontal compensation step indeed tend to zero.

The continuous-time analog of a result from Foster \cite[Theorem~1]{Foster1953_Ergodicity_condition}, shown in \cref{thmMP:Foster}, states that if the solution $p(m,n)$ satisfies all balance equations, is non-zero, and
\begin{equation}%
\sum_{m \ge 0} \sum_{n \ge 1} |p(m,n)| \le \sum_{m \ge 0} \sum_{n \ge 1} \Bigl( \sum_{i \ge 0} | h_i \al_i^m \be_i^n | + \sum_{i \ge 0} | v_i \al_{i + 1}^m \be_i^n | \Bigr) < \infty, \label{eqnJSQ:normalization_condition_what_we_need_to_prove}
\end{equation}%
then the solution can be normalized to produce the equilibrium distribution. The solution is non-zero because
\begin{equation}%
p(m,n) = \al_0^m \be_0^n + \BigO( \al_1^m \be_0^n ), \quad m \ge 0, ~ n \ge 1,
\end{equation}%
and for $m$ large $p(m,n)$ is positive. We prove that \eqref{eqnJSQ:normalization_condition_what_we_need_to_prove} holds. Since the summands in \eqref{eqnJSQ:normalization_condition_what_we_need_to_prove} are positive, we can interchange the order of the summations to obtain
\begin{align}%
&\sum_{m \ge 0} \sum_{n \ge 1} \Bigl( \sum_{i \ge 0} | h_i \al_i^m \be_i^n | + \sum_{i \ge 0} | v_i \al_{i + 1}^m \be_i^n | \Bigr) \notag \\
&= \sum_{i \ge 0} \frac{|h_i|}{1 - |\al_i|} \frac{|\be_i|}{1 - |\be_i|} + \sum_{i \ge 0} \frac{|v_i|}{1 - |\al_{i + 1}|} \frac{|\be_i|}{1 - |\be_i|}.
\end{align}%
We show that the two series converge. To that end, define
\begin{equation}%
R_3 \defi \lim_{i \to \infty} \biggl| \frac{\frac{|h_{i + 1}|}{1 - |\al_{i + 1}|} \frac{|\be_{i + 1}|}{1 - |\be_{i + 1}|}}{\frac{|h_i|}{1 - |\al_i|} \frac{|\be_i|}{1 - |\be_i|}} \biggr|, ~ R_4 \defi \lim_{i \to \infty} \biggl| \frac{\frac{|v_{i + 1}|}{1 - |\al_{i + 2}|} \frac{|\be_{i + 1}|}{1 - |\be_{i + 1}|}}{\frac{|v_i|}{1 - |\al_{i + 1}|} \frac{|\be_i|}{1 - |\be_i|}} \biggr|,
\end{equation}%
which can be written as
\begin{align}%
R_3 &= \lim_{i \to \infty} \biggl| \frac{\frac{|h_{i + 1}|}{|v_i|} \frac{1}{1 - |\al_{i + 1}|} \frac{1}{1 - |\be_{i + 1}|} \frac{|\be_{i + 1}|}{|\al_{i + 1}|}}{ \frac{|h_i|}{|v_i|} \frac{1}{1 - |\al_i|} \frac{1}{1 - |\be_i|} \frac{|\be_i|}{|\al_{i + 1}|}} \biggr|, \\
R_4 &= \lim_{i \to \infty} \biggl| \frac{\frac{|v_{i + 1}|}{|h_{i + 1}|} \frac{1}{1 - |\al_{i + 2}|} \frac{1}{1 - |\be_{i + 1}|} \frac{|\be_{i + 1}|}{|\al_{i + 1}|}}{ \frac{|v_i|}{|h_{i + 1}|} \frac{1}{1 - |\al_{i + 1}|} \frac{1}{1 - |\be_i|} \frac{|\be_i|}{|\al_{i + 1}|}} \biggr|.
\end{align}%
By applying the results of \cref{lemJSQ:asymptotic_ratios_al_be,lemJSQ:asymptotic_ratios_h_v} and the fact that $\al_i \to 0$ and $\be_i \to 0$ as $i \to \infty$, we find
\begin{equation}%
R_3 = R_4 = \frac{1 - \theta_\smallminus}{\theta_\smallplus - 1} < 1,
\end{equation}%
so that \eqref{eqnJSQ:normalization_condition_what_we_need_to_prove} holds.

In conclusion, due to \cref{thmMP:Foster}, the series in \eqref{eqnJSQ:series_expression_p(m,n)} is the unique (up to a multiplicative constant) solution to the balance equations \eqref{eqnJSQ:balance_equations_interior}, \eqref{eqnJSQ:balance_equations_vertical}, \eqref{eqnJSQ:balance_equations_interior_boundary_simplified} and \eqref{eqnJSQ:balance_equations_vertical_boundary_simplified} and can be normalized to produce the equilibrium distribution. Divide \eqref{eqnJSQ:series_expression_p(m,n)} by the normalization constant $C$ and merge the two series to obtain
\begin{equation}%
p(m,n) = C^{-1} \sum_{i \ge 0} (h_i \al_i^m + v_i \al_{i + 1}^m) \be_i^n, \quad m \ge 0, ~ n \ge 1. \label{eqnJSQ:normalized_series_expression_p(m,n)}
\end{equation}%
%

%%%%%%%%%%%%%%%%%%%%%%%%%%%%%%%%%%%%%%%%%%%%%%%%%%%%%%%
%%%%%%%%%%%%%%%%%%%%%%%%%%%%%%%%%%%%%%%%%%%%%%%%%%%%%%%
%%%%%%%%%%%%%%%%%%%%% NEW SECTION %%%%%%%%%%%%%%%%%%%%%
%%%%%%%%%%%%%%%%%%%%%%%%%%%%%%%%%%%%%%%%%%%%%%%%%%%%%%%
%%%%%%%%%%%%%%%%%%%%%%%%%%%%%%%%%%%%%%%%%%%%%%%%%%%%%%%

\subsection{Normalization constant}%
\label{subsecJSQ:normalization_constant}%

We use the PGF $\PGF{x,y}$ to determine the normalization constant $C$. First, eliminate the $p(m,0), ~ m \ge 0$ in the definition of $\PGF{x,y}$ using \eqref{eqnJSQ:balance_equations_horizontal} and \eqref{eqnJSQ:balance_equations_origin} to get
\begin{align}%
\PGF{x,y} &= p(0,0) + \sum_{m \ge 1} p(m,0) x^m + \sum_{m \ge 0} \sum_{n \ge 1} p(m,n) x^m y^n \notag \\
&= \frac{1}{\rho} p(0,1) + \frac{1}{1 + \rho} \sum_{m \ge 1} \bigl( 2\rho p(m - 1,1) + p(m,1) \bigr) x^m \notag \\
&\quad + \sum_{m \ge 0} \sum_{n \ge 1} p(m,n) x^m y^n. \label{eqnJSQ:probability_generating_function_no_p(m,0)}
\end{align}%
Second, substituting the series expression \eqref{eqnJSQ:normalized_series_expression_p(m,n)} into \eqref{eqnJSQ:probability_generating_function_no_p(m,0)} gives
\begin{align}%
\PGF{x,y} &= C^{-1} \Bigl[ \, \frac{1}{\rho} \sum_{i \ge 0} (h_i + v_i) \be_i \notag \\
&\quad + \frac{1}{1 + \rho} \sum_{m \ge 1} \sum_{i \ge 0} \bigl( h_i (2\rho + \al_i) \al_i^{m - 1} + v_i (2\rho + \al_{i + 1}) \al_{i + 1}^{m - 1} \bigr) \be_i x^m \notag \\
&\quad + \sum_{m \ge 0} \sum_{n \ge 1} \sum_{i \ge 0} (h_i \al_i^m + v_i \al_{i + 1}^m) \be_i^n x^m y^n \Bigr].
\end{align}%
Third, changing the order of the summations and simplifying the geometric series finally gives
\begin{align}%
\PGF{x,y} &= C^{-1} \Bigl[ \, \frac{1}{\rho} \sum_{i \ge 0} (h_i + v_i) \be_i \notag \\
&\quad + \frac{1}{1 + \rho} \sum_{i \ge 0} \bigl( h_i \frac{(2\rho + \al_i) x}{1 - \al_i x} + v_i \frac{(2\rho + \al_{i + 1}) x}{1 - \al_{i + 1} x} \bigr) \be_i \notag \\
&\quad + \sum_{i \ge 0} \bigl( h_i \frac{1}{1 - \al_i x} + v_i \frac{1}{1 - \al_{i + 1} x} \bigr) \frac{\be_i y}{1 - \be_i y} \Bigr]. \label{eqnJSQ:probability_generating_function_in_terms_of_compensation_parameters}
\end{align}%
Notice that the PGF $\PGF{x,y}$ is valid for $|x| < 1/\al_0$ and $|y| < 1/\be_0$. The expression \eqref{eqnJSQ:probability_generating_function_in_terms_of_compensation_parameters} is called a \textit{partial fraction decomposition} of the PGF $\PGF{x,y}$. This decomposition shows that $x = 1/\al_i$ and $y = 1/\be_i$ are the simple poles of $\PGF{x,y}$, which implies that the function $\PGF{x,y}$ approaches infinity as $x$ approaches $1/\al_i$ or $y$ approaches $1/\be_i$.

We determine the normalization constant by deriving two expressions for the leading term in the asymptotic expansion of $\PGF{x,0}$ as $x \uparrow 1/\al_0$. To that end, we set $y = 0$ in \eqref{eqnJSQ:probability_generating_function_in_terms_of_compensation_parameters} to obtain
\begin{align}%
\PGF{x,0} &= C^{-1} \Bigl[ \, \frac{1}{\rho} \sum_{i \ge 0} (h_i + v_i) \be_i \notag \\
&\quad + \frac{1}{1 + \rho} \sum_{i \ge 0} \bigl( h_i \frac{(2\rho + \al_i) x}{1 - \al_i x} + v_i \frac{(2\rho + \al_{i + 1}) x}{1 - \al_{i + 1} x} \bigr) \be_i \Bigr].
\end{align}%
Now, as $x \uparrow 1/\al_0 = 1/\rho^2$,
\begin{align}%
\PGF{x,0} &= C^{-1} \frac{1}{1 + \rho} h_0 \frac{(2\rho + \al_0)\frac{1}{\rho^2}}{1 - \al_0 x} \be_0 + \BigO(1) \notag \\
&= \frac{1}{C \rho (1 + \rho) (\frac{1}{\rho^2} - x)} + \BigO(1), \label{eqnJSQ:P(x,0)_asymptotic_expansion_1}
\end{align}%
where we used that $h_0 = 1$ and $\be_0 = \rho^2/(2 + \rho)$ and property \eqref{eqnJSQ:ordering_al_be}.

For a second expression for the leading term, we investigate the functional equation \eqref{eqnJSQ:functional_equation}. If we pick the pair $(x,y)$ such that $h_1(x,y) = 0$ and $|x| < 1/\al_0$, $|y| < 1/\be_0$, then we find that $\PGF{x,0}$ and $\PGF{0,y}$ are related according to
\begin{equation}%
0 = h_2(x,y) \PGF{x,0} + h_3(x,y) \PGF{0,y}. \label{eqnJSQ:relation_P(x,0)_and_P(0,y)}
\end{equation}%
Apply relation \eqref{eqnJSQ:relation_P(x,0)_and_P(0,y)} to three pairs $(x,y)$ in the following order: $(1/(2\rho),1)$, $(1/(2\rho),1/\rho)$ and $(1/\rho^2,1/\rho)$. All three pairs satisfy $h_1(x,y) = 0$. For the first pair $(x,y) = (1/(2\rho),1)$ we have
\begin{equation}%
0 = h_2(\frac{1}{2\rho},1) \PGF{\frac{1}{2\rho},0} + h_3(\frac{1}{2\rho},1) \PGF{0,1}. \label{eqnJSQ:relation_P(1/2rho,0)_and_P(0,1)}
\end{equation}%
Notice that $\PGF{0,1}$ is the fraction of time the first server is idle. The offered load to the system is $2\rho$ per unit time, so that by symmetry we know that $\PGF{0,1} = 1 - \rho$. So, from \eqref{eqnJSQ:relation_P(1/2rho,0)_and_P(0,1)} we obtain that $\PGF{1/(2\rho),0} = 1 - \rho$. For the second pair $(x,y) = (1/(2\rho),1/\rho)$ we have
\begin{equation}%
0 = h_2(\frac{1}{2\rho},\frac{1}{\rho}) \PGF{\frac{1}{2\rho},0} + h_3(\frac{1}{2\rho},\frac{1}{\rho}) \PGF{0,\frac{1}{\rho}} \label{eqnJSQ:relation_P(1/2rho,0)_and_P(0,1/rho)}
\end{equation}%
and find $\PGF{0,1/\rho} = (1 - \rho)(2 - \rho)$. Now, for the third pair $(x,y) = (1/\rho^2,1/\rho)$, we let $x \uparrow 1/\rho^2$ and $y \to 1/\rho$. To that end, we need the solution of $h_1(x,y) = 0$ for a fixed $x$. This solution is given by $y = \upsilon(x)$ with
\begin{equation}%
\upsilon(x) = (1 + \rho) x - \sqrt{x(x(1 + \rho^2) - 1)}.
\end{equation}%
Observe that if $x \uparrow 1/\rho^2$, then $\upsilon(x) \to 1/\rho$. Substituting the pair $(x,y) = (x,\upsilon(x))$ into \eqref{eqnJSQ:relation_P(x,0)_and_P(0,y)} gives the relation
\begin{equation}%
\PGF{x,0} = - \frac{h_3(x,\upsilon(x))}{h_2(x,\upsilon(x))} \PGF{0,\upsilon(x)}. \label{eqnJSQ:relation_P(x,0)_and_P(0,upsilon(x))}
\end{equation}%
Then, as $x \uparrow 1/\rho^2$ we find that $\PGF{0,\upsilon(x)} \to \PGF{0,1/\rho} = (1 - \rho)(2 - \rho)$, $h_3(x,\upsilon(x)) \to h_3(1/\rho^2,1/\rho) = (1 - 1/\rho) / \rho^2$, and
\begin{equation}%
h_2(x,\upsilon(x)) = - \frac{(1 - \rho)(2 + \rho)}{2\rho} \bigl( x - \frac{1}{\rho^2} \bigr) + \SmallO(x - \frac{1}{\rho^2}).
\end{equation}%
By combining these asymptotic results, we obtain from \eqref{eqnJSQ:relation_P(x,0)_and_P(0,upsilon(x))} a second expression for the leading term in the asymptotic expansion of $\PGF{x,0}$. For $x \uparrow 1/\rho^2$,
\begin{equation}%
\PGF{x,0} = \frac{2(1 - \frac{1}{\rho})(2 - \rho)}{\rho (2 + \rho)(x - \frac{1}{\rho^2})} + \BigO(1). \label{eqnJSQ:P(x,0)_asymptotic_expansion_2}
\end{equation}%
Finally, combining \eqref{eqnJSQ:P(x,0)_asymptotic_expansion_1} and \eqref{eqnJSQ:P(x,0)_asymptotic_expansion_2} gives, as $x \uparrow 1/\rho^2$,
\begin{equation}%
\frac{1}{C \rho (1 + \rho) (\frac{1}{\rho^2} - x)} = \frac{2(1 - \frac{1}{\rho})(2 - \rho)}{\rho (2 + \rho)(x - \frac{1}{\rho^2})}.
\end{equation}%
Solving this relation for $C$ gives the explicit expression
\begin{equation}%
C = \frac{\rho(2 + \rho)}{2(1 - \rho^2)(2 - \rho)}.
\end{equation}%
%

%%%%%%%%%%%%%%%%%%%%%%%%%%%%%%%%%%%%%%%%%%%%%%%%%%%%%%%
%%%%%%%%%%%%%%%%%%%%%%%%%%%%%%%%%%%%%%%%%%%%%%%%%%%%%%%
%%%%%%%%%%%%%%%%%%%%% NEW SECTION %%%%%%%%%%%%%%%%%%%%%
%%%%%%%%%%%%%%%%%%%%%%%%%%%%%%%%%%%%%%%%%%%%%%%%%%%%%%%
%%%%%%%%%%%%%%%%%%%%%%%%%%%%%%%%%%%%%%%%%%%%%%%%%%%%%%%

\section{Comparison with random routing}%
\label{secJSQ:comparison_with_random_routing}%

The compensation procedure allows us to easily calculate the equilibrium distribution using \cref{algJSQ:compensation_approach}. From the equilibrium distribution we can determine performance measures such as the expected number of jobs in the system. Let $X$ denote the total number of jobs in the system in equilibrium. Then,
\begin{align}%
\Prob{X = 0} &= p(0,0), \label{eqnJSQ:equilibrium_number_of_jobs_prob_0}\\
\Prob{X = x} &= \sum_{m = 0}^x p(m,x - m) + \sum_{m = 0}^{x - 1} p(m,m - x)  \notag \\
&= p(x,0) + 2 \sum_{m = 0}^{x - 1} p(m,x - m), \quad x \ge 1, \label{eqnJSQ:equilibrium_number_of_jobs_prob_l}
\end{align}%
where we used $p(m,x - m) = p(m,m - x)$ by symmetry, and therefore
\begin{equation}%
\E{X} = \sum_{x \ge 0} x \bigl( p(x,0) + 2 \sum_{m = 0}^{x - 1} p(m,x - m) \bigr). \label{eqnJSQ:expected_number_of_jobs}
\end{equation}%

For numerical purposes the number of compensation steps needs to be finite and the infinite summation in \eqref{eqnJSQ:expected_number_of_jobs} should be truncated. We first present a simple method to perform an appropriate number of compensation steps, see \cref{algJSQ:number_of_compensation_steps}. Essentially, \cref{algJSQ:number_of_compensation_steps} is the same as \cref{algJSQ:compensation_approach}, but now selects the number $K$ according to some preset target level: when the relative change in the equilibrium probability $p(m,n)$ goes below a certain threshold $\epsilon$, the compensation procedure is terminated.

\begin{algorithm}%
\caption{Number of compensation steps}%
\label{algJSQ:number_of_compensation_steps}%
\begin{algorithmic}[1]%
\State Select $\epsilon$ small and positive and a state $(m,n), ~ m \ge 0, ~ n \ge 1$
\State Set $h_0 = 1$, $\al_0 = \rho^2$, $\be_0 = \rho^2/(2 + \rho)$
\State Perform a vertical compensation step
\State Calculate $p_0(m,n) = (h_0 \al_0^m + v_0 \al_1^m) \be_0^n$
\State Perform a horizontal and vertical compensation step
\State Calculate $p_1(m,n) = \sum_{i = 0}^1 (h_i \al_i^m + v_i \al_{i + 1}^m) \be_i^n$
\State $K = 1$
\While{$(|p_{K}(m,n) - p_{K - 1}(m,n)|)/(|p_{K - 1}(m,n)|) > \epsilon$}
    \State $K = K + 1$
    \State Perform a horizontal and vertical compensation step
    \State Calculate
    \begin{align}%
    p_K(m,n) &= p_{K - 1}(m,n) + (h_K \al_K^m + v_K \al_{K + 1}^m) \be_K^n \notag \\
    &= \sum_{i = 0}^K (h_i \al_i^m + v_i \al_{i + 1}^m) \be_i^n
    \end{align}%
\EndWhile
\end{algorithmic}%
\end{algorithm}%

One way to choose the truncation level of the infinite series \eqref{eqnJSQ:expected_number_of_jobs} is described in \cref{algJSQ:truncation_level_EL}. We base the truncation level on the criterion that almost all probability mass is captured in the distribution of $X$.

\begin{algorithm}%
\caption{Truncation level $\E{X}$}%
\label{algJSQ:truncation_level_EL}%
\begin{algorithmic}[1]%
\State Select $\epsilon$ small and positive
\State Use \cref{algJSQ:number_of_compensation_steps} to construct the required equilibrium probabilities
\State $K = 0$
\State Calculate $\Prob{X = 0}$ using \eqref{eqnJSQ:equilibrium_number_of_jobs_prob_0}
\While{$\sum_{x = 0}^K \Prob{X = x} < 1 - \epsilon$}
    \State $K = K + 1$
    \State Calculate $\Prob{X = K}$ using \eqref{eqnJSQ:equilibrium_number_of_jobs_prob_l}
\EndWhile
\end{algorithmic}%
\end{algorithm}%

\cref{algJSQ:number_of_compensation_steps,algJSQ:truncation_level_EL} allow us to determine $\E{X}$ to any prescribed accuracy. We can compare these results with a naive random routing policy and demonstrate that the join the shortest queue routing policy is superior.

Random routing means that each job joins either queue with equal probability, irrespective of the number of jobs at each server. Due to the Poisson splitting, random routing ensures that each queue operates as an $M/M/1$ queue with arrival rate $\rho$ and equilibrium probabilities $(1 - \rho)\rho^i$. We denote by $X_{\textup{RR}}$ the total number of jobs in the system with random routing and derive
\begin{equation}%
\Prob{X_{\textup{RR}} = x} = \sum_{k = 0}^x (1 - \rho) \rho^{x - k} (1 - \rho) \rho^k = (x + 1) (1 - \rho)^2 \rho^x.
\end{equation}%
Then, we get that
\begin{equation}%
\E{X_{\textup{RR}}} = \sum_{x \ge 0} x \Prob{X_{\textup{RR}} = x} = (1 - \rho)^2 \sum_{x \ge 0} x (x + 1) \rho^x = \frac{2\rho}{1 - \rho}.
\end{equation}%
This result is also easily derived from the fact that under random routing both servers have independent Poisson input and the expected total number of jobs is the sum of the expected number of jobs in each queue ($\rho/(1 - \rho)$).

\cref{figJSQ:JSQ_vs_RR} compares join the shortest queue routing to random routing for various values of $\rho$. In terms of the expected number of jobs in the system, join the shortest queue routing is superior to random routing. For small $\rho$, an arriving job usually finds an empty system. In that case, both routing policies operate equally well. For larger $\rho$, join the shortest queue routing outperforms random routing. This routing policy balances the number of jobs at each server, and therefore utilizes the servers more efficiently than the random routing policy. Moreover, as $\rho \uparrow 1$ the join the shortest queue system behaves as a pooled system, which means that it behaves as if there is a single queue served by two servers instead of two separate queues with one server each.

%\begin{table}%
%\centering%
%\begin{tabular}{c|*{8}{c}}%
% & \multicolumn{8}{c}{$\rho$} \\
% & 0.3 & 0.4 & 0.5 & 0.6 & 0.7 & 0.8 & 0.9 & 0.95 \\
% \hline
%$\E{L}$ & 0.54 & 0.75 & 1.0 & 1.3 & 1.8 & 2.8 & 5.4 & 10 \\
%$\E{L_{\textup{RR}}}$ & 0.86 & 1.3 & 2.0 & 3.0 & 4.7 & 8.0 & 18 & 38 \\
%$\Delta$ & -37\% & -44\% & -50\% & -56\% & -61\% & -65\% & -70\% & -73\%
%\end{tabular}%
%\caption{Comparison of join the shortest queue routing and random routing with $\Delta = \frac{\E{L} - \E{L_{\textup{RR}}}}{\E{L_{\textup{RR}}}} \cdot 100 \%$. We use \cref{algJSQ:number_of_compensation_steps} with $\epsilon = 10^{-10}$ and $(m,n) = (0,1)$ and \cref{algJSQ:truncation_level_EL} with $\epsilon = 10^{-10}$.}%
%\label{tblJSQ:JSQ_vs_RR}%
%\end{table}%

\begin{figure}
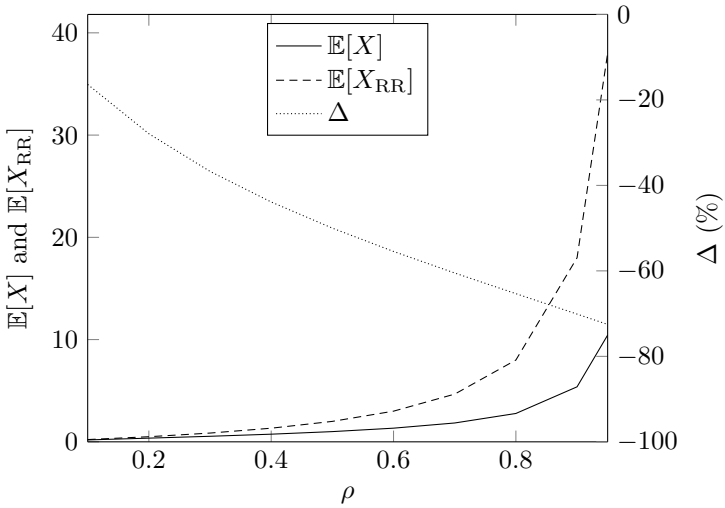
%
\centering%
\includestandalone{Chapters/JSQ/TikZFiles/comparison_JSQ_and_RR}%
\caption{Comparison of join the shortest queue routing and random routing with $\Delta = (\E{X} - \E{X_{\textup{RR}}})/\E{X_{\textup{RR}}} \cdot 100 \%$. We use \protect\cref{algJSQ:number_of_compensation_steps} with $\epsilon = 10^{-10}$ and $(m,n) = (0,1)$ and \protect\cref{algJSQ:truncation_level_EL} with $\epsilon = 10^{-10}$.}%
\label{figJSQ:JSQ_vs_RR}%
\end{figure}%

%%%%%%%%%%%%%%%%%%%%%%%%%%%%%%%%%%%%%%%%%%%%%%%%%%%%%%%
%%%%%%%%%%%%%%%%%%%%%%%%%%%%%%%%%%%%%%%%%%%%%%%%%%%%%%%
%%%%%%%%%%%%%%%%%%%%% NEW SECTION %%%%%%%%%%%%%%%%%%%%%
%%%%%%%%%%%%%%%%%%%%%%%%%%%%%%%%%%%%%%%%%%%%%%%%%%%%%%%
%%%%%%%%%%%%%%%%%%%%%%%%%%%%%%%%%%%%%%%%%%%%%%%%%%%%%%%

\section{Takeaways}%
\label{secJSQ:what_have_we_learned}%

The straightforward choice of taking the number of jobs at each queue as the dimensions of the Markov process led to an inhomogeneous transition rate structure. By performing a simple coordinate transformation and using the symmetry of the two servers and the join the shortest queue routing we were able to formulate a Markov process that did have a homogeneous transition rate structure in the interior. Due to this symmetry, we only needed to determine the equilibrium probabilities for the states $(m,n)$ with $m \ge 0$ and $n \ge 1$.

The compensation approach worked by linearly combining product-form solutions to satisfy all balance equations. These product-form solutions all satisfied the balance equations of the interior. In each step of the compensation procedure, a single product-form solution was added to the linear combination so that the resulting linear combination satisfied the balance equations on one of the two boundaries. In the next step, a single product-form solution was added to satisfy the balance equations on the other boundary. This process was repeated and finally led to an infinite sum of product-form solutions. Then, showing that this infinite sum converged, established that it was the unique equilibrium distribution.

For the gated single-server system in \cref{ch:gated_single-server}, compensation was only necessary on a single boundary. For the join the shortest queue model, however, we had to compensate on two boundaries. This creates two different, alternating compensation steps. The compensation approach applied to the gated single-server system is therefore inherently `simpler', which was demonstrated by the fact that the parameters $\al_i$ and $\be_i$ can be obtained explicitly, whereas this was not possible for the join the shortest queue system. Furthermore, for the gated single-server system $\al_i,\be_i$ did not tend to zero, while the coefficients $c_i$ did, and for the join the shortest queue system this is reversed: $\al_i,\be_i$ tended to zero, while the coefficients $h_i,v_i$ did not.

The compensation approach is not limited to the join the shortest queue system. It applies to a more general class of models, which we now briefly describe. For a Markov process in the positive quadrant, the compensation approach can be applied when it obeys the following conditions: (i) there should be only transitions to neighboring states; (ii) in the interior of the state space, there should be no transitions to the North, North-East, and East; and (iii) there should a homogeneous structure in terms of the transitions, i.e., the transition structure and the rate at which these transitions occur should be the same for all states in the interior, for all states on the vertical boundary, and for all states on the horizontal boundary. It can be shown that these conditions imply that $\al_i,\be_i \to 0$, which, as we saw in \cref{ch:gated_single-server}, is not necessary for convergence of the series expression for $p(m,n)$. For the gated single-server system of \cref{ch:gated_single-server}, the first and second condition are violated, but in this case convergence of the infinite sum of product forms is guaranteed by convergence to zero of the coefficients.

The compensation approach is also applied in \cite{Adan1996_SED_Erlang_servers}, which considers a system with Erlang-$r$ distributed service times and arriving jobs joining the queue with the least number of remaining service phases. The Markov process associated with this queueing system has transitions in the interior that are not restricted to neighboring states, but the compensation approach can be still be applied to determine the equilibrium probabilities. Hence, we know that the compensation approach also applies to some models that do not fit within the above class of models.

%%%%%%%%%%%%%%%%%%%%%%%%%%%%%%%%%%%%%%%%%%%%%%%%%%%%%%%
%%%%%%%%%%%%%%%%%%%%%%%%%%%%%%%%%%%%%%%%%%%%%%%%%%%%%%%
%%%%%%%%%%%%%%%%%%%%%%%% NOTES %%%%%%%%%%%%%%%%%%%%%%%%
%%%%%%%%%%%%%%%%%%%%%%%%%%%%%%%%%%%%%%%%%%%%%%%%%%%%%%%
%%%%%%%%%%%%%%%%%%%%%%%%%%%%%%%%%%%%%%%%%%%%%%%%%%%%%%%

%\theendnotes%
%\setcounter{endnote}{0}
\printendnotes%
%

%%%%%%%%%%%%%%%%%%%%%%%%%%%%%%%%%%%%%%%%%%%%%%%%%%%%%%%
%%%%%%%%%%%%%%%%%%%%%%%%%%%%%%%%%%%%%%%%%%%%%%%%%%%%%%%
%%%%%%%%%%%%%%%%%%%%%%% APPENDIX %%%%%%%%%%%%%%%%%%%%%%
%%%%%%%%%%%%%%%%%%%%%%%%%%%%%%%%%%%%%%%%%%%%%%%%%%%%%%%
%%%%%%%%%%%%%%%%%%%%%%%%%%%%%%%%%%%%%%%%%%%%%%%%%%%%%%%

\addtocontents{toc}{\protect\vspace{1.25em}}
%\appendix

%\input{Appendix/interchange_order}

%\input{Appendix/complex_analysis}

%\bibliographystyle{plain}%
{\small%
}%

%%%%%% Add some sort of \phantom part here to fix the ToC issue

\chapter*[Notation index]{Notation index}%
\addcontentsline{toc}{chapter}{Notation index}
\label{ch:notation_index}%

Vectors are denoted by bold lowercase letters or numbers. Matrices are denoted by uppercase letters. Unless stated otherwise, indexing of vectors and matrices starts at 0. Aside from the named number sets, all sets are denoted by calligraphic letters such as $\set{A}$.

\begin{longtable}{ p{.18\textwidth} p{.73\textwidth} }
$\defi$ & defined as \\
$\dequal$ & equal in distribution \\
%$\dconv$ & convergence in distribution
$\zerob$ & vector of zeros of appropriate dimension \\
$\ind{A}$ & indicator function of the event $A$ \\
$\oneb$ & vector of ones of appropriate dimension \\
$(A)_{i,j}$ & element $(i,j)$ of matrix $A$ \\
$A^{-1}$ or $(A)^{-1}$ & inverse of a matrix $A$ \\
$\set{A}^c$ & complement of a set $\set{A}$ \\
%$\Ber{p}$ & Bernoulli distribution with parameter $p$ and support $\{0,1\}$ \\
%$A \intersection B$ & intersection of sets $A$ and $B$ \\
%$A \union B$ & union of sets $A$ and $B$ \\
$\Complex$ & set of complex numbers \\
$\det(A)$ & determinant of a matrix $A$ \\
$\eb{i}$ & vector of zeros of appropriate dimension with a 1 at position $i$ \\
$\E{X}$ & expectation of a random variable $X$ \\
$\E{x}{f(X)}$ & expectation of a functional of a process $\{ X(t) \}_{t \ge 0}$ given $X(0) = x$ \\
$\Erl{n}{\la}$ & Erlang-$n$ distribution with parameter $\la$ \\
$\Exp{\la}$ & exponential distribution with parameter $\la$ \\
$f_X(\cdot)$ & probability density function of a random variable $X$ \\
$F_X(\cdot)$ & cumulative distribution function of a random variable $X$ \\
%$G$ & auxiliary matrix of the matrix-analytic method \\
$\Geo{p}$ & geometric distribution with failure probability $p$ and support $\Nat_0$ (or sometimes $\Nat$)  \\
$\ImagPart{\PGFarg}$ & imaginary part of $\PGFarg \in \Complex$ \\
$\complexunit$ & complex unit \\
$\La_n$ & transition rate submatrices in a QBD or QSF process from level $i$ to level $i + n$, independent of $i$ \\
$\La_n^{(i)}$ & transition rate submatrices in a QBD or QSF process from level $i$ to level $i + n$ \\
$\LST{X}{\LSTarg}$ & Laplace-Stieltjes transform of the random variable $X$ evaluated at the point $\LSTarg$ \\
$\Nat$, $\Nat_0$ & $\Nat = \{ 1,2,3,\ldots \}$, $\Nat_0 = \{ 0 \} \cup \Nat$ \\
$\PGF{X}{\PGFarg}$ & probability generating function of the random variable $X$ evaluated at the point $\PGFarg$ \\
%$\Normal{\mu,\sigma^2}$ & normal distribution with mean $\mu$ and variance $\sigma^2$ \\
$\Poisson{\la}$ & Poisson distribution with parameter $\la$ \\
$\Prob{A}$ & probability of event $A$ \\
$\Prob{A \mid B}$ & conditional probability \\
$\Prob{x}{f(X)}$ & probability of a functional of a process $\{ X(t) \}_{t \ge 0}$ given $X(0) = x$ \\
$Q$ & transition rate matrix of a Markov process \\
%$R$ & rate matrix of the matrix-geometric method \\
$\Real$ & set of real numbers \\
$\RealPart{\PGFarg}$ & real part of $\PGFarg \in \Complex$ \\
$\statespace$ & state space of a Markov process \\
$\Std{X}$ & standard deviation of a random variable $X$ \\
$\closedunitdisc$, $\unitcircle$ & closed unit disc and unit circle \\
$\vca{v}^\transpose$ & transpose of a vector $\vca{v}$ \\
$\Var{X}$ & variance of a random variable $X$ \\
$X \sim \mu$ & the random variable $X$ has distribution $\mu$ \\
$\Int$ & set of integer numbers \\
\end{longtable} 

\chapter*[Abbreviation index]{Abbreviation index}%
\addcontentsline{toc}{chapter}{Abbreviation index}
\label{ch:abbreviation_index}%

\begin{longtable}{ p{.15\textwidth}  p{.75\textwidth} }
BD & birth--and--death \\
%c.c.d.f. & complementary cumulative distribution function \\
%c.d.f. & cumulative distribution function \\
%CTMC & continuous time Markov chain \\
%DTMC & discrete time Markov chain \\
iff & if and only if \\
i.i.d. & independent and identically distributed \\
LST & Laplace-Stieltjes transform \\
LT & Laplace transform \\
PASTA & Poisson arrivals see time-averages \\
%p.d.f. & probability density function \\
%p.m.f. & probability mass function \\
PGF & probability generating function \\
QBD & quasi-birth--and--death
\end{longtable} 

\end{document}% 